\documentclass[a4paper]{amsart}
\usepackage[english]{babel}
\usepackage[utf8]{inputenc}
\usepackage[T1]{fontenc}
\usepackage{csquotes} 
\usepackage{amssymb}
\usepackage{float} 
\usepackage{amsthm}
\usepackage[top=2cm, bottom=2cm, left=2cm, right=2cm,twoside=false]{geometry}
\usepackage{mathtools} 
\usepackage[usenames,x11names]{xcolor}
\usepackage[all]{xy} 
\usepackage[normalem]{ulem}
\usepackage{multirow} 

\usepackage{array}
\usepackage{longtable}

\usepackage[shortlabels]{enumitem}
\setlist[1]{leftmargin=*}
\setlist[enumerate,1]{label=(\alph*)}
\setlist[enumerate,2]{label=(\roman*), ref=(\alph{enumi}.\roman*)}
\newlist{enumerate-alt}{enumerate}{1}
\setlist[enumerate-alt,1]{label=(\roman*)}
\newlist{longlist}{enumerate}{1}
\setlist[longlist]{label=\small{(\arabic*)}, itemsep=0.2em}
\newlist{shortlist}{enumerate}{1}
\setlist[shortlist]{label=\small{(\Alph*)}}
\newlist{parlist}{enumerate}{1}
\setlist[parlist]{leftmargin=0cm, itemindent=2\parindent, label=(\alph*), itemsep=0.5em}
\newlist{steps}{enumerate}{1}
\setlist[steps]{align=left, listparindent=\parindent, parsep=\parskip, leftmargin=0em, labelwidth=0pt, itemindent=1em,labelsep=.4em, topsep=.4em, itemsep=.4em}	
\setlist[steps,1]{label={\textbf{Step~\arabic*.}},  ref=\textbf{\arabic*}}
\newlist{parts}{enumerate}{3}
\setlist[parts,1]{leftmargin=0cm, itemindent=2\parindent, label=(\textbf{\arabic*}), ref=(\textbf{\arabic*}), itemsep=0.2em}
\setlist[parts,2]{label=(\alph*), ref=(\arabic{partsi}.\alph*), topsep=0.2em}
\setlist[parts,3]{label=(\roman*), ref=(\arabic{partsi}.\alph{partsii}.\roman*)}

\newcounter{foo} 

\newlist{casesp}{enumerate}{5} 
\setlist[casesp]{align=left, 
	listparindent=\parindent, 
	parsep=\parskip, 
	font=\normalfont\bfseries, 
	leftmargin=0pt, 
	labelwidth=0pt, 
	itemindent=.4em,labelsep=.4em, 
	topsep=.6em, 
	itemsep=.6em, 
}
\setlist[casesp,1]{label=Case~\arabic*:,ref=\arabic*}
\setlist[casesp,2]{label=Subcase~\thecasespi.\arabic*:,ref=\thecasespi.\arabic*}
\setlist[casesp,3]{label=Subcase~\thecasespii.\arabic*:,ref=\thecasespii.\arabic*}
\setlist[casesp,4]{label=Subcase~\thecasespiii.\arabic*:,ref=\thecasespiii.\arabic*}
\setlist[casesp,5]{label=Subcase~\thecasespiv.\arabic*:,ref=\thecasespiv.\arabic*}
\newcommand\litem[1]{\item{\bfseries #1.\enspace}}

\usepackage[tableposition=below]{caption}
\usepackage{subcaption} 
\usepackage{hyperref} 
\hypersetup{
	linkcolor=DodgerBlue3, 
	citecolor  = teal,
	urlcolor   = DodgerBlue4,
	colorlinks = true, 
	hyperfootnotes =false,
	bookmarksdepth=3,
	linktoc=all,
	hypertexnames=false
}
\makeatletter 
\renewcommand\subsection{\@startsection{subsection}{3}
	\z@{.5\linespacing\@plus.7\linespacing}{.5\linespacing}
	{\bfseries\itshape}} 
\renewcommand\paragraph{\@startsection{paragraph}{4}%
	\z@{.5\linespacing\@plus.7\linespacing}{-.5\linespacing}%
	{\normalfont\bfseries}}
\renewcommand\subparagraph{\@startsection{subparagraph}{5}%
	\z@{.3\linespacing\@plus.3\linespacing}{-.5\linespacing}%
	{\normalfont\bfseries}}
\makeatother

\usepackage{indentfirst}

\makeatletter \renewenvironment{proof}[1][\proofname]{
	\par\pushQED{\qed}\normalfont
	\topsep6\p@\@plus6\p@\relax
	\trivlist\item[\hskip\labelsep\bfseries#1\@addpunct{.}]
	\ignorespaces}{
	\popQED\endtrivlist\@endpefalse} \makeatother

\theoremstyle{plain}
\newtheorem{theorem}{Theorem}[section]
\newtheorem*{theorem*}{Theorem}
\newtheorem{theoremA}{Theorem} 
\newtheorem{propositionA}[theoremA]{Proposition} 

\theoremstyle{definition}

\newtheorem{definition}[theorem]{Definition}
\newtheorem{lemma}[theorem]{Lemma}
\newtheorem{proposition}[theorem]{Proposition}
\newtheorem{example}[theorem]{Example}
\newtheorem{notation}[theorem]{Notation}
\newtheorem{remark}[theorem]{Remark}
\newtheorem*{remark*}{Remark}
\newtheorem{claim}{Claim}
\newtheorem*{claim*}{Claim}

\let\sec\S 

\newcommand{\A}{\mathbb{A}}
\newcommand{\C}{\mathbb{C}}
\newcommand{\F}{\mathbb{F}}

\renewcommand{\P}{\mathbb{P}}
\newcommand{\Q}{\mathbb{Q}}

\newcommand{\Z}{\mathbb{Z}}

\newcommand{\cC}{\mathcal{C}}
\newcommand{\cD}{\mathcal{D}}

\newcommand{\cN}{\mathcal{N}}
\newcommand{\cO}{\mathcal{O}}
\newcommand{\cP}{\mathcal{P}}

\newcommand{\Phtl}{\cP_{\height\leq 3}}
\newcommand{\Pht}{\cP_{\height=3}}
\newcommand{\Phtres}{\tilde{\cP}_{\height=3}}

\newcommand{\cS}{\mathcal{S}}

\newcommand{\cX}{\mathcal{X}}

\newcommand{\rA}{\mathrm{A}}
\newcommand{\rD}{\mathrm{D}}
\newcommand{\rE}{\mathrm{E}}

\usepackage[scr=boondoxo]{mathalfa}

\renewcommand{\ll}{\mathscr{l}}

\newcommand{\cc}{\mathscr{c}}
\newcommand{\qq}{\mathscr{q}}

\newcommand{\hh}{\mathscr{h}}
\newcommand{\kk}{\mathscr{k}}

\newcommand{\ww}{\mathscr{w}}

\renewcommand{\epsilon}{\varepsilon}
\renewcommand{\phi}{\varphi}
\renewcommand{\theta}{\vartheta}

\DeclareFontFamily{U}{mathx}{}
\DeclareFontShape{U}{mathx}{m}{n}{<-> mathx10}{}
\DeclareSymbolFont{mathx}{U}{mathx}{m}{n}
\DeclareMathAccent{\widehat}{0}{mathx}{"70}
\DeclareMathAccent{\widecheck}{0}{mathx}{"71}

\renewcommand{\tilde}{\widetilde}
\renewcommand{\check}{\widecheck}
\renewcommand{\hat}{\widehat}
\renewcommand{\bar}{\overline}

\newcommand{\citestacks}[1]{\cite[\href{https://stacks.math.columbia.edu/tag/#1}{Tag #1}]{stacks-project}}

\renewcommand{\leq}{\leqslant}
\renewcommand{\geq}{\geqslant}

\renewcommand{\to}{\longrightarrow}
\newcommand{\map}{\dashrightarrow}

\newcommand{\sqto}{\rightsquigarrow}

\newcommand{\cha}{\operatorname{char}} 

\newcommand{\Aut}{\operatorname{Aut}}
\newcommand{\Sing}{\operatorname{Sing}}

\newcommand{\NS}{\operatorname{NS}}
\newcommand{\Exc}{\operatorname{Exc}}
\newcommand{\Supp}{\operatorname{Supp}}

\newcommand{\redd}{_{\mathrm{red}}} 
\newcommand{\Bs}{\operatorname{Bs}} 
\newcommand{\reg}{^{\mathrm{reg}}} 
\newcommand{\trp}{^{\scriptscriptstyle{\top}}} 

\newcommand{\PGL}{\mathrm{PGL}}

\newcommand{\ld}{\operatorname{ld}} 
\newcommand{\cf}{\operatorname{cf}}

\newcommand{\lts}[2]{\mathcal{T}_{#1}(-\log #2)} 

\newcommand{\hor}{_{\mathrm{hor}}} 
\renewcommand{\vert}{_{\mathrm{vert}}} 
\newcommand{\ftip}[1]{\mathrm{tip}^{+}(#1)} 
\newcommand{\ltip}[1]{\mathrm{tip}^{-}(#1)} 

\newcommand{\cp}[1]{^{(#1)}} 

\newcommand{\height}{\operatorname{ht}} 
\newcommand{\width}{\operatorname{wd}}

\newcommand{\id}{\mathrm{id}}

\newcommand{\Sec}{\Xi}
\newcommand{\bs}[1]{\boldsymbol{#1}} 
\newcommand{\ub}[1]{\uline{\bs{#1}}} 

\newcommand{\dec}[1]{^{|#1|}}
\newcommand{\ldec}[1]{\prescript{|#1|}{}} 

\newcommand{\adec}[1]{^{\langle #1 \rangle}} 
\newcommand{\aadec}[1]{^{\lbrace #1 \rbrace}} 

\newcommand{\pr}{\operatorname{pr}}

\newcommand{\Astst}{\P^1\setminus \{0,1,\infty\}}

\newcommand{\de}{\coloneqq} 

\usepackage{xspace}
\newcommand{\tables}{Tables \ref{table:ht=3_char=0}--\ref{table:ht=3_char=2}\xspace}

\makeatletter
\newbox\LT@firstfoot
\def\endfirstfoot{\LT@end@hd@ft\LT@firstfoot}
\newdimen\LT@footdiff
\def\LT@start{%
	\let\LT@start\endgraf
	\endgraf\penalty\z@
	\vskip\LTpre\endgraf
	\LT@footdiff-\ht\LT@foot
	\advance\LT@footdiff\ht\LT@firstfoot
	\dimen@\pagetotal
	\advance\dimen@ \ht\ifvoid\LT@firsthead\LT@head\else\LT@firsthead\fi
	\advance\dimen@ \dp\ifvoid\LT@firsthead\LT@head\else\LT@firsthead\fi
	\advance\dimen@ \ht\ifvoid\LT@firstfoot\LT@foot\else\LT@firstfoot\fi
	\dimen@ii\vfuzz
	\vfuzz\maxdimen
	\setbox\tw@\copy\z@
	\setbox\tw@\vsplit\tw@ to \ht\@arstrutbox
	\setbox\tw@\vbox{\unvbox\tw@}%
	\vfuzz\dimen@ii
	\advance\dimen@ \ht
	\ifdim\ht\@arstrutbox>\ht\tw@\@arstrutbox\else\tw@\fi
	\advance\dimen@\dp
	\ifdim\dp\@arstrutbox>\dp\tw@\@arstrutbox\else\tw@\fi
	\advance\dimen@ -\pagegoal
	\ifdim \dimen@>\z@\vfil\break\fi
	\global\@colroom\@colht
	\ifvoid\LT@firstfoot
	\ifvoid\LT@foot
	\else
	\advance\vsize-\ht\LT@foot
	\global\advance\@colroom-\ht\LT@foot
	\dimen@\pagegoal\advance\dimen@-\ht\LT@foot\pagegoal\dimen@
	\maxdepth\z@
	\fi
	\else
	\advance\vsize-\ht\LT@firstfoot
	\global\advance\@colroom-\ht\LT@firstfoot
	\dimen@\pagegoal\advance\dimen@-\ht\LT@firstfoot\pagegoal\dimen@
	\maxdepth\z@
	\fi
	\ifvoid\LT@firsthead\copy\LT@head\else\box\LT@firsthead\fi\nobreak
	\output{\LT@output}%
}
\def\LT@output{%
	\ifnum\outputpenalty <-\@Mi
	\ifnum\outputpenalty > -\LT@end@pen
	\LT@err{floats and marginpars not allowed in a longtable}\@ehc
	\else
	\setbox\z@\vbox{\unvbox\@cclv}%
	\ifdim \ht\LT@lastfoot>\ht\LT@foot
	\dimen@\pagegoal
	\advance\dimen@-\ht\LT@lastfoot
	\ifdim\dimen@<\ht\z@
	\setbox\@cclv\vbox{\unvbox\z@\copy\LT@foot\vss}%
	\@makecol
	\@outputpage
	\setbox\z@\vbox{\box\LT@head}%
	\fi
	\fi  
	\global\@colroom\@colht
	\global\vsize\@colht   
	\vbox
	{%
		\unvbox\z@
		\box
		\ifvoid\LT@lastfoot
		\ifvoid\LT@firstfoot
		\LT@foot
		\else
		\LT@firstfoot
		\fi
		\else
		\LT@lastfoot
		\fi
	}%
	\fi
	\else
	\ifvoid\LT@firstfoot
	\setbox\@cclv\vbox{\unvbox\@cclv\copy\LT@foot\vss}%
	\@makecol
	\@outputpage
	\global\vsize\@colroom
	\else
	\setbox\@cclv\vbox{\unvbox\@cclv\box\LT@firstfoot\vss}%
	\@makecol
	\@outputpage
	\global\advance\@colroom\LT@footdiff
	\global\vsize\@colroom
	\fi
	\copy\LT@head\nobreak
	\fi
}
\makeatother

\usepackage{tikz}
\usetikzlibrary{calligraphy,snakes,calc,intersections,arrows.meta}
\tikzset{
	partial ellipse/.style args={#1:#2:#3}{
		insert path={+ (#1:#3) arc (#1:#2:#3)}
	},
	thick/.style=      {line width=0.8pt},
	add/.style args={#1 and #2}{
		to path={%
			($(\tikztostart)!-#1!(\tikztotarget)$)--($(\tikztotarget)!-#2!(\tikztostart)$)%
			\tikztonodes},add/.default={.2 and .2}}
}
\usepackage{tikz-cd}
\usepackage{pgfplots}

\begin{document}
	\title[del Pezzo surfaces III]{Classification of del {P}ezzo surfaces of rank one \\  III. Height 3} 
	
\author{Karol Palka}
\address{Institute of Mathematics, Polish Academy of Sciences, \'{S}niadeckich 8, 00-656 Warsaw, Poland}
\email{palka@impan.pl}
	\thanks{This project was funded by the National Science Centre, Poland, grant number 2021/41/B/ST1/02062. For the purpose of Open Access, the authors have applied a CC-BY public copyright license to any Author Accepted Manuscript version arising from this submission.}

\author{Tomasz Pe{\l}ka}
\address{University of Warsaw, Faculty of Mathematics, Informatics and Mechanics, Banacha 2, 02-097 Warsaw, Poland}
\email{tpelka@mimuw.edu.pl}


\begin{abstract}
	This article is a part of a series aimed at classifying normal del Pezzo surfaces of Picard rank one over an algebraically closed field of arbitrary characteristic, up to an isomorphism. The key invariant guiding our classification is the \emph{height}, defined as the minimal number $\hh$ such that the minimal resolution of singularities admits a $\P^1$-fibration whose fiber meets the exceptional divisor $\hh$ times. It is expected that every singular del Pezzo surface of rank one is of height $\hh\leq 4$, with minor exceptions in characteristics $2$ and $3$.
	
	Having settled the case $\hh\leq 2$ in \cite{PaPe_ht_2}, we now give a classification in case the height equals $3$.
\end{abstract}
%

\maketitle
\setcounter{tocdepth}{1}
\tableofcontents
\section{Introduction}

A normal surface $\bar{X}$ is \emph{del Pezzo} if its anti-canonical divisor $-K_{\bar{X}}$ is ample. This article continues a series started in \cite{PaPe_ht_2}, whose goal is to give a self-contained classification of del Pezzo surfaces of Picard rank one, over an algebraically closed field $\kk$ of arbitrary characteristic. Such surfaces, especially the log terminal ones, are important building blocks in the birational classification of varieties. They have a rich, interesting geometry on their own: for instance, if $\kk=\C$ then their smooth loci are covered by images of $\A^1$ \cite{Keel-McKernan_rational_curves} and have finite fundamental groups \cite{GurZha_1}. Despite many partial results, log terminal del Pezzo surfaces of rank one are still not classified, for instance, no list of their singularity types is available yet. 

Our classification is guided by a new geometric invariant of $\bar{X}$, its \emph{height}. It is introduced in \cite[Definition 1.1]{PaPe_MT} as follows. For a smooth projective surface $X$ and a reduced simple normal crossings divisor $D$ on $X$ we define the \emph{height} of $(X,D)$ as 
\begin{equation}\label{eq:height}
		\height(X,D)\de \inf\{\hh : \mbox{ there is a } \P^{1}\mbox{-fibration of } X \mbox{ with fiber } F \mbox{ satisfying } F\cdot D=\hh\}.
\end{equation}	
For a normal surface $\bar{X}$ we put $\height(\bar{X})\de \height(X,D)$, where $X\to \bar{X}$ is the minimal resolution of singularities and $D$ is its reduced exceptional divisor. Its key expected property is that it is bounded. More precisely, it is expected 
that any singular del Pezzo surface $\bar{X}$ has height at most $4$, with some exceptions if $\cha\kk=2,3$. In \cite{PaPe_ht_2} we classified del Pezzo surfaces of height at most $2$. We also introduced and described a class of del Pezzo surfaces of rank one which have \emph{descendants with elliptic boundary}, that is, such that $X$ admits a morphism onto a normal surface $\bar{Y}$ mapping $D$ onto an irreducible curve of arithmetic genus $1$ contained in the smooth locus of~$\bar{Y}$, see Definition \ref{def:GK}.

In this article we settle the case when $\height(\bar{X})=3$ and $\bar{X}$ has no descendant with elliptic boundary. By \cite[Proposition B]{PaPe_ht_2} this condition implies that $\bar{X}$ is log terminal, with some exceptions if $\cha\kk=2$. 
Our classification is summarized in Theorem \ref{thm:ht=3}. 
Within the proof we establish the following results:
\begin{itemize}
	\item In Proposition \ref{prop:ht=3_models} we construct a birational morphism $\psi\colon X\to Z$, where $Z$, a minimal model of $X$, is isomorphic to $\P^1\times \P^1$ or $\P^2$, and $\psi_{*}D$ is a simple configuration of curves of low degree (drawn in the second-to-last column of Table \ref{table:intro}).
	\item In Proposition \ref{prop:ht=3_swaps} we factor $\psi$ as $\begin{tikzcd}
		[cramped,column sep=scriptsize] X \ar[r,"\phi_0"] & X_1 \ar[r,"\phi_1"] & \dots \ar[r,"\phi_{n-1}"] & X_n \ar[r] & Z
	\end{tikzcd}$, where each $\phi_j$ is a blowup, each $X_j$ is a minimal resolution of a del Pezzo surface of rank one, and $X_n$, a vertically primitive model of $X$ with a fixed $\P^1$-fibration of height $3$ with respect to $D$, is particularly simple (shown in the last column of Table \ref{table:intro}). 
\end{itemize}

We note that in case $\cha\kk\neq 2,3$, Lacini \cite{Lacini} obtained an independent rough classification of all log terminal del Pezzo surfaces. Extending the method \cite{Keel-McKernan_rational_curves} he arranged such surfaces in 24 non-disjoint series. Nonetheless, the question of uniqueness of surfaces of a given singularity type has not been addressed, and an explicit list of such types has not been given.  We compare this result with our Theorem~\ref{thm:ht=3} in Section~\ref{sec:comparison}. 

Numerous particular classes of del Pezzo surfaces are well described in the literature, for instance, those of index $\leq 3$ are classified in \cite{Alexxev-Nikulin_delPezzo-index-2,Fujita-Yasutake_delPezzo-index-3}. Other partial classification results such as \cite{Zhang_dP3,Kojima_Sing-1,Belousov_4-sings} are surveyed in \cite[\sec 7]{PaPe_ht_2}.

\subsection{Main result} 

We now state the main result of this article. Since non--log terminal del Pezzo surfaces of rank one are described in \cite[Proposition B]{PaPe_ht_2}, see Remark below, we restrict our attention to the log terminal ones.

A \emph{singularity type} of such a surface is the weighted graph of the exceptional divisor of its minimal resolution. We denote it using conventions summarized in Section \ref{sec:notation}; in particular, $[a_1,\dots,a_n]$ stands for a rational chain whose subsequent components have self-intersection numbers $-a_1,\dots, -a_n$. We also use standard notation $\rA_{k},\rD_k,\rE_k$ for the canonical singularity types. Given a singularity type $\cS$, we denote by $\Phtl(\cS)$ the set of isomorphism classes of normal del Pezzo surfaces of rank one, height at most $3$ and singularity type $\cS$. 

\begin{theoremA}[Del Pezzo surfaces of height $3$]\label{thm:ht=3}
	Let $\bar{X}$ be a log terminal del Pezzo surface of rank~$1$ and height~$3$, with no descendant with elliptic boundary. Let $\cS$ be the singularity type of $\bar{X}$. Then the following hold.
	\begin{enumerate}
		\item\label{item:thm-lt-types} The type $\cS$ is listed in Lemmas \ref{lem:w=3}, \ref{lem:w=2_cha_neq_2}, \ref{lem:w=2_cha=2}, \ref{lem:w=1_cha-neq-3} or  \ref{lem:w=1_cha=3}, see \tables. 
		\item\label{item:thm-lc-uniqueness} We either have $\#\Phtl(\cS)=1$, 
		or one of the following exceptional cases holds.
		\begin{enumerate}
			\item \label{item:uniq_exotic} 
			$\cS=[2,2,3,(2)_{5}]+[3,2]$. Then $\#\Phtl(\cS)=2$ and $\bar{X}$ is as in 
			Example \ref{ex:ht=3_pair},
			\item \label{item:uniq_cha=2} 
			$\cha\kk=2$, $\cS$ is 
			listed in Table \ref{table:ht=3_char=2_moduli}, and $\Phtl(\cS)$ has moduli dimension $1$.
		\end{enumerate}
	\end{enumerate}
\end{theoremA}
\noindent
For the definition of \emph{moduli dimension} see \cite[Definition 1.10]{PaPe_ht_2}, we recall it in Section \ref{sec:moduli}. 
\begin{remark}[{cf.\ \cite[Remark 1.14]{PaPe_ht_2}}]
In a forthcoming article we will show that for each singularity type $\cS$ as in Theorem \ref{thm:ht=3}, the set $\Phtl(\cS)$ consists of isomorphism classes of \emph{all} del Pezzo surfaces of rank one and singularity type $\cS$, i.e.\ $\Phtl(\cS)=\cP(\cS)$. 
\end{remark}

\subsection{Choosing witnessing $\P^1$-fibrations and their primitive models}\label{sec:intro-structure}

In the lemmas cited in Theorem \ref{thm:ht=3} we describe the singularity types $\cS$  together with the structure of a specific  $\P^1$-fibration $p\colon X\to \P^1$ of the minimal log resolution $(X,D)$ of $\bar{X}$ which realizes the infimum in formula \eqref{eq:height}, i.e.\ whose fiber $F$ satisfies $F\cdot D=3$. We call such $p$ a \emph{witnessing} $\P^1$-fibration. 

Blowing down $(-1)$-curves in degenerate fibers of $p$ and in their images we obtain a morphism onto some Hirzebruch surface $\F_{m}$; in case $m=1$ we contract the negative section, too. A description of this \emph{minimalization} morphism $\psi\colon X\to Z$ provides a recipe to reconstruct $(X,D)$, hence the original surface $\bar{X}\in \Phtl(\cS)$. 

As a first step towards proving Theorem \ref{thm:ht=3}, we show how to choose a witnessing $\P^1$-fibration $p\colon X\to \P^1$ and the minimalization $\psi\colon X\to Z$ in such a way that $Z$ is isomorphic to $\P^1\times \P^1$ or $\P^2$ and $\psi_{*}D$ is one of the simple configurations shown in Table \ref{table:intro} below. This is the content of Proposition \ref{prop:ht=3_models} below.

It is convenient to choose $p$ in such a way that $D$ has a maximal number of horizontal components among all witnessing $\P^1$-fibrations: we call this number the \emph{width} of $\bar{X}$ and denote by $\width(\bar{X})$, see \cite[Definition~2.2]{PaPe_ht_2}. Clearly, $1\leq \width(\bar{X})\leq \height(\bar{X})$. This way, we avoid as much as possible the situation when $D$ contains a multi-section. If $\width(\bar{X})\leq 2$, i.e.\ if $D$ necessarily contains a multi-section $H$, the classification splits into cases depending on whether the finite morphism $p|_{H}\colon H\to \P^1$ is separable or not, the latter case leading to many additional examples in characteristics $2$ and $3$, cf.\ Table \ref{table:intro}.

\begin{propositionA}[Choosing minimal models of $X$, see Table \ref{table:intro}]\label{prop:ht=3_models}
	Let $\bar{X}$ be a log terminal del Pezzo surface of rank one and height $3$, and let $(X,D)$ be its minimal log resolution. Then there is a birational morphism $\psi\colon X\to Z$ such that one of the following holds.
	\begin{enumerate}
		\item  \label{item:w=3_models} $\width(\bar{X})=3$, $Z = \P^1\times \P^1$, and $\psi_{*}D$ is a sum of three vertical and three horizontal lines, see Figure \ref{fig:w=3}.
		\item \label{item:w=2_models} $\width(\bar{X})=2$, $Z=\P^2$, and $\psi_{*}D$ is a sum of a conic $\cc$ and lines $\ll_1,\dots,\ll_{\nu}$ meeting at a common point off $\cc$, such that one of the following holds:
		\begin{enumerate}
			\item \label{item:w=2_cha-not-2} $\cha\kk\neq 2$, $\nu=3$, the lines $\ll_1,\ll_2$ are tangent to $\cc$ and $\ll_3$ is not, see Figure \ref{fig:w=2_cha_neq_2},
			\item \label{item:w=2_cha=2} $\cha\kk=2$, $\nu\in \{3,4\}$ and all lines $\ll_1,\dots,\ll_{\nu}$ are tangent to $\cc$, see Figure \ref{fig:w=2_cha=2}.
		\end{enumerate}
		\item\label{item:w=1_models} $\width(\bar{X})=1$, $Z=\P^2$, and $\psi_{*}D=\qq+\ll_1+\ll_2+\ll_{3}$, where $\qq$ is a cuspidal  cubic and $\ll_{1},\ll_2,\ll_3$ are lines meeting at a common point away from $\qq$, such that $\ll_1$ is tangent to the cusp of $\qq$, and $\ll_2,\ll_3$ are tangent to $\qq\reg$ with multiplicity $2$ if $\cha\kk\neq 3$ and with multiplicity $3$ if $\cha\kk=3$, see Figure \ref{fig:w=1}.
	\end{enumerate}
\end{propositionA}

To state the next result, we recall the definition of an \emph{elementary vertical swap} from  \cite[Definition 1.3]{PaPe_ht_2}. 
Let $X$ be a smooth projective surface and let $D$ be a reduced snc divisor on $X$. Fix a $\P^1$-fibration $p$ of $X$. Assume that there is a vertical $(-1)$-curve $L\not\subseteq D$ such that $L\cdot D\leq 2$ and $L$ meets exactly one $(-2)$-curve $C\subseteq D$. Assume that $C$ is vertical, too (this is a typical situation, cf.\  Lemma \ref{lem:fibrations-Sigma-chi}). Let $\phi\colon X\to X'$ be the contraction of $L$. Then $\phi(C)$ is a new vertical $(-1)$-curve; let $D'=\phi_{*}(D-C)$. The operation $(X,D)\sqto (X',D')$ is called an \emph{elementary vertical swap}. A \emph{vertical swap} is a composition of elementary vertical swaps with respect to the induced $\P^1$-fibrations. We say that $(X,D,p)$ is \emph{vertically primitive} if it admits no vertical swaps.

We say that $(X,D)$ \emph{swaps vertically} to $(Y,D_Y)$ if there is a vertical swap $(X,D)\sqto (Y,D_Y)$. If $(Y,D_Y)$ with the induced $\P^1$-fibration is vertically primitive, we call it a \emph{vertically primitive model} of $(X,D,p)$, and we often omit $p$ in the notation if it is clear from the context. If both $(X,D)$ and $(Y,D_Y)$ are minimal log resolutions of normal surfaces $\bar{X}$ and $\bar{Y}$, we abuse the terminology by saying that $\bar{X}$ vertically swaps to $\bar{Y}$, and call $\bar{Y}$ vertically primitive if $(Y,D_Y)$ with the induced $\P^1$-fibration is vertically primitive.

If $\bar{X}$ is a log terminal del Pezzo surface of rank 1, and $(X,D)\sqto (X_1,D_1)\sqto \dots \sqto (X_n,D_n)=(Y,D_Y)$ is a sequence of elementary vertical swaps then by Lemma \ref{lem:cascades}\ref{item:cascades-still-dP} each $(X_i,D_i)$ is the minimal log resolution of a log terminal del Pezzo surface of rank 1, too, so we get a \enquote{cascade} of del Pezzo surfaces $\bar{X}\sqto\bar{X}_1\sqto\dots\sqto \bar{Y}$. 
\smallskip

If $\bar{X}$ is a log terminal del Pezzo surface of rank one and height at most $2$ then \cite[Theorem A(1)]{PaPe_ht_2} shows that $\bar{X}$ vertically swaps to a canonical surface, with some exceptions in case $\cha\kk=2$. An analogous statement in case $\height(\bar{X})=3$ is our next intermediate result, Proposition \ref{prop:ht=3_swaps}. It implies that $\bar{X}$ vertically swaps to a surface $\bar{Y}$ which is canonical or has a descendant with elliptic boundary, with some exceptions if $\cha\kk=3$.

\begin{propositionA}[Choosing vertically primitive models]\label{prop:ht=3_swaps}
	Let $\bar{X}$ be a log terminal  del Pezzo surface of rank~$1$ and height~$3$. Then for some witnessing $\P^1$-fibration of its minimal log resolution, the surface $\bar{X}$ swaps vertically to one of the following vertically primitive surfaces $\bar{Y}$, see Table \ref{table:intro}. We write $\cS(\bar{Y})$ for the singularity type of $\bar{Y}$. 
	\begin{enumerate}
		\item\label{item:swap-to-can} The surface $\bar{Y}$ is canonical, and one of the following holds:
		\begin{enumerate}
			\item\label{item:ht=w=3} $\width(\bar{X})=3$, $\cS(\bar{Y})$ equals $\rA_{1}+\rA_{2}+\rA_{5}$ or $2\rA_{4}$ and $\bar{Y}$ is as in Example \ref{ex:w=3},
			\item\label{item:ht=3,w=2,cha-neq-2_can} $\cha\kk\neq 2$, $\width(\bar{X})=2$, $\cS(\bar{Y})$ equals $\rA_{1}+\rA_{2}+\rA_{5}$ or $2\rA_{1}+2\rA_{3}$, and $\bar{Y}$ is as  in Example \ref{ex:w=2_cha_neq_2}\ref{item:w=2_A1+A2+A5}, \ref{item:w=2_2A1+2A3}.
		\end{enumerate}
		\item\label{item:swap_to_GK} The surface $\bar{Y}$ has a descendant of singularity type $\cD$ with elliptic boundary, and one of the following holds:
		\begin{enumerate}
			\item\label{item:ht=3,w=2,cha-neq-2_GK} $\cha\kk\neq 2$, $\width(\bar{X})=2$, $\cD=\rA_1+\rA_7$, $\cS(\bar{Y})=\cD+[3]$ and $\bar{Y}$ is as  in Example  \ref{ex:w=2_cha_neq_2}\ref{item:ht=3_exception},
			\item\label{item:swap_ht=3,w=1_cha-neq-3} $\cha\kk\neq 2,3$, $\width(\bar{X})=1$, $\cD=\rA_{1}+\rA_{2}+\rA_{5}$, $\cS(\bar{Y})=\cD+[3]$ and $\bar{Y}$ is as  in Example \ref{ex:w=1}\ref{item:w=1_cha-neq-3},
			\item\label{item:swap_ht=3,w=1_cha=3_GK} $\cha\kk=3$, $\width(\bar{X})=1$, $\cD=3\rA_2$, $\cS(\bar{Y})=\cD+\rA_1+2\cdot [3]$ and $\bar{Y}$ is as  in Example \ref{ex:w=1}\ref{item:w=1_cha=3_GK},
			\item\label{item:swap_cha=2_GK} $\cha\kk=2$, $\width(\bar{X})=2$, and either $\cD=3\rA_1+\rD_4$ and  $\cS(\bar{Y})$ equals $\cD+[3,2]$ or $\cD+\rA_{1}+[3]+[4]$; or $\cD=2\rA_{1}+\rA_3$, $\cS(\bar{Y})=\cD+2\rA_{1}+[3]$. In either case, $\bar{Y}$ is as in Example \ref{ex:w=2_cha=2}.
		\end{enumerate}
		\item\label{item:swap_cha=3} $\cha\kk=3$, $\width(\bar{X})=1$, $\cS(\bar{Y})=4\cdot [3]+3\rA_{2}$ and $\bar{Y}$ is as in Example \ref{ex:w=1}\ref{item:w=1_cha=3_not-GK}.
	\end{enumerate}
\end{propositionA}

Once Proposition \ref{prop:ht=3_swaps} is proved, to get Theorem \ref{thm:ht=3} we reconstruct the cascade over each $\bar{Y}$. Note that by Proposition \ref{prop:primitive}\ref{item:primitive-uniqueness-Y} each surface $\bar{Y}$ above is uniquely determined, up to an isomorphism, by its singularity type $\cS(\bar{Y})$, except for those constructed in Examples \ref{ex:w=2_cha=2}\ref{item:swap-to-nu=4_small} and  \ref{ex:w=2_cha=2}\ref{item:swap-to-nu=4_large}, which form families of moduli dimension $1$.  To reconstruct the cascade from $\bar{Y}$ we blow up points on vertical $(-1)$-curves as long as the resulting anti-canonical divisor is ample. The latter condition turns out to be so restrictive that (in most cases) we need to stop before we can perform an outer blowup, hence each blown up point is uniquely determined by the singularity type of the surface we are reconstructing, see \cite[Observation 2.10]{PaPe_ht_2}. This is the reason behind the uniqueness part of Theorem \ref{thm:ht=3}. The same principle allows us to infer the following Proposition \ref{prop:rigidity}, which describes the cohomology of the logarithmic tangent sheaf $\lts{X}{D}$. For basic properties of $\lts{X}{D}$, in particular for its relation with infinitesimal deformations of $(X,D)$, we refer to \cite{FZ-deformations,Kawamata_deformations}.

\begin{propositionA}[Rigidity and unobstructedness, cf.\ {\cite[Proposition D]{PaPe_ht_2}}]\label{prop:rigidity}
	Let $\bar{X}$ be a log terminal del Pezzo surface of rank one and height $3$ which does not have a descendant with elliptic boundary. Let $(X,D)$ be its minimal log resolution and let $h^{i}\de h^{i}(\lts{X}{D})$. Then the following hold. 
	\begin{enumerate}
		\item\label{item:rigidity-h0} We have $h^0=0$. 
		\item\label{item:rigidity-h1} We have $h^1=0$, unless $\cha\kk=2$ and $\bar{X}$ is as in Theorem \ref{thm:ht=3}\ref{item:uniq_cha=2}, in which case $h^1=1$.
		\item\label{item:rigidity-h2} We have $h^2=0$, 
		unless $\cha\kk\in \{2,3\}$ and $\bar{X}$ vertically swaps to a vertically primitive surface $\bar{Y}$ from Examples \ref{ex:w=2_cha=2}\ref{item:swap-to-nu=4_small}, \ref{ex:w=2_cha=2}\ref{item:swap-to-nu=3} or Examples \ref{ex:w=2_cha=2}\ref{item:swap-to-nu=4_large},  \ref{ex:w=1}\ref{item:w=1_cha=3_not-GK}, in which case $h^2$ equals $1$ or $2$, respectively. 
	\end{enumerate}
\end{propositionA}

We remark that in case $\height(\bar{X})\leq 2$ the ampleness of $-K_{\bar{X}}$ follows simply from the log terminality, see \cite[Lemma 2.6]{PaPe_ht_2} or Lemma \ref{lem:delPezzo_criterion}, so the structure of del Pezzo surfaces of rank one and height $\leq 2$ is much less rigid. Indeed, even if $\cha\kk=0$ there exist positive-dimensional families of such surfaces which are pairwise non-isomorphic but have the same singularity type, see \cite[Corollary 1.6(a)]{PaPe_ht_2}.

\begin{remark}[Non--log terminal del Pezzo surfaces of rank $1$ and height $3$]
	Assume that $\bar{X}$ is a non--log terminal del Pezzo surface of rank one and height~$3$. Then by \cite[Proposition B]{PaPe_ht_2}, we have $\cha\kk\in \{2,3,5\}$. Assume furthermore that $\bar{X}$ has no descendant with elliptic boundary. Then $\cha\kk=2$ and $\bar{X}$ swaps vertically to the surface from Example \ref{ex:w=2_cha=2}\ref{item:swap-to-nu=3}, hence satisfies the conclusion of Propositions \ref{prop:ht=3_models}\ref{item:w=2_cha=2} and \ref{prop:ht=3_swaps}\ref{item:swap_cha=2_GK}. Moreover, if $\bar{X}$ is log canonical then its isomorphism class is uniquely determined by its singularity type, 
	listed in  Table~\ref{table:ht=3_non-lt}.
\end{remark}

\subsection{Outline of the proof}

In Section \ref{sec:basic} we construct vertically primitive models of $\bar{X}$ from Proposition \ref{prop:ht=3_swaps}. 
In Sections \ref{sec:w=3}, \ref{sec:w=2} and \ref{sec:w=1} we prove Theorem \ref{thm:ht=3} in cases $\width(\bar{X})=3,2,1$ respectively. In each case the proof is organized as follows.
\begin{steps}
	\item\label{step:choose} We choose a witnessing $\P^1$-fibration $p\colon X\to \P^1$ in such a way that contracting vertical $(-1)$-curves on $X$ and its images we get a minimalization morphism $\psi\colon X\to \P^2$ or $X\to \P^1\times \P^1$ as in Proposition \ref{prop:ht=3_models}.
	\item\label{step:adjust} We adjust the chosen witnessing $\P^1$-fibration $p$ so that the minimalization $\psi$ factors through a minimal resolution of a vertically primitive surface $\bar{Y}$ as in Proposition \ref{prop:ht=3_swaps}.
	\item\label{step:reverse} We reconstruct $(X,D)$ by reversing vertical swaps. 
	Here we need to control all possible \enquote{cascades} over $\bar{Y}$, ensuring that $(X,D)$ is the minimal log resolution of a log terminal surface $\bar{X}$ \emph{with ample $-K_{\bar{X}}$}.
\end{steps} 	

The proofs of Propositions \ref{prop:ht=3_models}, \ref{prop:ht=3_swaps} are summarized \hyperref[proofs]{at the end of Section \ref{sec:w=1_swaps}}; those of Theorem \ref{thm:ht=3} and of Proposition \ref{prop:rigidity} are summarized  \hyperref[final]{at the end of Section \ref{sec:w=1_cha=3}.} 
Eventually, in Section \ref{sec:comparison} we compare our classification with the one given by Lacini \cite{Lacini} in case $\cha\kk\not\in \{2,3\}$, see Tables \ref{table:Lacini} and \ref{table:ht=3_char=0}.

\begin{table}[htbp]
{\renewcommand{\arraystretch}{1.2}	
\begin{tabular}{r|c|p{2.1cm}|c|c}
	\multirow{2}{*}{$\width(\bar{X})$} & \multirow{2}{*}{$\cha\kk$} &  del Pezzo $\bar{X}$ & minimal model & vertically primitive model \\ 
	&& (Theorem \ref{thm:ht=3}) & (Proposition \ref{prop:ht=3_models}) & (Proposition \ref{prop:ht=3_swaps})  \\
	\hline\hline 
	3 & any &  Lemma~\ref{lem:w=3},\newline Table~\ref{table:ht=3_char=0}\newline (top part)  &  
	\raisebox{-.6\height}{	
		\begin{tikzpicture}[scale=0.6]
			\draw (0,2.5) -- (3,2.5);
			\draw (0,1.5) -- (3,1.5);
			\draw (0,0.5) -- (3,0.5);
			\draw (0.5,3) -- (0.5,0);
			\draw (1.5,3) -- (1.5,0);
			\draw (2.5,3) -- (2.5,0);
			\node at (1.5,-0.5) {\small{$\P^1\times \P^1$}};
		\end{tikzpicture}
	}
	&
	\raisebox{-.6\height}{	
		\begin{tikzpicture}[scale=0.6]
		\path[use as bounding box] (-0.5,-1.4) rectangle (4,3.5);	
			\draw (-0.2,3) -- (2.6,3);
			\draw (0.2,3.2) -- (0,2);
			\draw[dashed] (0,2.2) -- (0.2,1);
			\draw (0.2,1.2) -- (0,0);
			\draw[dashed] (1.2,3.2) -- (1,2);
			\draw (1,2.2) -- (1.2,1);
			\draw[dashed] (1.2,1.2) -- (1,0);
			\draw[dashed] (2.4,3.2) -- (2.2,2.2);
			\draw (2.2,2.4) -- (2.4,1.4);
			\draw (2.4,1.6) -- (2.2,0.6);
			\draw[dashed] (2.2,0.8) -- (2.4,-0.2);
			\draw (-0.2,0.7) -- (0.2,0.7) to[out=0,in=180] (1,2.4) -- (1.2,2.4) to[out=0,in=180] (2.4,0.2) -- (2.6,0.2);
			\draw (-0.2,0.2) -- (1.2,0.2) to[out=0,in=-120] (1.75,0.75);   
			\draw (1.95,1.15) to[out=60,in=180] (2.4,1.8) -- (2.6,1.8);
			\node at (1.2,-0.4) {\small{$\rA_1+\rA_2+\rA_5$}};
			\node at (1.2,-1) {\small{Ex.\ \ref{ex:w=3}\ref{item:ht=3_A1+A2+A5}}};
	\end{tikzpicture} 
	\begin{tikzpicture}[scale=0.6]
	\path[use as bounding box] (-0.5,-1.4) rectangle (3.2,3.5);
	\draw (-0.2,3) -- (2.6,3);
	\draw (0.2,3.2) -- (0,2);
	\draw[dashed] (0,2.2) -- (0.2,1);
	\draw (0.2,1.2) -- (0,0);
	\draw[dashed] (1.2,3.2) -- (1,2);
	\draw (1,2.2) -- (1.2,1);
	\draw[dashed] (1.2,1.2) -- (1,0);
	\draw (2.2,3.2) -- (2,2.2);
	\draw (2,2.4) -- (2.2,1.4);
	\draw[dashed] (2.2,1.6) -- (2,0.6);
	\draw[dashed] (2,2.8) to[out=0,in=80] (2.4,0);
	\draw (-0.2,0.7) -- (0.2,0.7) to[out=0,in=180] (1,1.4) -- (1.2,1.4) to[out=0,in=180] (2,0.8)-- (2.2,0.8);
	\draw (-0.2,0.2) -- (2.6,0.2);
	\node at (1.2,-0.4) {\small{$2\rA_4$}};
	\node at (1.2,-1) {\small{Ex.\ \ref{ex:w=3}\ref{item:ht=3_2A4}}};
\end{tikzpicture} }
	\\ \hline 2 & $\neq 2$  & Lemma~\ref{lem:w=2_cha_neq_2},\newline Table \ref{table:ht=3_char=0}\newline (middle part) & 
	\raisebox{-.6\height}{	
		\begin{tikzpicture}[scale=0.6]
			\draw (0,0) circle (1);
			\draw[add= 0.1 and 1] (0,2) to (-0.866,0.5);
			\draw[add= 0.1 and 1] (0,2) to (0.866,0.5);
			\draw (0,2.2) -- (0,-1.2);
			\node at (0,-1.5) {\small{$\P^2$}};
		\end{tikzpicture}
	}
	&
	\raisebox{-.6\height}{	
		\begin{tikzpicture}[scale=0.6]
			\path[use as bounding box] (-0.7,-1.4) rectangle (4,3.5);
			\draw (0,3) -- (3.2,3);
			\draw[dashed] (0.2,3.2) -- (0,2.2);
			\draw (0,2.4) -- (0.2,1.4);
			\node at (-0.4,1.9) {\small{$-3$}};
			\draw[dashed] (0.2,1.6) -- (0,0.6); 
			\draw (0,0.8) -- (0.2,-0.2);
			\draw (1.4,3.2) -- (1.2,2);
			\draw[dashed] (1.2,2.2) -- (1.4,1);
			\draw (1.4,1.2) -- (1.2,0); 
			\draw (3,3.2) -- (2.8,2);
			\draw (2.8,2.2) -- (3,1);
			\draw[dashed] (3,1.5) -- (1.8,1.7);
			\draw (3,1.2) -- (2.8,0); 
			\draw (-0.2,1.15) to[out=0,in=180] (0.1,1.1) to[out=0,in=180] (1.2,2.8) -- (1.4,2.8) to[out=0,in=170] (2.4,1.6) to[out=-10,in=180] (2.6,1.65);
			\draw (-0.2,1.05) to[out=0,in=180] (0.1,1.1) to[out=0,in=180] (1.2,0.4) -- (1.4,0.4) to[out=0,in=170] (2.4,1.6) to[out=-10,in=180] (2.6,1.45);
			\node at (1.4,-0.4) {\small{$\rA_1+\rA_7+[3]$}};
			\node at (1.4,-1) {\small{Ex.\ \ref{ex:w=2_cha_neq_2}\ref{item:ht=3_exception}}};
		\end{tikzpicture}
	\begin{tikzpicture}[scale=0.6]
		\path[use as bounding box] (-0.6,-1.4) rectangle (4,3.5);
		\draw (-0.4,3) -- (2.8,3);
		\draw (-0.2,3.2) -- (-0.4,2);
		\draw (-0.4,2.2) -- (-0.2,1);			
		\draw[dashed] (-0.4,1.5) -- (0.8,1.7);
		\draw (-0.2,1.2) -- (-0.4,0);
		\draw[dashed] (1.4,3.2) -- (1.2,2);
		\draw (1.2,2.2) -- (1.4,1);
		
		\draw[dashed] (1.4,1.2) -- (1.2,0); 
		\draw (2.6,3.2) -- (2.4,2);
		\draw[dashed] (2.4,2.2) -- (2.6,1);
		\draw (2.6,1.2) -- (2.4,0);
		\draw (0.2,1.7) to[out=-10,in=-170] (0.5,1.65) to[out=10,in=180] (1,1.8) -- (1.8,1.8)  to[out=0,in=180] (2.5,1.55) to[out=0,in=180] (2.8,1.6);
		\draw (0.2,1.5) to[out=20,in=-170] (0.5,1.65) to[out=10,in=180] (1.2,0.4) -- (1.4,0.4) to[out=0,in=190] (2.5,1.55) to[out=0,in=180] (2.8,1.5);
		\node at (1.2,-0.4) {\small{$\rA_1+\rA_2+\rA_5$}};
		\node at (1.2,-1) {\small{Ex.\ \ref{ex:w=2_cha_neq_2}\ref{item:w=2_A1+A2+A5}}};
	\end{tikzpicture} 
	\begin{tikzpicture}[scale=0.6]
		\path[use as bounding box] (-0.6,-1.4) rectangle (3,3.5);
		\draw (-0.4,3) -- (2.8,3);
		\draw (-0.2,3.2) -- (-0.4,2);
		\draw (-0.4,2.2) -- (-0.2,1);			
		\draw[dashed] (-0.4,1.5) -- (0.8,1.7);
		\draw (-0.2,1.2) -- (-0.4,0);
		\draw[dashed] (1.4,3.2) -- (1.2,2);
		\draw (1.2,2.2) -- (1.4,1);
		\draw[dashed] (1.4,1.2) -- (1.2,0); 
		\draw (2.6,3.2) -- (2.4,2);
		\draw[dashed] (2.4,2.2) -- (2.6,1);
		\draw (2.6,1.2) -- (2.4,0);
		\draw (0.2,1.7) to[out=-10,in=-170] (0.5,1.65) to[out=10,in=180] (1,1.8) -- (1.8,1.8)  to[out=0,in=180] (2.5,1.55) to[out=0,in=180] (2.8,1.6);
		\draw (0.2,1.5) to[out=20,in=-170] (0.5,1.65) to[out=10,in=180] (1.2,0.4) -- (1.4,0.4) to[out=0,in=190] (2.5,1.55) to[out=0,in=180] (2.8,1.5);
		\node at (1.1,-0.4) {\small{$2\rA_1+2\rA_3$}};
		\node at (1.1,-1) {\small{Ex.\ \ref{ex:w=2_cha_neq_2}\ref{item:w=2_2A1+2A3}}};
	\end{tikzpicture}
	}
   \\  \cline{2-5}  & $=2$ & Lemma~\ref{lem:w=2_cha=2},\newline Tables \ref{table:ht=3_char=2_moduli}, \ref{table:ht=3_char=2}   & 
	\raisebox{-.6\height}{	
		\begin{tikzpicture}[scale=0.6]
			\draw (0,0) circle (1);
			\draw[add= 0.1 and 0.8] (0,2) to (-0.866,0.5);
			\draw[add=0.1 and -0.65] (0,2) to (-0.5,-0.6);
			\draw[add=-0.45 and 0] (0,2) to (-0.5,-0.6);
			\draw (-0.5,-0.6) to[out=-107,in=156] (-0.4,-0.92) to[out=-24,in=-107] (-0.15,-0.6);
			\draw[add=0.1 and -0.65] (0,2) to (0.5,-0.6);
			\draw[add=-0.45 and 0] (0,2) to (0.5,-0.6);
			\draw (0.5,-0.6) to[out=-73,in=24] (0.4,-0.92) to[out=196,in=-73] (0.15,-0.6);
			\draw[add= 0.1 and 0.8] (0,2) to (0.866,0.5);
			\node at (0,-1.4) {\small{$\P^2$}};
		\end{tikzpicture}
	}
	&
	\raisebox{-.6\height}{	
		\begin{tikzpicture}[scale=0.6]
			\path[use as bounding box] (-0.3,-1.4) rectangle (5,3.5);
			\draw (0,3) -- (4,3);
			\draw[dashed] (0.2,3.2) -- (0,2);
			\draw (0,2.2) -- (0.2,1);			
			\draw[dashed] (0.2,1.2) -- (0,0);
			\draw (1.4,3.2) -- (1.2,2);
			\draw[dashed] (1.2,2.2) -- (1.4,1);
			\draw (1.4,1.2) -- (1.2,0); 
			\draw (2.6,3.2) -- (2.4,2);
			\draw[dashed] (2.4,2.2) -- (2.6,1);
			\draw (2.6,1.2) -- (2.4,0); 
			\draw (3.8,3.2) -- (3.6,2);
			\draw[dashed] (3.6,2.2) -- (3.8,1);
			\draw (3.8,1.2) -- (3.6,0); 
			\draw[thick] (0,1.1) -- (0.4,1.1) to[out=0,in=180] (1,1.6) -- (3.8,1.6);
			\node at (3.1,1.9) {\small{$-3$}};
			\node at (1.8,-0.4) {\small{$3\rA_1+\rD_4+[3,2]$}};
			\node at (1.8,-1) {\small{Ex.\ \ref{ex:w=2_cha=2}\ref{item:swap-to-nu=4_small}}};
		\end{tikzpicture}
	\begin{tikzpicture}[scale=0.6]
		\path[use as bounding box] (-0.7,-1.4) rectangle (4.2,3.5);
		\draw (0,3) -- (4,3);
		\draw[dashed] (0.2,3.2) -- (0,2.2);
		\draw (0,2.4) -- (0.2,1.4);
		\node at (-0.35,2) {\small{$-3$}};
		\draw[dashed] (0.2,1.6) -- (0,0.6); 
		\draw (0,0.8) -- (0.2,-0.2);
		\draw (1.4,3.2) -- (1.2,2);
		\draw[dashed] (1.2,2.2) -- (1.4,1);				
		\draw (1.4,1.2) -- (1.2,0); 
		\draw (2.6,3.2) -- (2.4,2);
		\draw[dashed] (2.4,2.2) -- (2.6,1);				
		\draw (2.6,1.2) -- (2.4,0); 
		\draw (3.8,3.2) -- (3.6,2);
		\draw[dashed] (3.6,2.2) -- (3.8,1);	
		\draw (3.8,1.2) -- (3.6,0); 
		\draw[thick] (0,1.1) -- (0.4,1.1) to[out=0,in=180] (1,1.6) -- (3.8,1.6);
		\node at (3.1,1.9) {\small{$-4$}};
		\node at (1.8,-0.4) {\small{$4\rA_1+\rD_4+[4]+[3]$}};
		\node at (1.8,-1) {\small{Ex.\ \ref{ex:w=2_cha=2}\ref{item:swap-to-nu=4_large}}};
	\end{tikzpicture}
	}
	\\ \cline{4-5}  &&&
	\raisebox{-.6\height}{	
		\begin{tikzpicture}[scale=0.6]
			\draw (0,0) circle (1);
			\draw[add= 0.1 and 0.8] (0,2) to (-0.866,0.5);
			\draw[add=0.1 and -0.65] (0,2) to (-0.5,-0.6);
			\draw[add=-0.45 and 0] (0,2) to (-0.5,-0.6);
			\draw (-0.5,-0.6) to[out=-107,in=156] (-0.4,-0.92) to[out=-24,in=-107] (-0.15,-0.6);
			\draw[add= 0.1 and 0.8] (0,2) to (0.866,0.5);
			\node at (0,-1.4) {\small{$\P^2$}};
		\end{tikzpicture}
	}
	&
	\raisebox{-.6\height}{	
		\begin{tikzpicture}[scale=0.6]
			\path[use as bounding box] (-0.7,-1) rectangle (3.2,3.5);
			\draw (0,3) -- (2.8,3);
			\draw[dashed] (0.2,3.2) -- (0,2.2);
			\draw (0,2.4) -- (0.2,1.4);
			\node at (-0.35,1.95) {\small{$-3$}};
			\draw[dashed] (0.2,1.6) -- (0,0.6); 
			\draw (0,0.8) -- (0.2,-0.2);
			\draw (1.4,3.2) -- (1.2,2);
			\draw[dashed] (1.2,2.2) -- (1.4,1);			
			\draw (1.4,1.2) -- (1.2,0); 
			\draw (2.6,3.2) -- (2.4,2);
			\draw[dashed] (2.4,2.2) -- (2.6,1);
			\draw (2.6,1.2) -- (2.4,0); 
			\draw[thick] (0,1.1) -- (0.4,1.1) to[out=0,in=180] (1,1.6) -- (2.6,1.6);
			\node at (1.3,-0.5) {\small{$4\rA_1+\rA_3+[3]$, Ex.\ \ref{ex:w=2_cha=2}\ref{item:swap-to-nu=3}}};
		\end{tikzpicture}
	}
	  \\ \hline 1 & $\neq 2,3$ & Lemma~\ref{lem:w=1_cha-neq-3},\newline Table \ref{table:ht=3_char=0}\newline (bottom row)  & 
	\raisebox{-.6\height}{	
		\begin{tikzpicture}[scale=0.6]
			\draw[add=0.1 and 0.8] (0,2) to (0,0.5);
			\draw[add= 0.1 and 0.6] (0,2) to (-0.866,0.5);
			\draw[add= 0.1 and 0.6] (0,2) to (0.866,0.5);
			\draw (-0.866,2) to[out=-60,in=60] (-0.3,1) to[out=-120,in=60] (-0.866,0.5) to[out=-120,in=90] (0,-0.5) to[out=90,in=-60] (0.866,0.5) to[out=120,in=-60] (0.3,1) to[out=120,in=-120]  (0.866,2);
			\node at (0,-1) {\small{$\P^2$}};
		\end{tikzpicture}
	}
	&\raisebox{-.6\height}{	
		\begin{tikzpicture}[scale=0.6]
			\path[use as bounding box] (-0.4,-1) rectangle (4.3,3.5);
			\draw (0.2,3.2) -- (0,2);
			\draw (0,2.2) -- (0.2,1);
			\draw[dashed] (0,1.6) -- (1.2,1.8);
			\draw (0.2,1.2) -- (0,0);
			\draw (2.4,3.2) -- (2.2,2.2);
			\draw (2.2,2.4) -- (2.4,1.4);
			\draw[dashed] (2.4,1.6) -- (2.2,0.6);
			\draw (2.2,0.8) -- (2.4,-0.2); 
			\node at (1.9,0.2) {\small{$-3$}};
			\draw (4,3.2) -- (3.8,2);
			\draw[dashed] (3.8,2.2) -- (4,1);
			\draw (4,1.2) -- (3.8,0); 
			\draw (0.3,1.75) to[out=0,in=-170] (0.6,1.7) to[out=10,in=180] (1,2) -- (1.2,2) to[out=0,in=180] (2.3,1.1)  to[out=0,in=180] (3.2,1.7) -- (3.5,1.7) to[out=0,in=180] (3.9,1.6) to[out=0,in=180] (4.1,1.65);
			\draw (0.3,1.55) to[out=0,in=-160] (0.6,1.7) to[out=20,in=180] (1.4,0.9) -- (2,0.9) to[out=0,in=180] (2.3,1.1) to[out=0,in=180] (2.8,0.9) -- (3,0.9) to[out=0,in=180] (3.9,1.6) to[out=0,in=180] (4.1,1.55);
			\draw (0,3) -- (1,3) to[out=0,in=90] (2.1,1.6) to[out=-90,in=180] (2.3,1.1) to[out=0,in=-90] (2.6,1.6) to[out=90,in=180] (3.9,3) -- (4.1,3);
			\node at (2,-0.6) {\small{$\rA_1+\rA_2+\rA_5+[3]$, Ex.\ \ref{ex:w=1}\ref{item:w=1_cha-neq-3}}};
		\end{tikzpicture}
	}
	\\ \cline{2-5} & $=3$ & Lemma \ref{lem:w=1_cha=3},\newline Table \ref{table:ht=3_char=3} &
\raisebox{-.6\height}{	
	\begin{tikzpicture}[scale=0.6]
		\draw[add=0.1 and 0.8] (0,2) to (0,0.5);
		\draw[add= 0.1 and 0.6] (0,2) to (-0.866,0.5);
		\draw[add= 0.1 and 0.6] (0,2) to (0.866,0.5);
		\draw (-0.866,2) to[out=-60,in=60] (-0.866,0.5) to[out=-120,in=90] (0,-0.5) to[out=90,in=-60] (0.866,0.5) to[out=120,in=-120]  (0.866,2);
		\node at (0,-1) {\small{$\P^2$}};
	\end{tikzpicture}
}
&
\raisebox{-.6\height}{	
	\begin{tikzpicture}[scale=0.6]
		\path[use as bounding box] (-0.9,-1.4) rectangle (4.5,3.5);	
		\draw (0.2,3.2) -- (0,2.2);
		\draw[dashed] (0,2.4) -- (0.2,1.4);
		\draw (0.2,1.6) -- (0,0.6);
		\draw (0,0.8) -- (0.2,-0.2); 
		\node at (-0.3,2.7) {\small{$-3$}};
		\draw (1.4,3.2) -- (1.2,2.2);
		\draw[dashed] (1.2,2.4) -- (1.4,1.4);
		\draw (1.4,1.6) -- (1.2,0.6);
		\draw (1.2,0.8) -- (1.4,-0.2); 
		\node at (0.9,2.7) {\small{$-3$}};
		\draw (2.6,3.2) -- (2.3,1.8);
		\draw[dashed] (2.3,2) -- (2.6,0.8);
		\draw (2.6,1) -- (2.3,-0.2); 
		\draw[thick] (-0.2,1.9) -- (2.8,1.9);
		\node at (1.3,-0.6) {\small{$3\rA_2+\rA_1+2\cdot [3]$}};
		\node at (1.3,-1.2) {\small{Ex.\ \ref{ex:w=1}\ref{item:w=1_cha=3_GK}}};
	\end{tikzpicture}
	\begin{tikzpicture}[scale=0.6]
		\path[use as bounding box] (-0.7,-1.4) rectangle (3,3.5);
		\draw (0.2,3.2) -- (0,2.2);
		\draw[dashed] (0,2.4) -- (0.2,1.4);
		\draw (0.2,1.6) -- (0,0.6);
		\draw (0,0.8) -- (0.2,-0.2); 
		\node at (-0.3,2.7) {\small{$-3$}};
		\draw (1.4,3.2) -- (1.2,2.2);
		\draw[dashed] (1.2,2.4) -- (1.4,1.4);
		\draw (1.4,1.6) -- (1.2,0.6);
		\draw (1.2,0.8) -- (1.4,-0.2); 
		\node at (0.9,2.7) {\small{$-3$}};
		\draw (2.6,3.2) -- (2.4,2.2);
		\draw[dashed] (2.4,2.4) -- (2.6,1.4);
		\draw (2.6,1.6) -- (2.4,0.6);
		\draw (2.4,0.8) -- (2.6,-0.2); 
		\node at (2.1,2.7) {\small{$-3$}};
		\draw[thick] (-0.2,1.9) -- (3.6,1.9);
		\node at (3.2,2.1) {\small{$-3$}};
		\node at (1.5,-0.6) {\small{$4\cdot [3]+3\rA_2$}};
		\node at (1.5,-1.2) {\small{Ex.\ \ref{ex:w=1}\ref{item:w=1_cha=3_not-GK}}};
	\end{tikzpicture}
}
\end{tabular}
}
	\caption{The structure of the chosen $\P^1$-fibrations.}
	\label{table:intro}
\end{table}

\clearpage
\section{Preliminaries}

We work over an algebraically closed field $\kk$ of arbitrary characteristic. In this section, we briefly recall the notation used in the previous part \cite{PaPe_ht_2}, which mostly follows \cite{Fujita-noncomplete_surfaces} and \cite{Palka_almost_MMP}.

\subsection{Divisors on surfaces}\label{sec:log_surfaces}

Curves and surfaces are understood to be irreducible and reduced. Given a birational morphism $\psi\colon X\to Y$ between normal surfaces we write $\Exc \psi$ for its reduced exceptional divisor, and denote  by $\Bs\psi^{-1}$ the base locus of the inverse map $\psi^{-1}\colon Y\map X$. We often abuse the notation and write infinitely near points as elements of $\Bs\psi^{-1}$: for instance, if $\psi=\psi_{1}\circ\psi_{2}$, where $\psi_{1}$ is a blowup at $y\in Y$ and $\psi_{2}$ is a blowup at $y'\in \Exc\psi_{1}$ then we write $y,y'\in \Bs\psi^{-1}$. 
We denote by $\rho(\psi)$ the relative Picard rank of $\psi$, i.e.\ $\rho(\psi)=\rho(X)-\rho(Y)$ is the number of irreducible components of $\Exc\psi$.

A \emph{log surface} $(X,D)$ consists of a proper normal surface $X$ and a \emph{boundary} Weil divisor $D$ with coefficients in the interval $[0,1]$, such that $K_{X}+D$ is $\Q$-Cartier. In this article, we consider only log surfaces with reduced boundary. 
A birational morphism of log surfaces $\psi\colon (X,D)\to (\bar{X},\bar{D})$ is a birational morphism $\psi\colon X\to \bar{X}$ such that $\psi_{*}D=\bar{D}$. It is a \emph{log resolution of $(\bar{X},\bar{D})$} if $X$ is smooth, $D=\psi^{-1}_{*}\bar{D}+\Exc\psi$, and $D$ is a simple normal crossing (snc) divisor. 
	A log resolution is \emph{minimal} if it does not dominate birationally any other log resolution. A (minimal) log resolution of a normal surface $\bar{X}$ is a (minimal) log resolution of $(\bar{X},0)$.
\smallskip

Let $D$ be an effective divisor on a normal surface $X$. We say that $D$ is \emph{connected} if $\Supp D$ is connected (in the Zariski topology). We write $D\redd$ for $D$ with reduced structure. By a \emph{component} of $D$ we mean an irreducible component of $D\redd$; we denote their number by $\#D$. A \emph{subdivisor} of $D$ is an effective divisor $T$ such that $D-T$ is effective, too. For two effective divisors $D_1,D_2$ we write $D_1\cap D_2$ for the intersection of their supports. If they are reduced, we write $D_1\wedge D_2$ for the sum of their common components. 

Let $X$ be a smooth projective surface. A \emph{curve} $C$ on $X$ is called an \emph{$n$-curve} if $C\cong \P^1$ and $C^2=n$. Let $D$ be a reduced snc divisor on $X$, and let $C$ be a component of $D$. The \emph{branching number} of $C$ in $D$ is $\beta_{D}(C)\de C\cdot (D-C)$, i.e.\ the number of common points of $C$ with the remaining part of $D$. We say that $C$ is a \emph{tip of $D$} if $\beta_{D}(C)\leq 1$, and $C$ is \emph{branching in} $D$ if $\beta_{D}(C)\geq 3$.

We say that $D$ is \emph{negative definite} if so is its intersection matrix $[T_{i}\cdot T_{j}]_{1\leq i,j\leq \#D}$, where $T_1,\dots, T_{\#D}$ are components of $D$. A \emph{discriminant} of $D$ is defined as $d(D)\de \det[-T_{i}\cdot T_{j}]_{1\leq i,j\leq \#D}$ if $D\neq 0$; and as $1$ if $D=0$. 
\begin{lemma}[{Recursive formula for discriminants, see \cite[Section 3]{Fujita-noncomplete_surfaces}}]\label{lem:discriminants}
	Let $D_{1}$ and $D_{2}$ be reduced snc divisors on $X$ with no common component. Assume that $D_{1}\cdot D_{2}=1$, and for $j\in \{1,2\}$ let $C_{j}$ be the unique component of $D_{j}$ which meets $D_{3-j}$. Then
		\begin{equation*}
			d(D_{1} + D_{2}) = d(D_{1})d(D_{2})-d(D_{1}-C_{1})d(D_{2}-C_{2}).
		\end{equation*}
\end{lemma} 

Let $T$ be an snc divisor on a smooth projective surface $X$. We say that $T$ is \emph{rational} all its components are. If $T$ is connected and has no branching component then $T$ is called a \emph{chain} if it has a tip and \emph{circular} if it does not. A connected snc divisor is a \emph{tree} if it has no circular subdivisor. A \emph{fork} is a tree with exactly one branching component, which additionally has branching number $3$.

Let $T$ be a chain. It is \emph{ordered} if it has a distinguished \emph{first} tip, denoted by $\ftip{T}$. We write $T\cp{i}$ for the $i$-th component of $T$ in the natural order, given by $T\cp{1}=\ftip{T}$, $T\cp{i}\cdot T\cp{i+1}=1$, $1\leq i\leq \#T-1$. Then $\ltip{T}\de T\cp{\#T}$ is the \emph{last} tip of $T$. We write $T\trp$ for a chain $T$ with the  opposite order. 

An ordered rational subchain $T$ of a reduced divisor $D$ is a \emph{twig} of $D$ if $\ftip{T}$ is a tip of $D$, and no component of $T$ is branching in $D$. In this case, either $T=D$ or $\ltip{T}$ meets $D-T$. We say that a twig of $D$ is \emph{maximal} if it is not properly contained in any other twig of $D$. A \emph{$(-2)$-chain} (\emph{$(-2)$-fork}, \emph{$(-2)$-twig}) is a chain (fork, twig) whose all components are $(-2)$-curves. 

\subsection{Types of chains and forks}\label{sec:types}

A \emph{type} of an ordered rational chain $T$ is a sequence of integers $[a_1,\dots,a_{\#T}]$, where $a_{i}=-(T\cp{i})^{2}$.  We write $(a)_{m}$ for an integer $a$ repeated $m$ times. We say that $T$ is \emph{admissible} if $a_1,\dots,a_{\#T}\geq 2$. 

For types $T_1=[a_1,\dots, a_k]$, $T_2=[b_1,\dots, b_l]$ we write $[T_1,T_2]=[T_1,b_1,\dots,b_l]=[a_1,\dots,a_k,b_1,\dots,b_l]$ etc.  We  also use the following conventions:
\begin{equation}\label{eq:conventions_Tono}
	\begin{split} 
	[a_{1},\dots,a_{k}]*[b_{1},\dots,b_{l}]& \de [a_{1},\dots, a_{k-1},a_{k}+b_{1}-1,b_{2},\dots, b_{l}],\\
	[(2)_{-1}]*[b_{1},\dots, b_{l}]& \de [b_{1}+1,b_2\dots, b_{l}],\quad [(2)_{-1},b_{1},\dots,b_{l}]\de [b_{2},\dots, b_{l}].
	\end{split}
\end{equation}

If $T$ is an admissible chain or $T=[1]$, there is a unique type $T^{*}$  such that a chain of type $[T,1,T^{*}]$ can be blown down to a $0$-curve, see \cite[Proposition 4.7]{Fujita-noncomplete_surfaces} or \cite{Russell_formal-aspects}. A chain $[T,1,T']$ blows down to a smooth point if and only if $T'=T^{*}*[(2)_{k}]$ for some $k\geq -1$. Moreover, if a curve $C$ meets a chain $R=[T,1,T^{*}]*[(2)_{k}]$ once, normally, in $\ftip{R}$, then blowing down $R$ to a smooth point increases the self-intersection number of $C$ by $k+2$. We refer to \cite{Russell_formal-aspects} for more details on the combinatorics of such blowdowns.
\smallskip

Let $T$ be a fork. Then $T=B+T_1+T_2+T_3$, where $\beta_{T}(B)=3$, and $T_{j}$ are maximal twigs of $T$. Assume that $T$ is rational, so $B=[b]$ for some integer $b$. Then we say that $T$ is of \emph{type} $\langle b;T_1,T_2,T_3 \rangle$. We say that $T$ is \emph{admissible} if $b\geq 2$, each twig $T_j$ is admissible, and $\sum_{j=1}^{3}\frac{1}{d(T_j)}>1$. Solving this inequality one gets the following known result, cf.\ \cite[I.5.3.4]{Miyan-OpenSurf}. 

\begin{lemma}[{Admissible forks}
	] \label{lem:admissible_forks}
	Let $F=\langle b;T_1,T_2,T_3\rangle$ be an admissible fork. Then the triple $\{d(T_1),d(T_2),d(T_3)\}$ is one of the following: $\{2,2,k\}$ for some $k\geq 2$; $\{2,3,3\}$, $\{2,3,4\}$, $\{2,3,5\}$. In particular, the following hold.
	\begin{enumerate}
		\item\label{item:has_-2} At least one of the twigs of $F$ is of type $[2]$.
		\item\label{item:chains-with-low-d} For each $i$ we have either $T_{i}\in \{[d],[(2)_{d-1}]\}$, where $d=d(T_i)\in \{2,3,4,5\}$, or $T_i\in \{[2,3],[3,2]\}$, $d(T_i)=5$. 
		\item\label{item:long-twig} If $\#T_{i}\geq 3$ then either $F=\langle b;[2],[2],T_i\rangle$ or $T_i=[(2)_{d-1}]$ for some $d\in \{4,5\}$.
	\end{enumerate}
\end{lemma}

\subsection{Log discrepancies}

Let $\pi\colon X\to \bar{X}$ be a resolution of singularities of a normal surface, with exceptional divisor $D$. Let $T_{1},\dots, T_{k}$ be the components of $D$. Since $D$ is negative definite, the formula
\begin{equation*}
	\phi^{*}K_{\bar{X}}=K_{X}+\sum_{i=1}^{k}\cf(T_{i})T_{i}
\end{equation*}
uniquely defines rational numbers $\cf(T_i)$, called the \emph{coefficients} of $T_{i}$ (with respect to $\bar{X}$). The \emph{log discrepancy} of $T_{i}$ is $\ld(T_i)\de 1-\cf(T_i)$. A normal surface is \emph{log terminal} if for some (or, equivalently, any) resolution, every component of the exceptional divisor has positive log discrepancy \cite[Definition 2.8]{Kollar_singularities_of_MMP}. By \cite[3.40]{Kollar_singularities_of_MMP}, a singularity is log terminal if and only if the exceptional divisor of its minimal resolution is an admissible chain or an admissible fork. If $D$ is a disjoint sum of admissible chains and admissible forks, and $C$ is a component of $D$, we write $\ld_{D}(C)$ for the log discrepancy of $C$ with respect to the log terminal surface obtained by contracting $D$. These numbers can be easily computed using Lemma \ref{lem:discriminants} and formulas in Lemma \ref{lem:ld_formulas} below. We will use them very frequently, often without explicit reference.

\begin{lemma}[{Formulas for log discrepancies, cf.\ \cite[II.3.3]{Miyan-OpenSurf} or \cite[3.2]{Flips_and_abundance}}]\label{lem:ld_formulas}\ 
	\begin{enumerate}
		\item\label{item:ld_chain} Let $T$ be an ordered admissible chain. Put 
		\begin{equation*}
			T\cp{<j}=\sum_{i=0}^{j-1}T\cp{i},\quad T\cp{>j}=\sum_{i=j+1}^{\#T}T\cp{i}. \qquad \mbox{Then} \quad 
			\ld_{T}(T\cp{j})=\frac{d(T\cp{>j})+d(T\cp{<j})}{d(T)}.
		\end{equation*}
	\item\label{item:ld_fork} Let $T=\langle b;T_1,T_2,T_3\rangle$ be an admissible fork with a branching component $B$. Put
	\begin{equation*}
		\delta=\sum_{i=1}^{3}\frac{1}{d(T_{i})},\quad e=\sum_{i=1}^{3}\frac{d(T_i-\ltip{T_i})}{d(T_i)}.
	\end{equation*}
	Then $\delta>1$, $e<2\leq b$, and 
	\begin{equation*}
		\ld_{T}(B)=\frac{\delta-1}{b-e},\quad \ld_{T}(T_{i}\cp{j})=\frac{\ld_{T}(B)\cdot d(T_i\cp{<j})+d(T_i\cp{>j})}{d(T_i)}.
	\end{equation*}
	\end{enumerate}
\end{lemma}

The \emph{weighted graph} $\Gamma$ of a reduced snc divisor $D$ has a vertex $v_{C}$ of weight $-C^2$ for each component $C$ of $D$, and an edge between $v_{C}$, $v_{C'}$ for each point of $C\cap C'$. We say that a subgraph $\Gamma'$ of $\Gamma$ is a \emph{weighted subgraph} of $\Gamma$ if the weight of each vertex of $\Gamma'$ does not exceed its weight as a vertex of $\Gamma$.

\begin{lemma}[{Log discrepancies do not decrease in weighted subgraphs \cite[L.1(2)]{Keel-McKernan_rational_curves}, cf.\ 
		\cite[Lemma 7.6]{Palka_almost_MMP}}] \label{lem:Alexeev}
		Let $D_1,D_2$ be rational chains or rational forks whose components have self-intersection numbers at most $-2$. Assume that $D_2$ is admissible, and the weighted graph of $D_1$ is a weighted subgraph of the weighted graph of $D_2$. Then $D_1$ is admissible, too, and if $C_1$ is a component of $D_1$ corresponding to a component $C_2$ of $D_2$, then 
	\begin{equation*}
		\ld_{D_1}(C_1)\geq \ld_{D_2}(C_2).
	\end{equation*}
\end{lemma}

A normal surface $\bar{X}$ is \emph{log terminal} if all its singularities are log terminal. The exceptional divisor of the minimal resolution of $\bar{X}$ is a disjoint union of admissible chains and forks, say $T_1,\dots, T_k$. The \emph{singularity type} of $\bar{X}$ is the collection of their types; we write it as $T_1+\dots+T_k$.

\subsection{\texorpdfstring{$\P^1$-}{P1-}-fibrations} \label{sec:P1-fibrations}

Let $X$ be a smooth projective surface. A \emph{$\P^{1}$-fibration} of $X$ is a morphism $p\colon X\to B$ onto a curve $B$ whose general fiber $F$ is isomorphic to $\P^1$. Let $D$ be a reduced snc divisor on $X$. The number $F\cdot D$ is called the \emph{height} of $p$ (with respect to $D$). The height of $(X,D)$, denoted by $\height(X,D)$, is the infimum of heights with respect to $D$ of all $\P^1$-fibrations of $X$, see formula~\eqref{eq:height}. A $\P^1$-fibration of $X$ whose height with respect to $D$ is equal to $\height(X,D)$ is called a \emph{witnessing} $\P^1$-fibration. 
\smallskip

Assume that $X$ admits a $\P^1$-fibration, and fix one such. For a reduced divisor $T$ we write $T\vert$ for the sum of components of $T$ which are \emph{vertical}, i.e.\ are contained in fibers. We call $T\hor\de T-T\vert$ the \emph{horizontal part} of $T$. A horizontal curve $H$ is called an $n$-section if $H\cdot F=n$ for some (hence every) fiber $F$. 

The \emph{width} of $(X,D)$, denoted by $\width(X,D)$, is the maximal number $\#D\hor$ among all witnessing $\P^1$-fibrations, see \cite[Definition 2.2]{PaPe_ht_2}. Clearly, $\width(X,D)\leq \height(X,D)$, and the equality holds if and only if $X$ admits a witnessing $\P^1$-fibration such that $D\hor$ consists of $1$-sections. If $(X,D)$ is the minimal log resolution of a normal surface $\bar{X}$ then we write $\height(\bar{X})\de \height(X,D)$, $\width(\bar{X})\de \width(X,D)$.
\smallskip

A fiber $F$ which is not isomorphic to $\P^1$ is called \emph{degenerate}. Such fibers are constructed by blowing up over $0$-curves, so their structure is easy to understand, see \cite[\S 4]{Fujita-noncomplete_surfaces}. In particular, the $(-1)$-curves in $F\redd$ are non-branching, and if $F\redd$ contains a unique $(-1)$-curve $L$ then $F$ has exactly two components of multiplicity one, and they are tips of $F\redd$. Furthermore, if those tips lie in different connected components of $F\redd-L$ then $F\redd$ is a chain of type $[T,1,T^{*}]$.

We will frequently use the following consequence of \cite[4.16]{Fujita-noncomplete_surfaces} proved in \cite[Lemma 2.6]{PaPe_MT}.

\begin{lemma}[Vertical $(-1)$-curves]
	\label{lem:fibrations-Sigma-chi}
	Let $(X,D)$ be the minimal log resolution of a del Pezzo surface of rank one. Fix a $\P^1$-fibration of $X$, let $F$ be a degenerate fiber and let $\sigma(F)$ be the number of $(-1)$-curves in $F$. 
	\begin{enumerate}
		\item \label{item:-1_curves}  A component of $F$ is a $(-1)$-curve if and only if it is not contained in $D$.
		\item \label{item:Sigma} We have $\#D\hor-1=\sum_{F}(\sigma(F)-1)$, where the sum runs over all degenerate fibers.
	\end{enumerate}
\end{lemma}

In the proof of Theorem \ref{thm:ht=3} (Step \ref{step:reverse}) 
we will reconstruct cascades of minimal log resolutions of del Pezzo surfaces of rank 1 by blowing up within degenerate fibers, or in other words, reversing vertical swaps, see Section \ref{sec:intro-structure}. The following elementary results tell us when to stop. More precisely, once we reach a log surface for which inequality \eqref{eq:ld_phi_H} fails, then Lemma \ref{lem:cascades}\ref{item:cascades-still-dP} implies that blowing up further cannot yield a del Pezzo surface, so the whole cascade has been reconstructed.

\begin{lemma}[{Criterion for ampleness of $-K_{\bar{X}}$, \cite[Lemma 2.6]{PaPe_ht_2}}]\label{lem:delPezzo_criterion}
	Let $\bar{X}$ be a normal surface of rank 1 such that $K_{\bar{X}}$ is $\Q$-Cartier. Let $\pi\colon X\to \bar{X}$ be its resolution, $D\de\Exc\pi$. Assume that $X$ admits a $\P^1$-fibration, let $F$ be its fiber and let $H_1,\dots, H_{h}$ be all horizontal components of $D$. Then $\bar{X}$ is del Pezzo if and only if
	\begin{equation}\label{eq:ld_phi_H}
		\sum_{j=1}^{h}\ld(H_j)\, H_j\cdot F>D\cdot F-2.
	\end{equation}
\end{lemma}

\begin{lemma}[{Log discrepancies do not decrease after swaps, \cite[Lemma 2.9]{PaPe_ht_2}}]\label{lem:cascades}
	Let $(X,D)$ be the minimal log resolution of a log terminal del Pezzo surface of rank one. Let $(X,D)\sqto (X',D')$ be a vertical swap.
	\begin{enumerate}
		\item\label{item:cascades-ld} Let $T'$ be a component of $D'$ and let $T$ be its proper transform on $X$. Then $\ld_{D'}(T')\geq \ld_{D}(T)$.
		\item\label{item:cascades-still-dP} The log surface $(X',D')$ is the minimal log resolution of a log terminal del Pezzo surface of rank one, too.
	\end{enumerate}
\end{lemma}
\begin{proof}
	Part \ref{item:cascades-ld} follows from Lemma \ref{lem:Alexeev}, and \ref{item:cascades-still-dP} follows from \ref{item:cascades-ld} and Lemma \ref{lem:delPezzo_criterion}, see \cite[Lemma 2.9]{PaPe_ht_2}.
\end{proof}

\subsection{Representing families and moduli dimension}\label{sec:moduli}

We now recall the definition of \emph{moduli dimension} used in Theorem \ref{thm:ht=3}\ref{item:uniq_cha=2}, see \cite[Definitions 1.9, 1.10]{PaPe_ht_2} for details. Let $\cC$ be a set consisting of some isomorphism classes of normal surfaces. We say that a smooth family $f\colon (\cX,\cD)\to B$, where $\cD$ is a divisor on $\cX$, \emph{represents} $\cC$ if the following hold. 
\begin{enumerate-alt}
	\item\label{item:moduli-irreducible} The restriction of $f$ to each component of  $\cD$ is a smooth surjective morphism with irreducible fibers.
	\item\label{item:moduli-iso-classes} The set of isomorphism classes of fibers $(X_b,D_b)\de (f^{-1}(b),\cD|_{f^{-1}(b)})$ equals the set of isomorphism classes of minimal log resolutions of surfaces in $\cC$. 
	\setcounter{foo}{\value{enumerate-alti}}
\end{enumerate-alt}
We say that the set $\cC$ \emph{has moduli dimension $d$} if it is represented by a family $f\colon (\cX,\cD)\to B$ with $\dim B=d$, which is \emph{almost universal}, that is, satisfies the following additional properties. 

\begin{enumerate-alt}	
	\setcounter{enumerate-alti}{\value{foo}}
	\item\label{item:moduli-group-action} The morphism $f$ is equivariant with respect to the action of some finite group $G$, such that two fibers $(X_b,D_b)$ and $(X_{b'},D_{b'})$ are isomorphic if and only if $b$ and $b'$ lie in the same $G$-orbit. 
	\item\label{item:moduli-semiuniversal} The formal germ of $f$ at each $b\in B$ is a semiuniversal deformation of $(X_b,D_b)$. \\ In particular, $h^{1}(\lts{X_b}{D_b})=d$.
\end{enumerate-alt}
The image in $\Aut(B)$ of the finite group $G$ from \ref{item:moduli-group-action} is called the \emph{symmetry group} of $f$.

\begin{remark}\label{rem:bases}
We will see in Lemma \ref{lem:w=2_cha=2}\ref{item:w=2-cha=2-moduli} that each set $\Phtl(\cS)$ in Theorem~\ref{thm:ht=3}\ref{item:uniq_cha=2} is represented by an almost universal families over $\Astst$, with symmetry group listed in Table \ref{table:ht=3_char=2_moduli}.
\end{remark} 

\begin{remark}\label{rem:exotic}
Theorem \ref{thm:ht=3} and \cite[Proposition C]{PaPe_ht_2} imply that for every log terminal singularity type $\cS$ other than $[2,2,3,(2)_{5}]+[3,2]$ the set $\Phtl(\cS)$ is represented by a family.

For $\cS= [2,2,3,(2)_{5}]+[3,2]$ this is not the case. To see this write $\Phtl(\cS)=\{\bar{X}_{1},\bar{X}_2\}$. By Proposition~\ref{prop:rigidity}\ref{item:rigidity-h1} the minimal log resolution $(X_i,D_i)$ of $\bar{X}_i$ for each $i=1,2$ has $h^{1}(\lts{X_i}{D_i})=0$, so its infinitesimal deformations are trivial. We note that if $\kk=\C$ then the smooth loci of $\bar{X}_{1}$ and $\bar{X}_2$ are not even homeomorphic in the Euclidean topology, as $H_{1}(\bar{X}_1\reg;\Z)=\Z/3$ and $H_1(\bar{X}_2\reg;\Z)=0$ by Proposition \ref{prop:exception}\ref{item:exception_H1}.
\end{remark}

\subsection{Descendants with elliptic boundary}

We now recall the definition of descendants with elliptic boundary from \cite[Definition 1.11]{PaPe_ht_2}, and introduce we the notion of an \emph{elliptic tie}, which will be used to distinguish del Pezzo surfaces with such descendants. Those surfaces are excluded in the assumptions of Theorem \ref{thm:ht=3}, because their common structure makes them easier to understand together, regardless of their height. They have been classified in \cite[Theorem E]{PaPe_ht_2}.

\begin{definition}[{Descendant with elliptic boundary, see \cite[Definition 1.11]{PaPe_ht_2}}]\label{def:GK}
	Let $(\bar{Y},\bar{T})$ be a log surface. We say that the boundary $\bar{T}$ is \emph{elliptic} if it is an irreducible curve of arithmetic genus one contained in the smooth locus of $\bar{Y}$. We say that $(\bar{Y},\bar{T})$ is a \emph{descendant} of a normal projective surface $\bar{X}$ if the minimal log resolution $(X,D)$ of $\bar{X}$ admits a birational morphism $\phi\colon X\to \bar{Y}$ such that $\phi_{*}D=\bar{T}$.
\end{definition}

\begin{definition}[An elliptic tie]\label{def:tie}
	Let $X$ be a smooth projective surface, and let $D$ be a reduced snc divisor on $X$. Let $L$ be a $(-1)$-curve on $X$ which is not contained in $D$, and let $E$ be the connected component of $D+L$ containing $L$. We say that $L$ is an \emph{elliptic tie} on $(X,D)$ if there is a birational morphism $\sigma\colon X\to Y$ onto a smooth surface $Y$ such that $\Exc\sigma\subseteq E$ and $\sigma_{*}E$ is an irreducible curve of arithmetic genus 1.
\end{definition}

\begin{lemma}\label{lem:tie}
	A normal projective surface of rank one has a descendant with elliptic boundary if and only if its minimal log resolution has an elliptic tie.
\end{lemma} 
\begin{proof}
	Let $(X,D)$ be the minimal log resolution of a normal surface $\bar{X}$. Assume that $(X,D)$ has an elliptic tie. Let $L$, $E$ and  $\sigma\colon X\to Y$ be as in Definition \ref{def:tie}. Since $\sigma$ is an isomorphism in a neighborhood of $D+L-E$, there is a birational morphism $\alpha\colon Y\to \bar{Y}$ with $\Exc\alpha=\sigma_{*}(D+L-E)$, so $(\bar{Y},(\alpha\circ\sigma)_{*}E)$ is a descendant of $\bar{X}$ with elliptic boundary. The converse follows from \cite[Lemma 1.12(b)]{PaPe_ht_2}, see Notation 6.1 loc.\ cit.
\end{proof}

We will see that del Pezzo surfaces of rank one and types listed in Theorem \ref{thm:ht=3} have 
no descendants with elliptic boundaries. For most types, this is a consequence of the following criterion.

\begin{lemma}[{Types with no elliptic-boundary descendants, \cite[Theorem E(b)]{PaPe_ht_2}}]\label{lem:no-deb}
	Let $\bar{X}$ be a del Pezzo surface of rank one and height at least $3$. Assume that the minimal resolution of all canonical singularities of $\bar{X}$ has at most $5$ exceptional components. Then $\bar{X}$ has no descendant with elliptic boundary.
\end{lemma}

\subsection{Summary of the notation}\label{sec:notation}

Let $(X,D)$ be the minimal log resolution of a log terminal del Pezzo surface $\bar{X}$ of rank~1. In the lemmas quoted in Theorem \ref{thm:ht=3} we describe the singularity type of $\bar{X}$ together with the structure of a certain witnessing $\P^1$-fibration $p\colon X\to \P^1$ (chosen so that it leads to vertical swaps in  Proposition  \ref{prop:ht=3_swaps}). More precisely, we describe the weighted graph of $D+\sum_{j}A_j$, where $A_1,\dots A_{n}$ are all vertical $(-1)$-curves, together with its decomposition into horizontal and vertical parts. We refer to this data as the \emph{combinatorial type of $(X,D,p)$}. To present it in an efficient way, we first write down the weighted graph of $D$ 
using notation from Section \ref{sec:types}, and then we decorate it as explained in \cite[\sec 2G]{PaPe_ht_2}. We now briefly recall this notation.
\smallskip 

The preimage of each singular point of $\bar{X}$ is an admissible chain or fork. We write down its type as follows. 
\begin{itemize}
	\item $[a_1,\dots,a_n]$ is a rational chain whose subsequent components have self-intersection numbers $-a_1,\dots,-a_n$,
	\item $\langle b;T_1,T_2,T_3\rangle$ is a rational fork with a branching component $[b]$ and twigs $T_1,T_2,T_3$,
	\item $(m)_{k}$ stands for an integer $m$ repeated $k$ times,
	\item we use the operator $*$ and convention for $(2)_{-1}$ introduced in formula \eqref{eq:conventions_Tono}.
\end{itemize}
We write the singularity type of $\bar{X}$ as a formal sum of types of its singularities, and decorate it as follows.
	\begin{itemize}
		\item We put in boldface those numbers which correspond to horizontal components of $D$. 
		\item If $D$ contains a $1$- and a $2$-section, then we underline the (bold) number corresponding to the $2$-section.
		\item We decorate by $\dec{j}$ those numbers which correspond to the components meeting the vertical $(-1)$-curve $A_j$. 
		\item We write $\ldec{j}[a_1,\dots,a_n]\de [a_{1}\dec{j},a_2,\dots,a_n]$ and $[a_1,\dots,a_n]\dec{j}\de [a_1,\dots,a_{n-1},a_{n}\dec{j}]$.
	\end{itemize} 
In the figures, we use the following notation:
\begin{itemize}
	\item A solid line with label \enquote{$n$} denotes an $n$-curve in $D$. Label $n=-2$ is skipped.
	\item A dashed line denotes one of the vertical $(-1)$-curves $A_j$. 
\end{itemize}
Sometimes, especially in Section \ref{sec:w=2}, we will enumerate the curves $A_j$ starting from $0$, not $1$.

\begin{notation}\label{not:P}
	For a singularity type $\cS$ and for $w\in \{1,2,3\}$ we denote by $\Pht^{\width=w}(\cS)$ the set of isomorphism classes of del Pezzo surfaces of rank one, height $3$ and width $w$.
\end{notation}

\clearpage
\section{Vertically primitive \texorpdfstring{$\P^1$}{P1}-fibrations of height $3$}\label{sec:basic} 

In this section we construct del Pezzo surfaces $\bar{Y}$ which serve as vertically primitive models in Proposition \ref{prop:ht=3_swaps}. More precisely, we construct log surfaces $(Y,D_Y)$ which are minimal log resolutions of del Pezzo surfaces $\bar{Y}$ of singularity types listed in Proposition \ref{prop:ht=3_swaps}, together with $\P^1$-fibrations $p_Y\colon Y\to \P^1$ of height $3$ with respect to $D_Y$, such that $(Y,D_Y,p_Y)$ is vertically primitive, see Section \ref{sec:intro-structure} for a definition. The log surface $(Y,D_Y)$, together with all degenerate fibers of $p_Y$, is shown in the last column of Table \ref{table:intro}. We put $\ww\de \#(D_Y)\hor$. Proposition \ref{prop:ht=3_swaps} asserts that the minimal log resolution of a log terminal del Pezzo surface of rank one, height $3$ and width $\ww$, with some witnessing $\P^1$-fibration, swaps vertically to some $(Y,D_Y,p_Y)$ constructed below.
\smallskip

In each case, the construction proceeds as follows. We start with a log surface $(Z,B)$ as in Proposition \ref{prop:ht=3_models} and define a morphism $\phi\colon (Y,D_Y)\to (Z,B)$ where $D_{Y}$ equals $\phi^{*}B\redd$ minus all $(-1)$-curves. The required $\P^1$-fibration of height $3$ on $(Y,D_Y)$ will be a pullback of the first projection in case $Z=\P^1\times \P^1$, and of a pencil of lines through a point $p_0$ in case $Z=\P^2$.

We write $\cS(\bar{Y})$ for the singularity type of $\bar{Y}$. Some basic properties of $\bar{Y}$ are listed in Proposition \ref{prop:primitive}.

\subsection{Case $\ww=3$, see Propositions \ref{prop:ht=3_models}\ref{item:w=3_models} and \ref{prop:ht=3_swaps}\ref{item:w=3_swaps}}\label{sec:basic-w=3}

\begin{example}[$\ww=3$, see Figure \ref{fig:w=3}]\label{ex:w=3}
	Let $Z=\P^1\times \P^1$ and $B=\sum_{j=1}^{3}(V_j+H_j)$, where $V_{j}$ and $H_{j}$ are 
	vertical and horizontal lines, respectively. Write $\{p_{ij}\}=V_{i}\cap H_{j}$ and let $v_{ij}$, $h_{ij}$ be the points infinitely near to $p_{ij}$ on the proper transforms of $V_{i}$ and $H_{j}$, respectively. Let $\phi\colon Y\to Z$ be the composition of the following blowups. 
	\begin{enumerate}
		\item \label{item:ht=3_A1+A2+A5} Blow up at $p_{13},v_{13},p_{21},h_{21},p_{32},h_{32},p_{33}$. Then $\cS(\bar{Y})=\rA_{1}+\rA_{2}+\rA_{5}$, see Figure \ref{fig:w=3_A1+A2+A5}.
		\item\label{item:ht=3_2A4} Blow up at $p_{13},v_{13},p_{23},p_{22},p_{32},p_{31},h_{31}$. Then $\cS(\bar{Y})=2\rA_{4}$, see Figure \ref{fig:w=3_2A4}. 
	\end{enumerate}
\begin{figure}[ht]
	\subcaptionbox{$\rA_{1}+\rA_{2}+\rA_{5}$ \label{fig:w=3_A1+A2+A5}}[.53\linewidth]{
	\begin{tikzpicture}
	\path[use as bounding box] (0,0) rectangle (8,3.2);
		\begin{scope}
			\draw (0,3) -- (3.2,3);
			\node at (0.5,3.2) {\small{$H_3$}};
			\draw (0,1.6) -- (3.2,1.6);
			\node at (0.5,1.8) {\small{$H_2$}};
			\draw (0,0.2) -- (3.2,0.2);
			\node at (0.5,0.4) {\small{$H_1$}};
			\draw (0.2,3.2) -- (0.2,0);
			\node at (0,0.8) {\small{$V_1$}};
			\draw (1.6,3.2) -- (1.6,0);
			\node at (1.4,0.8) {\small{$V_2$}};
			\draw (3,3.2) -- (3,0);
			\node at (2.8,0.8) {\small{$V_3$}};
			\filldraw (0.2,3) circle (0.06);
			\node at (-0.05,2.8) {\small{$p_{13}$}};
			\draw[-stealth] (0.3,2.9) -- (0.3,2.5);
			\node at (0.55,2.7) {\small{$v_{13}$}};
			\filldraw (1.6,0.2) circle (0.06);
			\node at (1.4,0) {\small{$p_{21}$}};
			\draw[-stealth] (1.5,0.3) -- (1.1,0.3);
			\node at (1.3,0.5) {\small{$h_{21}$}};
			\filldraw (3,3) circle (0.06);
			\node at (2.75,2.8) {\small{$p_{33}$}};
			\filldraw (3,1.6) circle (0.06);
			\node at (2.75,1.4) {\small{$p_{32}$}};
			\draw[-stealth] (2.9,1.7) -- (2.5,1.7);
			\node at (2.7,1.9) {\small{$h_{32}$}};
			\node at (4,1.75) {\small{$\phi$}};
			\draw[<-] (3.6,1.6) -- (4.4,1.6);
		\end{scope}
		\begin{scope}[shift={(5,0)}]
			\draw (-0.8,3) -- (2.6,3);
			\node at (-0.7,3.2) {\small{$H_3$}};
			\draw (0.2,3.2) -- (0,2);
			\draw[dashed] (0,2.2) -- (0.2,1);
			\draw (0.2,1.2) -- (0,0);
			\node at (-0.1,0.5) {\small{$V_1$}};
			\draw[dashed] (1.2,3.2) -- (1,2);
			\node at (0.9,2.7) {\small{$V_2$}};
			\draw (1,2.2) -- (1.2,1);
			\draw[dashed] (1.2,1.2) -- (1,0);
			\draw[dashed] (2.4,3.2) -- (2.2,2.2);
			\draw (2.2,2.4) -- (2.4,1.4);
			\draw (2.4,1.6) -- (2.2,0.6);
			\draw[dashed] (2.2,0.8) -- (2.4,-0.2);
			\node at (2.1,2) {\small{$V_3$}}; 
			\draw (-0.8,0.9) -- (0.2,0.9) to[out=0,in=180] (1,2.4) -- (1.2,2.4) to[out=0,in=180] (2.4,0.2) -- (2.6,0.2);
			\node at (-0.7,1.1) {\small{$H_2$}};
			\draw (-0.8,0.2) -- (1.2,0.2) to[out=0,in=-120] (1.8,0.85);   
			\draw (1.9,1.05) to[out=60,in=180] (2.4,1.8) -- (2.6,1.8);
			\node at (-0.7,0.4) {\small{$H_1$}};
		\end{scope}
	\end{tikzpicture}
	}
	\subcaptionbox{$2\rA_{4}$ \label{fig:w=3_2A4}}[.46\linewidth]{
		\begin{tikzpicture}
		\path[use as bounding box] (0,0) rectangle (7.6,3.2);
			\begin{scope}
				\draw (0,3) -- (3.2,3);
				\node at (0.5,3.2) {\small{$H_3$}};
				\draw (0,1.6) -- (3.2,1.6);
				\node at (0.5,1.8) {\small{$H_2$}};
				\draw (0,0.2) -- (3.2,0.2);
				\node at (0.5,0.4) {\small{$H_1$}};
				\draw (0.2,3.2) -- (0.2,0);
				\node at (0,0.8) {\small{$V_1$}};
				\draw (1.6,3.2) -- (1.6,0);
				\node at (1.4,0.8) {\small{$V_2$}};
				\draw (3,3.2) -- (3,0);
				\node at (2.8,0.8) {\small{$V_3$}};
				\filldraw (0.2,3) circle (0.06);
				\node at (-0.05,2.8) {\small{$p_{13}$}};
				\draw[-stealth] (0.3,2.9) -- (0.3,2.5);
				\node at (0.55,2.7) {\small{$v_{13}$}};
				\filldraw (1.6,3) circle (0.06);
				\node at (1.9,2.8) {\small{$p_{23}$}};
				\filldraw (1.6,1.6) circle (0.06);
				\node at (1.9,1.8) {\small{$p_{22}$}};
				\filldraw (3,1.6) circle (0.06);
				\node at (2.75,1.8) {\small{$p_{32}$}};
				\filldraw (3,0.2) circle (0.06);
				\node at (2.8,0) {\small{$p_{31}$}};
				\draw[-stealth] (2.9,0.3) -- (2.5,0.3);
				\node at (2.7,0.5) {\small{$h_{31}$}};
				\node at (4,1.75) {\small{$\phi$}};
				\draw[<-] (3.6,1.6) -- (4.4,1.6);
			\end{scope}
			\begin{scope}[shift={(5,0)}]
				\draw (-0.8,3) -- (2.6,3);
				\node at (-0.7,3.2) {\small{$H_3$}};
				\draw (0.2,3.2) -- (0,2);
				\draw[dashed] (0,2.2) -- (0.2,1);
				\draw (0.2,1.2) -- (0,0);
				\node at (-0.1,0.5) {\small{$V_1$}};
				\draw[dashed] (1.2,3.2) -- (1,2);
				\draw (1,2.2) -- (1.2,1);
				\node at (0.9,1.6) {\small{$V_2$}};
				\draw[dashed] (1.2,1.2) -- (1,0);
				\draw (2.2,3.2) -- (2,2.2);
				\node at (1.85,2.65) {\small{$V_3$}};
				\draw (2,2.4) -- (2.2,1.4);
				\draw[dashed] (2.2,1.6) -- (2,0.6);
				\draw[dashed] (2,2.8) to[out=0,in=80] (2.4,0);
				\draw (-0.8,0.9) -- (0.2,0.9) to[out=0,in=180] (1,1.4) -- (1.2,1.4) to[out=0,in=180] (2,0.8)-- (2.2,0.8);
				\node at (-0.7,1.1) {\small{$H_1$}};
				\draw (-0.8,0.2) -- (2.6,0.2);
				\node at (-0.7,0.4) {\small{$H_2$}};
			\end{scope}
		\end{tikzpicture}
	}
	\caption{Example \ref{ex:w=3}: vertically primitive surfaces in Proposition \ref{prop:ht=3_swaps}\ref{item:ht=w=3}, $\ww=3$.}
	\label{fig:w=3}
\end{figure}
\end{example}

\subsection{Case $\ww=2$, see Propositions \ref{prop:ht=3_models}\ref{item:w=2_models} and \ref{prop:ht=3_swaps}\ref{item:ht=3,w=2,cha-neq-2_can},\ref{item:ht=3,w=2,cha-neq-2_GK},\ref{item:swap_cha=2_GK}}\label{sec:basic-w=2}

\begin{example}[$\ww=2$, $\cha\kk\neq 2$, see Figure \ref{fig:w=2_cha_neq_2}]\label{ex:w=2_cha_neq_2}
	Assume $\cha\kk\neq 2$. Let $B\subseteq \P^2$ be the sum of a conic $\cc$ and three lines $\ll_1,\ll_2,\ll_3$ meeting at a point $p_0\not\in \cc$, such that for some points $p_1,p_2,p_3$ we have $(\ll_{1}\cdot \cc)_{p_{1}}=1$ and $(\ll_{j}\cdot \cc)_{p_{j}}=2$ for $j\in \{2,3\}$. 
	For $i,j\in \{0,\dots, 3\}$ let $p_{ij}$ and $p_{i}',p_{i}'',\dots$ be the points infinitely near to $p_i$ on the proper transforms of $\ll_{j}$ and $\cc$; respectively. Let $\phi\colon Y\to \P^2$ be the composition of the following blowups. 
	\begin{enumerate}
		\item\label{item:ht=3_exception} Blow up at $p_0,p_{02},p_{1},p_{11},p_{2},p_{2}',p_{3},p_{3}',p_{3}''$. Then $\cS(\bar{Y})=\rA_{1}+\rA_{7}+[3]$, see Figure \ref{fig:ht=3_exception}.
		\item \label{item:w=2_A1+A2+A5} Blow up at $p_0,p_{01},p_{1},p_{2},p_{2}',p_{2}'',p_{3},p_{3}'$. Then $\cS(\bar{Y})=\rA_{1}+\rA_{2}+\rA_{5}$, see Figure \ref{fig:w=2_A1+A2+A5}.
		\item\label{item:w=2_2A1+2A3} Blow up at $p_0,p_{01},p_{1},p_{1}',p_{2},p_{2}',p_{3},p_{3}'$. Then $\cS(\bar{Y})=2\rA_{1}+2\rA_{3}$, see Figure \ref{fig:w=2_2A1+2A3}.
	\end{enumerate}	
\end{example} 

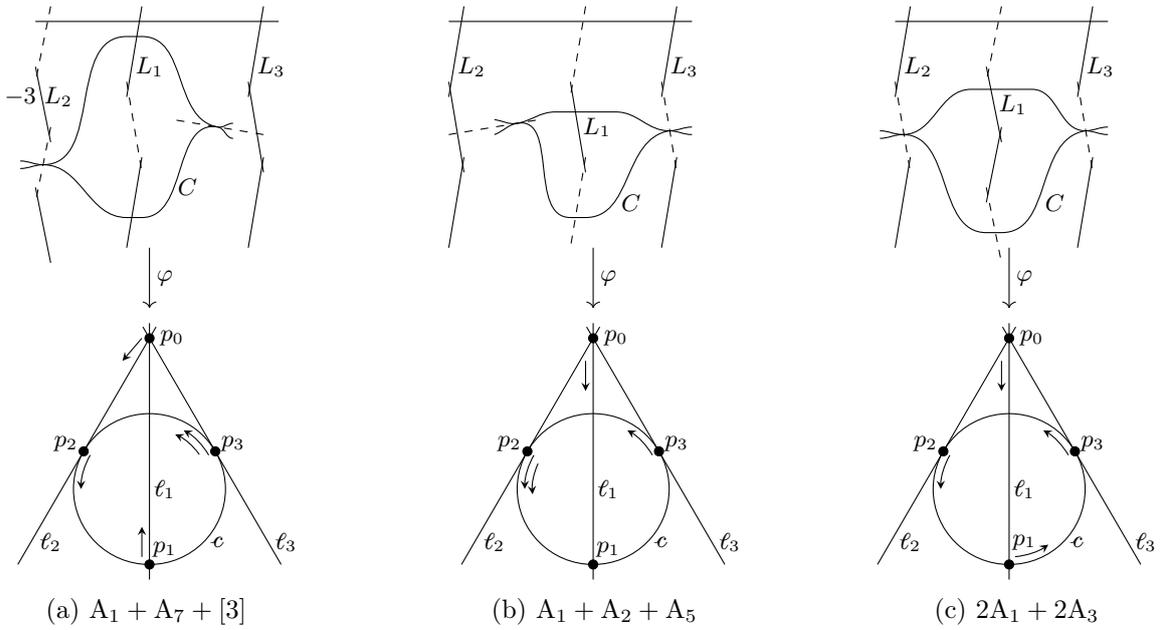
\begin{figure}
	\subcaptionbox{$\rA_1+\rA_7+[3]$ \label{fig:ht=3_exception}}[.34\linewidth]
	{
	\begin{tikzpicture}
		\begin{scope}
			\draw (0,3) -- (3.2,3);
			\draw[dashed] (0.2,3.2) -- (0,2.2);
			\draw (0,2.4) -- (0.2,1.4);
			\node at (-0.2,2) {\small{$-3$}};
			\node at (0.3,2) {\small{$L_2$}};
			\draw[dashed] (0.2,1.6) -- (0,0.6); 
			\draw (0,0.8) -- (0.2,-0.2);
			\draw (1.4,3.2) -- (1.2,2);
			\node at (1.5,2.4) {\small{$L_1$}};
			\draw[dashed] (1.2,2.2) -- (1.4,1);
			\draw (1.4,1.2) -- (1.2,0); 
			\draw (3,3.2) -- (2.8,2);
			\node at (3.1,2.4) {\small{$L_3$}};
			\draw (2.8,2.2) -- (3,1);
			\draw[dashed] (3,1.5) -- (1.8,1.7);
			\draw (3,1.2) -- (2.8,0); 
			\draw (-0.2,1.15) to[out=0,in=180] (0.1,1.1) to[out=0,in=180] (1.2,2.8) -- (1.4,2.8) to[out=0,in=170] (2.4,1.6) to[out=-10,in=180] (2.6,1.65);
			\draw (-0.2,1.05) to[out=0,in=180] (0.1,1.1) to[out=0,in=180] (1.2,0.4) -- (1.4,0.4) to[out=0,in=170] (2.4,1.6) to[out=-10,in=180] (2.6,1.45);
			\node at (2,0.8) {\small{$C$}};
			\draw[->] (1.5,0) -- (1.5,-0.8);
			\node at (1.7,-0.4) {\small{$\phi$}};
		\end{scope}
		\begin{scope}[shift={(1.5,-3.2)}]
			\draw (0,0) circle (1);
			\node at (0.9,-0.7) {\small{$\cc$}};
			\draw[add= 0.1 and 1] (0,2) to (-0.866,0.5);
			\node at (-1.3,-0.7) {\small{$\ll_2$}};
			\draw[add= 0.1 and 1] (0,2) to (0.866,0.5);
			\node at (1.8,-0.7) {\small{$\ll_3$}};
			\draw (0,2.2) -- (0,-1.2);
			\node at (0.2,0) {\small{$\ll_1$}};
			\filldraw (0,2) circle (0.06);
			\node at (0.3,2) {\small{$p_0$}};
			\draw[-stealth] (-0.1,2) -- (-0.355,1.7);
			\filldraw (-0.866,0.5) circle (0.06);
			\node at (-1.1,0.6) {\small{$p_2$}};
			\draw (0,0) [-stealth, partial ellipse=150:180:0.9 and 0.9];
			\filldraw (0.866,0.5) circle (0.06);
			\node at (1.1,0.6) {\small{$p_3$}};
			\draw (0,0) [-stealth, partial ellipse=30:60:0.9 and 0.9];
			\draw (0,0) [-stealth, partial ellipse=35:65:0.8 and 0.8];
			\filldraw (0,-1) circle (0.06);
			\node at (0.2,-0.8) {\small{$p_1$}};
			\draw[-stealth] (-0.1,-0.9) -- (-0.1,-0.5); 
		\end{scope}
	\end{tikzpicture}
	}
	\subcaptionbox{$\rA_1+\rA_2+\rA_5$ \label{fig:w=2_A1+A2+A5}}[.34\linewidth]
	{
		\begin{tikzpicture}
			\begin{scope}
				\draw (-0.4,3) -- (2.8,3);
				\draw (-0.2,3.2) -- (-0.4,2);
				\node at (-0.1,2.4) {\small{$L_2$}};	
				\draw (-0.4,2.2) -- (-0.2,1);			
				\draw[dashed] (-0.4,1.5) -- (0.8,1.7);
				\draw (-0.2,1.2) -- (-0.4,0);
				\draw[dashed] (1.4,3.2) -- (1.2,2);
				\draw (1.2,2.2) -- (1.4,1);
				\node at (1.55,1.6) {\small{$L_1$}};				
				\draw[dashed] (1.4,1.2) -- (1.2,0); 
				\draw (2.6,3.2) -- (2.4,2);
				\node at (2.7,2.4) {\small{$L_3$}};
				\draw[dashed] (2.4,2.2) -- (2.6,1);
				\draw (2.6,1.2) -- (2.4,0);
				\draw (0.2,1.7) to[out=-10,in=-170] (0.5,1.65) to[out=10,in=180] (1,1.8) -- (1.8,1.8)  to[out=0,in=180] (2.5,1.55) to[out=0,in=180] (2.8,1.6);
				\draw (0.2,1.5) to[out=20,in=-170] (0.5,1.65) to[out=10,in=180] (1.2,0.4) -- (1.4,0.4) to[out=0,in=190] (2.5,1.55) to[out=0,in=180] (2.8,1.5);
				\node at (2,0.6) {\small{$C$}};
				\draw[->] (1.5,0) -- (1.5,-0.8);
				\node at (1.7,-0.4) {\small{$\phi$}};
			\end{scope}
			\begin{scope}[shift={(1.5,-3.2)}]
				\draw (0,0) circle (1);
				\node at (0.9,-0.7) {\small{$\cc$}};
				\draw[add= 0.1 and 1] (0,2) to (-0.866,0.5);
				\node at (-1.3,-0.7) {\small{$\ll_2$}};
				\draw[add= 0.1 and 1] (0,2) to (0.866,0.5);
				\node at (1.8,-0.7) {\small{$\ll_3$}};
				\draw (0,2.2) -- (0,-1.2);
				\node at (0.2,0) {\small{$\ll_1$}};
				\filldraw (0,2) circle (0.06);
				\node at (0.3,2) {\small{$p_0$}};
				\draw[-stealth] (-0.1,1.7) -- (-0.1,1.3);
				\filldraw (-0.866,0.5) circle (0.06);
				\node at (-1.1,0.6) {\small{$p_2$}};
				\draw (0,0) [-stealth, partial ellipse=150:180:0.9 and 0.9];
				\draw (0,0) [-stealth, partial ellipse=155:185:0.8 and 0.8];
				\filldraw (0.866,0.5) circle (0.06);
				\node at (1.1,0.6) {\small{$p_3$}};
				\draw (0,0) [-stealth, partial ellipse=30:60:0.9 and 0.9];
				\filldraw (0,-1) circle (0.06);
				\node at (0.2,-0.8) {\small{$p_1$}};
			\end{scope}
		\end{tikzpicture}
	}
	\subcaptionbox{$2\rA_1+2\rA_3$ \label{fig:w=2_2A1+2A3}}[.3\linewidth]
	{
		\begin{tikzpicture}
			\begin{scope}
				\draw (0,3) -- (2.8,3);
				\draw (0.2,3.2) -- (0,2);
				\node at (0.3,2.4) {\small{$L_2$}};	
				\draw[dashed] (0,2.2) -- (0.2,1);			
				\draw (0.2,1.2) -- (0,0);
				\draw[dashed] (1.4,3.2) -- (1.2,2.2);
				\draw (1.2,2.4) -- (1.4,1.4);
				\node at (1.55,1.9) {\small{$L_1$}};
				\draw (1.4,1.6) -- (1.2,0.6); 
				\draw[dashed] (1.2,0.8) -- (1.4,-0.2);
				\draw (2.6,3.2) -- (2.4,2);
				\node at (2.7,2.4) {\small{$L_3$}};
				\draw[dashed] (2.4,2.2) -- (2.6,1);
				\draw (2.6,1.2) -- (2.4,0); 
				\draw (-0.2,1.55) to[out=0,in=180] (0.1,1.5) to[out=0,in=180] (1,2.1) -- (1.8,2.1)  to[out=0,in=180] (2.5,1.55) to[out=0,in=180] (2.8,1.6);
				\draw (-0.2,1.45) to[out=0,in=180] (0.1,1.5) to[out=0,in=180] (1.2,0.2) -- (1.4,0.2) to[out=0,in=190] (2.5,1.55) to[out=0,in=180] (2.8,1.5);
				\node at (2.1,0.6) {\small{$C$}};
				\draw[->] (1.5,0) -- (1.5,-0.8);
				\node at (1.7,-0.4) {\small{$\phi$}};
		\end{scope}
		\begin{scope}[shift={(1.5,-3.2)}]
				\draw (0,0) circle (1);
				\node at (0.9,-0.7) {\small{$\cc$}};
				\draw[add= 0.1 and 1] (0,2) to (-0.866,0.5);
				\node at (-1.3,-0.7) {\small{$\ll_2$}};
				\draw[add= 0.1 and 1] (0,2) to (0.866,0.5);
				\node at (1.8,-0.7) {\small{$\ll_3$}};
				\draw (0,2.2) -- (0,-1.2);
				\node at (0.2,0) {\small{$\ll_1$}};
				\filldraw (0,2) circle (0.06);
				\node at (0.3,2) {\small{$p_0$}};
				\draw[-stealth] (-0.1,1.7) -- (-0.1,1.3);
				\filldraw (-0.866,0.5) circle (0.06);
				\node at (-1.1,0.6) {\small{$p_2$}};
				\draw (0,0) [-stealth, partial ellipse=150:180:0.9 and 0.9];
				\filldraw (0.866,0.5) circle (0.06);
				\node at (1.1,0.6) {\small{$p_3$}};
				\draw (0,0) [-stealth, partial ellipse=30:60:0.9 and 0.9];
				\filldraw (0,-1) circle (0.06);
				\node at (0.2,-0.7) {\small{$p_1$}};
				\draw (0,0) [-stealth, partial ellipse=275:305:0.9 and 0.9];
			\end{scope}
		\end{tikzpicture}
	}
	\caption{Example \ref{ex:w=2_cha_neq_2}: vertically primitive surfaces in Proposition  \ref{prop:ht=3_swaps}\ref{item:ht=3,w=2,cha-neq-2_GK}: $\ww=2$, $\cha\kk\neq 2$.}
	\label{fig:w=2_cha_neq_2}
\end{figure}

\begin{example}[$\ww=2$, $\cha\kk=2$, see Figure \ref{fig:w=2_cha=2}]\label{ex:w=2_cha=2}
	Assume $\cha\kk=2$. Let $B\subseteq \P^2$ be a sum of a conic $\cc$ and $\nu$ lines $\ll_{1},\dots,\ll_{\nu}$ tangent to $\cc$. Since $\cha\kk=2$, the lines $\ll_1,\dots,\ll_{\nu}$ meet at a common point $p_0\not\in \cc$. Write $\{p_j\}=\ll_j\cap \cc$. Let $p_{ij}, p_{j}'$ be the infinitely near points as in Example \ref{ex:w=2_cha_neq_2}. Let $\phi\colon Y\to \P^2$ be the composition of the following blowups. 
	\begin{enumerate}
		\item\label{item:swap-to-nu=4_small} Take $\nu=4$. Blow up at $p_{0}$, $p_{01}$, $p_{1}$, $p_{j},p_{j}'$, $j\in \{2,3,4\}$. Then $\cS(\bar{Y})=3\rA_{1}+\rD_{4}+[2,3]$, see Figure \ref{fig:swap-to-nu=4_small}. 
		\item\label{item:swap-to-nu=4_large}  Take $\nu=4$. Blow up at $p_{0}$, $p_{01}$, $p_{j},p_{j}'$, $j\in \{1,2,3,4\}$. Then $\cS(\bar{Y})=4\rA_{1}+\rD_{4}+[4]+[3]$, see Figure  \ref{fig:swap-to-nu=4_large}.
		\item\label{item:swap-to-nu=3} Take $\nu=3$. Blow up at $p_{0}$, $p_{01}$, $p_{j},p_{j}'$, $j\in \{1,2,3\}$. Then $\cS(\bar{Y})=4\rA_{1}+\rA_{3}+[3]$, see Figure \ref{fig:swap-to-nu=3}.
	\end{enumerate}
\end{example}

\begin{figure}
	\subcaptionbox{$3\rA_{1}+\rD_{4}+[2,3]$ \label{fig:swap-to-nu=4_small}}[.34\linewidth]
	{
		\begin{tikzpicture}
			\begin{scope}
				\draw (0,3) -- (4,3);
				\draw[dashed] (0.2,3.2) -- (0,2);
				\node at (0.3,1.6) {\small{$L_1$}};	
				\draw (0,2.2) -- (0.2,1);			
				\draw[dashed] (0.2,1.2) -- (0,0);
				\draw (1.4,3.2) -- (1.2,2);
				\draw[dashed] (1.2,2.2) -- (1.4,1);
				\node at (1.5,2.4) {\small{$L_2$}};				
				\draw (1.4,1.2) -- (1.2,0); 
				\draw (2.6,3.2) -- (2.4,2);
				\draw[dashed] (2.4,2.2) -- (2.6,1);
				\node at (2.7,2.4) {\small{$L_3$}};				
				\draw (2.6,1.2) -- (2.4,0); 
				\draw (3.8,3.2) -- (3.6,2);
				\draw[dashed] (3.6,2.2) -- (3.8,1);
				\node at (3.9,2.4) {\small{$L_4$}};				
				\draw (3.8,1.2) -- (3.6,0); 
				\draw[thick] (0,1.1) -- (0.4,1.1) to[out=0,in=180] (1,1.6) -- (3.8,1.6);
				\node at (3.1,1.8) {\small{$-3$}};
				\node at (3.1,1.4) {\small{$C$}};
				\draw[->] (1.9,0) -- (1.9,-0.8);
				\node at (2.1,-0.4) {\small{$\phi$}};
			\end{scope}
			\begin{scope}[shift={(1.9,-3.2)}]
				\draw (0,0) circle (1);
				\node at (1.05,-0.5) {\small{$\cc$}};
				\draw[add= 0.1 and 0.8] (0,2) to (-0.866,0.5);
				\node at (-1.5,-0.1) {\small{$\ll_1$}};
				\filldraw (-0.866,0.5) circle (0.06);
				\node at (-1.1,0.6) {\small{$p_1$}};
				\draw[add=0.1 and -0.65] (0,2) to (-0.5,-0.6);
				\draw[add=-0.45 and 0] (0,2) to (-0.5,-0.6);
				\draw (-0.5,-0.6) to[out=-107,in=156] (-0.4,-0.92) to[out=-24,in=-107] (-0.15,-0.6);
				\node at (-0.2,-0.1) {\small{$\ll_2$}};
				\filldraw (-0.4,-0.92) circle (0.06);
				\node at (-0.65,-1.05) {\small{$p_2$}};
				\draw (0,0) [-stealth, partial ellipse=246:276:1.1 and 1.1];
				\draw[add=0.1 and -0.65] (0,2) to (0.5,-0.6);
				\draw[add=-0.45 and 0] (0,2) to (0.5,-0.6);
				\draw (0.5,-0.6) to[out=-73,in=24] (0.4,-0.92) to[out=196,in=-73] (0.15,-0.6);
				\node at (0.65,-0.1) {\small{$\ll_3$}};
				\filldraw (0.4,-0.92) circle (0.06);
				\node at (0.8,-1.05) {\small{$p_3$}};
				\draw (0,0) [-stealth, partial ellipse=294:324:1.1 and 1.1];
				\draw[add= 0.1 and 0.8] (0,2) to (0.866,0.5);
				\node at (1.5,-0.1) {\small{$\ll_4$}};
				\filldraw (0.866,0.5) circle (0.06);
				\node at (1.1,0.6) {\small{$p_4$}};	
				\draw (0,0) [-stealth, partial ellipse=30:60:0.9 and 0.9];			
				\filldraw (0,2) circle (0.06);
				\node at (0.3,2) {\small{$p_0$}};
				\draw[-stealth] (-0.1,2) -- (-0.355,1.7);
			\end{scope}
		\end{tikzpicture}
	}
	\subcaptionbox{$4\rA_{1}+\rD_{4}+[4]+[3]$ \label{fig:swap-to-nu=4_large}}[.34\linewidth]
	{
		\begin{tikzpicture}
			\begin{scope}
				\draw (0,3) -- (4,3);
				\draw[dashed] (0.2,3.2) -- (0,2.2);
				\draw (0,2.4) -- (0.2,1.4);
				\node at (-0.2,2) {\small{$-3$}};
				\node at (0.3,2) {\small{$L_1$}};
				\draw[dashed] (0.2,1.6) -- (0,0.6); 
				\draw (0,0.8) -- (0.2,-0.2);
				\draw (1.4,3.2) -- (1.2,2);
				\draw[dashed] (1.2,2.2) -- (1.4,1);
				\node at (1.5,2.4) {\small{$L_2$}};				
				\draw (1.4,1.2) -- (1.2,0); 
				\draw (2.6,3.2) -- (2.4,2);
				\draw[dashed] (2.4,2.2) -- (2.6,1);
				\node at (2.7,2.4) {\small{$L_3$}};				
				\draw (2.6,1.2) -- (2.4,0); 
				\draw (3.8,3.2) -- (3.6,2);
				\draw[dashed] (3.6,2.2) -- (3.8,1);
				\node at (3.9,2.4) {\small{$L_4$}};				
				\draw (3.8,1.2) -- (3.6,0); 
				\draw[thick] (0,1.1) -- (0.4,1.1) to[out=0,in=180] (1,1.6) -- (3.8,1.6);
				\node at (3.1,1.8) {\small{$-4$}};
				\node at (3.1,1.4) {\small{$C$}};
				\draw[->] (1.9,0) -- (1.9,-0.8);
				\node at (2.1,-0.4) {\small{$\phi$}};
			\end{scope}
		\begin{scope}[shift={(1.9,-3.2)}]
				\draw (0,0) circle (1);
				\node at (1.05,-0.5) {\small{$\cc$}};
				\draw[add= 0.1 and 0.8] (0,2) to (-0.866,0.5);
				\node at (-1.5,-0.1) {\small{$\ll_1$}};
				\filldraw (-0.866,0.5) circle (0.06);
				\node at (-1.1,0.6) {\small{$p_1$}};
				\draw (0,0) [-stealth, partial ellipse=150:180:0.9 and 0.9];
				\draw[add=0.1 and -0.65] (0,2) to (-0.5,-0.6);
				\draw[add=-0.45 and 0] (0,2) to (-0.5,-0.6);
				\draw (-0.5,-0.6) to[out=-107,in=156] (-0.4,-0.92) to[out=-24,in=-107] (-0.15,-0.6);
				\node at (-0.2,-0.1) {\small{$\ll_2$}};
				\filldraw (-0.4,-0.92) circle (0.06);
				\node at (-0.65,-1.05) {\small{$p_2$}};
				\draw (0,0) [-stealth, partial ellipse=246:276:1.1 and 1.1];
				\draw[add=0.1 and -0.65] (0,2) to (0.5,-0.6);
				\draw[add=-0.45 and 0] (0,2) to (0.5,-0.6);
				\draw (0.5,-0.6) to[out=-73,in=24] (0.4,-0.92) to[out=196,in=-73] (0.15,-0.6);
				\node at (0.65,-0.1) {\small{$\ll_3$}};
				\filldraw (0.4,-0.92) circle (0.06);
				\node at (0.8,-1.05) {\small{$p_3$}};
				\draw (0,0) [-stealth, partial ellipse=294:324:1.1 and 1.1];
				\draw[add= 0.1 and 0.8] (0,2) to (0.866,0.5);
				\node at (1.5,-0.1) {\small{$\ll_4$}};
				\filldraw (0.866,0.5) circle (0.06);
				\node at (1.1,0.6) {\small{$p_4$}};	
				\draw (0,0) [-stealth, partial ellipse=30:60:0.9 and 0.9];			
				\filldraw (0,2) circle (0.06);
				\node at (0.3,2) {\small{$p_0$}};
				\draw[-stealth] (-0.1,2) -- (-0.355,1.7);
			\end{scope}
		\end{tikzpicture}
	}
	\subcaptionbox{$4\rA_{1}+\rA_{3}+[3]$, see \cite[3.9]{PaPe_ht_2} \label{fig:swap-to-nu=3}}[.3\linewidth]
	{
		\begin{tikzpicture}
			\begin{scope}
				\draw (0,3) -- (2.8,3);
				\draw[dashed] (0.2,3.2) -- (0,2.2);
				\draw (0,2.4) -- (0.2,1.4);
				\node at (-0.2,2) {\small{$-3$}};
				\node at (0.3,2) {\small{$L_1$}};
				\draw[dashed] (0.2,1.6) -- (0,0.6); 
				\draw (0,0.8) -- (0.2,-0.2);
				\draw (1.4,3.2) -- (1.2,2);
				\draw[dashed] (1.2,2.2) -- (1.4,1);
				\node at (1.5,2.4) {\small{$L_2$}};				
				\draw (1.4,1.2) -- (1.2,0); 
				\draw (2.6,3.2) -- (2.4,2);
				\draw[dashed] (2.4,2.2) -- (2.6,1);
				\node at (2.7,2.4) {\small{$L_3$}};				
				\draw (2.6,1.2) -- (2.4,0); 
				\draw[thick] (0,1.1) -- (0.4,1.1) to[out=0,in=180] (1,1.6) -- (2.6,1.6);
				\node at (1.9,1.4) {\small{$C$}};
				\draw[->] (1.5,0) -- (1.5,-0.8);
				\node at (1.7,-0.4) {\small{$\phi$}};
			\end{scope}
			\begin{scope}[shift={(1.5,-3.2)}]
				\draw (0,0) circle (1);
				\node at (1.05,-0.5) {\small{$\cc$}};
				\draw[add= 0.1 and 0.8] (0,2) to (-0.866,0.5);
				\node at (-1.5,-0.1) {\small{$\ll_1$}};
				\filldraw (-0.866,0.5) circle (0.06);
				\node at (-1.1,0.6) {\small{$p_1$}};
				\draw (0,0) [-stealth, partial ellipse=150:180:0.9 and 0.9];
				\draw[add=0.1 and -0.65] (0,2) to (-0.5,-0.6);
				\draw[add=-0.45 and 0] (0,2) to (-0.5,-0.6);
				\draw (-0.5,-0.6) to[out=-107,in=156] (-0.4,-0.92) to[out=-24,in=-107] (-0.15,-0.6);
				\node at (-0.2,-0.1) {\small{$\ll_2$}};
				\filldraw (-0.4,-0.92) circle (0.06);
				\node at (-0.65,-1.05) {\small{$p_2$}};
				\draw (0,0) [-stealth, partial ellipse=246:276:1.1 and 1.1];
				\draw[add= 0.1 and 0.8] (0,2) to (0.866,0.5);
				\node at (1.5,-0.1) {\small{$\ll_3$}};
				\filldraw (0.866,0.5) circle (0.06);
				\node at (1.1,0.6) {\small{$p_3$}};	
				\draw (0,0) [-stealth, partial ellipse=30:60:0.9 and 0.9];			
				\filldraw (0,2) circle (0.06);
				\node at (0.3,2) {\small{$p_0$}};
				\draw[-stealth] (-0.1,2) -- (-0.355,1.7);
			\end{scope}
		\end{tikzpicture}
	}
	\caption{Example \ref{ex:w=2_cha=2}: vertically primitive surfaces in  Prop.\ \ref{prop:ht=3_swaps}\ref{item:ht=3,w=2,cha-neq-2_GK},\ref{item:swap_cha=2_GK}: {$\ww=2$,~$\cha\kk=2$}.} 
	\label{fig:w=2_cha=2}
\end{figure}
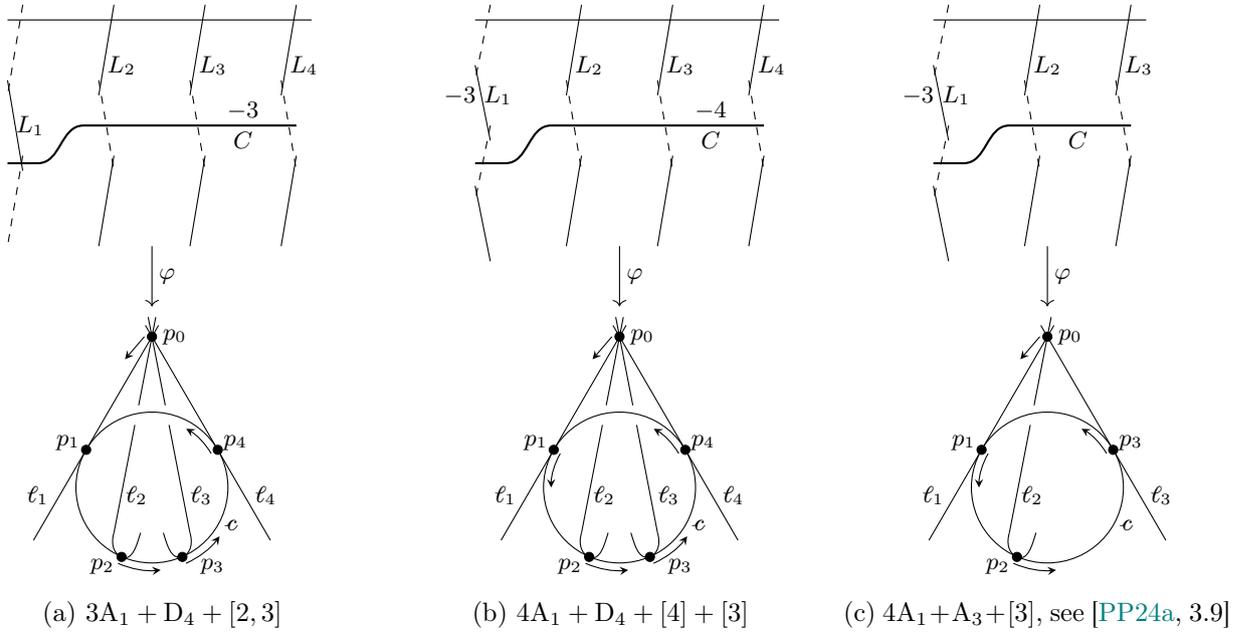

\subsection{Case $\ww=1$, see Propositions \ref{prop:ht=3_models}\ref{item:w=1_models} and \ref{prop:ht=3_swaps}\ref{item:swap_ht=3,w=1_cha-neq-3},\ref{item:swap_ht=3,w=1_cha=3_GK},\ref{item:swap_cha=3}}\label{sec:basic-w=1}

\begin{example}[$\ww=1$, see Figure \ref{fig:w=1}]\label{ex:w=1}
	Assume $\cha\kk\neq 2$. Let $B\subseteq \P^2$ be a sum of a cubic $\qq$ with a cusp $p_1$, a line $\ll_1$ tangent to $\qq$ at $p_1$, and lines $\ll_2,\ll_3$ which meet at a point $p_0\in \ll_1$ and are tangent to $\qq$ at some $p_2,p_3\in \qq\reg$. A direct computation, see \cite[Lemma 5.5]{PaPe_MT}, shows that such a divisor $B$ is unique up to a projective equivalence, and for $j\in \{2,3\}$ we have $(\ll_j\cdot \qq)_{p_j}=2$ if $\cha\kk\neq 3$, and $(\ll_j\cdot \qq)_{p_j}=3$ if $\cha\kk=3$. 
	
	We define $\phi\colon Y\to \P^2$ as follows. Blow up once at $p_0$, three times at $p_1$, $p_2$ and $m$ times at $p_3$, each time on the proper transform of $\qq$. We get the following surfaces $\bar{Y}$.
	\begin{enumerate}
		\item\label{item:w=1_cha-neq-3} Assume $\cha\kk\neq 3$ and take $m=2$. Then $\cS(\bar{Y})=\rA_{1}+\rA_{2}+\rA_{5}+[3]$, see Figure \ref{fig:w=1_cha-neq-3}.
		\item\label{item:w=1_cha=3_GK} Assume $\cha\kk=3$ and take $m=2$. Then $\cS(\bar{Y})=3\rA_{2}+\rA_1+2\cdot [3]$, see Figure \ref{fig:w=1_cha=3_GK}.
		\item\label{item:w=1_cha=3_not-GK} Assume $\cha\kk=3$ and take $m=3$. Then $\cS(\bar{Y})=4\cdot [3]+3\rA_{2}$, see Figure \ref{fig:w=1_cha=3_not-GK}. 
	\end{enumerate}
\end{example}
\begin{figure}[ht]
	\subcaptionbox{$\rA_{1}+\rA_{2}+\rA_{5}+[3]$, $\cha\kk\neq 2,3$ \label{fig:w=1_cha-neq-3}}[.38\linewidth]{
		\begin{tikzpicture}
			\begin{scope}
				\draw (0.2,3.2) -- (0,2);
				\node at (0.35,2.7) {\small{$L_2$}};
				\draw (0,2.2) -- (0.2,1);
				\draw[dashed] (0,1.6) -- (1.2,1.8);
				\draw (0.2,1.2) -- (0,0);
				\draw (2.4,3.2) -- (2.2,2.2);
				\node at (2.55,2.7) {\small{$L_1$}};
				\draw (2.2,2.4) -- (2.4,1.4);
				\draw[dashed] (2.4,1.6) -- (2.2,0.6);
				\draw (2.2,0.8) -- (2.4,-0.2); 
				\node at (2,0.2) {\small{$-3$}};
				\draw (4,3.2) -- (3.8,2);
				\node at (4.15,2.7) {\small{$L_3$}};
				\draw[dashed] (3.8,2.2) -- (4,1);
				\draw (4,1.2) -- (3.8,0); 
				\draw (0.3,1.75) to[out=0,in=-170] (0.6,1.7) to[out=10,in=180] (1,2) -- (1.2,2) to[out=0,in=180] (2.3,1.1)  to[out=0,in=180] (3.2,1.7) -- (3.5,1.7) to[out=0,in=180] (3.9,1.6) to[out=0,in=180] (4.1,1.65);
				\draw (0.3,1.55) to[out=0,in=-160] (0.6,1.7) to[out=20,in=180] (1.4,0.9) -- (2,0.9) to[out=0,in=180] (2.3,1.1) to[out=0,in=180] (2.8,0.9) -- (3,0.9) to[out=0,in=180] (3.9,1.6) to[out=0,in=180] (4.1,1.55);
				\draw (0,3) -- (1,3) to[out=0,in=90] (2.1,1.6) to[out=-90,in=180] (2.3,1.1) to[out=0,in=-90] (2.6,1.6) to[out=90,in=180] (3.9,3) -- (4.1,3);
				\node at (1.5,1.5) {\small{$Q$}};
				\draw [->] (2,-0.5) -- (2,-1.5);
				\node at (2.2,-1) {\small{$\phi$}}; 
			\end{scope}
			\begin{scope}[shift={(2,-4)}]
				\draw[add=0.1 and 0.8] (0,2) to (0,0.5);
				\node at (0.2,0.5) {\small{$\ll_1$}};
				\filldraw (0,-0.5) circle (0.06);
				\node at (-0.2,-0.5) {\small{$p_1$}};	
				\draw[add= 0.1 and 0.6] (0,2) to (-0.866,0.5);
				\node at (-1.5,-0.1) {\small{$\ll_2$}};
				\filldraw (-0.866,0.5) circle (0.06);
				\node at (-1.1,0.6) {\small{$p_2$}};
				\draw[add= 0.1 and 0.6] (0,2) to (0.866,0.5);
				\node at (1.5,-0.1) {\small{$\ll_3$}};
				\filldraw (0.866,0.5) circle (0.06);
				\node at (1.1,0.6) {\small{$p_3$}};	
				\filldraw (0,2) circle (0.06);
				\node at (0.3,2) {\small{$p_0$}};
				\draw (-0.866,2) to[out=-60,in=60] (-0.3,1) to[out=-120,in=60] (-0.866,0.5) to[out=-120,in=90] (0,-0.5) to[out=90,in=-60] (0.866,0.5) to[out=120,in=-60] (0.3,1) to[out=120,in=-120]  (0.866,2);
				\node at (-0.75,-0.1) {\small{$\qq$}};
				\draw[-stealth] (-0.6,0.6) to[out=-150,in=90] (-0.8,0.4) to[out=-90,in=150] (-0.6,0.2);
				\draw[-stealth] (-0.45,0.6) to[out=-150,in=90] (-0.65,0.4) to[out=-90,in=150] (-0.45,0.2);
				\draw[stealth-] (0.6,0.6) to[out=-30,in=90] (0.8,0.4) to[out=-90,in=30] (0.6,0.2);
				\draw[-stealth] (0.1,-0.45) to[out=60,in=-150] (0.5,-0.15);
				\draw[-stealth] (0.15,-0.55) to[out=60,in=-150] (0.55,-0.25);
			\end{scope}
		\end{tikzpicture}
	}
	\subcaptionbox{$3\rA_{2}+\rA_1+2\cdot [3]$, $\cha\kk=3$ \label{fig:w=1_cha=3_GK}}[.3\linewidth]{
		\begin{tikzpicture}
			\begin{scope}
				\draw (0.2,3.2) -- (0,2.2);
				\node at (0.35,2.7) {\small{$L_2$}};
				\draw[dashed] (0,2.4) -- (0.2,1.4);
				\draw (0.2,1.6) -- (0,0.6);
				\draw (0,0.8) -- (0.2,-0.2); 
				\node at (-0.2,2.7) {\small{$-3$}};
				\draw (1.4,3.2) -- (1.2,2.2);
				\node at (1.55,2.7) {\small{$L_1$}};
				\draw (1.2,2.4) -- (1.4,1.4);
				\draw[dashed] (1.4,1.6) -- (1.2,0.6);
				\draw (1.2,0.8) -- (1.4,-0.2); 
				\node at (1,0.2) {\small{$-3$}};
				\draw (2.6,3.2) -- (2.4,2);
				\node at (2.75,2.7) {\small{$L_3$}};
				\draw[dashed] (2.4,2.2) -- (2.6,1);
				\draw (2.6,1.2) -- (2.4,0); 
				\draw[thick] (-0.2,1.9) -- (0.2,1.9) to[out=0,in=180] (1.2,1.1) -- (1.6,1.1) to[out=0,in=180] (2.4,2.1) -- (2.8,2.1);
				\node at (0.9,1.5) {\small{$Q$}};
				\draw [->] (1.4,-0.5) -- (1.4,-1.5);
				\node at (1.6,-1) {\small{$\phi$}}; 
			\end{scope}
			\begin{scope}[shift={(1.4,-4)}]
				\draw[add=0.1 and 0.8] (0,2) to (0,0.5);
				\node at (0.2,0.5) {\small{$\ll_1$}};
				\filldraw (0,-0.5) circle (0.06);
				\node at (-0.2,-0.5) {\small{$p_1$}};	
				\draw[add= 0.1 and 0.6] (0,2) to (-0.866,0.5);
				\node at (-1.5,-0.1) {\small{$\ll_2$}};
				\filldraw (-0.866,0.5) circle (0.06);
				\node at (-1.1,0.6) {\small{$p_2$}};
				\draw[add= 0.1 and 0.6] (0,2) to (0.866,0.5);
				\node at (1.5,-0.1) {\small{$\ll_3$}};
				\filldraw (0.866,0.5) circle (0.06);
				\node at (1.1,0.6) {\small{$p_3$}};	
				\filldraw (0,2) circle (0.06);
				\node at (0.3,2) {\small{$p_0$}};
				\draw (-0.866,2) to[out=-60,in=60] (-0.866,0.5) to[out=-120,in=90] (0,-0.5) to[out=90,in=-60] (0.866,0.5) to[out=120,in=-120]  (0.866,2);
				\node at (-0.75,-0.1) {\small{$\qq$}};
				\draw[-stealth] (-0.65,0.6) to[out=-120,in=90] (-0.8,0.4) to[out=-90,in=150] (-0.6,0.2);
				\draw[-stealth] (-0.5,0.6) to[out=-120,in=90] (-0.65,0.4) to[out=-90,in=150] (-0.45,0.2);
				\draw[stealth-] (0.65,0.6) to[out=-60,in=90] (0.8,0.4) to[out=-90,in=30] (0.6,0.2);
				\draw[-stealth] (0.1,-0.45) to[out=60,in=-150] (0.5,-0.15);
				\draw[-stealth] (0.15,-0.55) to[out=60,in=-150] (0.55,-0.25);
			\end{scope}
		\end{tikzpicture}
	}	
	\subcaptionbox{$4\cdot [3]+3\rA_{2}$, $\cha\kk=3$ \label{fig:w=1_cha=3_not-GK}}[.3\linewidth]{
		\begin{tikzpicture}
			\begin{scope}
				\draw (0.2,3.2) -- (0,2.2);
				\node at (0.35,2.7) {\small{$L_2$}};
				\draw[dashed] (0,2.4) -- (0.2,1.4);
				\draw (0.2,1.6) -- (0,0.6);
				\draw (0,0.8) -- (0.2,-0.2); 
				\node at (-0.2,2.7) {\small{$-3$}};
				\draw (1.4,3.2) -- (1.2,2.2);
				\node at (1.55,2.7) {\small{$L_1$}};
				\draw (1.2,2.4) -- (1.4,1.4);
				\draw[dashed] (1.4,1.6) -- (1.2,0.6);
				\draw (1.2,0.8) -- (1.4,-0.2); 
				\node at (1,0.2) {\small{$-3$}};
				\draw (2.6,3.2) -- (2.4,2.2);
				\node at (2.75,2.7) {\small{$L_3$}};
				\draw[dashed] (2.4,2.4) -- (2.6,1.4);
				\draw (2.6,1.6) -- (2.4,0.6);
				\draw (2.4,0.8) -- (2.6,-0.2); 
				\node at (2.2,2.7) {\small{$-3$}};
				\draw[thick] (-0.2,1.9) -- (0.2,1.9) to[out=0,in=180] (1.2,1.1) -- (1.6,1.1) to[out=0,in=180] (2.4,1.9) -- (2.8,1.9);
				\node at (0.5,1.25) {\small{$-3$}};
				\node at (0.9,1.5) {\small{$Q$}};
				\draw [->] (1.4,-0.5) -- (1.4,-1.5);
				\node at (1.6,-1) {\small{$\phi$}}; 
			\end{scope}
			\begin{scope}[shift={(1.4,-4)}]
				\draw[add=0.1 and 0.8] (0,2) to (0,0.5);
				\node at (0.2,0.5) {\small{$\ll_1$}};
				\filldraw (0,-0.5) circle (0.06);
				\node at (-0.2,-0.5) {\small{$p_1$}};	
				\draw[add= 0.1 and 0.6] (0,2) to (-0.866,0.5);
				\node at (-1.5,-0.1) {\small{$\ll_2$}};
				\filldraw (-0.866,0.5) circle (0.06);
				\node at (-1.1,0.6) {\small{$p_2$}};
				\draw[add= 0.1 and 0.6] (0,2) to (0.866,0.5);
				\node at (1.5,-0.1) {\small{$\ll_3$}};
				\filldraw (0.866,0.5) circle (0.06);
				\node at (1.1,0.6) {\small{$p_3$}};	
				\filldraw (0,2) circle (0.06);
				\node at (0.3,2) {\small{$p_0$}};
				\draw (-0.866,2) to[out=-60,in=60] (-0.866,0.5) to[out=-120,in=90] (0,-0.5) to[out=90,in=-60] (0.866,0.5) to[out=120,in=-120]  (0.866,2);
				\node at (-0.75,-0.1) {\small{$\qq$}};
				\draw[-stealth] (-0.65,0.6) to[out=-120,in=90] (-0.8,0.4) to[out=-90,in=150] (-0.6,0.2);
				\draw[-stealth] (-0.5,0.6) to[out=-120,in=90] (-0.65,0.4) to[out=-90,in=150] (-0.45,0.2);
				\draw[stealth-] (0.65,0.6) to[out=-60,in=90] (0.8,0.4) to[out=-90,in=30] (0.6,0.2);
				\draw[stealth-] (0.5,0.6) to[out=-60,in=90] (0.65,0.4) to[out=-90,in=30] (0.45,0.2);
				\draw[-stealth] (0.1,-0.45) to[out=60,in=-150] (0.5,-0.15);
				\draw[-stealth] (0.15,-0.55) to[out=60,in=-150] (0.55,-0.25);
			\end{scope}
		\end{tikzpicture}
	}	
	\caption{Example \ref{ex:w=1}: vertically primitive surfaces in Proposition  \ref{prop:ht=3_swaps}\ref{item:swap_ht=3,w=1_cha=3_GK},\ref{item:swap_cha=3}: $\ww=1$.}
	\label{fig:w=1}
\end{figure}

\begin{remark}[Alternative construction of the surface from Example \ref{ex:w=1}\ref{item:w=1_cha=3_not-GK}]\label{rem:Bernasconi}
	The surface $\bar{Y}$ from Example \ref{ex:w=1}\ref{item:w=1_cha=3_not-GK} has been first constructed by Bernasconi in \cite{Bernasconi_char_3-example} as a counterexample to the Kawamata--Viehweg vanishing in characteristic~$3$. In fact, Theorem 3.6 loc.\ cit.\ shows that $H^{1}(\bar{Y},\cO_{\bar{Y}}(-\bar{A}))\neq 0$ for the ample divisor  $\bar{A}=\bar{A}_1+\bar{A}_2-\bar{A}_3$, where $\bar{A}_j$ is the image of the $(-1)$-curve over $p_j$.
	
	We now recall the construction from \cite[\sec 3.1]{Bernasconi_char_3-example} and show that it indeed yields a surface isomorphic to $\bar{Y}$ from Example \ref{ex:w=1}\ref{item:w=1_cha=3_not-GK}. This construction will be useful in the proof of Lemma \ref{lem:w=1_cha=3_not-GK_Aut}. Fix coordinates $[x:y]$ and $[z:w]$ on each factor of $\P^1\times \P^1$, and put $H=\{xz^{3}=yw^3\}\subseteq \P^1\times \P^1$. Since $\cha\kk=3$, the curve $H$ is tangent with multiplicity $3$ to each fiber of the  first projection $\P^1\times \P^1\to \P^1$. Fix three such fibers, say $F_1,F_2,F_3$, let $(X,\tilde{D})\to (\P^1\times \P^1,H+F_1+F_2+F_3)$ be the minimal log resolution, and let $X\to \bar{X}$ be the contraction of $D\de \tilde{D}-A_1-A_2-A_3$, where $A_j$ is the $(-1)$-curve in the preimage of $F_j$. Now $\bar{X}$ is the surface from \cite{Bernasconi_char_3-example}. To see that it is isomorphic to the one from Example \ref{ex:w=1}\ref{item:w=1_cha=3_not-GK}, let $\sigma\colon \P^1\times \P^1\map \P^2$ be a blowup at the point $F_1\cap H$, followed by the contraction of the proper transform of $F_1$ and of the horizontal line passing through $F_1\cap H$. Then $\sigma_{*}(H+\sum_{i}F_i)$ is projectively equivalent to $\qq+\sum_{i}\ll_i$, so $\sigma$ lifts to the required isomorphism.
\end{remark}

We now summarize basic properties of surfaces $\bar{Y}$ constructed above. Most importantly, we note that each $\bar{Y}$ is unique up to an isomorphism, except in Examples \ref{ex:w=2_cha=2}\ref{item:swap-to-nu=4_small},\ref{item:swap-to-nu=4_large}, where the choice of the fourth line yields a one-parameter family. This observation will be crucial for the proof of the uniqueness part of Theorem \ref{thm:ht=3}. Proposition \ref{prop:primitive}\ref{item:primitive-uniqueness} below is a slightly stronger version of this statement, formulated using the language of \cite[\sec 2F]{PaPe_ht_2} which will be convenient later.

Furthermore, we note that even though the $\P^1$-fibrations $p_Y$ constructed above are vertically primitive, i.e.\ do not admit vertical swaps, the surfaces $\bar{Y}$ itself may be non-primitive, i.e.\ $(Y,D_Y)$ may admit swaps which do not preserve the chosen $\P^1$-fibrations $p_Y$, see \cite[Definition 1.3]{PaPe_ht_2}. Similarly, we may have $\height(\bar{Y})=2$, so $p_Y$ may be a non-witnessing $\P^1$-fibration. Such cases are listed in  Proposition \ref{prop:primitive}\ref{item:primitive} and \ref{item:primitive-ht}, respectively.

\begin{proposition}[{Properties of vertically primitive models from Proposition~\ref{prop:ht=3_swaps}}]\label{prop:primitive}
	Let $\bar{Y}$ be a surface as in one of Examples \ref{ex:w=3}--\ref{ex:w=1} above. Then the following hold.
	\begin{enumerate}
		\item\label{item:primitive-deb} The surface $\bar{Y}$ has a descendant with elliptic boundary, unless it is as in Example \ref{ex:w=1}\ref{item:w=1_cha=3_not-GK}. 
		\item\label{item:primitive} The surface $\bar{Y}$ is primitive, see \cite[Definition 1.3(e)]{PaPe_ht_2}, 
		unless it is as in in Examples \ref{ex:w=2_cha_neq_2}\ref{item:ht=3_exception} or \ref{ex:w=2_cha=2}\ref{item:swap-to-nu=4_large},\ref{item:swap-to-nu=3}.
		\item\label{item:primitive-ht} We have $\height(\bar{Y})=2$ in Examples \ref{ex:w=3}, \ref{ex:w=2_cha_neq_2}\ref{item:w=2_A1+A2+A5},\ref{item:w=2_2A1+2A3} and \ref{ex:w=2_cha=2}\ref{item:swap-to-nu=3}. In all the other cases we have $\height(\bar{Y})=3$.
		\item\label{item:primitive-uniqueness}  Let $\check{D}_Y$ be the sum of $D_Y$ and all vertical $(-1)$-curves except those meeting $D$ in a node (i.e.\ except the one over $p_1$ in Example \ref{ex:w=2_cha=2}\ref{item:swap-to-nu=4_small} and $p_3$ in Example \ref{ex:w=1}\ref{item:w=1_cha=3_not-GK}). Let $\check{\cS}$ be the combinatorial type of $(Y,\check{D}_Y)$. Then $\#\cP_{+}(\check{\cS})=1$, unless $\check{\cS}$ is as in Example \ref{ex:w=2_cha=2}\ref{item:swap-to-nu=4_small},\ref{item:swap-to-nu=4_large}. In the latter case $\cP_{+}(\check{\cS})$ is represented by an $\Aut(\check{\cS})$-faithful universal family over $\Astst$, parametrized by the choice of the fourth line $\ll_{4}$.
		\item\label{item:primitive-uniqueness-Y} The isomorphism class of $\bar{Y}$ is uniquely determined by the singularity type of $\bar{Y}$, unless $\bar{Y}$ is as in Example \ref{ex:w=2_cha=2}\ref{item:swap-to-nu=4_small} or \ref{ex:w=2_cha=2}\ref{item:swap-to-nu=4_large}. In the latter cases, the set of isomorphism classes of surfaces $\bar{Y}$ has moduli dimension $1$.
		\item\label{item:primitive-hi} Put $h^i\de h^i(\lts{Y}{D_Y})$. Then $h^0=0$ and the following hold.
		\begin{enumerate}
			\item $h^1=1$ in Examples \ref{ex:w=2_cha=2}\ref{item:swap-to-nu=4_small},\ref{item:swap-to-nu=4_large}; and $h^1=0$ otherwise.
			\item $h^2=2$ in Examples  \ref{ex:w=2_cha=2}\ref{item:swap-to-nu=4_large} and  \ref{ex:w=1}\ref{item:w=1_cha=3_not-GK}; $h^2=1$ in Examples  \ref{ex:w=2_cha=2}\ref{item:swap-to-nu=4_small},\ref{item:swap-to-nu=3} and 	
			\ref{ex:w=1}\ref{item:w=1_cha=3_GK}; and $h^2=0$ otherwise.
		\end{enumerate}
		In particular, if $\cha\kk\neq 2,3$ then $h^i=0$ for all $i\geq 0$.
	\end{enumerate}
\end{proposition}
\begin{proof} 
	\ref{item:primitive-deb}
	The surface from Example \ref{ex:w=1}\ref{item:w=1_cha=3_not-GK} has no descendant with elliptic boundary because it has more than two non-canonical singularities, which is excluded in \cite[Theorem E(c)]{PaPe_ht_2}. In the remaining cases we need to find an elliptic tie, see Definition \ref{def:tie}. We check directly that each of the following is an elliptic tie:
	\begin{itemize}
		\item In Example \ref{ex:w=3}\ref{item:ht=3_A1+A2+A5}: the $(-1)$-curve over $p_{32}$, see Figure \ref{fig:w=3_A1+A2+A5}.
		\item In Example \ref{ex:w=3}\ref{item:ht=3_2A4}: the $(-1)$-curve over $p_{22}$, see Figure \ref{fig:w=3_2A4}.
		\item In Example \ref{ex:w=2_cha_neq_2}\ref{item:ht=3_exception}: the proper transform of the unique conic $\cc'$ such that $(\cc'\cdot \ll_{1})_{p_{1}}=2$, $(\cc'\cdot \cc)_{p_{3}}=3$ (we note that in this case, we get the surface from \cite[Proposition 6.4(a)]{PaPe_ht_2}, with $\cS_Y=\rA_1+\rA_7$).
		\item In Examples \ref{ex:w=2_cha_neq_2}\ref{item:w=2_A1+A2+A5},\ref{item:w=2_2A1+2A3} and \ref{ex:w=2_cha=2}: the $(-1)$-curve over $p_1$, see Figures \ref{fig:w=2_A1+A2+A5}, \ref{fig:w=2_2A1+2A3} and \ref{fig:w=2_cha=2}.
		\item In  Example \ref{ex:w=1}\ref{item:w=1_cha-neq-3}: the proper transform of a cubic $\cc$ with a double point at $p_1$, such that $(\cc\cdot \qq)_{p_2}=3$ and $(\cc\cdot \qq)_{p_{3}}=2$, see Figure \ref{fig:w=1_cha-neq-3}. The existence of $\cc$ is proved below (in this case, we get the surface from \cite[Proposition 6.4(a)]{PaPe_ht_2}, with $\cS_Y=\rA_1+\rA_2+\rA_5$).
		\item In  Example \ref{ex:w=1}\ref{item:w=1_cha=3_GK}: the proper transform of the line joining $p_1$ with $p_3$, see Figure \ref{fig:w=1_cha=3_GK} (in this case, we get the surface from \cite[Proposition 6.4(b)]{PaPe_ht_2}, with $\cS_Y=3\rA_2$).
	\end{itemize} 
	It remains to construct the cubic $\cc$ used in the case \ref{ex:w=1}\ref{item:w=1_cha-neq-3} above. Fix four points $q_0,r_1,r_2,r_3\in \P^2$ in a general position. For $i,j\in \{1,2,3\}$ let $\bar{\ll}_i,\bar{\ll}_{ij}$ be the lines joining $r_i$ with $q_0$ and $r_j$, respectively. For $\{i,j,k\}=\{1,2,3\}$ write $\{q_i\}=\bar{\ll}_i\cap \bar{\ll}_{jk}$. Let $\bar{\qq}$ be the unique conic passing through $r_1$, $q_2$, $q_3$ and tangent to $\bar{\ll}_{23}$ at $q_1$, and let $\bar{\cc}$ be the unique conic passing through $r_1$, $q_0$, $q_3$ and tangent to $\bar{\qq}$ at $q_2$. Recall that in Example \ref{ex:w=1}\ref{item:w=1_cha-neq-3} we assume $\cha\kk\neq 2,3$: this assumption implies that the conics $\bar{\cc}$, $\bar{\qq}$ are smooth. Indeed, in the coordinates $[x:y:z]$ on $\P^2$ such that $q_0=[1:1:1]$,  $r_1=[1:0:0]$, $r_2=[0:1:0]$, $r_3=[0:0:1]$, we have $\bar{\qq}=\{(y-z)^2=x(y+z)\}$ and $\bar{\cc}=\{y(y-x)=3z(x-z)\}$. The first conic is smooth since $\cha\kk\neq 2$, and the second one since $\cha\kk\neq 2,3$. Let $\sigma\colon \P^2\map \P^2$ be the standard quadratic transformation centered at $r_1,r_2,r_3$, i.e.\ we blow up those points and blow down the proper transforms of the lines joining them. Put $\ll_i=\sigma_{*}\bar{\ll}_i$, $\cc=\sigma_{*}\bar{\cc}$, $\qq=\sigma_{*}\bar{\qq}$. Then $(\P^2,\qq+\ll_1+\ll_2+\ll_3)$ is as in Example  \ref{ex:w=1}\ref{item:w=1_cha-neq-3}, and $\cc$ is the required cubic. 
	\smallskip

	\ref{item:primitive} The surface $\bar{Y}$ from Examples \ref{ex:w=2_cha_neq_2}\ref{item:ht=3_exception} and  
	\ref{ex:w=2_cha=2}\ref{item:swap-to-nu=4_large},\ref{item:swap-to-nu=3} is not primitive because we can swap the $(-1)$-curve over $p_0$. For the remaining surfaces $\bar{Y}$, primitivity follows from  \cite[Remark 5.7]{PaPe_ht_2} if $\bar{Y}$ is canonical, and from Proposition 6.15 loc.\ cit.\ if it is non-canonical but has a descendant with elliptic boundary. Eventually, the surface $\bar{Y}$ from Example \ref{ex:w=1}\ref{item:w=1_cha=3_not-GK} is obviously primitive: indeed, if $L\subseteq Y$ was a $(-1)$-curve as in \cite[Definition 1.3(a)]{PaPe_ht_2} then $L+D_Y$ would be nonpositive definite, which is impossible.
	\smallskip

	\ref{item:primitive-ht} Clearly, we have $\height(\bar{Y})\leq 3$. An elementary exercise, or an application of \cite[Lemma 4.1]{PaPe_ht_2}, shows that $\height(\bar{Y})>1$. Witnessing $\P^1$-fibrations of height $2$ are obtained by pulling back the following pencils: 
	\begin{itemize}
		\item In  Example \ref{ex:w=1}\ref{item:ht=3_A1+A2+A5}: second projection $\P^1\times \P^1\to\P^1$. 
		\item In  Example \ref{ex:w=1}\ref{item:ht=3_2A4}: pencil of curves of type $(1,1)$ passing through $p_{22}$, $p_{32}$. 
		\item In  Examples \ref{ex:w=2_cha_neq_2}\ref{item:w=2_A1+A2+A5},\ref{item:w=2_2A1+2A3} or \ref{ex:w=2_cha=2}\ref{item:swap-to-nu=3} pencil of conics which are tangent to  $\ll_3$ at $p_3$ and 
		\begin{itemize}
			\item pass through $p_1,p_2$ in  Example \ref{ex:w=2_cha_neq_2}\ref{item:w=2_A1+A2+A5}, 
			\item are tangent to $\ll_2$ at $p_2$ in  Examples \ref{ex:w=2_cha_neq_2}\ref{item:w=2_2A1+2A3},  \ref{ex:w=2_cha=2}\ref{item:swap-to-nu=3}.
		\end{itemize}
	\end{itemize}	
	In the remaining cases when $\bar{Y}$ has a descendant with elliptic boundary we have $\height(\bar{Y})\geq 3$ by \cite[Lemma 6.3]{PaPe_ht_2}. By \ref{item:primitive-deb} it remains to prove that $\height(\bar{Y})\geq 3$ in  Example    \ref{ex:w=1}\ref{item:w=1_cha=3_not-GK}. Suppose the contrary, so $\height(\bar{Y})\leq 2$. Then some degenerate fiber $F$ of a witnessing $\P^1$-fibration contains a $(-3)$-curve. By Lemma \ref{lem:fibrations-Sigma-chi} $F\redd$ is a sum of some components of $D_Y$ and at most two $(-1)$-curves, so $F\redd=[3,1,2,2]$ or $[1,3,1,2]$. We have seen in \ref{item:primitive} that $\bar{Y}$ is primitive, so each $(-1)$-curve in $F$ meets a component of $(D_Y)\hor$. Thus  $F\cdot D_Y\geq 3$, a contradiction.
	\smallskip

	\ref{item:primitive-uniqueness} Fix an element of $\cP_{+}(\check{\cS})$, i.e.\ a log surface $(Y,\check{D}_Y)$ together with the fixed order of components of $\check{D}_Y$, see \cite[\sec 2F]{PaPe_ht_2}. Then a subdivisor of $\check{D}_Y$ supports a fiber of a $\P^1$-fibration as above. Thus we get a morphism $\phi\colon (Y,\check{D}_Y)\to (Z,B)$ as above such that the center of each blowup in the decomposition of $\phi$ is uniquely determined by $(Z,B)$ and the fixed order of components of $B$. If $\bar{Y}$ is not as in Example \ref{ex:w=2_cha=2}\ref{item:swap-to-nu=4_small},\ref{item:swap-to-nu=4_large}, then the latter is unique up to an isomorphism 
	so by the universal property of blowing up we get $\#\cP_{+}(\check{\cS})=1$, as needed. 
	
	Consider $\bar{Y}$ as in  Example \ref{ex:w=2_cha=2}\ref{item:swap-to-nu=4_small},\ref{item:swap-to-nu=4_large}. We have seen in \ref{item:primitive-deb} that $\bar{Y}$ admits a descendant with elliptic boundary of type $3\rA_1+\rD_4$. Hence the required family is constructed in \cite[Lemma 6.14 and Proposition 6.17(b)]{PaPe_ht_2}. To be precise, loc.\ cit.\ yields families representing the sets of isomorphism classes of log surfaces $(Y,D_Y)$, but it is clear from their construction that their total spaces contain divisors corresponding to the $(-1)$-curves $\check{D}_Y-D_Y$.

	\ref{item:primitive-uniqueness-Y} This follows from \ref{item:primitive-uniqueness} and \cite[Lema 2.20(c)]{PaPe_ht_2}. Alternatively, in all cases except \ref{ex:w=1}\ref{item:w=1_cha=3_not-GK} parts \ref{item:primitive-deb}, \ref{item:primitive-ht} show that  $\height(\bar{Y})\leq 2$ or $\bar{Y}$ has a descendant with elliptic boundary, so \ref{item:primitive-uniqueness-Y} follows from the classification \cite{PaPe_ht_2}. 
	
	\ref{item:primitive-hi} If $\height(\bar{Y})=2$ or $\height(\bar{Y})=3$ and $\bar{Y}$ has a descendant with elliptic boundary then the numbers $h^i$ are computed in \cite[Proposition D and Theorem E]{PaPe_ht_2}. Thus by  \ref{item:primitive-deb} and \ref{item:primitive-ht} we can assume that $\bar{Y}$ is as in Example  \ref{ex:w=1}\ref{item:w=1_cha=3_not-GK}. Let $A=[1]$, $G=[2]$ be the third and the second exceptional curve over $p_3$, and let $\sigma\colon (Y,D_Y-G)\to (Y',D_{Y'})$ be the contraction of $A$, so $(Y',D_{Y'})$ is as in Example  \ref{ex:w=1}\ref{item:w=1_cha=3_GK}, see Figure \ref{fig:w=1}. The blowup $\sigma$ is inner, so $h^{i}(\lts{Y}{(D_Y-G)})=h^{i}(\lts{Y'}{D_{Y'}})$, see \cite[Lemma 2.11(a),(b)]{PaPe_ht_2}. Thus $h^{i}(\lts{Y}{(D_Y-G)})=0$ for $i=0,1$ and $h^{2}(\lts{Y}{(D_Y-G)})=1$. We have an exact sequence $0\to\lts{Y}{D_Y}\to \lts{Y}{(D_Y-G)}\to \cN_{G}\to 0$, where $\cN_{G}$ is the normal bundle to $G$ in $Y$, see Lemma 2.12(a) loc.\ cit. Since $h^{1}(\cN_{G})=1$ and $h^i(\cN_G)=0$ for $i\neq 1$, we get $h^{i}=0$ for $i=0,1$ and $h^{2}=2$, as claimed.
\end{proof}

\subsection{Automorphisms of primitive surfaces in case $\ww=1$, $\cha\kk=3$}

The last result of this section describes the automorphism groups of surfaces from Example \ref{ex:w=1}\ref{item:w=1_cha=3_GK},\ref{item:w=1_cha=3_not-GK}. Some of these automorphisms will be used in the proof of Lemma \ref{lem:w=1_cha=3}, i.e.\ in case $\width(\bar{X})=1$, $\cha\kk=3$ of Theorem~\ref{thm:ht=3}.

\begin{lemma}[$\Aut(\textnormal{\ref{ex:w=1}\ref{item:w=1_cha=3_GK}})\cong \rD_{6}$]\label{lem:w=1_cha=3_GK_Aut}
	Let $\bar{Y}$ be the surface from Example \ref{ex:w=1}\ref{item:w=1_cha=3_GK}, and let $(Y,D_Y)$ be its minimal log resolution. Then $\Aut(\bar{Y})=\Aut(Y,D_Y)$ is isomorphic to the dihedral group of order $12$, and is generated by the following automorphisms, cf.\ Remark \ref{rem:hexagon}:
	\begin{enumerate}
			\item\label{item:Aut_rotation} $(G_1,G_1',G_2,G_2',H,H',T_1,T_2)\mapsto
			(H,H',G_1',G_1,G_2,G_2',T_2,T_1)$,
			\item\label{item:Aut_symmetry} $
			(G_1,G_1',G_2,G_2',H,H',T_1,T_2)\mapsto
			(G_2',G_2,G_1',G_1,H',H,T_1,T_2)$,
	\end{enumerate}
	where $H=\phi^{-1}_{*}\qq$, $H+H'=[2,2]$; and for $i\in \{1,2\}$ we have $(\phi^{*}\ll_i)\redd=[3,1,2,2]= T_i+A_i+G_{i}+G_{i}'$.
\end{lemma}
\begin{proof}
	Let $L\subseteq Y$ be the proper transform of the line joining $p_1$ with $p_3$, and let $G$ be the connected component of $D_Y$ of type $[2]$. For every automorphism $\alpha\in \Aut(Y,D_Y)$ we have $\alpha(G)=G$ and $\alpha(\{T_1,T_2\})=\{T_1,T_2\}$, so $\alpha(L)$ is a $(-1)$-curve meeting $D_Y$ normally in $G$, $T_1$ and $T_2$. Since $K_Y$ and the components of $D_Y$ generate the group $\NS_{\Q}(Y)$, we get $\alpha(L)\equiv L$, hence $\alpha(L)\cdot L=L^2<0$ and therefore $\alpha(L)=L$. Thus every automorphism of $(Y,D_Y)$ fixes $L$ and $G$. It follows that $\Aut(Y,D_Y)$ is a lift of the stabilizer subgroup of $\Aut(Z,D_Z)$ with respect to the point $\gamma(L)$, where $\gamma\colon (Y,D_Y-G)\to (Z,D_Z)$ is the contraction of $L$. 
	
	The log surface $(Z,D_Z)$ is the minimal log resolution of the canonical del Pezzo surface $\bar{Z}$ of rank one and type $4\rA_2$. An explicit construction of $\bar{Z}$ is given for instance in \cite[Example 7.1]{PaPe_MT}. The curves $G_1\cp{1}$, $G_{1}\cp{2}$, $G_2\cp{2}$, $G_2\cp{1}$, $L_1$, $L_2$, $Q$ and $R$ from loc.\ cit.\ can be identified with the images of $G_{1}$, $G_{1}'$, $G_{2}$, $G_{2}'$, $H$, $H'$, $T_1$ and $T_2$, respectively (note the reverse order of $G_2\cp{j}$). Under this identification we have $\gamma(L)=Q\cap R$. The stabilizer of the point $\gamma(L)$ in $\Aut(Z,D_Z)$ is generated by the automorphisms (i), (ii), (iv) from \cite[Lemma 7.2(d)]{PaPe_MT}. They lift to $\sigma_{\textnormal{\ref{item:Aut_rotation}}}^{3}$, $\sigma_{\textnormal{\ref{item:Aut_rotation}}}^{-2}$,  $\sigma_{\textnormal{\ref{item:Aut_symmetry}}}$, where $\sigma_{\textnormal{\ref{item:Aut_rotation}}}$, $\sigma_{\textnormal{\ref{item:Aut_symmetry}}}$ are the automorphisms described in \ref{item:Aut_rotation}, \ref{item:Aut_symmetry}, respectively.
\end{proof}

\begin{remark}[Visualizing $\Aut(\textnormal{\ref{ex:w=1}\ref{item:w=1_cha=3_GK}})$, see Figure \ref{fig:hexagon}]\label{rem:hexagon}
	One can view the isomorphism $\Aut(Y,D_Y)\cong \rD_{6}$ from Lemma \ref{lem:w=1_cha=3_GK_Aut} as follows, see Figure \ref{fig:hexagon}. Consider a regular hexagon with vertices labeled subsequently by $G_1,H,G_2,G_{1}',H',G_{2}'$, and edges labeled alternatively by $T_1$, $T_2$. Now the automorphisms from Lemma \ref{lem:w=1_cha=3_GK_Aut}\ref{item:Aut_rotation} and \ref{item:Aut_symmetry} are, respectively, the $60^{\circ}$ rotation and a symmetry through the axis perpendicular to $HH'$.  
\begin{figure}[htbp]
	\begin{tikzpicture}
		\begin{scope}[scale=1.2]
			\draw (0.2,3.2) -- (0,2.2);
			\node at (0.35,2.8) {\small{$T_1$}};
			\node at (0.12,2.8) {\tiny{$\spadesuit$}};
			\node at (0.07,2.55) {\tiny{$\diamondsuit$}};
			\node at (-0.15,2.8) {\small{$-3$}};
			\draw[dashed] (0,2.4) -- (0.2,1.4);
			\node at (0.3,2.1) {\small{$A_1$}};
			\draw (0.2,1.6) -- (0,0.6);
			\node at (0.35,1.1) {\small{$G_1$}};
			\node at (0.12,1.2) {\small{$\circ$}};
			\node at (0.08,1) {\tiny{$\clubsuit$}};
			\draw (0,0.8) -- (0.2,-0.2); 
			\node at (0.35,0.3) {\small{$G_1'$}};
			\node at (0.08,0.4) {\tiny{$\spadesuit$}};
			\node at (0.13,0.15) {\tiny{$\heartsuit$}};
			\node at (0.18,-0.1) {$\bullet$};
			\draw (1.8,3.2) -- (1.6,2.2);
			\node at (1.95,2.8) {\small{$T_2$}};
			\node at (1.72,2.8) {\tiny{$\heartsuit$}};
			\node at (1.67,2.55) {\tiny{$\clubsuit$}};
			\node at (1.45,2.8) {\small{$-3$}};
			\draw[dashed] (1.6,2.4) -- (1.8,1.4);
			\node at (1.9,2.1) {\small{$A_2$}};
			\draw (1.8,1.6) -- (1.6,0.6);
			\node at (1.95,1.1) {\small{$G_2$}};
			\node at (1.72,1.2) {\small{$\circ$}};
			\node at (1.68,1) {\tiny{$\spadesuit$}};
			\draw (1.6,0.8) -- (1.8,-0.2); 
			\node at (1.95,0.3) {\small{$G_2'$}};
			\node at (1.68,0.4) {\tiny{$\diamondsuit$}};
			\node at (1.73,0.15) {\tiny{$\clubsuit$}};
			\node at (1.78,-0.1) {$\bullet$};
			\draw (3.4,3) -- (3.2,1.8);
			\node at (3.55,2.4) {\small{$H'$}};
			\node at (3.36,2.76) {\tiny{$\heartsuit$}};
			\node at (3.32,2.52) {\tiny{$\diamondsuit$}};
			\node at (3.28,2.28) {\small{$\circ$}};
			\draw[dashed] (3.2,2) -- (3.4,0.8);
			\node at (3.28,1.52) {\tiny{$\spadesuit$}};
			\node at (3.32,1.28) {\small{\tiny{$\clubsuit$}}};
			\node at (3.55,1.3) {\small{$A_3$}};
			\draw (3.4,1) -- (3.2,-0.2);
			\node at (3.5,0.3) {\small{$G$}};
			\node at (3.32,0.52) {\small{$\circ$}};
			\node at (3.28,0.28) {$\bullet$};
			\draw (-0.6,1.9) -- (3.5,1.9);
			\node at (-0.4,2.1) {\small{$H$}};
			\node at (0.9,1.9) {$\bullet$};
			\draw[dashed] (0,3.1) -- (3.8,3.1) to[out=0,in=90] (4,3) --(4,0.2) to[out=-90,in=0] (3.8,0) -- (3,0);
			\node at (4.2,1.3) {\small{$L$}};
			\draw[dashed] (-0.6,1.75) -- (0.2,1.75) to[out=0,in=180] (1.5,0.7) -- (2.5,0.7) to[out=0,in=-150] (3.383,0.9) -- (3.583,1);
			\node at (-0.4,1.55) {\small{$V_1$}};
			\node at (2.4,0.7) {\tiny{$\diamondsuit$}};
			\draw[dashed] (-0.6,0.7) -- (0.4,0.7) to[out=0,in=-120]  (0.87,0.94);
			\draw[dashed] (1.03,1.26) to[out=60,in=180] (1.5,1.75) -- (2.5,1.75) to[out=0,in=120] (3.383,0.9) -- (3.483,0.7);
			\node at (-0.4,0.9) {\small{$V_2$}};
			\node at (2.4,1.75) {\tiny{$\heartsuit$}};
		\end{scope}
		\begin{scope}[shift={(10,1.7)}, scale=0.8]
			\coordinate (H) at (-2,0);
			\coordinate (G2) at (-1,-1.732);
			\coordinate (G1') at (1,-1.732);
			\coordinate (H') at (2,0);
			\coordinate (G2') at (1,1.732);
			\coordinate (G1) at (-1,1.732);
			\draw (H) -- (G2);
			\draw[thick] (G2) -- (G1');
			\draw (G1') -- (H');
			\draw[thick] (H') -- (G2');
			\draw (G2') -- (G1);
			\draw[thick] (G1) -- (H);
			\filldraw (H) circle (0.1);
			\draw (G2) circle (0.1);
			\filldraw (G1') circle (0.1);
			\draw (H') circle (0.1);
			\filldraw (G2') circle (0.1);
			\draw (G1) circle (0.1);
			\node[left] at (H) {\small{$H$}};
			\node[right] at (H') {\small{$H'$}};
			\node[above left] at (G1) {\small{$G_1$}};
			\node[above right] at (G2') {\small{$G_2'$}};
			\node[below left] at (G2) {\small{$G_2$}};
			\node[below right] at (G1') {\small{$G_1'$}};
			\draw (H) -- (H');
			\node at ($(H)+(1,0.2)$) {\small{$A_3$}};
			\draw (G1) -- (G1');
			\node at ($(G1)+(0.6,-0.6)$) {\small{$V_2$}};
			\draw (G2) -- (G2');
			\node at ($(G2)+(0.65,0.6)$) {\small{$V_1$}};
			\node at (0.3,0.15) {\tiny{$G$}};
			\filldraw (0,0) circle (0.04);
			\node at (-1.7,1.1) {\small{$A_1$}};
			\node at (-1.2,0.8) {\tiny{$T_1$}};
			\node at (-1.7,-1.1) {\small{$A_2$}};
			\node at (-1.2,-0.8) {\tiny{$T_2$}};
			\node at (1.7,-1.1) {\small{$\heartsuit$}};
			\node at (1.2,0.8) {\tiny{$T_1$}};
			\node at (1.7,1.1) {\small{$\diamondsuit$}};
			\node at (1.2,-0.8) {\tiny{$T_2$}};
			\node at (0,2) {\small{$\clubsuit$}};
			\node at (0,1.5) {\tiny{$T_2$}};
			\node at (0,-2) {\small{$\spadesuit$}};
			\node at (0,-1.5) {\tiny{$T_1$}};
		\end{scope}
	\end{tikzpicture}
	\caption{Remark \ref{rem:hexagon}: the automorphism group of the surface from Example \ref{ex:w=1}\ref{item:w=1_cha=3_GK} viewed as the symmetry group of a regular hexagon.}
	\label{fig:hexagon}
\end{figure}
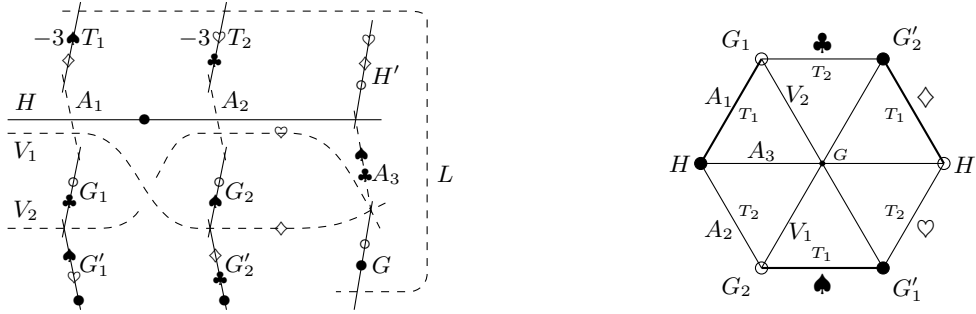	

	This identification leads to a convenient description of the action of $\Aut(Y,D_Y)$ on the set of $(-1)$-curves on $Y$. The $(-1)$-curve $L$ is fixed by $\Aut(Y,D_Y)$. The $(-1)$-curve $A_1$ meets $D_Y$ on $T_1$, $G_1$ and $H$. Acting by $\Aut(Y,D_Y)$, we see that for every edge of the hexagon, there is a $(-1)$-curve on $Y$ meeting $D_Y$ exactly three times, on the $(-3)$-curves corresponding to that edge, and the $(-2)$-curves corresponding to its vertices. Those $(-1)$-curves are: $A_1,A_2$ and, say, $A_{\spadesuit}$, $A_{\heartsuit}$, $A_{\diamondsuit}$ and $A_{\clubsuit}$: in Figure \ref{fig:hexagon}, we mark the common points of $A_{\star}$ with $D_Y$ by $\star\in \{\spadesuit,\heartsuit,\diamondsuit,\clubsuit\}$. Note that $A_{\heartsuit}=\phi^{-1}(p_0)$. Those $(-1)$-curves are pairwise disjoint since $A_1\cap A_2=\emptyset$. 
	
	Similarly, the orbit of $A_3$ is  $\{A_3,V_1,V_2\}$, where for $i\in \{1,2\}$ the $(-1)$-curve $V_i$ satisfies $V_i\cdot A_i=1$, $V_i\cdot D_Y=3$ and passes through $G_{3-i}\cap G_{3-i}'$ and $A_3\cap G$. To see them in the hexagon,  label the diagonals $HH'$ and $G_iG_i'$ by $A_3$ and $V_i$, and label their common point by $G$. Now, each of those three $(-1)$-curves meets three $(-2)$-curves whose labels lie on the chosen diagonal; and two $(-1)$-curves corresponding to edges parallel to it. 
	
	Eventually, let $A_{\bullet}$ and $A_{\circ}$ be the proper transforms of, respectively, the line joining $p_2$ with $p_3$ and the conic passing through $p_1$ and tangent to $\ll_i$ at $p_i$, $i\in \{2,3\}$. Then $\{A_{\circ},A_{\bullet}\}$ is another orbit of $\Aut(Y,D_Y)$. To see it in the hexagon, we label the vertices $H,G_1',G_{2}'$ by \enquote{$\bullet$} and the remaining ones by  \enquote{$\circ$}, so the $(-1)$-curve $A_{\star}$ meets all $(-2)$-curves labeled by $\star\in \{\bullet,\circ\}$. 
	
	The  points where some two of the above twelve  $(-1)$-curves meet are: $V_1\cap V_2\cap V_3$, and the orbit of $A_3\cap A_{\spadesuit}$.
	
	These twelve $(-1)$-curves exhaust all $(-1)$-curves on $Y$. We will not use this fact, so we leave the proof as an exercise for the reader. One can see this either from a direct computation in $\NS_{\Q}(Y)$, or from the description of the $(-1)$-curves on the minimal resolution of del Pezzo surface of type $4\rA_2$ given e.g.\ in \cite[p.\ 595]{BBD_canonical}.
\end{remark}

\begin{remark}[$\Aut(\textnormal{\ref{ex:w=1}\ref{item:w=1_cha=3_GK}})\leq \PGL_{3}(\F_3)$]
	The automorphism group of the surface $\bar{Y}$ from Example \ref{ex:w=1}\ref{item:w=1_cha=3_GK} can be described in yet another way, as follows. By \cite[Proposition 5.1(5)]{KN_Pathologies}, cf.\ \cite[Remark 7.3]{PaPe_MT}, the minimal log resolution $(Z,D_Z)$ of a del Pezzo surface of rank one and type $4\rA_2$ is isomorphic to the blowup of $\P^2$ in $p_1,\dots,p_{8}$, where $p_0,\dots,p_{8}$ are all $\F_3$-rational points which do not lie on a fixed $\F_3$-rational line $\ll$. The divisor  $D_Z$ is the proper transform of all $\F_3$-rational lines not passing through $p_0$ and different from $\ll$. Now  $\Aut(\bar{Y})$ can be identified with the stabilizer in $\PGL_{3}(\F_3)$ of $p_0$, $\ll$, and  one $\F_3$-rational point on $\ll$. 
\end{remark}

\begin{lemma}[$\Aut(\textnormal{\ref{ex:w=1}\ref{item:w=1_cha=3_not-GK}})\cong S_4$]\label{lem:w=1_cha=3_not-GK_Aut}
	Let $\bar{Y}$ be the surface from Example \ref{ex:w=1}\ref{item:w=1_cha=3_not-GK}, and let $(Y,D_Y)$ be its minimal log resolution. Then $\Aut(\bar{Y})=\Aut(Y,D_Y)$ is isomorphic to the symmetric group $S_4$, permuting the $(-3)$-curves in $D_Y$. It is generated by all permutations of the three degenerate fibers $\phi^{*}\ll_i$, $i\in \{1,2,3\}$ and the involution
	\begin{equation}\label{eq:Bernasconi_involution}
		(G_{1},G_{1}',G_{2},G_{2}',G_3,G_3',H,T_3,T_1,T_2)\mapsto
		(G_{1}',G_{1},G_{2}',G_{2},G_3,G_3',T_3,H,T_2,T_1),
	\end{equation}
	where $H=\phi^{-1}_{*}\qq$ and for $i\in \{1,2,3\}$ we have $(\phi^{*}\ll_i)\redd=[3,1,2,2]= T_i+A_i+G_{i}+G_{i}'$, cf.\ Remark \ref{rem:cube}.
\end{lemma}
\begin{proof}
	Consider the group homomorphism $\sigma\colon \Aut(Y,D_Y)\to S_4$ given by the action of $\Aut(Y,D_Y)$ on the set of $(-3)$-curves in $D_Y$. We will show that $\sigma$ is surjective and injective.
	
	The original construction of $\bar{Y}$ from \cite{Bernasconi_char_3-example}, recalled in Remark \ref{rem:Bernasconi}, shows that $\Aut(Y,D_Y)$ contains the symmetric group $S_3$ permuting the degenerate fibers. To construct the involution \eqref{eq:Bernasconi_involution}, let $(Y,D_Y-G_3)\to (Z,D_Z)$ be the contraction of $A_3$. Then $(Z,D_Z)$ is the minimal log resolution of the surface from Example \ref{ex:w=1}\ref{item:w=1_cha=3_GK}, and the involution $\sigma_{\textnormal{\ref{item:Aut_rotation}}}^3$, where $\sigma_{\textnormal{\ref{item:Aut_rotation}}}$ is as in Lemma \ref{lem:w=1_cha=3_GK_Aut}\ref{item:Aut_rotation}, lifts to \eqref{eq:Bernasconi_involution}. We conclude that $\sigma$ is surjective.
	
	To see that $\sigma$ is injective fix $\alpha\in\ker\sigma$. Then $\alpha(H)=H$, so we have $\alpha(F)\cdot H=3$ for a fiber $F$ and $\alpha(F)\cdot V=0$ for every component $V$ of $D_Y-H$. Since $K_Y$ and the components of $D_Y$ generate $\NS_{\Q}(Y)$, it follows that $\alpha(F)\equiv F$, so $\alpha$ maps fibers to fibers. Since $\alpha\in \ker\sigma$, it follows that $\alpha$ fixes each component of $D_Y$, so it acts trivially on $\NS_{\Q}(Y)$, and therefore fixes each $(-1)$-curve on $Y$. We conclude that $\alpha$ descends to an automorphism of $\P^2$ fixing $\phi_{*}D_Y=\qq+\sum_{i}\ll_i$ componentwise, so $\alpha=\id$, as claimed.
\end{proof}

\begin{remark}[Visualizing $\Aut(\textnormal{\ref{ex:w=1}\ref{item:w=1_cha=3_not-GK}})$, see Figure \ref{fig:cube}]\label{rem:cube}
	One can view the isomorphism $\Aut(Y,D_Y)\cong S_4$ from Lemma \ref{lem:w=1_cha=3_not-GK_Aut} as follows. Recall that $S_4$ is the rotation group of a cube, permuting its four diagonals \cite[Exercise 5.5]{Serre_linear-reps}. We label those diagonals by the $(-3)$-curves $H,T_1,T_2,T_3$. Next, we label the faces of the cube by the $(-2)$-curves $G_i$, $G_i'$, $i\in \{1,2,3\}$, so that $G_{i}$ is opposite to $G_{i}'$; the diagonal $T_i$ passes through the common vertex of $G_i$ and $G_{j}'$ for both $j\in \{1,2,3\}\setminus \{i\}$; and $H$ passes through the common vertex of $G_1$, $G_2$ and $G_3$. Now the rotation group $S_4$ acts on those labels as in the same way as $\Aut(Y,D_Y)$ acts on the components of $D_Y$. For example, the involution \eqref{eq:Bernasconi_involution} is a  $180^{\circ}$ rotation about the axis joining the centers of faces $G_3$ and $G_3'$.
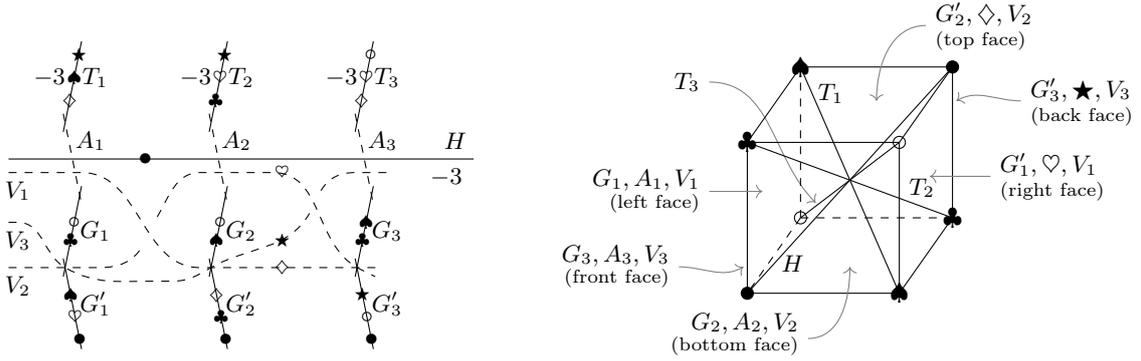
\begin{figure}[htbp]
	\begin{tikzpicture}
		\begin{scope}[scale=1.2]
			\draw (0.2,3.2) -- (0,2.2);
			\node at (0.35,2.8) {\small{$T_1$}};
			\node at (0.17,3.05) {\tiny{$\bigstar$}};
			\node at (0.12,2.8) {\tiny{$\spadesuit$}};
			\node at (0.07,2.55) {\tiny{$\diamondsuit$}};
			\node at (-0.15,2.8) {\small{$-3$}};
			\draw[dashed] (0,2.4) -- (0.2,1.4);
			\node at (0.3,2.1) {\small{$A_1$}};
			\draw (0.2,1.6) -- (0,0.6);
			\node at (0.35,1.1) {\small{$G_1$}};
			\node at (0.12,1.2) {\small{$\circ$}};
			\node at (0.08,1) {\tiny{$\clubsuit$}};
			\draw (0,0.8) -- (0.2,-0.2); 
			\node at (0.35,0.3) {\small{$G_1'$}};
			\node at (0.08,0.4) {\tiny{$\spadesuit$}};
			\node at (0.13,0.15) {\tiny{$\heartsuit$}};
			\node at (0.18,-0.1) {$\bullet$};
			\draw (1.8,3.2) -- (1.6,2.2);
			\node at (1.95,2.8) {\small{$T_2$}};
			\node at (1.77,3.05) {\tiny{$\bigstar$}};
			\node at (1.72,2.8) {\tiny{$\heartsuit$}};
			\node at (1.67,2.55) {\tiny{$\clubsuit$}};
			\node at (1.45,2.8) {\small{$-3$}};
			\draw[dashed] (1.6,2.4) -- (1.8,1.4);
			\node at (1.9,2.1) {\small{$A_2$}};
			\draw (1.8,1.6) -- (1.6,0.6);
			\node at (1.95,1.1) {\small{$G_2$}};
			\node at (1.72,1.2) {\small{$\circ$}};
			\node at (1.68,1) {\tiny{$\spadesuit$}};
			\draw (1.6,0.8) -- (1.8,-0.2); 
			\node at (1.95,0.3) {\small{$G_2'$}};
			\node at (1.68,0.4) {\tiny{$\diamondsuit$}};
			\node at (1.73,0.15) {\tiny{$\clubsuit$}};
			\node at (1.78,-0.1) {$\bullet$};
			\draw (3.4,3.2) -- (3.2,2.2);
			\node at (3.55,2.8) {\small{$T_3$}};
			\node at (3.37,3.05) {\small{$\circ$}};
			\node at (3.32,2.8) {\tiny{$\heartsuit$}};
			\node at (3.27,2.55) {\tiny{$\diamondsuit$}};
			\node at (3.05,2.8) {\small{$-3$}};
			\draw[dashed] (3.2,2.4) -- (3.4,1.4);
			\node at (3.5,2.1) {\small{$A_3$}};
			\draw (3.4,1.6) -- (3.2,0.6);
			\node at (3.55,1.1) {\small{$G_3$}};
			\node at (3.32,1.2) {\tiny{$\spadesuit$}};
			\node at (3.28,1) {\tiny{$\clubsuit$}};
			\draw (3.2,0.8) -- (3.4,-0.2); 
			\node at (3.55,0.3) {\small{$G_3'$}};
			\node at (3.28,0.4) {\tiny{$\bigstar$}};
			\node at (3.33,0.15) {\small{$\circ$}};
			\node at (3.38,-0.1) {$\bullet$};
			\draw (-0.6,1.9) -- (4.5,1.9);
			\node at (4.3,2.1) {\small{$H$}};
			\node at (4.2,1.7) {\small{$-3$}};
			\node at (0.9,1.9) {$\bullet$};
			\draw[dashed] (-0.6,1.75) -- (0.2,1.75) to[out=0,in=180] (1.5,0.7) -- (2.5,0.7) to[out=0,in=180] (3.2,0.7) -- (3.5,0.7);
			\node at (-0.5,1.55) {\small{$V_1$}};
			\node at (2.4,0.7) {\tiny{$\diamondsuit$}};
			\draw[dashed] (-0.6,1.2) -- (-0.5,1.2) to[out=0,in=150] (0.02,0.7) to[out=-30,in=180] (1,0.55) -- (1.2,0.55) to[out=0,in=-150] (1.62,0.7) to[out=30,in=-120] (2.7,1.3);
			\draw[dashed] (2.825,1.55) to[out=60,in=180] (3.2,1.75) -- (3.6,1.75);
			\node at (-0.5,1) {\small{$V_3$}};
			\node at (2.4,1) {\tiny{$\bigstar$}};
			\draw[dashed] (-0.6,0.7) -- (0.4,0.7) to[out=0,in=-120]  (0.87,0.94);
			\draw[dashed] (1.03,1.26) to[out=60,in=180] (1.5,1.75) -- (2.3,1.75) to[out=0,in=150] (3.22,0.7) -- (3.42,0.6);
			\node at (-0.5,0.5) {\small{$V_2$}};
			\node at (2.4,1.75) {\tiny{$\heartsuit$}};
		\end{scope}
		\begin{scope}[shift={(9,0.5)}]
			\coordinate (H) at (0,0);
			\coordinate (T1) at (2,0);
			\coordinate (T3) at (0.7,1);
			\coordinate (T2) at (2.7,1);
			\coordinate (T2') at (0,2);
			\coordinate (T3') at (2,2);
			\coordinate (T1') at (0.7,3);
			\coordinate (H') at (2.7,3);
			\node at ($(H)+(0.6,0.4)$) {\small{$H$}};
			\node at ($(T2)+(-0.4,0.4)$) {\small{$T_2$}};
			\node at ($(T1')+(0.4,-0.35)$) {\small{$T_1$}};
			%
			\filldraw (H) circle (0.08);
			\filldraw (H') circle (0.08);
			\node at (T1) {\small{$\spadesuit$}};
			\node at (T1') {\small{$\spadesuit$}};
			\node at (T2) {\small{$\clubsuit$}};
			\node at (T2') {\small{$\clubsuit$}};
			\draw (T3) circle (0.08);
			\draw (T3') circle (0.08);
			\draw[gray, ->] (-0.5,2.8) to[out=0,in=135] ($(T3)+(0.2,0.2)$);
			\node[left] at (-0.5,2.8)  {\small{$T_3$}};
			\draw (H) -- (T1) -- (T3') -- (T2') -- (H);
			\draw[dashed] (T3) -- (H);
			\draw[dashed] (T3) -- (T2);
			\draw[dashed] (T3) -- (T1');
			\draw (H') -- (T1') -- (T2');
			\draw (H') -- (T3');
			\draw (H') -- (T2) -- (T1);
			\draw (H) -- (H');
			\draw (T1) -- (T1');
			\draw (T2) -- (T2');
			\draw (T3) -- (T3');
			\node[left] at (-0.5,1.5) {\small{$G_1,A_1,V_1$}};
			\node[left] at (-0.5,1.2) {\scriptsize{(left face)}};
			\draw[->,gray] (-0.5,1.35) -- (0.2,1.35);
			\node[left] at (-0.9,0.5) {\small{$G_3,A_3,V_3$}};
			\node[left] at (-0.9,0.2) {\scriptsize{(front face)}};
			\draw[->,gray] (-0.9,0.35) to[out=0,in=-150] (-0.05,0.35);
			\node[right] at (3.2,1.7) {\small{$G_1',\heartsuit,V_1$}};
			\node[right] at (3.2,1.35) {\scriptsize{(right face)}};
			\draw[->,gray] (3.2,1.55) -- (2.4,1.55);
			\node[right] at (2.35,3.7) {\small{$G_2',\diamondsuit,V_2$}};
			\node[right] at (2.35,3.35) {\scriptsize{(top face)}};
			\draw[->,gray] (2.35,3.55) to[out=180,in=90] (1.7,2.55);
			\node[right] at (3.6,2.7) {\small{$G_3',\bigstar,V_3$}};
			\node[right] at (3.6,2.35) {\scriptsize{(back face)}};
			\draw[->,gray] (3.6,2.55) to[out=180,in=30] (2.75,2.55);
			\node[left] at (0.8,-0.4) {\small{$G_2,A_2,V_2$}};
			\node[left] at (0.8,-0.7) {\scriptsize{(bottom face)}};
			\draw[->,gray] (0.8,-0.6) to[out=0,in=-90] (1.35,0.35);
		\end{scope}
	\end{tikzpicture}

	\caption{Remark \ref{rem:cube}: the automorphism group of the surface from Example \ref{ex:w=1}\ref{item:w=1_cha=3_not-GK} viewed as the rotation group of a cube.}
	\label{fig:cube}
\end{figure}	

Using this identification, we can describe the action of $\Aut(Y,D_Y)$ on the set of $(-1)$-curves on $Y$, just like we did in Remark \ref{rem:cube} for the surface from Example  \ref{ex:w=1}\ref{item:w=1_cha=3_GK}. The orbit of $A_{1}$ is $\{A_1,A_2,A_3,A_{\heartsuit},A_{\diamondsuit},A_{\bigstar}\}$, where $A_{\bigstar}$ is the proper transform of the line joining $p_1$ and $p_3$, denoted by $L$ in Remark \ref{rem:hexagon}, and $A_{\heartsuit}$, $A_{\diamondsuit}$ are proper transforms of $(-1)$-curves denoted in Remark \ref{rem:hexagon} by the same letters. Each of these $(-1)$-curves meets exactly one $(-2)$-curve, which gives another labeling of the faces of the cube. more explicitly, we add labels $A_i$ to the face $G_i$, and $\heartsuit,\diamondsuit,\bigstar$ to faces $G_1',G_2',G_3'$, respectively. In Figure \ref{fig:cube}, we denote by $\heartsuit,\diamondsuit,\bigstar$ the common points of $D_Y$ with $A_{\heartsuit}$, $A_{\diamondsuit}$ and $A_{\bigstar}$, respectively.

Let $V_i$ be the proper transform of the curve $V_i$ from Remark \ref{rem:hexagon}, $i\in \{1,2\}$. Its orbit is $\{V_1,V_2,V_3\}$, where each $V_i$ passes through $G_j\cap G_{j}'$ for both $j\in \{1,2,3\}\setminus \{i\}$. To see this orbit in the cube, we add a label $V_i$ to both opposite faces $G_i$ and $G_i'$.

Eventually, let $A_{\bullet}$ be the proper transform of the line joining $p_2$ and $p_3$. Its orbit is $\{A_{\bullet},A_{\circ},A_{\spadesuit},A_{\clubsuit}\}$, where we denote the proper transforms of $(-1)$-curves from Remark \ref{rem:hexagon} by the same letters. Each of those $(-1)$-curves meets exactly one $(-3)$-curve, and we add its label to the corresponding diagonal of the cube: to keep Figure \ref{fig:cube} legible, we put those labels at the ends of each diagonal.

Now the action of $S_4$ on this labeled cube is exactly the action of $\Aut(Y,D_Y)$ on the components of $D_Y$ and the thirteen $(-1)$-curves described above. A direct computation in $\NS_{\Q}(Y)$ shows that those are the only $(-1)$-curves on $Y$.
\end{remark}

\clearpage

\section{Case $\width(\bar{X})=3$}\label{sec:w=3}

We begin the proof of Theorem \ref{thm:ht=3}. 
In this section we study the case $\width=3$, that is, we assume that
\begin{equation}\label{eq:assumption_ht=3}
	\parbox{.9\textwidth}{
		$\bar{X}$ is a log terminal del Pezzo surface of rank one, $\height(\bar{X})=3$; $(X,D)$ is the  minimal log resolution of $\bar{X}$, 
		and $p\colon X\to \P^1$ is a $\P^1$-fibration such that $D\hor$ consists of three $1$-sections.
	}
\end{equation}
Here $D=D\hor+D\vert$ is the decomposition into a horizontal and vertical part, see Section \ref{sec:P1-fibrations}.

As we have outlined in the introduction, we first prove the relevant case of Proposition \ref{prop:ht=3_models}, that is, we choose $p$ so that it factors as $p=\pr_1\circ\psi$ for some $\psi\colon (X,D)\to (\P^1\times \P^1,B)$, where $\pr_1 \colon \P^1\times \P^1\to \P^1$ is the first projection, and $B$ is a sum of three vertical and three horizontal lines. Next, we prove an analogous statement of Proposition \ref{prop:ht=3_swaps}, that is, we choose $p$ so that $\psi$ can be factored as $\psi=\phi\circ \phi_{+}$, where $\phi\colon (Y,D_Y)\to (\P^1\times \P^1,B)$ is as in Example \ref{ex:w=3}, and $\phi_{+}\colon (X,D)\sqto (Y,D_Y)$ is a vertical swap.

\subsection{Proof of Proposition \ref{prop:ht=3_models}: constructing the minimalizations \texorpdfstring{$\psi\colon (X,D)\to (\P^1\times \P^1,B)$}{onto P1xP1}} 

The aim of this section is to prove Proposition \ref{prop:ht=3_models} in case $\width(\bar{X})=3$. We keep the assumption \eqref{eq:assumption_ht=3}. 
First, we summarize basic properties of $p$ and we fix some notation for the remaining part of this section.

\begin{lemma}[Degenerate fibers]\label{lem:w=3_basics}
	Assume $(X,D,p)$ is as in \eqref{eq:assumption_ht=3}. Let $F_{1},\dots, F_{\nu}$ be all degenerate fibers of $p$. For $j=1,\dots, \nu$, let $\sigma(F_{j})$ be the number of $(-1)$-curves in $(F_{j})\redd$. Let $H_1,H_2,H_3$ be the components of $D\hor$. 
	\begin{enumerate}
		\item\label{item:w=3_Sigma} 	Put $\nu_{k}=\#\{j:\sigma(F_j)=k\}$ for $k\geq 1$. Then $(\nu_{2},\nu_{3})=(2,0)$ or $(0,1)$, and $\nu_{k}=0$ for $k\geq 4$.
		\item\label{item:ht=3_D'} Contract all vertical $(-1)$-curves on $X$ and its images which are disjoint from the images of $D\hor$, and denote the resulting morphism by $\phi\colon (X,D)\to (\hat{X},\hat{D})$. For a degenerate fiber $F$ of $p$ put $\hat{F}=\phi_{*}F$ and $V_{F}=\phi^{-1}_{*}\hat{F}$. Then one of the following holds.
		\begin{enumerate}
			\item\label{item:sigma=1_rivet} $\sigma(F)=1$, $\hat{F}=[0]$, $V_{F}\subseteq D\vert$, $V_{F}\cdot H_{j}=1$ and $H_{i}\cdot H_{j}=0$ for any $i,j\in \{1,2,3\}$, $i\neq j$.
			\item\label{item:sigma=1} $\sigma(F)=1$, $\hat{F}=[1,1]$, $V_{F}\subseteq D\vert$.
			\item\label{item:sigma=2} $\sigma(F)=2$, $\hat{F}=[1,(2)_{s},1]$ for some $s\geq 0$.
			\item\label{item:sigma=3} $\sigma(F)=3$, $\hat{F}=\langle 3;[1,(2)_{s_{1}}],[1,(2)_{s_{2}}],[1,(2)_{s_{3}}]\rangle$ for some $s_{1},s_{2},s_{3}\geq 0$.
		\end{enumerate}
	\end{enumerate}
\end{lemma}
\begin{proof}
	\ref{item:w=3_Sigma} This follows from Lemma \ref{lem:fibrations-Sigma-chi}\ref{item:Sigma}.
	
	\ref{item:ht=3_D'}. Assume $\hat{F}=[0]$. Then $V_{F}$ meets all components of $D\hor$, so $V_{F}\subseteq D$ and $\sigma(F)=1$ by \cite[Lemma 2.8(c)]{PaPe_ht_2}. Since $D$ has no circular subdivisor, the components of $D\hor$ are pairwise disjoint, so \ref{item:sigma=1_rivet} holds. 
	
	Assume $\hat{F}\neq [0]$, and let $\hat{\sigma}$ be the number of $(-1)$-curves in $\hat{F}$. Every such $(-1)$-curve meets $\hat{D}\hor$, so its proper transform in $V_{F}$ meets $D\hor$, in particular has multiplicity one in $\hat{F}$. It follows that $\hat{\sigma}\geq 2$. If $\hat{\sigma}>\sigma(F)$ then two $(-1)$-curves in $\hat{F}$ meet at a base point of $\phi^{-1}$, so \ref{item:sigma=1} holds. In the other case, we have $\hat{\sigma}\leq \sigma(F)\leq 3$ by \ref{item:w=3_Sigma}, and since $\hat{F}$ blows down to a $0$-curve, we get \ref{item:sigma=2} or \ref{item:sigma=3}.
\end{proof}

The next lemma is a key step in this section. 
Once it is proved, to infer the relevant part of Proposition~\ref{prop:ht=3_models} it is enough to show that one can choose $p$ with only three degenerate fibers, which we do in Lemma \ref{lem:nu_1=1} below.

\begin{lemma}
	\label{lem:H_disjoint}
	Let $(X,D,p)$ be as in \eqref{eq:assumption_ht=3}. One can choose the $\P^1$-fibration $p$ so that the components $H_1$, $H_2$, $H_3$ of $D\hor$ are pairwise disjoint. 
	
	Moreover, there is a birational morphism $\psi\colon (X,D)\to(\P^{1}\times \P^{1},B)$ such that  $p=\pr_1\circ\psi$, where $\pr_1\colon \P^1\times \P^1\to\P^1$ is the first projection,  $B\hor$ is a sum of three horizontal lines,  and $\#B\vert\geq 3$.
\end{lemma}
\begin{proof}
	Contract all vertical $(-1)$-curves in $X$ and its images which meet the horizontal part of the boundary at most once. Denote the resulting morphism by $\psi\colon (X,D)\to (Z,B)$. Then $Z$ is smooth, $B$ is snc and $B\hor\cong D\hor$. 
	Suppose that the induced $\P^1$-fibration of $Z$ has a degenerate fiber, say $F$. Since $F\cdot \psi_{*}D\hor=3$, $F$ has exactly one $(-1)$-curve, say $L$. Such $L$ has multiplicity at least $2$ in $F$, so $L\cdot \psi_{*}D\hor \leq 1$, a contradiction. Thus $Z$ is isomorphic to a Hirzebruch surface $\F_{n}$ for some $n\geq 0$. Put $\bar{H}_i=\psi(H_i)$, $\bar{V}_{i}=\psi_{*}F_{i}$ and $V_{i}=\psi^{-1}_{*}\bar{V}_i$. Then $V_i$ is a component of the fiber $F_i$. Put $\eta= \#B\vert$ and order the degenerate fibers of $p$ so that $B\vert=\sum_{i=1}^{\eta}\bar{V}_{i}$. 
	
	If the components of $D\hor$ are pairwise disjoint, then $n=0$ and the other projection of $\F_{0}=\P^{1}\times \P^{1}$ pulls back to a $\P^{1}$-fibration of height $\eta$, so $\eta \geq \height(\bar{X})=3$; as needed. Hence it remains to find $p$ with disjoint $H_{j}$'s.
	
	\begin{claim}\label{cl:H_not_chain}
		We can assume that $D\hor$ is not a chain.
	\end{claim}
	\begin{proof}
		Assume  that $D\hor$ is a chain. Then $B\hor\cong D\hor$ is a chain, too. Say that $\bar{H}_j=B\hor\cp{j}$. Since each $\bar{H}_j$ is a $1$-section, $\bar{H}_i-\bar{H}_j$ is linearly equivalent to a multiple of a fiber. In particular $0=(\bar{H}_1-\bar{H}_2)\cdot (\bar{H}_1-\bar{H}_3)=\bar{H}_{1}^2-\bar{H}_1\cdot (\bar{H}_2+\bar{H}_3)+\bar{H}_2\cdot\bar{H}_3=\bar{H}_1^2$, so $n=0$, i.e.\ $Z\cong \P^1\times \P^1$. It follows that  $B\hor=[0,-2,0]$. 
		Recall that the divisor $B=\psi_{*}D$ is snc, so since $D$ has no circular subdivisor, the fiber $\bar{V}_{j}$ for each $j\in \{1,\dots, \eta\}$ contains at least two base points of $\psi^{-1}$. Its preimage $F_{j}$ has at least two $(-1)$-curves, so $\eta\leq \nu_2\leq 2$ by Lemma \ref{lem:w=3_basics}\ref{item:w=3_Sigma}.
		
		Let $\tilde{p}$ be the pullback to $X$ of the other $\P^{1}$-fibration of $\P^{1}\times \P^{1}$, and let $\tilde{D}\hor$ be the  horizontal part of $D$ with respect to $\tilde{p}$. Then $\tilde{D}\hor=H_{2}+\sum_{i=1}^{\eta}V_{i}$. We have $3\geq \eta+1=\#\tilde{D}\hor\geq \height(X,D)=3$, so $\tilde{D}\hor$ consists of three $1$-sections and $\eta=\nu_2$. It follows that $\eta=\nu$: indeed, by Lemma \ref{lem:w=3_basics}\ref{item:w=3_Sigma} we have $\nu_{k}=0$ for $k\geq 3$, and by Lemma \ref{lem:w=3_basics}\ref{item:ht=3_D'} every degenerate fiber $F$ with $\sigma(F)=1$ meets $D\hor$ in components of $D\vert$, so its image is contained in $B\vert$. In particular, every base point of $\psi^{-1}$ lying on $B\hor$ lies on $B\vert$, too. Since $H_{2}^{2}\leq -2<2=\bar{H}_{2}^{2}$, one of those base points lies in $\bar{H}_{2}\cap B\vert$. Thus $\tilde{D}\hor$ is not a chain, and replacing $p$ with $\tilde{p}$ proves the claim. 
	\end{proof}
	
	Since no subdivisor of $D$ is circular, Claim \ref{cl:H_not_chain} allows us to assume that $H_{1}\cdot H_{2}=1$ and $H_{1}\cdot H_{3}=H_{2}\cdot H_{3}=0$. It follows that $Z\cong \F_1$ and $\bar{H}_3=[1]$ is the negative section.
	
	\begin{claim}\label{claim:eta>=3} We can assume that $\eta\geq 3$; in particular, $\nu\geq 3$.\end{claim}
	\begin{proof}
		Let $\tau\colon X\to \P^{2}$ be the composition of $\psi$ with the contraction of $\bar{H}_{3}$. Then $\ll_{i}\de\tau(H_{i})$ $i\in \{1,2\}$ and $\ll_{j}'\de \tau_{*}F_j$, $j\in \{1,\dots,\eta\}$ are lines, and $\ll_{1}'\dots,\ll_{\eta}'$ meet at a common point $\tau(H_3)\not\in \ll_1\cup\ll_2$. Write $\{p_{ij}\}=\ll_{i}\cap \ll_{j}'$.
		
		Since $D$ has no circular subdivisor, for every $j\in \{1,\dots, \eta\}$ there is an $i\in \{1,2\}$ such that $p_{ij}\in\Bs\tau^{-1}$, so, say, $p_{11}\in \Bs\tau^{-1}$. If $\eta \geq 2$ then since $H_{i}^{2}\leq -2<1=\ell_{i}^{2}$, we can assume that $p_{22}\in \Bs\tau^{-1}$, too. Moreover, interchanging $p_{11}$ with $p_{22}$, if needed, we can assume that $\tau$ factors through at least two blowups over $p_{11}$, so the proper transform of the first exceptional curve over $p_{11}$, call it $E$, is a component of $D$.
		
		The pencil of lines through $p_{11}$ pulls back to a $\P^{1}$-fibration of $X$, such that the horizontal part of $D$ consists of $1$-sections, namely: $E$, $\tau^{-1}_{*}\ll_2$, and $\tau^{-1}_{*}\ll_{2}',\dots,\tau^{-1}_{*}\ll_{\eta}'$. Thus $\eta+1\geq \height(\bar{X})=3$, so $\eta\geq 2$. If the equality holds then, since $p_{22}\in \Bs\tau^{-1}$, the above three $1$-sections are pairwise disjoint, as needed.
	\end{proof}

	For a degenerate fiber $F_i$, let $\hat{F}_i\subseteq \hat{X}$ be as in Lemma \ref{lem:w=3_basics}\ref{item:ht=3_D'}; and for a divisor $\hat{G}$ on $\hat{X}$ let $G$ be its proper transform on $X$. Since $\nu\geq 3$ by Claim \ref{claim:eta>=3}, Lemma \ref{lem:w=3_basics}\ref{item:w=3_Sigma} implies that $\nu_1\geq 1$, i.e.\  $p$ has a degenerate fiber $F_i$ with exactly one $(-1)$-curve. Since $H_1\cdot H_2=1$, this fiber is as in Lemma \ref{lem:w=3_basics}\ref{item:sigma=1}, i.e.\ $\hat{F}_i=[1,1]$, and both tips of $\hat{F}_i$, namely $\hat{V}_i$ and, say, $\hat{G}_i$ meet $\hat{D}\hor$. We have $V_{i},G_{i}\subseteq  D\vert$. Since $(V_i+G_i)\cdot D\hor=F_i\cdot D\hor=3$, we have $V_{i}\cdot D\hor=2$ and $G_i\cdot D\hor=1$. Since $D$ has no circular subdivisor, $V_{i}$ meets $H_3$. It follows that all components of $D\hor$ lie in the same connected component of $D$, call it $D_{0}$.
	
	The fact that no subdivisor of $D$ is circular implies that there is only one fiber of the above type, i.e.\ $\nu_1=1$. By Lemma \ref{lem:w=1_basics}\ref{item:w=3_Sigma} we get $\nu_2=2$ and $\nu_k=0$ for $k\geq 3$, so $3=\nu=\eta$ by  Claim \ref{claim:eta>=3}. 
	We order the degenerate fibers so that $\sigma(F_1)=\sigma(F_2)=2$, $\sigma(F_3)=1$. Write $\{1,2\}=\{j_3,\hat{j}_3\}$, where $V_{3}$ meets $H_{j_3}$ and $G_3$ meets $H_{\hat{j}_3}$.
	
	Using this notation, we conclude that $D$ contains a chain $[H_3,V_3,H_{j_3},H_{\hat{j}_3},G_3]$. Since $D$ has no circular subdivisor, for each $i\in \{1,2\}$ the curve $V_{i}$ meets at most one component of this chain. On the other hand, the image of this chain on $\F_1$ meets $\bar{V}_i$ in three points, so since $\sigma(F_i)=2$, two of those points are base points of $\psi^{-1}$. As a consequence, $F_i$ meets each section $H_j$ in a different component, say $L_{i,j}$. Note that one of the curves $L_{i,j}$ equals $V_i$, and the other two are contained in different connected components of $\Exc\psi$.
	
	\begin{claim}\label{cl:H1H2=1_structure}
		For $i\in \{1,2\}$ we have $F_{i}=[1,(2)_{s_i},1]$ for some $s_i\geq 1$. Moreover, after possibly interchanging $F_1$ with $F_2$ and $H_1$ with $H_2$, we can pick an order on each chain $F_1,F_2$ so that the following hold (cf.\  Figure \ref{fig:H1H2=1_structure}).
		\begin{enumerate}
			\item $H_1$ meets $\ftip{F_1}$ and $\ftip{F_2}$,
			\item $H_2$ meets $\ltip{F_2}$ and some tip of $F_1\wedge D\vert$,
			\item $H_3$ meets $\ltip{F_1}$ and some tip of $F_2\wedge D\vert$.
		\end{enumerate}
	\end{claim}
	\begin{proof}	
		By Lemma \ref{lem:w=3_basics}\ref{item:ht=3_D'} we have $ \hat{F}_{i}=[1,(2)_{s_i},1]$ for some $s_{i}\geq 0$. By the definition of $\hat{F}_i$, each tip of $\hat{F}_i$ equals $\hat{L}_{i,j}$ for some $j\in \{1,2,3\}$. On the other hand, since $\hat{F}_i$ has only two tips, for some $j_{i}\in \{1,2,3\}$ the curve $\hat{L}_{i,j_i}$ is not a tip of $\hat{F}_i$. Now, definition of $\psi$ implies that the image of $\hat{L}_{i,j_i}$ equals $\bar{V}_{i}$, so $L_{i,j_{i}}=V_{i}$.
		
		Suppose $j_1=j_2=3$. Then the contraction of $F_{i}-V_{i}$ is an isomorphism in a neighborhood of $H_3$. The same is true for the contraction of $F_3-V_3$, since the latter is disjoint from $H_3$. Thus $H_3^2=\bar{H}_3^2=-1$; a contradiction.

		Therefore, interchanging $F_1$ with $F_2$ and $H_1$ with $H_2$, if needed, we can assume that $j_2=1$, i.e.\  $V_{1}$ meets $H_2$. Since $\bar{X}$ is log terminal, $D$ is a sum of admissible chains and forks, so $\beta_{D}(H_2)\leq 3$. Since $H_2$ meets both $H_1$ and $V_3+G_3$, $D_0$ is a fork with branching component $H_2$, so $L_{2,2}\not\subseteq D$ and $\beta_{D}(H_j)\leq 2$ for $j=1,3$.
		
		Since $\beta_{D}(H_1)\leq 2$, and $H_1$ meets $H_2$ and $V_3+G_3$, we have $L_{i,1}\not\subseteq D$ for both $i\in \{1,2\}$. Similarly, since $\beta_{D}(H_{3})\leq 2$ and $H_{3}$ meets $V_3+G_3$, we have $L_{i_0,3}\not\subseteq D$ for some $i_0\in \{1,2\}$. It follows that the curves $\hat{L}_{2,2}$, $\hat{L}_{i_{0},3}$ and $\hat{L}_{i,1}$ for $i\in \{1,2\}$ are not contained in $\hat{D}$, so they are tips of $\hat{F}_{2}$, $\hat{F}_{i_0}$ and $\hat{F}_{i}$, respectively. We order the chain $\hat{F}_i$ so that $\hat{L}_{i,1}=\ftip{\hat{F}_i}$. Now since $\hat{L}_{i,j}\neq \hat{L}_{i',j'}$ for $(i,j)\neq (i',j')$, we get $\hat{L}_{1,2}=\ltip{\hat{F}_{1}}$, $i_0=1$ and $\hat{L}_{1,3}=\ltip{\hat{F}_1}$. In particular, for $i\in \{1,2\}$ both tips of $\hat{F}_i$ are not contained in $\hat{D}$, hence $F_i=\hat{F}_i$. 
		
		The remaining curves $L_{i,j}$ are equal to $V_i$, so they lie in $D_0$. Since $D_0$ is a fork with a branching component $H_2$, we get $\beta_{D}(V_i)\leq 2$ and thus $V_i$ is a tip of  $D\vert$, as claimed.
	\end{proof}
\begin{figure}[htbp]
	\begin{tikzpicture}[scale=0.9]
		\path[use as bounding box] (-1.2,-0.8) rectangle (10.2,3);
		\begin{scope}
			\draw[dashed] (0,3)-- (0.2,2);
			\node at (-0.2,2.7) {\small{$-1$}};
			\draw (0.2,2.2) -- (0,1.2);
			\node at (0.1,1.2) {$\vdots$};
			\node at (-0.5,1.1) {\small{$[(2)_{s_1}]$}};
			\draw (0,1) -- (0.2,0);
			\node at (0.3,0.65) {\small{$V_1$}};
			\draw[dashed] (0.2,0.2) -- (0,-0.8);
			\node at (-0.2,-0.3) {\small{$-1$}};
			\draw[dashed] (2.2,3)-- (2.4,2);
			\node at (2,2.5) {\small{$-1$}};
			\draw (2.4,2.2) -- (2.2,1.2);
			\node at (2,1.5) {\small{$V_2$}};
			\node at (2.3,1.2) {$\vdots$};
			\node at (2.9,1.1) {\small{$[(2)_{s_2}]$}};
			\draw (2.2,1) -- (2.4,0);
			\draw[dashed] (2.4,0.2) -- (2.2,-0.8);
			\node at (2,-0.3) {\small{$-1$}};
			\draw (4.2,3)-- (4.4,2);
			\node at (4.5,2.65) {\small{$V_3$}};
			\draw[dashed] (4.4,2.2) -- (4.2,1.2);
			\node at (4,1.7) {\small{$-1$}};
			\draw (4.2,1.4) -- (4.4,0.4);
			\node at (4.3,0.4) {$\vdots$};
			\draw (4.4,0.2) -- (4.2,-0.8);
			\node at (4.55,-0.3) {\small{$G_3$}};
			\draw (-1.2,-0.6) -- (4.6,-0.6);
			\node at (-1,-0.4) {\small{$H_1$}};
			\draw (-0.6,-0.8) -- (-0.6,0) to[out=90,in=180] (-0.2,0.4) -- (0.4,0.4) to[out=0,in=180] (1.6,2.9) -- (4.6,2.9);
			\node at (-0.75,0.4) {\small{$H_2$}};
			\draw (-1.2,2.4)-- (0.9,2.4);
			\draw (1.2,2.4) to[out=0,in=180] (2.2,1.85) -- (2.4,1.85) to[out=0,in=180] (3.3,2.4) -- (4.6,2.4);
			\node at (-1,2.6) {\small{$H_3$}};
			\draw[->] (5,1.1) -- (6,1.1);
			\node at (5.5,1.3) {\small{$\psi$}}; 
		\end{scope}
		\begin{scope}[shift={(7,-0.4)}]
			\begin{scope}
				\draw (0,3) -- (3.2,3);
				\node at (0.9,3.2) {\small{$\bar{H}_3$}};
				\node at (0.9,2.8) {\small{$-1$}};
				\draw (-0.6,0) -- (-0.6,0.4) to[out=90,in=180] (0,1.6) -- (3.2,1.6);
				\node at (0.9,1.8) {\small{$\bar{H}_2$}};
				\node at (0.9,1.4) {\small{$1$}};
				\draw (-0.8,0.2) -- (3.2,0.2);
				\node at (0.9,0.4) {\small{$\bar{H}_1$}};
				\node at (0.9,0) {\small{$1$}};
				\draw (0.2,3.2) -- (0.2,0);
				\node at (0,0.8) {\small{$\bar{V}_{\! 1}$}};
				\draw (1.6,3.2) -- (1.6,0);
				\node at (1.4,0.8) {\small{$\bar{V}_{\! 2}$}};
				\draw (3,3.2) -- (3,0);
				\node at (2.8,0.8) {\small{$\bar{V}_{\! 3}$}};
				\filldraw (0.2,0.2) circle (0.06);
				\filldraw (0.2,3) circle (0.06);
				\filldraw (1.6,0.2) circle (0.06);
				\filldraw (1.6,1.6) circle (0.06);
				\filldraw (3,0.2) circle (0.06);
			\end{scope}
		\end{scope}
	\end{tikzpicture}
	\caption{A possible shape of $D$ in Claim \ref{cl:H1H2=1_structure} (case $j_3=2$, $k_1=2$, $k_2=s_2+1$).}
	\label{fig:H1H2=1_structure}
\end{figure}
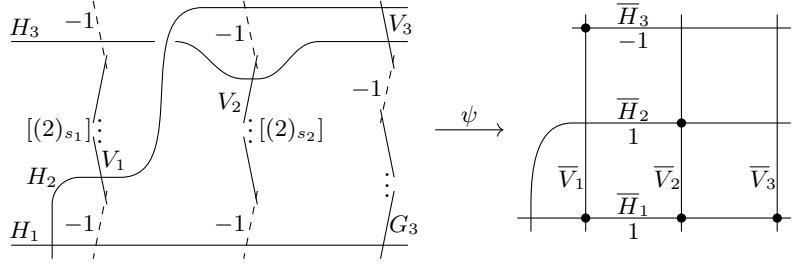
	
	Claim \ref{cl:H1H2=1_structure} implies that $D_0$ is a fork, whose branching component $H_2$ meets $H_1,V_{1}$ and $V_3+G_3$.  Moreover, for $i\in \{1,2\}$ we have $V_{i}=F_{i}\cp{k_i}$ for some $k_i\in \{2,s_i+1\}$. Recall that $V_3$ meets $H_{3}$ and $H_{j_{3}}$ for some $j_{3}\in \{1,2\}$.
	
	We have $\Exc\psi \wedge F_{i}=F_{i}-V_{i}$. Since $F_{1}-V_{1}$ is disjoint from $H_2$, we have $\bar{H}_2^{2}-H_2^{2}=(s_{2}+2-k_{2})+(2-j_3)\leq s_{2}-k_{2}+3$. Since $\bar{H}_2^2-H_2\geq 1-(-2)=3$, we get $k_{2}\leq s_{2}$. It follows that  $k_{2}=2$, $s_{2}\geq 1+j_3$. 
	
	Let $T_{j}$ be the maximal twig of $D_0$ containing $H_{j}$; $j\in \{2,3\}$. If $j_3=2$ then $\#T_{1}\geq \#(H_1+G_3)=2$ and $\#T_{3}\geq \#(F_{2}\wedge D\vert+H_{3}+V_3)=s_{2}+2\geq 5$; a contradiction with Lemma \ref{lem:admissible_forks}\ref{item:long-twig}. Hence $j_3=1$. Now $\#T_{3}=\#(F_{2}\wedge D\vert+H_{3}+V_3+H_1)=s_{2}+3\geq 5$, so by Lemma \ref{lem:admissible_forks}\ref{item:long-twig}, the remaining twigs of $D_{0}$ are $[2]$, $[2]$. In particular, $F_1\wedge D\vert=[2]$, so $k_{1}=2$. We have $3\leq \bar{H}_{1}^{2}-H_1^{2}=(k_{1}-1)+(k_{2}-1)=2$; a contradiction.
\end{proof}

\begin{notation}\label{not:untwisted}
	For the remaining part of Section \ref{sec:w=3}, we fix a birational morphism $\psi\colon (X,D)\to(\P^1\times \P^{1},B)$ as in Lemma \ref{lem:H_disjoint}. We put $\eta\de \#B\vert$, $\bar{V}_{i}= \psi_{*}F_{i}$, $V_i=\psi^{-1}_{*}\bar{V}_{i}$ for $i\in \{1,\dots,\nu\}$ and $\bar{H}_j=\psi(H_j)$ for $j\in \{1,2,3\}$.
\end{notation}

\begin{remark}[Switching projections of $\P^1\times \P^1$]\label{rem:tilde_p}
	 Let $\pr_1,\pr_2\colon \P^1\times \P^1\to \P^1$ be the two projections. Recall that $p=\psi\circ \pr_1$. The horizontal part of $D$ with respect to the $\P^1$-fibration $\tilde{p}\de \pr_2\circ\psi\colon X\to \P^{1}$ consists of $\eta$ disjoint $1$-sections, namely $V_1,\dots, V_{\eta}$. We have 
	\begin{equation*}
		\nu\geq \eta \geq \height(\bar{X})=3.
	\end{equation*}
	In particular, if $\nu=3$ then we can replace $p$ with $\tilde{p}$ whenever needed.
\end{remark}

\begin{lemma}[Case \ref{lem:w=3_basics}\ref{item:sigma=1_rivet} does not occur]\label{lem:no-rivet}
	We can further assume that no component of $D\vert$ meets all components of $D\hor$, so all degenerate fibers of $p$ are as in Lemma \ref{lem:w=3_basics}\ref{item:sigma=1}--\ref{item:sigma=3}.
\end{lemma}
\begin{proof}
	Suppose some component $R$ of $D\vert$ meets all components of $D\hor$. Let $F_{R}, D_R$ be the fiber and the connected component of $D$ containing $R$. Then $D_R=\langle R;T_{1},T_{2},T_{3}\rangle$, where $\ltip{T_{j}}=H_{j}$. 
	
	Suppose there is another degenerate fiber $F\neq F_{R}$ with $\sigma(F)=1$. Then by Lemma \ref{lem:w=3_basics}\ref{item:ht=3_D'} some component of $F\redd \wedge D\vert$ meets two components of $D\hor$. This is impossible, since $D_R$ is a tree. Thus $\nu_{1}=1$. Since $\nu\geq 3$ by Lemma \ref{lem:H_disjoint}, Lemma \ref{lem:w=3_basics}\ref{item:w=3_Sigma} implies that $\nu=3$, $\nu_1=1$ and $\nu_2=2$. 
	
	We can assume that all curves $V_{j}$ meet the same component of $D-\tilde{D}\hor$, say $\tilde{R}$: indeed, otherwise replacing $p$ with $\tilde{p}$ as in Remark \ref{rem:tilde_p} settles the claim. Now $\psi(\tilde{R})$ meets all $V_{j}$'s, so, say, $\psi(\tilde{R})=\bar{H}_{1}$; and $\bar{H}_1\cap B\vert$ is disjoint from $\Bs\psi^{-1}$. Since $\Bs\psi^{-1}\subseteq B\vert$, we get $\bar{H}_1\cap \Bs\psi^{-1}=\emptyset$, so $\tilde{R}^{2}=\bar{H}_{1}^{2}=0$; a contradiction.
\end{proof}

\begin{lemma}[Proposition \ref{prop:ht=3_models}, case $\width(\bar{X})=3$]\label{lem:nu_1=1}
	We have $\nu=3$, so the morphism $\psi$ is as in Proposition \ref{prop:ht=3_models}\ref{item:w=3_models}.
\end{lemma}
\begin{proof}
	By Lemma \ref{lem:H_disjoint} we have $\nu\geq 3$. Suppose $\nu\geq 4$. Then $\nu_1\geq 2$ by Lemma \ref{lem:w=3_basics}\ref{item:w=3_Sigma}. We order the degenerate fibers $F_{1},\dots, F_{\nu}$ so that  $\sigma(F_{j})=1$ for $j\leq \nu_{1}$. By Lemma \ref{lem:no-rivet}, every such $F_j$ is as in \ref{lem:w=3_basics}\ref{item:sigma=1}, so it contains two components of $D\vert$, namely $V_j$ with $V_j\cdot D\hor$ and, say, $G_j$ with $G_{j}\cdot D\hor =1$. Since $D$ has no circular subdivisor, we get $\nu_{1}=2$, and we can assume that $V_j$ meets $H_3$ and $H_j$ for $j\in \{1,2\}$, see Figure \ref{fig:nu_1=1_sigma=1}. Lemma \ref{lem:w=3_basics}\ref{item:w=3_Sigma} implies that $\nu=4$ and $\nu_{2}=2$, so $\sigma(F_{i})=2$ for $i\in \{3,4\}$. 
	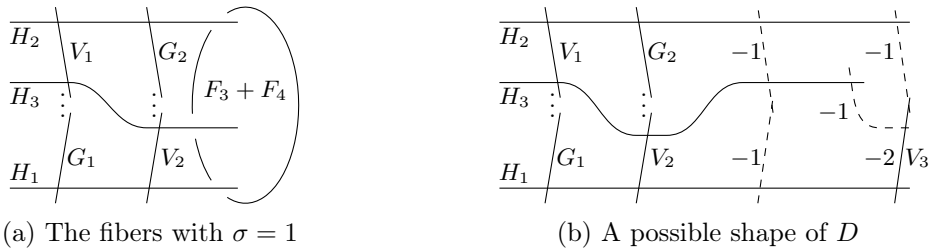
\begin{figure}[htbp]
		\subcaptionbox{The fibers with $\sigma=1$ \label{fig:nu_1=1_sigma=1}}[.3\textwidth]{
		\begin{tikzpicture}
			\path[use as bounding box] (-0.6,0) rectangle (3.2,2.6);
			\draw (0.2,1.4) -- (0,2.6);
			\node at (0.35,2) {\small{$V_{1}$}};
			\node at (0.1,1.4) {\small{$\vdots$}};
			\draw (0,0) -- (0.2,1.2);
			\node at (0.35,0.6) {\small{$G_1$}};
			\draw (1.4,1.4) -- (1.2,2.6);
			\node at (1.55,2) {\small{$G_{2}$}};
			\node at (1.3,1.4) {\small{$\vdots$}};
			\draw (1.2,0) -- (1.4,1.2);
			\node at (1.55,0.6) {\small{$V_2$}};
			\draw (-0.6,2.4) -- (2.4,2.4);
			\node at (-0.4,2.2) {\small{$H_2$}};
			\draw (-0.6,1.6) -- (0.2,1.6) to[out=0,in=180] (1.2,1) -- (2.4,1);
			\node at (-0.4,1.4) {\small{$H_3$}};
			\draw (-0.6,0.2) -- (2.4,0.2);
			\node at (-0.4,0.4) {\small{$H_1$}};
			\draw (2.5,1.3) [partial ellipse=110:-110:0.7 and 1.3];
			\draw (2.5,1.3) [partial ellipse=-130:-160:0.7 and 1.3];
			\draw (2.5,1.3) [partial ellipse=130:185:0.7 and 1.3];
			\node at (2.5,1.5) {\small{$F_3+F_4$}};
		\end{tikzpicture}	
		}
		\subcaptionbox{A possible shape of $D$ \label{fig:nu_1=1_all}}[.55\textwidth]{
		\begin{tikzpicture}
			\path[use as bounding box] (-0.6,0) rectangle (5,2.6);
			\draw (0.2,1.4) -- (0,2.6);
			\node at (0.35,2) {\small{$V_1$}};
			\node at (0.1,1.4) {\small{$\vdots$}};
			\draw (0,0) -- (0.2,1.2);
			\node at (0.35,0.6) {\small{$G_1$}};
			\draw (1.4,1.4) -- (1.2,2.6);
			\node at (1.55,2) {\small{$G_{2}$}};
			\node at (1.3,1.4) {\small{$\vdots$}};
			\draw (1.2,0) -- (1.4,1.2);
			\node at (1.55,0.6) {\small{$V_2$}};
			\draw (-0.6,2.4) -- (4.8,2.4);
			\node at (-0.4,2.2) {\small{$H_2$}};
			\draw (-0.6,1.6) -- (0.2,1.6) to[out=0,in=180] (1.2,0.9) -- (1.6,0.9) to[out=0,in=180] (2.6,1.6) -- (4.2,1.6);
			\node at (-0.4,1.4) {\small{$H_3$}};
			\draw (-0.6,0.2) -- (4.8,0.2);
			\node at (-0.4,0.4) {\small{$H_1$}};
			\draw[dashed]  (2.8,0) -- (3,1.4);
			\node at (2.65,0.6) {\small{$-1$}};
			\draw[dashed] (3,1.2) -- (2.8,2.6);
			\node at (2.65,2) {\small{$-1$}};
			%
			\draw  (4.6,0) -- (4.8,1.4);
			\node at (4.4,0.6) {\small{$-2$}};
			\node at (4.9,0.6) {\small{$V_3$}};
			\draw[dashed] (4.8,1.2) -- (4.6,2.6);
			\node at (4.4,2) {\small{$-1$}};
			\draw[dashed] (4,1.8) to[out=-80, in=180] (4.4,1) -- (4.8,1);
			\node at (3.8,1.2) {\small{$-1$}};
	\end{tikzpicture}
		}
		\caption{Proof of Lemma \ref{lem:nu_1=1}.}
		\label{fig:nu_1=1}
	\end{figure}
	
	The inequality $\eta\geq 3$ in Remark \ref{rem:tilde_p} implies that, say, $V_{3}\subseteq D\vert$. Since $\sigma(F_3)=2$, $V_3$ meets $H_i$ for some $i\in \{1,2,3\}$. Thus $\beta_{D}(H_i)=3$, so the connected component $D_0$ of $D$ containing $H_i$ is an admissible fork, with branching component $H_i$. For $j\neq i$ let $T_{j}$ be the twig of $D_0$ containing $H_j$. If $i=3$ then $T_{1}\neq T_{2}$ and $\#T_{j}\geq \#(G_{j}+H_j+V_{3-j})=3$ for $j\in \{1,2\}$, which is impossible by Lemma \ref{lem:admissible_forks}\ref{item:long-twig}. Thus, say, $i=1$. Since $\#T_{3}\geq \#(G_{2}+H_2+V_1+H_3+V_2)=5$, Lemma \ref{lem:admissible_forks} implies that the remaining twigs of $D$ meeting $H_1$ are of type $[2]$. In particular, $V_3=[2]$ and $F_{3}=[1,2,1]$. Moreover, $F_4\wedge D\vert$ is disjoint from $D_0$, so $F_4=[1,2,\dots,2,1]$ meets $D\hor$ in tips. Say that $H_1$ meets $\ftip{F_4}$. Then $V_4=\ltip{F_4}$ meets $H_2$ and $H_3$, since otherwise $H_1^2=\bar{H}_1^2-1=-1$, which is impossible. Now $H_3^2=\bar{H}_3^2-1=-1$, a contradiction. 
\end{proof}

\subsection{Proof of Proposition \ref{prop:ht=3_swaps}: vertical swaps to canonical surfaces}

Let $(X,D)$ be as in \eqref{eq:assumption_ht=3}, that is, the minimal log resolution of a log terminal del Pezzo surface of rank one, height $3$ and width $3$. In Lemma \ref{lem:nu_1=1} we have shown that $(X,D)$ admits a minimalization $\psi\colon (X,D)\to (\P^1\times \P^1,B)$ as in Proposition \ref{prop:ht=3_models}\ref{item:w=1_models}, i.e.\ with $B$ being a sum of $3$ vertical and $3$ horizontal lines. We will now show that $\psi$ factors through one of the morphisms $\phi$ used in  Example \ref{ex:w=3} to construct a vertically primitive log surface $(Y,D_Y)$, thus proving Proposition \ref{prop:ht=3_swaps} in case $\width=3$. This is done in part \ref{item:w=3_swaps} of Lemma \ref{lem:w=3_swaps}. Part \ref{item:w=3_additional-base-point} gives a more detailed description of the minimalization $\psi$, which will be useful later.

We keep Notation \ref{not:untwisted}. That is, we denote the horizontal and vertical lines in $B$ by $\bar{H}_i$ and  $\bar{V}_i$, respectively, $i\in \{1,2,3\}$. By Lemma \ref{lem:nu_1=1} the degenerate fibers of our fixed $\P^1$-fibration are $F_i=\psi^{*}V_i$, $i=1,2,3$. We put $V_{j}=\psi^{-1}_{*}\bar{V}_j$, $H_j=\psi^{-1}_{*}\bar{H}_j$. As in Example \ref{ex:w=3} we write $\{p_{ij}\}=\bar{V}_{i}\cap \bar{H}_{j}$ and we denote by $v_{ij}$, $h_{ij}$ the point infinitely near to $p_{ij}$ lying on the proper transform of $V_i$ and $H_j$, respectively. The underlining of $v_{ij}$ of $h_{ij}$ in Lemma \ref{lem:w=3_swaps}\ref{item:w=3_additional-base-point} bears no meaning other than a reference to be used in \ref{lem:w=3_swaps}\ref{item:w=3_swaps} and later in the proof. 

\begin{lemma}[Proposition \ref{prop:ht=3_swaps}, case $\width(\bar{X})=3$]\label{lem:w=3_swaps}
	We can further assume that the following hold.
	\begin{enumerate}
		\item\label{item:w=3_nu} $\sigma(F_{1})=1$, $\sigma(F_{2})=2$, $\sigma(F_{3})=2$, and $V_{1}$ meets both $H_{1}$ and $H_{2}$. 
		\item\label{item:w=3_additional-base-point} 
		One of the following holds, see Figure \ref{fig:w=3_swaps}.
		\begin{enumerate}
			\item\label{item:swap-to_A1+A2+A5_rivet} $\{p_{13},v_{13},p_{21},h_{21},p_{32},h_{32},p_{33},\uline{v_{21}}\}\subseteq \Bs\psi^{-1}$.
			\item\label{item:swap-to_A1+A2+A5_no-rivet} $\{p_{13},v_{13},p_{21},h_{21},p_{32},h_{32},p_{33},\uline{p_{23}}\}\subseteq \Bs\psi^{-1}$.
			\item\label{item:swap-to_2A4-v} $\{p_{13},v_{13},p_{22},p_{23},p_{31},h_{31},p_{32},\uline{v_{23}}\}\subseteq \Bs\psi^{-1}$.
			\item\label{item:swap-to_2A4-b} $\{p_{13},v_{13},p_{22},p_{23},p_{31},h_{31},p_{32},\uline{h_{23}}\}\subseteq \Bs\psi^{-1}$.
		\end{enumerate}
		\item\label{item:w=3_swaps} The surface $\bar{X}$ swaps vertically to a surface $\bar{Y}$ from Example \ref{ex:w=3}. More precisely, write $\psi=\phi\circ \phi_{+}$, where $\phi$ is a blowup at the points listed in \ref{item:w=3_additional-base-point} except the underlined one. Then $\phi_{+}\colon (X,D)\sqto (Y,D_{Y})$ is a vertical swap onto $(Y,D_{Y})$ as in Example \ref{ex:w=3}\ref{item:ht=3_A1+A2+A5} in cases \ref{item:swap-to_A1+A2+A5_rivet}, \ref{item:swap-to_A1+A2+A5_no-rivet} and as in \ref{ex:w=3}\ref{item:ht=3_2A4} in cases \ref{item:swap-to_2A4-v}, \ref{item:swap-to_2A4-b}.
	\end{enumerate}
\end{lemma}
\begin{proof}
	\ref{item:w=3_nu} We use the notation from Lemma \ref{lem:w=3_basics}\ref{item:w=3_Sigma}. The preimage of each base point of $\psi^{-1}$ contains at least one vertical $(-1)$-curve. Thus  $\psi^{-1}$ has at most $\sum_{j\geq 1}j\nu_{j}=5$ base points (not counting the infinitely near ones). Suppose $\sigma(F_3)=3$. Then by Lemma \ref{lem:w=3_basics}\ref{item:sigma=3}, $p_{3j}\in \Bs\psi^{-1}$ for all $j\in \{1,2,3\}$. Replacing $p$ with $\tilde{p}$, if needed, see Remark \ref{rem:tilde_p}, we can assume $p_{j3}\in \Bs\psi^{-1}$, too. Hence $\Bs\psi^{-1}\subseteq \bar{V}_{3}+\bar{B}_{3}$, so $V_{1}+V_{2}+H_1+H_2$ is a circular subdivisor of $D$; a contradiction. Thus $\nu_{3}=0$, so by Lemma \ref{lem:w=3_basics}\ref{item:w=3_Sigma} $\nu_{1}=1$, $\nu_{2}=2$, i.e.\ we can order $F_{1},F_{2},F_{3}$ so that $\sigma(F_1)=1$, $\sigma(F_j)=2$ for $j\in \{1,2\}$. By Lemma \ref{lem:w=3_basics}\ref{item:ht=3_D'}, $V_{1}$ meets two components of $D\hor$, as needed.
		
	 \ref{item:w=3_additional-base-point}, \ref{item:w=3_swaps} By \ref{item:w=3_nu}, we have $p_{11},p_{12}\not\in \Bs\psi^{-1}$. Since $V_1\subseteq D$, we have $\bar{V}_1^2-V_1^2\geq 2$, so $p_{13},v_{13}\in \Bs\psi^{-1}$. Since $\sigma(F_1)=1$, the condition $v_{13}\in \Bs\psi^{-1}$ implies that $h_{13}\not\in \Bs\psi^{-1}$. Because $\bar{H}_{3}^2-H_3^2\geq 2$, we infer that $p_{23}\in \Bs\psi^{-1}$ or $p_{33}\in \Bs\psi^{-1}$. Interchanging  $F_2$ with $F_3$, if needed, we can assume that $p_{33}\in \Bs\psi^{-1}$.
	
	\begin{casesp}
	\litem{$V_i\cdot D\hor\geq 2$ for some $i\in \{2,3\}$} If $V_i$ meets $H_1$ and $H_2$ then by \ref{item:w=3_nu} the subdivisor $V_1+V_i+H_1+H_2$ of $D$ is circular, which is impossible. Hence, say, $V_i\cdot H_1=0$, so $p_{i1}\in \Bs\psi^{-1}$ and $p_{i2},p_{i3}\not\in \Bs\psi^{-1}$. By assumption, $p_{33}\in \Bs\psi^{-1}$, so $i=2$. Since $\bar{H}_{2}^{2}-H_{2}^{2}\geq 2$ and $p_{12},p_{22}\not\in \Bs\psi^{-1}$, we infer that $p_{32},h_{32}\in \Bs\psi^{-1}$. Since $p_{32},p_{33}\in \Bs\psi^{-1}$ and $\sigma(F_{3})=2$, we have $p_{31}\not\in \Bs\psi^{-1}$. Recall that \ref{item:w=3_nu} gives $p_{11}\not\in \Bs\psi^{-1}$, too, so the inequality $\bar{H}_{1}^{2}-H_{1}^{2}\geq 2$ implies $h_{21}\in \Bs\psi^{-1}$. Thus $\psi=\phi\circ\phi_{+}$, where $\phi$ is a blowup at $p_{13},v_{13},p_{21},p_{31},p_{32},h_{32},p_{32},h_{32},p_{33}$; and  $\phi_{+}$ defines a vertical swap $(X,D)\sqto (Y,D_{Y})$ such that $(Y,D_Y)$ is as in Example  
	\ref{ex:w=3}\ref{item:ht=3_A1+A2+A5}. This proves part  \ref{item:w=3_swaps} in this case. To prove \ref{item:w=3_additional-base-point}, we will show that $v_{21}\in \Bs\psi^{-1}$.
	
	Put $V_{2}^{\phi}=\phi^{-1}_{*}\bar{V}_{2}$. Then $V_{2}^{\phi}=[1]$, so $V_{2}^{\phi}\not\in D_{Y}$, see Figure \ref{fig:w=3_A1+A2+A5}. But the proper transform $V_{2}$ of $V_{2}^{\phi}$ is a component of $D$, so $V_{2}^{\phi}$ contains a base point of $\phi_{+}^{-1}$, say $q$. Suppose $q\neq v_{21}$. Since by assumption $V_2$ meets $H_{2}$ and $H_{3}$, we get $q\not\in D_{Y}$. Then $V_2$ meets three twigs of $D$, namely $T_{j}$ with $\ltip{T_j}=H_j$, $j=2,3$, and one contained in $\psi^{-1}(p_{21})$. We have $\#T_{3}\geq 2$, since $v_{13}\in\Bs\psi^{-1}$; and $\#T_{2}\geq \#(H_{2}+V_{1}+H_1+V_{3})+1\geq 5$, since $b_{23}\in \Bs\psi^{-1}$, see Figure \ref{fig:w=3_A1+A2+A5}. This is  contradiction with Lemma \ref{lem:admissible_forks}\ref{item:long-twig}. Thus $q=v_{21}$ and \ref{item:swap-to_A1+A2+A5_rivet} holds.
	
	\litem{$V_i\cdot D\hor\leq 1$ for both $i\in \{2,3\}$} We interchange $F_2$ with $F_3$ so the assumption $p_{33}\in \Bs\psi^{-1}$ turns into $p_{23}\in \Bs\psi^{-1}$. Since $\sigma(F_i)=2$, we have $V_{i}\cdot D\hor=1$. Interchanging $H_1$ with $H_2$, if needed, we can assume $p_{22}\in \Bs\psi^{-1}$ and $p_{21}\not\in \Bs\psi^{-1}$. Since $\bar{H}_{1}^{2}-H_{1}^{2}\geq 2$ and $p_{11},p_{21}\not\in \Bs\psi^{-1}$, we have $p_{31},h_{31}\in\Bs\psi^{-1}$.
	
	Assume $V_3$ meets $H_2$, so $p_{32}\not\in \Bs\psi^{-1}$. Since $V_3\cdot D\hor=1$, we get $p_{33}\in \Bs\psi^{-1}$. Since $\bar{H}_{2}^{2}-H_{2}^{2}\geq 2$ and $p_{12},p_{32}\not\in \Bs\psi^{-1}$, we have $h_{22}\in \Bs\psi^{-1}$. After interchanging $H_{1}$ with $H_{2}$, we get \ref{item:swap-to_A1+A2+A5_no-rivet} and hence \ref{item:w=3_swaps} holds.
	
	Assume $V_3$ meets $H_3$, so $p_{32}\in \Bs\psi^{-1}$. Blowing up at the known base points of $\psi^{-1}$, namely at $p_{13}$, $v_{13}$, $p_{22}$, $p_{23}$, $p_{31}$, $h_{31}$, $p_{32}$, we get a morphism $\phi$ from Example \ref{ex:w=3}\ref{item:ht=3_2A4}, which proves part \ref{item:w=3_swaps} in this case. 
	
	Suppose \ref{item:swap-to_2A4-v} and \ref{item:swap-to_2A4-b} fail, so $v_{23},h_{23}\not\in \Bs\psi^{-1}$. Suppose $\psi^{-1}$ has a base point $q$ over $p_{23}$. Then $q\not\in D_{Y}$. Now $H_3$ meets three twigs of $D$, namely: one containing an exceptional curve over $p_{13}$, one containing $V_3$, which has length at least two, and one containing $H_2+V_{1}+H_{1}+V_{2}$, which has length at least $5$, see Figure \ref{fig:w=3_2A4}. This is a contradiction with Lemma \ref{lem:admissible_forks}\ref{item:long-twig}. Thus $\psi^{-1}$ has no further base points over $p_{23}$. Replacing $p$ with $\tilde{p}$, see Remark \ref{rem:tilde_p}, we can assume that $p$ has no further base points over $p_{32}$, either. Now the pencil of curves of type $(1,1)$, passing through $p_{23}$ and $p_{32}$, pulls back to a $\P^{1}$-fibration of height two on $(X,D)$; a contradiction.\qedhere
	\end{casesp}
\end{proof}

\begin{notation}[Reconstructing the minimalization $\psi$]\label{not:phi_+_H}
	In each of case of Lemma \ref{lem:w=3_swaps}\ref{item:w=3_additional-base-point}, let $\tau$ be the blowup at all points listed there, see Figure \ref{fig:w=3_swaps}, and let $\psi$ be as in Proposition \ref{prop:ht=3_models}\ref{item:w=3_models}, see Lemma \ref{lem:nu_1=1}. We study factorizations
	\begin{equation*}
	\begin{tikzcd}
		\psi\colon (X,D)
		\ar[r, "\gamma_{+}", squiggly]
		& (X_{\gamma},D_{\gamma})
		\ar[r, "\gamma"] 
		& (\P^1\times \P^1,B)
	\end{tikzcd}	
	\end{equation*}
	 where $\gamma$ is a composition of $\tau$ with some blowups over the $(-1)$-curves in $\Exc\tau$. We put $\check{D}_{\gamma}=(\gamma^{*}B)\redd$, let $D_{\gamma}$ be $\check{D}_{\gamma}$ minus all $(-1)$-curves, so $\gamma_{+}\colon (X,D)\sqto (X_{\gamma},D_{\gamma})$ is a vertical swap and $\Bs\gamma^{-1}_{+}\subseteq \check{D}_{\gamma}-D_{\gamma}$. By Lemma \ref{lem:w=3_swaps}\ref{item:w=3_additional-base-point}, one such factorization is given by $\psi=\tau\circ \tau_{+}$. The 
	 log surfaces $(X_{\tau},D_{\tau})$ are shown in Figure \ref{fig:w=3_swaps}.
	
	For a curve $G$ on $X$, or on some $X_{\eta}$, we denote its proper transform on $X_{\gamma}$ by $G^{\gamma}$. 
	Lemma \ref{lem:w=3_swaps} implies that $\gamma^{-1}(p_{ij})$ has at most one $(-1)$-curve, which we call $A_{ij}^{\gamma}$, unless  $(i,j)=(2,1)$ in case  \ref{item:swap-to_A1+A2+A5_rivet}: we then write $A_{v}^{\gamma}$, $A_{h}^{\gamma}$ for the $(-1)$-curve over $v_{21}$, $h_{21}$, respectively. 
	We skip the superscript whenever $\gamma$ is clear from the context.
\end{notation}
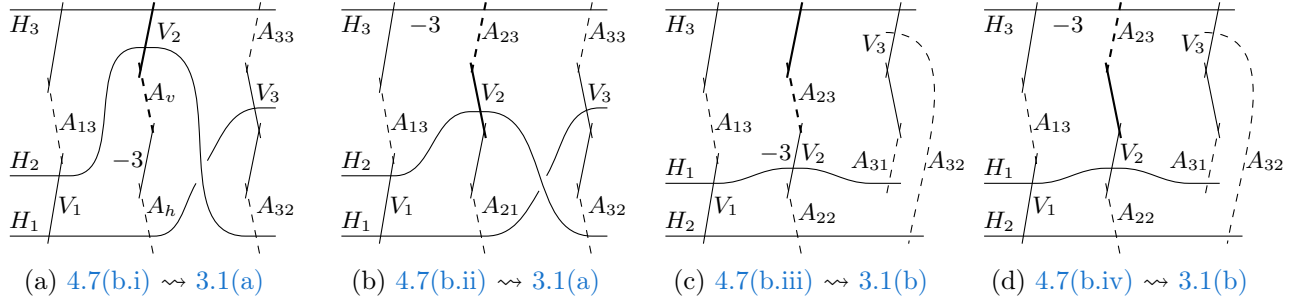
\begin{figure}[htbp]\vspace{-0.5em}
	\subcaptionbox{\ref{lem:w=3_swaps}\ref{item:swap-to_A1+A2+A5_rivet} $\sqto$ \ref{ex:w=3}\ref{item:ht=3_A1+A2+A5} \label{fig:swaps_to_A1_A2_A5_rivet}}[.25\linewidth]{
		\begin{tikzpicture}
			\path[use as bounding box] (-0.8,-0.1) rectangle (3.5,3.2);
			\draw (-0.5,3.1) -- (3,3.1);
			\node at (-0.3,2.9) {\small{$H_3$}};
			\draw (0.2,3.2) -- (0,2);
			\draw[dashed] (0,2.2) -- (0.2,1);
			\node at (0.4,1.6) {\small{$A_{13}$}};
			\draw (0.2,1.2) -- (0,0);
			\node at (0.3,0.5) {\small{$V_1$}};
			\draw[thick] (1.4,3.2) -- (1.2,2.2);
			\node at (1.6,2.8) {\small{$V_2$}}; 
			\draw[thick, dashed]  (1.2,2.4) -- (1.4,1.4);
			\node at (1.5,2) {\small{$A_{v}$}}; 
			\draw (1.4,1.6) -- (1.2,0.6);
			\node at (1.05,1.1) {\small{$-3$}};
			\draw[dashed] (1.2,0.8) -- (1.4,-0.2);
			\node at (1.5,0.5) {\small{$A_h$}}; 
			\draw[dashed] (2.8,3.2) -- (2.6,2.2);
			\node at (3,2.8) {\small{$A_{33}$}}; 
			\draw (2.6,2.4) -- (2.8,1.4);
			\node at (2.9,2) {\small{$V_3$}}; 
			\draw (2.8,1.6) -- (2.6,0.6);
			\draw[dashed] (2.6,0.8) -- (2.8,-0.2);
			\node at (3,0.5) {\small{$A_{32}$}};
			\draw (-0.5,0.9) -- (0.3,0.9) to[out=0,in=180] (1.2,2.6) -- (1.4,2.6) to[out=0,in=180] (2.6,0.1) -- (3,0.1);
			\node at (-0.3,1.1) {\small{$H_2$}};
			\draw (-0.5,0.1) -- (1.4,0.1) to[out=0,in=-120] (1.95,0.8);   
			\draw (2.1,1.1) to[out=60,in=180] (2.8,1.8) -- (3,1.8);
			\node at (-0.3,0.3) {\small{$H_1$}};
		\end{tikzpicture}
	}	
	\subcaptionbox{\ref{lem:w=3_swaps}\ref{item:swap-to_A1+A2+A5_no-rivet} $\sqto$ \ref{ex:w=3}\ref{item:ht=3_A1+A2+A5} \label{fig:swaps_to_A1_A2_A5_no-rivet}}[.25\linewidth]{
		\begin{tikzpicture}
			\path[use as bounding box] (-0.8,-0.1) rectangle (3.5,3.2);
			\draw (-0.5,3.1) -- (3,3.1);
			\node at (-0.3,2.9) {\small{$H_3$}};
			\node at (0.6,2.9) {\small{$-3$}};
			\draw (0.2,3.2) -- (0,2);
			\draw[dashed] (0,2.2) -- (0.2,1);
			\node at (0.4,1.6) {\small{$A_{13}$}};
			\draw (0.2,1.2) -- (0,0);
			\node at (0.3,0.5) {\small{$V_1$}};
			\draw[thick, dashed] (1.4,3.2) -- (1.2,2.2);
			\node at (1.6,2.8) {\small{$A_{23}$}}; 
			\draw[thick]  (1.2,2.4) -- (1.4,1.4);
			\node at (1.5,2) {\small{$V_2$}};
			\draw (1.4,1.6) -- (1.2,0.6);
			\draw[dashed] (1.2,0.8) -- (1.4,-0.2);
			\node at (1.6,0.5) {\small{$A_{21}$}}; 
			\draw[dashed] (2.8,3.2) -- (2.6,2.2);
			\node at (3,2.8) {\small{$A_{33}$}}; 
			\draw (2.6,2.4) -- (2.8,1.4);
			\node at (2.9,2) {\small{$V_3$}}; 
			\draw (2.8,1.6) -- (2.6,0.6);
			\draw[dashed] (2.6,0.8) -- (2.8,-0.2);
			\node at (3,0.5) {\small{$A_{32}$}};
			\draw (-0.5,0.9) -- (0.2,0.9) to[out=0,in=180] (1.2,1.75) -- (1.4,1.75) to[out=0,in=180] (2.8,0.1) -- (3,0.1);
			\node at (-0.3,1.1) {\small{$H_2$}};
			\draw (-0.5,0.1) -- (1.4,0.1) to[out=0,in=-120] (2.1,0.7);   
			\draw (2.2,0.9) to[out=60,in=180] (2.9,1.8) -- (3,1.8);
			\node at (-0.3,0.3) {\small{$H_1$}};
		\end{tikzpicture} 
	}
	\subcaptionbox{\ref{lem:w=3_swaps}\ref{item:swap-to_2A4-v} $\sqto$ \ref{ex:w=3}\ref{item:ht=3_2A4} \label{fig:swap-to_2A4-v}}[.24\linewidth]{
		\begin{tikzpicture}
			\path[use as bounding box] (-0.7,-0.1) rectangle (3.5,3.2);
			\draw (-0.5,3.1) -- (2.8,3.1);
			\node at (-0.3,2.9) {\small{$H_3$}};
			\draw (0.2,3.2) -- (0,2);
			\draw[dashed] (0,2.2) -- (0.2,1);
			\node at (0.4,1.6) {\small{$A_{13}$}};
			\draw (0.2,1.2) -- (0,0);
			\node at (0.3,0.5) {\small{$V_1$}};
			\draw[thick] (1.3,3.2) -- (1.1,2.2);
			\draw[thick, dashed]  (1.1,2.4) -- (1.3,1.4);
			\node at (1.5,2) {\small{$A_{23}$}}; 
			\draw (1.3,1.6) -- (1.1,0.6);
			\node at (0.95,1.2) {\small{$-3$}};
			\node at (1.45,1.2) {\small{$V_2$}};
			\draw[dashed] (1.1,0.8) -- (1.3,-0.2);
			\node at (1.5,0.4) {\small{$A_{22}$}}; 
			\draw (2.6,3.2) -- (2.4,2.2);
			\node at (2.25,2.65) {\small{$V_3$}};
			\draw (2.4,2.4) -- (2.6,1.4);
			\draw[dashed] (2.6,1.6) -- (2.4,0.6);
			\node at (2.2,1.1) {\small{$A_{31}$}};
			\draw[dashed] (2.4,2.8) to[out=0,in=80] (2.7,0);
			\node at (3.2,1.1) {\small{$A_{32}$}};
			\draw (-0.5,0.8) -- (0.2,0.8) to[out=0,in=180] (1.1,1) -- (1.3,1) to[out=0,in=180] (2.2,0.8)-- (2.6,0.8);
			\node at (-0.3,1) {\small{$H_1$}};
			\draw (-0.5,0.1) -- (2.9,0.1);
			\node at (-0.3,0.3) {\small{$H_2$}};
		\end{tikzpicture}  
	}		
	\subcaptionbox{\ref{lem:w=3_swaps}\ref{item:swap-to_2A4-b} $\sqto$ \ref{ex:w=3}\ref{item:ht=3_2A4} \label{fig:swap-to_2A4-b}}[.24\linewidth]{
		\begin{tikzpicture}
			\path[use as bounding box] (-0.7,-0.1) rectangle (3.5,3.2);
			\draw (-0.5,3.1) -- (2.8,3.1);
			\node at (-0.3,2.9) {\small{$H_3$}};
			\node at (0.6,2.9) {\small{$-3$}};
			\draw (0.2,3.2) -- (0,2);
			\draw[dashed] (0,2.2) -- (0.2,1);
			\node at (0.4,1.6) {\small{$A_{13}$}};
			\draw (0.2,1.2) -- (0,0);
			\node at (0.3,0.5) {\small{$V_1$}};
			\draw[thick, dashed] (1.3,3.2) -- (1.1,2.2);
			\node at (1.5,2.8) {\small{$A_{23}$}}; 
			\draw[thick]  (1.1,2.4) -- (1.3,1.4);
			\node at (1.45,1.2) {\small{$V_2$}};
			\draw (1.3,1.6) -- (1.1,0.6);
			\draw[dashed] (1.1,0.8) -- (1.3,-0.2);
			\node at (1.5,0.4) {\small{$A_{22}$}}; 
			\draw (2.6,3.2) -- (2.4,2.2);
			\node at (2.25,2.65) {\small{$V_3$}};
			\draw (2.4,2.4) -- (2.6,1.4);
			\draw[dashed] (2.6,1.6) -- (2.4,0.6);
			\node at (2.2,1.1) {\small{$A_{31}$}};
			\draw[dashed] (2.4,2.8) to[out=0,in=80] (2.7,0);
			\node at (3.2,1.1) {\small{$A_{32}$}};
			\draw (-0.5,0.8) -- (0.2,0.8) to[out=0,in=180] (1.1,1) -- (1.3,1) to[out=0,in=180] (2.2,0.8)-- (2.6,0.8);
			\node at (-0.3,1) {\small{$H_1$}};
			\draw (-0.5,0.1) -- (2.9,0.1);
			\node at (-0.3,0.3) {\small{$H_2$}};
		\end{tikzpicture}  
	}
	\caption{Log surfaces $(X_{\tau},D_{\tau})$ of height $3$ and width $3$ obtained by blowing up all points listed in Lemma \ref{lem:w=3_swaps}\ref{item:w=3_additional-base-point}, see Notation \ref{not:phi_+_H}. Swapping the thick curves (i.e.\ contracting the thick $(-1)$-curve (dashed) and removing the image of the thick $(-2)$-curve (solid) from the boundary) gives a log resolution of a canonical surface from Example \ref{ex:w=3}, see Lemma \ref{lem:w=3_swaps}\ref{item:w=3_swaps}.}\vspace{-1em}
	\label{fig:w=3_swaps}
\end{figure}

\subsection{An exotic pair of del Pezzo surfaces of rank 1 and height 3: Theorem \ref{thm:ht=3}\ref{item:uniq_exotic}}

We now construct the two non-isomorphic del Pezzo surfaces from Theorem \ref{thm:ht=3}\ref{item:uniq_exotic}, i.e.\ the ones of type 
\begin{equation}\tag{$\star$}\label{eq:25}
	[2,2,3,2,2,2,2,2]+[2,3].
\end{equation}
To do this, we define $\tau_{+}$ as in Notation \ref{not:phi_+_H} in such a way that the resulting log surfaces $(X,D)$ happen to have the same weighted graph of $D$, but different structure of the induced $\P^1$-fibrations, see Figure \ref{fig:ht=3_pair}.

\begin{example}[Construction]\label{ex:ht=3_pair}
For $\bar{X}_1$, take $\tau$ as in Lemma \ref{lem:w=3_swaps}\ref{item:swap-to_A1+A2+A5_rivet}, see Figure \ref{fig:swaps_to_A1_A2_A5_rivet}, and let $\tau_{+}$ be a blowup at $A_{v}^{\tau}\cap V_{2}^{\tau}$. For $\bar{X}_2$, take $\tau$ as in Lemma \ref{lem:w=3_swaps}\ref{item:swap-to_A1+A2+A5_no-rivet}, see Figure \ref{fig:swaps_to_A1_A2_A5_no-rivet}, and let $\tau_{+}$ be a blowup at $A_{32}^{\tau}\cap H_{2}^{\tau}$. This way, we get minimal log resolutions $(X_1,D_1)$ and $(X_2,D_2)$ of surfaces $\bar{X}_1$ and  $\bar{X}_2$, both of rank one and type \eqref{eq:25}. 
	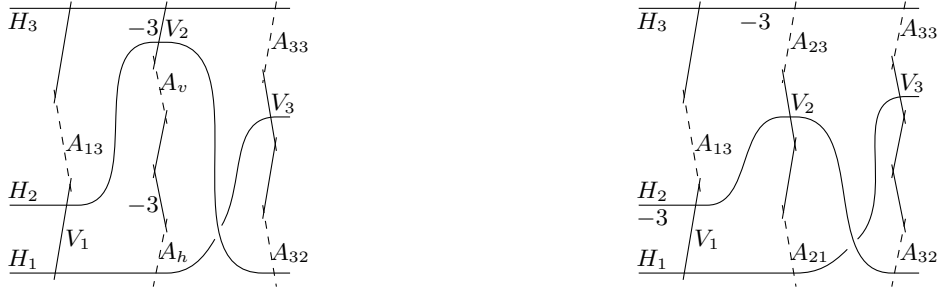
\begin{figure}[htbp]
		\subcaptionbox{$(X_1,D_1)$, see Lemma \ref{lem:w=3}\ref{item:rivet_A}, $k=3$. \label{fig:pair_2}}[.48\linewidth]{
			\begin{tikzpicture}[scale=0.9]
				\path[use as bounding box] (-1.5,-1) rectangle (4,3.2);
				\draw (-0.7,3.1) -- (3.4,3.1);
				\node at (-0.5,2.9) {\small{$H_3$}};
				\draw (0.2,3.2) -- (-0.05,1.7);
				\draw[dashed] (-0.05,1.9) -- (0.2,0.4);
				\node at (0.4,1.1) {\small{$A_{13}$}};
				\draw (0.2,0.6) -- (-0.05,-0.9);
				\node at (0.3,-0.3) {\small{$V_1$}};
				\draw (1.6,3.2) -- (1.4,2.2);
				\node at (1.25,2.8) {\small{$-3$}}; 
				\node at (1.75,2.8) {\small{$V_2$}}; 
				\draw[dashed]  (1.4,2.4) -- (1.6,1.4);
				\node at (1.7,2) {\small{$A_{v}$}}; 
				\draw (1.6,1.6) -- (1.4,0.6);
				\node at (1.25,0.2) {\small{$-3$}};
				\draw (1.4,0.8) -- (1.6,-0.2);
				\draw[dashed] (1.6,0) -- (1.4,-1);
				\node at (1.7,-0.5) {\small{$A_h$}}; 
				\draw[dashed] (3.2,3.2) -- (3,2);
				\node at (3.4,2.6) {\small{$A_{33}$}}; 
				\draw (3,2.2) -- (3.2,1);
				\node at (3.3,1.7) {\small{$V_3$}}; 
				\draw (3.2,1.2) -- (3,0);
				\draw[dashed] (3,0.2) -- (3.2,-1);
				\node at (3.4,-0.5) {\small{$A_{32}$}};
				\draw (-0.7,0.2) -- (0.3,0.2) to[out=0,in=180] (1.4,2.6) -- (1.6,2.6) to[out=0,in=180] (3,-0.8) -- (3.4,-0.8);
				\node at (-0.5,0.4) {\small{$H_2$}};
				\draw (-0.7,-0.8) -- (1.6,-0.8) to[out=0,in=-120] (2.3,-0.3);   
				\draw (2.4,-0.1) to[out=60,in=180] (3.2,1.5) -- (3.4,1.5);
				\node at (-0.5,-0.6) {\small{$H_1$}};
			\end{tikzpicture}
		}	
		\subcaptionbox{$(X_2,D_2)$, see Lemma \ref{lem:w=3}\ref{item:nu_3=1_c2}, $k=3$.  \label{fig:pair_5}}[.48\linewidth]{
			\begin{tikzpicture}[scale=0.9]
				\path[use as bounding box] (-1.5,-1) rectangle (4,3.2);			
				\draw (-0.7,3.1) -- (3.4,3.1);
				\node at (-0.5,2.9) {\small{$H_3$}};
				\node at (1,2.9) {\small{$-3$}};
				\draw (0.2,3.2) -- (-0.05,1.7);
				\draw[dashed] (-0.05,1.9) -- (0.2,0.4);
				\node at (0.4,1.1) {\small{$A_{13}$}};
				\draw (0.2,0.6) -- (-0.05,-0.9);
				\node at (0.3,-0.3) {\small{$V_1$}};
				\draw[dashed] (1.6,3.2) -- (1.4,2);
				\node at (1.8,2.6) {\small{$A_{23}$}}; 
				\draw (1.4,2.2) -- (1.6,1);
				\node at (1.7,1.7) {\small{$V_2$}}; 
				\draw (1.6,1.2) -- (1.4,0);
				\draw[dashed] (1.4,0.2) -- (1.6,-1);
				\node at (1.8,-0.5) {\small{$A_{21}$}};
				\draw[dashed] (3.2,3.2) -- (3,2.2);
				\node at (3.4,2.8) {\small{$A_{33}$}}; 
				\draw  (3,2.4) -- (3.2,1.4);
				\node at (3.3,2) {\small{$V_3$}}; 
				\draw (3.2,1.6) -- (3,0.6);
				\draw (3,0.8) -- (3.2,-0.2); 
				\draw[dashed] (3.2,0) -- (3,-1);
				\node at (3.4,-0.5) {\small{$A_{32}$}};
				\draw (-0.7,0.2) -- (0.3,0.2) to[out=0,in=180] (1.4,1.5) -- (1.6,1.5) to[out=0,in=180] (3,-0.8) -- (3.4,-0.8);
				\node at (-0.5,0.4) {\small{$H_2$}};
				\node at (-0.5,0) {\small{$-3$}};
				\draw (-0.7,-0.8) -- (1.6,-0.8) to[out=0,in=-135] (2.35,-0.45);   
				\draw (2.5,-0.3) to[out=45,in=180] (3.2,1.8) -- (3.4,1.8);
				\node at (-0.5,-0.6) {\small{$H_1$}};
			\end{tikzpicture}
		}			\vspace{-0.5em}
		\caption{Example \ref{ex:ht=3_pair}: two del Pezzo surfaces of type \eqref{eq:25} from Theorem \ref{thm:ht=3}\ref{item:uniq_exotic}.}\vspace{-1em}
		\label{fig:ht=3_pair}
	\end{figure}	
\end{example}
\begin{proposition}[Properties]\label{prop:exception}
	Let $\bar{X}_1$ and $\bar{X}_2$ be the surfaces constructed in Example \ref{ex:ht=3_pair}. 
	\begin{enumerate}
		\item\label{item:exception_del-Pezzo} $\bar{X}_1$ and $\bar{X}_2$ are del Pezzo.
		\item\label{item:exception_non-iso} $\bar{X}_1$ is not isomorphic to $\bar{X}_2$. 
		\item\label{item:exception_H1-etale} Assume $\cha\kk\neq 3$. Then 
			 $H^{1}_{\textnormal{\'et}}(\bar{X}_1\reg;\Z/3)\cong \Z/3$ and $H^1_{\textnormal{\'et}}(\bar{X}_2\reg;\Z/3)=0$.
			 		\item\label{item:exception_H1} Assume $\kk=\C$. Then $H_1(\bar{X}_1\reg;\Z)\cong \Z/3$ and $H_1(\bar{X}_2\reg;\Z)=0$. 
	\end{enumerate}
\end{proposition}
\begin{proof}
	\ref{item:exception_del-Pezzo} By Lemma \ref{lem:delPezzo_criterion} it is enough to show that $\sum_{i=1}^{3}\ld(H_i)>1$. We compute log discrepancies using Lemmas \ref{lem:ld_formulas} and \ref{lem:discriminants}. For $\bar{X}_1$ we have $\ld(H_1)=\frac{2}{3}$, $\ld(H_2)=\frac{4}{9}$ and $\ld(H_3)=\frac{5}{9}$, so $\sum_{i=1}^{3}\ld(H_i)=\frac{5}{3}>1$, as needed. For $\bar{X}_2$ we have $\ld(H_1)=\frac{5}{9}$, $\ld(H_2)=\frac{1}{3}$ and $\ld(H_3)=\frac{3}{5}$, so $\sum_{i=1}^{3}\ld(H_i)=\frac{67}{45}>1$, as claimed. 
	
	\ref{item:exception_non-iso}
	Suppose $\bar{X}_1$ is isomorphic to $\bar{X}_2$, so $(X_1,D_1)\cong (X_2,D_2)$. Then $X_1$ admits two $\P^1$-fibrations $p$, $p'$ coming from the construction of $\bar{X}_1$ and $\bar{X}_2$. Let $F'$ be a general fiber of $p'$, and let $F_i$ be the fiber of $p$ containing $V_i$, see Notation \ref{not:phi_+_H} or Figure \ref{fig:pair_2}. Then $F_2$ contains two $(-3)$-curves, which are horizontal for $p'$, see Figure \ref{fig:ht=3_pair}. Hence $2\leq F_2\cdot F'=F_1\cdot F'=1+2A_{13}\cdot F'$, so  $A_{13}$ is horizontal for $p'$. 
	
	Let $A_{21}'$  be the image on $X_1$ of the curve  $A_{21}\subseteq X_2$, see Figure \ref{fig:ht=3_pair}. Since $A_{21}'$ is vertical for $p'$, we have $A_{13}\neq A_{21}'$, so $A_{13}\cdot A_{21}'\geq 0$. We check that $A_{13}\cdot G=A_{21}'\cdot G$ for every component $G$ of $D_1$. Since $\NS_{\Q}(X_1)$ is generated by $K_{X_1}$ and the components of $D_1$, we get $A_{13}\equiv A_{21}'$, so $A_{13}\cdot A_{21}'=A_{13}^2=-1$; a contradiction.
	
	\ref{item:exception_H1-etale},\ref{item:exception_H1} By the Poincar\'e duality \cite[Theorem 24.1(b)]{Milne_etale} we have $H^{1}_{\textnormal{\'et}}(\bar{X}_i\reg;\Z/3)\cong H^{3}_{c}(\bar{X}_i\reg;\Z/3)$, 
	where $H^*_c$ denotes \'etale cohomology with compact supports. Similarly, if $\kk=\C$ then the Lefschetz duality gives $H_1(\bar{X}_i\reg;\Z)\cong H^3_c(\bar{X}_i\reg;\Z)$. Thus it is enough to compute $H^3_{c}(\bar{X}_i\reg)$. We use the same notation for \'etale cohomology with coefficients in $\Z/3$ and, if $\kk=\C$, for singular cohomology with coefficients in $\Z$. 
	
	We have $\bar{X}_i\reg\cong X_i\setminus D_i$. The long exact sequence in cohomology with compact supports  \cite[Theorem  I.8.7(3)]{FK_etale} gives an exact sequence $H^2(X_i)\to H^2(D_i)\to H^3_{c}(X_i\setminus D_i)\to H^3(X_i)$.
		
	Using the computation of $H^{*}(\P^1)$ and the K\"unneth formula \cite[Proposition 4.12 and Corollary 22.2]{Milne_etale}, we get $H^{3}(\P^1\times \P^1)=0$ and that $H^{2}(\P^1\times \P^1)$ is freely generated by a class of a general vertical line $L_v$ and a general horizontal line $L_h$ (more precisely: by the images of the Gysin maps $H^0(L_{\star})\to H^2(\P^1\times \P^1)$, see \cite[\sec 23]{Milne_etale}). Applying \citestacks{0EW3} to each blowup in the decomposition of $\psi\colon X_{i}\to \P^1\times \P^1$, we infer that $H^3(X_i)=0$ and $H^2(X_i)$ is freely generated by the proper transforms of $L_v$, $L_h$, all vertical $(-1)$-curves, and all components of $\Exc\psi \wedge D$. In turn, by the Mayer--Vietoris sequence, see \cite[Theorem 10.8]{Milne_etale}, $H^{2}(D_i)$ is freely generated by $H_{1},H_{2},H_3$, $V_1,V_2,V_3$, and all components of $\Exc\psi\wedge D$. We order these curves as in Figure \ref{fig:ht=3_pair}, from bottom to top, e.g.\ for $i=1$ we order the $(-1)$-curves as $A_{13},A_h,A_v,A_{32},A_{33}$. Thus $H^{3}_{c}(X_i\setminus D_i)$ is the cokernel of the restriction map $r_i\colon H^2(X_i)\to H^2(D_i)$, which in the above bases is given by a matrix
	\begin{equation*}
			\left[
			\begin{array}{@{}*{11}{r}@{}}
				1 & 0 & 0 & 1 & 0 & 0 & 0 & 0 & 0 & 0 & 0 \\
				1 & 0 & 0 & 0 & 0 & 1 & 0 & 0 & 0 & 0 & 0 \\	
				1 & 0 & 0 & 0 & 0 & 0 & 1 & 1 & 0 & 0 & 0 \\
				0 & 1 & 1 & 0 & 0 & 0 & 0 & 0 & 0 & 0 & 0 \\
				0 & 1 & 0 & 0 & 1 & 0 & 0 & 0 & 0 & 0 & 0 \\
				0 & 1 & 0 & 0 & 0 & 0 & 1 & 0 & 0 & 0 & 1 \\
				0 & 0 & 1 & 0 & 0 & 0 & 0 & -2 & 0 & 0 & 0 \\
				0 & 0 & 0 & 1 & 0 & 0 & 0 & 0 & -2 & 1 & 0 \\
				0 & 0 & 0 & 0 & 1 & 0 & 0 & 0 & 1 & -3 & 0 \\
				0 & 0 & 0 & 0 & 0 & 1 & 0 & 0 & 0 & 0 & -2 
			\end{array}
			\right]
			\quad\mbox{and}\quad
			\left[
\begin{array}{@{}*{11}{r}@{}}
	1 & 0 & 0 & 1 & 0 & 0 & 0 & 0 & 0 & 0 & 0 \\
	1 & 0 & 0 & 0 & 0 & 1 & 0 & 0 & 0 & 0 & 0 \\	
	1 & 0 & 0 & 0 & 1 & 0 & 1 & 1 & 0 & 0 & 0 \\
	0 & 1 & 1 & 0 & 0 & 0 & 0 & 0 & 0 & 0 & 0 \\
	0 & 1 & 0 & 0 & 1 & 0 & 0 & 0 & 0 & 0 & 0 \\
	0 & 1 & 0 & 0 & 0 & 0 & 1 & 0 & 0 & 0 & 1 \\
	0 & 0 & 1 & 0 & 0 & 0 & 0 & -2 & 0 & 0 & 0 \\
	0 & 0 & 0 & 1 & 0 & 0 & 0 & 0 & -2 & 0 & 0 \\
	0 & 0 & 0 & 0 & 0 & 1 & 0 & 0 & 0 & -2 & 1 \\
	0 & 0 & 0 & 0 & 0 & 0 & 0 & 0 & 0 & 1 & -2 
\end{array}
\right], 
	\end{equation*}
for $i=1$ and $2$, respectively. A direct computation shows that its cokernel has order $3$ and $1$, respectively.
\end{proof}

\subsection{The list of singularity types}\label{sec:w=3_list}

In this section we complete the proof of Theorem \ref{thm:ht=3} in case $\width(\bar{X})=3$. 
We keep the assumption \eqref{eq:assumption_ht=3} and Notation \ref{not:phi_+_H}. Our aim is to get the minimal log resolution $(X,D)$ of $\bar{X}$ by reconstructing a sequence of elementary vertical swaps $\tau_{+}\colon (X,D)\sqto (X_{\tau},D_{\tau})$, where $(X_{\tau},D_{\tau})$ is as in Lemma \ref{lem:w=3_swaps}\ref{item:w=3_additional-base-point}, see Figure \ref{fig:w=3_swaps}. 

Lemma \ref{lem:w=3_uniqueness}\ref{item:w=3_no_C1} below, which we will often use without further  comment, asserts that the center of each elementary swap is one of the finitely many common points of the boundary and vertical $(-1)$-curves. To reconstruct $(X,D)$ we blow up over each of these points and repeat this as long as $D$ is a sum of admissible chains and forks and the inequality \eqref{eq:ld-bound_ht=3} holds. The resulting list is given in Lemma \ref{lem:w=3}.

\begin{lemma}[Reconstructing the minimalization $\psi$ step by step]\label{lem:w=3_uniqueness}
	Let $\bar{X}$ be a del Pezzo surface of rank one, height $3$ and width $3$. Let $\psi=\gamma\circ \gamma_{+}$ be as in Notation \ref{not:phi_+_H}. Then the following hold.
	\begin{enumerate}
		\item\label{item:ld-bound_ht=3}The log surface $(X_{\gamma},D_{\gamma})$ is a minimal resolution of an lt del Pezzo surface $\bar{X}_{\gamma}$ of rank one. In particular, 
				\begin{equation}\label{eq:ld-bound_ht=3}
					\ld(H_{1}^{\gamma})+\ld(H_{2}^{\gamma})+\ld(H_{3}^{\gamma})>1.
				\end{equation}
		\item \label{item:w=3_no_C1} 
		Every base point of $\gamma_{+}^{-1}$ is a common point of $D_{\gamma}\reg$ and some vertical $(-1)$-curve.
		\item \label{item:w=3_AD=2} Every vertical $(-1)$-curve meets $D_{\gamma}$  in two points, normally.
		\item \label{item:w=3_Sing=2} The divisor $D_{\gamma}$ has exactly two connected components (equivalently, $\#\Sing \bar{X}_{\gamma}=2$). 
		\item \label{item:w=3_no-loops} If a vertical $(-1)$-curve contains a base point of $\gamma_{+}^{-1}$ then it meets both connected components of $D_{\gamma}$.  
		\item \label{item:ht=3_uniqueness} The isomorphism class of the log surface $(X_{\gamma},\check{D}_{\gamma})$ is uniquely determined by the weighted graph of $\check{D}_{\gamma}$.
		\item \label{item:ht=3_h1} We have $h^{i}(\lts{X_{\gamma}}{D_{\gamma}})=0$ for all $i\geq 0$.
	\end{enumerate}
\end{lemma}
\begin{proof}
	\ref{item:ld-bound_ht=3} The first statement follows from Lemma \ref{lem:cascades}\ref{item:cascades-still-dP}, the inequality \eqref{eq:ld-bound_ht=3} follows from  Lemma \ref{lem:delPezzo_criterion}.
	
	\ref{item:w=3_no_C1} Let $r\in \Bs\gamma_{+}^{-1}$. Since $\gamma_{+}$ is a vertical swap, $r$ lies on some vertical $(-1)$-curve $A=A_{\star}^{\gamma}$, see Notation \ref{not:phi_+_H}. In particular, $r\not\in \Sing D_{\gamma}$. Suppose $r\in A\setminus D_{\gamma}$. Let $\eta$ be a composition of $\gamma$ with a blowup at $r$. We will get a contradiction with the fact that $D_{\eta}$ is a sum of admissible chains and forks.
	
	If $\tilde{\eta}$ is a composition of $\tau$ with a blowup at $A_{\star}^{\tau}\setminus D_{\tau}$, then the weighted graph of  $D_{\tilde{\eta}}$ is a weighted subgraph of the one of $D_{\eta}$, so by Lemma \ref{lem:Alexeev} we can assume $\gamma=\tau$. We have 
	$A\cdot D_{\tau}=2$, see Figure \ref{fig:w=3_swaps}. Since $D_{\eta}$ has no circular subdivisor, $A$ meets both connected components of $D_{\tau}$; call them $U_{1}$, $U_{2}$. Let $U_{j}'$ be the component of $U_{j}$ meeting $A$, $j=1,2$. We have $D_{\eta}\cong D_{\tau}+A$, so $D_{\eta}$ is connected. Hence $D_{\eta}$ has at most one branching component. It follows that, say, $U_{1}$ is a chain and $U_{1}'$ is its tip. Suppose $U_{2}$ is a fork, i.e.\ case \ref{lem:w=3_swaps}\ref{item:swap-to_2A4-v} holds, see Figure \ref{fig:swap-to_2A4-v}. Then $U_{2}'$ is a tip of $U_{2}$, too, so $A=A_{23}^{\tau}$. But then the twigs of $D_{\eta}$ containing $V_{2}^{\eta}$ and $V_{3}^{\eta}$ have length $6$ and $2$, respectively, contrary to Lemma \ref{lem:admissible_forks}\ref{item:long-twig}. Thus $U_{2}$ is a chain. We check directly, see Figure \ref{fig:w=3_swaps}, that $U_{2}'$ is not a tip of $U_2$. It follows that $D_{\eta}$ is a fork with twigs of lengths $(\#U_{1}+1,\#U_{2}-k,k-1)$ for some $2\leq k\leq \frac{1}{2}\#U_2$. By Lemma \ref{lem:admissible_forks}\ref{item:has_-2} one of these twigs is $[2]$, so $k=2$. The remaining twigs have lengths $(2,6)$, $(3,5)$, and $(5,2)$ in cases \ref{item:swap-to_A1+A2+A5_rivet}, \ref{item:swap-to_A1+A2+A5_no-rivet}, and \ref{item:swap-to_2A4-b}, respectively, see Figure \ref{fig:w=3_swaps}. This contradicts Lemma \ref{lem:admissible_forks}\ref{item:long-twig}.
	
	\ref{item:w=3_AD=2}, \ref{item:w=3_Sing=2}, \ref{item:w=3_no-loops} For $\gamma=\tau$ parts \ref{item:w=3_AD=2}, \ref{item:w=3_Sing=2} are clear, see Figure \ref{fig:w=3_swaps}, and \ref{item:w=3_no-loops} holds because otherwise $D$ would have a circular subdivisor. In general, \ref{item:w=3_AD=2}--\ref{item:w=3_no-loops} follow from \ref{item:w=3_no_C1} by induction.

	\ref{item:ht=3_uniqueness} This follows from the uniqueness of $(\P^1\times \P^1,B)$ and from \ref{item:w=3_no_C1}. 
	More precisely, let $\check{\cS}_{\gamma}$ be the combinatorial type of $(X_{\gamma},\check{D}_{\gamma})$. Proposition \ref{prop:primitive}\ref{item:primitive-uniqueness} gives $\#\cP_{+}(\check{\cS}_{\phi})=1$. 
	By \ref{item:w=3_no_C1} the morphism $(X_{\gamma},\check{D}_{\gamma})\to (X_{\phi},\check{D}_{\phi})$ is inner, so $\#\cP_{+}(\check{\cS}_{\gamma})=1$ by the universal property of blowing up, 
	cf.\ \cite[Lemma 2.18]{PaPe_ht_2}, as needed.
	
	\ref{item:ht=3_h1} Since $\psi\colon (X_{\gamma},\check{D}_{\gamma})\to (\P^1\times \P^1,B)$ is inner, by \cite[Lemma 1.5(4)]{FZ-deformations} we have $h^{i}(\lts{X_{\gamma}}{D_{\gamma}})=h^{i}(\lts{\P^1\times \P^1}{B})=0$ for all $i$, see \cite[Lemmas 2.11(a),(b) and 2.12(d.iii)]{PaPe_ht_2}.
\end{proof}

We are now ready to list all singularity types $\cS$ of del Pezzo surfaces of rank one, height 3 and width 3 having no descendant with elliptic boundary, together with the structure of some witnessing $\P^1$-fibration $p\colon X\to \P^1$. 

We use notation summarized in Section \ref{sec:notation}. That is, we write each $\cS$ as a sum of rational chains $[a_1,\dots,a_n]$ and forks $\langle b;T_1,T_2,T_3\rangle$ which are connected components of $D$. To describe $p$ we put in boldface the numbers corresponding to horizontal components of $D$, and we add a superscript $\dec{j}$ to components meeting the $j$-th vertical $(-1)$-curve. The list of singularity types $\cS$ without decorations is given in Table \ref{table:ht=3_char=0} (top part). 

Recall that $\Pht^{\width=3}(\cS)$ is the set of isomorphism classes of del Pezzo surfaces of rank 1, height 3, and width 3.

\setcounter{claim}{0}
\begin{lemma}[Classification, case $\width=3$, see Table \ref{table:ht=3_char=0}]\label{lem:w=3}
	Let $\cS$ be a singularity type of a log terminal surface. Let $\bar{X}$ be a del Pezzo surface of rank $1$, height $3$, width $3$ and type $\cS$, i.e.\ $\bar{X}\in \Pht^{\width=3}(\cS)$. 
Assume that $\bar{X}$ has no descendant with elliptic boundary. 
	Then $\cS$ is listed below, and the following hold.
	\begin{parts}
		\item\label{item:w=3-uniqueness} Either $\#\Pht^{\width=3}(\cS)=1$, i.e.\ $\bar{X}$ is unique up to an isomorphism, or $\cS$ is as in Example \ref{ex:ht=3_pair} and ${\#\Pht^{\width=3}(\cS)=2}$.
		\item\label{item:w=3-classification} The minimal log resolution $(X,D)$ of $\bar{X}$ admits a $\P^1$-fibration $p$ such that $\bar{X}$ swaps vertically to one of the canonical surfaces $\bar{Y}$ from Example \ref{ex:w=3}, and the combinatorial type of $(X,D,p)$ is one of the following. 
	\end{parts} 
	\begin{enumerate}[itemsep=0.6em]
		\item\label{item:swap-to_A1+A2+A5_rivet-list} $\bar{Y}$ is of type $\rA_{1}+\rA_{2}+\rA_{5}$, see Example \ref{ex:w=3}\ref{item:ht=3_A1+A2+A5}; case \ref{lem:w=3_swaps}\ref{item:swap-to_A1+A2+A5_rivet} holds, and $(X,D,p)$ is one of the following.
	\begin{longlist}
		\item\label{item:rivet_A} $[2\dec{1},\bs{2}\dec{5},k\dec{2},\bs{2}\dec{4},2\dec{1},\bs{2}\dec{3},2\dec{5},2\dec{4}]+\ldec{2}[(2)_{k-2},3]\dec{3}$, $k\geq 3$ (for $k=3$ we get $\bar{X}_1$ from Example \ref{ex:ht=3_pair}),
		\item\label{item:rivet_AC} $[2\dec{1},\bs{2}\dec{5},3\dec{2},\bs{3}\dec{4},2\dec{1},\bs{2}\dec{3},2\dec{5},2,2\dec{4}]+\ldec{2}[2,3]\dec{3}$,	
		\item\label{item:rivet_0} $\ldec{1}[(2)_{k-1},\bs{2}\dec{5},2\dec{2},\bs{2}\dec{4},k\dec{1},\bs{2}\dec{3},2\dec{5},2\dec{4}]+[3]\dec{2,3}$, $k\in \{3,4\}$, 
	\end{longlist}\setcounter{foo}{\value{longlisti}}
	\item\label{item:swap-to_A1+A2+A5_no-rivet-list} $\bar{Y}$ is of type $\rA_{1}+\rA_{2}+\rA_{5}$, see Example \ref{ex:w=3}\ref{item:ht=3_A1+A2+A5}; case \ref{lem:w=3_swaps}\ref{item:swap-to_A1+A2+A5_no-rivet} holds, and $(X,D,p)$ is one of the following:
	\begin{longlist}\setcounter{longlisti}{\value{foo}}	
		\item\label{item:nu_3=1_c2} $
		[2\dec{4},2\dec{5},\bs{k}\dec{2},2\dec{1},\bs{2}\dec{4},2\dec{3},(2)_{k-1}]\dec{2}
		+[2\ldec{1},\bs{3}\dec{3,5}]$, $k\in \{3,4,5\}$ (for $k=3$ we get $\bar{X}_2$ from Example~\ref{ex:ht=3_pair}),
		\item\label{item:nu_3=1_e1} $
		[2\dec{4},2\dec{5},\bs{3}\dec{2},k\dec{1},\bs{2}\dec{4},2\dec{3},2,2\dec{2}]
		+\ldec{1}[(2)_{k-1},\bs{3}\dec{3,5}]$, $k\in \{3,4\}$,
		\item\label{item:nu_3=1_e2} $
		[2\dec{4},k\dec{5},\bs{3}\dec{2},2\dec{1},\bs{2}\dec{4},2\dec{3},2,2\dec{2}]
		+\ldec{1}[2,\bs{3}\dec{3},(2)_{k-2}]\dec{5}$, $k\in \{3,4\}$,
		\item\label{item:nu_3=1_s1} $
		[2\dec{4},k\dec{5},\bs{2}\dec{2},2\dec{1},\bs{2}\dec{4},2\dec{3},2\dec{2}]
		+\ldec{1}[2,\bs{3}\dec{3},(2)_{k-2}]\dec{5}$, $k\geq 3$,
		\item\label{item:nu_3=1_s3} $
		[2\dec{4},2\dec{5},\bs{2}\dec{2},k\dec{1},\bs{2}\dec{4},2\dec{3},2\dec{2}]
		+\ldec{1}[(2)_{k-1},\bs{3}\dec{3,5}]$, $k\geq 3$,
		\item\label{item:nu=3_C} $
		[2\dec{4},3\dec{5},\bs{2}\dec{2},3\dec{1},\bs{2}\dec{4},2\dec{3},2\dec{2}]
		+\ldec{1}[2,2,\bs{3}\dec{3},2]\dec{5}$,
		\item\label{item:nu=3_fork_s1=1} $\langle k,[2]\dec{4},[2]\dec{5},[2\dec{2},2\dec{3},\bs{2}\dec{4},2\dec{1},\bs{2}\dec{2}]\rangle+\ldec{1}[2,\bs{3}\dec{3},(2)_{k-3},3\dec{5}]$, $k\geq 3$,
		\setcounter{foo}{\value{longlisti}}
	\end{longlist}
	\item\label{item:swap-to_2A4-v_list} $\bar{Y}$ is of type $2\rA_4$, see Example \ref{ex:w=3}\ref{item:ht=3_2A4}; case \ref{lem:w=3_swaps}\ref{item:swap-to_2A4-v} holds, and $(X,D,p)$ is one of the following: 
	\begin{longlist}\setcounter{longlisti}{\value{foo}}	
		\item\label{item:both_B} $[2\dec{5},\bs{2}\dec{3},2\dec{1},\bs{3}\dec{4},3\dec{2,3}]+\langle \bs{2},[2]\dec{1},[2]\dec{2},\ldec{4}[2,2,3]\dec{5}\rangle$,
		\item\label{item:ht=3_YG} $[\bs{2}\dec{3,5},2\dec{1},\bs{3}\dec{4},3\dec{3},2\dec{2}]+
		\langle \bs{2};[2]\dec{1},[3]\dec{2},\ldec{4}[2,2,2]\dec{5}\rangle$,
		\item\label{item:ht=3_YDYD} $\ldec{3}[(2)_{k-3},\bs{2}\dec{5},2\dec{1},\bs{2}\dec{4},k\dec{3},2\dec{2}] + \langle \bs{2};[2]\dec{1},[3]\dec{2},\ldec{4}[2,2]\dec{5}\rangle$, $k\in \{4,5,6\}$,
		\item\label{item:ht=3_XY_b=3_v} $[2\dec{3},\bs{2}\dec{5},2\dec{1},\bs{2}\dec{4},3\dec{3}]*(T^{*})\dec{2}+
		\langle \bs{2};[2]\dec{1},\ldec{2}T,\ldec{4}[2,2]\dec{5}\rangle$, $T\in \{[3,2],[2,3],[4],[5]\}$,
		\setcounter{foo}{\value{longlisti}}
		\end{longlist}
		\item\label{item:swap-to_2A4-b_list} $\bar{Y}$ is of type $2\rA_4$, see Example \ref{ex:w=3}\ref{item:ht=3_2A4}; case \ref{lem:w=3_swaps}\ref{item:swap-to_2A4-b} holds, and $(X,D,p)$ is one of the following. 
		\begin{longlist}\setcounter{longlisti}{\value{foo}}		
		\item\label{item:ht=3_XE} $\langle 2;[2]\dec{3},[2]\dec{2},[\bs{3}\dec{2,5},2\dec{1},\bs{2}\dec{4}]\rangle+
		[2\dec{1},\bs{3}\dec{3},2\dec{5},2\dec{4}]$	,
		\item\label{item:ht=3_XA_a=3_c=3} $\langle 3;[2]\dec{1},[\bs{2}]\dec{2,5},[2\dec{3},2,2\dec{2},\bs{2}\dec{4}]\rangle+
		[3\dec{1},2,\bs{4}\dec{3},2\dec{5},2\dec{4}]$,
		\item\label{item:ht=3_XA_c=2} $\langle k;[2]\dec{1},[\bs{2}]\dec{2,5},[2\dec{3},2\dec{2},\bs{2}\dec{4}]\rangle+
		[3\dec{1},(2)_{k-2},\bs{3}\dec{3},2\dec{5},2\dec{4}]$, $k\geq 3$,
		\item\label{item:ht=3_XAA_T=[2,2]} $\langle k;[2\dec{1},2],[\bs{2}]\dec{2,5},[2\dec{3},2\dec{2},\bs{2}\dec{4}]\rangle+
		[4\dec{1},(2)_{k-2},\bs{3}\dec{3},2\dec{5},2\dec{4}]$, $k\geq 3$,
		\item\label{item:ht=3_XAA_T=[3]} $\langle k;[3]\dec{1},[\bs{2}]\dec{2,5},[2\dec{3},2\dec{2},\bs{2}\dec{4}]\rangle+
		[2\dec{1},3,(2)_{k-2},\bs{3}\dec{3},2\dec{5},2\dec{4}]$, $k\geq 3$,
		\item\label{item:ht=3_XY_a=3} $\ldec{3}[3,(2)_{k-3},2\dec{2},\bs{2}\dec{4},3\dec{1},\bs{2}\dec{2,5}]+
		\langle \bs{k};\ldec{3}[2],\ldec{1}[2,2],\ldec{4}[2,2]\dec{5}\rangle$, $k\geq 3$,
		\item\label{item:ht=3_XY_b=3} $\ldec{3}T^{*}*[(2)_{k-2},3\dec{2},\bs{2}\dec{4},2\dec{1},\bs{2}\dec{5},2\dec{2}]+
		\langle \bs{k};\ldec{3}T\trp,\ldec{1}[2],\ldec{4}[2,2]\dec{5}\rangle$, $k\geq 3$, $d(T)\leq 5$,
		\item\label{item:ht=3_XY_b=4_T-2} $\ldec{3}[3,4\dec{2},\bs{2}\dec{4},2\dec{1},\bs{2}\dec{5},2,2\dec{2}]+
		\langle \bs{3};\ldec{3}[2],\ldec{1}[2],\ldec{4}[2,2]\dec{5}\rangle$,
		\item\label{item:ht=3_chains} $\ldec{3}[(2)_{c-1},b\dec{2},\bs{d}\dec{4},a\dec{1},\bs{2}\dec{5},(2)_{b-2}]\dec{2}+\ldec{1}[(2)_{a-1},\bs{c+1}\dec{3},2\dec{5},(2)_{d-1}]\dec{4}$, with $(a,b,c,d)$ as in Table \ref{table:abcd}.
		\end{longlist}
\end{enumerate}
\begin{small}	
\begin{table}[htbp]	\vspace{-0.5em}
	\begin{tabular}{r||c|c|c|c|c|c|c|c|c|c|c|c|c|c|c|c|c}
		$a$ & $[3,\infty)$ & $[6,8]$ & $[3,7]$ & $5$ & $4$ & $4$ & $[3,4]$ & $3$ & $3$ & $3$ & $2$ & $2$ & $2$ & $2$ & $2$ & $2$ & $2$ \\ \hline 
		$b$ & $2$ & $2$ & $2$ & $2$ & $2$ & $2$ & $3$ & $2$ &  $2$ &  $2$ &  $[4,7]$ & $5$ & $4$ & $4$ & $3$ & $3$ & $3$ \\ \hline 
		$c$ & $2$ & $3$ & $2$ & $[3,4]$ & $[3,7]$ & $2$ & $2$ & $[3,\infty)$ & $2$ & $4$ & $2$ & $3$ & $[3,5]$ & $3$ & $[3,\infty)$ & $3$ & $2$ \\ \hline 
		$d$ & $2$ & $2$  & $3$ & $2$ & $2$  & $4$ & $2$ & $2$ & $[4,7]$ & $3$ & $2$ & $2$ & $2$ & $3$ & $2$ & $3$ & $[2,6]$       
	\end{tabular}\vspace{-0.5em}
	\caption{Values of $a,b,c,d$ in Lemma \ref{lem:w=3}\ref{item:ht=3_chains}: they are solutions to \eqref{eq:ld-bound_ht=3} with $(a,b)\neq (2,2)$.}
		\label{table:abcd} \vspace{-1em}
\end{table}
\end{small}
\end{lemma}
\begin{proof} 
First, 
	we show that \ref{item:w=3-classification} implies \ref{item:w=3-uniqueness}, i.e.\ we compute the number $\#\Pht^{\width=3}(\cS)$ given the above list of combinatorial types of $(X,D,p)$. Fix a singularity type $\cS$ as in the above list. Since no singularity in $\cS$ is canonical, Lemma \ref{lem:no-deb} implies that no surface $\bar{X}\in \Pht^{\width=3}(\cS)$ has a descendant with elliptic boundary. Hence by \ref{item:w=3-classification}, the minimal log resolution $(X,D)$ of each $\bar{X}\in \Pht^{\width=3}(\cS)$ admits a $\P^1$-fibration $p$ as in the above list. We check directly that if $\cS\neq \mbox{\eqref{eq:25}}$ then $\cS$ appears exactly once on that list, so $\cS$ uniquely determines the combinatorial type $\check{\cS}$ of $(X,D,p)$. The exceptional type $\mbox{\eqref{eq:25}}$ from Example \ref{ex:ht=3_pair}  appears twice: in \ref{item:rivet_A} and in \ref{item:nu_3=1_c2} for $k=3$, so we get two combinatorial types $\check{\cS}$. In any case, by Lemma \ref{lem:w=3_uniqueness}\ref{item:ht=3_uniqueness} a type $\check{\cS}$ uniquely determines the isomorphism class of $\bar{X}$, so $\#\Pht^{\width=3}(\cS)\leq 1$ for $\cS\neq \mbox{\eqref{eq:25}}$ and $\#\Pht^{\width=3}(\mbox{\ref{eq:25}})\leq 2$. To see that the equalities hold, note that each $(X,D,p)$ in the above list exists and satisfies the inequality \eqref{eq:ld_phi_H}, so $(X,D)$ is a minimal resolution of a del Pezzo surface $\bar{X}$ of rank one (for $\cS=\mbox{\eqref{eq:25}}$ this was proved in Proposition \ref{prop:exception}\ref{item:exception_del-Pezzo}, the computation in other cases is analogous). Since type $\cS$ does not appear in the classification \cite{PaPe_ht_2} of del Pezzo surfaces of rank one and height at most $2$, we have 
	$\height(\bar{X})=3$ and $\width(\bar{X})=3$, i.e.\ $\bar{X}\in \Pht^{\width=3}(\cS)$. Thus $\#\Pht^{\width=3}(\cS)\geq 1$, and $\#\Pht^{\width=3}(\mbox{\ref{eq:25}})\geq 2$ by Proposition \ref{prop:exception}\ref{item:exception_non-iso}, as needed. 
	\smallskip
	
	Therefore, it remains to prove \ref{item:w=3-classification}, i.e.\ to verify the completeness of the above list of combinatorial types. We choose a witnessing $\P^1$-fibration $p$ as in Lemma \ref{lem:w=3_swaps} and keep Notation  \ref{not:phi_+_H}.	We consider each of the four cases of Lemma \ref{lem:w=3_swaps}\ref{item:w=3_additional-base-point} separately.
	
	\paragraph{Case \ref{lem:w=3_swaps}\ref{item:swap-to_A1+A2+A5_rivet}, see Figure \ref{fig:swaps_to_A1_A2_A5_rivet}.} We claim that for any decomposition $\psi= \gamma\circ \gamma_{+}$ as in Notation \ref{not:phi_+_H} we have
		\begin{equation}\label{eq:forks_i}
		A_h^{\gamma}\cap \Bs\gamma_{+}^{-1}\subseteq H_1^{\gamma},\quad
		A_{13}^{\gamma}\cap \Bs\gamma_{+}^{-1}\subseteq V_{1}^{\gamma},\quad
		A_{v}^{\gamma}\cap \Bs\gamma_{+}^{-1}\subseteq V_{2}^{\gamma},\quad 
		A_{32}^{\gamma}\cap \Bs\gamma_{+}^{-1}\subseteq H_{2}^{\gamma}.
	\end{equation}

	Suppose the first inclusion fails. Then by Lemma \ref{lem:w=3_uniqueness}\ref{item:w=3_no_C1}, the point $A_h^{\gamma}\cap (D_{\gamma}-H_{1}^{\gamma})$ is a base point of~$\gamma_{+}^{-1}$. Let $\eta$ be a composition of $\gamma$ with a blowup at this point. Then we can see from Figure \ref{fig:swaps_to_A1_A2_A5_rivet} that $H_{1}^{\eta}$ is branching in $D_{\eta}$, and the  twigs of $D_{\eta}$ meeting $H_{1}^{\eta}$ in $V_{1}^{\eta}$, $V_{3}^{\eta}$ have lengths at least $5$ and $2$, respectively. This is a contradiction with Lemma \ref{lem:admissible_forks}\ref{item:long-twig}. We leave an analogous verification of the remaining parts of \eqref{eq:forks_i} to the reader.
	\smallskip

	Suppose $\psi=\tau$, so $D$ is as in Figure \ref{fig:swaps_to_A1_A2_A5_rivet}. 
	Let $\delta\colon \P^{1}\times \P^{1}\map \P^2$ be a blowup at $p_{21}$ followed by the contraction of the proper transforms of $H_{1}$ and $V_{2}$. Let $L\subseteq X$ be the proper transform of the conic tangent to $\delta(V_{1})$, $\delta(H_{2})$ at $\delta(p_{13})$, $\delta(p_{32})$ and passing through $\delta(p_{33})$. Then $L$ meets $D$ only in the $(-3)$-curve, twice, so $L$ is an elliptic tie and $\bar{X}$ has a descendant with elliptic boundary, a contradiction. Thus $\psi=\tau\circ \tau_{+}$ for some $\tau_{+}\neq \id$. 
	\smallskip
	 	
	We will use an involution of $(\P^1\times \P^1,B)$ which maps $(V_1,V_2,V_3)$ to $(H_2,H_1,H_3)$. In the notation of Lemma \ref{lem:w=3_swaps}, it maps $(p_{13},v_{13},p_{21},h_{21},p_{32},h_{32},p_{33},v_{21})$ to $(p_{32},h_{32},p_{21},v_{21},p_{13},v_{13},p_{33},h_{21})$. Therefore, it lifts to an involution $\iota\in \Aut(X_{\tau},D_{\tau})$ such that:
	\begin{equation*}
		\iota(A_{33})=A_{33},\ \iota(V_3)=H_3,\quad
		\iota(A_h)=A_v,\quad
		\iota(A_{13})=A_{32}.
	\end{equation*}
	
	Fix $q\in \Bs\tau_{+}^{-1}$. By Lemma \ref{lem:w=3_uniqueness}\ref{item:w=3_no_C1}, $q$ is a common point of $D_{\tau}$ and some vertical $(-1)$-curve. Suppose $q\in A_{33}$. Applying $\iota$ if necessary, we can assume $q\in H_{3}$. Let $\gamma$ be a composition of $\tau$ with a blowup  at $q$. Then $D_{\gamma}=\langle 2;[2],[2],[2,\bs{3},2,\bs{2},2,\bs{2}]\rangle+[3]$. Here, the branching component is $V_3^{\gamma}$, and the subsequent bold numbers correspond to $H_3^{\gamma}$, $H_{2}^{\gamma}$ and $H_{1}^{\gamma}$ (recall from Section \ref{sec:notation} that the bold numbers always correspond to the horizontal components). Using Lemma \ref{lem:ld_formulas}, we compute $\ld(H_3^{\gamma})=\frac{1}{17}(\frac{1}{3}\cdot 2+5)=\frac{1}{3}$, $\ld(H_{2}^{\gamma})=\frac{1}{17}(\frac{1}{3}\cdot 8+3)=\frac{1}{3}$, $\ld(H_{1}^{\gamma})=\frac{1}{17}(\frac{1}{3}\cdot 14+1)=\frac{1}{3}$, so $\sum_{i=1}^{3}\ld(H_{i}^{\gamma})=1$,  a contradiction with inequality  \eqref{eq:ld-bound_ht=3}.
	\smallskip
	
	Consider the case $\psi(q)=p_{21}$, i.e.\ $q\in A_{v}\cup A_h$. Applying $\iota$ if necessary, we can assume $q\in A_{v}$. Write $\tau_{+}=\check{\gamma}\circ\gamma_{+}$, $\gamma=\tau\circ\check{\gamma}$, where $\Bs\check{\gamma}^{-1}\subseteq A_{v}$ and $\Bs\gamma_{+}^{-1}\cap A_{v}^{\gamma}=\emptyset$. Put $k=\rho(\check{\gamma})+2\geq 3$. By the third assertion in \eqref{eq:forks_i}, 
	$\check{\gamma}$ is a composition of $k-2$ blowups centered on the proper transforms of $V_{2}$. Now $D_{\gamma}=[2,\bs{2},k,\bs{2},2,\bs{2},2,2]+[(2)_{k-2},3]$. If $\gamma_{+}=\id$ then $D$ is as in \ref{item:rivet_A}.
	Assume $\gamma_{+}\neq \id$, and take $r\in \Bs\gamma^{-1}_{+}$. We claim that $r\in A_{32}$. Suppose the contrary. We have shown that $r\not \in A_{33}$, and $r\not\in A_{v}$ by the definition of $\gamma_{+}$, so $r\in A_{13}\cup A_h$. By \eqref{eq:forks_i}, we have $\{r\}=A_{13}\cap V_{1}$ or $\{r\}=A_h\cap H_{1}$. Let $\eta$ be a composition of $\gamma$ with a blowup at $r$.
	To get a contradiction with the inequality \eqref{eq:ld-bound_ht=3} we can assume $k=3$: indeed, we first  preform $k-3$ elementary vertical swaps over $v_{13}$ (beginning with a swap of $A_{v}^{\eta}$), and only then we apply \eqref{eq:ld-bound_ht=3}. Now for $k=3$ we have   $D_{\eta}=[2,2,\bs{2},3,\bs{2},3,\bs{2},2,2]+[2,3]$ or $[2,\bs{2},3,\bs{2},2,\bs{3},2,2]+[2,3,2]$, so log discrepancies of its horizontal components (referred to by bold numbers) add up to one. This gives the required contradiction with  \eqref{eq:ld-bound_ht=3}.
	
	Thus $\gamma_{+}$ is a sequence of $l-2\geq 1$ blowups over $A_{32}$. By \eqref{eq:forks_i} all these blowups are centered on the proper transforms of $H_2$. Thus $D=[2,\bs{2},k,\bs{l},2,\bs{2},(2)_{l}]+[(2)_{k-2},3]$. Now \eqref{eq:ld-bound_ht=3} gives $k=l=3$, so $D$ is as in \ref{item:rivet_AC}.
	\smallskip
	
	It remains to consider the case when $\tau_{+}^{-1}$ has no base point over $p_{21}$, so $q\in\Bs\tau_{+}^{-1}\subseteq A_{13}\cup A_{32}$. Applying~$\iota$ we can take $q\in A_{13}$. As before, write $\tau_{+}=\check{\gamma} \circ \gamma_{+}$, $\gamma=\tau\circ \check{\gamma}$, where $\Bs\check{\gamma}^{-1}\subseteq A_{13}$ and $A_{13}^{\gamma}\cap \Bs\gamma^{-1}_{+}=\emptyset$. Put $k=\rho(\check{\gamma})+2\geq 3$. By \eqref{eq:forks_i}, each blowup in the decomposition of $\check{\gamma}$ is centered on a proper transform of $V_{1}$. Hence $D_{\gamma}=[(2)_{k-1},\bs{2},2,\bs{2},k,\bs{2},2,2]+[3]$. Inequality \eqref{eq:ld-bound_ht=3} gives  $k\leq 4$, so if $\gamma_{+}=\id$ then $D$ is as in \ref{item:rivet_0}. Suppose $\gamma_{+}\neq \id$. Condition \eqref{eq:forks_i} and definition of $\gamma$ imply that $\gamma_{+}$ is a sequence of blowups over $A_{32}\cap H_{2}$. Now it suffices to show that \eqref{eq:ld-bound_ht=3} fails for $k=3$, $\rho(\gamma_{+})=1$. In this case, we have $D=[2,2,\bs{2},2,\bs{3},3,\bs{2},2,2,2]+[3]$; so $\sum_{i=1}^{3}\ld(H_{i})=1$; a contradiction, as needed.
	
	
	\paragraph{Case \ref{lem:w=3_swaps}\ref{item:swap-to_A1+A2+A5_no-rivet}, see Figure \ref{fig:swaps_to_A1_A2_A5_no-rivet}.} As before, the fact that each connected component of $D$ is an admissible chain or an admissible fork implies that for any decomposition $\psi=\gamma\circ \gamma_{+}$ as in Notation \ref{not:phi_+_H} we have:
	\begin{equation}\label{eq:forks_ii}
		A_{13}^{\gamma}\cap \Bs \gamma_{+}^{-1}\subseteq V_{1}^{\gamma},\quad
		A_{21}^{\gamma}\cap \Bs\gamma_{+}^{-1}\subseteq H_{1}^{\gamma},\quad
		A_{32}^{\gamma}\cap \Bs\gamma_{+}^{-1}\subseteq H_{2}^{\gamma}.
	\end{equation}
	If $\psi=\tau$ then as in Case \ref{lem:w=3_swaps}\ref{item:swap-to_A1+A2+A5_rivet} we conclude that $\bar{X}$ has a descendant with elliptic boundary. Indeed, let $\delta\colon \P^1\times \P^1\map \P^2$ be a blowup at $p_{13}$, followed by the contraction of the proper transforms of $V_{1}$, $H_{3}$. Now the required elliptic tie is a proper transform of the conic tangent to $\delta(H_{1})$, $\delta(H_{2})$ at $\delta(p_{21})$, $\delta(p_{32})$ and passing through $\delta(H_3)$. 
	Thus $\psi=\tau\circ\tau_{+}$ for some $\tau_{+}\neq \id$. Let $r$ be a base point of $\tau_{+}^{-1}$. 

	\begin{casesp}
	\litem{$r\in A_{21}$}\label{item:A21} Write $\tau_{+}=\check{\upsilon}\circ\upsilon_{+}$, $\upsilon=\tau\circ\check{\upsilon}$,  where $\Bs\check{\upsilon}\subseteq A_{21}$ and $A_{21}^{\upsilon}\cap \Bs\upsilon_{+}^{-1}=\emptyset$. Put $k=\rho(\check{\upsilon})+2\geq 3$. Condition \eqref{eq:forks_ii} implies that all blowups in the decomposition of $\check{\upsilon}$ are centered on the proper transforms of $H_1$. Now $D_{\upsilon}=[2,2,\bs{k},2,\bs{2},(2)_{k}]+[2,\bs{3}]$, so inequality \eqref{eq:ld-bound_ht=3} gives $k\leq 5$. If $\psi=\upsilon$ then  $D$ is as in \ref{item:nu_3=1_c2}. Assume $\psi\neq \upsilon$ and let $\gamma$ be a composition of $\upsilon$ with a blowup at some $q\in \Bs\upsilon_{+}^{-1}$. We claim that
	\begin{equation}\label{eq:forks_ii_bis}
		A_{23}^{\upsilon}\cap \Bs\upsilon_{+}^{-1}\subseteq H_{3}^{\upsilon},\quad
		A_{32}^{\upsilon}\cap \Bs\upsilon_{+}^{-1}=\emptyset, \quad 
		A_{33}^{\upsilon}\cap \Bs\upsilon_{+}^{-1}\subseteq V_{3}^{\upsilon}. 
	\end{equation}
	Suppose one of the above conditions fails. Using Lemma \ref{lem:w=3_uniqueness}\ref{item:w=3_no_C1} and condition \eqref{eq:forks_ii}, we can take the base point $q\in \Bs\upsilon_{+}^{-1}$ to be one of the points $A_{23}\cap V_{2}$, $A_{32}\cap H_{2}$ or  $A_{33}\cap H_{3}$. Then for $k=3$ we get $D_{\gamma}=[2,2,\bs{3},2,\bs{2},3,2,2]+[2,\bs{3},2]$,  
	$[2,2,2,\bs{3},2,\bs{3},2,2,2]+[2,\bs{3}]$ or 
	$\langle 2;[2],[2],[2,2,2,\bs{2},2,\bs{3}]\rangle+[2,\bs{4}]$; respectively. In any case, $\sum_{i=1}^{3}\ld(H_{i}^{\gamma})=1$, so inequality  \eqref{eq:ld-bound_ht=3} fails for $k=3$. Since we can reduce any $\gamma$ with $k\geq 3$ to one with $k=3$ by a vertical swap, we get a contradiction with Lemma \ref{lem:w=3_uniqueness}\ref{item:ld-bound_ht=3}. This proves \eqref{eq:forks_ii_bis}.
	
	By the definition of $\upsilon$, we have $q\not\in A_{21}$. If $q\in A_{23}$ then $\{q\}=A_{23}\cap H_{3}$ by \eqref{eq:forks_ii_bis}, so $V_{2}^{\gamma}$ is branching in a non-admissible fork in $D_{\gamma}$, which is impossible. Thus $q\not\in A_{21}\cup A_{23}$. Combining this with \eqref{eq:forks_ii} and \eqref{eq:forks_ii_bis} we infer that $\Bs\upsilon_{+}^{-1}\subseteq (A_{13}\cap V_{1})\cup (A_{33}\cap V_{3})$. Thus $\{q\}=A_{j3}\cap V_{j}$ for some $j\in \{1,3\}$. 
	
	If $k=4$ then $D_{\gamma}=[(2)_{4},\bs{2},3,\bs{4},2,2]+[2,2,\bs{3}]$ if $j=1$ and   $D_{\gamma}=[(2)_{4},\bs{2},2,\bs{4},3,2]+[2,\bs{3},2]$ if $j=3$. Again, $\sum_{i=1}^{3}\ld(H_{i}^{\gamma})=1$, so \eqref{eq:ld-bound_ht=3} fails for $k=4$. As before, reducing any case with $k\geq 4$ to one with $k=4$ via a vertical swap, we infer that $k=3$.
	
	Write $\gamma_{+}=\check{\eta}\circ\eta_{+}$, $\eta=\check{\eta}\circ\gamma$, where $\check{\eta}$ is a composition of $l+2\geq 3$ blowups at $q$ and its infinitely near points on the proper transforms of $V_{j}$, such that $A_{j3}^{\eta}\cap V_{j}^{\eta}\not\subseteq \Bs\eta_{+}^{-1}$. Then $D_{\eta}=[2,2,\bs{3},l,\bs{2},2,2,2]+[(2)_{l-1},\bs{3}]$ if $j=1$ and $[2,l,\bs{3},2,\bs{2},2,2,2]+[2,\bs{3},(2)_{l-2}]$ if $j=3$. By \eqref{eq:ld-bound_ht=3}, $l\in \{3,4\}$. If $\eta_{+}=\id$ then $D$ is as in \ref{item:nu_3=1_e1} or \ref{item:nu_3=1_e2}. 
	
	Suppose $\eta_{+}\neq \id$, and let $\theta$ be a composition of $\eta$ with a blowup at some $s\in\Bs\eta_{+}^{-1}$. We have $s\in A_{13}^{\eta}\cup A_{33}^{\eta}$ because $\Bs\upsilon_{+}^{-1}\subseteq A_{13}^{\upsilon}\cup A_{33}^{\upsilon}$; and by the definition of $\eta$ we have $s\not\in V_j$. Thus either $\{s\}=A_{i3}\cap V_{i}$ for $\{i,j\}=\{1,3\}$; or $\{s\}=A_{j3}\cap(D_{\eta}-V_{j})$. Suppose $\{s\}=A_{i3}\cap V_{i}$, so $s$ is not infinitely near to $q$. To get a contradiction with inequality \eqref{eq:ld-bound_ht=3} we can assume $l=3$. Now $D_{\theta}=[2,3,\bs{3},3,\bs{2},2,2,2]+[2,2,\bs{3},2]$, so $\sum_{i}\ld(H_i)=\frac{592}{649}<1$; a contradiction with \eqref{eq:ld-bound_ht=3}. Thus $\{s\}=A_{j3}\cap(D_{\eta}-V_{j})$. Condition \eqref{eq:forks_ii} implies that $j\neq 1$, so $j=3$. Now $D_{\theta}=\langle l;[2],[2];[2,2,2,\bs{2},2,\bs{3}]\rangle+[2,\bs{3},(2)_{l-3},3]$; again contrary to \eqref{eq:ld-bound_ht=3}.
	
	\litem{$A_{21}\cap \Bs\tau_{+}^{-1}=\emptyset$} We will use an involution of $(\P^1\times \P^1,B)$ which maps  $(p_{13},p_{21},p_{32})$ to $(p_{13},p_{32},p_{21})$ and preserves the projections. Like before, it lifts to an involution $\iota\in \Aut(X_{\tau},D_{\tau})$ such that 
	\begin{equation*}
		\iota(A_{21})=A_{32},\quad  \iota(A_{23})=A_{33}. 
	\end{equation*}
	If $A_{32}\cap \Bs\tau_{+}^{-1}\neq \emptyset$ then applying $\iota$ we are back in Case \ref{item:A21} above, so we can assume $A_{32}\cap \Bs\tau_{+}^{-1}=\emptyset$.
	\smallskip
	
	Let again $r$ be a base point of $\tau_{+}^{-1}$. Suppose $\{r\}=A_{33}\cap H_{3}$. Let $\upsilon$ be a composition of $\tau$ with a blowup at $r$. If $\psi=\upsilon$ then $\bar{X}$ again has a descendant with elliptic boundary, with the same elliptic tie as in case $\psi=\tau$; a contradiction. Thus $\psi=\upsilon\circ \upsilon_{+}$ for some $\upsilon_{+}\neq \id$. Let $\gamma$ be a composition of $\upsilon$ with a blowup at $q\in \Bs\upsilon_{+}^{-1}$. We have $\{q\}\neq A_{23}\cap H_{3}$, since otherwise both $V_{2}^{\gamma}$ and $V_{3}^{\gamma}$ would be branching in the same connected component of $D_{\gamma}$, which is impossible. Similarly, $q\not \in A_{33}$, because otherwise $V_{3}^{\gamma}$ would be branching in a non-admissible fork. By assumption, $q\not\in A_{21}\cup A_{32}$, so condition \eqref{eq:forks_ii} shows that $\{q\}=A_{13}\cap V_{1}$ or  $A_{23}\cap V_{2}$. Now $D_{\gamma}=\langle 2;[2],[2],[2,2,\bs{2},3,\bs{2}]\rangle+[2,2,\bs{4}]$ or 
	$\langle 2;[2],[2],[2,3,\bs{2},2,\bs{2}]\rangle+[2,\bs{4},2]$ respectively, where, as usual, the bold numbers correspond to sections in $D_{\gamma}$. Their log discrepancies add up to one, a contradiction with \eqref{eq:ld-bound_ht=3}.  
	Thus $A_{33}\cap H_{3}\not\subseteq \Bs\tau_{+}^{-1}$, so $A_{33}\cap \Bs\tau_{+}^{-1}\subseteq V_{3}$. Applying $\iota$, we get $A_{j3}\cap \Bs\tau_{+}^{-1}\subseteq V_{j}$ for both $j\in \{2,3\}$. 
	\smallskip
	
	Consider the case $A_{33}\cap \Bs\tau_{+}^{-1}\neq \emptyset$. Take $r\in A_{33}\cap \Bs\tau_{+}^{-1}$, so $\{r\}=A_{33}\cap V_{3}$. Write $\psi=\upsilon\circ \upsilon_{+}$, where $\upsilon$ is a composition of $\tau$ with $k-2\geq 1$ blowups over $r$; and $V_{3}^{\upsilon}\cap \Bs\upsilon_{+}^{-1}=\emptyset$. If $\upsilon_{+}=\id$ then $D$ is as in \ref{item:nu_3=1_s1}. 
	
	Assume $\upsilon_{+}\neq \id$, fix $q\in \Bs\upsilon^{-1}$ and let $\gamma$ be a composition of $\upsilon$ with a blowup at $q$. Suppose $\{q\}=A_{23}\cap V_{2}$. Then for $k=3$ we get $D_{\gamma}=[2,3,\bs{2},2,\bs{2},3,2]+\langle \bs{3};[2],[2],[2]\rangle$, hence $\sum_{i}\ld(H_{i}^{\gamma})=1$. As before, we conclude that the inequality  \eqref{eq:ld-bound_ht=3} fails for any $k\geq 3$; a contradiction. 
	
	Hence $\Bs\upsilon_{+}^{-1}\subseteq A_{13}\cup A_{33}$. Assume $q\in A_{33}$, so $\{q\}=A_{33}\cap (D_{\upsilon}-V_{3})$ by the definition of $\upsilon$. If $\psi=\gamma$ then $D$ is as in \ref{item:nu=3_fork_s1=1}. Suppose $\gamma_{+}\neq \id$, and let $\eta_{+}$ be a composition of $\gamma_{+}$ with a blowup at some $s\in \Bs\gamma_{+}^{-1}$. We have $s\not\in A_{33}$, since otherwise $V_{3}^{\gamma}$ would meet a twig of $D_{\gamma}$ of length $5$ (the one containing $H_{1}$), and only one twig $[2]$; contrary to Lemma \ref{lem:admissible_forks}\ref{item:long-twig}. Thus $s\in A_{13}$, so condition \eqref{eq:forks_ii} gives $\{s\}=A_{13}\cap V_{1}$. Now $D_{\eta}=\langle k;[2],[2],[2,2,\bs{2},3,\bs{2}]\rangle+[2,2,\bs{3},(2)_{k-3},3]$; a contradiction with \eqref{eq:ld-bound_ht=3}.
	
	Thus we can assume $\Bs\upsilon_{+}^{-1}\subseteq A_{13}$. Condition \eqref{eq:forks_ii} implies that $\upsilon_{+}$ is a composition of $l-2\geq 1$ blowups on the proper transforms of $V_{1}$, so $D=[2,k,\bs{2},l,\bs{2},2,2]+[(2)_{l-1},\bs{3},(2)_{k-2}]$. By \eqref{eq:ld-bound_ht=3}, $k=l=3$, so $D$ is as in \ref{item:nu=3_C}. 
	\smallskip
	
	Eventually, assume $A_{33}\cap \Bs\tau_{+}^{-1}=\emptyset$. Applying $\iota$, we can assume $A_{23}\cap\Bs\tau_{+}^{-1}=\emptyset$, too, so $\Bs\tau_{+}^{-1}\subseteq A_{13}$. By \eqref{eq:forks_ii}, $\tau_{+}$ is a composition of some $k-2\geq 1$ blowups on the proper transforms of $V_{1}$. Thus $D$ is as in \ref{item:nu_3=1_s3}.
\end{casesp}
	

	\paragraph{Case \ref{lem:w=3_swaps}\ref{item:swap-to_2A4-v}, see Figure \ref{fig:swap-to_2A4-v}.} For any decomposition $\psi=\gamma\circ \gamma_{+}$ as in Notation \ref{not:phi_+_H}, let $U_{\gamma}$ and  $W_{\gamma}$ be the connected component of $D_{\gamma}$ containing $H_{1}^{\gamma}+H_{2}^{\gamma}$ and $H_{3}^{\gamma}$, respectively, cf.\ Lemma \ref{lem:w=3_uniqueness}\ref{item:w=3_Sing=2}. 
	We have 
	\begin{equation}\label{eq:witnesses_iii}
		\Bs\tau_{+}^{-1}\cap (A_{13}^{\tau}\cup A_{23}^{\tau}\cup A_{32}^{\tau})\neq \emptyset
		\quad\mbox{and}\quad
		\Bs\tau_{+}^{-1}\cap (A_{13}^{\tau}\cup A_{22}^{\tau}\cup A_{31}^{\tau})\neq \emptyset.
	\end{equation}
	Indeed, if the first condition fails then the pullback of $W_{\tau}+A_{31}^{\tau}=\langle 2;[2],[2],[1,2,2]\rangle$ to $X$ supports a fiber of a $\P^1$-fibration of height $2$ with respect to $D$; and if the second one fails then so does $T+A_{23}^{\tau}+U_{\tau}+A_{32}^{\tau}=[2,1,3,2,2,2,1]$, where $T$ is the $(-2)$-tip of $W_{\tau}$ meeting $A_{23}^{\tau}$. This proves \ref{eq:witnesses_iii}.
	\smallskip
	
	Assume that $\tau_{+}^{-1}$ has a base point $q\in A_{32}$. Then $q\in V_{3}$, because otherwise both $V_{3}^{\psi}$ and $H_{3}^{\psi}$ would be branching in $W^{\psi}$, which is impossible. Condition \eqref{eq:witnesses_iii} implies that $\tau_{+}^{-1}$ has another base point $q'\in A_{13}\cup A_{22}\cup A_{31}$. Let $\upsilon$ be a composition of $\tau$ with a blowup at $q$ and $q'$. 
	
	Suppose $q'\in A_{13}\cup A_{22}$. Then $D_{\upsilon}$ equals: 
	$\langle \bs{2};[2],[2],[4,\bs{2},2]\rangle +\langle \bs{2};[2],[2],[2,3]\rangle$ if $\{q'\}=A_{22}\cap V_{2}$; 
	$[2,\bs{3},2,\bs{2},3,2]+\langle \bs{2};[2],[2],[2,3]\rangle$ if $\{q'\}=A_{22}\cap H_{2}$;  
	$[2,\bs{2},3,\bs{2},3]+\langle \bs{2};[2],[2,2],[2,3]\rangle$ if $\{q'\}=A_{13}\cap V_1$; and 
	$\langle 2;[2],[2,\bs{2}],[3,\bs{2}]\rangle+\langle \bs{2};[2],[3],[2,3]\rangle$ if $\{q'\}=A_{13}\cap W_{\tau}$. As usual, bold numbers correspond to sections in $D_{\upsilon}$. Their log discrepancies add up to one or, in the latter case, to $\tfrac{122}{161}<1$; a contradiction with inequality \eqref{eq:ld-bound_ht=3}. 
	
	Thus $q'\in A_{31}$. If $q'\in W_{\tau}$ then $D_{\upsilon}=\langle \bs{2};[2],[2,\bs{2},2],[3]\rangle+\langle \bs{2};[2],[2],[3,3]\rangle$, so once again we get $\sum_{i}\ld(H_i^{\upsilon})=1$, contrary to \eqref{eq:ld_phi_H}. We conclude that $\{q'\}=A_{31}\cap U_{\tau}$, so $D_{\upsilon}=[2,\bs{2},2,\bs{3},3]+\langle \bs{2};[2],[2],[2,2,3]\rangle$. 
	
	If $\psi=\upsilon$ then $D$ is as in \ref{item:both_B}. Suppose $\psi=\upsilon\circ\upsilon_{+}$,  $\upsilon_{+}\neq \id$. Fix $r\in \Bs\upsilon_{+}^{-1}$ and let $\gamma$ be a composition of $\upsilon$ with a blowup at $r$. We have shown that $r\not\in A_{13}\cup A_{22}$, so $r\in A_{23}\cup A_{31}\cup A_{32}$. If $r\in A_{23}$ then $W_{\gamma}$ is a non-admissible fork, which is impossible. If $r\in A_{32}$ then since $V_3^{\gamma}$ cannot become branching in $W_{\gamma}$, we have $\{r\}=A_{32}\cap V_{3}$, so  $D_{\gamma}=[2,2,\bs{2},2,\bs{3},3]+\langle \bs{2};[2],[2],[2,2,4]\rangle$ and $\sum_{i}\ld(H_{i}^{\gamma})=1$, contrary to \eqref{eq:ld-bound_ht=3}. Thus $r\in A_{31}$. We have $D_{\gamma}=[2,\bs{2},2,\bs{4},3]+\langle \bs{2};[2],[2],[2,2,2,3]\rangle$ if $r\in H_{1}$ and $D_{\gamma}=\langle \bs{3};[2],[3],[2,\bs{2},2]\rangle+\langle \bs{2};[2],[2],[3,2,3]\rangle$ otherwise; so $\sum_{i}\ld(H_{i}^{\gamma})=1$ or $\tfrac{87}{119}<1$, respectively; a contradiction with \eqref{eq:ld-bound_ht=3}.
	\smallskip
	
	Thus we can assume $\Bs\tau_{+}^{-1}\cap A_{32}=\emptyset$. By condition \eqref{eq:witnesses_iii}, $\tau_{+}^{-1}$ has a base point $q\in A_{j3}$ for some $j\in \{1,2\}$. 
	The twigs of $W_{\psi}$ meeting $A_{j3}^{\psi}$ and $A_{31}^{\psi}$ are not of type $[2]$, so the remaining one is, i.e.\   $\Bs\tau_{+}^{-1}\cap A_{i3}=\emptyset$ for $\{i,j\}=\{1,2\}$. Moreover, $W_{\psi}$ is not a $(-2)$-fork, as otherwise $A_{22}^{\psi}$ would be an elliptic tie.
	
	\begin{casesp}
	\litem{$q\in U_{\tau}$} After a blowup at $q$ we have $W=\langle \bs{2},[2],[2,2],[2,2]\rangle$, so since $W_{\psi}$ is not a $(-2)$-fork, $\tau_{+}^{-1}$ has a base point $q'$ which  either is $A_{31}^{\tau}\cap W_{\tau}$, or is infinitely near to $q$, on the proper transform of $W_{\tau}$. Write $\psi=\gamma\circ \gamma_{+}$, where $\gamma$ is a composition of $\tau$ with a blowup at $q$ and $q'$. Then $W^{\gamma}=\langle \bs{2},[2],[2,2],[3,2]\rangle$, and  since $W^{\psi}$ is an admissible fork, $\gamma_{+}$ is an isomorphism near $W^{\psi}$, i.e.\ $\Bs\gamma_{+}\subseteq A_{22}^{\gamma}$. The inequality \eqref{eq:ld-bound_ht=3}  gives 
	\begin{equation}\label{eq:67}\ld(H_{1})+\ld(H_2)>1-\ld(H_3)=\tfrac{6}{7}.\end{equation}  
	
	Suppose $q\in A_{13}$. Then $U_{\gamma}=\langle \bs{2};[2],[3],[\bs{2},3]\rangle$ if $q'\in A_{31}$ and $\langle 3;[\bs{2}],[2],[3,\bs{2}]\rangle$ otherwise, so $\ld(H_1^{\gamma})+\ld(H_{2}^{\gamma})$ equals $\frac{15}{23}<\frac{6}{7}$ or $\tfrac{6}{7}$, contrary to \eqref{eq:67}. Hence $q\in A_{23}$. Like before, if $q'\in A_{31}$ then $U_{\gamma}=\langle \bs{2};[2],[4],[\bs{2},2]\rangle$, so $\ld(H_{1}^{\gamma})+\ld(H_{2}^{\gamma})=\tfrac{6}{7}$; a contradiction with \eqref{eq:67}. Thus $q'$ is infinitely near to $q$, and $U_{\gamma}=[2,4,\bs{2},2,\bs{2}]$. By \eqref{eq:witnesses_iii}, $\gamma_{+}^{-1}$ has a base point $s\in A_{22}^{\gamma}$. If $s\in H_{2}$ then after a blowup at $s$ we get $U=\langle 4;[2],[2],[\bs{3},2,\bs{2}]\rangle$, so $\ld(H_1)+\ld(H_2)=\frac{5}{8}<\frac{6}{7}$, contrary to \eqref{eq:67}. Hence $s\in V_{2}$. Write $\psi=\eta\circ\eta_{+}$, where $\eta$ is a composition of $\gamma$ with $k-4\geq 1$ blowups at $s$ and its infinitely near points on the proper transform of $V_2$ and $\Bs\eta_{+}^{-1}\cap V_{2}^{\eta}=\emptyset$. Then $U_{\eta}=[2,k,\bs{2},2,\bs{2},(2)_{k-4}]$, and \eqref{eq:67} implies that $k=5$. If $\eta_{+}=\id$ then $D$ is as in \ref{item:ht=3_XY_b=3_v} for $T=[3,2]$. Suppose  $\eta_{+}\neq \id$. Then the unique base point of $\eta_{+}^{-1}$ is $A_{22}^{\eta}\cap (U_{\eta}-V_{2}^{\eta})$. After a blowup at this point we get $U=\langle 5;[2],[2],[3,\bs{2},2,\bs{2}]\rangle$, so $\ld(H_{1})+\ld(H_{2})=\tfrac{14}{29}<\tfrac{6}{7}$; a contradiction with \eqref{eq:67}.
	
	\litem{$\Bs\tau_{+}^{-1}\cap U_{\tau}\subseteq A_{22}^{\tau}\cup A_{31}^{\tau}$} Write $\psi=\upsilon\circ \upsilon_{+}$, where $\upsilon$ is a composition of $\tau$ with a blowup at $q\in W_{\tau}$. 
	
	Suppose $q\in A_{13}$, so $\Bs\tau_{+}^{-1}\cap A_{23}=\emptyset$. If $\upsilon_{+}=\id$ then denoting by $T$ the component of $W_{\psi}$ meeting $A_{23}$ we see that the pullback of $U_{\psi}+A_{23}+T=\langle 2;[2],[2],[2,1,3,2]\rangle$ to $X$ supports a fiber of a $\P^1$-fibration of height $2$ with respect to $D$, which is impossible. Thus $\upsilon_{+}^{-1}\neq \id$. Let $\gamma$ be a composition of $\upsilon$ with a blowup at $q'\in \Bs\upsilon_{+}^{-1}$. We have $q'\not\in A_{22}$ by the admissibility of $U_{\gamma}$. 
	Suppose $q'\in A_{31}$. Then $q'\in H_1$, as otherwise both $V_{1}^{\gamma}$ and $H_{1}^{\gamma}$ would be branching in $U_{\gamma}$. Now $D_{\gamma}=\langle 2;[2],[3,\bs{3}],[\bs{2}]\rangle + \langle \bs{2};[2],[3],[2,2,2]\rangle$, so $\sum_{i}\ld(H_{i}^{\gamma})=1$; contrary to \eqref{eq:ld-bound_ht=3}.  
	Thus $q'\in A_{13}$, so $D_{\eta}$ equals  
	$\langle 2;[3],[3,\bs{2}],[\bs{2}]\rangle+\langle \bs{2};[2],[2,3],[2,2]\rangle$ if $q'\in U_{\upsilon}$, or
	$\langle 2;[2,2],[3,\bs{2}],[\bs{2}]\rangle+\langle \bs{2};[2],[4],[2,2]\rangle$ if $q'\in W_{\upsilon}$. Now $\sum_{i}\ld(H_{i}^{\gamma})=\tfrac{186}{221}$ or $1$, respectively, contrary to \eqref{eq:ld-bound_ht=3}.
	\smallskip
	
	We conclude that $\{q\}=A_{23}\cap W_{\tau}$; hence $A_{13}\cap\Bs\tau_{+}^{-1}=\emptyset$. Condition \eqref{eq:witnesses_iii} implies that $\tau_{+}$ has another base point $q'\in A_{22}\cup A_{31}$. Let $\gamma$ be a composition of $\tau$ with a  blowup at $q,q'$. If $\{q'\}=A_{31}\cap W_{\tau}$ then $D_{\gamma}=\langle \bs{2};[2],[2,3],[\bs{2},2]\rangle+\langle \bs{2};[2],[3],[3,2]\rangle$, so $\sum_{i}\ld(H_{i}^{\gamma})=\tfrac{183}{221}<1$, contrary to \eqref{eq:ld-bound_ht=3}. Similarly, if $\{q'\}=A_{22}\cap H_{2}$ then $D_{\gamma}=\langle 3;[2],[2],[\bs{3},2,\bs{2}]\rangle+\langle \bs{2};[2],[3],[2,2]\rangle$ and $\sum_{i}\ld(H_{i}^{\gamma})=1$, again contrary to \eqref{eq:ld-bound_ht=3}.

	Thus $q'$ equals $A_{31}\cap H_{1}$ or $A_{22}\cap V_{2}$. Consider the first case. Then $W_{\gamma}=\langle \bs{2},[2],[3],[2,2,2]\rangle$. If $\psi=\gamma$ then $D$ is as in \ref{item:ht=3_YG}. Suppose $\psi\neq \gamma$, and let $\eta$ be a composition of $\gamma$ with a blowup at some $s\in \Bs\gamma_{+}^{-1}$. Since $W_{\eta}$ is an admissible fork, we have $\{s\}=A_{31}\cap H_1$ or $s\in A_{22}$, and in the latter case we have shown that in fact $\{s\}= A_{22}\cap V_{2}$. Now $D_{\eta}=[2,3,\bs{4},2,\bs{2}]+\langle \bs{2},[2],[3],[2,2,2,2]\rangle$ or $[2,4,\bs{3},2,\bs{2},2]+\langle\bs{2},[2],[3],[2,2,2]\rangle$, respectively. In any case, $\sum_{i}\ld(H_{i}^{\eta})=1$; contrary to \eqref{eq:ld-bound_ht=3}. 
	
	It remains to consider the case $\{q'\}=A_{22}\cap V_{2}$, $\Bs\gamma_{+}^{-1}\subseteq A_{22}\cup A_{23}$. Write $\psi=\eta\circ\eta_{+}$, where $\eta$ is a composition of $\gamma$ with $k-3\geq 1$ blowups at $q'$ and its infinitely near points on the proper transforms of $V_{2}$, and $\Bs\eta_{+}^{-1}\cap V_{2}^{\eta}\cap A_{22}^{\eta}=\emptyset$. Now $D_{\eta}=[2,k,\bs{2},2,\bs{2},(2)_{k-3}]+\langle\bs{2};[2],[3],[2,2]\rangle$. Inequality \eqref{eq:ld-bound_ht=3} gives $k\leq 6$. 
	
	If $\psi=\eta$ then $D$ is as in \ref{item:ht=3_YDYD}. Assume $\psi\neq \eta$. Take $s\in \Bs\eta^{-1}$ and let $\theta$ be a composition of $\eta$ with a blowup at $s$. Recall that $s\in A_{22}\cup A_{23}$. If $s\in A_{22}$ then $s\not\in V_{2}^{\eta}$ by the definition of $\eta$, so $D_{\theta}=\langle k;[2],[2],[3,(2)_{k-4},\bs{2},2,\bs{2}]\rangle+\langle\bs{2};[2],[3],[2,2]\rangle$ and $\sum_{i}\ld(H_{i}^{\theta})=\tfrac{2}{k}+\tfrac{1}{3}<1$; contrary to \eqref{eq:ld-bound_ht=3}. Hence $\Bs\eta_{+}^{-1}\subseteq A_{23}$. If $s\in U_{\eta}$ then $D_{\theta}=[3,k,\bs{2},2,\bs{2},(2)_{k-3}]+\langle\bs{2};[2],[2,3],[2,2]\rangle$, so $k=4$ by \eqref{eq:ld-bound_ht=3}, and $\psi=\theta$ since $W_{\psi}$ is admissible. Hence $D$ is as in \ref{item:ht=3_XY_b=3_v} for $T=[2,3]$. Similarly, if $s\in W_{\eta}$ then by the admissibility of $W$ we get $D=[(2)_{l-2},k,\bs{2},2,\bs{2},(2)_{k-3}]+\langle\bs{2};[2],[l],[2,2]\rangle$, where $l=\rho(\eta_{+})+3\in \{4,5\}$. By \eqref{eq:ld-bound_ht=3} we get $k=4$ again, so $D$ is as in \ref{item:ht=3_XY_b=3_v} for $T=[l]$.
	\end{casesp}

	\paragraph{Case \ref{lem:w=3_swaps}\ref{item:swap-to_2A4-b}, see Figure \ref{fig:swap-to_2A4-b}.}
	Like before, for any decomposition $\psi=\gamma\circ\gamma_{+}$ as in Notation \ref{not:phi_+_H} we denote by $U_{\gamma}$, $W_{\gamma}$ the connected components of $D_{\gamma}$ containing $H_{1}^{\gamma}+H_{2}^{\gamma}$ and $H_{3}^{\gamma}$, respectively. We have 
	\begin{equation}\label{eq:witnesses_iv}
		\Bs\tau_{+}^{-1}\cap (A_{13}^{\tau}\cup A_{22}^{\tau}\cup A_{31}^{\tau})\neq \emptyset
	\end{equation}
	since otherwise the pullback of $A_{23}^{\tau}+U_{\tau}+A_{32}^{\tau}=[1,(2)_{5},1]$ to $X$ would induce a $\P^1$-fibration of height $2$ with respect to $D$, which is impossible. Having considered cases \ref{lem:w=3_swaps}\ref{item:swap-to_A1+A2+A5_rivet}--\ref{item:swap-to_2A4-v} above, we now can and will assume that any morphism $\psi\colon X\to \P^1\times \P^1$ as in Lemma \ref{lem:w=3_swaps}\ref{item:w=3_additional-base-point} falls into the current case \ref{lem:w=3_swaps}\ref{item:swap-to_2A4-b}.
	
	Denote by $C$ the proper transform of the curve of type $(1,1)$ on $\P^1\times \P^1$ passing through $p_{13},p_{22},$ and $p_{31}$.

\begin{claim}\label{cl:A32_a}
	We have $\Bs\tau_{+}^{-1}\cap A_{32}^{\tau}= \emptyset$, or $\Bs\tau_{+}^{-1}\cap (A_{22}^{\tau}\cup A_{31}^{\tau})=\emptyset$, or $\Bs\tau_{+}^{-1}\cap (A_{22}^{\tau}\cup A_{13}^{\tau})=\emptyset$. 
\end{claim}
\begin{proof}
	Assume that $A_{32}$ contains a base point of $\tau_{+}^{-1}$, say $q$. If $q\in H_2$ then the composition of $\psi$ with the involution switching the factors of $\P^1\times \P^1$ (i.e.\ mapping $\bar{V}_i$ to $\bar{H}_i$ for $i=1,2,3$) is as in case \ref{lem:w=3_swaps}\ref{item:swap-to_2A4-v}, contrary to our assumption. Hence $\{q\}=V_{3}\cap A_{32}$. Condition \eqref{eq:witnesses_iv} shows that $A_{13} \cup A_{22}\cup A_{31}$ contains a base point of $\tau_{+}^{-1}$, say $r$. Let $\gamma$ be a composition of $\tau$ with a blowup at $q$ and $r$.
	
	Suppose $r\in A_{22}^{\tau}$. Switching the factors of $\P^1\times \P^1$, if needed, we can assume that $\{r\}=A_{22}\cap V_{2}$. Now $D_{\gamma}=\langle \bs{2},[2],[2],[2,3,\bs{2},2]\rangle+[2,\bs{3},3,2]$, so   $\sum_{i}\ld(H_{i}^{\gamma})=1$, contrary to \eqref{eq:ld-bound_ht=3}. 
	
	Thus $A_{22}\cap \Bs\tau_{+}^{-1}=\emptyset$, so $r\in A_{13}\cup A_{31}$. As before, switching the factors of $\P^1\times \P^1$, if needed, we can assume that $r\in A_{13}$. 
	Suppose $r\in W_{\tau}$. Then $U_{\gamma}=\langle 2;[2],[2,\bs{2}],[2,2,\bs{2}]\rangle$.  	 The admissibility of $U_{\psi}$ implies that either $\gamma_{+}=\id$ or $\gamma_{+}$ is a single blowup at $A_{23}\cap H_{3}$. Thus $W_{\psi}=[3,\bs{k},3,2]$ for some $k\in\{3,4\}$; $C\cdot D=2$ and $C$ meets both tips of $W_{\psi}$. Thus $C$ is an elliptic tie, so $\bar{X}$ has a descendant with elliptic boundary; a contradiction.  
	
	Hence $\{r\}=A_{13}\cap V_1$. Suppose the claim fails, so $\gamma_{+}^{-1}$ has a base point $s\in A_{31}^{\gamma}$. Let $\eta$ be a composition of $\gamma$ with a blowup at $s$. If $s\in W_{\gamma}$ then $U_{\eta}$ is a non-admissible fork $\langle \bs{2},[2],[2,2],[2,\bs{2},3]\rangle$, which is false. Hence $s\in U_{\gamma}$ and $D_{\eta}=[2,2,\bs{3},3,\bs{2},2]+[2,2,\bs{3},3,2,2]$, so  $\sum_{i}\ld(H_{i}^{\eta})=1$, a contradiction with the inequality \eqref{eq:ld-bound_ht=3}.
\end{proof}

\begin{claim}\label{cl:flipping}
	We can further assume that the following two conditions hold:
	\begin{enumerate}
		\item\label{item:A32} $\Bs\tau_{+}^{-1}\cap A_{32}^{\tau}=\emptyset$,
		\item\label{item:U} $\Bs\tau_{+}^{-1}\cap A_{22}^{\tau}\neq \emptyset$ or $A_{13}^{\tau}\cap V_1^{\tau}\subseteq \Bs\tau_{+}^{-1}$.
	\end{enumerate}
\end{claim}
\begin{proof}
	If $\Bs\tau_{+}^{-1}\cap A_{22}^{\tau} \neq \emptyset$ then part \ref{item:A32} holds by Claim \ref{cl:A32_a}, and part \ref{item:U} holds automatically. So we can assume $A_{22}^{\tau}\cap \Bs\tau_{+}^{-1}=\emptyset$. If $p_{ij}\in \Bs\psi^{-1}$, denote by $T_{ij}$ the proper transform of  the first exceptional curve over $p_{ij}$. 
	
	Consider the $\P^1$-fibration induced by the pencil of curves of type $(1,1)$ passing through $p_{22}$ and $p_{31}$. With respect to this new $\P^{1}$-fibration, we have 
	\begin{equation*}
		D_{\tau}=[2\dec{1},2,2\dec{2},\bs{2}\dec{3},2\dec{4}]+[2\dec{3,5},\bs{3}\dec{1},2\dec{4},\bs{2}\dec{2,5}].
	\end{equation*} 
		The bold numbers refer to the curves $V_1$, $H_3$ and $T_{31}$, which are now horizontal. Decorations $\dec{1},\dots,\dec{5}$ indicate the common points of $D_{\tau}$ with the $(-1)$-curves $A_{23}$, $A_{31}$, $A_{13}$,  $A_{32}$, and $C$: those curves are now vertical. Let $\upsilon\colon X_{\tau}\to \P^1\times\P^1$ be the contraction  of $(A_{23}+T_{23}+V_2)+A_{31}+A_{13}+(A_{32}+H_{2})+C$. Choose coordinates on the target $\P^1\times \P^1$ so that $\bar{H}_1=\upsilon(H_3)$, $\bar{H}_{2}=\upsilon(T_{31})$, $\bar{H}_{3}=\upsilon(V_1)$ and $\bar{V}_1=\upsilon(V_3)$, $\bar{V}_2=\upsilon(T_{13})$, $\bar{V}_{3}=\upsilon(H_1)$. Write $\upsilon=\hat{\phi}\circ\check{\upsilon}$, where $\check{\upsilon}\colon X_{\tau}\to X_{\hat{\phi}}$ is the contraction of $A_{23}$. Then $\hat{\phi}\colon X_{\hat{\phi}}\to \P^1\times \P^1$ is as in Example  \ref{ex:w=3}\ref{item:ht=3_2A4}, with the images of $T_{23}$, $A_{31}$, $A_{13}$, $A_{32}$ and $C$  playing the role of $A_{31}$, $A_{32}$, $A_{23}$, $A_{13}$ and $A_{22}$, respectively. 
		We note that $\iota\circ\hat{\phi}$ is as in  Example  \ref{ex:w=3}\ref{item:ht=3_2A4}, too, where $\iota$ is the involution mapping $\bar{V}_i$ to $\bar{H}_i$ for $i=1,2,3$. 
		
		Put  $\hat{\psi}=\upsilon \circ\tau_{+}$. By assumption $\Bs\tau_{+}^{-1}\cap A_{22}^{\tau} =\emptyset$, so \eqref{eq:witnesses_iv} shows that  $A_{13}^{\tau}\cup A_{31}^{\tau}$ contains a base point of $\tau_{+}^{-1}$, call it $q$. Put $\tilde{\psi}=\hat{\psi}$ if $q\in A_{13}^{\tau}$ and $\tilde{\psi}=\iota\circ\hat{\psi}$ if $q\in A_{31}^{\tau}$.  If $q\in W_{\tau}$ then $\tilde{\psi}$ is as in case \ref{lem:w=3_swaps}\ref{item:swap-to_2A4-v}, contrary to our assumption. Hence $q\in U_{\tau}$, so $\tilde{\psi}$ is again in the current case \ref{lem:w=3_swaps}\ref{item:swap-to_2A4-b}. If $q\in A_{31}^{\tau}$ then $\tilde{\psi}^{-1}(p_{13})$ has at least $3$ components, namely $A_{23}$, $T_{23}$, and $V_2$, so replacing $\psi$ by $\tilde{\psi}$, if needed, we can assume $q\in A_{13}^{\tau}$ (so $\tilde{\psi}=\hat{\psi}$). 
		
		If $\Bs\tau_{+}^{-1}\cap A_{32}^{\tau}\neq \emptyset$ then Claim \ref{cl:A32_a} implies that $\Bs\tau_{+}^{-1}\cap A_{31}^{\tau}=\emptyset$, so $\hat{\psi}^{-1}$ has no base points infinitely near to $p_{32}$. Thus replacing $\psi$ by $\hat{\psi}$, if needed, we can assume that $\Bs\tau_{+}^{-1}\cap A_{32}^{\tau}= \emptyset$, i.e.\ \ref{item:A32}  holds. Since $\{q\}=A_{13}^{\tau}\cap U_{\tau}=A_{13}^{\tau}\cap V_{1}^{\tau}$ we get that \ref{item:U} holds, too, as needed.
	\end{proof}
	
	Write $\psi=\gamma\circ\gamma_{+}$ in such a way that $U_{\gamma}$, $W_{\gamma}$ are chains and $\rho(\gamma)$ is maximal possible. Since $\Bs\tau_{+}^{-1}\cap A_{32}^{\tau}=\emptyset$ by Claim \ref{cl:flipping}\ref{item:A32}, the morphism $\gamma$ is a composition of one blowup at $p_{32}$, and,  for some $a,b,c,d\geq 2$, exactly  $a,b-1,c,d$ blowups at $p_{13}$, $p_{22}$, $p_{23}$, $p_{31}$ and their infinitely near points  on the proper transforms of $\bar{V}_{1}$, $\bar{V}_{2}$, $\bar{H}_{3}$, $\bar{H}_{1}$; respectively. Moreover, those proper transforms contain no base points of $\gamma_{+}^{-1}$. Now 
	\begin{equation*}
	D_{\gamma}=\ldec{3}[(2)_{c-1},b\dec{2},\bs{d}\dec{4},a\dec{1},\bs{2}\dec{5},(2)_{b-2}]\dec{2}+\ldec{1}[(2)_{a-1},\bs{c+1}\dec{3},2\dec{5},(2)_{d-1}]\dec{4}. 
	\end{equation*}
	where the bold numbers correspond to $H_1^{\gamma},H_2^{\gamma},H_3^{\gamma}$; and decorations $\dec{1},\dots,\dec{5}$ indicate the common points of $D_{\gamma}$ with $A_{13}^{\gamma}$, $A_{22}^{\gamma}$, $A_{23}^{\gamma}$, $A_{31}^{\gamma}$ and $A_{32}^{\gamma}$, respectively. For this decomposition $\psi=\gamma\circ\gamma_{+}$, the values of $(a,b,c,d)$ satisfying inequality \eqref{eq:ld-bound_ht=3} are listed in Table \ref{table:abcd}. If $\psi=\gamma$ then $(a,b)\neq (2,2)$ by Claim \ref{cl:flipping}\ref{item:U}, so $D$ is as in \ref{item:ht=3_chains}. 
	
	Thus we can assume that $\gamma_{+}^{-1}$ has a base point, say $r$. Let $\eta$ be a composition of $\gamma$ with a blowup at $r$.
	
	\begin{casesp}
	\litem{$r\in A_{22}^{\gamma}$} By the definition of $\gamma$ and by Lemma \ref{lem:w=3_uniqueness}\ref{item:w=3_no_C1}, we have $\{r\}=A_{22}^{\gamma}\cap (U_{\gamma}-V_2^{\gamma})$. Hence $U_{\eta}$ is a fork, with a branching component $V_{2}$. The twig $[3,(2)_{b-2},a,d]$ of $U_{\eta}$ containing $H_{1}$ has length at least $3$ and is not a $(-2)$-twig, so by Lemma \ref{lem:admissible_forks}, the remaining twigs are of type $[2]$. It follows that $c=2$ and $\Bs\eta_{+}^{-1}\cap (A_{22}\cup A_{23})=\emptyset$. We have $\Bs\eta_{+}^{-1}\cap A_{32}=\emptyset$ by Claim \ref{cl:flipping}\ref{item:A32}; and $\Bs\eta^{-1}_{+}\cap (A_{31}\cup A_{13})=\emptyset$ because otherwise $U$ would have two branching components. Thus $\eta_{+}=\id$. Suppose $b\geq 3$.  To get a contradiction with \eqref{eq:ld-bound_ht=3}, we can apply vertical swaps over $p_{13},p_{23}$ and assume $a=d=2$. Now $D=\langle b;[2],[2],[3,(2)_{b-3},\bs{2},2,\bs{2}]\rangle+[2,\bs{3},2,2]$, so $\sum_{i=1}^{3}\ld(H_{i})=
	\tfrac{4b-6}{2b^2 - b - 4}+\tfrac{5}{11}\leq 1$; a contradiction with \eqref{eq:ld-bound_ht=3}. Thus $b=2$, so $D=\langle 2;[2],[2],[\bs{3},a,\bs{d}]\rangle+[(2)_{a-1},\bs{3},(2)_{d}]$. By \eqref{eq:ld-bound_ht=3} we get $a=d=2$, so $D$ is as in \ref{item:ht=3_XE}.
	
	\litem{$A_{22}\cap \Bs\gamma_{+}^{-1}=\emptyset$} Now Claim \ref{cl:flipping}\ref{item:A32} implies that $\Bs\gamma_{+}^{-1}\subseteq A_{13}\cup A_{23}\cup A_{31}$, so by the definition of $\gamma$ we have either 
		$\{r\}=A_{31}\cap W$, or  
		$\{r\}=A_{13}\cap W$, or 
		$\{r\}=A_{23}\cap U$. 
	Moreover, Claim \ref{cl:flipping}\ref{item:U} implies that 
	\begin{equation}\label{eq:ab}
	(a,b)\neq (2,2).
	\end{equation}
	
	\begin{casesp}
	\litem{$\{r\}=A_{31}\cap W$} Then $H_1$ is branching in the fork $U_{\eta}$, with twigs $T_{H_2}=[(2)_{b-2},\bs{2},a]$, $T_{V_2}=[(2)_{c-1},b]$ containing $H_2$, $V_2$; respectively. We have  $\#T_{H_2},\#T_{V_2}\geq 2$, so by Lemma \ref{lem:admissible_forks} one of these twigs is $[2,2]$. Since $(a,b)\neq (2,2)$, we have $T_{V_2}=[2,2]$, so $(b,c)=(2,2)$ and $a\geq 3$. Lemma \ref{lem:admissible_forks} gives $T_{H_2}=[2,3]$ and $a=3$. For $d=2$ we have $D_{\eta}=\langle \bs{2};[2],[2,2],[\bs{2},3]\rangle+[3,2,\bs{3},2,2]$, so $\sum_{i}\ld(H_i^{\eta})=1$ and inequality  \eqref{eq:ld-bound_ht=3} fails. The case $d\geq 2$ reduces to $d=2$ via vertical swaps, so in any case we get a contradiction with Lemma \ref{lem:w=3_uniqueness}\ref{item:ld-bound_ht=3}. 
	
	\litem{$\{r\}=A_{13}\cap W$, $\Bs\gamma_{+}^{-1}\subseteq A_{13}\cup A_{23}$} Then $U_{\eta}=\langle a;[2],[(2)_{b-2},\bs{2}],[(2)_{c-1},b,\bs{d}]\rangle$.  The admissibility of $U_{\eta}$ gives  $b=2$, so $a\geq 3$ by condition \eqref{eq:ab}, and $D_{\eta}=\langle a;[2],[\bs{2}],[(2)_{c},\bs{d}]\rangle+[3,(2)_{a-2},\bs{c+1},(2)_{d}]$.  Inequality \eqref{eq:ld-bound_ht=3} gives $d=2$ and either $(a,c)=(3,3)$ or $c=2$. If $\psi=\eta$ then $D$ is as in  \ref{item:ht=3_XA_a=3_c=3} or \ref{item:ht=3_XA_c=2}, respectively. 
	
	Assume $\psi\neq \eta$. Let $\theta$ be a composition of $\eta$ with a blowup at some $s\in \Bs\eta_{+}^{-1}\subseteq A_{13}\cup A_{23}$. Suppose $s\in A_{23}$, so $\{s\}=A_{23}\cap U$ by the definition of $\gamma$. 
	Then $W_{\theta}=\langle \bs{c+1},[2],[3,(2)_{a-2}],[2,2]\rangle$. Since $a\geq 3$ and the fork $W_{\theta}$ is admissible, we have $a=3$. Noe $D_{\theta}=\langle 3;[2],[\bs{2}],[3,(2)_{c-1},\bs{2}]\rangle+\langle \bs{c+1},[2],[3,2],[2,2]\rangle$, contrary to \eqref{eq:ld-bound_ht=3}. 
	
	Therefore, $\Bs\eta_{+}^{-1}\subseteq A_{13}$. Recall that $b=d=2$, so $D=\langle a;[\bs{2}],[(2)_{c},\bs{2}],T\trp\rangle+T^{*}*[(2)_{a-1},\bs{c+1},2,2]$ for some admissible chain $T$. Since $\psi\neq \eta$, we have $T\neq [2]$. Because the fork $U$ is admissible and $c\in \{2,3\}$, we have $d(T)\leq 3$, so $T=[2,2]$ or $[3]$. Recall that $a\geq 3$. The inequality \eqref{eq:ld-bound_ht=3} gives $c=2$, so $D$ is as in \ref{item:ht=3_XAA_T=[2,2]} if $T=[2,2]$ and as in \ref{item:ht=3_XAA_T=[3]} if $T=[3]$. 
	
	\litem{$\{r\}=A_{23}\cap U$, $\Bs\gamma_{+}^{-1}\subseteq A_{23}$} Then for some admissible chain $T$, we have 
	\begin{equation*}
		D=T^{*}*[(2)_{c-1},b,\bs{d},a,\bs{2},(2)_{b-2}]+\langle \bs{c+1};T\trp,[(2)_{a-1}],[(2)_{d}]\rangle  
	\end{equation*}
	Consider the case $a\geq 3$. Since $d\geq 2$, the admissibility of the fork $W$ implies that $T=[2]$ and either $a=3$, $d\in \{2,3,4\}$ or $a\in \{4,5\}$, $d=2$. 
	 Inequality \eqref{eq:ld-bound_ht=3} gives $(a,b,d)=(3,2,2)$, so $D$ is as in \ref{item:ht=3_XY_a=3}.

	Consider the case $a=2$. Then $b\geq 3$ by condition \eqref{eq:ab}. Substituting $T=[2]$ above, we get
	\begin{equation*}
		D_{\eta}=[3,(2)_{c-2},b,\bs{d},2,\bs{2},(2)_{b-2}]+\langle \bs{c+1};[2],[2],[(2)_{d}]\rangle,
	\end{equation*}
	so inequality \eqref{eq:ld-bound_ht=3} gives $d=2$, and either $c=3$ or $(b,c)=(4,2)$. In the first case, $D$ is as in \ref{item:ht=3_XY_b=3}. Assume the second case holds. If $\psi=\eta$ then $D$ is as in 	\ref{item:ht=3_XY_b=4_T-2}. Suppose $\psi\neq \eta$, and let $\theta$ be the composition of $\eta$ with a blowup at some $s\in \Bs\eta_{+}^{-1}$. We have $\{s\}=A_{23}\cap U$ or $\{s\}=A_{23}\cap W$, so 
	$D_{\theta}=[4,4,\bs{2},2,\bs{2},2,2]+\langle \bs{3},[2,2],[2],[2,2]\rangle$ or $[2,3,4,\bs{2},2,\bs{2},2,2]+\langle \bs{3},[3],[2],[2,2]\rangle$, respectively. Thus $\sum_{i}\ld(H_{i}^{\theta})=1$ or $\frac{713}{747}$, a contradiction with \eqref{eq:ld-bound_ht=3}.\qedhere
	\end{casesp}
\end{casesp}
\end{proof}
\clearpage

\section{Case $\width(\bar{X})=2$}\label{sec:w=2}

In this section we settle the case $\width=2$. 
We keep the following notation and assumptions.
\begin{equation}\label{eq:assumption_w=2}
	\parbox{.9\textwidth}{
		$\bar{X}$ is a log terminal del Pezzo surface of rank one such that $\height(\bar{X})=3$ and $\width(\bar{X})=2$; \\ $(X,D)$ is the minimal log resolution of $\bar{X}$, and $p\colon X\to \P^1$ is a $\P^1$-fibration such that \\ $D\hor=H_1+H_2$, where $H_1$ is a $1$-section and $H_2$ is a $2$-section.
	}
\end{equation}

The above conditions imply that a fiber $F$ of any $\P^1$-fibration of $X$ satisfies $F\cdot D\geq 3$, and if the equality holds then $D\hor$ has at most two components, see Section \ref{sec:P1-fibrations}.

We follow the same strategy as in case $\width=3$, settled in Section \ref{sec:w=3} above. First, in Lemma \ref{lem:w=2_psi}  we choose $p$ so that there is a minimalization $\psi\colon (X,D)\to (\P^2,B)$ as in Proposition \ref{prop:ht=3_models}\ref{item:w=2_models}, i.e.\ mapping $H_2$ to a conic and the fibers to lines passing through the point $\psi(H_1)$. Next, in Lemma \ref{lem:w=2_swaps} we prove the relevant part of Proposition~\ref{prop:ht=3_swaps}. That is, we adjust $p$ so that the above minimalization  $\psi$ factors as $\psi=\phi\circ \phi_{+}$, where $\phi_{+}\colon (X,D)\sqto(Y,D_Y)$ is a vertical swap, and $\phi\colon (Y,D_Y)\to (\P^2,B)$ is the construction of a vertically primitive log surface $(Y,D_Y)$ from Example \ref{ex:w=2_cha_neq_2} or \ref{ex:w=2_cha=2}. Eventually, we classify the surfaces $\bar{X}$ by describing all possible vertical swaps $\phi_{+}$.

In each step we need to consider two cases, depending on whether the morphism $p|_{H_{2}}\colon H_2\to \P^1$ of degree $2$ is separable or not. In the second case, which can only hold if $\cha\kk=2$, we get significantly more surfaces $\bar{X}$, see Lemma \ref{lem:w=2_cha=2}. Moreover, in this case we have moduli of vertically primitive log surfaces $(Y,D_Y)$, see Proposition \ref{prop:primitive}\ref{item:primitive-uniqueness}, which yield moduli of del Pezzo surfaces $\bar{X}$ of the same singularity type, see Theorem~\ref{thm:ht=3}\ref{item:uniq_cha=2}.

\subsection{Proof of Proposition \ref{prop:ht=3_models}: constructing the minimalizations \texorpdfstring{$\psi\colon (X,D)\to (\P^2,B)$}{onto P2}} 

To begin with, we give a preliminary description of degenerate fibers of $p$, and we fix the notation $\nu,k_i,l_i$ for the remaining part of Section \ref{sec:w=2}. Recall that for a degenerate fiber $F$ we denote by $\sigma(F)$ the number of components of $F\redd$ not contained in $D$, i.e.\ of $(-1)$-curves in $F\redd$, see Lemma \ref{lem:fibrations-Sigma-chi}.

\begin{lemma}[Degenerate fibers]\label{lem:w=2_fibers}
	Let $(X,D)$  be as in \eqref{eq:assumption_w=2}, and let $F_1,\dots, F_{\nu}$ be all degenerate fibers of $p$. After reordering these fibers, if needed, we have
	\begin{equation}\label{eq:w=2_Sigma}
		\sigma(F_1)=2\quad\mbox{and}\quad \sigma(F_i)=1\quad \mbox{for } i\in \{2,\dots,\nu\}.
	\end{equation}
	Let $\hat{\phi}$ (respectively, $\bar{\phi}$) be a contraction of all vertical $(-1)$-curves on $X$ and its images which do not meet the image of $D\hor$ (respectively, meet the image of $D\hor$ at most once, normally). 
	For $i= 1,\dots, \nu$ let $k_{i}$ (respectively, $l_{i}$) be the number of blowups in the decomposition of $\bar{\phi}$ which are not isomorphisms in any neighborhood of the image of $F_{i}$ or of $H_2$ (respectively, of $H_1$). Then one of the following holds.
		\begin{enumerate}
		\item\label{item:FH=2} $\hat{F}_1=[1\adec{2},(2)_{k_{1}-1},\uline{2}\adec{2},(2)_{l_1-1},1\adec{1}],\, k_{1},l_1\geq 1$; or $[\uline{1}\adec{2,2},(2)_{l_1-1},1\adec{1}],$ $k_{1}=0$, $l_{1}\geq 1$. 
		\item\label{item:FH=1} $\hat{F}_1=\langle 2;[1\adec{1},(2)_{l_1-1},\uline{3}],[1\adec{2},(2)_{k_1-3}],[2]\rangle$, $k_1\geq 3$, $l\geq 1$; or $[1\adec{1},(2)_{l_1-1},\uline{3},1\adec{2},2]$, $k_{1}=2$, $l_1\geq 1$.
		\item\label{item:dumb} $\hat{F}_1=[1\adec{2},(2)_{a},\uline{2}\adec{1},(2)_{k_{1}-a-2},1\adec{2}]$ for some $a\geq 0$, $k_{1}\geq a+2$, and $l_1=0$.
		\setcounter{foo}{\value{enumi}}
	\end{enumerate}
	Moreover, for each $i\in \{2,\dots,\nu\}$ we have $l_{i}=0$, and one of the following holds.
	\begin{enumerate}
		\setcounter{enumi}{\value{foo}}	
		\item\label{item:F_columnar} $\hat{F}_i=[\uline{1}\adec{1,2},1\adec{2}]$, so $(F_i)\redd=[T,1,T^{*}]$, $F_i$ meets $H_2$ in both tips and $H_1$ in one tip, and $k_i=1$.
		\item\label{item:F_fork} $\hat{F}_i=\langle 2;[1\adec{2},(2)_{k_i-3}];[2],[\uline{2}\adec{1}]\rangle$, $k_i\geq 3$; or $[2,1\adec{2},\uline{2}\adec{1}]$, $k_i=2$.
	\end{enumerate}
	Where we use the following notation:
	\begin{itemize}
		\item  The numbers decorated by $\adec{1}$ and $\adec{2}$ refer to components meeting  $\hat{\phi}(H_1)$ and $\hat{\phi}(H_2)$, respectively.
		\item The underlined number refers to the proper transform of $\bar{\phi}_{*}F_i$ (note that $\bar{\phi}_{*}F_{i}=[0]$).
	\end{itemize}

\end{lemma}
\begin{proof}
	The assertion \eqref{eq:w=2_Sigma} follows from Lemma \ref{lem:fibrations-Sigma-chi}\ref{item:Sigma}. 
	Fix a degenerate fiber $F$. Put $\hat{F}=\hat{\phi}_{*}F\redd$,  $\hat{H}_{i}=\hat{\phi}(H_{i})$.
	
	Suppose $\hat{F}=[0]$. Then $V\de \hat{\phi}^{-1}_{*}\hat{F}$ is a component of $F$ satisfying $V\cdot D\hor=F\cdot D\hor$, so $V\subseteq D$ by \cite[Lemma 2.8(c)]{PaPe_ht_2}. Since $D$ is snc and has no circular subdivisor, we get $1\geq H_{2}\cdot V=H_{2}\cdot F=2$, a contradiction. 
	
	Thus $\hat{F}$ contains a $(-1)$-curve, say $\hat{L}$. Assume that $\hat{L}$ meets $\hat{H}_1$.  Then $\hat{L}$ has multiplicity $1$ in $\hat{F}$, so $\hat{F}$ has another $(-1)$-curve, say $\hat{L}'$. Then $\hat{L}'\cdot \hat{H}_1=0$, so $\hat{L}'$ meets $\hat{H}_2$ by the definition of $\hat{\phi}$. Put $L=\hat{\phi}^{-1}_{*}\hat{L}$, $L'=\hat{\phi}^{-1}_{*}\hat{L}'$. 
	
	Consider the case $\sigma(F)=1$. 
	Then $\hat{F}\cap \Bs\hat{\phi}^{-1}=\hat{L}\cap \hat{L}'$. It follows that $L$ and $L'$ are not $(-1)$-curves, so they are components of $D$. Since $D$ is snc and contains no circular subdivisor, 
	both $L$ and $L'$ meet $H_2$, they lie in different connected components of $D\vert$, and have multiplicity $1$ in $F$, which implies that 
	$F$ is as in \ref{item:F_columnar}.
	
	Consider the case $\sigma(F)=2$. Since the curve $\hat{L}'$ meets $\hat{H}_{2}$, it has multiplicity $\mu\leq 2$ in $\hat{F}$. If $\mu=1$ then both $\hat{L}$ and $\hat{L}'$ are tips of $\hat{F}$, so \ref{item:FH=2} holds. Assume $\mu=2$. Contract $\hat{L}'$ and subsequent $(-1)$-curves in the images of $\hat{F}$, until the image of $\hat{H}_{2}$ meets the image of $\hat{F}$ in a node; and denote this contraction by $\tau$. Each component of $\Exc\tau$ has multiplicity $2$ in $\hat{F}$, so $\tau$ is an isomorphism 
	near $\hat{H}_1$, and $\tau_{*}\hat{F}$ has a $(-1)$-curve $C\neq\tau(\hat{L})$. By the definition of $\hat{\phi}$, both $\tau(\hat{L})$ and $C$ meet $\tau(\hat{H}_2)$, so their proper transforms have multiplicity $1$ in $F$. 
	We conclude that  $\tau_{*}\hat{F}=[1,(2)_{l},1]$, and $l\geq 1$ since 
	the component of $\hat{F}$ meeting $\hat{H}_1$ is 
	a $(-1)$-curve. Thus \ref{item:FH=1} holds.
	
	It remains to consider the case when no $(-1)$-curve in $\hat{F}$ meets $\hat{H}_{1}$. Assume $\sigma(F)=1$. Defining $\tau$ as before, we get that $\tau_{*}\hat{F}$ is a chain $[1,1]$ meeting $\tau_{*}\hat{H}_{1}$ in a tip and $\tau_{*}\hat{H}_{2}$ in a node, so \ref{item:F_fork} holds. Assume $\sigma(F)=2$. Then both $(-1)$-curves in $\hat{F}$ meet $\hat{H}_2$, so they have multiplicity one and \ref{item:dumb} holds.
\end{proof}

\begin{lemma}[The structure of $p$]\label{lem:w=2_basic}
	We keep the notation and assumptions of Lemma \ref{lem:w=2_fibers}. For $i\in \{1,\dots,\nu\}$ let $V_i=\bar{\phi}^{-1}_{*}(\bar{\phi}_{*}F_i)$, i.e.\ $V_{i}$ is the proper transform on $X$ of the component of $\hat{F}_i$ corresponding to the underlined number. 
	For $j\in \{1,2\}$ put $\bar{H}_j=\bar{\phi}(H_j)$. Then the following hold. 
	\begin{enumerate}
		\item \label{item:H1_meets} For all $i\in \{2,\dots,\nu\}$, the divisor $(F_{i})\redd \wedge D\vert$ meets the $1$-section $H_1$. In particular, $\nu\leq \beta_{D}(H_1)+1\leq 4$.
		\item\label{item:psi} The image $\bar{\phi}(X)$ is isomorphic to the Hirzebruch surface $\F_{m}$ for $m=1-H_{1}\cdot H_{2}\in \{0,1\}$.
		\item\label{item:psi_properties} We have $\bar{H}_1=[m]$, i.e.\ $\bar{H}_1$ is the negative section of $\F_m$. In turn, $\bar{H}_2^2=4$, $\bar{H}_2\cong \P^1$ and $\bar{H}_1\cdot\bar{H}_2=H_1\cdot H_2$.
		\item\label{item:H1} We have $H_{1}^{2}=H_{1}\cdot H_{2}-1-l_{1}$. In particular, $l_{1}\geq 1$, so case \ref{lem:w=2_fibers}\ref{item:dumb} does not occur.
		\item\label{item:H2} We have 
			$H_{2}^{2}=4-\sum_{i=1}^{\nu}k_{i}$. In particular, $\sum_{i=1}^{\nu}k_{i}\geq 6$.
		\item\label{item:eta} Put  $\eta=\#\{i:V_{i}\subseteq D\}=\#\bar{\phi}_{*}D\vert$. Then $\nu\geq\eta\geq 3$. \item\label{item:eta-not-nu} If $\nu\neq \eta$ then $(\nu,\eta)=(4,3)$ and $F_1$ is as in Lemma \ref{lem:w=2_fibers}\ref{item:FH=2} with $k_1=0$.
	\end{enumerate}
\end{lemma}
\begin{proof}
	\ref{item:H1_meets} Condition \eqref{eq:w=2_Sigma} implies that $A_i\de (F_i)\redd-(F_{i})\redd\wedge D\vert$ is the unique $(-1)$-curve in $F_i$. Hence $A_i$ has multiplicity at least $2$ in $F_i$, in particular, $A_i$ does not meet the $1$-section $H_1$, as claimed. It follows that $\nu-1\leq \beta_{D}(H_1)$. Eventually, $\beta_{D}(H_1)\leq 3$ as  $D$ is a disjoint union of chains and forks.
	
	\ref{item:psi}, \ref{item:psi_properties} By the definition of $\bar{\phi}$ we have $\bar{H}_1\cdot \bar{H}_2=H_1\cdot H_2$ and $\bar{H}_i\cong H_i\cong \P^1$ for $i\in \{1,2\}$. By Lemma \ref{lem:w=2_fibers} each $\bar{\phi}_{*}F_i$ is irreducible, so $\bar{\phi}(X)\cong \F_{m}$ for some $m\geq 0$. 
	Let $F$ be a fiber of $\F_m$. Then $(2\bar{H}_1-\bar{H}_2)\cdot F=0=(\bar{H}_2+K_{\F_{m}})\cdot F$, so  $2\bar{H}_1-\bar{H}_2$ and $\bar{H}_2+K_{\F_{m}}$ are linearly equivalent to a multiple of $F$. Thus $0=(2\bar{H}_1-\bar{H}_2)\cdot(\bar{H}_2+K_{\F_m})=2\bar{H}_1\cdot \bar{H}_2-2\bar{H}_1^{2}-4+2$. We get $\bar{H}_1^{2}=H_{1}\cdot H_{2}-1\leq 0$, so $\bar{H}_1$ is the negative section, and $m=-\bar{H}_1^2=1-H_{1}\cdot H_{2}$. Moreover, $0=(2\bar{H}_1-\bar{H}_2)^{2}=4\bar{H}_1^2-4H_{1}\cdot H_{2}+\bar{H}_2^{2}=\bar{H}_2^{2}-4$, so $\bar{H}_2^{2}=4$.
	
	\ref{item:H1}, \ref{item:H2} By the definition of $k_{i},l_{i}$ we have $\bar{H}_{2}^2=H_{2}^{2}+\sum_{i} k_{i}$, $\bar{H}_{1}^2=H_{1}^2+\sum_{i}l_{i}=H_{1}^2+l_{1}$, as $l_i=0$ for $i>1$ by Lemma \ref{lem:w=2_fibers}. Using \ref{item:psi} and \ref{item:psi_properties} we get the formulas for $H_{i}^2$. The inequalities follow because $H_{i}^2\leq -2$.

	\ref{item:eta} Clearly, $\nu\geq \eta$. Assume $H_1\cdot H_2=1$. Then $m=0$ by \ref{item:psi}. The other projection of $\P^1\times \P^1$ pulls back to a $\P^{1}$-fibration of $X$ such that the horizontal part of $D$ equals $H_2+\sum_{i}V_i\wedge D$, so it is a sum of $1+\eta$ sections. Thus $1+\eta\geq \height(\bar{X})=3$, and the inequality is strict, because $\width(\bar{X})<3$ by assumption \eqref{eq:assumption_w=2}.
	
	Assume now that $H_1\cdot H_2=0$, so $m=1$ by \ref{item:psi}. Let $\psi\colon X\to \P^2$ be the composition of $\bar{\phi}$ with the contraction of $\bar{H}_1$, and let $p$ be a base point of $\psi^{-1}$ lying on $\psi(H_2)$. If $\eta>0$ we can and will choose $p\in\psi(D\vert)$. Put $\epsilon=0$ if $p\in \psi(D\vert)$ and $\epsilon=1$ otherwise. 
	The pencil of lines passing through $p$ pulls back to a $\P^1$-fibration of $X$ such that the horizontal part of $D$ is a sum of at most $\eta+1+\epsilon$ sections, namely $H_2$, the first exceptional curve over $p$, and $V_i$ for all $i\in \{1,\dots,\nu\}$ such that $p\not\in \psi(F_i)$ and $V_i\subseteq D$. As before, the conditions $\height(\bar{X})=3$ and $\width(\bar{X})<3$ give $\eta+1+\epsilon\geq 4$. In particular, $\eta>0$, so $\epsilon=0$ by our choice of $p$. Thus $\eta\geq 3$, as claimed. 
	
	\ref{item:eta-not-nu} Assume $\nu\neq \eta$, so $V_i\not\subseteq D$ for some $i$. Then $V_i=[1]$, 
	so $\hat{\phi}(V_i)=[1]$ or $[0]$. Lemma~\ref{lem:w=2_fibers} shows that $i=1$ and $F_1$ is as in \ref{lem:w=2_fibers}\ref{item:FH=2}, with $k_1=0$. Moreover, $\nu\leq 4$ by \ref{item:H1_meets} and $\nu\geq \eta\geq 3$ by \ref{item:eta}, so $(\nu,\eta)=(4,3)$.
\end{proof}

As in the previous section, the key step of the proof of Proposition \ref{prop:ht=3_models} is the following. 

\begin{lemma}
	\label{lem:w=2_disjoint}
	Let $(X,D)$ be as in \eqref{eq:assumption_w=2}. One can choose $p$ so that $H_{1}\cdot H_{2}=0$.
\end{lemma}
\begin{proof}
	Assume $H_{1}\cdot H_{2}\neq 0$, so $H_1\cdot H_2=1$ by Lemma \ref{lem:w=2_basic}\ref{item:psi}.  We keep notation from Lemma \ref{lem:w=2_basic}. 
	
	By Lemma \ref{lem:w=2_basic}\ref{item:eta} we have $\nu\geq 3$. Lemma \ref{lem:w=2_basic}\ref{item:H1_meets} implies that $H_1$ meets some component of $(F_{i})\redd\wedge D\vert$ for $i\in \{2,\dots,\nu\}$. Thus the connected component of $D$ containing $H_1$, call it $U$, is an admissible fork with maximal twigs $T,T_2,T_3$ such that  $H_2=\ltip{T}$ and $\ltip{T_i}\subseteq F_i$ for $i\in \{2,3\}$, see Figure \ref{fig:w=2_disjoint}. It follows that $\nu=3$, and for $i\in \{2,3\}$ the chain $T_i$ is the connected component of $(F_{i})\redd\wedge D\vert$ meeting $H_1$. Since $T_i$ does not meet $H_2$, the fiber $F_i$ is as in Lemma \ref{lem:w=2_fibers}\ref{item:F_fork} for some $k_i\in \{2,3\}$. Since $\nu=3$, Lemma \ref{lem:w=2_basic}\ref{item:eta} implies that $\eta=\nu$, i.e.\ for each $i$, the component $V_i$ of $F_i$ corresponding to the underlined number in Lemma \ref{lem:w=2_fibers} lies in $D$.
	
	Suppose $F_{1}$ is as in \ref{lem:w=2_fibers}\ref{item:FH=2}. Then $V_1$ meets $H_2$, so since $V_1$ lies in $D$, it is the second-to-last component of the twig $T$. Hence $\beta_{D}(V_1)\leq 2$, so $k_1=1$ or $l_1=1$, and $\#T\geq k_{1}+l_1$. By Lemma \ref{lem:w=2_basic}\ref{item:H1} we have $l_1=-H_{1}^{2}\geq 2$, so $k_{1}=1$ and $\#T\geq 3$. The inequality in Lemma \ref{lem:w=2_basic}\ref{item:H2} yields $k_2+k_3\geq 5$, so since $k_i\in \{2,3\}$, we have, say, $k_2=2$, $k_3=3$. It follows that $\#T_3\geq 3$, too, which by Lemma \ref{lem:admissible_forks} contradicts the admissibility of the fork $U$.
	
	Thus $F_1$ is as in \ref{lem:w=2_fibers}\ref{item:FH=1}, as case \ref{lem:w=2_fibers}\ref{item:dumb} does not occur by Lemma \ref{lem:w=2_basic}\ref{item:H1}. We remark that this can happen only if $\cha\kk=2$: indeed, the restriction $p|_{H_2}$ is ramified at three points $H_2\cap F_i$, $i\in \{1,2,3\}$, which contradicts the Hurwitz formula \cite[Corollary IV.2.4]{Hartshorne_AG} unless $\cha\kk=2$ and $p|_{H_{2}}$ is totally ramified.
	
	Suppose $k_{3}=3$, so $\#T_{3}\geq 3$, see Figure \ref{fig:w=2_disjoint_k3=3}. Since $k_1,k_2\geq 2$, Lemma \ref{lem:w=2_basic}\ref{item:H2} implies that $H_{2}^{2}\leq -3$.  The weighted graph of $U$ is a weighted subgraph of a fork $\langle \bs{2},[2],[2,2,2],[\ub{3}]\rangle$, see Lemma \ref{lem:Alexeev}, where the bold numbers correspond to the horizontal components of $D$, and the underlined one corresponds to $H_2$. Using Lemmas \ref{lem:ld_formulas} and \ref{lem:Alexeev}, we get $\ld(H_1)\leq \frac{1}{5}$, $\ld(H_2)\leq \frac{2}{5}$, so $\ld(H_1)+2\ld(H_2)\leq 1$, a contradiction with Lemma \ref{lem:delPezzo_criterion}.
	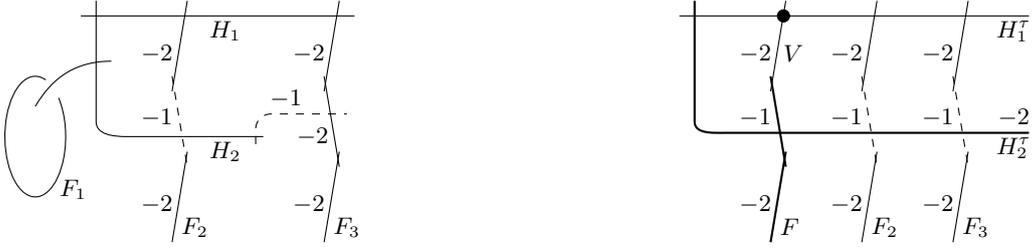
\begin{figure}[ht]
		\subcaptionbox{Case $(k_2,k_3)=(2,3)$: image of $(X,D)$ after a vertical swap $\hat{\phi}$ from Lemma \ref{lem:w=2_fibers}. \label{fig:w=2_disjoint_k3=3}}[.45\linewidth]{
			\begin{tikzpicture}
				\path[use as bounding box] (-1.3,0) rectangle (3.5,3.2);
				\draw (-0.2,3) -- (3.4,3);
				\node at (1.7,2.8) {\small{$H_{1}$}};
				\draw (0,3.2) -- (0,1.6) to[out=-90,in=180] (0.4,1.4) -- (2.2,1.4);
				\node at (1.7,1.2) {\small{$H_{2}$}};
				\draw (0.2,2.4) to[out=-180,in=60] (-0.8,1.8);
				\draw (-0.8,1.4) [partial ellipse= 70 : 400 : 0.4 and 0.8];
				\node at (-0.3,0.7) {\small{$F_1$}};
				\draw (1.2,3.2) -- (1,2);
				\node at (0.8,2.5) {\small{$-2$}};
				\draw[dashed] (1,2.2) -- (1.2,1);
				\node at (0.8,1.65) {\small{$-1$}};
				\draw (1.2,1.2) -- (1,0); 
				\node at (0.8, 0.5) {\small{$-2$}};
				\node at (1.3,0.2) {\small{$F_2$}};
				\draw[dashed] (2.1,1.3) -- (2.1,1.5) to[out=90,in=180] (2.3,1.7) -- (3.3,1.7);
				\node at (2.5,1.9) {\small{$-1$}};
				\draw (3.2,3.2) -- (3,2);
				\node at (2.8,2.5) {\small{$-2$}};
				\draw (3,2.2) -- (3.2,1);
				\node at (2.85,1.4) {\small{$-2$}};
				\draw (3.2,1.2) -- (3,0); 
				\node at (2.8, 0.5) {\small{$-2$}};
				\node at (3.3,0.2) {\small{$F_3$}};
			\end{tikzpicture}
		}
	\hfill
		\subcaptionbox{Case $(k_2,k_3)=(2,2)$: the new $\P^1$-fibration. The marked point is a base point of $\tau^{-1}$, so $\tau^{-1}_{*}V\cap H_1=\emptyset$. \label{fig:w=2_disjoint_k3=2}}[.5\linewidth]{
			\begin{tikzpicture}
				\path[use as bounding box] (-0.2,0) rectangle (4.5,3.2);
				\draw (-0.2,3) -- (4.4,3);
				\node at (4.2,2.8) {\small{$H_{1}^{\tau}$}};
				\draw[thick] (0,3.2) -- (0,1.6) to[out=-90,in=180] (0.4,1.45) -- (4.4,1.45);
				\node at (4.2,1.25) {\small{$H_{2}^{\tau}$}};
				\node at (4.2,1.65) {\small{$-2$}};
				\draw (1.2,3.2) -- (1,2);
				\node at (0.8,2.5) {\small{$-2$}};
				\node at (1.3,2.5) {\small{$V$}};
				\filldraw (1.17,3) circle (0.08);
				\draw[thick] (1,2.2) -- (1.2,1);
				\node at (0.8,1.65) {\small{$-1$}};
				\draw[thick] (1.2,1.2) -- (1,0); 
				\node at (0.8, 0.5) {\small{$-2$}};
				\node at (1.25,0.2) {\small{$F$}};
				\draw (2.4,3.2) -- (2.2,2);
				\node at (2,2.5) {\small{$-2$}};
				\draw[dashed] (2.2,2.2) -- (2.4,1);
				\node at (2,1.65) {\small{$-1$}};
				\draw (2.4,1.2) -- (2.2,0); 
				\node at (2, 0.5) {\small{$-2$}};
				\node at (2.5,0.2) {\small{$F_2$}};
				\draw (3.6,3.2) -- (3.4,2);
				\node at (3.2,2.5) {\small{$-2$}};
				\draw[dashed] (3.4,2.2) -- (3.6,1);
				\node at (3.2,1.65) {\small{$-1$}};
				\draw (3.6,1.2) -- (3.4,0); 
				\node at (3.2, 0.5) {\small{$-2$}};
				\node at (3.7,0.2) {\small{$F_3$}};
			\end{tikzpicture}
		}\vspace{-1em}
		\caption{Proof of Lemma  \ref{lem:w=2_disjoint}.}\vspace{-0.5em}
		\label{fig:w=2_disjoint}
	\end{figure}
	
	Thus $k_{2}=k_{3}=2$. Lemma \ref{lem:w=2_basic}\ref{item:H2} gives $k_{1}=-H_{2}^{2}$. Let $\tau$ be the contraction of $(-1)$-curves in $F_1$ and its images so that $F\de\tau_{*}F_1$ satisfies $F\redd=[2,1,2]$. Put $H_j^{\tau}=\tau(H_j)$ for $j\in \{1,2\}$, and let $V\de \tau(V_1)$, so $V$ is the tip of $F\redd$ meeting $H_1^{\tau}$. We have $(H_2^{\tau})^2=H_2^2+k_1-2=-2$. Put $F'=H_{2}^{\tau}+F-V$. Then $F'\redd=[2,1,2]$ meets the remaining part of $\tau_{*}D$ normally, twice: at $H_{2}^{\tau}\cap H_{1}^{\tau}$ and at $(F\redd-V)\cap V$, see Figure \ref{fig:w=2_disjoint_k3=2}, where $F'$ is the thickened chain. Thus $|\tau^{*}F'|$ induces a $\P^1$-fibration $\tilde{p}$ of $X$  such that the horizontal part of $D$ consists of a $1$-section $H_{1}$ and a $2$-section $V_1=\tau^{-1}_{*}V$. Since $V$ is the image of the $(-3)$-curve in $\hat{F}_1$ (the one underlined in Lemma \ref{lem:w=2_fibers}\ref{item:FH=1}), its proper transform $V_1$ is disjoint from $H_1$. Thus $\tilde{p}$ has 
	the required property.
\end{proof}

\begin{lemma}[Proposition \ref{prop:ht=3_models}, case $\width(\bar{X})=2$]\label{lem:w=2_psi} 
	Let $(X,D)$ be as in \eqref{eq:assumption_w=2}, that is, the minimal log resolution of a log terminal del Pezzo surface of height $3$ and width $2$. Then one can choose a witnessing $\P^1$-fibration $p$ so that there exists a morphism $\psi\colon X\to \P^2$ with the following properties. 
	\begin{enumerate}
		\item\label{item:psi_H1} The image of the $1$-section in $D$ is a point, say $p_0$.
		\item\label{item:psi_H2} The image of the $2$-section in $D$ is a conic, say $\cc$, not passing through $p_0$.
		\item\label{item:psi_l} Let $F_{1},\dots,F_{\nu}$ be the degenerate fibers of $p$, and let $\ll_{i}=\psi_{*}F_i$. Then for all $i\in \{1,\dots,\nu\}$, $\ll_i$ is a line contained in $\psi_{*}D$, and the set $\ll_i\cap \cc$ contains exactly one base point of $\psi^{-1}$, say $p_i$.
		\item\label{item:psi_p_i} The base points of $\psi^{-1}$ on $\P^2$ are $p_0,\dots,p_{\nu}$, and the preimage of each $p_i$ contains exactly one $(-1)$-curve.
		\item\label{item:psi_cha-neq-2} If $\cha\kk\neq 2$ then $\nu=3$ and exactly two lines in $\psi_{*}D$ are tangent to $\cc$.
		\item\label{item:psi_cha=2} If $\cha\kk=2$ then $\nu\in \{3,4\}$ and all lines passing through $p_0$ are tangent to $\cc$.
	\end{enumerate}
\end{lemma}
\begin{proof}
	Let $\bar{\phi}$ be as in  Lemma \ref{lem:w=2_fibers}. By Lemma \ref{lem:w=2_disjoint}, we can assume $H_1\cdot H_2=0$, so  by Lemma \ref{lem:w=2_basic}\ref{item:psi} $\bar{\phi}\colon X\to \F_1$ maps $H_1$ to the negative section. Let $\psi\colon X\to \P^2$ be the composition of $\bar{\phi}$ with the contraction $\bar{\phi}(H_1)$. Then $\psi$ maps the fibers of $p$ to  lines meeting at $\{p_0\}= \psi(H_1)$; and we conclude that 
	\ref{item:psi_H1} and \ref{item:psi_H2} hold. 
	
	We claim that we can choose $p$ such that $\eta=\nu$, see Lemma \ref{lem:w=2_basic}\ref{item:eta}. Assume $\nu\neq \eta$. By Lemma \ref{lem:w=2_basic}\ref{item:eta-not-nu} $\nu=4$, so by Lemma \ref{lem:w=2_basic}\ref{item:H1_meets} $H_1$ meets twigs $T_2,T_3,T_4$ of $D$ such that $\ltip{T_i}\subseteq F_i$. By Lemma \ref{lem:admissible_forks} we have, say, $T_2=[2]$, so $(F_2)\redd=[2,1,2]$.  In particular, $H_{2}\not\subseteq T_2$. Interchanging $T_3$ with $T_4$, if needed, we can assume that $H_{2}\not\subseteq T_3$, either, so $F_3$ is as in Lemma \ref{lem:w=2_fibers}\ref{item:F_fork}. Thus $(\ll_i\cdot \cc)_{p_i}=2$ for $i\in \{2,3\}$. The pencil of conics tangent to $\ll_{i}$ at $p_{i}$ for $i=3,4$ pulls back to a $\P^{1}$-fibration $\tilde{p}$ of $X$ such that the horizontal part of $D$ consists of a $2$-section $V_{4}=\psi^{-1}_{*}\ll_4$ and a $1$-section $H$, where $H$ is the second exceptional curve over $p_3$ (note that $H\subseteq D$ because $\height(\bar{X})=3$). Since $H\subseteq F_3$ and $V_4\subseteq F_4$, we have $V_4\cdot H=0$. Hence $\tilde{p}$ still satisfies the condition of Lemma  \ref{lem:w=2_disjoint}, which implies \ref{item:psi_H1} and \ref{item:psi_H2}. For $i\in \{3,4\}$, the unique $(-1)$-curve $A_i$ in $\psi^{-1}(p_i)$ belongs to the fiber $F$ of $\tilde{p}$ containing $H_2$. In particular, $F$ contains two $(-1)$-curves, so it plays the role of $F_1$ in  \eqref{eq:w=2_Sigma}. For $i\in \{3,4\}$ the fiber $F_i$ is as Lemma \ref{lem:w=2_fibers}\ref{item:F_columnar} or \ref{item:F_fork}, so $A_{i}\cdot G\leq 1$ for every component $G$ of $D$. It follows that $F$ is \emph{not} as in Lemma \ref{lem:w=2_fibers}\ref{item:FH=2} with $k_1=0$, so by Lemma \ref{lem:w=2_basic}\ref{item:eta-not-nu} the new $\P^1$-fibration $\tilde{p}$ has $\eta=\nu$, as needed.
	\smallskip
	
	The equality $\eta=\nu$ means that $\ll_i\subseteq\psi_{*}D$ for all $i$. Since $\ll_i\cdot \cc=2$, the set $\ll_i\cap \cc$ contains a base point of $\psi^{-1}$. Thus $\psi^{-1}$ has at least $\nu+1$ base points on $\P^2$ (not counting the infinitely near ones). On the other hand, all $(-1)$-curves in $\Exc\psi$ are vertical and by \eqref{eq:w=2_Sigma} there is at most $\nu+1$ of them, so $\psi^{-1}$ has exactly $\nu+1$ base points, and the preimage of each of them contains exactly one $(-1)$-curve. Thus \ref{item:psi_l} and \ref{item:psi_p_i} hold.
	
	By Lemma \ref{lem:w=2_basic}\ref{item:H1_meets},\ref{item:eta} we have $\nu\in \{3,4\}$. To prove \ref{item:psi_cha-neq-2} and \ref{item:psi_cha=2}, it remains to prove that at most one line $\ll_i$ is not tangent to $\cc$. If $\#\ll_{i}\cap \cc=2$ for some $i>1$ then $F_i$ is as in Lemma \ref{lem:w=2_fibers}\ref{item:F_columnar}, so since $\beta_{D}(H_2)\leq 3$, there can be at most one such $i>1$. Thus we can assume that $\#\ll_{i}\cap \cc=2$ for $i\in \{1,2\}$ and $\#\ll_{j}\cap \cc=1$ for $j\geq 3$. It follows that $F_2$ is as in  \ref{lem:w=2_fibers}\ref{item:F_columnar}, so $k_2=1$; $F_j$ for $j\geq 3$ are as in \ref{lem:w=2_fibers}\ref{item:F_fork}; and  $F_1$ is as in \ref{lem:w=2_fibers}\ref{item:FH=2}: recall that case \ref{lem:w=2_fibers}\ref{item:dumb} does not occur by Lemma \ref{lem:w=2_basic}\ref{item:H1}. Since each $\ll_{i}\cap \cc$ contains a base point of $\psi^{-1}$, we have $k_i\geq 1$ for all $i$. 
	
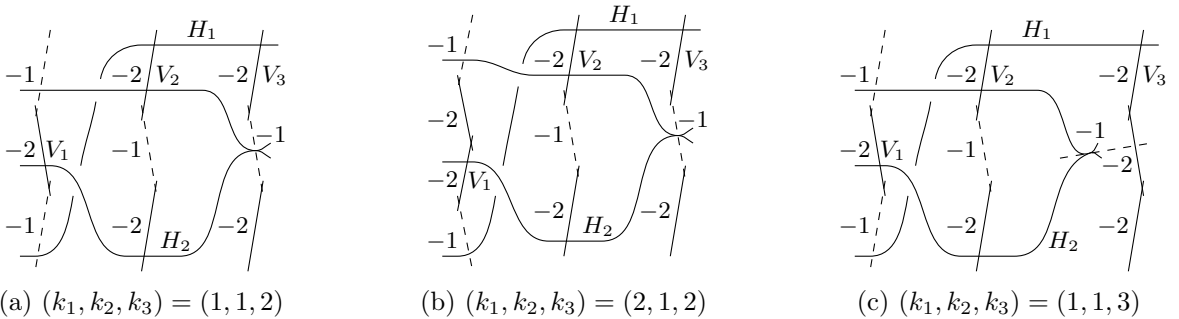
\begin{figure}[ht]
	\subcaptionbox{$(k_1,k_2,k_3)=(1,1,2)$ \label{fig:w=2_models_k2=2}}[.32\textwidth]{
		\begin{tikzpicture}
			\draw[dashed] (0.2,3.2) -- (0,2);
			\node at (-0.2,2.6) {\small{$-1$}};
			\draw (0,2.2) -- (0.2,1);
			\node at (-0.2,1.6) {\small{$-2$}};
			\node at (0.3,1.6) {\small{$V_1$}};
			\draw[dashed] (0.2,1.2) -- (0,0);
			\node at (-0.2,0.6) {\small{$-1$}};
			\draw (1.6,3.2) -- (1.4,2);
			\node at (1.2,2.6) {\small{$-2$}};
			\node at (1.75,2.6) {\small{$V_2$}};
			\draw[dashed] (1.4,2.2) -- (1.6,1);
			\node at (1.2,1.6) {\small{$-1$}};
			\draw (1.6,1.2) -- (1.4,0);
			\node at (1.2,0.6) {\small{$-2$}};
			\draw (3,3.2) -- (2.8,2);
			\node at (2.6,2.6) {\small{$-2$}};
			\node at (3.15,2.6) {\small{$V_3$}};
			\draw[dashed] (2.8,2.2) -- (3,1);
			\node at (3.1,1.8) {\small{$-1$}};
			\draw (3,1.2) -- (2.8,0);
			\node at (2.6,0.6) {\small{$-2$}};
			\draw (-0.2,0.2) -- (0,0.2) to[out=0,in=-100] (0.5,1);
			\draw (0.6,1.4) to[out=80,in=-100] (0.8,2.25);
			\draw (0.86,2.55) to[out=80,in=180] (1.4,3) -- (3.2,3);
			\node at (2.2,3.2) {\small{$H_1$}};
			\draw (-0.2,2.4) -- (2.2,2.4) to[out=0,in=180] (2.9,1.6) to[out=0,in=-150] (3.1,1.7);
			\draw (-0.2,1.4) -- (0.2,1.4) to[out=0,in=180] (1.2,0.2) -- (1.9,0.2) to[out=0,in=180] (2.9,1.6) to[out=0,in=150] (3.1,1.5);
			\node at (1.85,0.4) {\small{$H_2$}};
		\end{tikzpicture}		
	}
	\subcaptionbox{$(k_1,k_2,k_3)=(2,1,2)$ \label{fig:w=2_models_k2=2b}}[.32\textwidth]{
	\begin{tikzpicture}
		\draw[dashed] (0.2,3.2) -- (0,2.2);
		\node at (-0.2,2.8) {\small{$-1$}};
		\draw (0,2.4) -- (0.2,1.4);
		\node at (-0.2,1.8) {\small{$-2$}};
		\draw (0.2,1.6) -- (0,0.6);
		\node at (-0.2,1) {\small{$-2$}};
		\node at (0.3,1) {\small{$V_1$}};
		\draw[dashed] (0,0.8) -- (0.2,-0.2);
		\node at (-0.2,0.2) {\small{$-1$}};
		\draw (1.6,3.2) -- (1.4,2);
		\node at (1.2,2.6) {\small{$-2$}};
		\node at (1.75,2.6) {\small{$V_2$}};
		\draw[dashed] (1.4,2.2) -- (1.6,1);
		\node at (1.2,1.6) {\small{$-1$}};
		\draw (1.6,1.2) -- (1.4,0);
		\node at (1.2,0.6) {\small{$-2$}};
		\draw (3,3.2) -- (2.8,2);
		\node at (2.6,2.6) {\small{$-2$}};
		\node at (3.15,2.6) {\small{$V_3$}};
		\draw[dashed] (2.8,2.2) -- (3,1);
		\node at (3.1,1.8) {\small{$-1$}};
		\draw (3,1.2) -- (2.8,0);
		\node at (2.6,0.6) {\small{$-2$}};
		\draw (-0.2,0) -- (0,0) to[out=0,in=-100] (0.5,0.8);
		\draw (0.6,1.2) to[out=80,in=-100] (0.8,2.25);
		\draw (0.86,2.55) to[out=80,in=180] (1.4,3) -- (3.2,3);
		\node at (2.2,3.2) {\small{$H_1$}};
		\draw (-0.2,2.6) --  (0.2,2.6) to[out=0,in=180] (1,2.4) -- (2.2,2.4) to[out=0,in=180] (2.9,1.6) to[out=0,in=-150] (3.1,1.7);
		\draw (-0.2,1.25) -- (0.2,1.25) to[out=0,in=180] (1.2,0.2) -- (1.9,0.2) to[out=0,in=180] (2.9,1.6) to[out=0,in=150] (3.1,1.5);
		\node at (1.85,0.4) {\small{$H_2$}};
	\end{tikzpicture}		
}
\subcaptionbox{$(k_1,k_2,k_3)=(1,1,3)$ \label{fig:w=2_models_k3=3}}[.34\textwidth]{
	\begin{tikzpicture}
		\draw[dashed] (0.2,3.2) -- (0,2);
		\node at (-0.2,2.6) {\small{$-1$}};
		\draw (0,2.2) -- (0.2,1);
		\node at (-0.2,1.6) {\small{$-2$}};
		\node at (0.3,1.6) {\small{$V_1$}};
		\draw[dashed] (0.2,1.2) -- (0,0);
		\node at (-0.2,0.6) {\small{$-1$}};
		\draw (1.6,3.2) -- (1.4,2);
		\node at (1.2,2.6) {\small{$-2$}};
		\node at (1.75,2.6) {\small{$V_2$}};
		\draw[dashed] (1.4,2.2) -- (1.6,1);
		\node at (1.2,1.6) {\small{$-1$}};
		\draw (1.6,1.2) -- (1.4,0);
		\node at (1.2,0.6) {\small{$-2$}};
		\node at (2.9,1.85) {\small{$-1$}};
		\draw[dashed] (2.5,1.5) -- (3.7,1.7);
		\draw (3.6,3.2) -- (3.4,2);
		\node at (3.2,2.6) {\small{$-2$}};
		\node at (3.75,2.6) {\small{$V_3$}};
		\draw (3.4,2.2) -- (3.6,1);
		\node at (3.25,1.4) {\small{$-2$}};
		\draw (3.6,1.2) -- (3.4,0);
		\node at (3.2,0.6) {\small{$-2$}};
		\draw (-0.2,0.2) -- (0,0.2) to[out=0,in=-100] (0.5,1);
		\draw (0.6,1.4) to[out=80,in=-100] (0.8,2.25);
		\draw (0.86,2.55) to[out=80,in=180] (1.4,3) -- (3.8,3);
		\node at (2.2,3.2) {\small{$H_1$}};
		\draw (-0.2,2.4) -- (2.2,2.4) to[out=0,in=-160] (2.9,1.57) to[out=20,in=-120] (3,1.7);
		\draw (-0.2,1.4) -- (0.2,1.4) to[out=0,in=180] (1.2,0.2) -- (1.9,0.2) to[out=0,in=-160] (2.9,1.57) to[out=20,in=150] (3.05,1.5);
		\node at (2.55,0.4) {\small{$H_2$}};
	\end{tikzpicture}		
}
	\caption{Proof of Lemma \ref{lem:w=2_psi}: examples of $D$'s excluded using Lemma \ref{lem:w=2_basic}\ref{item:H2}.}
	\label{fig:w=2_models}
\end{figure}	
	We have $\beta_{D}(H_2)=3$: indeed, $H_{2}$ meets both tips of $F_{2}$, and a component $V_1$ of $F_1$. Hence the connected component of $D$ containing $H_2$, call it $U$, is an admissible fork. 
	Let $T_{1}$ be the  maximal twig of $D$ containing $V_1$. Since $F_1$ is as in Lemma  \ref{lem:w=2_fibers}\ref{item:FH=2}, we have $\#T_{1}\geq k_{1}$. Recall that the $1$-section $H_1$ meets a tip of $F_2$, too, so $H_{1}$ lies in another twig of $U$, call it $T_{H}$. In particular $\beta_{D}(H_1)\leq 2$. Since $\nu\geq 3$, the inequality in Lemma \ref{lem:w=2_basic}\ref{item:H1_meets} yields $\nu=3$ and $\beta_{D}(H_1)=2$, so $\#T_{H}\geq 3$, see Figure \ref{fig:w=2_models}. Since the fork $U$ is admissible, we have $\#T_{1}\leq 2$ by Lemma \ref{lem:admissible_forks}. Thus $k_{1}\leq 2$. Lemma \ref{lem:w=2_basic}\ref{item:H2} gives $k_{3}\geq 3$, cf.\ Figure \ref{fig:w=2_models_k2=2b} or \ref{fig:w=2_models_k3=3}. Since $F_3$ is as in \ref{lem:w=2_fibers}\ref{item:F_fork} and the connected component of $(F_3)\redd \wedge D\vert$ is a chain, we get $k_3=3$ and $\#T_{H}\geq 5$, so $\#T_{1}=1$ by the admissibility of $U$, see Figure \ref{fig:w=2_models_k3=3}. Thus $\sum_{i} k_{i}=5$, a contradiction with Lemma \ref{lem:w=2_basic}\ref{item:H2}.
\end{proof}

\subsection{Proof of Proposition \ref{prop:ht=3_swaps}: vertical swaps \texorpdfstring{$(X,D)\sqto (Y,D_Y)$}{}}

We now prove the case $\width=2$ of Proposition \ref{prop:ht=3_swaps}, that is, we show that any del Pezzo surface of rank $1$, height $3$ and width $2$ swaps vertically to a del Pezzo surface as in Example \ref{ex:w=2_cha_neq_2} or \ref{ex:w=2_cha=2}. More precisely, we choose a witnessing $\P^1$-fibration $p$ so that the minimalization $\psi\colon X\to \P^2$ constructed in Lemma \ref{lem:w=2_psi} factors as $\psi=\phi\circ\phi_{+}$, where $\phi$ is as in Example \ref{ex:w=2_cha_neq_2} or \ref{ex:w=2_cha=2}, and $\phi_{+}$ is the required vertical swap. 

\begin{notation}\label{not:w=2}
	We fix a birational morphism $\psi\colon X\to \P^2$ as in Lemma \ref{lem:w=2_psi}, and keep the notation introduced there. In particular, $\{p_{0}\}=\psi(H_1)$, and $p_i$ for $i\in \{1,\dots,\nu\}$ is the base point of $\psi^{-1}$ at the intersection of the line $\ll_i=\psi(F_i)$ and the conic $\cc=\psi(H_2)$. As in  Example \ref{ex:w=2_cha_neq_2}, we denote by $p_{ij}$ and  $p_{i}',p_{i}'',\dots$ the points infinitely near to $p_{i}$ lying on the proper transforms of $\ll_{j}$ and $\cc$, respectively. As in Lemma \ref{lem:w=3_swaps}, we sometimes underline these symbols without changing their meaning, just so that we can easily refer to them later: for instance, the symbol $\uline{p_{1}''}$ in Lemma \ref{lem:w=2_swaps}\ref{item:ht=3_exception_swap} still refers to the point $p_{1}''$.
	
	We order the degenerate fibers $F_i$ as in Lemma \ref{lem:w=2_fibers} (so $F_1$ has two $(-1)$-curves, and $F_i$ for $i\in \{2,\dots,\nu\}$ has one), and use the numbers $k_{i}$ introduced there. Note that the resulting order of the lines $\ll_i$ slightly differs from the one in Example \ref{ex:w=2_cha_neq_2}: namely, the line $\ll_1$ is now the image of the fiber of $p$ containing two $(-1)$-curves, which in Example \ref{ex:w=2_cha_neq_2}\ref{item:ht=3_exception} is called $\ll_2$, see Figure \ref{fig:ht=3_exception}. Hence in case \ref{ex:w=2_cha_neq_2}\ref{item:ht=3_exception}, we interchange the lines $\ll_1$ and $\ll_2$. In all the remaining cases the notation in Examples \ref{ex:w=2_cha_neq_2}, \ref{ex:w=2_cha=2} is compatible with the current one.
\end{notation}

\begin{lemma}[Proposition \ref{prop:ht=3_swaps}, case $\width(\bar{X})=2$]\label{lem:w=2_swaps}
	Let $\bar{X}$ be as in \eqref{eq:assumption_w=2}, i.e.\ a log terminal del Pezzo surface of rank one, height $3$ and width $2$. Then $\bar{X}$ swaps vertically to one of the surfaces $\bar{Y}$ from Example \ref{ex:w=2_cha_neq_2} or \ref{ex:w=2_cha=2}.
	
	More precisely, one can choose a witnessing $\P^1$-fibration of the minimal resolution $X$ of $\bar{X}$ such that for the birational morphism $\psi\colon X\to \P^2$ from  Lemma \ref{lem:w=2_psi}, one of the following holds, see Notation \ref{not:w=2}.
		\begin{enumerate}
			\item\label{item:ht=3_exception_swap} $\cha\kk\neq 2$ and  $\{p_0,p_{01},p_{1},p_{1}',\uline{p_{1}''},p_{2},p_{22},p_{3},p_{3}',p_{3}''\}\subseteq \Bs\psi^{-1}$, see Figure \ref{fig:ht=3_exception_swap}, 
			\item \label{item:w=2_A1+A2+A5_swap} $\cha\kk\neq 2$ and  $\{p_0,p_{01},p_{1},\uline{p_{11}},p_{2},p_{2}',p_{2}'',p_{3},p_{3}',\uline{q}\}\subseteq \Bs\psi^{-1}$ where $q$ is a point infinitely near to $p_3$ which lies on the proper transform of $\cc$ or of $\ll_3$, see Figure \ref{fig:w=2_A1+A2+A5_swap} or \ref{fig:w=2_A1+A2+A5_swap_b},
			\item\label{item:w=2_2A1+2A3_swap} $\cha\kk\neq 2$ and 
			$\{p_0,p_{01},p_{1},p_{1}',p_{2},p_{2}',p_{3},p_{3}',\uline{q}\}\subseteq \Bs\psi^{-1}$, where $q$ is a point infinitely near to $p_3$  which lies on the proper transform of $\cc$, or of $\ll_3$, or of the first exceptional curve over $p_3$, see Figures \ref{fig:w=2_2A1+2A3_swap_H}--\ref{fig:w=2_2A1+2A3_swap_T},
			\item\label{item:w=2_cha=2_nu=4} $\cha\kk=2$, $\nu=4$ and $\{p_{0},p_{01},p_{1},p_{2},p_{2}',p_{3},p_{3}',p_{4},p_{4}'\}\subseteq \Bs \psi^{-1}$, see Figure \ref{fig:swap-to-nu=4_small} or \ref{fig:swap-to-nu=4_large},
			\item\label{item:w=2_cha=2_nu=3} $\cha\kk=2$, $\nu=3$ and $\{p_{0},p_{01},p_{1},p_{1}',p_{2},p_{2}', p_{3},p_{3}'\}\subseteq \Bs \psi^{-1}$, see Figure \ref{fig:swap-to-nu=3}.
		\end{enumerate}
	Write $\psi=\phi\circ\phi_{+}$, where $\phi$ is a blowup at all the above points except for the underlined ones and, if $p_{1}'\in \Bs\psi^{-1}$ and case \ref{item:w=2_cha=2_nu=4} holds, at $p_{1}'$. Then $\phi_{+}\colon (X,D)\sqto (Y,D_Y)$ is a vertical swap onto a minimal resolution of $\bar{Y}$.
\end{lemma}
\begin{figure}[ht]
	\subcaptionbox{\ref{lem:w=2_swaps}\ref{item:ht=3_exception_swap} $\sqto$ \ref{ex:w=2_cha_neq_2}\ref{item:ht=3_exception} \label{fig:ht=3_exception_swap}}[.3\linewidth]{
		\begin{tikzpicture}
			\node at (1.9,3.2) {\small{$H_1$}};
			\draw (-0.6,3) -- (3.1,3);
			\draw[dashed] (-0.4,3.2) -- (-0.6,2.2);
			\node at (-0.3,2.6) {\small{$A_0$}};
			\draw (-0.6,2.4) -- (-0.4,1.4);
			\node at (-0.8,2) {\small{$-3$}};
			\node at (-0.3,2) {\small{$L_1$}};
			\draw[thick] (-0.4,1.6) -- (-0.6,0.6); 
			\draw[dashed, thick] (-0.6,1.2) -- (0.4,1);
			\node at (0,1.35) {\small{$A_1$}};
			\draw (-0.6,0.8) -- (-0.4,-0.2);
			\draw (1.3,3.2) -- (1.1,2);
			\node at (1.4,2.6) {\small{$L_2$}};
			\draw[dashed] (1.1,2.2) -- (1.3,1);
			\node at (1.4,1.7) {\small{$A_2$}};
			\draw (1.3,1.2) -- (1.1,0); 
			\draw (3,3.2) -- (2.8,2);
			\node at (3.1,2.6) {\small{$L_3$}};
			\draw (2.8,2.2) -- (3,1);
			\node at (2.4,1.75) {\small{$A_3$}};
			\draw[dashed] (3,1.6) -- (1.8,1.4);
			\draw (3,1.2) -- (2.8,0); 
			\draw (-0.1,1.2) to[out=-30,in=170] (0.1,1.05) to[out=-10,in=180] (1.2,2.4) -- (1.4,2.4) to[out=0,in=-170] (2.4,1.5) to[out=10,in=-150] (2.7,1.65);
			\draw (-0.1,0.95) to[out=10,in=170] (0.1,1.05) to[out=-10,in=180] (1.2,0.4) -- (1.4,0.4) to[out=0,in=-170] (2.4,1.5) to[out=10,in=150] (2.7,1.45);
			\node at (2.1,0.8) {\small{$H_2$}};
			\node at (1.6,0.85) {\small{$-3$}};
		\end{tikzpicture}
	}	
	\subcaptionbox{\ref{lem:w=2_swaps}\ref{item:w=2_A1+A2+A5_swap}, $q\in H_2$ $\sqto$ \ref{ex:w=2_cha_neq_2}\ref{item:w=2_A1+A2+A5}
		\label{fig:w=2_A1+A2+A5_swap}}[.3\linewidth]{
		\begin{tikzpicture}
			\draw (-0.4,3) -- (3.1,3);
			\node at (1.9,3.2) {\small{$H_1$}};
			\draw (-0.2,3.2) -- (-0.4,2);
			\node at (-0.1,2.4) {\small{$L_2$}};	
			\draw (-0.4,2.2) -- (-0.2,1);			
			\draw[dashed] (-0.4,1.4) -- (0.8,1.6);
			\node at (0,1.7) {\small{$A_2$}};
			\draw (-0.2,1.2) -- (-0.4,0);
			\draw[dashed] (1.3,3.2) -- (1.1,2.2);
			\node at (1.4,2.6) {\small{$A_0$}};
			\draw (1.1,2.4) -- (1.3,1.4);
			\node at (0.9,2) {\small{$-3$}};
			\node at (1.4,2) {\small{$L_1$}};				
			\draw[dashed, thick] (1.3,1.6) -- (1.1,0.6);
			\node at (1.4,1.1) {\small{$A_1$}};
			\draw[thick] (1.1,0.8) --  (1.3,-0.2);
			\draw (3,3.2) -- (2.8,2);
			\node at (3.1,2.4) {\small{$L_3$}};
			\draw[thick] (2.8,2.2) -- (3,1);
			\node at (2.4,1.75) {\small{$A_3$}};
			\draw[thick, dashed] (3,1.6) -- (1.8,1.4);
			\draw (3,1.2) -- (2.8,0); 
			\draw (0.3,1.6) to[out=-10,in=-170] (0.5,1.55) to[out=10,in=180] (1.1,1.8) -- (1.4,1.8) 
			to[out=0,in=-170] (2.4,1.5) to[out=10,in=-150] (2.7,1.65);
			\draw (0.3,1.4) to[out=20,in=-170] (0.5,1.55) to[out=10,in=180] (1.2,0.1) -- (1.4,0.1) to[out=0,in=-170] (2.4,1.5) to[out=10,in=150] (2.7,1.45);
			\node at (2.1,0.6) {\small{$H_2$}};
			\node at (1.6,0.6) {\small{$-3$}};
		\end{tikzpicture}
	}
	\subcaptionbox{\ref{lem:w=2_swaps}\ref{item:w=2_A1+A2+A5_swap}, $q\in L_3$ $\sqto$ \ref{ex:w=2_cha_neq_2}\ref{item:w=2_A1+A2+A5}
		\label{fig:w=2_A1+A2+A5_swap_b}}[.3\linewidth]{
		\begin{tikzpicture}
			\draw (-0.4,3) -- (2.7,3);
			\node at (1.9,3.2) {\small{$H_1$}};
			\draw (-0.2,3.2) -- (-0.4,2);
			\node at (-0.1,2.4) {\small{$L_2$}};	
			\draw (-0.4,2.2) -- (-0.2,1);			
			\draw[dashed] (-0.4,1.4) -- (0.8,1.6);
			\node at (0,1.7) {\small{$A_2$}};
			\draw (-0.2,1.2) -- (-0.4,0);
			\draw[dashed] (1.3,3.2) -- (1.1,2.2);
			\node at (1.4,2.6) {\small{$A_0$}};
			\draw (1.1,2.4) -- (1.3,1.4);
			\node at (0.9,2) {\small{$-3$}};
			\node at (1.4,2) {\small{$L_1$}};			
			\draw[dashed, thick] (1.3,1.6) -- (1.1,0.6);
			\node at (1.4,1) {\small{$A_1$}};
			\draw[thick] (1.1,0.8) --  (1.3,-0.2);
			\draw (2.6,3.2) -- (2.4,2.2);
			\node at (2.7,2.6) {\small{$L_3$}};
			\node at (2.2,2.6) {\small{$-3$}};
			\draw[thick, dashed] (2.4,2.4) -- (2.6,1.4);
			\node at (2.7,2) {\small{$A_3$}};
			\draw[thick] (2.6,1.6) -- (2.4,0.6); 
			\draw (2.4,0.8) -- (2.6,-0.2);
			\draw (0.3,1.6) to[out=-10,in=-170] (0.5,1.55) to[out=10,in=180] (1,1.8) -- (1.4,1.8)   to[out=0,in=180] (2.5,1.1) to[out=0,in=180] (2.8,1.15);
			\draw (0.3,1.4) to[out=20,in=-170] (0.5,1.55) to[out=10,in=180] (1.2,0.1) -- (1.4,0.1) to[out=0,in=180] (2.5,1.1) to[out=0,in=180] (2.8,1.05);
			\node at (2.1,0.3) {\small{$H_2$}};
		\end{tikzpicture}	
	}	
	\smallskip
		
	\subcaptionbox{\ref{lem:w=2_swaps}\ref{item:w=2_2A1+2A3_swap}, $q\in H_2$ $\sqto$ \ref{ex:w=2_cha_neq_2}\ref{item:w=2_2A1+2A3}
		\label{fig:w=2_2A1+2A3_swap_H}}[.3\linewidth]{
		\begin{tikzpicture}
			\draw (0,3) -- (3.1,3);
			\node at (2.1,3.2) {\small{$H_1$}};
			\draw (0.2,3.2) -- (0,2);
			\node at (0.3,2.4) {\small{$L_2$}};	
			\draw[dashed] (0,2.2) -- (0.2,1);
			\node at (0.3,1.8) {\small{$A_2$}};			
			\draw (0.2,1.2) -- (0,0);
			\draw[dashed] (1.3,3.2) -- (1.1,2.2);
			\node at (1.4,2.6) {\small{$A_0$}};
			\draw (1.1,2.4) -- (1.3,1.4);
			\node at (1.4,2) {\small{$L_1$}};				
			\draw (1.3,1.6) -- (1.1,0.6);
			\node at (1.4,0.35) {\small{$A_1$}};
			\draw[dashed] (1.1,0.8) --  (1.3,-0.2);
			\draw (3,3.2) -- (2.8,2);
			\node at (3.1,2.4) {\small{$L_3$}};
			\draw[thick] (2.8,2.2) -- (3,1);
			\draw[dashed,thick] (3,1.6) -- (1.8,1.4);
			\node at (2.4,1.75) {\small{$A_3$}};
			\draw (3,1.2) -- (2.8,0); 
			\draw (-0.2,1.55) to[out=0,in=180] (0.1,1.5) to[out=0,in=180] (1,1.8) -- (1.4,1.8) 
			to[out=0,in=-170] (2.4,1.5) to[out=10,in=-150] (2.7,1.65);
			\draw (-0.2,1.45) to[out=0,in=180] (0.1,1.5) to[out=0,in=180] (1.2,0) -- (1.4,0) to[out=0,in=-170] (2.4,1.5) to[out=10,in=150] (2.7,1.45);
			\node at (2.15,0.8) {\small{$H_2$}};
			\node at (1.65,0.8) {\small{$-3$}};
		\end{tikzpicture}
	}		
	\subcaptionbox{\ref{lem:w=2_swaps}\ref{item:w=2_2A1+2A3_swap}, $q\in L_3$ $\sqto$ \ref{ex:w=2_cha_neq_2}\ref{item:w=2_2A1+2A3}
		\label{fig:w=2_2A1+2A3_swap_L}}[.3\linewidth]{
		\begin{tikzpicture}
			\draw (0,3) -- (2.7,3);
			\node at (2.1,3.2) {\small{$H_1$}};
			\draw (0.2,3.2) -- (0,2);
			\node at (0.35,2.6) {\small{$L_2$}};	
			\draw[dashed] (0,2.2) -- (0.2,1);
			\node at (0.3,1.8) {\small{$A_2$}};				
			\draw (0.2,1.2) -- (0,0);
			\draw[dashed] (1.3,3.2) -- (1.1,2.2);
			\node at (1.4,2.6) {\small{$A_0$}};
			\draw (1.1,2.4) -- (1.3,1.4);
			\node at (1.4,2) {\small{$L_1$}};				
			\draw (1.3,1.6) -- (1.1,0.6);
			\node at (1.4,0.35) {\small{$A_1$}};
			\draw[dashed] (1.1,0.8) --  (1.3,-0.2);
			\draw (2.6,3.2) -- (2.4,2.2);
			\node at (2.7,2.6) {\small{$L_3$}};
			\node at (2.2,2.6) {\small{$-3$}};
			\draw[thick, dashed] (2.4,2.4) -- (2.6,1.4);
			\node at (2.7,2) {\small{$A_3$}};
			\draw[thick] (2.6,1.6) -- (2.4,0.6); 
			\draw (2.4,0.8) -- (2.6,-0.2);
			\draw (-0.2,1.55) to[out=0,in=180] (0.1,1.5) to[out=0,in=180] (1,1.8) -- (1.6,1.8)  to[out=0,in=180] (2.5,1.1) to[out=0,in=180] (2.8,1.15);
			\draw (-0.2,1.45) to[out=0,in=180] (0.1,1.5) to[out=0,in=180] (1.2,0) -- (1.4,0) to[out=0,in=180] (2.5,1.1) to[out=0,in=180] (2.8,1.05);
			\node at (2.1,0.3) {\small{$H_2$}};
		\end{tikzpicture}	
	}			
	\subcaptionbox{\ref{lem:w=2_swaps}\ref{item:w=2_2A1+2A3_swap}, $q\in G_3$ $\sqto$ \ref{ex:w=2_cha_neq_2}\ref{item:w=2_2A1+2A3} \label{fig:w=2_2A1+2A3_swap_T}}[.3\linewidth]{
		\begin{tikzpicture}
			\draw (0,3) -- (2.7,3);
			\node at (2.05,3.2) {\small{$H_1$}};
			\draw (0.2,3.2) -- (0,2);
			\node at (0.35,2.6) {\small{$L_2$}};	
			\draw[dashed] (0,2.2) -- (0.2,1);	
			\node at (0.3,1.8) {\small{$A_2$}};			
			\draw (0.2,1.2) -- (0,0);
			\draw[dashed] (1.3,3.2) -- (1.1,2.2);
			\node at (1.4,2.6) {\small{$A_0$}};
			\draw (1.1,2.4) -- (1.3,1.4);
			\node at (1.4,2) {\small{$L_1$}};				
			\draw (1.3,1.6) -- (1.1,0.6);
			\node at (1.4,0.35) {\small{$A_1$}};
			\draw[dashed] (1.1,0.8) --  (1.3,-0.2);
			\draw (2.6,3.2) -- (2.4,2.2);
			\node at (2.7,2.6) {\small{$L_3$}};
			\draw[thick] (2.4,2.4) -- (2.6,1.4);
			\node at (2.7,1) {\small{$A_3$}};
			\draw[thick, dashed] (2.6,1.6) -- (2.4,0.6); 
			\draw (2.4,0.8) -- (2.6,-0.2);
			\node at (2.3,0) {\small{$-3$}};
			\draw (-0.2,1.55) to[out=0,in=180] (0.1,1.5) to[out=0,in=180] (1,1.8) -- (2.5,1.8) to[out=0,in=180] (2.8,1.85);
			\draw (-0.2,1.45) to[out=0,in=180] (0.1,1.5) to[out=0,in=180] (1.1,0) -- (1.3,0) to[out=0,in=180] (2.5,1.8) to[out=0,in=180] (2.8,1.75);
			\node at (2.05,0.6) {\small{$H_2$}};
		\end{tikzpicture}
	}
	\caption{Log surfaces $(X_{\tau},D_{\tau})$ of height $3$ and width $2$, obtained in case $\cha\kk\neq 2$ by blowing up all points listed in Lemma \ref{lem:w=2_swaps}, see Notation \ref{not:phi_+_h=2}. Swapping bold curves gives a log resolution of a surface from Example \ref{ex:w=2_cha_neq_2}, see Lemma \ref{lem:w=2_swaps}.}
	\label{fig:w=2_swaps}
\end{figure}
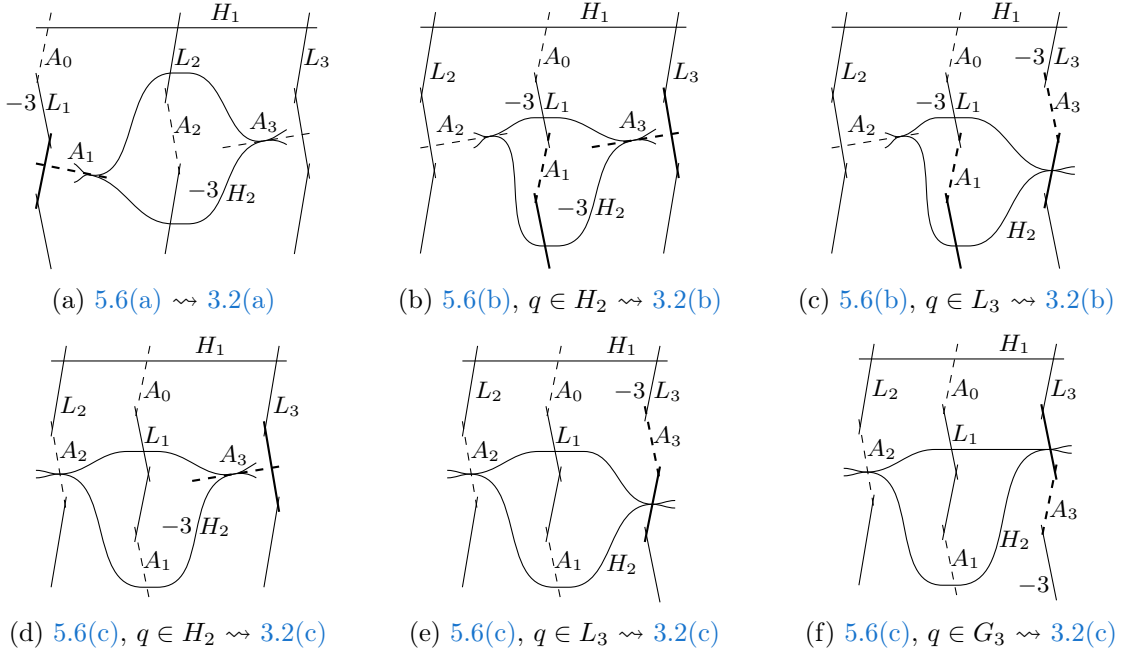
\begin{proof}
	By Lemma \ref{lem:w=2_psi}\ref{item:psi_p_i}, $\Bs\psi^{-1}$ consists of $p_0,\dots,p_{\nu}$ and some infinitely near points. Recall from \eqref{eq:w=2_Sigma} that the fiber $F_1$ has two $(-1)$-curves, so by Lemma \ref{lem:w=2_psi}\ref{item:psi_p_i}, one of them lies in $\psi^{-1}(p_1)$, and the other in $\psi^{-1}(p_0)$. It follows that $p_{01}\in \Bs\psi^{-1}$. 
	Moreover, for $i\in \{2,\dots,\nu\}$ we have $p_{ii}\in \Bs\psi^{-1}$ because otherwise $\psi^{-1}_{*}\ll_{i}$ would be a $(-1)$-curve in $D$, which is impossible. Note that if $\ll_i$ is tangent to $\cc$ then $p_{ii}=p_{i}'$.
	\smallskip
	
	Assume first that $\cha\kk=2$. By Lemma \ref{lem:w=2_psi}\ref{item:psi_cha=2} we have $\nu\in \{3,4\}$. If $\nu=4$ then \ref{item:w=2_cha=2_nu=4} holds. Assume $\nu=3$. If $p_{1}'\not\in \Bs\psi^{-1}$ then the pencil of lines passing through $p_1$ pulls back to a $\P^1$-fibration of $X$ such that $D\hor$ consists of three $1$-sections, namely $\psi^{-1}_{*}\cc$, $\psi^{-1}_{*}\ll_{2}$, $\psi^{-1}_{*}\ll_{3}$, contrary to the assumption $\width(\bar{X})=2$. Thus \ref{item:w=2_cha=2_nu=3} holds.
	\smallskip
	
	Assume now that $\cha\kk\neq 2$. By Lemma \ref{lem:w=2_psi}\ref{item:psi_cha-neq-2}, $\nu=3$ and exactly one of the lines $\ll_i$ is not tangent to $\cc$.
	
	\begin{casesp}
	\litem{$\ll_1$ is not tangent to $\cc$}\label{case:non-tangent} Then $F_{1}$ and $F_{j}$, $j\in \{2,3\}$ are as in Lemma \ref{lem:w=2_fibers}\ref{item:FH=2} and \ref{item:F_fork}, respectively (recall that case \ref{lem:w=2_fibers}\ref{item:dumb} does not occur by Lemma \ref{lem:w=2_basic}\ref{item:H1}). It follows that $p_{01},p_{2}',p_{3}'\in \Bs\psi^{-1}$. We have $k_1\geq 1$ since $p_{1}\in \Bs\psi^{-1}$. Assume that $k_{1}\geq 2$, so $p_{1}'\in \Bs\psi^{-1}$. If $\psi^{-1}$ has no base points infinitely near to $p_{2}'$, $p_{3}'$ then the pencil of conics tangent to $\ll_{j}$ at $p_{j}$, $j\in \{2,3\}$ pulls back to a $\P^{1}$-fibration of $X$ of height~2 with respect to $D$, which is impossible. Thus after possibly reversing the order of $\{F_2,F_3\}$, we get that $\psi^{-1}$ has a base point $q$ infinitely near to $p_{3}'$. Let $\eta$ be a blowup at $p_{0},p_{01},p_1,p_{1}',p_{2},p_{2}',p_{3},p_{3}'$, so $\psi=\eta\circ\eta_{+}$ for a vertical swap $\eta_{+}\colon (X,D)\sqto (X_{\eta},D_{\eta})$, and $q\in \Bs\eta_{+}^{-1}$. To prove that  \ref{item:w=2_2A1+2A3_swap} holds, it remains to show that $q\in D_{\eta}$. Suppose the contrary, and let $(X',D')\sqto (X_{\eta},D_{\eta})$ be a vertical swap given by a blowup at $q$. Then $D'=[2]+\langle 2;[2],[2,2,2],[2,2,2]\rangle$ contains a non-admissible fork, a contradiction with Lemma \ref{lem:cascades}\ref{item:cascades-still-dP}. 
	
	Assume now that $k_{1}=1$, so $p_{1}'\not\in \Bs\psi^{-1}$. If $\psi^{-1}$ has no base points infinitely near to $p_{1}$, then the pencil of lines through $p_{1}$ pulls back to a $\P^{1}$-fibration of $X$ such that $D\hor$ consists of three $1$-sections, namely $\psi^{-1}_{*}\cc$, $\psi^{-1}_{*}\ll_{2}$ and $\psi^{-1}_{*}\ll_{3}$: this is impossible by assumption $\width(\bar{X})=2$. Since $D$ has no circular subdivisor, we get $p_{11}\in \Bs\psi^{-1}$. Since $k_1=1$, Lemma \ref{lem:w=2_basic}\ref{item:H2} gives $k_{2}+k_{3}\geq 5$, so, say, $k_{3}\geq 3$, i.e.\ $p_{3}''\in \Bs\psi^{-1}$. Let $\eta$ be a blowup at those base points of $\psi^{-1}$ which we have already found, that is, at $p_0$, $p_{01}$, $p_{1}$, $p_{11}$, $p_{2}$, $p_{2}'$, $p_{3}$, $p_{3}'$ and $p_3''$, see Figure \ref{fig:w=2_swaps_k1=1}. As before, write $\psi=\eta\circ \eta_{+}$ for a vertical swap $\eta_{+}\colon (X,D)\sqto (X_{\eta},D_{\eta})$. Put $H_{j}^{\eta}=\eta_{+}(H_{j})$, $L_{j}=\eta^{-1}_{*}\ll_{j}$, and let $A_{j}\subseteq X_{\eta}$ be the $(-1)$-curve in $\eta^{-1}(p_j)$, cf.\ Notation \ref{not:phi_+_h=2}. 
	
	Assume that $A_{2}$ contains a base point of $\psi^{-1}$, say $q$. If $q\in H_{2}^{\eta}\cup L_2$ then after interchanging $F_2$ with $F_3$ we get \ref{item:w=2_A1+A2+A5_swap}. Assume $q\not\in H_{2}^{\eta}\cup L_2$. Let $(X',D')\sqto (X_{\eta},D_{\eta})$ be the vertical swap given by the blowup at $q$. Let $U'$ be the connected component of $D'$ containing the proper transform of $A_2$. If $q\not\in D_{\eta}$ then the proper transforms of both $H_2$ and $A_2$ are branching in $U'$; and if $q\in D_{\eta}$ then $U'$ is a non-admissible fork $\langle 2;[2],[3],[(2)_{7}]\rangle$. In either case, $D'$ does not contract to log terminal singularities, a contradiction with Lemma  \ref{lem:cascades}\ref{item:cascades-still-dP}.
	
	Thus we can assume $A_2\cap \Bs\eta_{+}^{-1}=\emptyset$. Put $L=\eta^{-1}_{*}\ll$, where $\ll$ is the line joining $p_{2}$ with $p_{3}$. For $j\in \{1,2,3\}$ let $G_j\subseteq D_{\eta}$ be the first exceptional curve over $p_j$, and let $T\subseteq D_{\eta}$ be the second exceptional curve over $p_{3}$. Let $\tilde{p}$ be the $\P^{1}$-fibration of $X_{\eta}$ induced by the pencil of conics tangent to $\cc$ at $p_{2}$, $p_3$. The horizontal part of $D_{\eta}+\sum_{j}A_{j}$ consists of a $2$-section $L_{1}$ and $1$-sections $T$, $A_{2}$. Since $\Bs\eta_{+}^{-1}\cap A_2=\emptyset$, we have $(\eta_{+}^{-1})_{*}A_{2}\not\subseteq D$. It follows that $\tilde{p}\circ\eta_{+}$ is a $\P^{1}$-fibration of $X$ of height $3$ with respect to $D$, such that $\#D\hor=2$.
\begin{figure}[htbp]
		\begin{tikzpicture}
		\begin{scope}
			\draw (-0.5,3) -- (3.1,3);
			\node at (-0.3,3.2) {\small{$H_1^{\eta}$}};
			\draw (0.2,3.2) -- (0,2);
			\node at (0.3,2.4) {\small{$L_2$}};	
			\draw[dashed, thick] (0,2.2) -- (0.2,1);
			\node at (0.3,1.85) {\small{$A_2$}};			
			\draw (0.2,1.2) -- (0,0);
			\node at (0.3,0.2) {\small{$G_2$}};
			\draw[dashed] (1.4,3.2) -- (1.2,2.2);
			\node at (1.55,2.6) {\small{$A_0$}};
			\draw[very thick] (1.2,2.4) -- (1.4,1.4);
			\node at (1.5,2.1) {\small{$L_1$}};
			\node at (0.95,2.1) {\small{$-3$}};
			\draw[dashed] (1.4,1.6) -- (1.2,0.6);
			\node at (1.5,1) {\small{$A_1$}};
			\draw (1.2,0.8) -- (1.4,-0.2); 
			\node at (1.65,0) {\small{$G_1$}};
			\draw (3,3.2) -- (2.8,2);
			\node at (3.1,2.4) {\small{$L_3$}};
			\draw[thick] (2.8,2.2) -- (3,1);
			\node at (3.05,1.75) {\small{$T$}};
			\draw[dashed] (3,1.5) -- (1.8,1.7);
			\node at (2.5,1.85) {\small{$A_3$}};
			\draw (3,1.2) -- (2.8,0); 
			\node at (3.1,0.2) {\small{$G_3$}};
			\draw (-0.6,1.55) -- (-0.2,1.55) to[out=0,in=180] (0.1,1.5) to[out=0,in=180] (1,1.9) -- (1.8,1.9)  to[out=0,in=170] (2.4,1.6) to[out=-10,in=180] (2.6,1.65);
			\draw (-0.6,1.45) -- (-0.2,1.45) to[out=0,in=180] (0.1,1.5) to[out=0,in=180] (1.2,0.4) -- (1.4,0.4) to[out=0,in=170] (2.4,1.6) to[out=-10,in=180] (2.6,1.45);
			\node at (-0.4,1.8) {\small{$H_2^{\eta}$}};
			\draw[->] (-1,1.5) -- (-2,1.5);
			\node at (-1.5,1.7) {\small{$\eta$}};
			\draw[->] (3.6,1.5) -- (4.6,1.5);
			\node at (4.1,1.7) {\small{$\tilde{\eta}$}};
			\draw[densely dashed] (-0.6,0.5) -- (0.2,0.5) to[out=0,in=-120] (0.65,0.8);
			\draw[densely dashed] (0.75,1) to[out=60,in=180] (1.3,1.75) to[out=0,in=120] (1.85,1.1);  
			\draw[densely dashed] (1.95,0.9) to[out=-60,in=180] (2.9,0.5) -- (3.1,0.5);
			\node at (-0.5,0.7) {\small{$L$}};
		\end{scope}
		\begin{scope}[shift={(-4,1)}]
			\draw (0,0) circle (1);
			\node at (0.9,-0.7) {\small{$\cc$}};
			\draw[add= 0.1 and 1] (0,2) to (-0.866,0.5);
			\node at (-1.3,-0.7) {\small{$\ll_2$}};
			\draw[add= 0.1 and 1] (0,2) to (0.866,0.5);
			\node at (1.8,-0.7) {\small{$\ll_3$}};
			\draw (0,2.2) -- (0,-1.2);
			\node at (0.2,0) {\small{$\ll_1$}};
			\filldraw (0,2) circle (0.06);
			\node at (0.3,2) {\small{$p_0$}};
			\draw[-stealth] (-0.1,1.7) -- (-0.1,1.3);
			\filldraw (-0.866,0.5) circle (0.06);
			\node at (-1.1,0.6) {\small{$p_2$}};
			\draw (0,0) [-stealth, partial ellipse=150:180:0.9 and 0.9];
			\filldraw (0.866,0.5) circle (0.06);
			\node at (1.1,0.6) {\small{$p_3$}};
			\draw (0,0) [-stealth, partial ellipse=30:60:0.9 and 0.9];
			\draw (0,0) [-stealth, partial ellipse=35:65:0.8 and 0.8];
			\draw (0,0) [-stealth] (-0.1,-0.9) -- (-0.1,-0.5);
			\filldraw (0,-1) circle (0.06);
			\node at (0.2,-0.8) {\small{$p_1$}};
			\draw[densely dashed] (-1.7,0.5) -- (1.5,0.5);
			\node at (-1.6,0.7) {\small{$\ll$}}; 
		\end{scope}
		\begin{scope}[shift={(6.4,1)}]
			\draw (0,0) circle (1);
			\node at (1,-1) {\small{$\tilde{\eta}(L_1)$}};
			\draw[add= 0.1 and 1] (0,2) to (-0.866,0.5);
			\node at (-1.2,-1) {\small{$\tilde{\eta}(G_3)$}};
			\draw[add= 0.1 and 1] (0,2) to (0.866,0.5);
			\node at (2.1,-0.7) {\small{$\tilde{\eta}(L_3)$}};
			\draw (0,2.2) -- (0,-1.2);
			\node at (0.5,0) {\small{$\tilde{\eta}(H_2^{\eta})$}};
			\filldraw (0,2) circle (0.06);
			\node at (0.5,2) {\small{$\tilde{\eta}(T)$}};
			\draw[-stealth] (-0.1,1.7) -- (-0.1,1.3);
			\filldraw (-0.866,0.5) circle (0.06);
			\draw (0,0) [-stealth, partial ellipse=150:180:0.9 and 0.9];
			\filldraw (0.866,0.5) circle (0.06);
			\draw (0,0) [-stealth, partial ellipse=30:60:0.9 and 0.9];
			\draw (0,0) [-stealth, partial ellipse=35:65:0.8 and 0.8];
			\filldraw (0,-1) circle (0.06);
			\draw (0,0) [-stealth, partial ellipse=275:305:0.9 and 0.9];
			\draw[densely dashed] (-1.5,0.5) -- (1.7,0.5);
			\node at (1.6,0.7) {\small{$\tilde{\eta}(A_2)$}}; 
		\end{scope}
	\end{tikzpicture}
	\caption{Proof of Lemma \ref{lem:w=2_swaps}, Case \ref{case:non-tangent}: new $\P^1$-fibration with $k_1\geq 2$.}
	\label{fig:w=2_swaps_k1=1}
\end{figure}
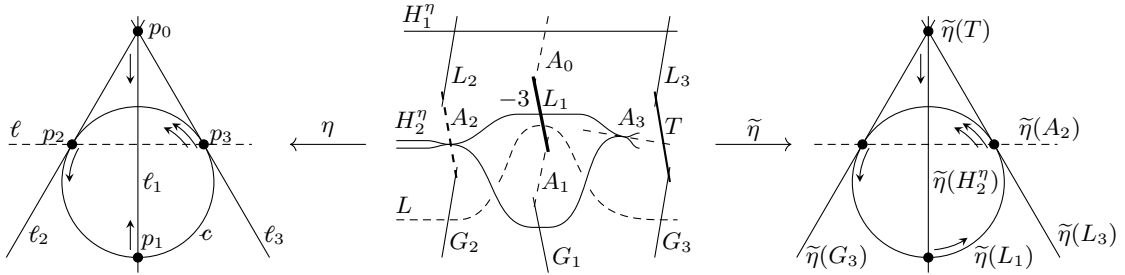

	 Degenerate fibers of $\tilde{p}$ are supported on 
	$\tilde{F}_{1}\de A_{3}+H_{2}^{\eta}+G_1+A_1=[1\aadec{1},2\aadec{2},2,1\aadec{2}]$; 
	$\tilde{F}_{2}\de G_2+L+G_3=[2,1\aadec{2},2\aadec{1}]$ and 
	$\tilde{F}_{3}\de H_{1}^{\eta}+L_{2}+A_0+L_3=\langle 2;[2],[1\aadec{2}],[2\aadec{1}]\rangle$, where the numbers decorated by $\aadec{1}$ and $\aadec{2}$ refer to the components meeting the $1$-section $T$ and the $2$-section $L_1$, respectively. Define $\tilde{\eta}\colon X_{\eta} \to \P^{2}$ as the contraction of $(\tilde{F}_{1}-H_{2}^{\eta})+(\tilde{F}_{2}-G_{3})+(\tilde{F}_{3}-L_3)+T$, and choose coordinates on the target $\P^2$ so that $\ll_{j}=\tilde{\eta}(\tilde{F}_{j})$, $\cc=\tilde{\eta}(L_1)$. Replacing $p$ with $\tilde{p}\circ\eta_{+}$ and $\psi$ by $\tilde{\eta}\circ\eta_{+}$, we are again in the current Case \ref{case:non-tangent} when $\ll_1$ is not tangent to $\cc$. However, for this new $\P^1$-fibration we have $k_1\geq 2$: indeed, the fiber $\tilde{F}_1$ is as in Lemma \ref{lem:w=2_fibers}\ref{item:FH=2} with $k_1=2$, and further blowups within $\eta_{+}$ can only increase $k_1$. So, we already know that \ref{item:w=2_2A1+2A3_swap} holds for these new $p$ and $\psi$. In fact, $\tilde{\eta}$ is exactly a blowup at the points listed in \ref{item:w=2_2A1+2A3_swap}, with $q=p_{3}''$.

	\litem{$\ll_1$ is tangent to $\cc$} Interchanging $F_2$ with $F_3$, if needed, we can assume that $\ll_2$ is not tangent to $\cc$. Now the fibers $F_{1}$, $F_{2}$, $F_{3}$ are as in Lemma \ref{lem:w=2_fibers}\ref{item:FH=1}, \ref{item:F_columnar}, \ref{item:F_fork}, respectively, so   $p_{1}',p_{22}\in \Bs\psi^{-1}$ and $k_2=1$, $k_1,k_3\geq 2$. If $k_{1},k_{3}\geq 3$ then $p_{1}'',p_{3}''\in \Bs\psi^{-1}$; so \ref{item:ht=3_exception_swap} holds. Thus we can assume $k_{a}=2$ for some $a\in \{1,3\}$.
	
	Write $\{1,3\}=\{a,b\}$. Since $k_{2}=1$, $k_a=2$, Lemma \ref{lem:w=2_basic}\ref{item:H2} implies that $k_{b}\geq 3$, i.e.\ $p_{b}''\in \Bs\psi^{-1}$. Let $\eta$ be a blowup at $p_{0}$, $p_{01}$, $p_{1}$, $p_{1}'$, $p_2$, $p_{22}$, $p_3$, $p_{3}'$, $p_{b}''\in \Bs\psi^{-1}$, so $\psi=\eta\circ\eta_{+}$ for some vertical swap $\eta_{+}\colon  (X,D)\sqto (X_{\eta},D_{\eta})$. We keep the notation as in Case \ref{case:non-tangent}, that is, $H_{j}^{\eta}=\eta_{+}(H_j)$, $L_j=\eta^{-1}_{*}\ll_j$, $A_j$, $G_j$ are the $(-1)$-curve and the first exceptional curve over $p_j$, and $T$ is the second exceptional curve over $p_{b}$. Let $\tilde{p}$ be the $\P^1$-fibration of $X_{\eta}$ induced by the pencil of conics tangent to $\cc$ at $p_{1}$, $p_{3}$. As before, the horizontal part of $D_{\eta}+\sum_{j}A_j$ consists of a $2$-section $L_2$ and $1$-sections $A_{a}$ and $T$.
	
	Suppose that $A_{a}$ contains a base point of $\eta_{+}^{-1}$, call it $q$. Let $\theta$ be the composition of $\eta$ with the blowup at $q$, so $\psi=\theta\circ \theta_{+}$ for some vertical swap $(X,D)\sqto (X_{\theta},D_{\theta})$, cf.\ Notation \ref{not:phi_+_h=2}. If $q\not \in D_{\eta}$ then the proper transforms of both $H_2$ and $A_{a}$ are branching in the same connected component of $D_{\theta}$, which is impossible. Since $k_{a}=2$, we have $q\not \in H_{2}^{\eta}$, so $q\in G_{a}\cup L_{a}$. Note that $q\not \in G_3$ because $D_{\theta}$ has no circular subdivisor. The curve $\theta_{+}(H_2)$ is branching in $D_{\theta}$, and meets a maximal twig of $D_{\theta}$ containing $\theta_{+}(H_1)$. The latter twig is of type $[(2)_{5}]$ if $a=1$, or of type $[3,2,2]$ if $a=3$. The twig meeting $\theta_{+}(H_2)$ in the proper transform of $A_{a}$ has length $2$. Hence the connected component of $D_{\theta}$ containing it is a non-admissible fork, a contradiction.
	
	Thus $A_{a}\cap \Bs\eta_{+}^{-1}=\emptyset$, so $(\eta_{+}^{-1})_{*}A_{a}\not\subseteq D$. As in Case \ref{case:non-tangent}, it follows that $\tilde{p}\circ \eta_{+}$ is a $\P^1$-fibration of height $3$ with $\#D\hor=2$. Degenerate fibers of $\tilde{p}$ are supported on 
	$\tilde{F}_{1}\de A_{b}+H_{2}^{\eta}+G_2+A_2=[1\aadec{1},2\aadec{2},2,1\aadec{2}]$;
	$\tilde{F}_{2}\de G_1+L+G_3=[2,1\aadec{2},2\aadec{1}]$, where $L$ is the proper transform of the line joining $p_1$ with $p_3$; and 
	$\tilde{F}_{3}\de L_{3}+H_{1}^{\eta}+A_{0}+L_{1}=[2\aadec{1},2\aadec{2},1,3]$ if $a=1$ or $[2,2\aadec{2},1,3\aadec{1}]$ if $a=3$. As before, decorations $\aadec{1}$, $\aadec{2}$ indicate the components meeting the $1$-section $T$ and the $2$-section $L_2$, respectively. Let $\tilde{\eta}\colon X_{\eta}\to \P^2$ be the contraction of $(\tilde{F}_{1}-H_{2}^{\eta})+(\tilde{F}_{2}-G_{a})+(\tilde{F}_{3}-L_{a})+T$. Replacing $p$ and $\psi$ with  $\tilde{p}\circ \eta_{+}$ and $\tilde{\eta}\circ\eta_{+}$ leads back to Case \ref{case:non-tangent}, which ends the proof.\qedhere
	\end{casesp}
\end{proof}

\subsection{Further restrictions and notation}

We proceed with the proof of Theorem \ref{thm:ht=3} under the assumption $\width(\bar{X})=2$. 
Cases $\cha\kk\neq 2$ and $\cha\kk=2$ are settled in Lemmas  \ref{lem:w=2_cha_neq_2} and \ref{lem:w=2_cha=2}, respectively. We follow the same strategy as in Section \ref{sec:w=3_list}. In particular, we use the following notation, analogous to \ref{not:phi_+_H}.

\begin{notation}[Reconstructing the minimalization $\psi$, cf.\ Notation \ref{not:phi_+_H}]\label{not:phi_+_h=2}
	We keep Notation \ref{not:w=2}. By Lemma \ref{lem:w=2_swaps} we have a factorization  $\psi=\phi\circ\phi_{+}\colon X\to \P^2$, where $\phi$ is as in Example \ref{ex:w=2_cha_neq_2} or \ref{ex:w=2_cha=2}. We study factorizations $\psi=\gamma\circ\gamma_{+}$, where $\gamma$ is a composition of $\phi$ with some further blowups over $p_0,\dots,p_{\nu}$; and $\gamma_{+}\colon (X,D)\sqto (X_{\gamma},D_{\gamma})$ is a vertical swap. One such factorization is given by $\psi=\tau\circ \tau_{+}$, where $\tau\colon X_{\tau}\to \P^2$ is the blowup at \emph{all} points listed there (including the underlined ones), see Figures \ref{fig:w=2_swaps} (case $\cha\kk\neq 2$) and \ref{fig:w=2_cha=2}  (case $\cha\kk=2$).
	
	For a given factorization $\psi=\gamma\circ\gamma_{+}$ as above, we put $H_{j}^{\gamma}=\gamma_{+}(H_{j})$, $j\in \{1,2\}$, and denote by $U_{\gamma}$, $W_{\gamma}$ the connected component of $D_{\gamma}$ containing $H_{1}^{\gamma}$ and $H_{2}^{\gamma}$, respectively. For $i\in \{1,\dots,\nu\}$ we put $L_{i}^{\gamma}=\gamma^{-1}_{*}\ll_{i}$, and for $j\in \{0,\dots,\nu\}$ we denote by $A_{j}^{\gamma}$ the $(-1)$-curve in $\gamma^{-1}(p_{j})$: it is unique by Lemma \ref{lem:w=2_psi}\ref{item:psi_p_i}. For $i\in \{1,\dots,\nu\}$ we denote by $G_{i}^{\gamma}$ the first exceptional curve over $p_i$. We skip the superscript if $\gamma$ is clear from the context. 
	
	We put $\check{D}\de D+\sum_{j}A_{j}$, where the sum runs over all $j$ such that $A_j$ meets $D$ normally. That is, $j\in \{0,\dots,\nu\}$ in all cases except \ref{lem:w=2_swaps}\ref{item:w=2_cha=2_nu=4} with $p_{1}'\not\in \Bs\psi^{-1}$, in which case $j\in \{0,2,3,4\}$ because $A_1$ meets $D$ in a node, see Figure \ref{fig:swap-to-nu=4_small}. We denote by $\check{\cS}$ the combinatorial type of $(X,\check{D})$, see \cite[\sec 2F]{PaPe_ht_2} for definitions and notation.
\end{notation}

As in the case $\width=3$, we recover the vertical swap $\tau_{+}$ by reversing elementary swaps one by one as long as the inequality in Lemma \ref{lem:delPezzo_criterion} holds: this inequality now reads as \eqref{eq:ld-bound_h=2}. The following result, which is an analogue of Lemma \ref{lem:w=3_uniqueness}, shows that this process is essentially combinatorial, with minor exceptions if $\cha\kk=2$. 

\begin{lemma}[Reconstructing the minimalization $\psi$ step by step, cf.\ Lemma \ref{lem:w=3_uniqueness}]\label{lem:w=2_uniqueness}
	Let $\psi=\gamma\circ \gamma_{+}$ be as in Notation \ref{not:phi_+_h=2}. Then  the following hold.
	\begin{enumerate}
		\item\label{item:ld-bound_h=2} The log surface $(X_{\gamma},D_{\gamma})$ is the minimal log resolution of an lt del Pezzo surface of rank one. In particular,
		\begin{equation}\label{eq:ld-bound_h=2}
			\ld(H_{1}^{\gamma})+2\ld(H_{2}^{\gamma})>1.
		\end{equation}
		\item\label{item:w=2_A} We have $\Bs\gamma_{+}^{-1}\subseteq \sum_{i=0}^{\nu}A_{i}^{\gamma}$.
		\item\label{item:w=2_D} Assume that $\gamma_{+}$ has a base point $r\not\in D_{\gamma}$. Then the following hold.
		\begin{enumerate}
			\item\label{item:w=2_moduli} We have $\cha\kk=2$ and $p_{1}'\in \Bs\psi^{-1}$, see Figures \ref{fig:swap-to-nu=4_large}, \ref{fig:swap-to-nu=3}.
			\item\label{item:w=2_moduli_further} For any decomposition 
			$\gamma_{+}\colon 
			\begin{tikzcd}[cramped,column sep=scriptsize]
				(X,D)  \ar[r,"\theta_{+}",squiggly] & (X_{\theta},D_{\theta}) \ar[r,squiggly] & (X_{\gamma},D_{\gamma})
			\end{tikzcd}
			$
			we have $\Bs\theta_{+}^{-1}\subseteq D_{\theta}$.
			\item\label{item:w=2_GK} Assume furthermore that $\bar{X}$ has no descendant with elliptic boundary. Then $\nu=3$.
		\end{enumerate}
		\item\label{item:w=2_uniq} Put $d=\nu-3+\epsilon$, where  $\epsilon=1$ if $\gamma_{+}$ has a base point off $D_{\gamma}$ (see \ref{item:w=2_D} above), and $\epsilon=0$ otherwise. Then the following hold.
		\begin{enumerate}
			\item\label{item:w=2_uniq-hat}  The set $\cP_{+}(\check{\cS})$ has moduli dimension $d$, i.e.\ it is represented by a universal $\Aut(\check{\cS})$-faithful family over a base $B$ of dimension $d$ (see Section \ref{sec:moduli} for definitions). In particular, if $d=0$ then $\#\cP_{+}(\check{\cS})=1$, i.e.\ $\check{\cS}$ uniquely determines the isomorphism class of $(X,\check{D})$. 
			In case $d=1$ we have $B\cong \Astst$.
			\item\label{item:w=2_uniq-hi} Put $h^{i}=h^{i}(\lts{X}{D})$. Then $h^0=0$ and $h^1=d$. Moreover, $h^{2}=0$ if the corresponding surface $\bar{Y}$ from Lemma \ref{lem:w=2_swaps} is as in Example \ref{ex:w=2_cha_neq_2}, $h^2=1$ if it is as in  \ref{ex:w=2_cha=2}\ref{item:swap-to-nu=4_small},\ref{item:swap-to-nu=3}  and $h^2=2$ if it is as in \ref{ex:w=2_cha=2}\ref{item:swap-to-nu=4_large}.
		\end{enumerate}
	\end{enumerate}
\end{lemma}

\begin{remark}[Moduli dimension in Theorem \ref{thm:ht=3} is at most $1$]
	The moduli dimension $d$ in Lemma \ref{lem:w=2_uniqueness}\ref{item:w=2_uniq} is at most $2$. It equals $2$ if and only if $\nu=4$ and $\epsilon=1$, in which case Lemma \ref{lem:w=2_uniqueness}\ref{item:w=2_GK} shows that  $\bar{X}$ has a descendant $(\bar{Y},\bar{T})$ with elliptic boundary. In these cases the $(-1)$-curve $A_1$ is an elliptic tie, cf.\ Figure \ref{fig:swap-to-nu=4_large}, so $\bar{T}$ is cuspidal and $\bar{Y}$ is of type $3\rA_1+\rD_4$ or $2\rA_1+\rD_6$.  
 	And indeed, case $(d,t)=1$ of \cite[Theorem E(e)]{PaPe_ht_2} shows that the class such surfaces $\bar{X}$, of a fixed singularity type, has moduli dimension $2$.
 	
 	We conclude that in the setting of Theorem \ref{thm:ht=3}, if $\#\cP_{+}(\check{\cS})>1$ then $\cP_{+}(\check{\cS})$ has moduli dimension $1$. More precisely, it is represented by an almost universal family over $\Astst$.
\end{remark}

\begin{proof}[Proof of Lemma \ref{lem:w=2_uniqueness}]
	\ref{item:ld-bound_h=2} As in the proof of Lemma \ref{lem:w=3_uniqueness}, the first statement follows from  Lemma \ref{lem:cascades}\ref{item:cascades-still-dP}, and the inequality \eqref{eq:ld-bound_h=2} is a restatement of \eqref{eq:ld_phi_H}. 
	
	\ref{item:w=2_A} This follows from Lemma \ref{lem:w=2_psi}\ref{item:psi_p_i}. 
	
	\ref{item:w=2_D} Let $\eta$ be a composition of $\gamma$ with a blowup at $r$. In case $\cha\kk\neq 2$, looking at Figure \ref{fig:w=2_swaps} we see that $D_{\eta}$ is not a sum of admissible chains and forks, contrary to  \ref{item:ld-bound_h=2}. Assume $\cha\kk=2$. Then a similar analysis of Figure \ref{fig:w=2_cha=2} shows that $p_{1}'\in \Bs\psi^{-1}$, i.e.\ $D_{\tau}$ is as in \ref{ex:w=2_cha=2}\ref{item:swap-to-nu=4_large} or \ref{item:swap-to-nu=3}; and either $r\in A_{1}^{\gamma}$ or $\nu=3$. Repeating this analysis after a blowup at $r$, we infer that each further blowup is centered on the boundary, i.e.\ \ref{item:w=2_moduli_further} holds. Furthermore, in case $\nu=4$ the morphism $\tau_{+}$ is a composition of blowups over $r$ and, possibly, one blowup over $A_j\cap H_2$ for some $j\in \{2,3,4\}$, so the curve $A_1$ is an elliptic tie, see Lemma \ref{lem:tie}, so \ref{item:w=2_GK} holds. We leave the details, analogous to the proof of Lemma \ref{lem:w=3_uniqueness}\ref{item:w=3_no_C1}, to the reader.
	
	\ref{item:w=2_uniq} 
	 Write $\psi=\phi\circ \phi_{+}$, where $\phi_{+}\colon (X,D)\sqto (Y,D_Y)$ is a vertical swap onto a log surface as in Example \ref{ex:w=2_cha_neq_2} or \ref{ex:w=2_cha=2}, see Lemma \ref{lem:w=2_swaps}. Let $\check{D}_{\phi}=\phi_{+}(\check{D})$ and let $\check{\cS}_{\phi}$ be the combinatorial type of $(X_{\phi},\check{D}_{\phi})$. By Proposition \ref{prop:primitive}\ref{item:primitive-uniqueness}  $\cP_{+}(\check{\cS}_{\phi})$ is represented by a universal $\Aut(\check{\cS}_{\phi})$-faithful family of dimension $\nu-3$, where $\Aut(\check{\cS}_{\phi})$ permutes degenerate fibers of the same type. The base of this family is a point if $\nu=3$ and $\Astst$ if $\nu=4$.

	The only automorphisms of the weighted graph of $\check{D}$ are permutations of degenerate fibers, so the morphism  $\phi_+\colon (X,\check{D})\to (X_\phi,\check{D}_\phi)$ is $\Aut(\check{\cS})$-equivariant. If $\epsilon=0$ then $\phi_+$ is inner, so \cite[Lemma 2.18]{PaPe_ht_2} yields \ref{item:w=2_uniq-hat} and the equality $h^i=h^i(\lts{X_{\phi}}{D_{\phi}})$, which together with Proposition \ref{prop:primitive}\ref{item:primitive-hi} implies \ref{item:w=2_uniq-hi}. Assume $\epsilon=1$, i.e.\ we are in the setting of part \ref{item:w=2_D}. Then $\phi_{+}$ is a composition of inner blowups and a unique outer one, centered at a point of $A^{\circ}\de A_{1}^{\gamma}\setminus D_{\gamma}$, see Figures \ref{fig:swap-to-nu=4_large}, \ref{fig:swap-to-nu=3}. Every automorphism of the weighted graph of $\check{D}_{\gamma}$ fixes the three points $A_{1}^{\gamma}\cap D_{\gamma}$, so the group $\Aut(X_{\gamma},\check{D}_{\gamma})$ acts trivially on $A^{\circ}$. Now we infer  \ref{item:w=2_uniq} by applying  \cite[Lemma 2.19(1)]{PaPe_ht_2} to this outer blowup and Lemma 2.18 loc.\ cit.\ to the remaining inner ones. If $d=1$ (so $\nu=3$) we get that $\cP_{+}(\check{\cS})$ is represented by an almost universal family over $A^{\circ}\cong \Astst$.
\end{proof}

Recall that we restrict our attention to surfaces which do not have descendants with elliptic boundary, since the latter are classified in \cite{PaPe_ht_2}. This additional restriction has the following easy consequence. We prove it now in case $\cha\kk\neq 2$ and postpone the other case until \hyperref[U-2-cha=2]{the end of this section}.

\begin{lemma}\label{lem:U-2}
	Let $\bar{X}$ be as in \eqref{eq:assumption_w=2}. Assume furthermore that $\bar{X}$ has no descendant with elliptic boundary. Then the connected component $U$ of $D$ containing $H_1$ is neither a $(-2)$-chain nor a $(-2)$-fork.
\end{lemma}
\begin{proof}[Proof in case $\cha\kk\neq 2$]
	Suppose $U$ consists of $(-2)$-curves. Then for every decomposition $\psi=\gamma\circ \gamma_{+}$ as in Notation \ref{not:phi_+_h=2}, $U_{\gamma}$ consists of $(-2)$-curves and $\Bs\gamma_{+}^{-1}\cap U_{\gamma}=\emptyset$. By Lemma \ref{lem:tie} we need to find an elliptic tie. 
	
	Since $\cha\kk\neq 2$, Lemma \ref{lem:w=2_swaps} shows that $D_{\tau}$ is as in Figure \ref{fig:w=2_swaps}.  Since $U_{\tau}$ consists of $(-2)$-curves, $D_{\tau}$ is as in Figure  \ref{fig:w=2_A1+A2+A5_swap}, \ref{fig:w=2_2A1+2A3_swap_H} or \ref{fig:w=2_2A1+2A3_swap_T}. Consider the first two cases. Then $D_{\tau}$ has two connected components, $U_{\tau}$ and $W_{\tau}$. By Lemma \ref{lem:w=2_uniqueness}\ref{item:w=2_moduli} $\Bs\tau_{+}^{-1}\subseteq D_{\tau}$, so in fact $\Bs\tau_{+}^{-1}\subseteq W_{\tau}$, and by induction we get $\Bs\gamma_{+}^{-1}\subseteq W_{\gamma}$ for any $\gamma$. 
	
	Assume $\Bs\tau_{+}^{-1}\cap A_{0}^{\tau}=\emptyset$. Then every blowup in the decomposition of $\tau_{+}$ is centered over $A_{1}$ or away from $W+A_{1}-H_2$. Hence $W+A_1-H_2$ blows down to $W^{\tau}+A_{1}^{\tau}-H_{2}^{\tau}$. The latter is a chain $[3,1,2]$ or $[1,2,2]$, see Figure  \ref{fig:w=2_A1+A2+A5_swap} or \ref{fig:w=2_2A1+2A3_swap_H}, so blows down to a smooth point. Since $H_{2}^{\tau}$ meets this chain in tips, the contraction of $W+A_1-H_2$ maps $H_2$ to a nodal curve of arithmetic genus one, so $A_{1}$ is an elliptic tie. 
	
	Assume now that $A_{0}^{\tau}$ contains a base point $q$ of $\tau_{+}^{-1}$. Let $\gamma$ be a composition of $\tau$ with a blowup at $q$. In the first case (Figure  \ref{fig:w=2_A1+A2+A5_swap}) $U_{\gamma}$ is not admissible, so the second case holds (Figure \ref{fig:w=2_2A1+2A3_swap_H}). The admissibility of $U$ implies that $\Bs\gamma_{+}\cap W_{\gamma}\subseteq A_{1}^{\gamma}$. As before, we conclude that $W+A_1-L_1$ blows down to $W^{\gamma}+A_{1}^{\gamma}-L_{1}^{\gamma}=[3,1,2]$, hence to a smooth point. This blowdown maps $L_1$ to a nodal curve of arithmetic genus $1$. So $A_1$ an elliptic tie. 
	
	Consider the last case (Figure \ref{fig:w=2_2A1+2A3_swap_T}). Since $U$ is admissible and consists of $(-2)$-curves, we get $\psi=\tau$. We claim that there is  a $(-1)$-curve $A\subseteq X$ such that $A\cdot D=A\cdot G_3=2$: such $A$ will be the required elliptic tie.
	
	For $i,j\in \{1,2,3\}$ let $G_{j}'$ be the second exceptional curve over $p_j$, and let $L_{ij}$ be the proper transform of the line joining $p_i$ with $p_j$. Let $R$ be the proper transform of the conic passing through $p_0,p_2,p_3$ and tangent to $\cc$ at $p_1$. Consider a $\P^1$-fibration of $X$ induced by $|2G_{3}'+2A_3+L_3+H_2|$. The horizontal part of $D$ consists of $1$-sections $L_1$ and $H_1$, and a $2$-section $G_3$. Degenerate fibers are supported on $G_{3}'+A_{3}+L_{3}+H_{2}=\langle 2;[1],[2],[2]\rangle$, on $G_{1}+L_{13}+L_2=[2,1,2]$, on $L_{23}+G_{2}+R=[1,2,1]$ and on $A_0+A$ for some divisor $A$ which has no common components with $D$. By Lemma \ref{lem:fibrations-Sigma-chi}\ref{item:Sigma} there are no more degenerate fibers, and $A_0+A=[1,1]$. Since $A_0$ meets $L_{1}$ and $H_1$ we get $A\cdot D=A\cdot G_3=2$, as needed.
\end{proof}

We now focus on the case $\cha\kk=2$, $\nu=3$. In Lemma \ref{lem:further-swaps} below, we use the assumption \eqref{eq:assumption_w=2} and symmetries of $(X,D)$ to impose some additional restrictions of the map $\phi_{+}^{-1}$ which we need to reconstruct. 

\begin{lemma}[Additional blowup in case \ref{lem:w=2_swaps}\ref{item:w=2_cha=2_nu=3}]\label{lem:further-swaps}
	Assume $\cha\kk=2$ and $\nu=3$, see Figure \ref{fig:swap-to-nu=3}. Then one can choose the morphism $\psi$ as in Notation \ref{not:phi_+_h=2} so that $\phi_{+}^{-1}$ has a base point $q_{j}\in A_{j}^{\phi}$ for $j=1,2$. 
\end{lemma}

The proof of Lemma \ref{lem:further-swaps} is based on the following observation, which will be useful in Section \ref{sec:w=2_cha=2}, too. 

\begin{remark}[Reduction via symmetries of the Fano plane]
	\label{rem:7A1}
	Assume $\cha\kk=2$ and $\nu=3$. 
	Write $\psi=\alpha\circ\alpha_{+}$, where $\alpha\colon X_{\alpha}\to \P^2$ is a blowup at $p_0$ and $p_{j},p_{j}'$ for $j\in \{1,2,3\}$, so $(X_{\phi},D_{\phi})\sqto (X_{\alpha},D_{\alpha})$ is a swap contracting $A_0$. Then $X_{\alpha}$ is the minimal resolution of a del Pezzo surface of rank 1 and type $7\rA_1$. 
	It can be obtained by a blowup $\beta\colon X_{\alpha}\to \P^2$ of all seven $\F_{2}$-rational points on $\P^2$, see \cite[Corollary 5.11]{KN_Pathologies}. 
	
	The exceptional divisor of $\beta$ is a sum of seven disjoint $(-1)$-curves. They are: the $(-1)$-curves $A_{j}^{\alpha}$ over $p_j$, $j\in \{0,1,2,3\}$, and the proper transform of the lines joining $p_{i}$ with $p_{j}$, $i,j\in \{1,2,3\}$.  
	In turn, $D_{\alpha}$ is the proper transform of all $\F_{2}$-rational lines. 
	Put $q_j=\beta(A_{j}^{\alpha})\in \P^2$. Note that $q_1,q_2,q_3$ lie on a line $\beta(H_2)$, and $q_0\not\in \beta(H_2)$. 
	
	Choose $i_0,\bar{i}_0\in \{1,2,3\}$, and write $\{i,j,k\}=\{0,1,2,3\}\setminus \{i_0\}$, $\{\bar{i},\bar{j},\bar{k}\}=\{0,1,2,3\}\setminus \{\bar{i}_0\}$. Since the points $\{q_i,q_j,q_k\}$ and $\{q_{\bar{i}},q_{\bar{j}},q_{\bar{k}}\}$ are not collinear, there is a unique element of $\operatorname{GL}_3(\F_2)$ mapping $(q_i,q_j,q_k)$ to $(q_{\bar{i}},q_{\bar{j}},q_{\bar{k}})$; its lift to $(X_\alpha,D_{\alpha})$ maps $(A_{i}^{\alpha},A_{j}^{\alpha},A_{k}^{\alpha})$ to $(A_{\bar{i}}^{\alpha},A_{\bar{j}}^{\alpha},A_{\bar{k}}^{\alpha})$. If such an automorphism lifts to some intermediate log surface $(X_{\gamma},D_{\gamma})$ introduced in Notation \ref{not:phi_+_h=2}, then it maps $(A_{i}^{\gamma},A_{j}^{\gamma},A_{k}^{\gamma})$ to $(A_{\bar{i}}^{\gamma},A_{\bar{j}}^{\gamma},A_{\bar{k}}^{\gamma})$, and we denote it simply by $(A_i,A_j,A_k)\mapsto (A_{\bar{i}},A_{\bar{j}},A_{\bar{k}})$. 
\end{remark}

\begin{example}[See Figure \ref{fig:reduction}]\label{ex:7A1}
	 Assume that $\gamma$ is a composition of $\phi$ with a blowup at $A_2^{\phi}\cap H_{2}^{\phi}$. Then the log surface $(X_{\gamma},D_{\gamma})$ admits an involution $(A_0,A_1,A_2)\mapsto (A_2,A_1,A_0)$. Looking at its action on the $\F_2$-rational lines, we see that it maps $H_2$, $L_1$, $L_2$, $L_3$, $G_1$, $G_2$, $G_3$, to $L_1$, $H_2$, $L_2$, $G_2$, $G_1$, $L_3$, $G_3$, respectively.	
	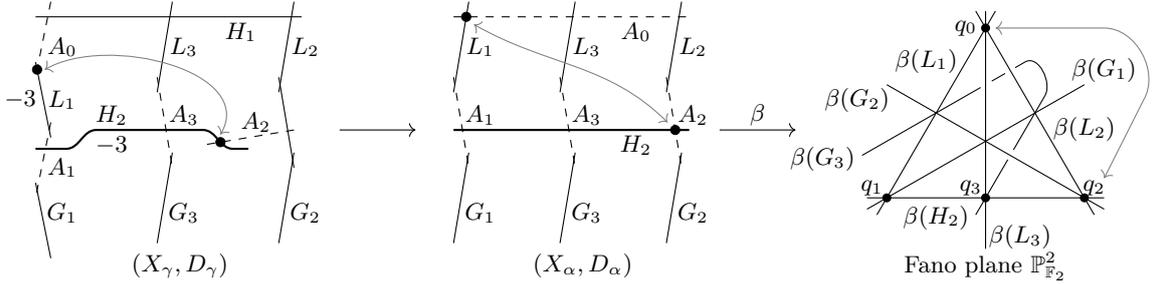
\begin{figure}[htbp]
		\begin{tikzpicture}
			\begin{scope}
				\draw (0,3) -- (3.5,3);
				\node at (2.7,2.8) {\small{$H_1$}};
				\draw[dashed, name path=A0] (0.2,3.2) -- (0,2.2);
				\node at (0.35,2.6) {\small{$A_0$}};
				\draw[name path = L1] (0,2.4) -- (0.2,1.4);
				\node at (0.35,1.9) {\small{$L_1$}};
				\node at (-0.2,1.9) {\small{$-3$}};
				\draw[dashed] (0.2,1.6) -- (0,0.6);
				\node at (0.35,1) {\small{$A_1$}};
				\draw (0,0.8) -- (0.2,-0.2);
				\node at (0.35,0.4) {\small{$G_1$}};
				\draw (1.8,3.2) -- (1.6,2);
				\node at (1.95,2.6) {\small{$L_3$}};
				\draw[dashed] (1.6,2.2) -- (1.8,1);
				\node at (1.95,1.7) {\small{$A_3$}};
				\draw (1.8,1.2) -- (1.6,0);
				\node at (1.95,0.4) {\small{$G_3$}}; 
				\draw (3.4,3.2) -- (3.2,2);
				\node at (3.55,2.6) {\small{$L_2$}};
				\draw (3.2,2.2) -- (3.4,1);
				\draw[dashed, name path=A2] (3.4,1.5) -- (2.2,1.3);
				\node at (2.9,1.6) {\small{$A_2$}};
				\draw (3.4,1.2) -- (3.2,0);
				\node at (3.55,0.4) {\small{$G_2$}};
				\draw[thick, name path=H] (0,1.25) -- (0.4,1.25) to[out=0,in=180] (0.8,1.5) -- (2.2,1.5) to[out=0,in=180] (2.6,1.25)-- (2.8,1.25);
				\node at (1,1.7) {\small{$H_2$}};
				\node at (1,1.3) {\small{$-3$}};
				\path [name intersections={of=A2 and H, by=X}];
				\filldraw (X) circle (0.06);
				\path [name intersections={of=A0 and L1, by=Y}];
				\filldraw (Y) circle (0.06);
				\draw[<->, gray] ($(X)+(0,0.1)$) to[out=60,in=30] ($(Y)+(0.1,0)$);
				\node at (1.9,-0.3) {\small{$(X_{\gamma},D_{\gamma})$}};
				\draw[->] (4,1.5) -- (5,1.5);
			\end{scope}
			\begin{scope}[shift={(5.5,0)}]
				\draw[dashed, name path = A0] (0,3) -- (3.1,3);
				\node at (2.4,2.8) {\small{$A_0$}};
				\draw[name path = L1] (0.2,3.2) -- (0,2);
				\node at (0.35,2.6) {\small{$L_1$}};
				\draw[dashed] (0,2.2) -- (0.2,1);
				\node at (0.35,1.7) {\small{$A_1$}};
				\draw (0.2,1.2) -- (0,0);
				\node at (0.35,0.4) {\small{$G_1$}};
				\draw (1.6,3.2) -- (1.4,2);
				\node at (1.75,2.6) {\small{$L_3$}};
				\draw[dashed] (1.4,2.2) -- (1.6,1);
				\node at (1.75,1.7) {\small{$A_3$}};
				\draw (1.6,1.2) -- (1.4,0);
				\node at (1.75,0.4) {\small{$G_3$}}; 
				\draw (3,3.2) -- (2.8,2);
				\node at (3.15,2.6) {\small{$L_2$}};
				\draw[dashed, name path = A2] (2.8,2.2) -- (3,1);
				\node at (3.15,1.7) {\small{$A_2$}};
				\draw (3,1.2) -- (2.8,0);
				\node at (3.15,0.4) {\small{$G_2$}};
				\draw[thick, name path = H] (0,1.5) -- (3.1,1.5);
				\node at (2.4,1.3) {\small{$H_2$}};
				\path [name intersections={of=A2 and H, by=X}];
				\filldraw (X) circle (0.06);
				\path [name intersections={of=A0 and L1, by=Y}];
				\filldraw (Y) circle (0.06);
				\draw[<->, gray] ($(Y)+(0.1,-0.1)$) to[out=-30,in=135] ($(X)+(-0.1,0.1)$);
				\node at (1.7,-0.3) {\small{$(X_{\alpha},D_{\alpha})$}};
				\draw[->] (3.5,1.5) -- (4.5,1.5);
				\node at (4,1.7) {\small{$\beta$}};
			\end{scope}
			\begin{scope}[shift={(12.5,0.6)}, scale=1.3]
				\draw[add = 0.1 and 0.1] (1,0) to (-1,0);
				\draw[add = 0.1 and 0.1] (1,0) to (0,1.732);
				\draw[add = 0.1 and 0.1] (-1,0) to (0,1.732);
				\draw[add = 0.3 and 0.1] (0,0) to (0,1.732);
				\draw[add = 0.2 and 1] (1,0) to (0,0.577);
				\node at (-1.3,1) {\small{$\beta(G_2)$}};
				\draw[add = 0.2 and 1] (-1,0) to (0,0.577);
				\node at (1.2,1.3) {\small{$\beta(G_1)$}};
				\filldraw (1,0) circle (0.04);
				\node at (1.1,0.1) {\small{$q_2$}};
				\filldraw (0,0) circle (0.04);
				\node at (-0.15,0.1) {\small{$q_3$}};
				\filldraw (-1,0) circle (0.04);
				\node at (-1.15,0.1) {\small{$q_1$}};
				\filldraw (-0,1.732,0) circle (0.04);
				\node at (-0.2,1.732) {\small{$q_0$}};
				\node at (-0.5,-0.2) {\small{$\beta(H_2)$}};
				\node at (1,0.7) {\small{$\beta(L_2)$}};
				\node at (0.35,-0.4) {\small{$\beta(L_3)$}};
				\node at (-0.6,1.4) {\small{$\beta(L_1)$}};
				\node at (0,-0.7) {\small{Fano plane $\P^2_{\F_2}$}};
				\draw[add = 0.2 and -0.55] (0,0) to (0.5,0.866);
				\draw[add = -0.47 and 0] (0,0) to (0.6,1.04);
				\draw[add = 1 and -0.4] (-0.5,0.866) to (0.25,1.299);
				\draw[add = -0.75 and -0.07] (-0.5,0.866) to (0.25,1.299);
				\draw (0.6,1.04) to[out=60,in=30] (0.4,1.386) -- (0.3,1.328);
				\node at (-1.65,0.4) {\small{$\beta(G_3)$}};
				\draw[<->, gray] (1.2,0.2) -- (1.6,1) to[out=60,in=0] (1.2,1.732) -- (0.1,1.732); 
 			\end{scope}
		\end{tikzpicture}\vspace{-1em}
		\caption{Example \ref{ex:7A1}: involution $(A_0,A_1,A_2)\mapsto (A_2,A_1,A_0)$ from Remark \ref{rem:7A1}.}\vspace{-0.5em}
		\label{fig:reduction}
	\end{figure}	
\end{example}

\begin{example}[How we use Remark \ref{rem:7A1}]\label{ex:7A1-application}
	Automorphisms from Remark \ref{rem:7A1} allow us to narrow down the possibilities for the morphism $\gamma_{+}$. To see how this works, let $\gamma$ be as in Example \ref{ex:7A1}. Assume that, for some reason, we know that  $\Bs\gamma_{+}^{-1}\cap A_{3}^{\gamma}=\emptyset$. Since $U_{\gamma}$ is a $(-2)$-chain and, by Lemma \ref{lem:U-2}, $U$ is not, we see that $\gamma_{+}^{-1}$ has a base point on $A_{0}^{\gamma}$ or $A_{2}^{\gamma}$. Thus applying $(A_0,A_1,A_2)\mapsto (A_2,A_1,A_0)$, we can assume that $A_{0}^{\gamma}\cap \Bs\gamma_{+}^{-1}\neq \emptyset$. 
	
	Note that in this reduction we have implicitly used the condition $A_{3}^{\gamma}\cap \Bs\gamma_{+}^{-1}=\emptyset$. Indeed, our automorphism maps $A_{3}^{\gamma}$ away from the base locus of the morphism $\gamma_{+}^{-1}$ which we reconstruct. In fact, it maps $A_{3}^{\gamma}$ to the $(-1)$-curve over the point $\beta(L_1)\cap \beta(G_2)\cap \beta(G_3)$, i.e.\ to the proper transform of the line joining $p_2$ with $p_3$. 
\end{example}

\begin{proof}[Proof of Lemma \ref{lem:further-swaps}]
If $A_i^{\phi}\cap \Bs\phi_{+}^{-1}=\emptyset$ for two indices $i\in \{1,2,3\}$ then the pencil of conics tangent to $\cc$ at both points $p_{i}$ pulls back to a $\P^1$-fibration of height $2$, which is impossible. Thus we can assume that $\phi_{+}^{-1}$ has a base point $r_{i}^{\phi}\in A_i^{\phi}$ for $i=2,3$. Since $D$ has no circular subdivisor, we have $r_i^{\phi}\in (H_2^{\phi}\cup L_i^{\phi})$ for some $i\in \{2,3\}$, and interchanging $\ll_2$ with $\ll_3$, if needed, we can assume $r_2^{\phi}\in (H_2^{\phi}\cup L_2^{\phi})$.

Let $\sigma\colon X_{\phi}\to X_{\alpha}$ be the contraction of $A_{1}$, and let $\alpha_{+}=\sigma\circ \phi_{+}$, so we have a factorization $\psi=\alpha\circ \alpha_{+}$ as in  Remark \ref{rem:7A1}. Put $r_{0}=\sigma(A_1)\in A_{0}^{\alpha}$ and $r_{i}=\sigma(r_{i}^{\phi})\in A_{i}^{\alpha}$ for $i=2,3$, so $r_{0},r_{2},r_{3}\in \Bs\alpha_{+}^{-1}$, and $r_{2}\in H_{2}^{\alpha}$ or $r_{2}\in L_{2}^{\alpha}$. Define $\iota\in \Aut(X_{\alpha},D_{\alpha})$ as $(A_{0},A_{2},A_{3})\mapsto (A_{2},A_{0},A_{1})$ if $r_{2}\in H_{2}$ or $(A_{0},A_{2},A_{3})\mapsto(A_{1},A_{0},A_{2})$ if $r_{2}\in L_{2}$; and let $\tilde{\alpha}_{+}=\iota^{-1}\circ\alpha_{+}$. Then $\iota(r_{2})=r_{0}$, so $\tilde{\phi}_{+}\de \sigma^{-1}\circ \tilde{\alpha}_{+}$ is regular, and $\iota(r_{i})\in A_{1},\iota(r_{j})\in A_{2}$ for $\{i,j\}=\{0,2\}$, so $\tilde{\phi}_{+}^{-1}$ has base points on $A_{1},A_{2}$. Thus replacing $\psi=\phi\circ\phi_{+}$ by $\phi\circ\tilde{\phi}_{+}$ we get the claim.
\end{proof}

\begin{proof}[Proof of Lemma \ref{lem:U-2} in case $\cha\kk=2$]\phantomsection\label{U-2-cha=2}
	Suppose $U$ consists of $(-2)$-curves. To get a contradiction, by Lemma \ref{lem:tie} we need to find an elliptic tie. By Lemma \ref{lem:w=2_swaps} we have $\nu=4$ or $\nu=3$ and $D_{\phi}$ is as in Figure \ref{fig:w=2_cha=2}.  
	
	If $\nu=4$ then since $U$ is admissible and consists of $(-2)$-curves, we see from Figure \ref{fig:w=2_cha=2} that $\Bs\tau_{+}^{-1}\subseteq A_{1}$ and $A_{1}$ is an elliptic tie. Thus we can assume $\nu=3$, so $D_{\phi}$ is as in Figure \ref{fig:swap-to-nu=3}. By Lemma \ref{lem:further-swaps} we can assume that $\phi_{+}^{-1}$ has a base point $q_{i}\in A_{i}^{\phi}$ for $i\in \{1,2\}$. The admissibility of $U$ and the fact that $U$ consists of $(-2)$-curves imply that $q_{1},q_{2}\in D_{\phi}-U_{\phi}$, in particular $q_{2}\in H_{2}^{\phi}\cup G_{2}^{\phi}$. 
	
	Consider the case $q_{2}\in G_{2}^{\phi}$. Since $U$ consists of $(-2)$-curves we have $q_{1}\in L_{1}^{\phi}$, and since $U$ is admissible we conclude that $\phi_{+}$ is a blowup at $q_{1},q_{2}$ and possibly at a point infinitely near to $q_{1}$ on the proper transform of $L_{1}$. Denoting by $A$ the proper transform of the line joining $p_{2}$ with $p_{3}$ we see that $A+L_{1}+G_{2}+G_{3}$ is a connected component of $D+A$, of type $\langle 1;[2],[3],[k]\rangle$ for some $k\in \{4,5\}$, so $A$ is an elliptic tie, as needed.
	
	Consider the case $q_{2}\in H_{2}^{\phi}$. Let $\eta$ be a composition of $\phi$ with a blowup at $q_{2}$. Since $U$ consists of $(-2)$-curves, all base points of $\eta_{+}^{-1}$, including the infinitely near ones, lie on the proper transforms of $L_{1}^{\eta}+H_{2}^{\eta}$ or on the preimage of $A_{1}^{\eta}$. Moreover, the admissibility of $U$ implies that if $L_{1}^{\eta}\cap \Bs\eta_{+}^{-1}\not\subseteq A_{1}^{\eta}$ then $H_{2}^{\eta}\cap \Bs\eta_{+}^{-1}\subseteq A_{1}^{\eta}$. Thus $L_{1}+H_{2}+G_{1}+\phi_{+}^{-1}(A_{1})$ is a connected component of $D+A_{1}$ which blows down to a fork $\langle 1;[2],[3],[k]\rangle$, where the $(-k)$-curve is the proper transform of $H_{2}$ or $L_{1}$. Hence $A_{1}$ is an elliptic tie, as needed.
\end{proof} 

\subsection{The list of singularity types: case $\cha\kk\neq 2$}\label{sec:w=2_cha_neq_2}

We can now state the classification results. Like in Lemma \ref{lem:w=3}, we describe the structure of the chosen witnessing $\P^1$-fibration by writing the singularity type (i.e.\ the weighted graph of $D$) as a sum of admissible chains $[a_1,\dots,a_n]$ and forks $\langle b;T_1,T_2,T_3\rangle$, and decorating it as explained in Section \ref{sec:notation}, that is: the bold numbers correspond to horizontal components, the underlined one to the $2$-section, and the ones decorated by $\dec{j}$ to the components meeting the vertical $(-1)$-curve $A_j$, see Notation \ref{not:phi_+_h=2}. The list of singularity types without decorations is given in Tables \ref{table:ht=3_char=0}--\ref{table:ht=3_char=2}.  We begin with case $\cha\kk\neq 2$, summarized in the second part of Table \ref{table:ht=3_char=0}.

\begin{lemma}[Classification, case $\width=2$, $\cha\kk\neq 2$, see Table \ref{table:ht=3_char=0}]\label{lem:w=2_cha_neq_2}
	Assume $\cha\kk\neq 2$. Let $\cS$ be a singularity type of a log terminal surface. Let $\bar{X}$ be a del Pezzo surface of rank $1$, height $3$, width $2$ and singularity type~$\cS$. Assume $\bar{X}$ has no descendant with elliptic boundary. Then $\cS$ is as below, and the following hold. 
	\begin{parts}
		\item\label{item:w=2-uniqueness} We have $\#\Pht^{\width=2}(\cS)=1$, i.e.\ $\bar{X}$ is unique up to an isomorphism.
		\item\label{item:w=2-classification} The minimal log resolution $(X,D)$ of $\bar{X}$ admits a $\P^1$-fibration $p$ such that $\bar{X}$ swaps vertically to one of the canonical surfaces $\bar{Y}$ from Example \ref{ex:w=2_cha_neq_2}, and the combinatorial type of $(X,D,p)$ is one of the following. 
	\end{parts} 
	\begin{enumerate}[itemsep=0.6em]
		\item\label{item:ht=3_exception_class} $\bar{Y}$ is of type $\rA_{1}+\rA_{7}+[3]$, see Example \ref{ex:w=2_cha_neq_2}\ref{item:ht=3_exception},  case \ref{lem:w=2_swaps}\ref{item:ht=3_exception_swap} holds, and $(X,D,p)$ is  
		\begin{longlist}
			\item\label{item:ht=3_exc_type} $[2,2\dec{3},2,\bs{2}\dec{0},2\dec{2},\ub{3}\dec{1,3},2\dec{2}]+[3\dec{0},2\dec{1},2]$,
		\setcounter{foo}{\value{longlisti}}
		\end{longlist}
		\item \label{item:w=2_A1+A2+A5_class} $\bar{Y}$ is of type $\rA_1+\rA_2+\rA_5$, see Example \ref{ex:w=2_cha_neq_2}\ref{item:w=2_A1+A2+A5}, case \ref{lem:w=2_swaps}\ref{item:w=2_A1+A2+A5_swap} holds, and $(X,D,p)$ is
		\begin{longlist}\setcounter{longlisti}{\value{foo}}	
			\item\label{item:ht=3_b}  $[2,2\dec{2},2,\bs{2}\dec{0},3\dec{3}]+\langle \ub{2}\dec{2},[2]\dec{1},[2,2]\dec{3},[3]\dec{0,1} \rangle$,
		\setcounter{foo}{\value{longlisti}}
		\end{longlist}
		\item\label{item:w=2_2A1+2A3_class} $\bar{Y}$ is of type $2\rA_1+2\rA_3$, see Example \ref{ex:w=2_cha_neq_2}\ref{item:w=2_2A1+2A3}, case \ref{lem:w=2_swaps}\ref{item:w=2_2A1+2A3_swap} holds, and $(X,D,p)$ is one of the following: 
		\begin{longlist}\setcounter{longlisti}{\value{foo}}	
		\item\label{item:ht=3_A22} $[2,3\dec{2},2,\bs{2}\dec{0},2,2\dec{3},2]+[2\dec{2},\ub{4}\dec{1,3},2\dec{0},2\dec{1}]$,
		\item\label{item:ht=3_A2B} $[2\dec{2},3,\bs{2}\dec{0},2,2\dec{3},2]+\ldec{1}[2,2\dec{0},\ub{3}\dec{1,3},3\dec{2},2]$,
		\item\label{item:ht=3_A2} $[3\dec{2},\bs{2}\dec{0},2,2\dec{3},2]+\ldec{1}[(2)_{k-2},2\dec{0},\ub{k}\dec{1,3},2\dec{2},2]$, $k\in \{3,4\}$,
		\item\label{item:ht=3_C2} $[2\dec{1},2\dec{0},\ub{3}\dec{1,3},2\dec{2},2,\bs{2}\dec{0},2,2\dec{3},2]+[3]\dec{2}$,
		\item\label{item:ht=3_A3E_l=3} $[2,k\dec{3},2,\bs{2}\dec{0},2\dec{2}]+\ldec{3}[(2)_{k-2},\ub{3}\dec{1,2},2\dec{0},2]\dec{1}+[2]\dec{2}$, $k\geq 3$,
		\item\label{item:ht=3_A3E_l=4} $[2,k\dec{3},2,\bs{2}\dec{0},2\dec{2}]+\ldec{3}[(2)_{k-2},\ub{4}\dec{1,2},2\dec{0},2,2]\dec{1}+[2]\dec{2}$, $k\in \{3,4\}$,
		\item\label{item:ht=3_A3E_l=5} $[2,3\dec{3},2,\bs{2}\dec{0},2\dec{2}]+\ldec{3}[2,\ub{5}\dec{1,2},2\dec{0},2,2,2]\dec{1}+[2]\dec{2}$,
		\item\label{item:ht=3_BF} $[2\dec{3},3,\bs{2}\dec{0},2\dec{2}]+[2,3\dec{3},\ub{3}\dec{1,2},2\dec{0},2,2\dec{1}]+[2]\dec{2}$,
		\item\label{item:ht=3_BD} $[3\dec{3},\bs{2}\dec{0},2\dec{2}]+[2,2\dec{3},\ub{k}\dec{1,2},2\dec{0},(2)_{k-1}]\dec{1}+[2]\dec{2}$, $k\in \{3,4\}$,
		\item\label{item:ht=3_C1} $[2\dec{1},2,2\dec{0},\ub{3}\dec{1,2},2\dec{3},2,\bs{2}\dec{0},2\dec{2}]+[3]\dec{3}+[2]\dec{2}$,
		\item\label{item:ht=3_C3} $[2\dec{1},2\dec{0},\ub{2}\dec{1,2},k\dec{3},2,\bs{2}\dec{0},2\dec{2}]+[3,(2)_{k-2}]\dec{3}+[2]\dec{2}$, $k\geq 3$,
		\item\label{item:ht=3_A0} $[2\dec{2},\bs{k}\dec{0},2,2\dec{3},2]+\langle 2;\ldec{0}[(2)_{k-2}],[2]\dec{1},[\ub{3}]\dec{1,2,3}\rangle+[2]\dec{2}$, $k\in \{3,4\}$,
		\item\label{item:ht=3_A3F} $\langle 2;[2],[3]\dec{3},[2\dec{2},\bs{2}\dec{0},2]\rangle+[2\dec{3},\ub{4}\dec{1,2},2\dec{0},2\dec{1}]+[2]\dec{2}$,
		\item\label{item:ht=3_A3EF} $\langle k;[2],\ldec{3}T,[2\dec{2},\bs{2}\dec{0},2]\rangle +\ldec{3}T^{*}*[(2)_{k-2},\ub{3}\dec{1,2},2\dec{0},2\dec{1}]+[2]\dec{2}$, $k\geq 3$, $d(T)\leq 3$.
	\end{longlist}
	\end{enumerate}
\end{lemma}
\begin{proof}
	As in the proof of Lemma \ref{lem:w=3}, we first prove that \ref{item:w=2-classification} implies \ref{item:w=2-uniqueness}. Fix a singularity type $\cS$ as in the above list. By Lemma \ref{lem:no-deb} no surface $\bar{X}\in \Pht^{\width=2}(\cS)$ has a descendant with elliptic boundary, so by \ref{item:w=2-classification} its minimal log resolution $(X,D)$ has a $\P^1$-fibration $p$ as above. Since $\cS$ appears exactly once on the above list, it uniquely determines the combinatorial type of $(X,D,p)$. The latter uniquely determines the isomorphism class of $\bar{X}$ by Lemma \ref{lem:w=2_uniqueness}\ref{item:w=2_uniq-hat}, case $\nu=3$, $\epsilon=0$. Thus $\#\Pht^{\width=2}(\cS)\leq 1$. To see that the equality holds we note that a triple $(X,D,p)$ of each combinatorial type listed above exists and satisfies inequality \eqref{eq:ld_phi_H}, so $(X,D)$ is indeed the minimal log resolution of a del Pezzo surface $\bar{X}$ of rank 1. Since $\cS$ does not appear in the classification of del Pezzo surfaces of rank 1 and height at most $2$ in \cite{PaPe_ht_2}, we have $\height(\bar{X})=3$, and since $\cS$ does not appear in Lemma \ref{lem:w=3}, either, we get $\width(\bar{X})=2$, so $\bar{X}\in \Pht^{\width=2}(\cS)$, as needed. 
	\smallskip 
	
	Thus it remains to prove \ref{item:w=2-classification}, i.e.\ to choose $p$ such that one of the above cases \ref{item:ht=3_exception_class}--\ref{item:w=2_2A1+2A3_class} hold. We keep Notation \ref{not:phi_+_h=2}, and consider each case of Lemma \ref{lem:w=2_swaps} separately.
	
	\smallskip
	\noindent\textbf{Case \ref{lem:w=2_swaps}\ref{item:ht=3_exception_swap}, see Figure \ref{fig:ht=3_exception_swap}.} If $\psi=\tau$ then $D$ is as in \ref{item:ht=3_exc_type}. Suppose $\psi\neq \tau$, and let $\eta$ be a composition of $\tau$ with a blowup at some $r\in \Bs\tau_{+}^{-1}$. By Lemma \ref{lem:w=2_uniqueness}\ref{item:w=2_A},\ref{item:w=2_D}, we have $r\in D_{\tau}\cap A_{j}^{\tau}$ for some $j\in \{0,\dots,3\}$. Recall from Notation \ref{not:phi_+_h=2} that $U_{\eta}$ denotes the connected component of $D_{\eta}$ containing $H_{1}^{\eta}$. Like all connected components of $D_{\eta}$, it is an admissible chain or an admissible fork, see Lemma \ref{lem:w=2_uniqueness}\ref{item:ld-bound_h=2}.
	
	Suppose $r\in A_{0}$. We have $r\in H_{1}^{\tau}$, since otherwise $U_{\eta}$ would be a non-admissible fork $\langle 2,[2],[3,2,2],[2,2,2]\rangle$. Now $U_{\eta}=[2,2,2,\bs{3},2,\ub{3},2]$, where, as usual, the bold numbers correspond to the horizontal components of $D_{\eta}$; and the underlined one corresponds to $H_{2}^{\eta}$. We compute $\ld(H_{1}^{\eta})+2\ld(H_{2}^{\eta})=\tfrac{11}{13}<1$, contrary to \eqref{eq:ld-bound_h=2}.
	
	Suppose $r\in A_2$. Like before, the admissibility of $U_{\eta}$ implies that $r\in L_{3}^{\tau}$. Now $U_{\eta}=[2,2,2,\bs{2},3,\ub{3},2,2]$, so $\ld(H_{1}^{\eta})+2\ld(H_{2}^{\eta})=\tfrac{25}{31}<1$, which again contradicts \eqref{eq:ld-bound_h=2}.
	
	Thus $r\in A_1$ or $r\in A_3$. It is sufficient to show that inequality \eqref{eq:ld-bound_h=2} fails in the first case. Indeed, in case $r\in A_3$ the weighted graph of $U_{\eta}$ is obtained from the one in case $r\in A_1$ either by attaching a vertex corresponding to a vertical curve, or by increasing a weight of such a vertex. In any case, by Lemma \ref{lem:Alexeev}, the log discrepancy of $H_{j}^{\eta}$ in case $r\in A_3$ does not exceed the corresponding one in case $r\in A_1$.  Therefore, assume $r\in A_1$. Now $U_{\eta}=[2,2,2,\bs{2},2,\ub{4},2]$ if $r\in H_{2}^{\tau}$ and $U_{\eta}=\langle \ub{3},[2],[2],[2,2,2,\bs{2},2]\rangle$ otherwise. The number $\ld(H_{1}^{\eta})+2\ld(H_{2}^{\eta})$ equals $1$ or $\tfrac{5}{7}<1$, so in both cases we get a contradiction with  \eqref{eq:ld-bound_h=2}, as needed.
	
	\smallskip
	\noindent\textbf{Case \ref{lem:w=2_swaps}\ref{item:w=2_A1+A2+A5_swap} with $q\in H_2$, see Figure \ref{fig:w=2_A1+A2+A5_swap}.} For $j\in \{2,3\}$ write $\{r_j\}=A_{j}^{\tau}\cap H_{2}^{\tau}$. If both $r_2$ and $r_3$ are base points of $\tau_{+}^{-1}$ then $U$ has two branching components, which is impossible. Hence interchanging $\ll_{2}$ with $\ll_{3}$ if necessary, we can assume $r_3\not\in \Bs\tau_{+}^{-1}$.
	
	Suppose $\Bs\tau_{+}^{-1}\subseteq \{r_2\}\cup A_{1}$.  If $r_{2}\not \in \Bs\tau_{+}^{-1}$ then $U$ is a $(-2)$-chain, contrary to Lemma \ref{lem:U-2}. If $r_2\in \Bs\tau_{+}^{-1}$ then $U$ is a fork, and the twig of $U$ containing $H_1$ has length $5$, so its remaining twigs are of type $[2]$. It follows that $\tau_{+}^{-1}$ has no base points infinitely near to $r_2$, so $U$ is a $(-2)$-fork, again contrary to Lemma \ref{lem:U-2}. 
	
	Thus $\tau_{+}^{-1}$ has a base point $r\neq r_2$, $r\not\in A_{1}$. Let $\eta$ be a composition of $\tau$ with a blowup at $r$. Suppose $r\in A_{3}$. Then by assumption $r\neq r_3$, so $r\in U_{\tau}$ and  $D_{\eta}=[2,3,2,\bs{2},2,2,2]+\langle \ub{3};[3],[2],[2]\rangle$, hence $\ld(H_{1}^{\eta})+2\ld(H_{2}^{\eta})=1$; a contradiction with  \eqref{eq:ld-bound_h=2}.  If $r\in A_2$ then by assumption $r\neq r_2$, so $r_2\not\in \Bs\tau_{+}^{-1}$, and interchanging  $\ll_2$ with $\ll_3$ we get back the case $r\in A_3$, excluded above. Thus $r\in A_{0}$. If $r\in L_{1}$ then $\beta_{U_{\eta}}(H_{1})=3$, and $H_1$ meets a twig $T_j\subseteq F_j$ with $\#T_j=3$ for $j\in \{2,3\}$, contrary to Lemma \ref{lem:admissible_forks}. Therefore, $\{r\}=H_{1}\cap A_{0}$, so $D_{\eta}=[(2)_{3},\bs{3},(2)_{3}]+[2,3,\ub{3},2]$, and again $\ld(H_{1}^{\eta})+2\ld(H_{2}^{\eta})=1$; a contradiction with \eqref{eq:ld-bound_h=2}.
	
	\smallskip
	\noindent\textbf{Case \ref{lem:w=2_swaps}\ref{item:w=2_A1+A2+A5_swap} with $q\in L_3$, see Figure \ref{fig:w=2_A1+A2+A5_swap_b}.} If $\psi=\tau$ then $D$ is as in \ref{item:ht=3_b}. Suppose $\psi\neq \tau$, and let $\eta$ be a composition of $\tau$ with a blowup at some $r\in \Bs\tau_{+}^{-1}$. Since $W_{\eta}$ is an admissible fork, we have $r\not\in A_1$, and $r\in W_{\tau}$ or $\{r\}=A_0\cap H_1$. In the latter case, $D_{\eta}=[2,2,2,\bs{3},3]+\langle \ub{2},[2],[2,2],[2,3]\rangle$, so $\ld(H_1^{\eta})+2\ld(H_2^{\eta})=\frac{137}{299}<1$, a contradiction with \eqref{eq:ld-bound_h=2}. Thus $r\in W_{\tau}\setminus A_1$. Now $D_{\eta}$ is $\langle \bs{2};[2],[2,2,2],[3]\rangle+\langle \ub{2},[2],[4],[2,2]\rangle$ if $r\in A_0$; $\langle 2;[2],[2],[3,\bs{2},2]\rangle+\langle \ub{3};[2],[3],[2,2]\rangle$ if $r\in A_2$, and $[2,2,2,\bs{2},3,2]+\langle \ub{2};[2],[3],[2,3]\rangle$ if $r\in A_3$. Thus $\ld(H_{1}^{\eta})+2\ld(H_{2}^{\eta})=\tfrac{17}{35}$, $\tfrac{13}{18}$ or $\tfrac{241}{391}$, respectively, so in each case we get a contradiction with inequality  \eqref{eq:ld-bound_h=2}.

	\smallskip
	\noindent\textbf{Case \ref{lem:w=2_swaps}\ref{item:w=2_2A1+2A3_swap} with $q\in H_2$, see Figure  \ref{fig:w=2_2A1+2A3_swap_H}.} By Lemma \ref{lem:U-2} $U$ is not a sum of $(-2)$-curves, so $\tau_{+}^{-1}$ has a base point $r\not\in A_{1}^{\tau}$. By Lemma \ref{lem:w=2_uniqueness}\ref{item:w=2_A},\ref{item:w=2_D} $r\in A_{j}^{\tau}\cap D_{\tau}$ for some $j\in \{0,2,3\}$. Let $\eta$ be a composition of $\tau$ with a blowup at some $r$; and for a base point $s$ of $\eta_{+}^{-1}$ let $\upsilon$ be a composition of $\eta$ with a blowup at $s$.

\begin{casesp}
	\litem{$r\in A_{0}^{\tau}$} Suppose $\{r\}=A_0\cap L_1$. Then $U_{\eta}$ is a $(-2)$-fork $\rD_{6}$. Since $U$ is not a $(-2)$-fork by Lemma \ref{lem:U-2}, the map $\eta_{+}^{-1}$ has a base point $s\not\in A_1^{\eta}$. After a blowup at $s$, the fork $U_{\upsilon}$ must still be admissible, so $s\in A_0$ or $\{s\}=A_j\cap U_{\eta}$ for some $j\in \{2,3\}$. If $\{s\}=A_0\cap L_1$ then $U_{\upsilon}$ is still a $(-2)$-fork, now of type $\rE_{7}$, and after one further blowup it becomes non-admissible, a contradiction. Thus $s\in U_{\eta}$. If $s\in A_{3}$ then $D_{\upsilon}=\langle \bs{2};[2],[2],[2,3,2]\rangle+[2,\ub{3},3,2]+[2]$, so $\ld(H_{1}^{\eta})+2\ld(H_{2}^{\eta})=1$, contrary to \eqref{eq:ld-bound_h=2}. Hence $s\in A_0\cup A_2$; so $U_{\upsilon}=\langle \bs{2},[2],[3],[2,2,2]\rangle$,  $\ld(H_{1}^{\upsilon})=\tfrac{1}{5}$ and therefore $\ld(H_{2}^{\upsilon})>\tfrac{2}{5}$ by \eqref{eq:ld-bound_h=2}. But $W_{\upsilon}=\langle 3;[2],[2],[\ub{3}]\rangle$ if $s\in A_0$ and $[2,3,\ub{3},2,2]$ if $s\in A_2$, so $\ld(H_{2}^{\upsilon})=\tfrac{2}{5}$ or $\tfrac{8}{29}<\tfrac{2}{5}$; a contradiction.
	
	Therefore, $\{r\}=A_0\cap H_1$. Suppose $\eta_{+}^{-1}$ has a base point $s\not \in A_{0}^{\tau}$. Consider the case $s\in A_1$. Then $U_{\upsilon}=[2,\bs{3},2,2,2]$, so $\ld(H_{1}^{\upsilon})=\frac{3}{7}$ and therefore $\ld(H_{2}^{\upsilon})>\frac{2}{7}$ by \eqref{eq:ld-bound_h=2}. But $W_{\upsilon}=\langle 2,[2],[2,2],[\ub{4}]\rangle$ if $s\in H_2$ or $\langle 2;[2],[3],[2,\ub{3}]\rangle$ otherwise, so $\ld(H_{2}^{\upsilon})=\frac{2}{7}$ or $\frac{5}{23}<\frac{2}{7}$, a contradiction. If $s\in A_{2}$ then $D_{\upsilon}$ equals $
	 [(2)_{3},\bs{3},(2)_{3}]+\langle 2;[2],[2],[\ub{4}]\rangle$ if $s\in H_2$,
	$[3,\bs{3},2,2,2]+\langle 2;[2],[2],[2,2,\ub{3}]\rangle$ if $s\in L_2$,
	and $\langle 2;[2],[2],[2,2,2,\bs{3},2,2,\ub{3}]\rangle+[3]$ otherwise, 
	so $\ld(H_{1}^{\upsilon})+2\ld(H_{2}^{\upsilon})=
	1$, 
	$\tfrac{37}{46}$ or 
	$\tfrac{1}{4}$; a contradiction with \eqref{eq:ld-bound_h=2}. Eventually, in case $s\in A_3$ we have 
	$D_{\upsilon}=\langle 2,[2],[2],[2,\bs{3},2]\rangle+\langle 2;[2],[2],[\ub{4}]\rangle+[2]$ if $s\in H_{2}$ and 
	 $D_{\upsilon}=[2,\bs{3},2,3,2]+\langle 2;[2],[2],[2,\ub{3}]\rangle+[2]$ otherwise,
	 so in either case $\ld(H_{1}^{\upsilon})+2\ld(H_{2}^{\upsilon})=1$; again a contradiction with \eqref{eq:ld-bound_h=2}.
	 
	 We conclude that $\Bs\eta_{+}^{-1}\subseteq A_0$. Let $\theta$ be a composition of $\tau$ with some $k-2\geq 1$ blowups at $r$ and its infinitely near points on the proper transforms of $H_1$, so that $A_{0}^{\theta}\cap H_{1}^{\theta}$ is not a base point of $\theta_{+}^{-1}$, i.e.\  $\Bs\theta_{+}^{-1}\subseteq A_0^{\theta}\cap W_{\theta}$. We have $U_{\theta}=[2,\bs{k},2,2,2]$ and $W_{\theta}=\langle 2;[2],[(2)_{k-2}],[\ub{3}]\rangle$. If $\psi=\theta$ then inequality \eqref{eq:ld-bound_h=2} gives $k\leq 4$, so $D$ is as in \ref{item:ht=3_A0}. Suppose $\psi\neq \theta$. Then by assumption $A_0^{\theta}\cap W_{\theta}$ is a base point of $\theta_{+}^{-1}$. Let $\gamma$ be a composition of $\theta$ with a blowup at that point. Then $D_{\gamma}=\langle \bs{k};[2],[2],[2,2,2]\rangle+\langle 2,[2],[3,(2)_{k-3}],[\ub{3}]\rangle+[2]$, a contradiction with 
	 \eqref{eq:ld-bound_h=2}.  
	
	\litem{$r\in A_{2}^{\tau}$ and $A_{0}^{\tau}\cap \Bs\tau_{+}^{-1}=\emptyset$} We have three cases: $r\in H_2$, $r\in L_2$ and $r\in G_2$.
	
	\begin{casesp}
	\litem{$\{r\}=A_{2}^{\tau}\cap H_2$} We have $U_{\eta}=[(2)_{3},\bs{2},(2)_{3}]$, so by Lemma \ref{lem:U-2}, $\eta_{+}^{-1}$ has a base point $s\not\in A_{1}^{\eta}$. Assumption  $A_{0}^{\tau}\cap \Bs\tau_{+}^{-1}=\emptyset$ implies that $s\in A_{2}\cup A_{3}$. Interchanging $\ll_2$ with $\ll_3$ if necessary, we can assume $s\in A_{2}$. Suppose $\{s\}=A_2\cap H_{2}$. Then after a blowup at $s$, $U_{\upsilon}$ is a $(-2)$-fork, so by Lemma \ref{lem:U-2}, $\upsilon_{+}^{-1}$ has a base point $z\in A_2\cup A_3$. Let $\theta$ be a composition of $\upsilon$ with a blowup at $z$. Since $U_{\theta}$ is an admissible fork, we have $\{z\}=A_3\cap U_{\upsilon}$. Now $D_{\theta}=\langle 2;[2],[2],[2,3,2,\bs{2},2]\rangle+[2,2,\ub{5},2]$, so $\ld(H_{1}^{\theta})+2\ld(H_{2}^{\theta})=\tfrac{53}{69}$, contrary to \eqref{eq:ld-bound_h=2}. 
	
	Thus $\{s\}=A_{2}\cap U_{\eta}$. If $\psi=\upsilon$ then $D$ is as in \ref{item:ht=3_A22}. Assume that $\psi\neq \upsilon$, and let $\theta$ be a composition of $\upsilon$ with a blowup at $z\in \Bs\upsilon_{+}^{-1}$. Suppose $z\in A_1$. Then $U_{\theta}=[2,3,2,\bs{2},2,2,2]$, so $\ld(H_{1}^{\theta})=\tfrac{3}{5}$ and thus  $\ld(H_{2}^{\theta})>\tfrac{1}{5}$ by \eqref{eq:ld-bound_h=2}. But $W_{\theta}=[2,2,2,\ub{5},2]$ if $z\in H_2$ and $\langle \ub{4};[2],[2],[3,2]\rangle$ otherwise, so $\ld(H_{2}^{\theta})=\tfrac{1}{5}$ or $\tfrac{1}{12}$; a contradiction. 
	
	Thus $z\in A_2\cup A_3$. If $z\in A_{2}$ then $D_{\theta}=
	[2,4,2,\bs{2},2,2,2]+[2,2,\ub{4},2,2]$ if $z\in U_{\upsilon}$ and 
	$\langle 3;[2],[2],[2,2,2,\bs{2},2]\rangle+[3,\ub{4},2,2]$ otherwise, so $\ld(H_{1}^{\theta})+2\ld(H_{2}^{\theta})=1$ or $\tfrac{55}{63}$; contrary to \eqref{eq:ld-bound_h=2}. Hence $z\in A_{3}$. If $z\in H_{2}$ then interchanging $\ll_2$ with $\ll_3$ we get back the case $\{s\}=A_2\cap H_2$, considered above, so we can assume $\{z\}=A_3\cap U_{\upsilon}$. Now $D_{\theta}=
	[2,3,2,\bs{2},2,3,2]+\langle \ub{4};[2],[2],[2,2]\rangle$, so $\ld(H_{1}^{\theta})+2\ld(H_{2}^{\theta})=\tfrac{13}{21}$; a contradiction with \eqref{eq:ld-bound_h=2}.

	\litem{$\{r\}=A_2^{\tau}\cap L_2$} Assume first that $\eta_{+}^{-1}$ has a base point $s\not\in A_{1}^{\eta}$. Suppose $s\in A_3$. Then $D_{\upsilon}=\langle 2;[2],[2],[3,\bs{2},2]\rangle+[2,2,\ub{4},2,2]$ if $s\in H_2$ and $[3,\bs{2},2,3,2]+\langle \ub{3};[2],[2,2],[2,2]\rangle$ otherwise, so $\ld(H_{1}^{\upsilon})+2\ld(H_{2}^{\upsilon})=1$ or $\tfrac{127}{175}$; a contradiction with \eqref{eq:ld-bound_h=2}. Thus $s\in A_2$, $\Bs\eta_{+}^{-1}\cap A_3^{\eta}=\emptyset$. If $s\in L_2$ then  $D_{\upsilon}=[4,\bs{2},2,2,2]+\langle 2;[2],[2],[2,2,\ub{3}]\rangle$,  so $\ld(H_{1}^{\upsilon})+2\ld(H_{2}^{\upsilon})=1$, contrary to \eqref{eq:ld-bound_h=2}. Hence $s\in W_{\eta}$, and $D_{\upsilon}=[2,3,\bs{2},2,2,2]+[2,2,\ub{3},3,2]$. If $\psi=\upsilon$ then $D$ is as in \ref{item:ht=3_A2B}. Assume $\psi\neq \upsilon$, and let $\theta$ be a composition of $\upsilon$ with a blowup at some $z\in \Bs\upsilon_{+}^{-1}$. If $z\in A_1$ then the admissibility of $W_{\theta}$ implies that $z\in H_2$, so $D_{\theta}=[2,3,\bs{2},2,2,2]+[2,2,2,\ub{4},3,2]$ and $\ld(H_{1}^{\theta})+2\ld(H_{2}^{\theta})=\tfrac{273}{323}$, contrary to \eqref{eq:ld-bound_h=2}. Thus $z\in A_2$, so $D_{\theta}=[3,3,\bs{2},2,2,2]+\langle 3;[2],[2],[2,2,\ub{3}]\rangle$ if $z\in U_{\upsilon}$ and 
	$[2,2,3,\bs{2},2,2,2]+[2,2,\ub{3},4,2]$ otherwise. We get  $\ld(H_{1}^{\eta})+2\ld(H_{2}^{\eta})=\tfrac{61}{77}$ or $\tfrac{933}{989}$, contrary to \eqref{eq:ld-bound_h=2}.
	
	It remains to consider the case $\Bs\eta_{+}^{-1}\subseteq A_{1}$. Then $U=[3,\bs{2},2,2,2]$, so $\ld(H_1)=\frac{7}{11}$, and therefore $\ld(H_2)>\frac{2}{11}$ by \eqref{eq:ld-bound_h=2}. Let $\theta$ be a composition of $\eta$ with some $k-3\geq 0$ blowups at $A_{1}^{\eta}\cap H_{2}^{\eta}$ and its infinitely near points on the proper transforms of $H_{2}^{\eta}$, so that $A_{1}^{\theta}\cap H_{2}^{\theta}\not\subseteq \Bs\theta_{+}^{-1}$. Then $W_{\theta}=[(2)_{k-1},\ub{k},2,2]$, and the inequality $\ld(H_2)>\frac{2}{11}$  gives $k\leq 4$. If $\psi=\theta$ then $D$ is as in \ref{item:ht=3_A2}, otherwise after one further blowup we get $W=\langle \bs{k};[2],[2,2],[3,(2)_{k-2}]\rangle$, so $k=3$ since $W$ is admissible, and $\ld(H_{2})=\frac{1}{37}<\frac{2}{11}$; a contradiction.
	
	\litem{$\{r\}=A_2^{\tau}\cap G_2$} We have $U_{\eta}=[2,2,2,\bs{2},2,2,\ub{3},2,2]$. If $\psi=\eta$ then $D$ is as in \ref{item:ht=3_C2}. Suppose $\psi\neq \eta$. Since $A_{0}\cap \Bs\eta_{+}^{-1}=\emptyset$ and  $U_{\upsilon}$ is admissible, we get 
	$\{s\}=
	 A_{1}\cap H_{2}$, 
	$A_{2}\cap U_{\eta}$ or 
	$A_{3}\cap H_2$. Now $U_{\upsilon}=
	 [2,2,2,\bs{2},2,2,\ub{4},2,2,2]$, 
	$[2,2,2,\bs{2},2,3,\ub{3},2,2]$ or 
	$\langle 2;[2],[2],[2,2,\ub{4},2,2,\bs{2},2]\rangle$, so $\ld(H_{1}^{\upsilon})+2\ld(H_{2}^{\upsilon})=
	\tfrac{57}{67}$, $\tfrac{65}{73}$ or $\tfrac{3}{7}<1$; a contradiction with inequality \eqref{eq:ld-bound_h=2}. 
	\end{casesp}
	
	\litem{$r\in A_{3}^{\tau}$ and $(A_{0}^{\tau}\cup A_{2}^{\tau})\cap \Bs\tau_{+}^{-1}=\emptyset$} Consider the case  $r\in H_2$. Then $U_{\eta}$ is a $(-2)$-fork. By Lemma \ref{lem:U-2}, $\eta_{+}^{-1}$ has a base point $s\not \in A_{1}$, so $s\in A_3$ by assumption. 
	If $s\in H_2$ then $U_{\upsilon}$ is still a $(-2)$-fork, so as before, $A_{3}^{\upsilon}$ contains a base point of $\upsilon_{+}^{-1}$, and after blowing up at this point we get a non-admissible $U$, which is impossible. Thus $\{s\}=A_{3}\cap U_{\eta}$, so $U_{\upsilon}=\langle 2;[2],[3],[2,\bs{2},2]\rangle$. In particular, $\ld(H_{1}^{\upsilon})=\frac{3}{5}$, so $\ld(H_{2})>\frac{1}{5}$ by \eqref{eq:ld-bound_h=2}. Like before, the admissibility of $U$ yields $\Bs\upsilon_{+}^{-1}\subseteq A_{1}$. If $\psi=\upsilon$ then $D$ is as in \ref{item:ht=3_A3F}, otherwise after one further blowup $W$ becomes $[2,\ub{5},2,2,2]$ or $\langle \ub{4};[2],[2],[3,2]\rangle$, so $\ld(H_2)=\frac{1}{5}$ or $\frac{1}{12}<\frac{1}{5}$, a contradiction.
	
	Consider the case $r\in U_{\tau}$. Replace $\eta$ with its composition with some $l-3\geq 0$ blowups over $A_1\cap H_2$, on the proper transforms of $H_2$, so that $A_{1}^{\eta}\cap H_{2}^{\eta}\not\subseteq \Bs\eta_{+}^{-1}$. Then $D_{\eta}=[2,3,2,\bs{2},2]+[2,\ub{l},(2)_{l-1}]+[2]$. If $A_1$ contains a base point of $\eta_{+}^{-1}$, say $s$, then $\{s\}=A_1\cap (D_{\tau}-H_2)$, so $D_{\upsilon}=[2,3,2,\bs{2},2]+\langle \ub{l};[2],[2],[3,(2)_{l-2}]\rangle+[2]$, contrary to \eqref{eq:ld-bound_h=2}. Thus $\Bs\eta_{+}^{-1}\subseteq A_{3}$. Again, replace $\eta_{+}$ with $k-2\geq 1$ blowups over $r$, on the proper transforms of $U_{\tau}$, so $D_{\eta}=[2,k,2,\bs{2},2]+[(2)_{k-2},\ub{l},(2)_{l-1}]+[2]$ for some $k,l\geq 3$. If $\psi=\eta$ then inequality \eqref{eq:ld-bound_h=2} gives $l=3$ or $l=4$, $k\in \{3,4\}$ or $(k,l)=(3,5)$, so $D$ is as in \ref{item:ht=3_A3E_l=3}--\ref{item:ht=3_A3E_l=5}. Assume $\psi\neq \eta$. Then the base point of $\eta_{+}^{-1}$ is $\{s\}=A_3\cap W_{\eta}$, and after blowing up at $s$ we get  $D_{\upsilon}=\langle k;[2],[2],[2,\bs{2},2]\rangle +[3,(2)_{k-3},\ub{l},(2)_{l-1}]+[2]$. Inequality \eqref{eq:ld-bound_h=2} gives $l=3$. By assumption, we have $\Bs\upsilon_{+}^{-1}\subseteq A_3$, so the admissibility of $U$ implies that $D$ is as in \ref{item:ht=3_A3EF}.
	\end{casesp}

	\noindent\textbf{Case \ref{lem:w=2_swaps}\ref{item:w=2_2A1+2A3_swap} with $q\in L_3$, see Figure  \ref{fig:w=2_2A1+2A3_swap_L}.} Assume first that $\tau_{+}^{-1}$ has a base point $r\in A_{2}$. The admissibility of $W$ implies that $r\in H_2$. Now interchanging $\ll_2$ with $\ll_3$ we are back in case \ref{lem:w=2_swaps}\ref{item:w=2_2A1+2A3_swap} with $q\in H_2$, considered above. Thus we can assume $A_{2}^{\tau}\cap \Bs\tau_{+}^{-1}=\emptyset$.
	
	Assume $A_{1}\cap \Bs\tau_{+}^{-1}=\emptyset$. Using a similar argument as in the proof of Lemma \ref{lem:w=2_swaps}, we will now find another witnessing $\P^1$-fibration which is again as in the previous case $q\in H_2$. 
	
	Recall that we assume $(A_{1}\cup A_{2})\cap \Bs\tau_{+}^{-1}=\emptyset$, so $A_1^{\tau},A_2^{\tau}\not\subseteq \tau_{+}(D)$. For $j\in \{1,2,3\}$ let $G_j$, $G_{j}'$ be the first and second exceptional curve over $p_j$. Let $\ll_{12}$, $\ll_{1}'$ be the line joining $p_1$ with $p_2$ and the line tangent to $\cc$ at $p_1$; and let $L_{12}=\tau^{-1}_{*}\ll_{12}$, $L_{1}'=\tau^{-1}_{*}\ll_{1}'$. The pencil of conics tangent to $\cc$ at $p_1$ and $p_2$ induces a $\P^{1}$-fibration $\tilde{p}\colon X_{\tau}\to \P^1$ such that $\tau_{+}(D)\hor$ consists of a $2$-section $L_3$ and $1$-section $L_1$, which are disjoint. Its degenerate fibers are supported on  $\tilde{F}_1\de A_0+H_1+L_2+L_{1}'=[1\aadec{1},2\aadec{2},2,1\aadec{2}]$, $\tilde{F}_2\de G_{1}+L_{12}+G_2=[2\aadec{1},1\aadec{2},2]$ and $\tilde{F}_3\de G_{3}'+G_{3}+H_{2}+A_{3}=\langle 2;[2],[2\aadec{1}],[1\aadec{2}]\rangle$, where the numbers decorated by $\aadec{1}$ and $\aadec{2}$ refer to the components meting the $1$-section $L_1$ and the $2$-section $L_3$, respectively. Let $\tilde{\tau}\colon X_{\tau}\to \P^2$ be the contraction of $(\tilde{F}_1-H_1)+(\tilde{F}_2-G_1)+(\tilde{F}_3-H_2)+L_1$. Replacing $p$ and $\psi$ with $\tilde{p}\circ\tau_{+}$ and $\tilde{\tau}\circ\tau_{+}$ we are back in case $q\in H_2$.
	
	Thus we can assume that $\tau_{+}^{-1}$ has a base point $r\in A_1$. As before, denote by $\eta$ the composition of $\tau$ with a blowup at $r$, and for a given base point $s$ of $\eta_{+}^{-1}$, denote by $\upsilon$ the composition of $\eta$ with a blowup at $s$.
	
	We have $\{r\}=A_1\cap H_2$, since otherwise $D_{\eta}=[3,\bs{2},2]+\langle \ub{2};[2],[2,2],[3,2]\rangle+[2]$ and $\ld(H_{1}^{\eta})+2\ld(H_{2}^{\eta})=1$, contrary to \eqref{eq:ld-bound_h=2}. Suppose $\eta_{+}^{-1}$ has a base point $s\in A_0$. Since $W_{\upsilon}$ is admissible, we get $s\in L_1$ and  $D_{\upsilon}=\langle \bs{2};[2],[2],[3]\rangle+[2,2,\ub{3},3,2,2]$, so again $\ld(H_{1}^{\upsilon})+2\ld(H_{2}^{\upsilon})=1$, contrary to \eqref{eq:ld-bound_h=2}. Therefore, $\Bs\eta_{+}^{-1}\subseteq A_{1}\cup A_{3}$.
	
	Assume that $\eta_{+}^{-1}$ has a base point $s\in A_{3}$. If $s\in L_{3}$ then $D_{\upsilon}=[4,\bs{2},2]+\langle 2;[2],[2],[2,2,2,\ub{3}]\rangle+[2]$, so $\ld(H_{1}^{\upsilon})+2\ld(H_{2}^{\upsilon})=1$, contrary to \eqref{eq:ld-bound_h=2}. Thus $\{s\}=A_{3}\cap W_{\eta}$. If $\psi=\upsilon$ then $D$ is as in \ref{item:ht=3_BF}. Suppose $\psi\neq \upsilon$, and let $\theta$ be a composition of $\upsilon$ with a blowup at some $z\in \Bs\upsilon_{+}^{-1}\subseteq  A_1\cup A_3$. If $z\in A_1$ then $z\in H_2$ since $W_{\theta}$ is admissible, so $D_{\theta}=[2,3,\bs{2},2]+[2,3,\ub{4},2,2,2,2]+[2]$ and $\ld(H_{1}^{\theta})+2\ld(H_{2}^{\theta})=\tfrac{71}{77}$, contrary to \eqref{eq:ld-bound_h=2}. Thus $z\in A_3$. We get $D_{\theta}=[3,3,\bs{2},2]+\langle 3;[2],[2],[2,2,2,\ub{3}]\rangle+[2]$ if $z\in U_{\upsilon}$ and $[2,2,3,\bs{2},2]+[2,4,\ub{3},2,2,2]+[2]$ if $z\in W_{\upsilon}$, so $\ld(H_{1}^{\theta})+2\ld(H_{2}^{\theta})=\tfrac{53}{63}$ or $1$; contrary to \eqref{eq:ld-bound_h=2}.
	
	It remains to consider the case $\Bs\eta_{+}^{-1}\subseteq A_{1}$. Replace $\eta$ with $k-2\geq 1$ blowups over $r$, on the proper transforms of $H_2$, so that $A_{1}\cap H_{2}\not\subseteq \Bs\eta_{+}^{-1}$. Then $D_{\eta}=[3,\bs{2},2]+[2,2,\ub{k},(2)_{k}]+[2]$. If $\psi\neq \eta$ then by construction $\{s\}=A_1\cap (D_{\eta}-H_{2})$ is a base point of $\eta_{+}^{-1}$, and the resulting fork $W_{\upsilon}$ is not admissible, which is impossible. Thus $\psi=\eta$. Inequality \eqref{eq:ld-bound_h=2} gives $k\leq 4$, so $D$ is as in \ref{item:ht=3_BD}.

\smallskip
\noindent\textbf{Case \ref{lem:w=2_swaps}\ref{item:w=2_2A1+2A3_swap} with $q\in G_3$, see Figure  \ref{fig:w=2_2A1+2A3_swap_T}.} Since $U_{\tau}$ is a $(-2)$-chain, by Lemma \ref{lem:U-2} we have $\psi\neq \tau$. Let $\eta$ be a composition of $\tau$ with a blowup at $r\in \Bs\tau_{+}^{-1}$. 

We can assume that $A_{2}^{\tau}\cap \Bs\tau_{+}^{-1}=\emptyset$. Indeed, if $r\in A_{2}^{\tau}$ then since $D$ has no circular subdivisor, we have $r\in H_2$ or $r\in L_2$, so interchanging $\ll_2$ with $\ll_3$ we are back in the case $q\in H_2$ or $q\in L_3$ considered above.

Suppose $r\in A_{0}$. Then $U_{\eta}=\langle \bs{2};[2],[2],[2,3,\ub{2},2,2] \rangle$ if $r\in L_{1}$ and $\langle 2;[2],[2],[2,\bs{3},2,2,\ub{2}]\rangle$ if $r\in H_1$. In both cases, we have  $\ld(H_{1}^{\eta})+2\ld(H_{2}^{\eta})=1$; a contradiction with \eqref{eq:ld-bound_h=2}. Thus $\Bs\tau_{+}^{-1}\subseteq A_1\cup A_3$.

Assume $r\in A_1$, so $r\in  H_2$ as $U$ is admissible. If $\psi=\eta$ then $D$ is as in \ref{item:ht=3_C1}. Suppose $\psi\neq \eta$, and let $\upsilon$ be a composition of $\eta$ with a blowup at $s\in \Bs\eta_{+}^{-1}\subseteq A_{1}\cup A_3$. Since $U_{\upsilon}$ is admissible, we have 
$\{s\}=A_1\cap H_2$ or $A_3\cap U_\eta$, so 
$U_{\upsilon}=[(2)_{4},\ub{4},2,2,\bs{2},2]$ or $[(2)_{3},\ub{3},3,2,\bs{2},2]$, hence $\ld(H_{1}^{\upsilon})+2\ld(H_{2}^{\upsilon})=1$, a contradiction with \eqref{eq:ld-bound_h=2}.

It remains to consider the case $\Bs\tau_{+}^{-1}\subseteq A_3$. The admissibility of $U$ implies that $\tau_{+}$ is a composition of $k-2\geq 1$ blowups on the proper transforms of $U_{\eta}$; hence $D$ is as in \ref{item:ht=3_C3}.
\end{proof}

\subsection{The list of singularity types: case $\cha\kk= 2$}\label{sec:w=2_cha=2}

We now turn to the case $\cha\kk= 2$. Here we need to reconstruct $(X,D)$ from log surfaces $(Y,D_Y)$ listed in Example \ref{ex:w=2_cha=2}, see Lemma \ref{lem:w=2_swaps}\ref{item:w=2_cha=2_nu=4},\ref{item:w=2_cha=2_nu=3}. The result is summarized in Lemma \ref{lem:w=2_cha=2}.  
As in Lemmas \ref{lem:w=3} and \ref{lem:w=2_cha_neq_2}, we list all singularity types together with decorations which encode the structure of the chosen witnessing $\P^1$-fibration, as explained in Section \ref{sec:notation}. Namely, the numbers decorated by $\dec{j}$ correspond to components of $D$ meeting the vertical $(-1)$-curve $A_j$, the bold ones correspond to the horizontal components, and the underlined one to the $2$-section. Singularity types without decorations are listed in Tables \ref{table:ht=3_char=2_moduli}--\ref{table:ht=3_char=2}. Recall from Notation \ref{not:P} that  $\Pht^{\width=2}(\cS)$ denotes the set of isomorphism classes of del Pezzo surfaces of rank one, height 3, and width 2.

We would like to reassure the reader that, although the resulting list turns out to be very long, the proof follows exactly the same lines as the proofs of Lemmas \ref{lem:w=2_cha_neq_2} and \ref{lem:w=3} above.

\begin{lemma}[Classification in case $\width=2$, $\cha\kk=2$]\label{lem:w=2_cha=2}
	Assume $\cha\kk= 2$. Let $\cS$ be a singularity type of a log terminal surface. Let $\bar{X}$ be a del Pezzo surfaces of rank 1, height 3, width 2 and type $\cS$, i.e.\ $\bar{X}\in \Pht^{\width=2}(\cS)$. Assume $\bar{X}$ has no descendant with elliptic boundary. Then $\cS$ is listed below, and the following hold.
	\begin{parts}
		\item\label{item:w=2-cha=2-moduli} If $\cS$ is as in \ref{item:ht=3_nu=4_id}--\ref{item:ht=3_nu_3_V_fork_off}, see Table \ref{table:ht=3_char=2_moduli}, then $\Pht^{\width=2}(\cS)$ has moduli dimension $1$. It is represented by an almost universal family over $\Astst$, with symmetry group $G$ listed in Table \ref{table:ht=3_char=2_moduli}; see Section \ref{sec:moduli} for definitions.
		\item\label{item:w=2-cha=2-uniqueness} If $\cS$ is as in \ref{item:ht=3_nu=3_XY'E}--\ref{item:qL3_[5]}, see Table \ref{table:ht=3_char=2}, then $\#\Pht^{\width=2}(\cS)=1$, i.e.\ $\bar{X}$ is unique up to an isomorphism.
		\item\label{item:w=2-cha=2-classification} The minimal log resolution $(X,D)$ of $\bar{X}$ admits a $\P^1$-fibration $p$ such that $\bar{X}$ swaps vertically to one of the  surfaces $\bar{Y}$ from Example \ref{ex:w=2_cha=2}, and the combinatorial type of $(X,D,p)$ is one of the following. 
	\end{parts}
	\begin{enumerate}[itemsep=0.6em]
		\item\label{item:list-swap-to-nu=4_small} $\bar{Y}$ is of type $3\rA_{1}+\rD_{4}+[2,3]$, see Example \ref{ex:w=2_cha=2}\ref{item:swap-to-nu=4_small}, and $(X,D,p)$ is one of the following: 
	\begin{longlist}
		\item \label{item:ht=3_nu=4_id} $\langle \bs{k}\dec{0};[2]\dec{2},[2]\dec{3},[2]\dec{4}\rangle+\ldec{0}[(2)_{k-2},2\dec{1},\ub{3}]\dec{1,2,3,4}+[2]\dec{2}+[2]\dec{3}+[2]\dec{4}$, $k\geq 3$,
		\item \label{item:ht=3_nu=4_id_C} $\langle \bs{2}\dec{0};[3]\dec{2},[2]\dec{3},[2]\dec{4}\rangle+[2,2\dec{2},\ub{3}\dec{1,3,4},2]\dec{0,1}+[2]\dec{3}+[2]\dec{4}$,
		\setcounter{foo}{\value{longlisti}}
	\end{longlist}
		\item \label{item:ht=3_nu=4} $\bar{Y}$ is of type $4\rA_{1}+\rD_{4}+[4]+[3]$, see Example \ref{ex:w=2_cha=2}\ref{item:swap-to-nu=4_large}, and $(X,D,p)$ is one of the following: 
	\begin{longlist}
		\setcounter{longlisti}{\value{foo}}	
		\item \label{item:ht=3_nu=4_id_X} $\langle \bs{k}\dec{0};[2]\dec{2},[2]\dec{3},[2]\dec{4}\rangle+\ldec{0}[(2)_{k-2},3]\dec{1}+[\ub{4}]\dec{1,2,3,4}+[2]\dec{1}+[2]\dec{2}+[2]\dec{3}+[2]\dec{4}$, $k\geq 3$,
		\item \label{item:ht=3_nu=4_V} $\langle \bs{3}\dec{0};[2]\dec{2},[2]\dec{3},[2]\dec{4}\rangle+[2\dec{0},3,2\dec{1},2]+[\ub{5}]\dec{1,2,3,4}+[2]\dec{2}+[2]\dec{3}+[2]\dec{4}$,
		\item \label{item:ht=3_nu=4_U} $\langle \bs{3}\dec{0};[2]\dec{2},[2]\dec{3},[2]\dec{4}\rangle+[2\dec{0},4\dec{1}]+[2,2\dec{1},\ub{4}]\dec{2,3,4}+[2]\dec{2}+[2]\dec{3}+[2]\dec{4}$,
		\item \label{item:ht=3_nu=4_id_XC} $\langle \bs{2}\dec{0};[3]\dec{2},[2]\dec{3},[2]\dec{4}\rangle+[3]\dec{0,1}+[2,2\dec{2},\ub{4}]\dec{1,3,4}+[2]\dec{1}+[2]\dec{3}+[2]\dec{4}$,
		\item \label{item:ht=3_nu=4_XC} $\langle \bs{2}\dec{0};[3]\dec{2},[2]\dec{3},[2]\dec{4}\rangle+[3\dec{0},2\dec{1},2]+[2,2\dec{2},\ub{5}]\dec{1,3,4}+[2]\dec{3}+[2]\dec{4}$,
		\setcounter{foo}{\value{longlisti}}
	\end{longlist}
	\item $\bar{Y}$ is of type $4\rA_{1}+\rA_{3}+[3]$, see Example \ref{ex:w=2_cha=2}\ref{item:swap-to-nu=3}, and $(X,D,p)$ is one of the following: 
	\begin{longlist}\setcounter{longlisti}{\value{foo}}
		\item\label{item:ht=3_nu_3_off}
		$\langle k\dec{2};[2],[\ub{3}]\dec{1,3},[2\dec{3},\bs{2}\dec{0},2]\rangle+[3\dec{0},2\dec{1},2]+[2]\dec{3}+[(2)_{k-2}]\dec{2}$, $k\geq 2$,
		\item\label{item:ht=3_nu_3_V_fork_off}
		$[2\dec{3},\bs{3}\dec{0},2,2\dec{2},2]+\langle k\dec{1};[2],[\ub{3}]\dec{3},[2\dec{0},3]\rangle+[(2)_{k-2}]\dec{1} +[2]\dec{3}$,
\medskip

		\item\label{item:ht=3_nu=3_XY'E} $\ldec{3}[(2)_{k-2},3\dec{3},\bs{3}\dec{0},2,2\dec{2},2]+[2,2\dec{1},\ub{3}\dec{2},k\dec{3},2]+[2\dec{0},4\dec{1}]$, $k\in \{2,3\}$,
		\item\label{item:ht=3_nu=3_XY'U_eps=0}
		$[3\dec{3},\bs{2}\dec{0},2,2\dec{2},2]+[2,k\dec{1},\ub{3}\dec{2},2\dec{3},2]+[4\dec{0},(2)_{k-2}]\dec{1}$, $k\leq 7$,
		\item\label{item:ht=3_nu=3_XY'U_eps=1}
		$[3\dec{3},\bs{2}\dec{0},2,2\dec{2},2]+[3\dec{0},k\dec{1},\ub{3}\dec{2},2\dec{3},2]+[3,(2)_{k-2}]\dec{1}$, $k\leq 5$,
		\item\label{item:ht=3_nu=3_XY'BE}
		$[3\dec{3},\bs{3}\dec{0},2,2\dec{2},2]+[2,2\dec{3},\ub{4}]\dec{1,2}+[2\dec{0},3,2\dec{1},2]$,
		\item\label{item:ht=3_nu=3_XY'BB} 
		$[3\dec{3},\bs{2}\dec{0},2,3\dec{2},2]+[2\dec{2},\ub{4}\dec{1},2\dec{3},2]+[3\dec{0},2\dec{1},2]$,
		\item\label{item:ht=3_nu=3_XY'BD} 	
		$\ldec{3}[(2)_{k-2},3,\bs{2}\dec{0},2,2\dec{2},2]+[2,k\dec{3},\ub{4}\dec{1,2}]+[3\dec{0},2\dec{1},2]$, $k\in \{3,4,5\}$,
		\item\label{item:ht=3_nu=3_XY'B_T=0}
		$[3\dec{3},\bs{2}\dec{0},2,2\dec{2},2]+\ldec{1}[(2)_{k-2},\ub{4}\dec{2},2\dec{3},2]+[3\dec{0},k\dec{1},2]$, $k\leq 9$,
		\item\label{item:ht=3_nu=3_ZY'A}
		$[2,2\dec{1},\ub{3}\dec{2},2\dec{3},2,\bs{3}\dec{0},2,2\dec{2},2]+[4]\dec{0,1}+[3]\dec{3}$,
		\item\label{item:ht=3_nu=3_ZY'B}
		$[\ub{4}\dec{1,2},k\dec{3},2,\bs{2}\dec{0},2,2\dec{2},2]+[3\dec{0},2\dec{1},2]+\ldec{3}[(2)_{k-2},3]$, $k\in \{2,3\}$,
		\item\label{item:ht=3_nu=3_ZY'BD}
		$[2\dec{1},\ub{4}\dec{2},2\dec{3},2,\bs{2}\dec{0},2,2\dec{2},2]+[3\dec{0},3\dec{1},2]+[3]\dec{3}$,
		\item\label{item:ht=3_nu=3_YYB}
		$[2,k\dec{2},2,\bs{2}\dec{0},2,2\dec{3},2]+[\ub{5}\dec{1,3},(2)_{k-2}]\dec{2}+[3\dec{0},2\dec{1},2]$, $k\in \{4,5\}$,
		\item\label{item:miss2}
		$[2,3\dec{2},2,\bs{2}\dec{0},2,2\dec{3},2]+[3\dec{0},2\dec{1},\ub{4},2\dec{2}]+[3]\dec{1}$,
		\item\label{item:ht=3_nu=3_YA}
		$[2,3\dec{2},2,\bs{2}\dec{0},2,2\dec{3},2]+[2,k\dec{1},\ub{4}\dec{3},2]\dec{2}+[4\dec{0},(2)_{k-2}]\dec{1}$, $k\in \{2,3\}$,
		\item\label{item:ht=3_nu=3_YB_T=0}
		$[2,3\dec{2},2,\bs{2}\dec{0},2,2\dec{3},2]+[2\dec{2},\ub{5}\dec{3},(2)_{k-2}]\dec{1}+[3\dec{0},k\dec{1},2]\dec{1}$, $k\in \{2,3,4\}$,
		\item\label{item:ht=3_nu=3_UA}
		$[2,2\dec{2},2,\bs{3}\dec{0},2,2\dec{3},2]+[2,2\dec{1},\ub{4}\dec{2,3}]+[2\dec{0},4\dec{1}]$,
		\item\label{item:ht=3_nu=3_UB}
		$[2,2\dec{2},2,\bs{3}\dec{0},2,2\dec{3},2]+[\ub{5}]\dec{1,2,3}+[2\dec{0},3,2\dec{1},2]$,
		\item \label{item:ht=3_nu=3_XX'} $[3\dec{2},\bs{k}\dec{0},3\dec{3}]+[2,2\dec{2},\ub{3}\dec{1},2\dec{3},2]+\ldec{0}[(2)_{k-2},3,2\dec{1},2]$, $k\in \{2,3\}$,
\medskip
		
		\item\label{item:ht=3_nu=3_XY'UV}
		$[3\dec{3},\bs{2}\dec{0},2,2\dec{2},2]+\langle 2;[2],[2]\dec{1};[2,2\dec{3},\ub{3}\dec{2}]\rangle+[5]\dec{0,1}$,
		\item\label{item:ht=3_nu=3_XY'BA}	
		$\langle 2;[2],[2]\dec{2},[3\dec{3},\bs{2}\dec{0},2]\rangle+[\ub{5}\dec{1,2},2\dec{3},2]+[3\dec{0},2\dec{1},2]$,
		\item\label{item:ht=3_nu=3_XY'BC} 
		$[4\dec{3},\bs{2}\dec{0},2,2\dec{2},2]+\langle 2;[2],[2]\dec{3},[\ub{4}]\dec{1,2}\rangle+[3\dec{0},2\dec{1},2]$,
		\item\label{item:miss} 
		$[2\dec{3},3,\bs{2}\dec{0},2,2\dec{2},2]+[\ub{5}\dec{1,2},3\dec{3},2]+\langle 2;[3]\dec{0},[2]\dec{1},[2]\rangle$,
		\item\label{item:ht=3_nu=3_XY'B_k=2}
		$[3\dec{3},\bs{2}\dec{0},2,2\dec{2},2]+[\ub{k}\dec{1,2},2\dec{3},2]+\langle 2;[2],[3]\dec{0},\ldec{1}[(2)_{k-4}]\rangle$, $k\in \{5,6,7\}$,
		\item\label{item:ht=3_nu=3_XY'B_k=2_T=[3]}
		$[3\dec{3},\bs{2}\dec{0},2,2\dec{2},2]+[2\dec{1},\ub{5}\dec{2},2\dec{3},2]+\langle 2;[2],[3]\dec{0},[3]\dec{1}\rangle$, 
		\item\label{item:ht=3_nu=3_XY'B_k=3}
		$[3\dec{3},\bs{2}\dec{0},2,2\dec{2},2]+[k\dec{1},\ub{4}\dec{2},2\dec{3},2]+\langle 3;[2],[3]\dec{0},\ldec{1}[(2)_{k-2}]\rangle$, $k\in \{3,4\}$, 
		\item\label{item:ht=3_nu=3_XY'B_k=4}
		$[3\dec{3},\bs{2}\dec{0},2,2\dec{2},2]+[3\dec{1},2,\ub{4}\dec{2},2\dec{3},2]+\langle 4;[2],[3]\dec{0},[2]\dec{1}\rangle$, 
		\item\label{item:ht=3_nu=3_ZY'BE}
		$[\ub{5}\dec{1,2},2\dec{3},2,\bs{2}\dec{0},2,2\dec{2},2]+\langle 2;[2],[2]\dec{1},[3]\dec{0}\rangle+[3]\dec{3}$,
		\item\label{item:ht=3_nu=3_YB_T-neq-0}
		$[2,3\dec{2},2,\bs{2}\dec{0},2,2\dec{3},2]+[2\dec{2},\ub{k}\dec{1,3}]+\langle 2;[2],[3]\dec{0},\ldec{1}[(2)_{k-5}]\rangle$, $k\in \{6,7\}$,
		\medskip		
		
		\item\label{item:ht=3_nu_3_VZ_chain_k=3}
		$[2\dec{3},\bs{3}\dec{0},2,2\dec{2},2]+[2\dec{0},3,k\dec{1},\ub{3}\dec{3}]+[3,(2)_{k-2}]\dec{1}+[2]\dec{3}$,
		\item\label{item:ht=3_nu_3_VZ_chain_k=4}
		$[2\dec{3},\bs{4}\dec{0},2,2\dec{2},2]+[2\dec{0},2,3,k\dec{1},\ub{3}\dec{3}]+[3,(2)_{k-2}]\dec{1}+[2]\dec{3}$, $k\in \{2,3,4,5\}$,
		\item\label{item:ht=3_nu_3_VZ_chain_m=2}
		$[2\dec{3},\bs{k}\dec{0},2,2\dec{2},2]+\ldec{0}[(2)_{k-2},3,2\dec{1},\ub{3}\dec{3}]+[3]\dec{1}+[2]\dec{3}$, $k\in \{5,6,7\}$,
		\item\label{item:ht=3_nu_3_YV_chain_l=2,k=3}
		$[2\dec{3},\bs{3}\dec{0},2,2\dec{2},2]+[2\dec{0},3,k\dec{1},2]+\ldec{1}[(2)_{k-2},\ub{4}\dec{2,3}]+[2]\dec{3}$,
		\item\label{item:ht=3_nu_3_YV_chain_l=2,k=4}
		$[2\dec{3},\bs{4}\dec{0},2,2\dec{2},2]+[2\dec{0},2,3,k\dec{1},2]+\ldec{1}[(2)_{k-2},\ub{4}\dec{2,3}]+[2]\dec{3}$, $k\leq 7$,
		\item\label{item:ht=3_nu_3_YV_chain_l=2,k>4}
		$[2\dec{3},\bs{k}\dec{0},2,2\dec{2},2]+\ldec{0}[(2)_{k-2},3,2\dec{1},2]+[\ub{4}]\dec{1,2,3}+[2]\dec{3}$, $k\in \{5,6\}$,	
		\item\label{item:ht=3_nu_3_YV_chain_k=2,l=3}
		$[2\dec{3},\bs{2}\dec{0},2,3\dec{2},2]+[3\dec{0},k\dec{1},2]+\ldec{1}[(2)_{k-2},\ub{4}\dec{3},2\dec{2}]+[2]\dec{3}$,
		\item\label{item:ht=3_nu_3_YV_chain_l=4}
		$[2\dec{3},\bs{2}\dec{0},2,4\dec{2},2]+[3\dec{0},k\dec{1},2]+\ldec{1}[(2)_{k-2},\ub{4}\dec{3},2,2]\dec{2}+[2]\dec{3}$, $k\leq 9$,	
		\item\label{item:ht=3_nu_3_YV_chain_k=2,l>4,m=2}
		$[2\dec{3},\bs{2}\dec{0},2,k\dec{2},2]+[3\dec{0},2\dec{1},2]+\ldec{1}[\ub{4}\dec{3},(2)_{k-2}]\dec{2}+[2]\dec{3}$, $k\geq 5$,
		\item\label{item:ht=3_nu_3_YV_chain_k=2,l>4,m=3}
		$[2\dec{3},\bs{2}\dec{0},2,k\dec{2},2]+[3\dec{0},3\dec{1},2]+\ldec{1}[2,\ub{4}\dec{3},(2)_{k-2}]\dec{2}+[2]\dec{3}$,  $10\geq k\geq 5$,
		\item\label{item:ht=3_nu_3_YV_chain_k=2,l>4,m=4}
		$[2\dec{3},\bs{2}\dec{0},2,5\dec{2},2]+[3\dec{0},4\dec{1},2]+\ldec{1}[2,2,\ub{4}\dec{3},2,2,2]\dec{2}+[2]\dec{3}$,
		\item\label{item:ht=3_nu_3_YV_chain_k=3,l>2,m=2}
		$[2\dec{3},\bs{3}\dec{0},2,k\dec{2},2]+[2\dec{0},3,2\dec{1},2]+\ldec{1}[\ub{4}\dec{3},(2)_{k-2}]\dec{2}+[2]\dec{3}$, $3\leq k\leq 7$,
		\item\label{item:ht=3_nu_3_YV_chain_k>3,l=3,m=2}
		$[2\dec{3},\bs{k}\dec{0},2,3\dec{2},2]+\ldec{0}[(2)_{k-2},3,2\dec{1},2]+\ldec{1}[\ub{4}\dec{3},2]\dec{2}+[2]\dec{3}$, $k\in \{4,5\}$,
		\item\label{item:ht=3_nu_3_WZ_l=2}
		$[2\dec{3},\bs{2}\dec{0},2,2\dec{2},\ub{2}\dec{3},k\dec{1},3\dec{0}]+[3,(2)_{k-2}]\dec{1}+[3]\dec{2}+[2]\dec{3}$, $k\geq 2$,
		\item\label{item:ht=3_nu_3_WZ_k=2}
		$[2\dec{3},\bs{2}\dec{0},2,k\dec{2},\ub{2}\dec{3},2\dec{1},3\dec{0}]+[3,(2)_{k-2}]\dec{2}+[3]\dec{1}+[2]\dec{3}$, $k\in \{3,4\}$,
		\item\label{item:ht=3_nu_3_WX_l=2}
		$[2\dec{3},\bs{2}\dec{0},2,k\dec{2},\ub{2}\dec{3},2\dec{1},2]+[4]\dec{0,1}+[3,(2)_{k-2}]\dec{2}+[2]\dec{3}$, $k\geq 3$,
		\item\label{item:ht=3_nu_3_WX_k=2}
		$[2\dec{3},\bs{k}\dec{0},2,2\dec{2},\ub{2}\dec{3},2\dec{1},2]+[4\dec{1},(2)_{k-2}]\dec{0}+[3]\dec{2}+[2]\dec{3}$, $k\geq 3$,
		\item\label{item:ht=3_nu_3_WX_k,l=3}
		$[2\dec{3},\bs{3}\dec{0},2,3\dec{2},\ub{2}\dec{3},2\dec{1},2]+[2\dec{0},4\dec{1}]+[2\dec{2},3]+[2]\dec{3}$,	
		\item\label{item:ht=3_nu_3_WY_k=2}
		$[2\dec{3},\bs{2}\dec{0},2,k\dec{2},\ub{3}\dec{1,3}]+[3\dec{0},2\dec{1},2]+[3,(2)_{k-2}]\dec{2}+[2]\dec{3}$,
		\item\label{item:ht=3_nu_3_WY_k=3}
		$[2\dec{3},\bs{3}\dec{0},2,k\dec{2},\ub{3}\dec{1,3}]+[2\dec{0},3,2\dec{1},2]+[3,(2)_{k-2}]\dec{2}+[2]\dec{3}$, $k\leq 5$,
		\item\label{item:ht=3_nu_3_WY_k>3}
		$[2\dec{3},\bs{k}\dec{0},2,l\dec{2},\ub{3}\dec{1,3}]+\ldec{0}[(2)_{k-2},3,2\dec{1},2]+[3,(2)_{l-2}]\dec{2}+[2]\dec{3}$, $k\in \{4,5,6\}$, $l\in \{2,3\}$,
		\item\label{item:qL3_A0_chain} $[3\dec{2},\bs{k}\dec{0},2\dec{3}]+[\ub{3}\dec{1,3},2\dec{2},2]+\ldec{0}[(2)_{k-2},3,2\dec{1},2]+[2]\dec{3}$,
		\item\label{item:qL3_k1} $[3\dec{2},\bs{k}\dec{0},2\dec{3}]+[2\dec{1},\ub{3}\dec{3},2\dec{2},2]+\ldec{0}[(2)_{k-2},3,3\dec{1},2]+[2]\dec{3}$, $3\leq k\leq 9$,
		\item\label{item:qL3_k2} $[3\dec{2},\bs{3}\dec{0},2\dec{3}]+\ldec{1}[(2)_{k-2},\ub{3}\dec{3},2\dec{2},2]+[2\dec{0},3,k\dec{1},2]+[2]\dec{3}$, $3\leq k\leq 9$,
		\item\label{item:qL3_4} $[3\dec{2},\bs{4}\dec{0},2\dec{3}]+\ldec{1}[2,2,\ub{3}\dec{3},2\dec{2},2]+[2\dec{0},2,3,4\dec{1},2]+[2]\dec{3}$,
\medskip

		\item\label{item:ht=3_nu_3_VZ_fork}
		$[2\dec{3},\bs{3}\dec{0},2,2\dec{2},2]+\langle k;[2]\dec{1},[\ub{3}]\dec{3},[2\dec{0},3]\rangle +[3,(2)_{k-3},3]\dec{1}+[2]\dec{3}$,
		\item\label{item:ht=3_nu_3_YV_fork_l=2,m=2}
		$[2\dec{3},\bs{4}\dec{0},2,2\dec{2},2]+\langle 2;[2],[2]\dec{1},\ldec{0}[2,2,3]\rangle+[\ub{5}]\dec{1,2,3}+[2]\dec{3}$,
		\item\label{item:ht=3_nu_3_YV_fork_l=2,m>2}
		$[2\dec{3},\bs{3}\dec{0},2,2\dec{2},2]+\langle k;[2],\ldec{1}T\trp,\ldec{0}[2,3]\rangle+\ldec{1}T^{*}*[(2)_{k-2},\ub{4}]\dec{2,3}+[2]\dec{3}$, $d(T)\leq 3$,
		\item\label{item:ht=3_nu_3_YV_fork_l=3}
		$[2\dec{3},\bs{2}\dec{0},2,3\dec{2},2]+\langle k;[2],\ldec{1}T\trp,\ldec{0}[3]\rangle+\ldec{1}T^{*}*[(2)_{k-2},\ub{4}\dec{3},2]\dec{2}+[2]\dec{3}$, $d(T)\leq 5$,
		\item\label{item:ht=3_nu_3_YV_fork_T=[2]_m=2}
		$[2\dec{3},\bs{2}\dec{0},2,k\dec{2},2]+\langle
		2;[2],[3]\dec{0},\ldec{1}[2]\rangle+[\ub{5}\dec{1,3},(2)_{k-2}]\dec{2}+[2]\dec{3}$, $k\geq 4$,
		\item\label{item:ht=3_nu_3_YV_fork_T=[2]_m>2}
		$[2\dec{3},\bs{2}\dec{0},2,4\dec{2},2]+\langle
		k;[2],[3]\dec{0},\ldec{1}[2]\rangle+[3\dec{1},(2)_{k-3},\ub{4}\dec{3},2,2]\dec{2}+[2]\dec{3}$, $k\in \{3,4\}$,	
		\item\label{item:ht=3_nu_3_YV_fork_T=[2,2]_m=3}
		$[2\dec{3},\bs{2}\dec{0},2,4\dec{2},2]+\langle
		3;[2],[3]\dec{0},\ldec{1}[2,2]\rangle+[4\dec{1},\ub{4}\dec{3},2,2]\dec{2}+[2]\dec{3}$,
		\item\label{item:ht=3_nu_3_YV_fork_T=[2,2]_m=2}
		$[2\dec{3},\bs{2}\dec{0},2,k\dec{2},2]+\langle
		2;[2],[3]\dec{0},\ldec{1}[2,2]\rangle+[\ub{6}\dec{1,3},(2)_{k-2}]\dec{2}+[2]\dec{3}$, $k\in \{4,5,6,7\}$,
		\item\label{item:ht=3_nu_3_YV_fork_T=[3]}
		$[2\dec{3},\bs{2}\dec{0},2,4\dec{2},2]+\langle 2;[2],[3]\dec{0},\ldec{1}[3]\rangle+[2\dec{1},\ub{5}\dec{3},2,2]\dec{2}+[2]\dec{3}$,
		\item\label{item:ht=3_nu_3_YV_fork_T=[2,2,2]}
		$[2\dec{3},\bs{2}\dec{0},2,4\dec{2},2]+\langle
		2;[2],[3]\dec{1},\ldec{0}[2,2,2]\rangle+[\ub{7}\dec{1,3},2,2]\dec{2}+[2]\dec{3}$,
		\item\label{item:ht=3_nu_3_ZV_U-3}
		$\langle \bs{2};[2]\dec{3},[3]\dec{0},[2,2\dec{2},2]\rangle+
		[2\dec{0},4,2\dec{1},\ub{3}\dec{2,3}]+[3]\dec{1}+[2]\dec{3}$,
		\item\label{item:ht=3_nu_3_ZV_U-2}
		$\langle \bs{3};[2]\dec{3},[2]\dec{0},[2,2\dec{2},2]\rangle+
		[3\dec{0},3,2\dec{1},\ub{3}\dec{2,3}]+[3]\dec{1}+[2]\dec{3}$,
		\item\label{item:ht=3_nu_3_YV_U-2_k=2}
		$\langle \bs{2};[2]\dec{3},[2]\dec{0},[2,3\dec{2},2]\rangle+
		[4\dec{0},2\dec{1},2]+
		[\ub{4}\dec{1,3},2\dec{2}]+[2]\dec{3}$,
		\item\label{item:ht=3_nu_3_YV_U-2_k>2}
		$\langle \bs{k};[2]\dec{3},[2]\dec{0},[2,2\dec{2},2]\rangle+
		[3\dec{0},(2)_{k-3},3,2\dec{1},2]+
		[\ub{4}]\dec{1,2,3}+[2]\dec{3}$, $k\geq 3$,
		\item\label{item:ht=3_nu_3_YV_U-2_m=3}
		$\langle \bs{3};[2]\dec{3},[2]\dec{0},[2,2\dec{2},2]\rangle+
		[3\dec{0},3,3\dec{1},2]+
		[2\dec{1},\ub{4}]\dec{2,3}+[2]\dec{3}$,
		\item\label{item:ht=3_nu_3_YV_U-3_k>2}
		$\langle \bs{k};[2]\dec{3},[3]\dec{0},[2,2\dec{2},2]\rangle+
		[2\dec{0},3,(2)_{k-3},3,2\dec{1},2]+
		[\ub{4}]\dec{1,2,3}+[2]\dec{3}$, $k\geq 3$,
		\item\label{item:ht=3_nu_3_YV_U-3_k=2}
		$\langle \bs{2};[2]\dec{3},[3]\dec{0},[2,2\dec{2},2]\rangle+
		[2\dec{0},4,k\dec{1},2]+
		\ldec{1}[(2)_{k-2},\ub{4}]\dec{2,3}+[2]\dec{3}$, $k\in \{2,3\}$,
		\item\label{item:ht=3_nu_3_YV_U'-2_l=3}
		$\langle 2;[2],[2]\dec{2},[2\dec{3},\bs{3}\dec{0},2]\rangle+
		[2\dec{0},3,2\dec{1},2]+
		[\ub{5}]\dec{1,2,3}+[2]\dec{3}$,
		\item\label{item:ht=3_nu_3_YV_U'-2_k=4}
		$\langle 4;[2],[2]\dec{2},[2\dec{3},\bs{2}\dec{0},2]\rangle+
		[3\dec{0},3\dec{1},2]+
		[2\dec{1},\ub{4}\dec{3},2,3\dec{2}]+[2]\dec{3}$,
		\item\label{item:ht=3_nu_3_YV_U'-2_k=3}
		$\langle 3;[2],[2]\dec{2},[2\dec{3},\bs{2}\dec{0},2]\rangle+
		[3\dec{0},k\dec{1},2]+
		\ldec{1}[(2)_{k-2},\ub{4}\dec{3},3\dec{2}]+[2]\dec{3}$, $k\in \{3,4\}$,
		\item\label{item:ht=3_nu_3_YV_U'-2_k>2_m=2}
		$\langle k;[2],[2]\dec{2},[2\dec{3},\bs{2}\dec{0},2]\rangle+
		[3\dec{0},2\dec{1},2]+
		[\ub{4}\dec{1,3},(2)_{k-3},3\dec{2}]+[2]\dec{3}$, $k\geq 3$,
		\item\label{item:ht=3_nu_3_YV_U'-3_chain_k>2}
		$\langle k;[2],[3]\dec{2},[2\dec{3},\bs{2}\dec{0},2]\rangle +
		[3\dec{0},2\dec{1},2]+
		[\ub{4}\dec{1,3},3,(2)_{k-3},3,2\dec{2}]+[2]\dec{3}$,
		\item\label{item:ht=3_nu_3_YV_U'-3_chain_k=2}
		$\langle 2;[2],[3]\dec{2},[2\dec{3},\bs{2}\dec{0},2]\rangle +
		[3\dec{0},k\dec{1},2]+
		\ldec{1}[(2)_{k-2},\ub{4}\dec{3},4,2\dec{2}]+[2]\dec{3}$, $k\in \{2,3,4\}$,
		\item\label{item:ht=3_nu_3_WXD}
		$\langle \bs{k};[2]\dec{0},[2]\dec{3},[2,2\dec{1},\ub{2}\dec{3},2\dec{2},2]\rangle+[4\dec{1},(2)_{k-3},3]\dec{0}+[3]\dec{2} + [2]\dec{3}$, $k\geq 3$,
		\item\label{item:ht=3_nu_3_WYV}
		$\langle k;[2]\dec{2},[\ub{3}]\dec{1,3},[2\dec{3},\bs{2}\dec{0},2]\rangle+ [3\dec{0},2\dec{1},2]+[3,(2)_{k-3},3]\dec{2}+[2]\dec{3}$, $k\geq 3$,
		\item\label{item:ht=3_nu_3_WYV_k=2} $\langle 2;[2]\dec{2},[\ub{3}]\dec{1,3},[2\dec{3},\bs{2}\dec{0},2]\rangle+ [3\dec{0},2\dec{1},2]+[4]\dec{2}+[2]\dec{3}$,
		\item\label{item:ht=3_nu_3_WYC}
		$\langle \bs{2};[2]\dec{0},[2]\dec{3},[\ub{3}\dec{1,3},2\dec{2},2]\rangle+[4\dec{0},2\dec{1},2]+[3]\dec{2}+[2]\dec{3}$,
		\item\label{item:ht=3_nu_3_WYT}
		$[2\dec{3},\bs{2}\dec{0},2,2\dec{2},\ub{3}\dec{3},(2)_{k-2}]*(T^{*})\dec{1}+\langle k;[2],[3]\dec{0},\ldec{1}T\rangle +[3]\dec{2}+[2]\dec{3}$, $d(T)\leq 5$,
		\item\label{item:ht=3_nu_3_WYU}
		$[2\dec{3},\bs{2}\dec{0},2,k\dec{2},\ub{4}\dec{1,3}]+\langle 2;[2],[3]\dec{0},[2]\dec{1}\rangle +[3,(2)_{k-2}]\dec{2}+[2]\dec{3}$, $k\geq 3$,
		\item\label{item:ht=3_nu_3_WYUU}
		$[2\dec{3},\bs{2}\dec{0},2,k\dec{2},\ub{5}\dec{1,3}]+\langle 2;[2],[3]\dec{0},[2\dec{1},2]\rangle +[3,(2)_{k-2}]\dec{2}+[2]\dec{3}$, $k\in \{3,4\}$
		\item\label{item:qL3_A0_fork} $\langle \bs{k};\ldec{0}T\trp,[3]\dec{2},[2]\dec{3}\rangle +[\ub{3}\dec{1,3},2\dec{2},2]+\ldec{0}T^{*}*[(2)_{k-2},3,2\dec{1},2]+[2]\dec{3}$, $d(T)\leq 5$,
		\item\label{item:qL3_rA0_T=[3]}
		$\langle \bs{2},[3]\dec{0},[3]\dec{2},[2]\dec{3}\rangle+\ldec{1}[2,\ub{3}\dec{3},2\dec{2},2]+[2\dec{0},4,3\dec{1},2]+[2]\dec{3}$,
		\item\label{item:qL3_rA0} $\langle \bs{2},[2]\dec{0},[3]\dec{2},[2]\dec{3}\rangle+\ldec{1}[(2)_{k-2},\ub{3}\dec{3},2\dec{2},2]+[4\dec{0},k\dec{1},2]+[2]\dec{3}$, $k\geq 3$,
		\item\label{item:qL3_rA0_l=3} $\langle \bs{k},[2]\dec{0},[3]\dec{2},[2]\dec{3}\rangle+[2\dec{1},\ub{3}\dec{3},2\dec{2},2]+[3\dec{0},(2)_{k-3},3\dec{1},2]+[2]\dec{3}$, $k\in \{3,4\}$,
		\item\label{item:qL3_rA0_s=2} $\langle \bs{2},[2\dec{0},2],[3]\dec{2},[2]\dec{3}\rangle+\ldec{1}[(2)_{k-2},\ub{3}\dec{3},2\dec{2},2]+[5\dec{0},k\dec{1},2]+[2]\dec{3}$, $3\leq k\leq 6$,
		\item\label{item:qL3_rA0_s=2_l=3} $\langle \bs{3},[2\dec{0},2],[3]\dec{2},[2]\dec{3}\rangle+[2\dec{1},\ub{3}\dec{3},2\dec{2},2]+[4\dec{0},2,3\dec{1},2]+[2]\dec{3}$,
		\item\label{item:qL3_rA0_s=3_l=3} $\langle \bs{2},[2\dec{0},2,2],[3]\dec{2},[2]\dec{3}\rangle+[2\dec{1},\ub{3}\dec{3},2\dec{2},2]+[6\dec{0},3\dec{1},2]+[2]\dec{3}$,
		\item\label{item:qL3_rA1} $[3\dec{2},\bs{k}\dec{0},2\dec{3}]+[\ub{4}\dec{1,3},2\dec{2},2]+\langle 2;\ldec{0}[(2)_{k-2},3],[2]\dec{1},[2]\rangle+[2]\dec{3}$, $k\in \{3,4\}$,
		\item\label{item:qL3_rA1_k=3} $[3\dec{2},\bs{3}\dec{0},2\dec{3}]+[3\dec{1},\ub{3}\dec{3},2\dec{2},2]+\langle 3;\ldec{0}[2,3],[2]\dec{1},[2]\rangle+[2]\dec{3}$,
		\medskip
		
		\item\label{item:ht=3_nu_3_YV_U'-2_T=[2]}
		$\langle k;[2],[2]\dec{2},[2\dec{3},\bs{2}\dec{0},2]\rangle+
		\langle 2;[2],[3]\dec{0},[2]\dec{1}\rangle+
		[\ub{5}\dec{1,3},(2)_{k-3},3\dec{2}]+[2]\dec{3}$, $k\geq 3$,
		\item\label{item:ht=3_nu_3_YV_U'-2_T=[2,2]}
		$\langle 3;[2],[2]\dec{2},[2\dec{3},\bs{2}\dec{0},2]\rangle+
		\langle 2;[2],[3]\dec{0},\ldec{1}[(2)_{k-4}]\rangle+
		[\ub{k}\dec{1,3},3\dec{2}]+[2]\dec{3}$, $k\in \{5,6\}$,
		\item\label{item:ht=3_nu_3_YV_U'-3_fork_k>2}
		$\langle k;[2],[3]\dec{2},[2\dec{3},\bs{2}\dec{0},2]\rangle +
		\langle 2;[2],[3]\dec{0},[2]\dec{2}\rangle+
		[\ub{5}\dec{1,3},3,(2)_{k-3},3,2\dec{2}]+[2]\dec{3}$, $k\geq 3$,
		\item\label{item:ht=3_nu_3_YV_U'-3_fork_k=2}
		$\langle 2;[2],[3]\dec{2},[2\dec{3},\bs{2}\dec{0},2]\rangle +
		\langle 2;[2],[3]\dec{0},\ldec{2}[(2)_{k-4}] \rangle+
		[\ub{k}\dec{1,3},4,2\dec{2}]+[2]\dec{3}$, $k\in \{5,6\}$,
		\item\label{item:qL3_[4]} $\langle \bs{2};[3]\dec{2},[2]\dec{0},[2]\dec{3}\rangle+ \ldec{1}T^{*}*[(2)_{k-2},\ub{3}\dec{3},2\dec{2},2]+ \langle k;[4]\dec{0},\ldec{1}T\trp,[2]\rangle+[2]\dec{3}$, $d(T)\leq 3$,
		\item\label{item:qL3_[5]} $\langle \bs{2};[3]\dec{2},[2\dec{0},2],[2]\dec{3}\rangle+ [\ub{4}\dec{1,3},2\dec{2},2]+ \langle 2;[5]\dec{0},[2]\dec{1},[2]\rangle+[2]\dec{3}$.
	\end{longlist}
	\end{enumerate}
\end{lemma}
\begin{proof}
	As usual, 
	we first explain how to deduce the uniqueness results \ref{item:w=2-cha=2-moduli}, \ref{item:w=2-cha=2-uniqueness} from the list \ref{item:w=2-cha=2-classification}. 
	
	 Fix a singularity type $\cS$ as in the above list. No surface $\bar{X}\in \Pht^{\width=3}(\cS)$ has a descendant with elliptic boundary: this follows from Lemma \ref{lem:no-deb} in each case except \ref{item:ht=3_nu_3_off}--\ref{item:ht=3_nu_3_V_fork_off}, where one directly compares $\cS$ with \cite[Theorem~E]{PaPe_ht_2}. Thus by \ref{item:w=2-cha=2-classification}, the minimal log resolution $(X,D)$ of $\bar{X}\in \Pht^{\width=3}(\cS)$ admits a $\P^1$-fibration $p$ as in the above list. Each $\cS$ appears there exactly once, so it uniquely determines the combinatorial type of $(X,D,p)$, hence the combinatorial type $\check{\cS}$ of $(X,\check{D})$, see Notation \ref{not:phi_+_h=2}. Thus the set  of isomorphism classes of minimal log resolutions of  surfaces in $\Pht^{\width=2}(\cS)$, call it $\Phtres^{\width=2}(\cS)$, is contained in the image of $\cP(\check{\cS})$ in $\cP(\cS)$, where $\cP(*)$ is the set of isomorphic classes of log surfaces of combinatorial type $*$, see \cite[\sec 2F]{PaPe_ht_2}. 
	 In fact, these two sets are equal. Indeed, if $(X,\check{D})\in \cP(\check{\cS})$ and $(X,D)$ is its image in $\cP(\cS)$, i.e.\ $D$ equals $\check{D}$ minus all $(-1)$-curves, then $(X,D)$ satisfies the inequality \eqref{eq:ld_phi_H}, so it is the minimal log resolution of a del Pezzo surface of rank one and type $\cS$, which has height $3$ because $\cS$ does not appear in \cite{PaPe_ht_2}, and has width $2$ because $\cS$ does not appear in Lemma \ref{lem:w=3}, so $(X,D)\in \Phtres^{\width=2}(\cS)$, 
	 as claimed. 
	 
	 Let $d=\nu-3+\epsilon$ be the moduli dimension specified in Lemma \ref{lem:w=2_uniqueness}\ref{item:w=2_uniq}. In cases \ref{item:ht=3_nu=4_id}--\ref{item:ht=3_nu=4_U} we have $\nu=4$, so $\epsilon=0$ by Lemma \ref{lem:w=2_uniqueness}\ref{item:w=2_GK}, and therefore $d=1$. In  cases \ref{item:ht=3_nu_3_off}--\ref{item:ht=3_nu_3_V_fork_off} we have $\nu=3$, $\epsilon=1$, so again $d=1$. In the remaining cases we have $\nu=3$, $\epsilon=0$, so $d=0$. 
	 
	 If $d=0$ then Lemma \ref{lem:w=2_uniqueness}\ref{item:w=2_uniq} gives $\#\cP_{+}(\check{\cS})=1$, so $\#\Phtres^{\width=2}(\cS)=1$, 
	 as claimed in \ref{item:w=2-cha=2-uniqueness}. Assume $d=1$. Then by Lemma \ref{lem:w=2_uniqueness}\ref{item:w=2_uniq} the set $\cP_{+}(\check{\cS})$ is represented by a universal  $\Aut(\check{\cS})$-faithful family over $B\cong \Astst$. We check directly that  any automorphism of the weighted graph of $D$ preserves the vertices corresponding to the $1$- and $2$-section. Thus by \cite[Lemmas 2.21, 2.20(c)]{PaPe_ht_2} our family, viewed as one representing $\Phtres^{\width=2}(\cS)$, is almost faithful, so $\Phtres^{\width=2}(\cS)$ has moduli dimension $1$, as needed. Moreover, it follows that this family, say $f$, is $G$-equivariant, where $G\leq \Aut(\check{\cS})$ is trivial if $\nu=3$ and permutes degenerate fibers of the same type if $\nu=4$; and two fibers of $f$ are isomorphic if and only if they lie in the same $G$-orbit. We check directly that $G$ is as in Table \ref{table:ht=3_char=2_moduli}. Eventually, we note that the homomorphism $G\to \Aut(B)\cong S_3$ is injective, because  $\Aut(\check{\cS})$ fixes the unique fiber with two $(-1)$-curves, so $G$ is indeed the symmetry group of $f$, see Section \ref{sec:moduli}.

	\smallskip

	Therefore, it remains to prove \ref{item:w=2-cha=2-classification}, i.e.\ to verify the completeness of the above list. 
	We keep Notation \ref{not:phi_+_h=2}. 
	
	\paragraph{Case $\nu=4$, see Lemma \ref{lem:w=2_swaps}\ref{item:w=2_cha=2_nu=4}. and Figures \ref{fig:swap-to-nu=4_small}, \ref{fig:swap-to-nu=4_large}} 
	 Put $\delta=1$ if $p_{1}'\not \in \Bs\psi^{-1}$, i.e.\ if $\bar{X}$ swaps to a surface $3\rA_{1}+\rD_{4}+[2,3]$, and $\delta=0$ otherwise, see Figures \ref{fig:swap-to-nu=4_small} and \ref{fig:swap-to-nu=4_large}, respectively. Let $\gamma$ be a composition of $\tau$ with $k-2\geq 0$ blowups over $A_0^{\tau}\cap H_1^{\tau}$, at the proper transforms of $H_{1}^{\tau}$, so that $A_0^{\gamma}\cap H_{1}^{\gamma}\not\subseteq \Bs\gamma_{+}^{-1}$. If $\psi=\gamma$ then $k\geq 3$ by Lemma \ref{lem:U-2}, hence $D$ is as in \ref{item:ht=3_nu=4_id} if $\delta=1$ and as in \ref{item:ht=3_nu=4_id_X} if $\delta=0$. Thus, we can assume that $\gamma_{+}^{-1}$ has a base point $r$. If $r\in A_{0}$ then by the definition of $\gamma$ and by Lemma \ref{lem:w=2_uniqueness}\ref{item:w=2_D} we have $\{r\}=A_{0}^{\gamma}\cap H_{1}^{\gamma}$, so $\beta_{D}(H_1)>3$, which is impossible. Thus $A_{0}\cap \Bs\gamma_{+}^{-1}=\emptyset$. Let $\eta$ be a composition of $\gamma$ with a blowup at $r\in \Bs\gamma_{+}^{-1}$, and if $\psi\neq\eta$ let $\upsilon$ be a composition of $\eta$ with a blowup at a given $s\in \Bs\eta_{+}^{-1}$. 
	
	Assume first that $r\in A_{j}$ for some $j\in \{2,3,4\}$, say $r\in A_2$. By Lemma \ref{lem:w=2_uniqueness}\ref{item:w=2_D} we have three cases: $r\in G_2$, $r\in L_2$ or $r\in H_2$. Note that to get a contradiction with inequality \eqref{eq:ld-bound_h=2}, it is enough to show it for some minimal value of $k$, since we can reduce $k$ by vertical swaps over $p_0$.
	
	Consider the case $\{r\}=A_2\cap G_2$. For $k=2$ we have $U_{\eta}=\langle \bs{2};[2],[2],[(2)_{\delta},\ub{4-\delta},2,2]\rangle$, so $\ld(H_{1}^{\eta})+2\ld(H_{2}^{\eta})=1$. Thus \eqref{eq:ld-bound_h=2} fails for $k=2$, hence it fails for any $k\geq 2$, too, a contradiction.
	
	Consider the case  $\{r\}=A_2\cap H_2$. If $k=3$ then $U_{\eta}=\langle \bs{3};[2],[2],[2,2,2]\rangle$, so $\ld(H_{1}^{\eta})=\tfrac{1}{5}$, which by \eqref{eq:ld-bound_h=2} gives $\ld(H_{2}^{\eta})>\tfrac{2}{5}$. But $W_{\eta}=[2,2,\ub{4}]$ or $[\ub{5}]$, so $\ld(H_{2}^{\eta})=\tfrac{2}{5}$, a contradiction. As before, we conclude that $k=2$. Lemma \ref{lem:U-2} implies that $\eta_{+}^{-1}$ has a base point $s\in A_{j}^{\eta}$ for some $j\in \{2,3,4\}$. Suppose $s\not \in A_{2}$. Then, say, $s\in A_{3}$, and since $U$ is an admissible fork, we conclude that $\{s\}=A_{3}\cap L_{3}$. After a blowup at $s$ we get  $U_{\upsilon}=\langle \bs{2};[2,2,2],[3],[2]\rangle$, so $\ld(H_{2}^{\upsilon})>\tfrac{1}{2}(1-\ld(H_{1}^{\upsilon}))=\tfrac{2}{5}$, and $W_{\upsilon}=[(2)_{\delta},\ub{5-\delta},2,2]$, so $\ld(H^{\upsilon}_{2})=\tfrac{5}{17}$ or $\tfrac{4}{11}$; a contradiction. Therefore, $s\in A_{2}$. Since $H_{1}$ is the only branching component of $U$, we have $s\in U_{\eta}$. Now $U_{\upsilon}=\langle \bs{2};[2,3,2],[2],[2]\rangle$, $W_{\upsilon}=[(2)_{\delta},\bs{5-\delta},2]$, so $\ld(H_{1}^{\upsilon})+2\ld(H_{2}^{\upsilon})=1$; a contradiction with \eqref{eq:ld-bound_h=2}.
	
	Eventually, consider the case $\{r\}=A_2\cap L_2$. We have $U_{\eta}=\langle \bs{k};[3],[2],[2]\rangle$ and  $W_{\eta}=[(2)_{(k-1)\delta},\ub{4-\delta},2,2]$. If $k=3$ then $\ld(H_{1}^{\eta})+2\ld(H_2^{\eta})=1$, so from \eqref{eq:ld-bound_h=2} we infer that $k=2$. If $\psi=\eta$ then $D$ is as in \ref{item:ht=3_nu=4_id_C} if $\delta=1$ and as in \ref{item:ht=3_nu=4_id_XC} if $\delta=0$. Suppose $\psi\neq \eta$, and as before fix $s\in \Bs\eta_{+}^{-1}$. If $s\in A_{j}$ for some $j\in \{3,4\}$ then interchanging $\ll_2$ with $\ll_j$ if necessary we can assume $s\in L_3$. Now $U_{\upsilon}=\langle \bs{2};[3],[3],[2]\rangle$, so $\ld(H_{1}^{\upsilon})=\tfrac{1}{5}$ and thus $\ld(H_{2}^{\upsilon})>\frac{2}{5}$ by \eqref{eq:ld-bound_h=2}, but $W_{\upsilon}=\langle \ub{3},[2],[2,2],[2,2]\rangle$ or $[2,2,\ub{4},2,2]$, so $\ld(H_{2}^{\upsilon})=\frac{1}{7}$ or $\frac{1}{4}$, a contradiction. Assume $s\in A_{2}$. Then $U_{\upsilon}=\langle \bs{2};[4],[2],[2]\rangle$ if $s\in L_{2}$ and $U_{\upsilon}=\langle \bs{2};[2,3],[2],[2]\rangle$ otherwise; hence $\ld(H_{1}^{\upsilon})=\tfrac{1}{3}$, so $\ld(H_{2}^{\upsilon})>\tfrac{1}{3}$ by \eqref{eq:ld-bound_h=2}. But $W_{\upsilon}=\langle 2;[2],[2],[(2)_{\delta},\ub{4-\delta}]\rangle$ or $[(2)_{\delta},\ub{4-\delta},3,2]$, so $\ld(H_{2}^{\upsilon})=\tfrac{1}{3}$; a contradiction. We conclude that $\delta=0$ and $\Bs\eta_{+}^{-1}\subseteq A_{1}$. Now $U_{\upsilon}=\langle \bs{2};[3],[2],[2]\rangle$, so $\ld(H_{1}^{\upsilon})=\tfrac{1}{2}$ and $\ld(H_{2}^{\upsilon})>\tfrac{1}{4}$ by \eqref{eq:ld-bound_h=2}. Recall from Lemma \ref{lem:w=2_uniqueness}\ref{item:w=2_D} that $s\in D_{\eta}$. In fact, $s\in H_{2}$, since otherwise $W_{\upsilon}=[2,2,\ub{4},2,2]$ or $[3,2,\ub{4},2,2]$, so $\ld(H_{2}^{\upsilon})\leq \tfrac{1}{4}$, which is impossible. If $\psi=\upsilon$ then $D$ is as in \ref{item:ht=3_nu=4_XC}. Otherwise, after one further blowup, $W$ becomes $[\ub{6},2,2]$ or $[2,\ub{5},2,2]$, so $\ld(H_{2}^{\upsilon})=\tfrac{1}{4}$ or $\tfrac{5}{23}$; a contradiction.
	\smallskip
	
	It remains to consider the case $\Bs\gamma_{+}^{-1}\subseteq A_{1}$. Since $\psi\neq \gamma$, we have  $\delta=0$, see Figure \ref{fig:swap-to-nu=4_large}. We have $U=\langle \bs{k};[2],[2],[2]\rangle$, so $k\geq 3$ by Lemma \ref{lem:U-2}, and $\ld(H_{2})>\tfrac{1}{2}(1-\ld(H_1))=\tfrac{k-2}{2k-3}\geq \tfrac{1}{3}$. If $r\in G_1$ then $W_{\eta}=[(2)_{k-2},3,2,\ub{4}]$, so for $k=3$ we have $\ld(H_{2}^{\eta})=\tfrac{1}{3}$; a contradiction. Hence $r\in H_2$ or $L_1$, so $W_{\eta}=[\ub{5}]$ or $[2,2,\ub{4}]$ and thus $\ld(H_{2}^{\upsilon})=\tfrac{2}{5}$. The inequality $\tfrac{k-2}{2k-3}<\ld(H_2^{\upsilon})=\frac{2}{5}$ gives $k=3$. If $\psi=\upsilon$ then $D$ is as in \ref{item:ht=3_nu=4_V} or \ref{item:ht=3_nu=4_U}. Suppose $\psi\neq \upsilon$. If $r\in H_2$ then $W_{\upsilon}=[\ub{6}]$ or $[2,\ub{5}]$, so $\ld(H_{2}^{\upsilon})=\tfrac{1}{3}$, which is false. Thus $r\in L_1$, $W_{\upsilon}=[2,3,\ub{4}]$ or $\langle 2;[2],[2],[\ub{4}]\rangle$, and again $\ld(H_{2}^{\upsilon})=\tfrac{1}{3}$; a contradiction.
	
	\paragraph{Case $\nu=3$, see Lemma \ref{lem:w=2_swaps}\ref{item:w=2_cha=2_nu=3} and Figure \ref{fig:swap-to-nu=3}.} By Lemma \ref{lem:further-swaps}, the map $\phi_{+}^{-1}$ has base points $q_1\in A_{1}^{\phi}$ and $q_2\in A_{2}^{\phi}$. Let $\tau$ be a composition of $\phi$ with a blowup at $\{q_1,q_2\}$.
	
	As a warm-up, we consider the case with moduli, i.e.\ when for some decomposition $\psi=\gamma\circ \gamma_{+}$ as in Notation \ref{not:phi_+_h=2}, the morphism $\gamma_{+}$ has a base point $r\not\in D_{\gamma}$, see Lemma \ref{lem:w=2_uniqueness}\ref{item:w=2_D}. Assume first that $r\in A_{2}^{\gamma}$. Then the admissibility of $U$ implies that $r=q_2$, i.e.\ $r$ is the center of the first blowup over $p_{2}'$. Again, from the admissibility of $U$ we see that $\{q_1\}=A_1\cap H_2$ and $A_3\cap \Bs\phi_{+}^{-1}=\emptyset$. Thus $U_{\tau}=\langle 2;[\ub{3}],[2],[2,\bs{2},2]\rangle$. Applying once again the fact that $U$ is an admissible fork, we infer that  $\tau_{+}^{-1}$ is a composition of some $k-2\geq 0$ blowups over $A_2\cap U_{\tau}$, on the proper transforms of $U_{\tau}$. Thus case \ref{item:ht=3_nu_3_off} holds.
	
	Assume that $r\not\in A_{2}^{\gamma}$. Interchanging $\ll_2$ with $\ll_3$, if needed, we can assume $r\not\in A_{3}^{\gamma}$, too, so $r\in A_{0}^{\gamma}$ or $r\in A_{1}^{\gamma}$. In fact, $r\in A_{1}^{\gamma}$ since $U$ is admissible. In turn, the admissibility of $W$ implies that $r$ is the center of the first blowup over $p_{1}'$, i.e.\ $r=q_1$; and $\{q_2\}=A_{2}^{\phi}\cap L_{2}^{\phi}$ or $A_{2}^{\phi}\cap H_{2}^{\phi}$. In the first case, the admissibility of $U$ implies that $\tau_{+}^{-1}$ is a composition of blowups over $r$, so $D$ is of the same (non-decorated) type as in case \ref{item:ht=3_nu_3_off} found above. Applying the automorphism $(A_0,A_1,A_2)\mapsto (A_1,A_2,A_0)$ from Remark \ref{rem:7A1}, we get back the case \ref{item:ht=3_nu_3_off}. 
	
	Thus we can assume $\{q_2\}=A_{2}^{\phi}\cap H_{2}^{\phi}$. Then $U_{\tau}$ is a $(-2)$-chain, so by Lemma \ref{lem:U-2}, $\tau_{+}^{-1}$ has a base point $s$ on $A_{0}^{\tau}$, $A_{2}^{\tau}$ or $A_{3}^{\tau}$. Let $\eta$ be a composition of $\tau$ with a blowup at $s$. If $s\in A_{3}^{\tau}$ then it is easy to see that the admissibility of $W$ forces $U$ to be a $(-2)$-chain or fork, contrary to Lemma \ref{lem:U-2}. Thus $A_{3}^{\tau}\cap \Bs\tau_{+}^{-1}=\emptyset$. Applying the automorphism $(A_0,A_1,A_2)\mapsto (A_2,A_1,A_0)$ from Remark \ref{rem:7A1}, we can assume that $s\in A_{0}^{\tau}$. If $s\in L_{1}^{\tau}$ then as before we get a contradiction with Lemma \ref{lem:U-2}, so $s\in H_{1}^{\tau}$. Since $W$ is admissible, the morphism $\eta_{+}$ is a composition of blowups over $r$, on the proper transform of $W_{\tau}$. Thus case \ref{item:ht=3_nu_3_V_fork_off} holds.
	\smallskip
	
	In the remaining part of the proof, we can and will assume that for any decomposition $\psi=\gamma\circ\gamma_{+}$ as in Notation \ref{not:phi_+_h=2}, the base locus of $\gamma_{+}^{-1}$ is contained in $D_{\gamma}$. Like in case $\nu=4$ or in the proofs of Lemmas \ref{lem:w=3} and \ref{lem:w=2_cha_neq_2}, we will use this assumption without further comment. We now split the proof into cases according to the placement of the base points $q_j\in A_j^{\phi}$ from Lemma \ref{lem:further-swaps}. 
	
	\begin{casesp}	
	\litem{$q_2\in H_2$}\label{case:qH2}
	
	\begin{casesp}
	\litem{The map $\tau_{+}^{-1}$ has a base point $q_3\in A_3$} Let $\sigma$ be a composition of $\tau$ with a blowup at $q_3$. 
	
	\begin{casesp}	
	\litem{$q_3\in L_3$} Recall that we have a base point $q_1\in A_1$, lying on $L_1$, $G_1$ or $H_2$. 
	
	\begin{casesp}
	\litem{$q_1\in L_1\cup G_{1}$}
	Note that the weighted graphs of $U,W$ obtained in case $q_1\in L_1$ are weighted subgraphs of those for $q_1\in G_{1}$, so by Lemma \ref{lem:Alexeev}, if \eqref{eq:ld-bound_h=2} fails for the former, it fails for the latter, too. 
	
	Assume first that $\psi\neq \sigma$, and let $\eta$ be a composition of $\sigma$ with a blowup at some $r\in \Bs\sigma_{+}^{-1}$. 
	
	Suppose $r\in A_{2}$. Then for $q_1\in L_{1}$ we get $U_{\eta}+W_{\eta}=\langle 2;[2],[2],[3,\bs{2},2]\rangle+[2,2,\ub{4},2,2]$ if $r\in H_2$ and $[3,\bs{2},2,3,2]+\langle \ub{3};[2],[2,2],[2,2]\rangle$ otherwise; hence $\ld(H_{1}^{\eta})+2\ld(H_{2}^{\eta})=1$ or $\tfrac{127}{175}$. Thus inequality \eqref{eq:ld-bound_h=2} fails in this case, so by our observation it fails in case $q_1\in G_1$, too; a contradiction. 
	
	Similarly, if $\{r\}=A_0\cap L_1$ or $A_3\cap L_3$ then for $q_1\in L_1$ we get $U_{\eta}+W_{\eta}=\langle \bs{2};[2],[3],[2,2,2]\rangle+[2,2,\ub{3},2,2]$ or $[4,\bs{2},2,2,2]+\langle 2;[2],[2],[2,2,\ub{3}]\rangle$, respectively, so $\ld(H_{1}^{\eta})+2\ld(H_{2}^{\eta})=1$. As before, it follows that inequality \eqref{eq:ld-bound_h=2} fails both for $q_1\in L_1$ and $q_1\in G_1$; a contradiction.
	
	Assume $\{r\}=A_0\cap H_1$. Then $U_{\eta}=[3,\bs{3},2,2,2]$, so $\ld(H_{2})>\tfrac{1}{2}(1-\ld(H_{1}^{\eta}))=\tfrac{8}{23}$. If $q_1\in G_1$ then $W_{\eta}=[2,3,2,\ub{3},2,2]$, so $\frac{8}{23}<\ld(H_{2}^{\eta})=\tfrac{11}{41}$, which is impossible. Thus $q_1\in L_{1}$. If $\psi=\eta$ then $D$ is as in \ref{item:ht=3_nu=3_XY'E} with $k=2$. Suppose $\psi\neq \eta$, and let $\theta$ be a composition of $\eta$ with a blowup at $s\in \Bs\eta_{+}^{-1}$. We know that $s\not\in A_2$ and $\{s\}\neq A_3\cap L_3$. If $s\in W_{\eta}$ then $W_{\theta}=[2,3,\ub{3},2,2]$, so $\frac{8}{23}<\ld(H_{2}^{\theta})=\tfrac{8}{29}$; a contradiction. If $s\in A_1$ then, since $s\not\in W_{\eta}$, we get $W_{\theta}=\langle 2;[2],[2],[2,2,\ub{3}]\rangle$, so $\frac{8}{23}<\ld(H_{2}^{\theta})=\tfrac{1}{4}$; a contradiction. Thus $s\in A_{0}$, hence  $\ld(H_{2}^{\theta})=\ld(H_{2}^{\eta})=\tfrac{2}{5}$ and therefore $\ld(H_{1}^{\theta})>\tfrac{1}{5}$ by \eqref{eq:ld-bound_h=2}. But $U_{\theta}=[3,\bs{4},2,2,2]$ or $\langle \bs{3};[2],[3],[2,2,2]\rangle$, so $\ld(H_{1}^{\theta})=\tfrac{1}{5}$ or $\frac{1}{17}$; a contradiction.
	
	Assume $\{r\}=A_3\cap W_{\sigma}$. Then $U_{\eta}=[2,3,\bs{2},2,2,2]$, hence $\ld(H_{2})>\tfrac{1}{2}(1-\ld(H_{1}^{\eta}))=\tfrac{4}{17}$. If $q_1\in G_1$ then $W_{\eta}=[3,2,\ub{3},3,2]$, so $\ld(H_{1}^{\eta})=\tfrac{1}{5}<\frac{4}{17}$, a contradiction. Thus $q_1\in L_1$. If $\psi=\eta$ then $D$ is as in \ref{item:ht=3_nu=3_XY'E} with $k=3$. Suppose $\psi\neq \eta$, and as before let $\theta$ be a composition of $\eta$ with a blowup at $s\in \Bs\eta_{+}^{-1}$. We know that $s\in A_1\cup A_3$. If $s\in A_{1}$ then $W_{\theta}=[2,3,\ub{3},3,2]$ if $s\in W_{\eta}$ and $\langle 2;[2],[2],[2,3,\ub{3}]\rangle$ if $s\in L_{1}$, so $\ld(H_{2}^{\theta})=\tfrac{2}{11}$ or $\tfrac{1}{8}$; a contradiction. If $s\in A_3$ then similarly $W_{\theta}=[2,2,\ub{3},4,2]$ or $\langle 3;[2],[2],[2,2,\ub{3}]\rangle$, so $\ld(H_{2}^{\theta})=\tfrac{10}{43}$ or $\tfrac{1}{11}$; a contradiction.
	
	Therefore, we can assume $\Bs\sigma_{+}^{-1}\subseteq A_{1}$. Now $U=[3,\bs{2},2,2,2]$, so $\ld(H_{2})>\tfrac{1}{2}(1-\ld(H_{1}))=\tfrac{2}{11}$. Let $\theta$ be a composition of $\sigma$ with $k-2\geq 0$ blowups at $A_{1}\cap W_{\sigma}$ and its infinitely near points on the proper transforms of $W_{\sigma}$, so that $A_1\cap W_{\theta}\not\subseteq \Bs\theta_{+}^{-1}$. Now $W_{\theta}=[2+\epsilon,k,\ub{3},2,2]$, where $\epsilon=0$ if $q_1\in L_1$ and $\epsilon=1$ if $q_1\in G_1$. We compute that $\ld(H_{1}^{\theta})>\tfrac{2}{11}$ if and only if $k\leq 7-2\epsilon$, so if $\psi=\theta$ then $D$ is as in \ref{item:ht=3_nu=3_XY'U_eps=0} if $\epsilon=0$ and \ref{item:ht=3_nu=3_XY'U_eps=1} if $\epsilon=1$. Assume $\psi\neq \theta$. Since $\Bs\sigma_{+}^{-1}\subseteq A_1$, definition of $\theta$ shows that the unique base point of $\theta_{+}^{-1}$ is $A_1\cap (D_{\theta}-W_{\theta})$. Since $W$ is admissible, $\theta_{+}$ is a single blowup, and $q_1\in L_1$. Now $W=\langle k;[2],[2],[2,2,\ub{3}]\rangle$, so the inequality $\frac{2}{11}<\ld(H_2)=\frac{1}{7k-10}$ gives $k=2$, and $D$ is as in  \ref{item:ht=3_nu=3_XY'UV}.
	
	\litem{$q_1\in H_2$}
	We have $U_{\sigma}=[3,\bs{2},2,2,2]$, $W_{\sigma}=[\ub{4},2,2]$. Inequality \eqref{eq:ld-bound_h=2} gives $\ld(H_1)>1-2\ld(H_2^{\sigma})=\frac{1}{5}$. For a given base point $r\in \Bs\sigma_{+}^{-1}$ we denote by $\eta$ a composition of $\sigma$ with a blowup at $r$; and for a given $s\in \Bs\eta_{+}^{-1}$ we denote by $\theta$ a composition of $\eta$ with a blowup at $s$.
	
	Consider the case $r\in A_{0}$. If $r\in L_1$ then $U_{\eta}=\langle \bs{2},[2],[3],[2,2,2]\rangle$, so $\ld(H_1)=\frac{1}{5}$, which is false. Thus $\{r\}=A_0\cap H_1$. If $\psi=\eta$ then $D$ is as in \ref{item:ht=3_nu=3_XY'BE}. Suppose $\psi\neq \eta$, and fix $s\in \Bs\eta_{+}^{-1}$. If $s\in A_0$ then $U_{\theta}=[3,\bs{4},2,2,2]$ if $s\in H_1$ and $\langle \bs{3},[2],[3],[2,2,2]\rangle$  otherwise, so $\ld(H_{1}^{\theta})=\frac{1}{5}$ or $\frac{1}{17}$, which is false. Thus $s\not\in A_0$. Recall that $U_{\eta}=[3,\bs{3},2,2,2]$, so $\ld(H_2)>\frac{1}{2}(1-\ld(H_1^{\eta}))=\frac{8}{23}$. If $s\in A_1\cup A_2$ then $W_{\theta}=[\ub{5},2,2]$ or $[2,\ub{4},2,2]$, so $\ld(H_{2}^{\eta})=\frac{4}{13}$ or $\frac{5}{17}$, in any case $\ld(H_2^{\theta})<\frac{8}{23}$, which is false. Thus $s\in A_{3}$, so $W_{\theta}=\langle 2;[\ub{4}],[2],[2]\rangle$ if $s\in U_{\eta}$ and $[\ub{4},3,2]$ if $s\in W_{\eta}$, so $\ld(H_{2}^{\theta})=\frac{1}{3}<\frac{8}{23}$; a contradiction.
	
	Thus we can assume $\Bs\sigma_{+}^{-1}\cap A_{0}=\emptyset$. Consider the case $\{r\}=A_2\cap W_{\tau}$. If $\psi=\eta$ then $D$ is as in \ref{item:ht=3_nu=3_XY'BA}. Suppose $\psi\neq \eta$, so $\eta_{+}^{-1}$ has a base point $s\not\in A_0$. Because $U$ is admissible, $s\not\in A_{2}$. If $s\in A_{3}$ then $U_{\eta}+W_{\eta}=\langle 2;[2],[2],[4,\bs{2},2]\rangle+\langle 2;[2],[2],[\bs{5}]\rangle$ if $s\in U_{\eta}$ and $\langle 2;[2],[2],[2,3,\bs{2},2]\rangle+[\ub{5},3,2]$, so $\ld(H_{1}^{\eta})+2\ld(H_{2}^{\eta})=\tfrac{5}{6}$ or $\tfrac{59}{69}$, contrary to \eqref{eq:ld-bound_h=2}. Thus $s\in A_{1}$. We have $\ld(H_{1}^{\eta})=\ld(H_{1}^{\eta})=\tfrac{1}{2}$, so $\ld(H_{2}^{\eta})>\tfrac{1}{4}$ by \eqref{eq:ld-bound_h=2}. But $W_{\eta}=[\ub{6},2,2]$ if $s\in W_{\eta}$ and $[2,\ub{5},2,2]$ otherwise, so $\ld(H_{2}^{\eta})=\tfrac{1}{4}$ or $\tfrac{5}{23}$; a contradiction. 
	
	Consider the case $\{r\}=A_2\cap U_{\tau}$. Then $U_{\eta}=[3,\bs{2},2,3,2]$, so $\ld(H_2)>\frac{1}{2}(1-\ld(H_1^{\eta}))=\frac{7}{25}$. If $\psi=\eta$ then $D$ is as in \ref{item:ht=3_nu=3_XY'BB}. 	
	Suppose $\eta_{+}^{-1}$ has a base point $s$. Recall that $s\not\in A_0$. If $s\in A_1$ then $W_{\eta}=[2,\ub{5},2,2]$ if $s\in W_{\eta}$ and $\langle \ub{4};[2],[2],[2]\rangle$ otherwise, so $\ld(H_{2}^{\eta})=\tfrac{5}{23}<\frac{7}{25}$ or $\tfrac{1}{5}<\frac{7}{25}$, which is impossible. If $s\in A_3$ then $W_{\eta}=[2,\ub{4},3,2]$ if $s\in W_{\eta}$ and $\langle 2;[2],[2],[2,\ub{4}]\rangle$ otherwise, so $\ld(H_2^{\eta})=\frac{7}{31}<\frac{7}{25}$ or $\frac{1}{5}$, which, again, is impossible. Thus $s\in A_3$, so $W_{\eta}= [3,\ub{4},2,2]$ if $s\in W_{\eta}$ and $[2,2,\ub{4},2,2]$ otherwise, so $\ld(H_2^{\eta})=\frac{2}{9}$ or $\frac{1}{4}$; a contradiction since $\ld(H_2^{\eta})>\frac{7}{25}$.
	
	Therefore, we can assume $\Bs\sigma_{+}^{-1}\subseteq A_{1}\cup A_{3}$. Consider the case $\{r\}=A_{3}\cap U_{\tau}$. Then $U_{\eta}=[4,\bs{2},2,2,2]$, so $\ld(H_2)>\frac{1}{2}(1-\ld(H_{1}^{\eta}))=\frac{1}{4}$. If $\psi=\eta$ then $D$ is as in \ref{item:ht=3_nu=3_XY'BC}. Suppose $\psi\neq \eta$, so $\eta_{+}^{-1}$ has a base point $s\in A_1\cup A_3$. If $s\in A_1$ then $W_{\eta}=\langle 2;[2],[2],[\ub{5}]\rangle$ if $s\in H_2$ and $\langle 2;[2],[2],[2,\ub{4}]\rangle$ otherwise, so $\ld(H_{2}^{\eta})=\tfrac{1}{4}$ or $\tfrac{1}{5}<\frac{1}{4}$, which is impossible. Thus $s\in A_{3}$. Now $U_{\eta}+W_{\eta}=[5,\bs{2},2,2,2]+\langle 2;[2],[2,2],[\ub{4}]\rangle$ if $s\in U_{\eta}$ and $[2,4,\bs{2},2,2,2]+\langle 2;[2],[3],[\ub{4}]\rangle$ otherwise, so $\ld(H_{1}^{\eta})+2\ld(H_{2}^{\eta})=1$ or $\tfrac{283}{297}$; a contradiction with \eqref{eq:ld-bound_h=2}. 
	
	Consider the case $\{r\}=A_{3}\cap W_{\tau}$. Let $\upsilon$ be the composition of $\eta$ with $k-2\geq 1$ blowups over $r$, on proper transforms of $W_{\eta}$, so that $A_{3}\cap W_{\upsilon}\not\subseteq \Bs\upsilon_{+}^{-1}$. Then $U_{\upsilon}+W_{\upsilon}=[(2)_{k-2},3,\bs{2},2,2,2]+[\ub{4},k,2]$. Inequality  \eqref{eq:ld-bound_h=2} gives $k\leq 5$. If $\psi=\upsilon$ then $D$ is as in \ref{item:ht=3_nu=3_XY'BD}. Suppose $\psi\neq \upsilon$, and let $\gamma$ be a composition of $\upsilon$ with a blowup at $z\in \Bs\upsilon_{+}^{-1}$. If $z\in A_3$ then by the definition of $\upsilon$ we have $\{z\}=A_3\cap U_{\upsilon}$, so $U_{\gamma}+W_{\gamma}=[3,(2)_{k-3},3,\bs{2},2,2,2]+\langle k;[2],[2],[\ub{4}]\rangle$ which yields a contradiction with \eqref{eq:ld-bound_h=2}. Thus $z\in A_1$. For $k=3$ we have $\ld(H_1^{\eta})=\frac{9}{17}$, so in any case $\ld(H_2)>\frac{4}{17}$. 
	
	If $z\not\in H_2$ then $W_{\gamma}=[2,\ub{4},k,2]$, so  $\ld(H_{2}^{\gamma})\leq \frac{7}{31}<\frac{4}{17}$, a contradiction. Thus $\{z\}=A_1\cap H_2$. Inequality \eqref{eq:ld-bound_h=2} gives $k=3$. If $\psi=\gamma$ then $D$ is as in \ref{item:miss}. If $\psi\neq \gamma$ then after one further blowup over $\Bs\gamma_{+}^{-1}\subseteq A_1$ we get  $W=[\ub{6},3,2]$ or $[2,\ub{5},3,2]$, so $\ld(H_2)=\frac{3}{14}$ or $\frac{7}{41}$, a contradiction since $\ld(H_2)>\frac{4}{17}$.
	
	Eventually, we can assume $\Bs\sigma_{+}^{-1}\subseteq A_{1}$. Then $U=[3,\bs{2},2,2,2]$, so $\ld(H_2)>\frac{1}{2}(1-\ld(H_1))=\frac{2}{11}$. Now $D-U=T^{*}*[(2)_{k-2},\ub{4},2,2]+\langle k;T\trp,[3],[2]\rangle$ or $[(2)_{k-2},\ub{4},2,2]+[3,k,2]$ for some $k\geq 2$ and an admissible chain $T$; we put $T=0$ in the latter case. The inequality $\ld(H_2)>\frac{2}{11}$ implies that one of the following holds. 
	
	If $T=0$ then $k\leq 9$, so $D$ is as in \ref{item:ht=3_nu=3_XY'B_T=0}. Assume $T\neq 0$. If $k=2$ then $T=[(2)_{l}]$, $l\in \{1,2,3\}$ or $T=[3]$, so $D$ is as in \ref{item:ht=3_nu=3_XY'B_k=2} or \ref{item:ht=3_nu=3_XY'B_k=2_T=[3]}, respectively. If $k=3$ then $T=[2]$ or $[2,2]$, so $D$ is as in \ref{item:ht=3_nu=3_XY'B_k=3}. Otherwise, $k=4$, $T=[2]$, and $D$ is as in \ref{item:ht=3_nu=3_XY'B_k=4}.
	\end{casesp}

	\litem{$q_3\in G_3$}
	If $q_1\in G_1$ then $U_{\sigma}=[3,2,\ub{3},2,2,\bs{2},2,2,2]$, so $\ld(H_{1}^{\sigma})+2\ld(H_{2}^{\sigma})=1$, contrary to \eqref{eq:ld-bound_h=2}. Assume $q_1\in L_1$. If $\psi=\sigma$ then $D$ is as in \ref{item:ht=3_nu=3_ZY'A}. Suppose $\psi\neq\sigma$, and let $\eta$ be a composition of $\sigma$ with a blowup at $r\in \Bs\sigma_{+}^{-1}$. Since $U$ is admissible, we have $\{r\}=A_2\cap H_2$, $A_3\cap U_{\sigma}$, $A_1\cap L_1$, $A_1\cap U_{\sigma}$ or $A_0\cap H_1$. Then $U_{\eta}=
	\langle 2;[2],[2],[2,2,\ub{4},2,2,\bs{2},2]\rangle$,
	$[2,2,\ub{3},3,2,\bs{2},2,2,2]$, 
	$\langle 2;[2],[2],[2,2,2,\bs{2},2,2,\ub{3}]\rangle$, 
	$[2,3,\ub{3},2,2,\bs{2},2,2,2]$, or 
	$[2,2,\ub{3},2,2,\bs{3},2,2,2]$; so $\ld(H_{1}^{\eta})+2\ld(H_{2}^{\eta})=
	\tfrac{3}{7}$, 
	$\tfrac{65}{73}$, 
	$\tfrac{3}{4}$, 
	$\tfrac{57}{61}$ or 
	$\tfrac{9}{13}$; a contradiction with \eqref{eq:ld-bound_h=2}.
	
	Thus we can assume $q_1\in H_2$. If $\psi=\sigma$ then $D$ is as in  \ref{item:ht=3_nu=3_ZY'B} with $k=2$. Assume $\psi\neq\sigma$, and as before let $\eta$ be a composition of $\sigma$ with a blowup at $r\in\Bs  \sigma_{+}^{-1}$. If $r\in A_{0}$ then $r\in H_1$ since $U$ is admissible, so $U_{\eta}=[\ub{4},2,2,\bs{3},2,2,2]$ and $\ld(H_{1}^{\eta})+2\ld(H_{2}^{\eta})=\tfrac{27}{31}$, contrary to \eqref{eq:ld-bound_h=2}. If $r\in A_{2}$ then $U_{\eta}=\langle 2;[2],[2],[\ub{5},2,2,\bs{2},2]\rangle$ if $r\in H_2$ and $[2,\ub{4},2,2,\bs{2},2,3,2]$ otherwise, so $\ld(H_{1}^{\eta})+2\ld(H_{2}^{\eta})=\tfrac{3}{4}$ or $\tfrac{9}{13}$, again contrary to \eqref{eq:ld-bound_h=2}. 
	
	Thus $\Bs\sigma_{+}^{-1}\subseteq A_1\cup A_3$. Assume $r\in A_{3}$. Since $U$ is admissible we have $r\in U_{\sigma}$. If $\psi=\eta$ then $D$ is as in \ref{item:ht=3_nu=3_ZY'B} with $k=3$. In turn, if $\psi\neq \eta$ then after one blowup at some $s\in \Bs\eta_{+}^{-1}\subseteq A_1\cup A_3$, $U$ becomes $[\ub{4},4,2,\bs{2},2,2,2]$ if $s\in A_3$, and $[2,\ub{4},3,2,\bs{2},2,2,2]$ or $[\ub{5},3,2,\bs{2},2,2,2]$ if $s\in A_1$, so $\ld(H_1)+2\ld(H_2)=1$, $\tfrac{65}{79}$ or $\tfrac{55}{59}$, contrary to \eqref{eq:ld-bound_h=2}. Thus we can assume $\Bs\sigma_{+}^{-1}\subseteq A_{1}$. If $\psi=\eta$ then $D$ is as in \ref{item:ht=3_nu=3_ZY'BE} if $r\in H_2$ or \ref{item:ht=3_nu=3_ZY'BD} otherwise. After one further blowup, $U$ becomes $[\ub{6},2,2,\bs{2},2,2,2]$ or $[2,\ub{5},2,2,\bs{2},2,2,2]$ in the first case, and $[3,\ub{4},2,2,\bs{2},2,2,2]$ or 	$[2,2,\ub{4},2,2,\bs{2},2,2,2]$ in the second case, so $\ld(H_1)+2\ld(H_2)=1$, $\tfrac{15}{17}$, $\tfrac{51}{59}$ or $\tfrac{45}{52}$, contrary to \eqref{eq:ld-bound_h=2}.

	\litem{$q_3\in H_2$}
	Suppose that for both $j\in \{2,3\}$, the curve $A_{j}$ contains a base point of $\sigma_{+}^{-1}$, say $s_j$. Let $\upsilon$ be a composition of $\phi$ with a blowup at $q_2,s_2,q_3,s_3$ (so we do not blow up at $q_1$ yet). Since $U_{\upsilon}$ has at most one branching component, we have, say, $s_2\in U_{\sigma}$.  Now $U_{\upsilon}+W_{\upsilon}=[2,3,2,\bs{2},2,3,2]+[\ub{6}]$ if $s_3\in U_{\sigma}$ and $\langle 2;[2],[2],[2,3,2,\bs{2},2]\rangle+[\ub{5},2]$ if $s_3\in H_2$, so in either case $\ld(H_1^{\upsilon})+2\ld(H_{2}^{\upsilon})=1$, a contradiction with \eqref{eq:ld-bound_h=2}. Thus interchanging $\ll_2$ with $\ll_3$ if necessary, we can assume $A_3\cap \Bs\sigma_{+}^{-1}=\emptyset$. Since by Lemma \ref{lem:U-2} $U$ is not a $(-2)$-chain, the map $\sigma_{+}^{-1}$ has a base point $r\not\in A_1$, so $r\in A_0\cup A_2$.
	
	\begin{casesp}
	\litem{$r\in A_2$} Suppose $r\in H_2$. By Lemma \ref{lem:U-2} $U$ is not a $(-2)$-fork, so $\sigma_{+}^{-1}$ has another base point off $A_1$, and since $U$ is admissible, this base point is $A_0\cap H_1$. Like before, letting $\upsilon$ be a composition of $\phi$ with a blowup at $q_2$, $q_3$, $s$ and that point, we get $U_{\upsilon}+W_{\upsilon}=\langle 2;[2],[2],[2,2,2,\bs{3},2]\rangle+[\ub{5}]$, so $\ld(H_1^{\upsilon})+2\ld(H_{2}^{\upsilon})=1$, a contradiction with \eqref{eq:ld-bound_h=2}. 
	
	Thus $\{r\} = A_2\cap U_{\sigma}$. Let $\eta$ be a composition of $\sigma$ with a blowup at $r$, and for a given $s\in \Bs\eta_{+}^{-1}$, let $\theta$ be a composition of $\eta$ with a blowup at $s$. Suppose $s\in A_0$. Since $U$ is admissible, we have $s\in H_1$. Now $U_{\theta}=[2,3,2,\bs{3},2,2,2]$, so inequality \eqref{eq:ld-bound_h=2} gives $\ld(H_2)>\frac{1}{2}(1-\ld(H_1^{\theta}))=\frac{5}{13}$. But $W_{\theta}=[\ub{5},2]$ if $q_1\in H_2$; $[2,2,\ub{4},2]$ if $q_1\in L_1$ and $[2,3,2,\bs{4},2]$ if $q_1\in G_1$, so $\ld(H_2^{\theta})=\frac{1}{3}$, $\frac{5}{17}$ or $\frac{5}{23}$; a contradiction. Thus $\Bs\eta_{+}^{-1}\subseteq A_1\cup A_2$.
	
	Assume that $\eta_{+}^{-1}$ has a base point $s\in A_2$. If $s\in H_2$ then $U_{\theta}=\langle 3,[2],[2],[2,2,2,\bs{2},2]\rangle$, so \eqref{eq:ld-bound_h=2} gives $\ld(H_2)>\frac{1}{2}(1-\ld(H_1^{\eta}))=\frac{2}{7}$, but we have $W_{\theta}=[\ub{5},3]$, $[2,2,\ub{4},3]$ or $[3,2,\ub{4},3]$, so $\ld(H_2^{\theta})=\frac{2}{7}$, $\frac{6}{29}$ or $\frac{8}{49}$, a contradiction. Thus $\{s\}=A_2\cap U_{\eta}$. We have $U_{\theta}=[2,4,2,\bs{2},2,2,2]$, so $\ld(H_2)>\frac{1}{2}(1-\ld(H_1^{\theta}))=\frac{1}{4}$. If $q_1\in L_1$ or $G_1$ then $W_{\theta}$ is a chain $[2,2,\ub{4},2,2]$ or has this chain as a weighted subgraph, so $\ld(H_2)\leq \frac{1}{4}$, which is false. Thus $q_1\in H_2$. If $\theta_{+}^{-1}$ has a base point on $A_1$ then after blowing up at this point we get $W=[\ub{6},2,2]$ or $[2,\ub{5},2,2]$, so $\ld(H_2)=\frac{1}{4}$ or $\frac{5}{23}$, which is impossible as $\ld(H_2)>\frac{1}{4}$. Thus $\Bs\theta_{+}^{-1}\subseteq A_2$. Let $\upsilon$ be a composition of $\theta$ with $k-4\geq 0$ further blowups over $s$, on the proper transforms of $U_{\eta}$, so that $A_2\cap U_{\upsilon}\not\subseteq \Bs\upsilon_{+}^{-1}$. Now  $U_{\upsilon}=[2,k,2,\bs{2},2,2,2]$ and $W_{\upsilon}=[\ub{5},(2)_{k-2}]$, so inequality \eqref{eq:ld-bound_h=2} gives $k\leq 5$. If $\psi=\upsilon$ then $D$ is as in \ref{item:ht=3_nu=3_YYB}. Otherwise, by the definition of $\upsilon$ the unique base point of $\upsilon_{+}^{-1}$ is $A_2\cap W_{\upsilon}$, and after blowing up at that point we get $U+W=\langle k;[2],[2],[2,2,2,\bs{2},2]\rangle+[\ub{5},(2)_{k-3},3]$, which leads to a contradiction with inequality  \eqref{eq:ld-bound_h=2}. 
	
	It remains to consider the case $\Bs\eta_{+}^{-1}\subseteq A_1$. Now $U=[2,3,2,\bs{2},2,2,2]$, so $\ld(H_{2})>\tfrac{1}{2}(1-\ld(H_{1}))=\tfrac{1}{5}$ by \eqref{eq:ld-bound_h=2}. Consider the case $q_1\in G_1$. Then the admissibility of $W$ gives $\Bs\eta_{+}^{-1}\subseteq W_{\eta}$. If $\psi=\eta$ then $D$ is as in \ref{item:miss2}, otherwise after a blowup at the (unique) base point of $\eta_{+}^{-1}$ we get $W=[3,3,\bs{4},2]$, so $\ld(H_2)=\frac{1}{5}$, which is false.
	
	Consider the case $q_1\in L_1$. Let $\theta$ be a composition of $\eta$ with $k-2\geq 0$ blowups over $q_1$, on the proper transforms of $W_{\eta}$, so that $W_{\theta}\cap A_{1}^{\theta}\not\subseteq \Bs\theta_{+}^{-1}$. We have $W_{\theta}=[2,k,\ub{4},2]$, so the inequality $\frac{1}{5}<\ld(H_2^{\theta})=\frac{2k+1}{14k+11}$ gives $k\leq 3$. If $\psi=\theta$ then $D$ is as in \ref{item:ht=3_nu=3_YA}, otherwise after a blowup at the unique base point of $\theta_{+}^{-1}$ we get $W=\langle k;[2],[2],[2,\ub{4}]\rangle$, so $\ld(H_2)=\frac{k-1}{7k-9}\leq \frac{1}{5}$, a contradiction.
	
	Eventually, consider the case  $q_1\in H_{2}$. Let $\theta$ be a composition of $\eta$ with $k-2\geq 0$ blowups over $q_1$, on the proper transforms of $D_{\eta}-W_{\eta}$, so that $\Bs\theta_{+}^{-1}\subseteq W_{\theta}\cap A_{1}^{\theta}$. Then $W_{\theta}=[(2)_{k-2},\ub{5},2]$, so the inequality $\frac{1}{5}<\ld(H_{2}^{\theta})=\frac{k+1}{7k-5}$ gives $k\leq 4$. If $\psi=\theta$ then $D$ is as in \ref{item:ht=3_nu=3_YB_T=0}. Assume $\psi\neq \theta$, and let $\upsilon$ be the composition of $\theta$ with $l-5\geq 1$ blowups over $A_1\cap W_{\theta}$, on the proper transforms of $W_{\theta}$. Then $W_{\upsilon}=[\ub{l},2]$ if $k=2$, and $[l-2,(2)_{k-3},\ub{5},2]$ otherwise. The inequality $\ld(H_2)>\frac{1}{5}$ implies that the first case holds and $l\leq 7$. If $\psi=\upsilon$ then $D$ is as in \ref{item:ht=3_nu=3_YB_T-neq-0}; otherwise after one further blowup we get $W=[2,\ub{l},2]$, so $\ld(H_2)=\frac{1}{l-1}\leq \frac{1}{5}$, which is false.

	\litem{$r\in A_0$, $\Bs\sigma_{+}^{-1}\cap A_2=\emptyset$}
	Let $\eta$ be a composition of $\sigma$ with $k-2$ blowups over $A_0\cap H_1$, on the proper transforms of $H_1$, so that $A_0\cap H_1\not\subseteq \Bs\eta_{+}^{-1}$. Since $U$ is admissible, we get $A_0\cap \Bs\eta_{+}^{-1}=\emptyset$. Our assumptions imply that $r\in H_1^{\tau}$, so $k-2\geq 1$; and $\Bs\eta_{+}^{-1}\subseteq A_1$. 
	
	We have $U=[(2)_{3},\bs{k},(2)_{3}]$, so \eqref{eq:ld-bound_h=2} gives  $\ld(H_{2})>\frac{1}{2}(1-\ld(H_1^{\eta}))=\frac{k-2}{2k-3}\geq \frac{1}{3}$. If $q_1\in G_1$ then for $k=3$ we have $W_{\eta}=[2,3,2,\bs{4}]$, so $\ld(H_2)=\frac{1}{3}$, which in fact yields a contradiction for all $k$ (apply \eqref{eq:ld-bound_h=2} for $\gamma$ being a composition of $\tau$ with one blowup at $r$). Assume $q_1\in L_1$. Then $W_{\eta}=[2,2,\ub{4}]$, so the inequality $\frac{k-2}{2k-3}<\ld(H_{2})=\tfrac{2}{5}$ gives $k=3$. If $\psi=\eta$ then $D$ is as in \ref{item:ht=3_nu=3_UA}. If $\psi\neq \eta$ then after one blowup over $q_1$ we get $W=\langle 2;[2],[2],[\ub{4}]\rangle$ or $[2,3,\ub{4}]$, so $\ld(H_{2})=\tfrac{1}{3}$; a contradiction. 
	If $q_1\in H_1$ then $W_{\sigma}=[\ub{5}]$, so as before we get $k=3$; if $\psi=\eta$ then $D$ is as in \ref{item:ht=3_nu=3_UB}, otherwise after one blowup $W=[\ub{6}]$ or $[2,\ub{5}]$, so $\ld(H_2)=\tfrac{1}{3}$; a contradiction.	
\end{casesp}
\end{casesp}

	\litem{$\Bs\tau_{+}^{-1}\cap A_{3}=\emptyset$} Let $\gamma$ be a composition of $\phi$ with a blowup at $q_2$. 
	Then $(X_{\gamma},D_{\gamma})$ admits an involution $(A_0,A_1,A_2)\mapsto (A_2,A_1,A_0)$ from Remark \ref{rem:7A1}; call it $\iota$. We have $\iota(G_1^{\gamma})=G_1^{\gamma}$ and $\iota(L_1^{\gamma})=H_2^{\gamma}$. Therefore, replacing $\gamma_{+}$ with $\iota\circ\gamma_{+}$, if needed, we can assume that $q_1\in G_1$ or $q_1\in H_2$.

	\begin{casesp}
		\litem{$U$ is a chain}
		Then $U=[2,\bs{k},2,l,2]$ for some $k,l\geq 2$, and $(k,l)\neq 2$ by Lemma \ref{lem:U-2}. Here the $(-k)$-curve meets $A_0$, and the $(-l)$-curve meets $A_2$.
		
		\begin{casesp}
		\litem{$q_1\in G_1$}\label{case:q1G1} The involution $\iota\in \Aut(X_{\gamma},D_{\gamma})$ maps the divisors $A_1^{\gamma}$, $G_1^{\gamma}$ and $W_{\gamma}$ to themselves, and interchanges $A_0^{\gamma}$ with $A_{2}^{\gamma}$. Thus replacing $\gamma_{+}$ with $\iota\circ\gamma_{+}$ we remain in Case \ref{case:q1G1}, but with the roles of $k$ and $l$ interchanged. Thus we can assume $k\geq l$. In particular, we have $k\geq 3$, i.e.\ $A_{0}^{\tau}\cap H_1\subseteq \Bs\tau_{+}^{-1}$. 
		
		Consider the case when $W$ is a chain. Then $W=[(2)_{k-2},3,m,\ub{3},(2)_{l-2}]$ for some $m\geq 2$. Inequality \eqref{eq:ld-bound_h=2} gives $l=2$, and either $k=3$; or $k=4$, $m\leq 5$; or $k\leq 7$, $m=2$, hence $D$ is as in \ref{item:ht=3_nu_3_VZ_chain_k=3}--\ref{item:ht=3_nu_3_VZ_chain_m=2}. If $W$ is a fork then since $W$ is admissible and $k\geq 3$ we have $W=\langle m;[2],[\ub{3}],[2,3]\rangle$, so $(k,l)=(3,2)$ and $D$ is as in \ref{item:ht=3_nu_3_VZ_fork}. 
		\litem{$q_1\in  H_2$} Denote by $L$ the connected component of $D$ containing $L_1$. Assume first that $L$ is a chain. Then $W=[(2)_{m-2},\ub{4},(2)_{l-2}]$ for some $m\geq 2$. Since $(k,l)\neq (2,2)$,  inequality \eqref{eq:ld-bound_h=2} shows that either $l=2$ and $D$ is as in  \ref{item:ht=3_nu_3_YV_chain_l=2,k=3}--\ref{item:ht=3_nu_3_YV_chain_l=2,k>4}, or $k=2$ and $D$ is as in \ref{item:ht=3_nu_3_YV_chain_k=2,l=3}--\ref{item:ht=3_nu_3_YV_chain_k=2,l>4,m=4}, or $k,l\geq 3$, $m=2$ and $D$ is as in \ref{item:ht=3_nu_3_YV_chain_k=3,l>2,m=2}--\ref{item:ht=3_nu_3_YV_chain_k>3,l=3,m=2}. 
		
		Assume now that $L$ is a fork, so $L=\langle m;[2],[(2)_{k-2},3],T\trp\rangle$ for some $m\geq 2$ and admissible $T$; and $W=T^{*}*[(2)_{m-2},\ub{4},(2)_{l-2}]$. Since $\rho(\bar{X})=1$, the surface $\bar{X}$ is del Pezzo if and only if $0>K_{\bar{X}}\cdot \pi(A_0)=1-\ld(H_1)-\ld(C)$, where $C$ is the tip of $L$ meeting $A_0$. This inequality is equivalent to 
		\begin{equation}\label{eq:A0C}
			\ld(H_1)+\ld(C)>1,
		\end{equation}
		 which will be now more convenient than inequality \eqref{eq:ld-bound_h=2}, equivalent to $K_{\bar{X}}\cdot \pi(F)<0$ for a fiber $F$.

		Consider the case $k\geq 3$. Suppose $l\geq 3$. Then by Lemma \ref{lem:Alexeev}, log discrepancies of $H_1$ and $C$ do not exceed those in case $(k,l,m)=(3,3,2)$; $T=[2]$, when $U=[2,\bs{3},2,3,2]$ and $L=\langle 2;[\mathit{2},3],[2],[2]\rangle$, where $C$ is the italicized tip. In this case, $\ld(H_1)+\ld(C)=1$; contrary to \eqref{eq:A0C}. Thus $l=2$. If $k\geq 4$ then $T=[2]$ by the admissibility of $L$, so $U=[2,\bs{k},2,2,2]$, $L=\langle m;[\mathit{2},(2)_{k-3},3],[2],[2]\rangle$, and inequality \eqref{eq:A0C} gives $(k,m)=(4,2)$, so  $D$ is as in \ref{item:ht=3_nu_3_YV_fork_l=2,m=2}. If $k=3$ then $d(T)\leq 3$ by the admissibility of $L$, so $D$ is as in \ref{item:ht=3_nu_3_YV_fork_l=2,m>2}.
		
		Consider now the case $k=2$, so  $l\geq 3$ by assumption and $d(T)\leq 5$ by the admissibility of $L$. Let $B$ be the branching component of $L$. Inequality \eqref{eq:A0C} is equivalent to 
		\begin{equation}\label{eq:A0}
			\ld(B)>\frac{2l-7}{4l-5}.
		\end{equation}
		If $l=3$ then \eqref{eq:A0} holds and $D$ is as in \ref{item:ht=3_nu_3_YV_fork_l=3}. Assume $l\geq 4$, so $\ld(B)>\frac{1}{11}$, which implies that $T\in \{[2],[2,2]\}$ or  $T\in \{[3],[2,2,2]\}$ and $m=2$. If $T=[2]$ then
		$\ld(B)=\frac{1}{3m-4}$, so $m\leq 4$ and \eqref{eq:A0} shows that $D$ is as in \ref{item:ht=3_nu_3_YV_fork_T=[2]_m=2} if $m=2$ (so $\ld(B)=\frac{1}{2}$) or \ref{item:ht=3_nu_3_YV_fork_T=[2]_m>2} if $m>2$. If $T=[2,2]$ then $\ld(B)=\frac{1}{6m-9}$, so by \eqref{eq:A0} either $(m,l)=(3,4)$ and $D$ is as in \ref{item:ht=3_nu_3_YV_fork_T=[2,2]_m=3}, or $m=2$, $l\leq 7$, and $D$ is as in \ref{item:ht=3_nu_3_YV_fork_T=[2,2]_m=2}. In the remaining cases we have $m=2$, $\ld(B)=\frac{1}{5}$, so $l=4$ by \eqref{eq:A0}, and $D$ is as in \ref{item:ht=3_nu_3_YV_fork_T=[3]} if $T=[3]$ and as in \ref{item:ht=3_nu_3_YV_fork_T=[2,2,2]} if $T=[2,2,2]$.
		\end{casesp}
		
		\litem{$U$ is a fork} Since $U$ is admissible and has a twig of length three, one of the following holds.
		\begin{enumerate}[itemsep=0em]
			\item \label{item:22} $U=\langle \bs{k};[2],[2],[2,l,2]\rangle$ or $\langle k;[2],[2],[2,\bs{l},2]\rangle$ for some $k,l\geq 2$, $(k,l)\neq (2,2)$ by Lemma \ref{lem:U-2}.
			\item \label{item:23} $U=\langle \bs{k};[2],[3],[2,2,2]\rangle$ or $\langle k;[2],[3],[2,\bs{2},2]\rangle$ for some $k\geq 2$. Put $l=2$ in this case.
		\end{enumerate}		
		\begin{casesp} 
		\litem{$q_1\in G_1$} Like in Case \ref{case:q1G1}, applying the automorphism $(A_0,A_1,A_2)\mapsto (A_2,A_1,A_0)$ from Remark \ref{rem:7A1} we can assume that the long twig of $U$ meets $A_2$, i.e.\ $\beta_{D}(H_1)=3$.
		
		Put $\epsilon=1$ in case \ref{item:22} and $\epsilon=2$ in case \ref{item:23}. The proper transform of $A_{1}^{\phi}$, which is a component of $W$, meets twigs of $W$ of types  $[(2)_{\epsilon}]*[(2)_{k-2},3]$ and $[(2)_{l-2},\ub{3}]$. If $W$ is a fork then the admissibility of $W$ implies that $k=2$ and $\epsilon=1$, which is impossible. Thus $W=[(2)_{\epsilon}]*[(2)_{k-2},3,m,\ub{3},(2)_{l-2}]$ for some $m\geq 2$. In case \ref{item:23}, inequality \eqref{eq:ld-bound_h=2} gives $k=m=2$, so $D$ is as in \ref{item:ht=3_nu_3_ZV_U-3}. In case \ref{item:22} we have $W=[3,(2)_{k-3},3,m,\ub{3},(2)_{l-2}]$ if $k\geq 3$ and $W=[4,m,\ub{3},(2)_{l-2}]$, $l\geq 3$ if $k=2$. By \eqref{eq:ld-bound_h=2}, the first case holds and $(k,l,m)=(3,2,2)$, so $D$ is as in \ref{item:ht=3_nu_3_ZV_U-2}.
		
		\litem{$q_1\in H_2$} As before, denote by $L$ the connected component of $D$ containing $L_1$. 
		
		Assume $\beta_{D}(H_1)=3$. Suppose $L$ is a fork. Let $T_U,T_L$ be the components of $U$ and $L$ meeting $A_0$. Since $\bar{X}$ is del Pezzo, we have $0>\pi(A_0)\cdot K_{\bar{X}}=1-\ld(T_U)-\ld(T_L)$ by adjunction, so $\ld(T_U)+\ld(T_L)>1$. Passing to a weighted subgraph as in Lemma \ref{lem:Alexeev}, we can assume that one of the following holds: in case \ref{item:22}, $k\geq 3$: 
		$U=\langle 3,[\mathit{2}],[2],[2,2,2]\rangle$, $L=\langle 2;[\mathit{3},3],[2],[2]\rangle$; 
		in case \ref{item:22}, $l\geq 3$: 
		$U=\langle 2;[\mathit{2}],[2],[2,3,2]\rangle$, $L=\langle 2;[\mathit{4}],[2],[2]\rangle$;
		and in case \ref{item:23}: 
		$U=\langle 2;[\mathit{3}],[2],[2,2,2]\rangle$, $L=\langle 2;[\mathit{2},4],[2],[2]\rangle$; 
		where the italicized number refers to $T_U$ or $T_L$. In each case, we get $\ld(T_U)+\ld(T_L)=1$; a contradiction.
		
		Thus $L$ is a chain, so $W=[(2)_{m-2},\ub{4},(2)_{l-2}]$ for some $m\geq 2$. Consider the case $l\geq 3$, so $U$ is as in \ref{item:22}. Then inequality \eqref{eq:ld-bound_h=2} gives $(k,l,m)=(2,3,2)$, and $D$ is as in \ref{item:ht=3_nu_3_YV_U-2_k=2}. Thus we can assume $l=2$. In case \ref{item:22} we have $k\geq 3$ by assumption, so inequality \eqref{eq:ld-bound_h=2} gives $m=2$ or $(k,m)=(3,3)$, and $D$ is as in \ref{item:ht=3_nu_3_YV_U-2_k>2} or \ref{item:ht=3_nu_3_YV_U-2_m=3}. In case \ref{item:23} inequality \eqref{eq:ld-bound_h=2} gives $m=2$ or $(k,m)=(2,3)$, so $D$ is as in \ref{item:ht=3_nu_3_YV_U-3_k>2} or \ref{item:ht=3_nu_3_YV_U-3_k=2}.
		\smallskip
		
		Assume $\beta_{D}(H_1)=2$. Consider case \ref{item:22}, $k=2$. Then $l\geq 3$. Write $\psi=\sigma\circ\sigma_{+}$ so that $U_{\sigma}=\langle 2;[2],[2],[2,\bs{3},2]\rangle$, $W_{\sigma}=[\ub{5}]$. If $\psi=\sigma$ then $D$ is as in \ref{item:ht=3_nu_3_YV_U'-2_l=3}. Suppose $\psi\neq \sigma$, and let $\eta$ be a composition of $\sigma$ with a blowup at $r\in \Bs\sigma_{+}^{-1}$. If $r\in A_{1}$ then $W_{\eta}=[2,\ub{5}]$ or $[\ub{6}]$, so $\ld(H_{2})=\frac{1}{3}=\frac{1}{2}(1-\ld(H_{1}^{\sigma}))$, contrary to \eqref{eq:ld-bound_h=2}. Thus $\{r\}=A_0\cap H_1$, so $U_{\eta}=\langle 2;[2],[2],[2,\bs{4},2]\rangle$ and $\ld(H_{1})=\frac{1}{5}=1-2\ld(H_{2}^{\sigma})$, again a contradiction with \eqref{eq:ld-bound_h=2}. 
		
		For the remaining cases, recall that $\bar{X}$ is del Pezzo if and only $K_{\bar{X}}\cdot\pi(A_0)<0$. It is equivalent to inequality \eqref{eq:A0C}, where $C$ is the tip of $L$ meeting $A_0$.  When writing the type of $L$, we refer to $C$ by a number in italics.
		
		Suppose $l\geq 3$. Then case \ref{item:22} holds and the weighted graph of $U+L$ has a weighted subgraph $\langle 3;[2],[2],[2,\bs{3},2]\rangle+[\mathit{2},3,2,2]$. Using Lemma \ref{lem:Alexeev}, we get  $\ld(H_1)+\ld(C)\leq \frac{3}{11}+\frac{8}{11}=1$, a contradiction with \eqref{eq:A0C}. 
		
		Thus $l=2$, so $U=\langle k;[2],[2],[2,\bs{2},2]\rangle$ for some $k\geq 3$ or $\langle k,[3],[2],[2,\bs{2},2]\rangle$ for some $k\geq 2$. We compute $\ld(H_1)=\frac{2k-3}{4k-7}$ or $\frac{6k-9}{12k-19}$, respectively. In any case $\ld(H_1)\leq \frac{3}{5}$, so $\ld(C)>\frac{2}{5}$ by \eqref{eq:A0C}. 
		
		Assume that $L$ is a chain. Then $L=[\mathit{3},m,2]$ for some $m\geq 2$. We have $\ld(C)=\frac{2m}{6m-5}$, so inequality \eqref{eq:A0} implies that $D$ is as in \ref{item:ht=3_nu_3_YV_U'-2_k=4}--\ref{item:ht=3_nu_3_YV_U'-2_k>2_m=2} in case \ref{item:22} and as in \ref{item:ht=3_nu_3_YV_U'-3_chain_k>2}--\ref{item:ht=3_nu_3_YV_U'-3_chain_k=2} in case \ref{item:23}. 
		
		Assume that $L$ is a fork. If $m\geq 3$ then the weighted graph of $L$ has a weighted subgraph $\langle 3;[\mathit{3}],[2],[2]\rangle$, so $\ld(C)\leq \frac{2}{5}$, which is false. Thus $m=2$, so $L=\langle 2;[\mathit{3}],[2],T\rangle$ for some admissible $T$, $d(T)\leq 5$. If $T=[(2)_{s}]$ for some $s\geq 1$ then $\ld(C)=\frac{4}{7+s}$, so inequality \eqref{eq:A0C} implies that $D$ is as in \ref{item:ht=3_nu_3_YV_U'-2_T=[2]}--\ref{item:ht=3_nu_3_YV_U'-2_T=[2,2]} in case \ref{item:22} or as in  \ref{item:ht=3_nu_3_YV_U'-3_fork_k>2}--\ref{item:ht=3_nu_3_YV_U'-3_fork_k=2} in case \ref{item:23}. In the remaining cases, the weighted graph of $L$ has a weighted subgraph $\langle 2;[\mathit{3}],[2],[3]\rangle$ or $\langle 2;[\mathit{3}],[2],[3,2]\rangle$, so $\ld(C)\leq \frac{2}{5}$ or $\frac{6}{17}$; a contradiction because $\ld(C)>\frac{2}{5}$.
		\end{casesp}
	\end{casesp}
	\end{casesp}
	
	\litem{$q_2\in G_2$}\label{case:q2G2}
	Suppose $A_3$ contains a base point of $\phi_{+}^{-1}$, say $q_3$. Since $D$ has no circular subdivisor, we have $q_3\in L_3\cup H_2$. If $q_3\in H_2$ then interchanging $\ll_2$ with $\ll_3$  we get back in Case \ref{case:qH2}, so we can assume $q_3\in L_3$. Since $U$ is admissible, we have $q_1\in H_2$. Now after a blowup at $q_1,q_2,q_3$ we get $U=[3,\bs{2},2,2,\ub{3},2,2]$, so  $\ld(H_{1})+2\ld(H_{2})=1$, a contradiction with inequality \eqref{eq:ld-bound_h=2}. Thus $A_3\cap \Bs\phi_{+}^{-1}=\emptyset$. 
	\begin{casesp}
		\litem{$q_1\in G_1$} Suppose $\tau_{+}^{-1}$ has a base point $r\in A_{0}$, and let $\sigma$ be a composition of $\tau$ with a blowup at $r$. Then $U_{\sigma}=[2,\bs{3},2,2,\ub{2},2,3,2]$ if $r\in H_{1}$ and $\langle \bs{2};[2],[2],[4,2,\ub{2},2,2]\rangle$ if $r\in L_{1}$, so $\ld(H_{1}^{\sigma})+2\ld(H_{2}^{\sigma})=1$; a contradiction with \eqref{eq:ld-bound_h=2}. Thus $\Bs\tau_{+}^{-1}\subseteq A_1\cup A_2$. Since $U$ is admissible, $\tau_{+}$ is a composition of $k_{j}-2\geq 0$ blowups at $A_{j}\cap U_{\tau}$, each time on the proper transforms of $U_{\tau}$. Hence $U=[3,k_{1},\ub{2},k_{2},2,\bs{2},2]$. If $k_2=2$ then $D$ is as in \ref{item:ht=3_nu_3_WZ_l=2}. Otherwise, $k_1=2$, $k_2\leq 4$ by \eqref{eq:ld-bound_h=2}, and $D$ is as in \ref{item:ht=3_nu_3_WZ_k=2}.
		\smallskip
		
		Let $V$ be the proper transform of the line joining $p_2$ with $p_3$. Its image $V^{\gamma}\subseteq X_{\gamma}$ is a $(-1)$-curve satisfying  $V^{\gamma}\cdot D_{\gamma}=3$ and meeting $D_{\gamma}$ on  $L_{1}^{\gamma},G_{2}^{\gamma}$ and $G_{3}^{\gamma}$. 
		Let $\pi_{\gamma}\colon X_{\gamma}\to \bar{X}_{\gamma}$ be the contraction of $D_{\gamma}$. By Lemma  \ref{lem:w=2_uniqueness}\ref{item:ld-bound_h=2} $\bar{X}_{\gamma}$ is a del Pezzo surface. Thus $0>\pi_{\gamma}(V^{\gamma})\cdot K_{\bar{X}_{\gamma}}=2-\ld(L_{1}^{\gamma})-\ld(G_{2}^{\gamma})-\ld(G_{3}^{\gamma})$ by adjunction. The assumption $A_{3}\cap \Bs\phi_{+}^{-1}=\emptyset$ implies that $G_{3}^{\gamma}=[2]$ is a connected component of $D_{\gamma}$, so $\ld(G_{3}^{\gamma})=1$. We get 
		\begin{equation}\label{eq:ld_bound_A1'}
			\ld(L_{1}^{\gamma})+\ld(G_{2}^{\gamma})>1.
		\end{equation}
	Note that both \eqref{eq:ld_bound_A1'} and \eqref{eq:ld-bound_h=2} are equivalent to ampleness of $-K_{\bar{X}_{\gamma}}$: in the remaining part of Case \ref{case:q2G2}, it will be more convenient use \eqref{eq:ld_bound_A1'}. To this end, we introduce the following notation. 
			
		Let $L_{\gamma}$, $G_{\gamma}$ be the connected components of $D_{\gamma}$ containing $L_{1}$ and $G_{2}$, respectively. We write their types italicizing the number corresponding to $L_{1}$ and $L_{3}$. For example, if $\eta$ is a composition of $\phi$ with a blowup at $q_2$ then $L_{\eta}=[\mathit{3}]$ and $G_{\eta}=[\mathit{3}]$. In particular, $\ld(G_{2}^{\eta})=\frac{2}{3}$, so Lemma \ref{lem:Alexeev} implies that $\ld(G_{2}^{\gamma})\leq \frac{2}{3}$ whenever $\gamma=\eta\circ \tilde{\gamma}$ for a morphism $\tilde{\gamma}$. Inequality \eqref{eq:ld_bound_A1'} gives $\ld(L_{1}^{\gamma})>\tfrac{1}{3}$.
		
		\litem{$q_1\in L_1$}\label{case:q2G2_q1L1}
		If both $A_{0}^{\tau}$ and $A_{1}^{\tau}$ contain base points of $\tau_{+}^{-1}$ then after blowing up at these points, $L$ becomes $[\mathit{6}]$, $[\mathit{5},2]$ or $[2,\mathit{4},2]$, so $\ld(L_{1})=\tfrac{1}{3}$, which is impossible. Thus $\Bs\tau_{+}^{-1}\subseteq A_j\cup A_2$ for some $j\in \{0,1\}$. Replacing, if needed, $\tau_{+}$ with $\iota\circ\tau_{+}$, where $\iota\in \Aut(X_{\tau},D_{\tau})$ is the involution $(A_0,A_1,A_2)\mapsto (A_1,A_0,A_2)$ from Remark \ref{rem:7A1}, we can assume that $\Bs\tau_{+}^{-1}\subseteq A_0\cup A_2$.
		
		Let $\sigma$ be a composition of $\tau$ with $k_{j}-2\geq 0$ blowups at $A_{j}\cap U_{\sigma}$, $j\in \{0,2\}$, each time on the proper transform of $U_{\tau}$, so that $A_{j}^{\sigma}\cap U_{\sigma}\not\subseteq \Bs\sigma_{+}^{-1}$. Since $U$ is admissible, we have either $\psi=\sigma$, or $\sigma_{+}$ is a single blowup at $A_0\cap L$. Thus $G=[\mathit{3},(2)_{k_2-2}]$. If $\psi=\sigma$ then $L=[\mathit{4},(2)_{k_0-2}]$, otherwise $L=[\mathit{4},(2)_{k_0-3},3]$ if $k_0\geq 3$ or $[\mathit{5}]$ if $k_0=2$. If $k_2=2$ then Lemma \ref{lem:U-2} implies that $k_0\geq 3$, so $D$ is as in \ref{item:ht=3_nu_3_WX_l=2} or \ref{item:ht=3_nu_3_WXD}. If $k_2\geq 3$ then inequality \eqref{eq:ld_bound_A1'} implies that $\psi=\sigma$  and either $k_0=2$ or $(k_0,k_2)=(3,3)$, so $D$ is as in \ref{item:ht=3_nu_3_WX_k=2} or  \ref{item:ht=3_nu_3_WX_k,l=3}.
		
		\litem{$q_1\in H_2$}\label{case:q2G2_q1H2} Let $\sigma$ be a composition of $\tau$ with $k_{j}-2\geq 0$ blowups over $A_j\cap U_{\tau}$, $j\in \{0,2\}$, each time on the proper transform of $U_{\tau}$, so that $A_{j}^{\sigma}\cap U_{\sigma}\not\subseteq \Bs\sigma_{+}^{-1}$. If $\psi\neq \sigma$, we denote by $\eta$ a composition of $\sigma$ with a blowup at a given $r\in \Bs\sigma_{+}^{-1}$. Note that $\{r\}=A_0\cap L$, $\{r\}=A_2\cap G$ or $r\in A_1$.
		
		If $r\in A_{2}$ then $U_{\eta}=\langle k_2;[2],[\ub{3}],[2,\bs{k_{0}},2]\rangle$, so since $U$ is admissible, $k_0=2$, $\psi=\eta$, and $D$ is as in \ref{item:ht=3_nu_3_WYV} or \ref{item:ht=3_nu_3_WYV_k=2}. Thus we can assume $\Bs\sigma_{+}^{-1}\subseteq A_{0}\cup A_{1}$; hence $G=[\mathit{3},(2)_{k_{2}-2}]$ and by \eqref{eq:ld_bound_A1'}, $\ld(L_1)>1-\ld(G_{2})=\tfrac{1}{2}-\tfrac{1}{4k_2-2}\geq \tfrac{1}{3}$. Assume $r\in A_{0}$. If $k_0\geq 3$ then $L_{\eta}=[3,(2)_{k_0-3},\mathit{3},2,2]$, so $\ld(L_{1})=\tfrac{k_0}{4k_0-3}\leq \tfrac{1}{3}$, which is false. Hence $k_0=2$. We have $L_{\eta}=[\mathit{4},2,2]$, so $\ld(L_{1})=\tfrac{2}{5}$ and $k_2=2$ by \eqref{eq:ld_bound_A1'}. If $\psi=\eta$ then $D$ is as in \ref{item:ht=3_nu_3_WYC}, otherwise  $\Bs\eta_{+}^{-1}\subseteq A_1$ as $U$ is admissible, and after one further blowup we get $L=[\mathit{4},3,2]$ or $\langle 2;[2],[2],[\mathit{4}]\rangle$, so $\ld(L_{1})=\tfrac{1}{3}$, a contradiction.
		
		Therefore, we can assume $\Bs\sigma_{+}^{-1}\subseteq A_{1}$. Recall that $L_{\sigma}=[(2)_{k_0-2},\mathit{3},2,2]$ and $G_{\sigma}=[(2)_{k_2-2},\mathit{3}]$, so inequality  \eqref{eq:ld_bound_A1'} gives $k_0=2$; $k_0=3$, $k_2\leq 5$; or $k_0\in \{4,5,6\}$, $k_2\leq 3$. If $\psi=\sigma$ then $D$ is as in \ref{item:ht=3_nu_3_WY_k=2}, \ref{item:ht=3_nu_3_WY_k=3} or \ref{item:ht=3_nu_3_WY_k>3}.
		
		Assume $\psi\neq \sigma$. We have $L_{\eta}=[(2)_{k_0-2},\mathit{3},3,2]$ if $r\in L_{\sigma}$ and $\langle 2;[2],[2],[(2)_{k_0-2},\mathit{3}]\rangle$ if $r\in H_2$. If $k_0= 3$ then $\ld(L_{1}^{\eta})=\frac{1}{3}$, which is impossible; and any $k_0\geq 3$ can be reduced to $k_0=3$ by a vertical swap. Thus $k_0=2$. If $k_2=2$ then $D$ is as in \ref{item:ht=3_nu_3_WYT} (note that inequality \eqref{eq:ld_bound_A1'} is satisfied for any possible twig $T$, since $G=[\mathit{3}]$, so $\ld(G_2)=\frac{2}{3}$; and $L_1$ is a maximal twig of type $[3]$, so $\ld(L_1)>\tfrac{1}{3}$). Assume $k_2\geq 3$, so $\ld(G_2)=\frac{k_2}{2k_2-1}\leq \frac{3}{5}$ and thus $\ld(L_1)>\frac{2}{5}$ by \eqref{eq:ld_bound_A1'}. If $r\in L_{\sigma}$ then $L_{\eta}=[3,\mathit{3},2]$, so $\frac{2}{5}<\ld(L_1^{\eta})=\tfrac{5}{13}$, which is impossible. Hence $r\in H_2$. If $\psi=\eta$ then $D$ is as in \ref{item:ht=3_nu_3_WYU}. Assume $\psi\neq \eta$, and let $\theta$ be a composition of $\eta$ with blowup at $s\in \Bs\eta_{+}^{-1}\subseteq A_{1}$. If $s\in L_{\eta}$ then $L_{\theta}=\langle 2;[2],[3],[\mathit{3}]\rangle$, so $\ld(L_{1}^{\theta})=\tfrac{2}{5}$ which is impossible. Hence $s\in H_2$, $L_{\theta}=\langle 2;[2],[2,2],[\mathit{3}]\rangle$, and $\frac{4}{9}=\ld(L_{1}^{\theta})>\frac{k_2-1}{2k_2-1}$, so $k_2\leq 4$. If $\psi=\theta$ then $D$ is as in \ref{item:ht=3_nu_3_WYUU}, otherwise after one further blowup we get $\ld(L_1)=\tfrac{2}{5}$ or $\tfrac{6}{17}$; a contradiction because $\ld(L_1)>\frac{2}{5}$.	
	\end{casesp}

	\litem{$q_2\in L_2$} 
	\begin{casesp}
	\litem{The map $\tau_{+}^{-1}$ has a base point $q_{3}\in A_{3}^{\tau}$} 
	If $q_3$ lies on $H_2$ or $G_2$ then after interchanging $\ll_2$ with $\ll_3$  we get back to Case \ref{case:qH2} or \ref{case:q2G2}, so we can assume $q_3\in L_3$.  Let $\sigma$ be a composition of $\tau$ with a blowup at $q_3$.
	
	The admissibility of $W$ implies that $q_1\in H_2$. Hence $U_{\sigma}=[3,\bs{2},3]$, $W_{\sigma}=[2,2,\ub{3},2,2]$. Suppose $\sigma_{+}^{-1}$ has a base point $r\not\in A_0$, and let $\eta$ be a composition of $\sigma$ with a blowup at $r$.
	
	If $r\in A_{1}$ then $U_{\eta}$ is still a chain $[3,\bs{2},3]$, so $\ld(H_{1}^{\eta})=\tfrac{1}{2}$ and therefore $\ld(H_{2}^{\eta})>\tfrac{1}{4}$ by \eqref{eq:ld-bound_h=2}. But $W_{\eta}=[2,2,\ub{4},2,2]$ if $r\in H_2$  and $\langle \ub{3},[2],[2,2],[2,2]\rangle$, so $\ld(H_{2}^{\eta})=\tfrac{1}{4}$ or $\tfrac{1}{7}$; a contradiction. If $r\in A_j$ for some $j\in \{2,3\}$ then $U_{\eta}+W_{\eta}=[3,\bs{2},4]+\langle 2;[2],[2],[2,2,\ub{3}]\rangle$ if $r\in U_{\sigma}$ and $[3,\bs{2},3,2]+[2,2,\ub{3},3,2]$ if $r\in W_{\sigma}$, so $\ld(H_{1}^{\eta})+2\ld(H_{2}^{\eta})=\tfrac{31}{34}$ or $\tfrac{536}{551}$, contrary to \eqref{eq:ld-bound_h=2}. We conclude that $\Bs\sigma_{+}^{-1}\subseteq A_0$, so $\ld(H_{2})=\ld(H_{2}^{\sigma})=\tfrac{2}{5}$ and therefore $\ld(H_1)>\tfrac{1}{5}$ by \eqref{eq:ld-bound_h=2}. Let $\theta$ be a composition of $\sigma$ with $k-2\geq 0$ blowups over $A_0\cap H_1$, on the proper transforms of $H_1$. Now $U_{\theta}=[3,\bs{k},3]$, so $\ld(H_{1}^{\theta})=\tfrac{2}{3k-2}$, and the inequality $\ld(H_1)>\tfrac{1}{5}$ gives $k\in \{2,3\}$. If $\psi\neq \theta$ then after one blowup at $\Bs\theta_{+}^{-1}=A_{0}\cap (D_{\theta}-H_{1})$ we get $U=\langle \bs{k};[2],[3],[3]\rangle$, so $\ld(H_1)=\tfrac{1}{6k-7}\leq \tfrac{1}{5}$; a contradiction. Thus $\psi=\theta$, and $D$ is as in \ref{item:ht=3_nu=3_XX'}.
	\litem{$A_{3}\cap \Bs\tau_{+}^{-1}=\emptyset$}
	If $q_{1}\in L_{1}$ or $q_1\in G_1$ then replacing $\tau_{+}$ with $\iota\circ\tau_{+}$, where $\iota$ is the automorphism $(A_0,A_1,A_2)\mapsto (A_1,A_2,A_0)$ from Remark \ref{rem:7A1}, we get back to Case \ref{case:qH2}. Thus we can assume $q_1\in H_2$. Suppose each of the $(-1)$-curves $A_0,A_1,A_2$ contains a base point of $\tau_{+}^{-1}$, and let $\sigma$ be a composition of $\tau$ with a blowup at those three points. Then $U_{\sigma}=[4,\bs{3},2]$, $[2,3,\bs{3},2]$, $\langle \bs{2};[4],[2],[2]\rangle$, or $\langle \bs{2},[2,3],[2],[2]\rangle$, so $\ld(H_{1}^{\sigma})=\frac{1}{3}$; and $W_{\sigma}=[\ub{4},3,2]$, $[2,\ub{3},3,2]$, $\langle 2;[\ub{4}],[2],[2]\rangle$ or $\langle 2;[2,\ub{3}],[2],[2]\rangle$, so $\ld(H_{2}^{\sigma})=\frac{1}{3}$; a contradiction with \eqref{eq:ld-bound_h=2}. Thus $A_{j}\cap \Bs\tau_{+}^{-1}=\emptyset$ for some $j\in \{0,1,2\}$; and using the automorphism $(A_0,A_1,A_2)\mapsto (A_2,A_0,A_1)$ from Remark \ref{rem:7A1} we can assume $j=2$, i.e.\ $\Bs\tau_{+}^{-1}\subseteq A_0\cup A_1$.
	
	If $\Bs\tau_{+}^{-1}\subseteq A_j$ for some $j\in \{0,1\}$ then applying the above automorphism $(A_0,A_1,A_2)\mapsto (A_2,A_0,A_1)$ we can assume $j=0$; so $D$ is as in \ref{item:qL3_A0_chain} or \ref{item:qL3_A0_fork}. Note that in this case $\ld(H_2)=\frac{4}{7}>\frac{1}{2}$, so inequality \eqref{eq:ld-bound_h=2} is automatically satisfied. Thus we can assume that $\tau_{+}^{-1}$ has a base point on each $A_{j}$, $j\in \{0,1\}$. 
	
	As before, denote by $L_{\gamma}$ the connected component of $D_{\gamma}$ containing $L_1$. Let $\sigma$ be a composition of $\tau$ with $k_j-2\geq 0$ blowups over $A_j\cap (D_{\tau}-L_{\tau})$, $j\in \{0,1\}$, each time on the proper transforms of $D_{\tau}-L_{\tau}$, so that $A_j^{\sigma}\cap \Bs\sigma_{+}^{-1}\subseteq L_{\sigma}$. If $\psi=\sigma$ then by assumption $k_0,k_1\geq 3$, and inequality \eqref{eq:ld-bound_h=2} shows that $D$ is as in \ref{item:qL3_k1}--\ref{item:qL3_4}.

	Assume $\psi\neq \sigma$. Consider first the case $\Bs\sigma_{+}^{-1}\subseteq A_0$, so by assumption $k_1\geq 3$; and $U+W=\langle \bs{k_0};T,[3],[2]\rangle+[(2)_{k_1-2},\ub{3},2,2]$ for some admissible $T$, $d(T)\leq 5$. If $T=[4]$, $[(2)_{4}]$, $[3,2]$ or $[2,3]$ then for $(k_0,k_1)=(2,3)$ we get $\ld(H_1)+2\ld(H_2)=1$, $1$, $\frac{181}{187}$ or $\frac{241}{253}$, respectively, which contradicts \eqref{eq:ld-bound_h=2}. Using Lemma \ref{lem:Alexeev}, we conclude that $T=[3]$ or $T=[(2)_{s}]$ for some $s\in \{1,2,3\}$. If $T=[3]$ then inequality \eqref{eq:ld-bound_h=2} gives $(k_0,k_1)=(2,3)$, so $D$ is as in \ref{item:qL3_rA0_T=[3]}. If $T=[(2)_{s}]$ then inequality \eqref{eq:ld-bound_h=2} implies that $D$ is as in \ref{item:qL3_rA0}--\ref{item:qL3_rA0_s=3_l=3}.
	
	Consider now the case $\Bs\sigma_{+}^{-1}\subseteq A_1$. Then by assumption $k_0\geq 3$ and $U+W=[3,\bs{k_0},2]+T^{*}*[(2)_{k_1-2},\ub{3},2,2]$ for some admissible $T$, such that $L=\langle k_1;[(2)_{k_0-2},3],T\trp,[2]\rangle$. The admissibility of $L$ implies that either $T=[2]$, or $d(T)=3$ and $k_0=3$. In the latter case, \eqref{eq:ld-bound_h=2} gives $\ld(H_2)>\frac{1}{2}(1-\ld(H_1))=\frac{4}{13}$, but we compute $\ld(H_2)=\frac{3k_1-2}{12k_1-11}\leq \frac{4}{13}$ if $T=[2,2]$ and $\ld(H_2)=\frac{3k-1}{12k-7}\leq \frac{5}{17}<\frac{4}{13}$ if $T=[3]$; a contradiction. Thus the first case holds, and \eqref{eq:ld-bound_h=2} shows that $D$ is as in \ref{item:qL3_rA1} or \ref{item:qL3_rA1_k=3}.
	
	It remains to consider the case when $A_j\cap L_{\sigma}$ is a base point of $\sigma_{+}^{-1}$ for both $j\in \{0,1\}$. Let $\eta$ be a composition of $\sigma$ with a blowup at those two points; so $D_{\eta}=\langle \bs{k_0};[3],[2],[2]\rangle$+$[3,(2)_{k_1-3},\ub{3},2,2]+\langle k_1;[4],[2],[2]\rangle+[2]$; where $[3,(2)_{-1},\ub{3}]\de [\ub{4}]$. We have $\ld(H_1^{\sigma})=\frac{1}{3k_0-4}$ and $\ld(H_2^{\sigma})=\frac{k_1}{4k_1-3}\leq\frac{2}{5}$, so $\ld(H_1)>\frac{1}{5}$ by \eqref{eq:ld-bound_h=2}; in particular $k_0=2$. If $\Bs\eta_{+}^{-1}\subseteq A_1$ then $D$ is as in \ref{item:qL3_[4]}: note that any such $\bar{X}$ is del Pezzo since $\pi(A_0)\cdot K_{\bar{X}}=1-\ld(C)-\ld(L_1)<1-\frac{3}{4}-\frac{1}{4}=0$, where $C$ is the new $(-2)$-twig of $U$.
	
	Assume that $\eta_{+}^{-1}$ has a base point $r\in A_0$, and let $\theta$ be a composition of $\eta$ with a blowup at $r$. If $r\in L_1$ then $U_{\theta}=\langle \bs{2},[3],[3],[2]\rangle$, so $\ld(H_1^{\theta})=\frac{1}{5}$, which is impossible. Thus $r\in U_{\eta}$. Now $U_{\theta}=\langle \bs{2},[2,2],[3],[2]\rangle$, so $\ld(H_1^{\theta})=\frac{1}{3}$, and the inequality $\frac{k_1}{4k_1-3}=\ld(H_2^{\theta})>1-2\ld(H_1^{\theta})=\frac{1}{3}$ yields $k_1=2$. If $\psi=\theta$ then $D$ is as in \ref{item:qL3_[5]}. Suppose $\psi\neq \theta$. If $\theta_{+}^{-1}$ has a base point on $A_0$ then after one blowup at this base point we get $U=\langle \bs{2},[2,2,2],[3],[2]\rangle$ or $\langle \bs{2},[3,2],[3],[2]\rangle$, so $\ld(H_1)=\frac{1}{5}$ or $\frac{1}{17}<\frac{1}{5}$, which is impossible. Thus $\Bs\theta_{+}^{-1}\subseteq A_1$ and  $\ld(H_2)>\frac{1}{2}(1-\ld(H_1))=\frac{1}{3}$. After one blowup at a base point of $\Bs\theta_{+}^{-1}$ we get $W=[\ub{5},2,2]$ or $[2,\ub{4},2,2]$, so $\ld(H_2)=\frac{4}{13}$ or $\frac{5}{17}$, so $\ld(H_2)<\frac{1}{3}$, a contradiction.
	\qedhere
	\end{casesp}
	\end{casesp}
\end{proof}

\clearpage
\section{Case $\width(\bar{X})=1$}\label{sec:w=1}

In the previous sections we have classified del Pezzo surfaces of rank $1$, height $3$ and width at least $2$. Thus it remains to settle the case $\width=1$. We keep the following notation.
\begin{equation}\label{eq:assumption_w=1}
	\parbox{.9\textwidth}{
		$\bar{X}$ is a log terminal del Pezzo surface of rank one such that $\height(\bar{X})=3$ and $\width(\bar{X})=1$, $(X,D)$ is its minimal log resolution and $p\colon X\to \P^1$ is a $\P^1$-fibration such that $H\de D\hor$ is a $3$-section.
	}
\end{equation}

Recall that the condition $\height(\bar{X})=3$ means that for any $\P^1$-fibration of $X$ we have $F\cdot D\geq 3$ for a fiber $F$; and the condition $\width(\bar{X})=1$ means that whenever the equality holds, the horizontal part $H$ of $D$ is irreducible. Like in Section \ref{sec:w=2}, the analysis splits into two cases, according to the separability of the morphism $p|_{H}\colon H\to \P^1$ of degree $3$. Of course, $p|_{H}$ can be inseparable only if $\cha\kk=3$.
\smallskip

The assumption \eqref{eq:assumption_w=1} is very restrictive. In fact, if $\cha\kk\neq 2,3$ then we will see in Lemma \ref{lem:w=1_cha-neq-3} that a surface $\bar{X}$ satisfying \eqref{eq:assumption_w=1} and having no descendant with elliptic boundary is unique up to an isomorphism. If $\cha\kk=2$ there is no such surface; and if $\cha\kk=3$ there are only a few of them, listed in Lemma \ref{lem:w=1_cha=3}. 

This section is organized in the same way as the previous ones. First, in Lemma \ref{lem:w=1_basics} we roughly describe degenerate fibers of $p$ and we fix the notation. Next, in Lemma \ref{lem:w=1_psi} we infer the remaining cases of Propositions \ref{prop:ht=3_models} and \ref{prop:ht=3_swaps}. 
To prove Theorem \ref{thm:ht=3}, we treat cases $\cha\kk\neq 3$ and $\cha\kk=3$ separately in Sections \ref{sec:w=1_cha-neq-3} and \ref{sec:w=1_cha=3}.

\subsection{Proof of Propositions \ref{prop:ht=3_models} and \ref{prop:ht=3_swaps}}\label{sec:w=1_swaps}

\begin{lemma}[Degenerate fibers]\label{lem:w=1_basics}
	Let $(X,D)$ be as in \eqref{eq:assumption_w=1}. 
	Let $F_1,\dots,F_{\nu}$ be all degenerate fibers of $p$. Then each $F_j$ has a unique $(-1)$-curve, say $L_{j}$. Let $\tau_{+}\colon (X,D)\to (X_{\tau},\check{D}_{\tau})$ be the contraction of $\sum_{j}L_{j}$ and of all non-branching vertical $(-1)$-curves in the subsequent images of $D$. Put $\hat{F}_j=\tau_{+}(F_j)\redd$, $\hat{H}=\tau_{+}(H)$. Then for each $j\in \{1,\dots,\nu\}$ one of the following holds. 
	\begin{enumerate}
		\item\label{item:r=3} $\hat{F}_j=[3,1,2,2]$, $\hat{H}\cdot \hat{F}=1$ and $\hat{H}$ meets $\hat{F}_j$ in the $(-1)$-curve. In this case put $\mu_j=3$.
		\item\label{item:r=2} $\hat{F}=[2,1,2]$, $\hat{H}\cdot \hat{F}=2$ and $\hat{H}$ meets $\hat{F}_j$ in a tip and in the $(-1)$-curve. In this case put $\mu_j=2$ if $\hat{H}$ meets $\hat{F}_j$ normally; and $\mu_j=3$ otherwise (i.e.\ when $\hat{H}$ meets $\hat{F}_j$ in a node).
	\end{enumerate}
\end{lemma}
\begin{proof}
	By Lemma \ref{lem:fibrations-Sigma-chi}\ref{item:Sigma} each $F_j$ has only one $(-1)$-curve, say $L_{j}$, and by Lemma \ref{lem:fibrations-Sigma-chi}\ref{item:-1_curves}, $(F_j)\redd-L_{j}\subseteq D$. 
	
	Suppose $\hat{F}_j=[0]$. Then \cite[Lemma 2.8(c)]{PaPe_ht_2} implies that $\hat{F}_j\subseteq \check{D}_{\tau}$, so since $\hat{D}_{\tau}$ is snc, $\hat{H}$ meets $\hat{F}_j$ in three points. As $\tau_{+}^{-1}$ has only one base point on $\hat{F}_j$, namely the image of $L_j$, the proper transform of $\hat{F}_j$ is a component of $D$ meeting $H$ at least twice; a contradiction. 
	
	Suppose $\hat{F}_j=[1,1]$. Since $F_j$ has only one $(-1)$-curve, the unique base point of $\tau_{+}^{-1}$ on $\hat{F}_j$ is its node. In particular, $\hat{F}_j\subseteq \check{D}_{\tau}$. Since $\check{D}_{\tau}$ is snc, $\hat{H}$ meets $\hat{F}_j$ normally, so it meets one of the tips of $\hat{F}_j$ at least twice. The proper transform of that tip is a component of $D$ meeting $H$ at least twice; a contradiction as before.
	 
	Therefore, $\hat{F}_j$ has a unique $(-1)$-curve, say $\hat{L}_j$, so $\hat{L}_j$ has multiplicity $\mu_j\geq 2$ in $\hat{F}_j$. By the definition of $\tau_{+}$, we have $3\leq \beta_{\check{D}_{\tau}}(\hat{L}_j)=\hat{L}_j\cdot \hat{H}+\beta_{(\hat{F}_j)\redd}(\hat{L}_j)\leq \tfrac{1}{\mu_j}F_j\cdot H+2<4$, so $\beta_{(\hat{F}_j)\redd}(\hat{L}_j)=2$ and $\hat{L}_j\cdot \hat{H}=1$. It follows that $\mu_j\leq 3$ and $\hat{F}_j$ 
	is as in \ref{item:r=3} if $\mu_j=3$ and as in \ref{item:r=2} if $\mu_j=2$.
\end{proof}

\begin{lemma}[Propositions \ref{prop:ht=3_models} and \ref{prop:ht=3_swaps}, case $\width=1$]\label{lem:w=1_psi}
	We keep the notation and assumptions of Lemma \ref{lem:w=1_basics}. Then the following hold.
	\begin{enumerate}
		\item\label{item:nu=3} We have $\nu=3$, i.e.\ $p$ has exactly three degenerate fibers. One can order them so that $F_1$ is as in \ref{lem:w=1_basics}\ref{item:r=3}.
		\item\label{item:w=1_psi}  There is a birational morphism  $\psi\colon X\to \P^2$ such that $\psi_{*}D=\qq+\ll_1+\ll_2+\ll_3$ , where $\qq\de \psi_{*}H$ is a cubic with a cusp $p_1$, and for each $j\in \{1,2,3\}$, $\ll_j\de \psi_{*}F_j$ is a line passing through a point $p_0\not\in \qq$, tangent to $\qq$ with multiplicity $\mu_j$ at some point $p_j$.
		\item\label{item:w=1_psi-cases} One of the following holds.
		\begin{enumerate}
			\item\label{item:cha-neq-3_psi} $\cha\kk\neq 2,3$ and $\mu_j=2$ for $j\in \{2,3\}$,
			\item\label{item:cha=3-psi} $\cha\kk=3$ and $\mu_j=3$ for $j\in \{2,3\}$.
		\end{enumerate}
		\item\label{item:w=1_swaps} 
		The minimalization $\psi$ from \ref{item:w=1_psi} factors as $\psi=\phi\circ\phi_{+}$, where $\phi_{+}\colon (X,D)\sqto (Y,D_Y)$ is a vertical swap onto the minimal log resolution of one of the surfaces from Example \ref{ex:w=1}. 
	\end{enumerate}
\end{lemma}
\begin{proof}
	Let $\nu_{3}$, $\nu_{2}$ be the number of degenerate fibers $F_j$ of $p$ such that $\hat{F}_j$ is as in Lemma \ref{lem:w=1_basics}\ref{item:r=3} and \ref{item:r=2}, respectively. Clearly, $\nu=\nu_2+\nu_3$. 
	Let $\bar{\phi}_{+}$ be the composition of $\tau_{+}$ with the contraction of all vertical $(-1)$-curves in the subsequent images of $D$ which meet the image of $H$ once, normally. Put $\bar{F}_j=(\bar{\phi}_{*}F_j)\redd$, $\bar{H}=\bar{\phi}(H)$. Note that $\bar{F}_j=[0]$ is the image of a tip of $\hat{F}_j$, so the proper transform $V_j\de \bar{\phi}_{*}^{-1}\bar{F}_j$ is a component of $D$. Put $\hat{V}_j=\tau_{+}(V_j)$. By the definition of $\mu_j$, we have $(\bar{F}_j\cdot \bar{H})_{q_j}=\mu_j$ for some point $q_j\in \bar{F}_j$.
	
	The curve $\bar{H}$ is a smooth, rational $3$-section on a Hirzebruch surface $\F_{m}$ for some $m\geq 0$. Write $\bar{H}\sim af+3C$, where $f$ is, fiber, $C$ is a section of self-intersection $m$, and $a=\bar{H}\cdot\Sec_{m}\geq 0$, where $\Sec_{m}$ is the negative section. By adjunction $-2=\bar{H}\cdot (\bar{H}+K_{\F_m})=4a-6m-6$, so $m=\tfrac{2}{3}(1-a)\leq \tfrac{2}{3}$. Thus $(m,a)=(0,1)$, so $\bar{H}^2=6$ and 
	\begin{equation}\label{eq:H}
		\hat{H}^{2}=6-2\nu_{2}-3\nu_{3}.
	\end{equation}
	
	\ref{item:nu=3} We need to show that $\nu=3$ and $\nu_3\geq 1$. The pullback of the other projection of $\P^1\times \P^1$ is a $\P^1$-fibration of $X$ such that $D\hor$ consists of $\nu+1$ $1$-sections, namely $H$ and $V_1,\dots,V_{\nu}$. Since $\width(\bar{X})=1$, we get $\nu\geq 3$. 
	
	Suppose $\nu\geq 4$. Formula \eqref{eq:H} gives $\hat{H}^2=6-2\nu-
	\nu_{3}\leq -2$. Thus $\tau_{+}$ defines a vertical swap $(X,D)\sqto (X_{\tau},D_{\tau})$, where $D_{\tau}$ equals $\check{D}_{\tau}$ minus all vertical $(-1)$-curves. By Lemma \ref{lem:cascades}\ref{item:cascades-still-dP}  $(X_{\tau},D_{\tau})$ is the minimal log resolution of some del Pezzo surface of rank one. Applying \eqref{eq:H} again, we get $-\hat{H}^2=3\nu-6-\nu_2$. By Lemma  \ref{lem:w=1_basics}\ref{item:r=2}, $\hat{H}$ meets a $(-2)$-tip of each of the $\nu_2$ fibers of type $[2,1,2]$. Thus $\nu_2=\beta_{D_{\tau}}(\hat{H})\leq 3$, and for $\nu_2=0,1,2$ or $3$ the connected component of $D_{\tau}$ containing $\hat{H}$ is of type $[\bs{3\nu-6}]$, $[2,\bs{3\nu-7}]$, $[2,\bs{3\nu-8},2]$ or $\langle \bs{3\nu-9},[2],[2],[2]\rangle$,  respectively. As usual, the bold number refers to $\hat{H}$. Lemma \ref{lem:Alexeev} implies that in any case $\ld(\hat{H})\leq \frac{2}{3\nu-6}\leq \frac{1}{3}$, a contradiction with Lemma \ref{lem:delPezzo_criterion}.
	
	Thus $\nu=3$. Suppose $\nu_3=0$, so $\nu_2=3$. By Lemma \ref{lem:w=1_basics}\ref{item:r=2}, $\hat{H}$ meets each $\hat{F}_j$ once in a $(-2)$-tip $\hat{V}_j$ and once in a $(-1)$-curve, call it $A_j$. In particular, $H$ meets $V_j$ for all $j\in \{1,2,3\}$. It follows that the connected component of $D$ containing $H$, call it $U$, is an admissible fork, with a branching component $H$ and maximal twigs, say, $T_j$, meeting $H$ in $V_j$. Moreover, by the definition of $\tau_{+}$ either $T_j=V_j$ and $\tau_{+}$ is an isomorphism near $F_j$, or $T_j=(\tau_{+}^{-1})_{*}\hat{F}_j=[2,2,2]$ and $(F_j)\redd\wedge \Exc\tau_{+}$ is a $(-1)$-tip of $(F_j)\redd$. Indeed, otherwise the proper transform of $A_{j}$ would be a branching component of $T_j$, which is impossible. Since $U$ is an admissible fork, we have $T_{j}=[2,2,2]$ for at most one $j$. Thus $\#\Exc\tau_{+}\leq 1$, so $H^2\geq \hat{H}^2-1=-1$ by \eqref{eq:H}; a contradiction.
	\smallskip
	
	\ref{item:w=1_psi} Part \ref{item:nu=3} implies that $(\bar{F}_1\cdot \bar{H})_{q_{1}}=3$. Let $\eta\colon \P^1\times \P^1\map \P^2$ be a composition of a blowup at $q_{1}$ and the contraction of the proper transforms of the horizontal and vertical lines through $q_{1}$. Clearly, $\psi\de \eta\circ\bar{\phi}$ is a morphism. Now $\qq\de \psi_{*}H$ is a cubic with a cusp at the image of $V_1$, call it $p_1$; and $\ll_{j}\de \psi_{*}F_{j}$ are lines meeting at $p_{0}\not\in \qq$. We have $(\ll_1\cdot \qq)_{p_1}=3$ and $(\ll_j\cdot \qq)_{p_j}=\mu_j$ for $j\in \{2,3\}$, where $p_{j}\de \eta(q_j)$. 
	\smallskip
	
	\ref{item:w=1_psi-cases} The morphism $p|_{H}\colon H\to \P^1$ of degree three is ramified at $H\cap F_{j}$ with multiplicity $\mu_{j}$. If $p|_{H}$ is not separable, then $\cha\kk=3$ and \ref{item:cha=3-psi} holds. Assume $p|_{H}$ is separable. By the Hurwitz formula, see \cite[Corollary IV.2.4]{Hartshorne_AG}, we have $\sum_{j=1}^{3}\mu_{j}\leq 7$. Since $\mu_1=3$, we get $\mu_{2}=\mu_{3}=2$. By \cite[Lemma 5.5]{PaPe_MT} a configuration $\qq+\ll_1+\ll_2+\ll_3$ as in \ref{item:w=1_psi} with $(\ll_j\cdot \qq)_{p_j}=\mu_j$ exists only if $\cha\kk\neq 2,3$, so  \ref{item:cha-neq-3_psi} holds.
	\smallskip
	
	\ref{item:w=1_swaps} By \ref{item:nu=3}, $\nu_{2}\leq 2$. If $\nu_2\leq 1$ then $\mu_{j}=3$ for some $j\in \{2,3\}$, so case \ref{item:cha=3-psi} holds, and $\tau_{+}$ defines a vertical swap onto a log surface as in Example \ref{ex:w=1}\ref{item:w=1_cha=3_GK} if $\nu_2=1$ and \ref{ex:w=1}\ref{item:w=1_cha=3_not-GK} if $\nu_2=0$.
	
	Assume $\nu_{2}=2$. Formula \eqref{eq:H} gives $\hat{H}^2=-1$. If $\Bs\tau_{+}^{-1}\subseteq \hat{F}_1$ then $\hat{V}_{2}+\hat{H}+\hat{V}_{3}=[2,1,2]$ supports a fiber of a $\P^1$-fibration whose pullback to $X$ has heoght $2$ with respect to $D$, which is impossible. Hence the $(-1)$-curve in, say, $\hat{F}_{2}$ contains a base point of $\tau_{+}^{-1}$, call it $r$. It follows from definition of $\tau_{+}$ that $\hat{H}$ meets $\hat{F}_{2}$ in two points, i.e.\ $\mu_2=2$, so case \ref{item:cha-neq-3_psi} holds, see Figure \ref{fig:w=1_cha-neq-3}. It remains to prove that $r\in \hat{H}$. Suppose the contrary. Since $D$ has no circular subdivisor, we have $r\in \hat{V}_2$. Since $H^2<-1=\hat{H}^2$, the curve $\hat{H}$ contains another base point of $\tau_{+}^{-1}$, say $s$; and interchanging $\ll_2$ with $\ll_3$ if necessary, we can assume $\{s\}=\hat{H}\cap \hat{F}_{1}$. Write $\tau_{+}=\sigma\circ \sigma_{+}$, where $\sigma$ is a blowup at $r$ and $s$. Then $H^{\sigma}\de \sigma_{+}(H)$ is a $(-2)$-curve, so $\sigma_{+}$ defines a vertical swap $(X,D)\sqto (X_{\sigma},D_{\sigma})$. By Lemma \ref{lem:cascades}\ref{item:cascades-still-dP}, $(X_{\sigma},D_{\sigma})$ is a minimal resolution of some del Pezzo surface of rank one. The connected component of $D_{\sigma}$ containing $H^{\sigma}$ is of type $\langle\bs{2},[2],[3],[2,2]\rangle$ with branching component $H^{\sigma}$, so $\ld(H^{\sigma})=\tfrac{1}{3}$; a contradiction with Lemma \ref{lem:delPezzo_criterion}.
\end{proof}

\phantomsection\label{proofs}
Given the above results, the proofs of Propositions \ref{prop:ht=3_models} and \ref{prop:ht=3_swaps} can now be completed as follows.

\begin{proof}[Proof of Proposition \ref{prop:ht=3_models}]
	Let $\bar{X}$ be a del Pezzo surface of rank one and height $3$, and let $(X,D)$ be its minimal log resolution.  f $\width(\bar{X})=3$ then $\bar{X}$ satisfies the assumption \eqref{eq:assumption_ht=3}, so by Lemma \ref{lem:nu_1=1} the minimalization $\psi\colon X\to \P^1\times \P^1$ from Lemma \ref{lem:H_disjoint} is as in Proposition \ref{prop:ht=3_models}\ref{item:w=3_models}. If $\width(\bar{X})=2$ then $\bar{X}$ satisfies~\eqref{eq:assumption_w=2}, so Lemma \ref{lem:w=2_psi} provides a minimalization $\psi\colon X\to \P^2$ as in Proposition \ref{prop:ht=3_models}\ref{item:w=2_models}. Eventually, if $\width(\bar{X})=1$ then the condition \eqref{eq:assumption_w=1} holds, and by Lemma \ref{lem:w=1_psi}\ref{item:w=1_psi-cases} the minimalization $\psi\colon X\to \P^2$ from  \ref{lem:w=1_psi}\ref{item:w=1_psi} is as in Proposition~\ref{prop:ht=3_models}\ref{item:w=1_models}. 
\end{proof}

\begin{proof}[Proof of Proposition \ref{prop:ht=3_swaps}]
	Like before, let $\bar{X}$ be a del Pezzo surface of rank one and height $3$, let $(X,D)$ be its minimal log resolution. If $\width(\bar{X})=3$ then by Lemma \ref{lem:w=3_swaps}\ref{item:w=3_swaps} the surface $\bar{X}$ vertically swaps to one of the canonical surfaces from Example \ref{ex:w=3}, so Proposition \ref{prop:ht=3_swaps}\ref{item:ht=w=3} holds. If $\width(\bar{X})=2$ then by Lemma \ref{lem:w=2_swaps} $\bar{X}$ vertically swaps to a surface from Example \ref{ex:w=2_cha_neq_2} if $\cha\kk\neq 2$ or \ref{ex:w=2_cha=2} if $\cha\kk=2$,  so \ref{prop:ht=3_swaps}\ref{item:ht=3,w=2,cha-neq-2_can},\ref{item:ht=3,w=2,cha-neq-2_GK} or \ref{prop:ht=3_swaps}\ref{item:swap_cha=2_GK} holds, respectively. Eventually, if $\width(\bar{X})=1$ then by Lemma  \ref{lem:w=1_psi}\ref{item:w=1_psi-cases} we have $\cha\kk\neq 2$; and by Lemma \ref{lem:w=1_psi}\ref{item:w=1_swaps} $\bar{X}$ vertically swaps to a surface from Example \ref{ex:w=1}\ref{item:w=1_cha-neq-3}. By Proposition \ref{prop:primitive}\ref{item:primitive-deb} the surfaces from Example \ref{ex:w=1}\ref{item:w=1_cha-neq-3},\ref{item:w=1_cha=3_GK} have descendants with elliptic boundary. We conclude that \ref{prop:ht=3_swaps}\ref{item:swap_ht=3,w=1_cha-neq-3},\ref{item:swap_ht=3,w=1_cha=3_GK} or \ref{prop:ht=3_swaps}\ref{item:swap_cha=3} holds, as claimed. 
\end{proof}

\subsection{Further restrictions and notation} 

We  now reconstruct the vertical swap $\phi_{+}$ from Lemma \ref{lem:w=1_psi}\ref{item:w=1_swaps}. To do this, we first introduce Notation \ref{not:phi_+_h=1} which is analogous to \ref{not:phi_+_H} and \ref{not:phi_+_h=2}. Next, we prove Lemma \ref{lem:w=1_uniqueness}, which, just like Lemma \ref{lem:w=3_uniqueness} or \ref{lem:w=2_uniqueness}, asserts that recovering $\phi_{+}$ is a combinatorial task: simply blow up at one of the common points of a vertical $(-1)$-curve and a boundary, as long as the inequality \eqref{eq:ld-bound_h=1} is satisfied. In particular, we will see that a surface $\bar{X}$ as in \eqref{eq:assumption_w=1} is determined uniquely, up to an isomorphism, by its singularity type. 

\begin{notation}[Recovering the minimalization $\psi$, cf.\ Notations \ref{not:phi_+_H}, \ref{not:phi_+_h=2}]\label{not:phi_+_h=1}
	Let $(X,D)$ be as in \eqref{eq:assumption_w=1}. Let $\psi\colon X\to \P^2$ be the minimalization from Lemma \ref{lem:w=1_psi}\ref{item:w=1_psi}, and let $\psi=\phi\circ\phi_{+}$ be its factorization given by Lemma \ref{lem:w=1_psi}\ref{item:w=1_swaps}. We study factorizations $\psi=\gamma\circ\gamma_{+}$, where $\gamma$ is a composition of $\phi$ with some morphism and $\gamma_{+}\colon (X,D)\sqto (X_{\gamma},D_{\gamma})$ is a vertical swap. For any such factorization we put $H^{\gamma}=(D_{\gamma})\hor$ and denote by $U_{\gamma}$ the connected component of $D_{\gamma}$ containing $H^{\gamma}$. For $j=1,2,3$ we denote by $A_{j}^{\gamma}$ the unique $(-1)$-curve in $\gamma^{-1}(p_j)$, so $\check{D}_{\gamma}\de \gamma_{+}(D)=D_{\gamma}+\sum_{j}A_{j}^{\gamma}$. We put $L_{j}=\phi^{-1}_{*}\ll_j$. As usual, we skip the superscript $\gamma$ if it is clear from the context.
\end{notation}

\begin{lemma}[Recovering $\psi$ step by step, cf.\ Lemmas \ref{lem:w=3_uniqueness} and  \ref{lem:w=2_uniqueness}]\label{lem:w=1_uniqueness}
	Let $\bar{X}$ be a del Pezzo surface as in \eqref{eq:assumption_w=1}; i.e.\ of height $3$ and width $1$. For every factorization $\psi=\gamma\circ\gamma_{+}$ as in Notation \ref{not:phi_+_h=1}, the following hold.
	\begin{enumerate}
		\item\label{item:ld-bound_ht=1} The log surface $(X_{\gamma},D_{\gamma})$ is a minimal resolution of an lt del Pezzo surface of rank one. In particular, 
		\begin{equation}\label{eq:ld-bound_h=1}
			\ld(H^{\gamma})>\frac{1}{3}.
		\end{equation}
		\item \label{item:w=1_no_C1} Every base point of $\gamma_{+}^{-1}$ is a common point of $D_{\gamma}\reg$ and of a vertical $(-1)$-curve $A_{j}^{\gamma}$ for some $j\in \{1,2,3\}$.
		\item \label{item:w=1_uniqueness} The isomorphism class of the log surface $(X,\check{D})$ is uniquely determined by the weighted graph of $\check{D}$.
		\item \label{item:w=1_hi} Put $h^i=h^i(\lts{X}{D})$. Then $h^0=h^1=0$ and $h^2=0,1,2$ if $\bar{X}$ swaps vertically to a surface from Example \ref{ex:w=1}\ref{item:w=1_cha-neq-3}, \ref{item:w=1_cha=3_GK}, \ref{item:w=1_cha=3_not-GK}, respectively; see Lemma \ref{lem:w=1_psi}\ref{item:w=1_swaps}.
	\end{enumerate}
\end{lemma}
\begin{proof}
	The proof is analogous to the proof of Lemma \ref{lem:w=3_uniqueness}, so we give only a sketch. Part \ref{item:ld-bound_ht=1} follows from Lemma \ref{lem:cascades}\ref{item:cascades-still-dP} and Lemma \ref{lem:delPezzo_criterion}. Part \ref{item:w=1_no_C1} is a consequence of the admissibility of $U$. Indeed, if \ref{item:w=1_no_C1} fails then some base point of $\gamma_{+}^{-1}$ lies on $A_{j}^{\gamma}\setminus D_{\gamma}$ for some $j\in \{1,2,3\}$ and looking at Figure \ref{fig:w=1} we see that the proper transform of $A_{j}^{\gamma}$ is either a component of a circular subdivisor of $U$, or a branching component of $U$ in case $U$ is not an admissible fork. Parts \ref{item:w=1_uniqueness} and \ref{item:w=1_hi} follow from  \ref{item:w=1_no_C1} and \cite[Lemma 2.18]{PaPe_ht_2} in the same way as Lemma \ref{lem:w=3_uniqueness}\ref{item:ht=3_uniqueness},\ref{item:ht=3_h1} follows from \ref{lem:w=3_uniqueness}\ref{item:w=3_no_C1}. Indeed, by  \ref{item:w=1_no_C1} the morphism $(X,\check{D})\to (Y,\check{D}_Y)$ is inner, where $(Y,D_Y)$ is as in Example \ref{ex:w=1} and $\check{D}_Y$ is the sum of $D_Y$ and $(-1)$-curves shown there. Since by Proposition \ref{prop:primitive}\ref{item:primitive-uniqueness} the latter is unique, so is $(X,\check{D})$. The number $h^i$ equals $h^i(\lts{Y}{D_Y})$, which is listed in Proposition \ref{prop:primitive}\ref{item:primitive-hi}.
\end{proof}

\begin{lemma}
	\label{lem:U-2_w=1}
	Let $\bar{X}$ be as in \eqref{eq:assumption_w=1}, and assume that $\bar{X}$ has no descendant with elliptic boundary. Then the connected component $U$ of $D$ containing $H$ is neither a $(-2)$-chain nor a $(-2)$-fork.
\end{lemma}
\begin{proof}
		Suppose the contrary. If $\cha\kk\neq 3$ then $\phi_{+}=\id$, 
		i.e.\ $\bar{X}$ is the surface from Example \ref{ex:w=1}\ref{item:w=1_cha-neq-3}, which by Proposition \ref{prop:primitive}\ref{item:primitive-deb} has a descendant with elliptic boundary; a contradiction. Thus $\cha\kk=3$. Now $\bar{X}$ swaps vertically to the surface from Example \ref{ex:w=1}\ref{item:w=1_cha=3_GK}, and, after interchanging  $F_1$ with $F_2$ if necessary (cf.\ Lemmas \ref{lem:w=1_cha=3_GK_Aut} and \ref{lem:w=1_cha=3_not-GK_Aut}) we have $\Bs\phi_{+}^{-1}\subseteq A_2^{\phi}\cap L_{2}^{\phi}$, see Figure \ref{fig:w=1_cha=3_GK}. In particular, $\phi_{+}$ is an isomorphism near the proper transform of the line joining $p_1$ with $p_2$, which therefore remains an elliptic tie; 
		a contradiction.
\end{proof}

\subsection{The list of singularity types: case $\cha\kk\neq 3$}\label{sec:w=1_cha-neq-3}

\begin{lemma}[Classification in case $\width=1$, $\cha\kk\neq 3$, see Table \ref{table:ht=3_char=0}]\label{lem:w=1_cha-neq-3}
	Assume $\cha\kk\neq 3$. Assume that $\bar{X}$ is a del Pezzo surface of rank one, height $3$, and width $1$, and $\bar{X}$ has no descendant with elliptic boundary. Then $\cha\kk\neq 2$ and the following hold.
	\begin{parts}
		\item\label{item:w=1-uniqueness} The surface $\bar{X}$ is unique up to an isomorphism.
		\item\label{item:w=1-classification} The minimal log resolution $(X,D)$ of $\bar{X}$ admits a $\P^1$-fibration $p$ such that $\bar{X}$ swaps vertically to the surface from Example \ref{ex:w=1}\ref{item:w=1_cha-neq-3}. The combinatorial type of $(X,D,p)$ is 
		\begin{equation*}
			[2,2,2\dec{1},3]+[2,2\dec{2},2,\bs{3}\dec{1,2,3},2\dec{3}]+[2]\dec{3},
		\end{equation*}
		where the numbers decorated by $\dec{j}$ correspond to components meeting the $j$-th vertical $(-1)$-curve, and the bold number corresponds to the horizontal component of $D$, see Section \ref{sec:notation}.
	\end{parts}
\end{lemma}
\begin{proof}
	We use Notation \ref{not:phi_+_h=1}. Lemma \ref{lem:U-2_w=1} implies that $\psi\neq \phi$.  Let $\sigma$ be a composition of $\phi$ with a blowup at $r\in \Bs\phi_{+}^{-1}$. Recall from Lemma \ref{lem:w=1_uniqueness}\ref{item:w=1_no_C1} that $r\in D_{\phi}\cap A_{j}^{\phi}$ for some $j\in \{1,2,3\}$.
	
	Suppose $r\in A_{3}$. We have $r\in L_1\cup H$, because $D$ has no circular subdivisor. If $r\in L_1$ then $U_{\sigma}$ is a non-admissible fork $\langle \bs{2};[3],[2,2],[2,2,2]\rangle$, which is impossible. Thus  $\{r\}=A_{3}\cap H$. Now $U_{\sigma}=[(2)_{3},\bs{3},(2)_{3}]$, so $\ld(H^{\sigma})=\tfrac{1}{3}$; a contradiction with inequality \eqref{eq:ld-bound_h=1}. 
	Suppose $r\in A_{2}$. Then $U_{\sigma}=\langle 2;[2],[2],[2,\bs{3},2]\rangle$ if $r\in H$ and $\langle \bs{2};[2],[2],[2,3,2]\rangle$ otherwise, so $\ld(H^{\sigma})=\tfrac{1}{3}$; again a contradiction with \eqref{eq:ld-bound_h=1}.
	
	Therefore, $\Bs\phi_{+}^{-1}\subseteq A_{1}$. The admissibility of $U$ implies that $r\in H$. If $\psi=\sigma$ then  \ref{item:w=1-classification} holds. Suppose $\psi\neq \sigma$, and let $\eta$ be a composition of $\sigma$ with a blowup at some $s\in \Bs\sigma_{+}^{-1}\subseteq A_1$. Now $U_{\eta}=[2,\bs{4},2,2,2]$ if $s\in H$ and $\langle \bs{3},[2],[2],[2,2,2]\rangle$ otherwise, hence $\ld(H^{\eta})=\tfrac{3}{11}<\frac{1}{3}$ or $\tfrac{1}{5}<\frac{1}{3}$; a contradiction with \eqref{eq:ld-bound_h=1}. 
	
	We conclude that the surface $\bar{X}$ is obtained by an elementary vertical swap from the surface constructed in  Example \ref{ex:w=1}\ref{item:w=1_cha-neq-3}. Such a surface has no descendant with elliptic boundary by Lemma \ref{lem:no-deb}. It has height $3$ and width $1$, because its singularity type does not appear in the classification \cite{PaPe_ht_2}, or in Lemmas \ref{lem:w=3}, \ref{lem:w=2_cha_neq_2}. It is unique up to an isomorphism by Lemma \ref{lem:w=1_uniqueness}\ref{item:w=1_uniqueness}.
\end{proof}

\subsection{The list of singularity types: case $\cha\kk=3$}\label{sec:w=1_cha=3}

\begin{lemma}[Classification in case $\width=1$, $\cha\kk=3$, see Table \ref{table:ht=3_char=3}]\label{lem:w=1_cha=3}
	Assume $\cha\kk=3$. Let $\cS$ be a singularity type of a log terminal surface. Let $\bar{X}$ be a del Pezzo surface of rank $1$, height $3$, width $1$ and type $\cS$, i.e.\ $\bar{X}\in \Pht^{\width=1}(\cS)$, see Notation \ref{not:P}. Assume $\bar{X}$ has no descendant with elliptic boundary. 
	Then $\cS$ is one of the types listed below, and the following hold.
	\begin{parts}
		\item\label{item:w=1-cha=3-uniqueness} We have $\#\Pht^{\width=1}(\cS)=1$, i.e.\ $\bar{X}$ is unique up to an isomorphism.
		\item\label{item:w=1-cha=3-classification} The minimal log resolution $(X,D)$ of $\bar{X}$ admits a $\P^1$-fibration $p$ such that $\bar{X}$ swaps vertically to a surface $\bar{Y}$ from Example \ref{ex:w=1}\ref{item:w=1_cha=3_GK} or \ref{ex:w=1}\ref{item:w=1_cha=3_not-GK}, and the combinatorial type of $(X,D,p)$ is one of the following. 
	\end{parts}
	\begin{enumerate}[itemsep=0.6em]
	\item\label{item:cha=3_not-GK} $\bar{Y}$ is of type $4\cdot [3]+3\rA_{2}$, see Example \ref{ex:w=1}\ref{item:w=1_cha=3_not-GK}, and $(X,D,p)$ is one of the following:
	\begin{longlist}
		\item\label{item:w=1_0_id} $[\bs{3}]\dec{1,2,3}+[2,2]\dec{1}+\ldec{1}[3]+[2,2]\dec{2}+\ldec{2}[3]+[2,2]\dec{3}+\ldec{3}[3]$,
		\item\label{item:w=1_0} $\ldec{2,3}[\bs{4},(2)_{k-2}]\dec{1}+[2,2,k\dec{1},3]+[2,2]\dec{2}+\ldec{2}[3]+[2,2]\dec{3}+\ldec{3}[3]$, $k\geq 2$,
		\item\label{item:w=1_0B_k=2} $[\bs{5}]\dec{1,2,3}+\langle 2;[2]\dec{1},[3],[2,2]\rangle +[2,2]\dec{2}+\ldec{2}[3]+[2,2]\dec{3}+\ldec{3}[3]$,	
		\item\label{item:w=1_0B} $\ldec{2,3}[\bs{4},(2)_{k-3},3]\dec{1}+\langle k;[2]\dec{1},[3],[2,2]\rangle +[2,2]\dec{2}+\ldec{2}[3]+[2,2]\dec{3}+\ldec{3}[3]$, $k\geq 3$,	
		\item\label{item:w=1_0Y} $[\bs{5}]\dec{1,2,3}+[2,2,2\dec{1},3]+[2,2,2\dec{2},3]+[2,2]\dec{3}+\ldec{3}[3]$,
		\item\label{item:w=1_0Z} $[2,2,2,\bs{4}]\dec{1,2,3}+[2,2,2\dec{1},3]+[4]\dec{2}+[2,2]\dec{3}+\ldec{3}[3]$,		
		\item\label{item:w=1_0X} $[3,2,\bs{4}]\dec{1,2,3}+[2,2,2\dec{1},3]+[2,3]\dec{2}+[2,2]\dec{3}+\ldec{3}[3]$,
		\item\label{item:w=1_0A} $[3,k\dec{1},\bs{3}]\dec{2,3}+[2,3,(2)_{k-2}]\dec{1}+[2,2]\dec{2}+[3]\dec{2}+[2,2]\dec{3}+\ldec{3}[3]$, $k\geq 2$,
		\item\label{item:w=1_0AA_k=2} $\langle 2;[2]\dec{1},[3],[\bs{3}]\dec{2,3}\rangle +[2,4]\dec{1}+[2,2]\dec{2}+[3]\dec{2}+[2,2]\dec{3}+\ldec{3}[3]$,		
		\item\label{item:w=1_0AA} $\langle k;[2]\dec{1},[3],[\bs{3}]\dec{2,3}\rangle +[2,3,(2)_{k-3},3]\dec{1}+[2,2]\dec{2}+[3]\dec{2}+[2,2]\dec{3}+\ldec{3}[3]$, $k\geq 3$,
		\setcounter{foo}{\value{longlisti}}
	\end{longlist}
	\item\label{item:cha=3_GK} $\bar{Y}$ is of type $2\cdot [3]+\rA_{1}+3\rA_{2}$, see Example \ref{ex:w=1}\ref{item:w=1_cha=3_GK}, and $(X,D,p)$ is one of the following:
	\begin{longlist}\setcounter{longlisti}{\value{foo}}	
		\item\label{item:w=1_1} $[2\dec{3},\bs{3}\dec{2,3},(2)_{k-2}]\dec{1}+[2,2,k\dec{1},3]+[2,2]\dec{2}+\ldec{2}[3]+[2]\dec{3}$, $k\geq 2$,
		\item\label{item:w=1_1B_k=2} $[2\dec{3},\bs{4}]\dec{1,2,3}+\langle 2;[2]\dec{1},[3],[2,2]\rangle +[2,2]\dec{2}+\ldec{2}[3]+[2]\dec{3}$,
		\item\label{item:w=1_1B} $[2\dec{3},\bs{3}\dec{2,3},(2)_{k-3},3]\dec{1}+\langle k;[2]\dec{1},[3],[2,2]\rangle +[2,2]\dec{2}+\ldec{2}[3]+[2]\dec{3}$, $k\geq 3$,
		\item\label{item:w=1_1Y} $[\bs{4}\dec{1,2,3},2\dec{3}]+[2,2,2\dec{1},3]+[2,2,2\dec{2},3]+[2]\dec{3}$,
		\item\label{item:w=1_1Z} $[2,2,2,\bs{3}\dec{1,2,3},2\dec{3}]+[2,2,2\dec{1},3]+[4]\dec{2}+[2]\dec{3}$,		
		\item\label{item:w=1_1X} $[3,2,\bs{3}\dec{1,2,3},2\dec{3}]+[2,2,2\dec{1},3]+[2,3]\dec{2}+[2]\dec{3}$,
		\item\label{item:w=1_1C} $[2,2,k\dec{1},\bs{2}\dec{2,3},2\dec{3}]+[4,(2)_{k-2}]\dec{1}+[2,2]\dec{2}+[3]\dec{2}+[2]\dec{3}$, $k\geq 3$,
		\item\label{item:w=1_1CB} $\langle k;[2]\dec{1},[2,2],[2\dec{3},\bs{2}]\dec{2,3}\rangle+[4,(2)_{k-3},3]\dec{1}+[2,2]\dec{2}+[3]\dec{2}+[2]\dec{3}$, $k\geq 3$.
	\end{longlist}
	\end{enumerate}
\end{lemma}
\begin{proof}
	As in the proof of Lemmas \ref{lem:w=3}, \ref{lem:w=2_cha_neq_2} or \ref{lem:w=2_cha=2}, we first note that the classification \ref{item:w=1-cha=3-classification} implies the uniqueness result \ref{item:w=1-cha=3-uniqueness}. Fix a singularity type $\cS$ as in the above list. Lemma \ref{lem:no-deb} and a direct application of \cite[Theorem E]{PaPe_ht_2} in case \ref{item:w=1_0_id}, shows that no surface $\bar{X}\in \Pht^{\width=2}(\cS)$ has a descendant with elliptic boundary. Hence by \ref{item:w=2-classification}, the minimal log resolution $(X,D)$ of $\bar{X}$ has a $\P^1$-fibration $p$ as above. Since $\cS$ appears exactly once on the above list, it uniquely determines the combinatorial type of $(X,D,p)$. The latter uniquely determines the isomorphism class of $\bar{X}$ by Lemma \ref{lem:w=1_uniqueness}\ref{item:w=1_uniqueness}, so $\#\Pht^{\width=1}(\cS)\leq 1$. We claim that the equality holds. Clearly, a triple $(X,D,p)$ of each combinatorial type above exists. We check that it satisfies the inequality \eqref{eq:ld_phi_H}, so $(X,D)$ is the minimal log resolution of a del Pezzo surface $\bar{X}$ of rank 1. We have $\height(\bar{X})=3$ because $\cS$ does not appear in \cite{PaPe_ht_2}, and $\width(\bar{X})=1$ since $\cS$ does not appear in Lemmas \ref{lem:w=3} or \ref{lem:w=2_cha_neq_2}. Thus $\bar{X}\in \Pht^{\width=1}(\cS)$, as needed. 
	\smallskip
	
	Therefore, it is enough to prove \ref{item:w=1-cha=3-classification}, i.e.\ to verify the completeness of the above list. 
	
	We use Notation \ref{not:phi_+_h=1}. By Lemma \ref{lem:w=1_psi}\ref{item:w=1_swaps} we have a factorization $\psi=\phi\circ \phi_{+}$, where $\phi_{+}\colon (X,D)\sqto (X_{\phi},D_{\phi})$ is a vertical swap onto the minimal log resolution of a surface $\bar{Y}$ from Example \ref{ex:w=1}\ref{item:w=1_cha=3_GK} or \ref{item:w=1_cha=3_not-GK}. Put $\epsilon=1$ in the first case and $\epsilon=0$ in the second case. By the definition of $\phi_{+}$ we have $\Bs\phi_{+}^{-1}\subseteq A_{1}^{\phi}\cup A_{2}^{\phi}\cup A_{3}^{\phi}$ if $\epsilon=0$ and $\Bs\phi_{+}^{-1}\subseteq A_{1}^{\phi}\cup A_{2}^{\phi}$ if $\epsilon=1$. In any case, $\Bs\phi_{+}^{-1}\subseteq D_{\phi}$ by Lemma \ref{lem:w=1_uniqueness}\ref{item:w=1_no_C1}. 	
	For any factorization $\psi=\gamma\circ\gamma_{+}$ as in Notation \ref{not:phi_+_h=1}, we denote by $T_{j}^{\gamma}$, $G_{j}^{\gamma}$ the proper transforms of the $(-3)$- and $(-2)$-curve meeting $A_j^{\phi}$. 
	
	If $\psi=\phi$ then, since $U$ is not a $(-2)$-chain by Lemma \ref{lem:U-2_w=1}, we have $\epsilon=0$ and $D$ is as in \ref{item:w=1_0_id}. Assume $\psi\neq \phi$. Reordering the degenerate fibers, if needed, we can assume that $\phi_{+}^{-1}$ has a base point $q\in A_1$. Let $\upsilon$ be a composition of $\phi$ with some $k-1\geq 1$ blowups over $q$, on the proper transforms of $A_{1}^{\phi}$. By Lemma \ref{lem:w=2_uniqueness}\ref{item:w=1_no_C1} we can choose $\upsilon$ so that $\Bs\upsilon_{+}^{-1}\cap A_{1}^{\upsilon}\subseteq U_{\upsilon}$. If $\psi\neq \upsilon$ then for a given $r\in \Bs\upsilon_{+}^{-1}$ we denote by $\sigma$ a composition of $\upsilon$ with a blowup at $r$; and similarly, if $\psi\neq \sigma$ we let $\eta$ be a composition of $\sigma$ with a blowup at $s\in \Bs\sigma_{+}^{-1}$.

	\begin{casesp}
	\litem{$q\in H^{\tau}$}\label{case:qH} If $\psi=\upsilon$ then $D$ is as in \ref{item:w=1_0} if $\epsilon=0$ and  as in \ref{item:w=1_1} if $\epsilon=1$. Assume $\psi\neq \upsilon$ and fix $r\in \Bs\upsilon_{+}^{-1}$.
	\begin{casesp}
	\litem{$r\in A_{1}$}\label{case:rA1} By the definition of $\upsilon$ we have $\{r\}=A_1\cap U_{\upsilon}$. Assume $\psi=\sigma$. If $k=2$ then $D$ is as in \ref{item:w=1_0B_k=2} or as in \ref{item:w=1_1B_k=2}, and if $k\geq 3$ then $D$ is as in \ref{item:w=1_0B} or as in \ref{item:w=1_1B}. Assume $\psi\neq\sigma$. Since all forks in $D$ are admissible, we have $s\not\in A_{1}$. After possibly interchanging  $F_2$ with $F_3$ in case $\epsilon=0$ we get $s\in A_{2}$. If $s\in U_{\sigma}$ then $U_{\eta}=[2]*[(2)_{k-2},\bs{5-\epsilon},(2)_{\epsilon}]$; so $\ld(H^{\eta})=\frac{k-1}{4k-5}\geq \frac{1}{3}$ if $\epsilon=0$ and $\frac{2k-1}{10k-11}\geq \frac{1}{3}$ if $\epsilon=1$; contrary to \eqref{eq:ld-bound_h=1}. If $s\in T_{2}$ then $U_{\eta}=[2]*[(2)_{k-2},\bs{4},2,2,2]$ if $\epsilon=0$; $[2,\bs{4},2,2,2]$ if $(\epsilon,k)=(1,2)$ and $\langle \bs{3};[2],[3],[2,2,2]\rangle$ otherwise, so $\ld(H^{\eta})=\frac{2k+1}{18k-19},\frac{3}{11}$, or $\frac{1}{17}$; contrary to \eqref{eq:ld-bound_h=1}. Thus $s\in G_2$ so $U_{\eta}=[2]*[(2)_{k-2},\bs{4},2,3]$ if $\epsilon=0$, $[2,\bs{4},2,3]$ if $(\epsilon,k)=(1,2)$ and $\langle \bs{3},[2],[3],[3,2]\rangle$ otherwise, so $\ld(H^{\eta})=\frac{k+1}{12k-13},\frac{7}{29}$ or $\frac{1}{47}$; a contradiction with \eqref{eq:ld-bound_h=1}. 
	
	\litem{$A_1\cap \Bs\upsilon_{+}^{-1}=\emptyset$} After possibly interchanging  $F_2$ with $F_3$ in case $\epsilon=0$, we can assume $r\in A_{2}$.
	 
	Suppose $k=3$. Then $U_{\sigma}=[(2)_{\epsilon},\bs{5-\epsilon},2]$ if $r\in H$; 
	$U_{\sigma}=[2,\bs{4},2,2,2]$ or $\langle \bs{3};[2],[2],[2,2,2]\rangle$ if $r\in T_2$; and $U_{\sigma}=[2,\bs{4},2,3]$ or $\langle \bs{3};[2],[2],[3,2]\rangle$ if $r\in G_2$. For such $U_{\sigma}$, the number $\ld(H^{\sigma})$ equals $\tfrac{1}{3}$; $\tfrac{3}{11}$,$\tfrac{1}{5}$ and $\tfrac{7}{29}$,$\tfrac{1}{7}$, so \eqref{eq:ld-bound_h=1} fails. If $k\geq 4$, we get a contradiction with \eqref{eq:ld-bound_h=1} after a vertical swap to the previous case $k\geq 3$. Thus $k=2$. If $\psi=\sigma$ then $D$ is as in \ref{item:w=1_0Y} or \ref{item:w=1_1Y} if $r\in H$, as in \ref{item:w=1_0Z} or \ref{item:w=1_1Z} if $r\in T_2$, and as in \ref{item:w=1_0X}, \ref{item:w=1_1X} if $r\in G_2$.
	
	Assume $\psi\neq \sigma$. If $r\in H$ then reordering the fibers, if needed we get back to the Case \ref{case:rA1}, with  $\{s\}=A_2\cap H_1$. 
	Thus we can assume $H\cap \Bs\sigma_{+}^{-1}=\emptyset$. 
	
	\begin{casesp}
	\litem{$r\in T_{2}$}\label{case:rL2} If $\epsilon=1$ then $s\in A_{2}$, so $U_{\eta}=[2,2,3,\bs{3},2]$ if $s\in U_{\sigma}$ and $U_{\eta}=\langle 2;[2],[2,2],[2,\bs{3}]\rangle$ otherwise, so $\ld(H^{\eta})=\tfrac{9}{29}$ or $\tfrac{3}{13}$, contrary to \eqref{eq:ld-bound_h=1}. Thus $\epsilon=0$. If $s\in A_2$ then $U_{\eta}=[2,2,3,\bs{4}]$ if $s\in U_{\sigma}$ and $U_{\eta}=\langle 2;[2],[2,2],[\bs{4}]\rangle$ if $s\in T_2$, so $\ld(H^{\eta})=\tfrac{8}{25}$ or $\tfrac{2}{7}$, contrary to \eqref{eq:ld-bound_h=1}. Hence $s\in A_3$. By assumption $s\not\in H$, so $s\in T_3$ or $s\in G_3$. We get $U_{\eta}=[(2)_{3},\bs{4},(2)_{3}]$ or $U_{\eta}=[(2)_{3},\bs{4},2,3]$, so $\ld(H^{\eta})=\tfrac{1}{5}$ or $\tfrac{9}{53}$, contrary to \eqref{eq:ld-bound_h=1}. 
	
	\litem{$r\in G_2$} Suppose $s\in A_{2}$. If $\epsilon=1$ then $s\in U_{\sigma}$ since $U_{\eta}$ is admissible; hence $U_{\eta}=[3,3,\bs{3},2]$ and $\ld(H^{\eta})=\tfrac{5}{17}<\frac{1}{3}$; contrary to \eqref{eq:ld-bound_h=1}. Thus $\epsilon=0$, so  $U_{\eta}=[3,3,\bs{4}]$ if $s\in U_{\sigma}$ and $\langle 2;[2],[3],[\bs{4}]\rangle$ otherwise, and $\ld(H^{\eta})=\tfrac{9}{29}$ or $\tfrac{3}{11}$; a contradiction with \eqref{eq:ld-bound_h=1}. Thus $s\in A_{3}$; in particular, $\epsilon=0$. We have $s\not \in H$ by assumption, and if $s\in T_3$ then interchanging  $F_2$ with $F_3$ we get back to Case \ref{case:rL2}. Thus we can assume $s\in G_3$, so $U_{\eta}=[3,2,\bs{4},2,3]$ and $\ld(H^{\eta})=\tfrac{1}{7}$; a contradiction with \eqref{eq:ld-bound_h=1}.
	\end{casesp}
	\end{casesp}

	\litem{$q\in T_{1}$, $H\cap \Bs\phi_{+}^{-1}=\emptyset$}  Suppose $\upsilon_{+}^{-1}$ has a base point $r\not\in A_1$, so, say, $r\in A_2$. By assumption $r\not\in H$, so $r\in T_2\cup G_2$. If $\epsilon=0$ then $U_{\sigma}$ is a non-admissible fork, which is impossible, so $\epsilon=0$. Now for $k=2$ we get 
	$U_{\eta}=[(2)_{3},\bs{3},(2)_{3}]$ if $r\in T_2$ and $U_{\eta}=[(2)_{3},\bs{3},2,3]$ if $r\in G_2$, so $\ld(H^{\eta})=\frac{1}{3}$ or $\frac{3}{11}$, contrary to \eqref{eq:ld-bound_h=1}. Thus for any $k\geq 2$, we get a contradiction with \eqref{eq:ld-bound_h=1} after a vertical swap. Thus $\Bs\upsilon_{+}^{-1}\subseteq A_1$. 
	
	Assume $\epsilon=0$. Conjugating the automorphism \eqref{eq:Bernasconi_involution} by a permutation $(F_1,F_2,F_3)\mapsto (F_3,F_2,F_1)$ we get an involution $\iota\in \Aut(X_{\phi},D_{\phi})$ mapping $(T_1,A_1)$ to $(H,A_1)$. Under the isomorphism $\Aut(X_{\phi},D_{\phi})\cong S_{4}$ explained in Remark \ref{rem:cube}, this involution corresponds to a $180^{\circ}$ rotation about the horizontal axis in Figure \ref{fig:cube}. Now replacing $\phi_{+}$ with $\iota\circ\phi_{+}$ we get back to Case \ref{case:qH}. 
	
	Therefore, we can assume $\epsilon=1$. If $\psi=\upsilon$ then $k\geq 3$ by Lemma \ref{lem:U-2_w=1}, and  $D$ is as in \ref{item:w=1_1C}. Assume $\psi\neq \upsilon$. By our assumptions, the unique base point of $\upsilon_{+}^{-1}$ is $\{r\}= A_{1}\cap T_{1}$. The admissibility of $U$ implies that $\upsilon_{+}$ is a single blowup at $r$. As before, Lemma \ref{lem:U-2_w=1} gives $k\geq 3$, so $D$ is as in \ref{item:w=1_1CB}.
	
	\litem{$q\in G_{1}^{\tau}$, $(H\cup \bigcup_{i}T_i)\cap \Bs\phi_{+}^{-1}=\emptyset$} Suppose $r\in A_{2}$. Then by assumption $\{r\}=A_{2}\cap G_{2}$, so the admissibility of $U$ implies that $\epsilon=0$. Now for $k=2$ we have $U_{\sigma}=[3,2,\bs{3},2,3]$, so  $\ld(H^{\sigma})=\tfrac{2}{9}<\tfrac{1}{3}$, contrary to \eqref{eq:ld-bound_h=1}; and as before the case $k\geq 3$ can be reduced to $k=2$ by a vertical swap. Thus $\Bs\phi_{+}^{-1}\subseteq A_{1}$. 
	
	Assume $\epsilon=1$. Let $\sigma_{\textnormal{\ref{item:Aut_symmetry}}}\in \Aut (X_{\phi},D_{\phi})$ be as in Lemma \ref{lem:w=1_cha=3_GK_Aut}\ref{item:Aut_symmetry}: under the isomorphism $\Aut(X_{\phi},D_{\phi})\cong D_{6}$ explained in Remark \ref{rem:hexagon}, $\sigma_{\textnormal{\ref{item:Aut_symmetry}}}$ corresponds to the symmetry through the vertical axis in Figure \ref{fig:hexagon}. Now, replacing  $\phi_{+}$ with $\sigma_{\textnormal{\ref{item:Aut_symmetry}}}\circ \phi_{+}$ we get back to Case \ref{case:qH}. 
	
	Thus we can assume $\epsilon=0$. If $\psi=\upsilon$ then $D$ is as in \ref{item:w=1_0A}. Assume $\psi\neq \upsilon$. Then the $\Bs\upsilon_{+}^{-1}=A_{1}\cap G_{1}$. The admissibility of $U$ implies that $\upsilon_{+}$ is a single blowup at $A_1\cap G_1$, so $D$ is as in \ref{item:w=1_0AA_k=2} 
	if $k=2$ and \ref{item:w=1_0AA} 
	if $k\geq 3$.\qedhere
	\end{casesp}
\end{proof}

\phantomsection\label{final}
\begin{proof}[Proof of Theorem \ref{thm:ht=3}]
	Let $\cS$ be a singularity type of a log terminal surface. 
	As in Notation \ref{not:P}, let $\Pht^{\width=w}(\cS)$ be the set of isomorphism classes of del Pezzo surfaces of rank $1$, height $3$, width $w$ and singularity type $\cS$. 
	
	Assume $\Pht^{\width=3}(\cS)\neq \emptyset$. Then by Lemma \ref{lem:w=3} $\cS$ is as in Table \ref{table:ht=3_char=0} (top part), and either $\#\Pht^{\width=3}(\cS)=1$ or $\cS$ is as in Theorem \ref{thm:ht=3}\ref{item:uniq_exotic} and $\Pht^{\width=3}(\cS)$ consists of the two surfaces from  Example \ref{ex:ht=3_pair}. 
	
	Assume $\Pht^{\width=2}(\cS)\neq \emptyset$. If $\cha\kk\neq 2$ then Lemma \ref{lem:w=2_cha_neq_2} shows that $\cS$ is as in Table \ref{table:ht=3_char=0} (part $\cha\kk\neq 2$), and $\#\Pht^{\width=2}(\cS)=1$. If $\cha\kk=2$, Lemma \ref{lem:w=2_cha=2} shows that either $\Pht^{\width=2}(\cS)$ has moduli dimension $1$ (with a representing family over $\Astst$) and $\cS$ is as in Table \ref{table:ht=3_char=2_moduli}, or $\#\Pht^{\width=2}(\cS)=1$ and $\cS$ is as in Table \ref{table:ht=3_char=2}. 
	
	Eventually, assume $\Pht^{\width=1}(\cS)\neq \emptyset$. If $\cha\kk\neq 3$ then by Lemma \ref{lem:w=1_cha-neq-3} we have  $\cha\kk\neq 2$, $\#\Pht^{\width=1}(\cS)=1$ and $\cS$ is as in the last row of Table \ref{table:ht=3_char=0}. If $\cha\kk=3$ then by Lemma \ref{lem:w=1_cha=3} $\#\Pht^{\width=1}(\cS)=1$ and  $\cS$ is as in Table~\ref{table:ht=3_char=3}.
	
	We check directly that each $\cS$ appears in at most one of the above lemmas, so for each $\cS$ only one of the sets $\Pht^{\width=w}(\cS)$ is nonempty. Moreover, the classification in \cite{PaPe_ht_2} shows that none of the above singularity types can be realized by a del Pezzo surface of rank one and height at most $2$, see Tables 7--10 loc.\ cit. Thus the set $\Pht^{\width=3}(\cS)\cup \Pht^{\width=2}(\cS) \cup \Pht^{\width=1}(\cS)$ described above equals the set $\Phtl(\cS)$ from Theorem \ref{thm:ht=3}, as needed. 
\end{proof}

\begin{proof}[Proof of Proposition \ref{prop:rigidity}]
	In cases $\width(\bar{X})=3,2$ and $1$ the numbers $h^i$ are computed in Lemmas \ref{lem:w=3_uniqueness}\ref{item:ht=3_h1}, \ref{lem:w=2_uniqueness}\ref{item:w=2_uniq-hi} and \ref{lem:w=1_uniqueness}\ref{item:w=1_hi}, respectively.
\end{proof}

\clearpage

\section{Comparison with Lacini's list}\label{sec:comparison}

Assume $\cha\kk\neq 2,3$. In this case, all (non-canonical) log terminal del Pezzo surfaces of rank one are arranged in 24 broad series described in \cite[\sec 6.1]{Lacini}. We note that the proof of Theorem \ref{thm:ht=3} is independent from \cite{Lacini}, which follows a different method, based on \cite{Keel-McKernan_rational_curves}. In this section, we check which surfaces in \cite[\sec 6.1]{Lacini} have height at least $3$, and identify those of height $3$ with the ones in Theorem \ref{thm:ht=3}. This comparison is summarized in Table \ref{table:Lacini} below, for the comparison in the other direction see Table \ref{table:ht=3_char=0}. 

We recall that the description in \cite{Lacini} does not address the question of uniqueness of del Pezzo surfaces constructed there. In fact, as we will see below, the 24 series in loc.\ cit.\ are not disjoint.
We will also see from the comparison that surfaces from Lemmas \ref{lem:w=3}\ref{item:ht=3_XA_a=3_c=3} and \ref{lem:w=2_cha_neq_2}\ref{item:ht=3_b} seem to be missing in \cite{Lacini}.

\begin{small}
\begin{table}[htbp]
\begin{minipage}[t]{.4\textwidth}\vspace{0pt}	
	\centering
		\begin{tabular}{r|l}
			\cite{Lacini} & this article \\ \hline\hline	
			1 & \ref{lem:w=1_cha-neq-3} \\
			2--4 & $\height=4$ \\ 
			5 & \ref{lem:w=2_cha_neq_2}\ref{item:ht=3_exc_type} \\
			6 & $\height=4$ \\ 
			7 & \ref{lem:w=2_cha_neq_2}\ref{item:ht=3_A0} \\ 
			8 & \ref{lem:w=3}\ref{item:ht=3_XE} \\
			9($\rA_1+\rA_5$) & \ref{lem:w=2_cha_neq_2}\ref{item:ht=3_A22} \\
			9($3\rA_2$) & $\height=4$ \\ 
			10 	& \ref{lem:w=2_cha_neq_2}\ref{item:ht=3_A3E_l=4}, \ref{item:ht=3_A3E_l=5} \\ 
			11 & \ref{lem:w=2_cha_neq_2}\ref{item:ht=3_A3E_l=4} \\
			12 & \ref{lem:w=2_cha_neq_2}\ref{item:ht=3_A0} \\
			13 & see 1st table on the right 
			\\
	14 & \ref{lem:w=3}\ref{item:ht=3_chains} \\
	15 & \ref{lem:w=3}\ref{item:ht=3_YDYD}\\
	16 & \ref{lem:w=3}\ref{item:ht=3_XY_b=4_T-2}\\
	17 (1),(2) & \ref{lem:w=3}\ref{item:ht=3_chains}\\
	17 (3) & \ref{lem:w=3}\ref{item:nu_3=1_e1}, \ref{item:nu=3_C}\\
	18& \ref{lem:w=3}\ref{item:ht=3_chains}\\
	19& \ref{lem:w=3}\ref{item:both_B}\\
	20 & $\height\leq 2$ \\
	21 & has d.e.b. \\ 
	22 & $\height \leq 2$ or has d.e.b. \\
	23 &  see 2nd table on the right \\
	24 & $\height\leq 2$
	\end{tabular}
\end{minipage}	
\begin{minipage}[t]{.5\textwidth}\vspace{0pt}
	\centering	
	\begin{tabular}{r|l}	
			\cite[13]{Lacini} & this article \\ \hline\hline	
			(1) & \ref{lem:w=3}\ref{item:rivet_AC}, \ref{item:rivet_0} \\
				(2) & \ref{lem:w=3}\ref{item:nu_3=1_c2}, \ref{item:nu_3=1_e2} \\
				 (3) & \ref{lem:w=2_cha_neq_2}\ref{item:ht=3_C2}\\
				 (4) & \ref{lem:w=2_cha_neq_2}\ref{item:ht=3_A2B}, \ref{item:ht=3_A2}\\
				 (5)--(8) & $\height=4$ \\ 
				 (9) & \ref{lem:w=2_cha_neq_2}\ref{item:ht=3_BF}, \ref{item:ht=3_BD} \\
				 (10) & \ref{lem:w=2_cha_neq_2}\ref{item:ht=3_C1}\\
				 \multicolumn{2}{c}{}\\ 
				 \multicolumn{2}{c}{}\\
		\cite[23]{Lacini} & this article \\ \hline\hline	
			 5.1(3) & \ref{lem:w=3}\ref{item:nu_3=1_s3}, \ref{item:ht=3_YG}, \ref{item:ht=3_XA_c=2},\\
			 & \ref{item:ht=3_XAA_T=[2,2]},  \ref{item:ht=3_XAA_T=[3]}, \ref{item:ht=3_chains}\\
			 5.1(4) &  \ref{lem:w=3}\ref{item:ht=3_XY_a=3}, \ref{item:ht=3_chains}\\
			 5.1(5: $\rA_1+\rA_2+\rA_5$)& \ref{lem:w=3}\ref{item:rivet_A}, \ref{item:nu_3=1_s1},
			 \ref{item:nu=3_fork_s1=1},
			\ref{item:ht=3_YDYD},\\
			&
			\ref{item:ht=3_XY_b=3_v},
			\ref{item:ht=3_XY_b=3},
			\ref{item:ht=3_chains}
			\\
			 5.1(5: $2\rA_1+2\rA_3$)& \ref{lem:w=2_cha_neq_2}\ref{item:ht=3_A3E_l=3}, \ref{item:ht=3_C3}, 
			 \ref{item:ht=3_A3F},  \ref{item:ht=3_A3EF}\\
			 5.1(5: $4\rA_2$)	& $\height=4$ \\ 
			others & $\height \leq 2$ or has d.e.b. 
		\end{tabular}
\end{minipage}	
		\caption{Del Pezzo surfaces from \cite{Lacini}, cf.\ Table \ref{table:ht=3_char=0}.}
		\label{table:Lacini}
	\end{table}
\end{small}

For completeness, we explain how to deduce the above comparison. Let $\bar{X}$ be a surface from \cite{Lacini}, and let $\cS$ be its singularity type. In each case, we first check if $\height(\bar{X})\leq 2$ or $\bar{X}$ has a descendant with elliptic boundary (this is usually apparent from the constructions in loc.\ cit.). Assume that neither of those conditions hold. If $\cS$ is not listed in Table \ref{table:ht=3_char=0}, then by Theorem \ref{thm:ht=3} we get $\height(\bar{X})\geq 4$. We will discuss such surfaces  in our forthcoming article, where we will show that in fact $\height(\bar{X})=4$. Assume now that $\cS$ is listed in Table \ref{table:ht=3_char=0}, and let $\bar{X}'$ be a corresponding surface from Theorem \ref{thm:ht=3}. The latter result shows that $\bar{X}'$ is uniquely determined by $\cS$, unless $\cS$ is the exceptional type \eqref{eq:25} from Theorem \ref{thm:ht=3}\ref{item:uniq_exotic}, in which case we get two surfaces $\bar{X}_1'$, $\bar{X}_2'$. Our goal is to show that, in case $\cS\neq\mbox{\eqref{eq:25}}$, there is an isomorphism $\bar{X}\cong \bar{X}'$, and in the exceptional case $\cS=\mbox{\eqref{eq:25}}$, there is an isomorphism $\bar{X}\cong \bar{X}_i'$ for some $i\in \{1,2\}$. (In fact, we will see that type \eqref{eq:25} is realized by two surfaces $\bar{X}_1,\bar{X}_2$ in \cite{Lacini}, and we will find isomorphisms $\bar{X}_i\cong \bar{X}_i'$ for $i=1,2$).
 	 
Let $(X,D)$ and $(X',D')$ be the minimal log resolutions of $\bar{X}$ and $\bar{X}'$, respectively. Lemmas \ref{lem:w=3}, \ref{lem:w=2_cha_neq_2} and \ref{lem:w=1_cha-neq-3} cited in Theorem \ref{thm:ht=3} show that on $X'$ there are  $(-1)$-curves $A_j'$ meeting $D'$ in components corresponding to numbers decorated by $\dec{j}$, see Section \ref{sec:notation} for the notation. Using the description of each series in \cite{Lacini}, we find $(-1)$-curves $A_{j}$ on $X$ which meet the exceptional divisor $D$ in the same way as the curves $A_{j}'$ meet $D'$, i.e.\ so that the weighted graphs of $D+\sum_{j}A_{j}$ and $D'+\sum_{j}A_{j}'$ are isomorphic. Contracting $\sum_{j}A_{j}'$ and some new $(-1)$-curves in the subsequent images of $D'$, we get a minimalization $(X',D')\to (Z',B')$ as in Proposition \ref{prop:ht=3_models}. Contracting the corresponding subdivisor of $D+\sum_{j}A_{j}$ we get a minimalization $(X,D)\to (Z,B)$ such that $\rho(Z)=\rho(Z')$ and the weighted graphs of $B$ and $B'$ are isomorphic. The list of minimal log surfaces $(Z',B')$ in Proposition \ref{prop:ht=3_models} shows that there is an isomorphism $(Z,B)\cong (Z',B')$. By the universal property of blowing up, cf.\ \cite[Lemma 2.18]{PaPe_ht_2}, this isomorphism  lifts to an isomorphism $(X,D)\cong (X',D')$, which descends to the required isomorphism $\bar{X}\cong \bar{X}'$.
\smallskip

Now, we follow this strategy for each series LDP1--LDP24 from \cite{Lacini}. We use the notation from loc.\ cit.

\paragraph{LDP 1} This is the surface from Lemma \ref{lem:w=1_cha-neq-3}: the $(-1)$-curves  over $r'$, $p'$ and $q'$ correspond to $\dec{1}$, $\dec{2}$, $\dec{3}$ in \ref{lem:w=1_cha-neq-3}.

\paragraph{LDP 2--4} These surfaces have singularity types $[2,2,3,2]+[(2)_{5},3]$ or, if $\cha\kk=5$:  $[2]+[3]+[2,3]+[2,3]+[2,2,2,2]$, $[2]+[2,3]+[2,4]+[(2)_{5},3]$ or $[2]+[2,4]+[2,3,2,2]+[2,2,2,2]$. They are of height $4$.

\paragraph{LDP 5} 
  This is the surface from Lemma \ref{lem:w=2_cha_neq_2}\ref{item:ht=3_exc_type}: the $(-1)$-curves over $L\cap C$, $p$, and the cusp of $C$ correspond to $\dec{3}$, $\dec{2}$ and $\dec{1}$; while the proper transform of the line joining $L\cap C$ with the cusp corresponds to $\dec{0}$.

\paragraph{LDP 6} This surface is of type $[2,2,3,2,2]+[3,2,2]+[2,2]$ and height $4$.

\paragraph{LDP 7}\label{LDP7} These are the surfaces from Lemma \ref{lem:w=2_cha_neq_2}\ref{item:ht=3_A0}. 
Let $r$ be the cusp of $C$ and let $s$ be the inflection point of $C$, i.e.\ $L\cap C=\{s\}$. The $(-1)$-curve over $r$ corresponds to $\dec{0}$. The proper transforms of the lines joining $r$ with $p$ and $s$ correspond to $\dec{2}$ and $\dec{3}$, respectively. It remains to find a $(-1)$-curve corresponding to $\dec{1}$, i.e.\ meeting $D$ exactly in the first and the second exceptional curve over $r$. 

Let $\sigma\colon \P^2\map \P^2$ be the standard quadratic transformation centered at $p,r,s$, and let $\ll_{p},\ll_{r},\ll_{s}$ be lines contracted by $\sigma^{-1}$ to $p,r,s$, respectively. Then $\cc\de \sigma_{*}C$ is a conic  tangent to $\ll_{r}$ and passing through $\ll_{p}\cap \ll_{s}$; and $\ll\de \sigma_{*}L$ is a line passing through $\ll_r\cap \ll_p$ and tangent to $\cc$ at $\cc\cap \ll_s$. Let $\cc'$ be the conic passing through $\ll_{r}\cap \cc$ and tangent to $\ll$, $\ll_{p}$ at their common points with $\ll_{s}$. We have $\cc'\neq \cc$ since $\cc$ is not tangent to $\ll_p$. Now, the proper transform of $\cc'$ corresponds to $\dec{1}$, as needed.

\paragraph{LDP 8}\label{LDP8} This is the surface from Lemma \ref{lem:w=3}\ref{item:ht=3_XE}. 
The $(-1)$-curves mapped to points of $B$ and $C$ correspond to $\dec{3}$ and $\dec{5}$, respectively. To see the remaining curves, let $\phi\colon \tilde{S}(2\rA_1+\rA_3)\to \F_2$ be the contraction of $B+C$ and the preimages of the $\rA_{1}$-points. Then $\phi(A)\in|2\sigma_2|$, where $|\sigma_2|$ is the linear system of positive sections, i.e.\ sections of self-intersection $2$. Let $F$ be the fiber passing through the node of $\phi(A)$, and let $H_B,H_C\in |\sigma_2|$ be the positive sections  tangent to the node of $\phi(A)$, at the branch which gets blown up twice, and passing through $\phi(B)$ and $\phi(C)$, respectively. Then the proper transforms of $F$, $H_B$ and $H_C$ are $(-1)$-curves corresponding to $\dec{1}$, $\dec{2}$ and $\dec{4}$, respectively.

\paragraph{LDP 9}\label{LDP9} Here we have two cases: $S_2=S(\rA_1+\rA_5)$ or $S_2=S(3\rA_2)$. In the latter case, the resulting surface is of type $[2,4,2,2]+[2,2,3,2,2]+[2,2]$, and height $4$. 

In case $S_2=S(\rA_1+\rA_5)$, we get the surface from Lemma \ref{lem:w=2_cha_neq_2}\ref{item:ht=3_A22}. The $(-1)$-curves over the node of $A$ and over $A\cap B$ correspond to $\dec{1}$ and $\dec{2}$, respectively. To see the remaining ones, let $C$ be the preimage of the $\rA_1$-point, and let $\phi\colon \tilde{S}(\rA_1+\rA_5)\to \P^2$ be the contraction of $B$ and the preimage of the $\rA_5$-point. Then $\phi(A)$ is a nodal cubic, and $\phi(C)$ is a conic tangent to $\phi(A)$ at some smooth point $p$, with multiplicity $6$. Let $\ll$ be the line tangent $\phi(C)$ at $p$; and let $\cc$ be the conic tangent to $\phi(C)$ at $p$ with multiplicity $4$ and passing though the node of $\phi(A)$. Then the proper transforms of $\cc$ and $\ll$ correspond to $\dec{0}$ and $\dec{3}$, respectively. 

\paragraph{LDP 10, 11} Here we get the surface from Lemma  \ref{lem:w=2_cha_neq_2}\ref{item:ht=3_A3E_l=5} if we blow up five times over the node of $A$ (so we are in case LDP 10), and the one from Lemma \ref{lem:w=2_cha_neq_2}\ref{item:ht=3_A3E_l=4} if we blow up four times (we get $k=3$ in case LDP 10 and $k=4$ in case LDP 11). The $(-1)$-curves over the node of $A$ and over $A\cap B$ correspond do $\dec{1}$ and $\dec{3}$, respectively; and the proper transforms of $F$ and $C$ from \hyperref[LDP8]{LDP 8} above correspond to $\dec{0}$ and $\dec{2}$, respectively.

\paragraph{LDP 12} This is the surface from Lemma \ref{lem:w=2_cha_neq_2}\ref{item:ht=3_A0} for $k=3$. 
The $(-1)$-curves over $A\cap B$ and over the node of $A$ correspond to $\dec{0}$ and $\dec{1}$, respectively. To see the remaining ones, let $\phi\colon \tilde{S}(\rA_1+\rA_5)\to \P^2$, $p\in \P^2$ and $\cc\subseteq \P^2$ be as in \hyperref[LDP9]{LDP 9} above; and let $\ll'$ be the line joining $p$ with the node of $\phi(A)$. Then the proper transforms of $\ll'$ and $\cc$ correspond to $\dec{2}$ and $\dec{3}$, respectively. Note that this surface has already appeared in \hyperref[LDP7]{LDP 7} above.

\paragraph{LDP 13 (1)}\label{LDP13(1)} This is the surface from Lemma \ref{lem:w=3}\ref{item:rivet_0} for $k=s-3\in \{3,4\}$ if $t=1$; and from \ref{lem:w=3}\ref{item:rivet_AC} if $t=2$. The $(-1)$-curves over the node of $A$ and over the $\rA_{1}$-point $p$ correspond to $\dec{1}$ and $\dec{3}$, respectively. 

To see the remaining ones, let $\phi\colon \tilde{S}(\rA_1+\rA_2)\to \P^2$ be the contraction of $B$ and the preimage of the $\rA_2$-point. Now $\qq\de \phi(A)$ is cubic with a node, call it $r$, and the total transform of the $\rA_{1}$-point is a line tangent to $\qq$ at an inflection point, cal it $q$. Let $\ll_{qr}$ be the line joining $q$ with $r$, let $\ll$ be the line tangent to the chosen branch $\beta$ of $\qq$ at the node $r$, and let $\cc$ be the conic such that $(\cc\cdot \qq)_{q}=2$, $(\cc\cdot \beta)_{r}=3$. 
Now the proper transforms of $\ll$, $\ll_{qr}$ and $\cc$ correspond to $\dec{2}$, $\dec{4}$ and $\dec{5}$, respectively.
\paragraph{LDP 13 (2)} This is the surface from Lemma \ref{lem:w=3}\ref{item:nu_3=1_c2} for $k=s-3$ if $t=1$, and from \ref{lem:w=3}\ref{item:nu_3=1_e2} for $k=t-1$ if $t>1$. The exceptional curves over the node of $A$ and over the $\rA_2$-point $p$ correspond to $\dec{2}$, $\dec{5}$, respectively; and the proper transforms of $\ll_{qr}$, $\cc$ and $\ll$ defined in case \hyperref[LDP13(1)]{(1)} above correspond to $\dec{1}$, $\dec{3}$ and $\dec{4}$.

\paragraph{LDP 13 (3), (4)} In case (3), we get the surface from Lemma  \ref{lem:w=2_cha_neq_2}\ref{item:ht=3_C2}. In case (4), we get the one from  \ref{lem:w=2_cha_neq_2}\ref{item:ht=3_A2}, with $k=s$ if $t=1$; and from \ref{lem:w=2_cha_neq_2}\ref{item:ht=3_A2B} if $t=2$. The $(-1)$-curves over the node of $A$ and over $p$ correspond to $\dec{1}$ and $\dec{2}$; respectively; and the proper transforms of $\cc$ and $\ll$ from \hyperref[LDP9]{LDP 9} above correspond to $\dec{0}$ and $\dec{3}$.

\paragraph{LDP 13 (5)--(8)} These surfaces are of types: $[2,2,3,2,2,2]+[3,2]+[2,2]$ in case (5); $[2,3,2,2,2]+[3,2,2,2,2]$ in case (6), $[2,3,(2)_{6}]+[3,2]$ in case (7); $[2,3,2,2]+[2,3,2]+[2,2,2]$ in case (8) if the curve $B$ meets the preimage of the $\rA_1$-point; and $[2,3,2,2,2,2]+[2,2,3]+[2]$ otherwise. They are of height $4$.

\paragraph{LDP 13 (9), (10)} In case (9) we get the surface from \ref{lem:w=2_cha_neq_2}\ref{item:ht=3_BD}, $k=s-1$ if $t=1$; and from  \ref{lem:w=2_cha_neq_2}\ref{item:ht=3_BF} if $t=2$. In case (10), we get the one from   \ref{lem:w=2_cha_neq_2}\ref{item:ht=3_C1}. The $(-1)$-curves over the node $A$ and over $p$ correspond to $\dec{1}$ and $\dec{3}$, respectively, and the proper transforms of curves $F$ and $C$ from \hyperref[LDP8]{LDP 8} above correspond to $\dec{0}$ and $\dec{2}$.

\paragraph{LDP 14--16} Here we get the surface from \ref{lem:w=3}\ref{item:ht=3_chains}, $(a,b,c,d)=(2,s-2,t,2)$ in case LDP 14 and the one from \ref{lem:w=3}\ref{item:ht=3_YDYD}, \ref{item:ht=3_XY_b=4_T-2}  in cases LDP 15, 16. 
The $(-1)$-curves over the node of $A$ and over $A\cap B$ correspond to $\dec{2}$ and $\dec{3}$ in case LDP 14, and $\dec{3}$, $\dec{2}$ in cases LDP 15, 16; while the proper transforms of $\ll$, $\ll_{qr}$ and $\cc$ defined in \hyperref[LDP13(1)]{LDP 13(1)} above correspond to $\dec{1}$, $\dec{4}$, and $\dec{5}$.

\paragraph{LDP 17}\label{LDP17} Write the morphism $\tilde{S}\to \F_{2}$ from loc.\ cit.\ as $\pi_1\circ \pi_2$, where $\pi_2$ is described in (1)--(3). Put $k=\rho(\pi_2)$. For $x\in \{p,u\}$ let $L_{x}\subseteq \tilde{S}$ be the proper transform of the exceptional $(-1)$-curve of $\pi_1$ over the point $x$, and let $L$ be the exceptional $(-1)$-curve of $\pi_{2}$. For $y\in \{p,q\}$ let $F_{y}\subseteq \tilde{S}$ be the proper transforms of the fibers of $\F_2$ passing through the point $y$; and for $x\in \{p,q\}$, $\{H,H'\}=\{A,B\}$ let $C_{x,H}\subseteq \tilde{S}$ be the proper transform of the unique section of self-intersection $2$ which is tangent to $H$ at $x$ and passes through $H'\cap F$. In cases (1)--(3) we get the following surfaces: 
\begin{description}
	\item[(1), $(s,r)=(3,2)$] \ref{lem:w=3}\ref{item:ht=3_chains} with $(a,b,c,d)=(3,2,k-1,2)$; 
	$L_p$, $L_u$, $L$, $F_p$, $F_q$ correspond to $\dec{4}$, $\dec{1}$, $\dec{3}$, $\dec{5}$, $\dec{2}$;
	\item[(1), $(s,r)=(4,1)$] \ref{lem:w=3}\ref{item:ht=3_chains} with $(a,b,c,d)=(k,2,3,2)$; 
	$L_p$, $L_u$, $L$, $F_p$, $F_q$ correspond to $\dec{3}$, $\dec{4}$, $\dec{1}$, $\dec{2}$, $\dec{5}$;
	\item[(2)] \ref{lem:w=3}\ref{item:ht=3_chains} with $(a,b,c,d)=(k,r+1,2,s-1)$: here 
	$L_p$, $L_u$, $L$, $F_p$, $C_{p,B}$ correspond to $\dec{4}$, $\dec{2}$, $\dec{1}$, $\dec{3}$, $\dec{5}$;
	\item[(3)] \ref{lem:w=3}\ref{item:nu=3_C} if $(s,r)=(3,2)$; \ref{lem:w=3}\ref{item:nu_3=1_e1} if $(s,r)=(4,1)$: here $L_p$, $L_u$, $L$, $F_p$, $C_{p,B}$ correspond to $\dec{2}$, $\dec{5}$, $\dec{1}$, $\dec{4}$, $\dec{3}$.
\end{description}

\paragraph{LDP 18} 
We use notation from \hyperref[LDP17]{LDP 17} above, and we treat the case $(s,r)=(3,2)$ as a subcase of (2). In cases (1)--(3) we get the surface from \ref{lem:w=3}\ref{item:ht=3_chains} for the following values of parameters $(a,b,c,d)$: 
\begin{description}
	\item[(1)] $(a,b,c,d)=(s+1,2,k-2,2)$: here $L_p$, $L_u$, $L$, $F_q$, $C_{q,A}$  correspond to $\dec{1}$, $\dec{5}$, $\dec{3}$, $\dec{4}$, $\dec{2}$;
	\item[(2)] $(a,b,c,d)=(2,s,r+1,k-1)$: here $L_p$, $L_u$, $L$, $F_p$, $C_{p,B}$ correspond to $\dec{2}$, $\dec{3}$, $\dec{4}$, $\dec{1}$, $\dec{5}$;
	\item[(3)] $(a,b,c,d)=(s,2,2,k-1)$: here $L_p$, $L_u$, $L$, $F_p$, $C_{p,A}$ correspond to $\dec{1}$, $\dec{2}$, $\dec{4}$, $\dec{3}$, $\dec{5}$.
\end{description}

\paragraph{LDP 19\protect\footnote{The second blowup over $u$ should be centered on the proper transform of $F$, not $A$. Moreover, the last blowup over $p$, the one \enquote{away from $A$}, is centered at the common point of the last two exceptional curves over $p$.}}  This is the surface from Lemma \ref{lem:w=3}\ref{item:both_B}. As in \hyperref[LDP17]{LDP 17} above, let $F_{p}$ be the proper transform of the fiber passing through $p$, let $L_{p},L$ be the $(-1)$-curves over $p$ such that $L_{p}$ meets the proper transform of $B$, and let $C$ be the proper transform of the section of self-intersection number $+2$ which passes through $u$ and is tangent to $B$ at $p$. Now $C$, $L$, $F_{p}$ and $L_{p}$ correspond to $\dec{1}$, $\dec{2}$, $\dec{4}$ and $\dec{5}$, respectively. 

To see $\dec{3}$, let $\phi\colon \F_{2}\map \P^2$ be a blowup at $p$ followed by the contraction of the proper transform of $F_p$ and of the negative section. Then $\phi_{*}A$, $\phi_{*}B$, $\phi_{*}F$ and the image of the exceptional curve over $p$, call it $\ll$, are lines in a general position. Let $\cc$ be the conic passing through $\ll\cap \phi_{*}A$ and tangent to $\phi_{*}F$, $\phi_{*}B$ at their common points with $\phi_{*}A$ and $\ll$, respectively. Then the proper transform of $\cc$ is the $(-1)$-curve corresponding to $\dec{3}$.

\paragraph{LDP 20} This surface has height at most $2$. Indeed, the  $\P^1$-fibration of $\F_n$ pulls back to a $\P^1$-fibration of $X$ such that the only horizontal components of $D$ are the proper transforms of two disjoint sections.

\paragraph{LDP 21} This surface has a descendant with elliptic boundary, described in Example B.11 loc.\ cit.

\paragraph{LDP 22--23} These surfaces have tigers, and are described in Theorem 6.2(2)--(4) loc.\ cit, as follows. Let $(S,C)$ be a log canonical log surface such that $S$ is a del Pezzo surface of rank one, and $C$ is reduced, and $K_{S}+C$ is anti-nef. They are classified in Section 5 loc.\ cit. The corresponding surface from LDP 22 is simply $S$, see Theorem 6.2(4) loc.\ cit. For LDP 23, let $(\tilde{S},\tilde{C}+D_{\tilde{S}})$ be the minimal log resolution of $(S,C)$, where $\tilde{C}$ is the proper transform of $C$; let $\pi\colon X\to \tilde{S}$ be a sequence of blowups over some  point $p\in \tilde{C}$, and define $D$ as $\pi^{*}(\tilde{C}+D_{\tilde{S}})$ minus the $(-1)$-curve. Contracting $D$, we get the corresponding del Pezzo surface $\bar{X}$ from LDP~23, see Theorem 6.2(3) loc.\ cit. 
The initial pair $(S,C)$ is as in one of the following results of \sec 5 loc.\ cit:

\subparagraph{5.1(1), 5.2(1)--(6), 5.3--5.5, 5.6(1),(3), 5.7(1),(3),(4), 5.9(2)--(4),(7),(8), 5.11(2),(3),(5)--(7),(9),(10)} Here by construction we have $\height(\bar{X})\leq \height(S,C)\leq 2$, or $S=\P^2$ and some pencil of lines pulls back to a $\P^1$-fibration of height at most $2$ on $(X,D)$.

\subparagraph{5.2(7), 5.6(2), 5.7(2), 5.9(1)} We claim that $\height(S,C)\leq 2$, so $\height(\bar{X})\leq 2$ as above.  

In these cases there is a birational morphism $\upsilon\colon (\tilde{S},\tilde{C}+D_{\tilde{S}})\to (Y,D_Y)$ contracting $\tilde{C}$, such that $D_Y=E_Y+T_0+T_1+T_2$, where $T_j$ are zero or disjoint admissible chains, $E_Y$ is a smooth rational curve meeting each nonzero $T_j$ once; and either $T_0=0$ or $T_1,T_2=[2]$. Moreover, $\upsilon^{-1}$ has exactly one base point, say $p$, such that $p\in E_Y$; either $p\in T_j$ or $\beta_{D_{Y}}(E_Y)\leq 1$, and if $T_0\neq 0$ then $p\in T_0$. Since $\tilde{C}+D_{\tilde{S}}$ is not negative definite, neither is $D_Y$, which implies that $E_Y^2\geq -1$. Since all components of $\tilde{C}+D_{\tilde{S}}$ have self intersection number at most $-1$, there is a decomposition $\upsilon=\theta\circ\theta_{+}$ such that either $\theta^{-1}_{*}E_{Y}=[0]$, or $\theta^{-1}_{*}(T_1+E_{Y}+T_2)=[2,1,2]$ supports a fiber of a $\P^1$-fibration whose pullback has height at most two, as needed.

\subparagraph{5.1(2), 5.10, 5.11(1),(4),(8)} In these cases, by construction $\bar{X}$ has a descendant with elliptic boundary.

\subparagraph{5.9(5)} We claim that $\bar{X}$ has a descendant with elliptic boundary, or $\height(\bar{X})\leq \height(S,C)\leq 2$ as above.

By construction, there is a morphism $(\tilde{S},\tilde{C}+D_{\tilde{S}})\to(Y,D_Y)$ contracting $\tilde{C}$ such that $(Y,D_Y)\to(S,E)$ is the minimal log resolution of a log canonical log surface such that $K_{S}+E\equiv 0$ and $E$ has exactly one singular point, which is a node. Lemma A.19(2) loc.\ cit.\ implies that the reduced total transform of $E$, call it $R$, is a rational circular connected component of $D_Y$. Let $\theta\colon (Y,D_Y)\to (Z,D_Z)$ be the contraction of $D_Y-R$ and all $(-1)$-curves in $R$ and its images. If $D_Z$ is irreducible then $(Z,D_Z)$ is a descendant of $\bar{X}$ with elliptic boundary. Otherwise, $D_Z$ is a rational circular divisor in $Z\reg$ with a component $G$ satisfying $G^2\geq 0$. As before, using the fact that all components of $\tilde{C}+D_{\tilde{S}}$ have negative self-intersection numbers, we conclude that the total transform of $G$ contains a support of a fiber of a $\P^1$-fibration of height $2$, as needed.

\subparagraph{5.9(6)} Here $(S,C)$ is as in Lemma 5.8 loc.\ cit. In cases 5.8(2),(3), we have by construction $\height(S,C)\leq 2$. In the remaining case 5.8(1), there is a birational morphism $(\tilde{S},\tilde{C})\to (S,A+B)$, where $S$ is a del Pezzo surface of rank one, and $A+B$ are two smooth rational curves meeting normally at two points; contained in the smooth locus of $S$ (see the proof of 5.8 loc.\ cit). Arguing as in case 5.2(7) above we see that $\height(\bar{X})\leq 2$.

\subparagraph{5.1(3)--(5)} Here $S$ is a canonical surface of type\footnote{Types $\rA_1+2\rA_3$, $4\rA_1$ in 5.1(5) loc.\ cit.\ should be replaced by $2\rA_1+2\rA_3$ and $4\rA_2$, respectively, see \cite[13.5]{Keel-McKernan_rational_curves}.} $\rA_1+2\rA_3$, $3\rA_2$, $\rA_1+\rA_2+\rA_5$, $2\rA_1+2\rA_3$, or $4\rA_2$. Recall that in case LDP 22 we put $\bar{X}=S$, so  $\height(\bar{X})\leq 2$ or $\bar{X}$ has a descendant with elliptic boundary.

It remains to analyze series LDP 23, with $S$ as in 5.1(3)--(5). In this case, we have a morphism $\pi\colon (X,D)\to (\tilde{S},D_{\tilde{S}}+\tilde{C})$. By the definition of $C$ in 5.1 we have $K_{S}+C=0$ and $C$ is as in \cite[19.2, 13.5]{Keel-McKernan_rational_curves}. Its proper transform $\tilde{C}$ on $\tilde{S}$ satisfies $\tilde{C}\cdot D_{\tilde{S}}=3$ and meets tips of  of three different chains which are connected components of $D_{\tilde{S}}$. Denote these connected components by $T_{i}$, $i\in \{1,2,3\}$, put $d_{i}=d(T_i)$, so $T_{i}=[(2)_{d_i-1}]$, and order each chain $T_i$ so that it meets $\tilde{C}$ in its last tip. Now the condition $K_{S}+C=0$ implies that
\begin{equation}\label{eq:C}
	\tilde{C}=-K_{\tilde{S}}-\sum_{i=1}^{3}\sum_{j=1}^{d_i-1}\frac{j}{d_{i}}T_{i}\cp{j},\qquad\mbox{so}\quad \tilde{C}=[\#D_{\tilde{S}}-7] \quad\mbox{and}\quad \sum_{i=1}^{3}\frac{1}{d_i}=1
\end{equation}
by Noether and adjunction formulas.  Write $\pi^{-1}_{*}\tilde{C}=[k]$, let $A$ be the $(-1)$-curve in $\Exc\pi$, let $U$ be the connected component of $D$ containing $\pi^{-1}_{*}\tilde{C}$, and let $\{p\}=\pi(A)\subseteq \tilde{C}$ be the base point of $\pi^{-1}$. Note that $p\in D_{\tilde{S}}$, since otherwise $U$ would contain a fork whose twigs $T_{i}$ satisfy $\sum_{i}\frac{1}{d(T_i)}=1$ by \eqref{eq:C}, which is impossible. 

 We use the construction of $\tilde{S}$ given in \cite[\sec 5A]{PaPe_ht_2} and \cite[Example 7.1]{PaPe_MT}, in particular, we use notation $A_j$  and $L_j$ for $(-1)$- and $(-2)$-curves on $\tilde{S}$ shown in \cite[Figures 19(a)--(c), 20(b)]{PaPe_ht_2} and \cite[Figure 6(a)]{PaPe_MT}.

\subparagraph{5.1(3): $S$ is of type $\rA_1+2\rA_3$, see \cite[Figure 19(b)]{PaPe_ht_2}} Say that $T_{i}=[2,2,2]$ for $i\in \{1,2\}$, so $T_3=L_1=[2]$. We have $(-1)$-curves $A_0,\dots, A_3$ meeting $D_{\tilde{S}}$ twice each, such that $A_{i-1}$ meets $L_1$ and $T_{i}\cp{2}$, $i\in \{1,2\}$, and the divisor $T_1+A_2+T_2+A_3$ is circular. Intersecting each $A_j$ with both sides of the equation \eqref{eq:C} we infer that $\tilde{C}\cdot A_j=0$, and the tips of $T_1,T_2$ meeting $\tilde{C}$ do not meet the same curve $A_{j}$. Applying an automorphism of $\tilde{S}$, if needed, we can assume that $\tilde{C}$ meets $L_1$, $L_2$ and the $(-2)$-curve over $p_3$, and $\{p\}=L_1\cap \tilde{C}$ or $\{p\}=L_2\cap \tilde{C}$. 

If $\{p\}=L_1\cap \tilde{C}$ then the admissibility of $U$ implies that $U$ is a chain and $\bar{X}$ is the surface from Lemma \ref{lem:w=3}\ref{item:nu_3=1_s3}, where the proper transforms of $A$, $A_0$, $A_1$, $A_2$ and $A_3$ correspond to $\dec{1}$, $\dec{5}$, $\dec{3}$, $\dec{2}$ and $\dec{4}$, respectively.

Assume that $\{p\}=L_2\cap \tilde{C}$. Then $U$ is a chain $[2,k,2,2,2]$ or a fork $\langle k;[2],[2,2,2],T\rangle$ with $d(T)\leq 3$; and $W\de D-U$ is a chain $[2,2,3,(2)_{k-1}]$ or $[2,2,3,(2)_{k-1}]*T^{*}$; respectively. Put $T=0$ in the former case. Let $L=[1]$ be the proper transform of the line joining $p_1$ with $\phi(C)\cap \ll_{2}$; it meets $D$ in the first and fourth component of $W$, and is disjoint from $A$, $A_0$, $A_2$ and $A_3$.
	
	Consider the case $k=2$. If $T=[3]$ then $\bar{X}$ is the surface from Lemma \ref{lem:w=3}\ref{item:ht=3_YG}: indeed, $A$, $A_0$, $A_2$, $A_3$ and $L$ correspond to $\dec{2}$, $\dec{1}$, $\dec{4}$, $\dec{5}$ and $\dec{3}$, respectively. Otherwise, we have $T=[(2)_{d-1}]$ for some $d\in \{1,2,3\}$; $W=[d,3,2,2]$ and $W+L$ is circular. If $d\in \{1,2\}$ then $W+L$ contains a chain $[d,1,(2)_{d-1}]$, which supports a fiber of a $\P^1$-fibration of height $2$. If $d=3$ then $L$ is an elliptic tie, so $\bar{X}$ has a descendant with elliptic boundary. 
	
	Consider the case $k\geq 3$. If $U$ is a chain then $\bar{X}$ is the surface from Lemma \ref{lem:w=3}\ref{item:ht=3_chains} for $(a,b,c,d)=(k,2,2,2)$: indeed, the $(-1)$-curves $A$, $A_0$, $A_1$, $A_2$ and $A_3$  correspond to $\dec{1}$, $\dec{5}$, $\dec{2}$, $\dec{3}$ and $\dec{4}$, respectively. If $U$ is a fork then $\bar{X}$ is the surface from Lemma \ref{lem:w=3}\ref{item:ht=3_XA_c=2}, \ref{item:ht=3_XAA_T=[2,2]}, or \ref{item:ht=3_XAA_T=[3]} if $T=[2]$, $[2,2]$ or $[3]$, respectively: indeed, the $(-1)$-curves $A$, $A_0$, $A_1$, $A_2$ and $A_3$ correspond to $\dec{1}$, $\dec{5}$, $\dec{2}$, $\dec{3}$ and $\dec{4}$, respectively.

\subparagraph{5.1(4): $S$ is of type $3\rA_2$, see \cite[Figure 19(a)]{PaPe_ht_2}} We have $(-1)$-curves $L_1$, $A_0$, $A_{1}$ on $\tilde{S}$ such that the divisor $T\de D_{\tilde{S}}+L_1+A_0+A_1$ is circular. 
Intersecting each component of $T$ with both sides of the equation \eqref{eq:C}, we see that $\tilde{C}$ meets every third component of $T$. Applying an automorphism of $\tilde{S}$, if needed, we can assume that $\tilde{C}$ meets $T$ in $L_2$ and in the first (resp.\ second) exceptional curve over $p_0$ (resp.\ $p_1$); and $\{p\}=\tilde{C}\cap L_2$. Thus $\phi(\tilde{C})$ is a conic passing through $p_0$ and tangent to $\cc$ at $p_1$. Let $L$ be the proper transform of the line joining $\phi(p)$ with $p_1$; and let $L'$ be the proper transform of the line tangent to $\phi(\tilde{C})$ at $\phi(p)$. 

We have $L=[1]$, $L'=[1]$. As before, we check that $U=[2,2,k,2,2]$ or $\langle k;[2,2],[2,2],[2]\rangle$; and $W\de D-U=[2,3,(2)_{k}]$ or $[2,3,(2)_{k-1},3]$, respectively. If $k=2$ then $W+L'$ is circular, so either it contains a chain $[2,1,2]$, or $L'$ is an elliptic tie, hence either $\height(\bar{X})\leq 2$ or $\bar{X}$ has a descendant with elliptic boundary. Assume $k\geq 3$. If $U$ is a fork then $\bar{X}$ is the surface from Lemma \ref{lem:w=3}\ref{item:ht=3_XY_a=3}: indeed the $(-1)$-curves, $A$, $A_0$, $A_1$, $L$ and $L'$ correspond to $\dec{3}$, $\dec{1}$, $\dec{5}$, $\dec{4}$ and $\dec{2}$, respectively. If $U$ is a chain then $\bar{X}$ is as in Lemma \ref{lem:w=3}\ref{item:ht=3_chains} for $(a,b,c,d)=(3,2,k-1,2)$: indeed, the $(-1)$-curves $A$, $A_0$, $A_1$, $L$ and $L'$ correspond to $\dec{3}$, $\dec{1}$, $\dec{5}$, $\dec{4}$ and $\dec{2}$, respectively.

\subparagraph{5.1(5) with $S$ of type $\rA_1+\rA_2+\rA_5$, see \cite[Figure 19(c)]{PaPe_ht_2}} Like before, intersecting both sides of the equation \eqref{eq:C} with $A_j$ and possibly applying an automorphism of $\tilde{S}$, we infer that $\tilde{C}\cdot A_j=0$ and $\tilde{C}$ meets $L_1$ and the $(-2)$-curve over $p_2$ and $q_3$, call them $V_2$ and $V_3$, respectively. Thus $\phi(\tilde{C})$ is the line joining $p_2$ with $q_3$. The admissibility of $D$ implies that $p\in D_{\tilde{S}}$, so we have three cases: $\{p\}=\tilde{C}\cap L_1$, $\tilde{C}\cap V_2$ or $\tilde{C}\cap V_{3}$.

\begin{casesp}
	\litem{$\{p\}=\tilde{C}\cap L_1$} Then $U=[2,k,(2)_{5}]$ or $\langle k;[2],[2],[(2)_{5}]\rangle$. If $k\geq 3$ then $\bar{X}$ is the surface from Lemma \ref{lem:w=3}\ref{item:nu_3=1_s1} or \ref{item:nu=3_fork_s1=1}, respectively: indeed, the $(-1)$-curves $A$, $A_0$, $A_1$, $A_2$ and $A_3$ correspond to $\dec{5}$, $\dec{3}$, $\dec{1}$, $\dec{2}$ and $\dec{4}$, respectively. If $k=2$ then $\bar{X}$ has a descendant with elliptic boundary: indeed, an elliptic tie is the proper transform of the conic passing through $p_1$ and tangent to the lines $\ll_2$ and $\ll_3$ at points $p_2$ and $q_3$.

\litem{$\{p\}=\tilde{C}\cap V_2$} Then $U=[2,k,2,2]$ or $\langle k;[2],[2,2],T\rangle$ for some admissible chain $T$ with $d(T)\leq 5$. If $U$ consists of $(-2)$-curves then $A_2$ is an elliptic tie, so $\bar{X}$ has a descendant with elliptic boundary. If $k\geq 3$ then $\bar{X}$ is the surface from Lemma \ref{lem:w=3}\ref{item:ht=3_chains} for $(a,b,c,d)=(2,3,k-1,2)$ if $U$ is a chain and from Lemma \ref{lem:w=3}\ref{item:ht=3_XY_b=3} if $U$ is a fork: indeed, the $(-1)$-curves $A$, $A_0$, $A_1$, $A_2$, $A_3$ correspond to $\dec{3}$, $\dec{5}$, $\dec{4}$, $\dec{2}$ and $\dec{1}$, respectively. Eventually, assume that $k=2$, $U$ is a fork and $T$ is not a $(-2)$-chain. Then $\bar{X}$ is the surface from Lemma \ref{lem:w=3}\ref{item:ht=3_YDYD} with $k=4$ if $T=[3]$ and from \ref{lem:w=3}\ref{item:ht=3_XY_b=3_v}, otherwise: indeed, the $(-1)$-curves $A$, $A_0$, $A_1$, $A_2$, $A_3$ correspond to $\dec{2}$, $\dec{5}$, $\dec{4}$, $\dec{3}$ and $\dec{1}$, respectively.

\litem{$\{p\}=\tilde{C}\cap V_3$} Then $U=[2,2,k,(2)_{5}]$. If $k\geq 3$ then $\bar{X}$ is as in Lemma \ref{lem:w=3}\ref{item:rivet_A}: indeed, the curves $A$, $A_0$, $A_1$, $A_2$ and $A_3$ correspond to $\dec{2}$, $\dec{5}$, $\dec{1}$, $\dec{4}$ and $\dec{3}$, respectively. Assume $k=2$. We claim that $\bar{X}$ has a descendant with elliptic boundary. To prove this, it is enough to find an elliptic tie, i.e.\ a $(-1)$-curve meeting $D$ only in the proper transform of $V_3$, twice. Let $V_i$ be the $(-2)$-curve over $p_i$, $i\in \{0,1\}$, and let $\tilde{\phi}\colon \tilde{S}\to \P^2$ be the contraction of $A_0+(A_1+V_1)+(A_2+L_2)+(A_3+L_3)+\tilde{C}$. Then $\tilde{\phi}(V_0)$, $\tilde{\phi}(L_1)$, $\tilde{\phi}(V_2)$ and $\tilde{\phi}(C)$ are lines in a general position, and $\tilde{\phi}(V_3)$ is the line joining the points $\tilde{\phi}(\tilde{C})=\tilde{\phi}(L_1)\cap \tilde{\phi}(V_2)$ and $\tilde{\phi}(A_3)=\tilde{\phi}(V_0)\cap \tilde{\phi}(C)$. The required elliptic tie is the proper transform of the conic passing through the point $\tilde{\phi}(A_0)=\tilde{\phi}(V_0)\cap \tilde{\phi}(L_1)$ and tangent to the lines $\tilde{\phi}(C)$, $\tilde{\phi}(V_2)$ at points $\tilde{\phi}(A_1)=\tilde{\phi}(C)\cap \tilde{\phi}(L_1)$ and $\tilde{\phi}(A_2)=\tilde{\phi}(V_2)\cap \tilde{\phi}(V_0)$. 
\end{casesp}

\subparagraph{5.1(5) with $S$ of type $2\rA_1+2\rA_3$, see \cite[Figure 20(b)]{PaPe_ht_2}} The curve $\tilde{C}$ meets either two connected components of $D_{\tilde{S}}$ of type $[2]$ and one of type $[2,2,2]$, or one of type $[2]$ and two of type $[2,2,2]$. In the first case, intersecting both sides of the equation \eqref{eq:C} with $A_2$ and $A_3$ we get that $\tilde{C}$ does not meet the connected components of $D_{\tilde{S}}$ of type $[2,2,2]$, which is false. Thus the second case holds. Applying an automorphism of $\tilde{S}$, if needed, we can assume $\tilde{C}=A_0$ and $p\not\in L_3$, so $p\in L_1\cup L_2$. The admissibility of $D$ implies that $p\in D_{\tilde{S}}$, so we have two cases: $\{p\}=\tilde{C}\cap L_1$ or $\tilde{C}\cap L_2$. 
Let $L=[1]$ be the proper transform of the line joining $p_2$ with $p_3$.

\begin{casesp}
\litem{$\{p\}= \tilde{C}\cap L_1$} If $U$ consists of $(-2)$-curves then the proper transform of $A_1$ is an elliptic tie, so $\bar{X}$ has a descendant with elliptic boundary. In the other case, $\bar{X}$ is the surface from Lemma \ref{lem:w=2_cha_neq_2}\ref{item:ht=3_A3EF} if $k=2$, from \ref{lem:w=2_cha_neq_2}\ref{item:ht=3_A3E_l=3} if $U$ is a chain, and from \ref{lem:w=2_cha_neq_2}\ref{item:ht=3_A3F} otherwise. Indeed, in each case the $(-1)$-curves $A$, $A_1$, $A_3$ and $L$ correspond to $\dec{3}$, $\dec{1}$, $\dec{0}$ and $\dec{2}$, respectively.

\litem{$\{p\}=\tilde{C}\cap L_2$} 
If $k=2$ then $\bar{X}$ has a descendant with elliptic boundary: indeed, an elliptic tie is the proper transform of a conic $\qq$ such that $(\qq\cdot \ll_1)_{p_1}=2$, $(\qq\cdot \cc)_{p_3}=3$. If $k\geq 3$ then $\bar{X}$ is the surface from Lemma \ref{lem:w=2_cha_neq_2}\ref{item:ht=3_C3}: as before, the $(-1)$-curves $A$, $A_1$, $A_3$ and $L$ correspond to $\dec{3}$, $\dec{1}$, $\dec{0}$ and $\dec{2}$, respectively.
\end{casesp}

\subparagraph{5.1(5) with $S$ of type $4\rA_2$, see \cite[Figure 6(a)]{PaPe_MT}} Applying an automorphism of $\tilde{S}$, see \cite[Lemma 7.2(d)]{PaPe_MT}, we can assume that $\tilde{C}=A_1$ and $\{p\}=\tilde{C}\cap Q$. If $k=2$ then the proper transform of $A_0$ is an elliptic tie, so $\bar{X}$ has a descendant with elliptic boundary. If $k\geq 3$, then $\bar{X}$ is of type $[2,2,k,2,2]+[(2)_{k-2},3,2]+[2,2]$ or $\langle k;[2],[2,2],[2,2]\rangle+[3,(2)_{k-3},3,2]+[2,2]$. We will see in a forthcoming article that in this case $\height(\bar{X})=4$.

\paragraph{LDP 24} By construction, surfaces in this class are of height at most two.

\clearpage
\newgeometry{top=2cm, bottom=1.8cm, left=2cm, right=2cm, twoside=false}
\section*{Tables}

We now list all singularity types of surfaces in Theorem \ref{thm:ht=3}, that is, del Pezzo surfaces of rank one and height $3$, which have no descendants with elliptic boundary. Each such singularity type is of the form $T_1+\dots+T_{k}$, where $T_{i}$ is a type of an admissible chain or an admissible fork, which is a minimal resolution graph of the corresponding singular point. We use conventions summarized in Section \ref{sec:notation}. Thus $[a_1,\dots,a_k]$ is a chain whose subsequent components have self-intersection numbers $-a_1,\dots,-a_n$; and $\langle b;T_1,T_2,T_3\rangle$ is a  fork with a branching component $[b]$ and twigs $T_1$, $T_2$, $T_3$. For an integer $k\geq 0$ we write $(2)_{k}$ for a sequence consisting of the integer $2$ repeated $k$ times, which corresponds to a chain of $k$ $(-2)$-curves. We also use notation the \enquote{$*$} and \enquote{$(2)_{-1}$} introduced in formula \eqref{eq:conventions_Tono}.

\begin{small}
\begin{table}[h!]
	\addtolength{\leftskip} {-1cm} 
	\addtolength{\rightskip}{-1cm}
	\begin{tabular}{c|r|lll|l|c|c}
		$\cha \kk$ && \multicolumn{4}{c|}{Singularities} & swaps to & \cite{Lacini} \\ \hline\hline
		any & \eqref{eq:25} & $[2,2,3,(2)_{5}]$ &$+\ [3,2]$ &&\ref{lem:w=3}\ref{item:rivet_A}, $k=3$ & \ref{ex:w=3}\ref{item:ht=3_A1+A2+A5} &23[5.1(5)] \\
		&&&&& \  or \ref{item:nu_3=1_c2}, $k=3$ && 13(2) \\
		& \ref{lem:w=3}\ref{item:rivet_A} & $[2,2,k,(2)_{5}]$ &$+\ [3,(2)_{k-2}]$ && $k\geq 4$ & & 23[5.1(5)] \\
		& \ref{item:rivet_AC} & $[2,2,3,3,(2)_{5}]$& $+\ [3,2]$ &&& & 13(1) \\		
		& \ref{item:rivet_0} & $[2,2,2,k,(2)_{k+2}]$ & $+\ [3]$ && $k\in \{3,4\}$ & & 13(1) \\
		& \ref{item:nu_3=1_c2} & $[2,2,k,(2)_{k+2}]$ & $+\ [3,2]$ && $k\in \{4,5\}$ & & 13(2)\\
		& \ref{item:nu_3=1_e1} & $[2,2,3,k,(2)_{4}]$ & $+\ [3,(2)_{k-1}]$&& $k\in \{3,4\}$ & & 17(3) \\
		& \ref{item:nu_3=1_e2} & $[2,k,3,(2)_{5}]$ & $+\ [2,3,(2)_{k-2}]$ && $k\in \{3,4\}$ & & 13(2) \\
		& \ref{item:nu_3=1_s1} & $[2,k,(2)_{5}]$ & $+\ [2,3,(2)_{k-2}]$ && $k\geq 3$ & & 23[5.1(5)] \\
		& \ref{item:nu_3=1_s3} & $[2,2,2,k,2,2,2]$ & $+\ [3,(2)_{k-1}]$ && $k\geq 3$ & & 23[5.1(3)] \\
		& \ref{item:nu=3_C} & $[2,3,2,3,2,2,2]$ & $+\ [2,3,2,2]$ &&& & 17(3) \\
		& \ref{item:nu=3_fork_s1=1} & $\langle k,[(2)_{5}],[2],[2]\rangle$ & $+\ [2,3,(2)_{k-3},3]$ && $k\geq 3$ & & 23[5.1(5)] \\\cline{3-8}
		& \ref{item:both_B} &  $\langle 2;[2,2,3],[2],[2]\rangle$ & $+\ [3,3,2,2,2]$ &&& \ref{ex:w=3}\ref{item:ht=3_2A4}
		& 19 \\
		& \ref{item:ht=3_YG}, \ref{item:ht=3_XAA_T=[3]} & $\langle k;[2,2,2],[3],[2]\rangle$ & $+\ [2,3,(2)_{k-2},3,2,2]$ && $k\geq 2$ & & 21[5.1(3)] \\
		& \ref{item:ht=3_YDYD} & $\langle 2;[2,2],[3],[2]\rangle$ & $+\ [2,k,(2)_{k}]$ && $k\in \{4,5,6\}$ & & 15, 23[5.1(5)] \\
		& \ref{item:ht=3_XY_b=3_v}, \ref{item:ht=3_XY_b=3}& $\langle k;T,[2,2],[2]\rangle$ & $+\ [(2)_{4},3,(2)_{k-2}]*T^{*}$ && $k\geq 3$, $d(T)\leq 5$; or & & 23[5.1(5)] \\
		&&&&& $k=2$, $ T\neq[(2)_{l}]$,& & \\
		&&&&& $\ \ d(T)\in \{4,5\}$ & & \\
		& \ref{item:ht=3_XE} & $\langle 2;[3,2,2],[2],[2]\rangle$ & $+\ [2,3,2,2]$ &&& & 8 \\
		& \ref{item:ht=3_XA_a=3_c=3} & $\langle 3;[2,2,2,2],[2],[2]\rangle$ & $+\ [3,2,4,2,2]$ && & 
		  & - \\
		& \ref{item:ht=3_XA_c=2} & $\langle k;[2,2,2],[2],[2]\rangle$ & $+\ [3,(2)_{k-2},3,2,2]$ && $k\geq 3$ & & 23[5.1(3)] \\
		& \ref{item:ht=3_XAA_T=[2,2]} & $\langle k;[2,2,2],[2,2],[2]\rangle$ & $+\ [4,(2)_{k-2},3,2,2]$ && $k\geq 3$ & & 23[5.1(3)] \\
		& \ref{item:ht=3_XY_a=3} & $\langle k;[2,2],[2,2],[2]\rangle$ & $+\ [3,(2)_{k-1},3,2]$ && $k\geq 3$ & & 23[5.1(4)] \\
		& \ref{item:ht=3_XY_b=4_T-2} & $\langle 3;[2,2],[2],[2]\rangle$ & $+\ [3,4,(2)_{5}]$ && & & 16 \\
		& \ref{item:ht=3_chains} & $[(2)_{c-1},b,d,a,(2)_{b-1}]$ & $+\ [(2)_{a-1},c+1,(2)_{d}]$ && see Table \ref{table:abcd} & & many, see \sec\ref{sec:comparison} \\ \hline \hline
		$\neq 2$ & \ref{lem:w=2_cha_neq_2}\ref{item:ht=3_exc_type}  & $[2,3,(2)_{5}]$ &$+\ [3,2,2]$ &&&
		\ref{ex:w=2_cha_neq_2}\ref{item:ht=3_exception} & 5 \\ \cline{3-8}		
		& \ref{item:ht=3_b} & $\langle 2;[2,2],[3],[2]\rangle$  & $+\ [3,(2)_{4}]$ &&& \ref{ex:w=2_cha_neq_2}\ref{item:w=2_A1+A2+A5} & - \\ \cline{3-8}
		& \ref{item:ht=3_A22}  & $[2,3,(2)_{5}]$ & $+\ [2,4,2,2]$ &&& \ref{ex:w=2_cha_neq_2}\ref{item:w=2_2A1+2A3} & 9 \\
		& \ref{item:ht=3_A2B}  & $[2,3,3,2,2]$ & $+\ [2,3,2,2,2,2]$ && & & 13(4) \\
		& \ref{item:ht=3_A2}  & $[2,2,k,(2)_{k-1}]$ & $+\ [3,2,2,2,2]$ && $k\in \{3,4\}$ & & 13(4) \\
		& \ref{item:ht=3_C2} & $[2,2,3,(2)_{6}]$ & $+\ [3]$ &&&
		 & 13(3) \\
		& \ref{item:ht=3_A3E_l=3}  & $[(2)_{k-2},3,2,2]$ & $+\ [2,k,2,2,2]$ & $+\ [2]$ & $k\geq 3$ & & 23[5.1(5)] \\
		& \ref{item:ht=3_A3E_l=4}  & $[(2)_{k-2},4,2,2,2]$ & $+\ [2,k,2,2,2]$ & $+\ [2]$ & $k\in \{3,4\}$ & & 10, 11 \\
		& \ref{item:ht=3_A3E_l=5}  & $[2,5,2,2,2,2]$ & $+\ [2,3,2,2,2]$ & $+\ [2]$ && & 10 \\
		& \ref{item:ht=3_BF}  & $[2,3,3,2,2,2]$ & $+\ [2,3,2,2]$ & $+\ [2]$ && & 13(9) \\
		& \ref{item:ht=3_BD}  & $[2,2,k,(2)_{k}]$ & $+\ [3,2,2]$ & $+\ [2]$ & $k\in \{3,4\}$ & & 13(9) \\
		& \ref{item:ht=3_C1}  & $[2,2,2,3,(2)_{4}]$ & $+\ [3]$ & $+\ [2]$ && & 13(10) \\
		& \ref{item:ht=3_C3}  & $[2,2,2,k,2,2,2]$ & $+\ [3,(2)_{k-2}]$ & $+\ [2]$ & $k\geq 3$ & & 23[5.1(5)] \\
		& \ref{item:ht=3_A0}  & $\langle 2;[(2)_{k-2}],[3],[2]\rangle$ & $+\ [2,k,2,2,2]$ & $+\ [2]$ & $k\in \{3,4\}$ & & 7, 12 \\
		& \ref{item:ht=3_A3F}  & $\langle 2;[2,2,2],[3],[2]\rangle$ & $+\ [2,4,2,2]$ & $+\ [2]$ && & 23[5.1(5)] \\ 
		& \ref{item:ht=3_A3EF}  & $\langle k;T,[2,2,2],[2]\rangle $ & $+\ [2,2,3,(2)_{k-2}]*T^{*}$ & $+\ [2]$ & $k\geq 3$ $d(T)\leq 3$ & & 23[5.1(5)] \\
		\hline \hline
		$\neq 2,3$ & \ref{lem:w=1_cha-neq-3} & $[2,3,2,2,2]$ & $+\ [3,2,2,2]$ & $+\ [2]$ && \ref{ex:w=1}\ref{item:w=1_cha-neq-3} & 1 \\
	\end{tabular}

	\caption{Del Pezzo surfaces of rank one and height $3$, not admitting descendants with elliptic boundary, occurring in $\cha\kk\neq 2,3$. Type \eqref{eq:25} is realized by two surfaces, see  Example \ref{ex:ht=3_pair}, all remaining types are realized by exactly one, up to an isomorphism.}
	\label{table:ht=3_char=0}
\end{table}
\clearpage

\begin{table}
	\begin{tabular}{r|lllllll|l|c}
		& \multicolumn{8}{c|}{Singularities \qquad($\cha\kk=3$)} & swaps to \\ \hline\hline
		\ref{lem:w=1_cha=3}\ref{item:w=1_0_id} & 
		$[2,2]$ & $+\ [2,2]$ &$+\ [2,2]$ & $+\ [3]$ & $+\ [3]$ & $+\ [3]$ & $+\ [3]$ && \ref{ex:w=1}\ref{item:w=1_cha=3_not-GK} \\
		\ref{item:w=1_0} & $[3,k,2,2]$ & $+\ [4,(2)_{k-2}]$ & $+\ [2,2]$ &$+\ [2,2]$ & $+\ [3]$ & $+\ [3]$ && $k\geq 2$ & \\
		\ref{item:w=1_0B_k=2} & $\langle 2;[2,2],[3],[2]\rangle$ & $+\ [2,2]$ & $+\ [2,2]$ &$+\ [5]$ & $+\ [3]$ & $+\ [3]$ &&& \\	
		\ref{item:w=1_0B} & $\langle k;[2,2],[3],[2]\rangle$ & $+\ [4,(2)_{k-3},3]$ & $+\ [2,2]$ &$+\ [2,2]$ & $+\ [3]$ & $+\ [3]$ && $k\geq 3$ & \\	
		\ref{item:w=1_0Y} & $[3,2,2,2]$ & $+\ [3,2,2,2]$ & $+\ [2,2]$ &$+\ [5]$ & $+\ [3]$ &&&& \\
		\ref{item:w=1_0Z} & $[4,2,2,2]$ & $+\ [3,2,2,2]$ & $+\ [2,2]$ &$+\ [4]$ & $+\ [3]$ &&&& \\		
		\ref{item:w=1_0X} & $[3,2,2,2]$ & $+\ [4,2,3]$ & $+\ [3,2]$ &$+\ [2,2]$ & $+\ [3]$  &&&& \\	
		\ref{item:w=1_0A} & $[2,3,(2)_{k-2}]$ & $+\ [3,k,3]$ &  $+\ [2,2]$ &$+\ [2,2]$ & $+\ [3]$ & $+\ [3]$ && $k\geq 2$ & \\
		\ref{item:w=1_0AA_k=2} & $\langle 2;[3],[3],[2]\rangle$ & $+\ [4,2]$ & $+\ [2,2]$ &$+\ [2,2]$ & $+\ [3]$ & $+\ [3]$  &&& \\
		\ref{item:w=1_0AA} & $\langle k;[3],[3],[2]\rangle$ & $+\ [3,(2)_{k-3},3,2]$ & $+\ [2,2]$ &$+\ [2,2]$ & $+\ [3]$ & $+\ [3]$ && $k\geq 3$  & \\
		\cline{2-10}					
		\ref{item:w=1_1} & $[3,k,2,2]$ & $+\ [2,3,(2)_{k-2}]$ & $+\ [2,2]$ &$+\ [3]$ & $+\ [2]$ &&& $k\geq 2$ &
		\ref{ex:w=1}\ref{item:w=1_cha=3_GK}	\\
		\ref{item:w=1_1B_k=2} & $\langle 2;[2,2],[3],[2]\rangle$  & $+\ [4,2]$ & $+\ [2,2]$ &$+\ [3]$ & $+\ [2]$ &&&& \\
		\ref{item:w=1_1B} & $\langle k;[2,2],[3],[2]\rangle$ & $+\ [3,(2)_{k-3},3,2]$ & $+\ [2,2]$ &$+\ [3]$ & $+\ [2]$ &&& $k\geq 3$ & \\
		\ref{item:w=1_1Y} & $[3,2,2,2]$ & $+\ [3,2,2,2]$ & $+\ [4,2]$ &$+\ [2]$ &&&&&\\
		\ref{item:w=1_1Z} & $[2,3,2,2,2]$ & $+\ [3,2,2,2]$ & $+\ [4]$ &$+\ [2]$	&&&&&\\	
		\ref{item:w=1_1X} & $[3,2,3,2]$ & $+\ [3,2,2,2]$ & $+\ [3,2]$ &$+\ [2]$ &&&&&\\
		\ref{item:w=1_1C} & $[2,2,k,2,2]$ & $+\ [4,(2)_{k-2}]$ & $+\ [2,2]$ &$+\ [3]$ & $+\ [2]$ &&& $k\geq 3$ &\\
		\ref{item:w=1_1CB} & $\langle k;[2,2],[2,2],[2]\rangle$ & $+\ [4,(2)_{k-3},3]$ & $+\ [2,2]$ &$+\ [3]$ & $+\ [2]$ &&& $k\geq 3$ &
	\end{tabular}
	\caption{Del Pezzo surfaces of rank one and height $3$, not admitting descendants with elliptic boundary, occurring only in $\cha\kk=3$. Each type is realized by a unique surface.}
	\label{table:ht=3_char=3}
\end{table}

\begin{table}
	\begin{tabular}{r|lllllll|l|c|c}
		& \multicolumn{8}{c|}{Singularities \qquad($\cha\kk=2$)} & swaps to & $G$ \\ \hline\hline
		\ref{lem:w=2_cha=2}\ref{item:ht=3_nu=4_id} & $\langle   k  ;[2] ,[2] ,[2] \rangle$ & $+\ [3,(2)_{k-1}]$ & $+\ [2] $ & $+\ [2]$ & $+\ [2]$ &&& $k\geq 3$ &\ref{ex:w=2_cha=2}\ref{item:swap-to-nu=4_small} & $S_3$ 
		\\ \ref{item:ht=3_nu=4_id_C} & $\langle 2  ;[3] ,[2] ,[2] \rangle$ & $+\ [2,3,2,2]$ & $+\ [2] $ & $+\ [2]$ & &&&& & $\Z/2$  
		\\ \cline{2-11} \ref{item:ht=3_nu=4_id_X} & 
		$\langle  k  ;[2] ,[2] ,[2] \rangle$ & $+\ [3,(2)_{k-2}] $ & $+\ [4]$ & $+\ [2]$ & $+\ [2]$ & $+\ [2] $
		& $+\ [2]$ & $k\geq 3$ & \ref{ex:w=2_cha=2}\ref{item:swap-to-nu=4_large} & $S_3$
		\\ \ref{item:ht=3_nu=4_V} & $\langle 3  ;[2] ,[2] ,[2] \rangle$ & $+\ [2, 3  ,2 ,2]$ & $+\ [5]$ & $+\ [2]$ & $+\ [2] $ & $+\ [2]$ &&& & 
		\\ \ref{item:ht=3_nu=4_U} & $\langle 3  ;[2] ,[2] ,[2] \rangle$ & $+\ [4,2,2]$ & $+\ [4,2]$ & $+\ [2]$ & $+\ [2] $
		& $+\ [2] $ &&& &
		 \\ 	\cline{11-11}  \ref{item:ht=3_nu=4_id_XC} & $\langle 2  ;[3] ,[2] ,[2] \rangle$ & $+\ [4,2,2]$ & $+\ [3]$ & $+\ [2]$ & $+\ [2]$ & $+\ [2]$ &&& & $\Z/2$
		\\ \ref{item:ht=3_nu=4_XC} & $\langle 2  ;[3] ,[2] ,[2] \rangle$ & $+\ [5,2 ,2]$ & $+\ [3,2,2]$ & $+\ [2]$ & $+\ [2]$ & $+\ [2]$ &&& & 
		\\ \cline{2-11}  \ref{item:ht=3_nu_3_off} &
		$\langle k ;[2,2,2],[3],[2] \rangle$ & $+\ [3  ,2 ,2]$ & $+\ [(2)_{k-2}] $ & $+\ [2] $ &&&& $k\geq 2$ & \ref{ex:w=2_cha=2}\ref{item:swap-to-nu=3} & $\{\id\}$
		\\ \ref{item:ht=3_nu_3_V_fork_off} &
		$\langle k ;[2,3],[3],[2]\rangle$ & $+\ [2 ,3  ,2,2 ,2]$ & $+\ [(2)_{k-2}]  $ & $+\ [2]$ &&&& $k\geq 2$ & &
	\end{tabular}
	\caption{Log terminal del Pezzo surfaces of rank one, height $3$, not admitting descendants with elliptic boundary, occurring only in $\cha\kk=2$. Part 1: types with moduli dimension $1$ (
	represented by a family over $\Astst$, with symmetry group $G$ as above).}
	\label{table:ht=3_char=2_moduli}
\end{table}

	\begin{longtable}{r|llll|l}\phantomsection\label{table:ht=3_char=2}
		& \multicolumn{4}{c}{Singularities \qquad($\cha\kk=2$)} & \\ \hline\hline		
		\endhead
		\caption{Log terminal del Pezzo surfaces of rank one, height $3$, not admitting descendants with elliptic boundary, occurring only in $\cha\kk=2$. Part 2: types realized by a unique surface. They all swap  vertically to the surface from Example \ref{ex:w=2_cha=2}\ref{item:swap-to-nu=3}.\hfill (\emph{continues on the next pages})}
		\endfirstfoot
		\caption{(\emph{continued}).}
		\endfoot
		\ref{lem:w=2_cha=2}\ref{item:ht=3_nu=3_XY'E} & $[(2)_{k-2},3 ,3  ,2,2 ,2]$ & $+\ [2,2 , 3 ,k ,2]$ & $+\ [4,2]$ && $k\in \{2,3\}$
		\\ \ref{item:ht=3_nu=3_XY'U_eps=0} &
		$[2,k , 3 ,2 ,2]$ & $+\ [3 ,2  ,2,2 ,2]$ & $+\ [4  ,(2)_{k-2}]$ && $7\geq k \geq 2$ 
		\\ \ref{item:ht=3_nu=3_XY'U_eps=1} &
		$[3  ,k , 3 ,2 ,2]$ & $+\ [3 ,2  ,2,2 ,2]$ & $+\ [3,(2)_{k-2}] $ && $5\geq k \geq 2$ 
		\\ \ref{item:ht=3_nu=3_XY'BE} &
		$[3 ,3  ,2,2 ,2]$ & $+\ [2,3,2,2]$ & $+\ [4,2,2]$ & &
		\\ \ref{item:ht=3_nu=3_XY'BB} &
		$[3 ,2  ,2,3 ,2]$ & $+\ [2 ,4 ,2 ,2]$ & $+\ [3  ,2 ,2]$ & &
		\\ \ref{item:ht=3_nu=3_XY'BD} &	
		$[(2)_{k-2},3,(2)_{4}]$ & $+\ [4,k ,2]$ & $+\ [3  ,2 ,2]$ && $k\in \{3,4,5\}$ 
		\\ \ref{item:ht=3_nu=3_XY'B_T=0} &
		$[3 ,2  ,2,2 ,2]$ & $+\ [2,2,4,(2)_{k-2}]$ & $+\ [3  ,k ,2]$ && $9\geq k\geq 2$ 
		\\ \ref{item:ht=3_nu=3_ZY'A} &
		$[2,2 , 3,(2)_{6}]$ & $+\ [4]  $ & $+\ [3] $ && 
		\\ \ref{item:ht=3_nu=3_ZY'B} &
		$[4,k,(2)_{5}]$ & $+\ [3  ,2 ,2]$ & $+\ [(2)_{k-2},3] $ && $k\in \{2,3\}$
		\\ \ref{item:ht=3_nu=3_ZY'BD} &
		$[2 ,4,(2)_{6}]$ & $+\ [3  ,3 ,2]$ & $+\ [3] $ &&
		\\ \ref{item:ht=3_nu=3_YYB} &
		$[2,k ,(2)_{5}]$ & $+\ [5,(2)_{k-2}] $ & $+\ [3  ,2 ,2]$ && $k\in \{4,5\}$ 
		\\ \ref{item:ht=3_nu=3_YA} &
		$[2,3 ,(2)_{5}]$ & $+\ [2,k ,4 ,2] $ & $+\ [4  ,(2)_{k-2}] $  && $k\in \{2,3\}$ 
		\\ \ref{item:ht=3_nu=3_YB_T=0} &
		$[2,3 ,(2)_{5}]$ & $+\ [2 ,5 ,(2)_{k-2}] $ & $+\ [3  ,k ,2] $ && $k\in \{2,3,4\}$ 
		\\ \ref{item:ht=3_nu=3_UA} &
		$[2,2 ,2,3  ,2,2 ,2]$ & $+\ [4,2,2]$ & $+\ [4,2]$ &&
		\\ \ref{item:ht=3_nu=3_UB} &
		$[2,2 ,2,3  ,2,2 ,2]$ & $+\ [2  ,3,2 ,2]$ & $+\ [5]$ &&
		\\ \ref{item:ht=3_nu=3_XX'} & $[2,2,3,(2)_{k-2}]$ & $+\ [2,2 , 3 ,2 ,2]$ & $+\ [3 ,  k  ,3 ]$ &&  $k\in \{2,3\}$
		\\[.5em] 
		 \ref{item:ht=3_nu=3_XY'UV} &
		$\langle 2;[2,2 ,3 ],[2],[2]\rangle$ &	$+\ [3 ,2  ,2,2 ,2]$ & $+\ [5]$ & &
		\\ \ref{item:ht=3_nu=3_XY'BA} &	
		$\langle 2;[3 ,2  ,2],[2],[2]\rangle$ & $+\ [5,2 ,2]$ & $+\ [3  ,2 ,2]$ && 
		\\ \ref{item:ht=3_nu=3_XY'BC} &
		$\langle 2;[4],[2] ,[2]\rangle$ & $+\ [4 ,2  ,2,2 ,2]$ & $+\ [3  ,2 ,2]$	 &&
		\\ \ref{item:miss} &
		$\langle 2;[3],[2],[2]\rangle$ & $+\ [2,3,2,2,2,2]$ & $+\ [5,3,2]$ && 		
		\\ \ref{item:ht=3_nu=3_XY'B_k=2} &
		$\langle 2;[(2)_{k-4}],[3]  ,[2]\rangle$ & $+\ [3 ,2  ,2,2 ,2]$ & $+\ [k,2 ,2]$ && $k\in \{5,6,7\}$ 
		\\ \ref{item:ht=3_nu=3_XY'B_k=2_T=[3]} &
		$\langle 2;[3],[3]  ,[2] \rangle$ & $+\ [3 ,2  ,2,2 ,2]$ & $+\ [2 ,5 ,2 ,2]$ & &
		\\ \ref{item:ht=3_nu=3_XY'B_k=3} &
		$\langle 3;[(2)_{k-2}],[3]  ,[2]\rangle$ & $+\ [3 ,2  ,2,2 ,2]$ & $+\ [k ,4 ,2 ,2]$ && $k\in \{3,4\}$ 
		\\ \ref{item:ht=3_nu=3_XY'B_k=4} &
		$\langle 4;[3],[2]  ,[2] \rangle$ & $+\ [3 ,2  ,2,2 ,2]$ & $+\ [3 ,2,4 ,2 ,2]$ &&
		\\ \ref{item:ht=3_nu=3_ZY'BE} &
		$\langle 2;[3],[2] ,[2]  \rangle$ & $+\ [5,(2)_{6}]$ & $+\ [3] $ &&		
		\\ \ref{item:ht=3_nu=3_YB_T-neq-0} &
		$\langle 2;[(2)_{k-5}],[3]  ,[2]\rangle$ & $+\ [2,3 ,(2)_{5}]$ & $+\ [2 ,k] $ &&  $k\in \{6,7\}$ 
		\\[.5em]
	    \ref{item:ht=3_nu_3_VZ_chain_k=3} &
		$[2 ,3  ,2,2 ,2]$ & $+\ [3,k,3,2]$ & $+\ [3,(2)_{k-2}] $ & $+\ [2] $ &  $k\geq 2$ 
		\\ \ref{item:ht=3_nu_3_VZ_chain_k=4} &
		$[2 ,4  ,2,2 ,2]$ & $+\ [3,k,3,2,2]$ & $+\ [3,(2)_{k-2}] $ & $+\ [2] $ & $5\geq k \geq 2$ 
		\\ \ref{item:ht=3_nu_3_VZ_chain_m=2} &
		$[2 ,  k  ,2,2 ,2]$ & $+\ [3,2,3,(2)_{k-2}]$ & $+\ [3] $ & $+\ [2] $ & $k\in \{5,6,7\}$ 
		\\  \ref{item:ht=3_nu_3_YV_chain_l=2,k=3} &
		$[2 ,3  ,2,2 ,2]$ & $+\ [2,k,3,2]$ & $+\ [4,(2)_{k-2}]$ & $+\ [2] $ & $k\geq 2$
		\\ \ref{item:ht=3_nu_3_YV_chain_l=2,k=4} &
		$[2 ,4  ,2,2 ,2]$ & $+\ [2,k,3,2  ,2]$ & $+\ [4,(2)_{k-2}]$ & $+\ [2] $ & $7\geq k \geq 2$ 
		\\ \ref{item:ht=3_nu_3_YV_chain_l=2,k>4} &
		$[2 ,  k  ,2,2 ,2]$ & $+\ [2,2,3,(2)_{k-2}]$ & $+\ [4]$ & $+\ [2] $ & $k\in \{5,6\}$ 		
		\\ \ref{item:ht=3_nu_3_YV_chain_k=2,l=3} &
		$[2,3,2,2,2]$ & $+\ [2,4,(2)_{k-2}]$ &  $+\ [3,k,2]$ & $+\ [2]$ & $k\geq 2$  
		\\ \ref{item:ht=3_nu_3_YV_chain_l=4} & 
		$[2,4,2,2,2]$ & $+\ [2,2,4,(2)_{k-2}]$ & $+\ [3,k,2]$ & $+\ [2]$ & $9\geq k\geq 2$	
		\\ \ref{item:ht=3_nu_3_YV_chain_k=2,l>4,m=2} &
		$[2,k,2,2,2]$ & $+\ [4,(2)_{k-2}]$ & $+\ [3,2,2]$ & $+\ [2]$ &  $k\geq 5$  
		\\ \ref{item:ht=3_nu_3_YV_chain_k=2,l>4,m=3} &
		$[2,k,2,2,2]$ & $+\ [2,4,(2)_{k-2}]$ & $+\ [3,3,2]$ & $+\ [2]$ & $10\geq k\geq 5$  
		\\ \ref{item:ht=3_nu_3_YV_chain_k=2,l>4,m=4} &
		$[2,5,2,2,2]$ & $+\ [2,2,4,2,2,2]$ & $+\ [3,4,2]$ & $+\ [2]$ &
		\\ \ref{item:ht=3_nu_3_YV_chain_k=3,l>2,m=2} & 
		$[2,3,2,k,2]$ & $+\ [2,3,2,2]$ & $+\ [4,(2)_{k-2}]$ &$+\ [2]$ & $7\geq k\geq 3$  
		\\ \ref{item:ht=3_nu_3_YV_chain_k>3,l=3,m=2} & 
		$[2,k,2,3,2]$ & $+\ [(2)_{k-2},3,2,2]$ & $+\ [4,2]$ & $+\ [2]$ & $k\in \{4,5\}$ 
		\\ \ref{item:ht=3_nu_3_WZ_l=2} &
		$[3,k,(2)_{5}]$ & $+\ [3,(2)_{k-2}] $ & $+\ [3] $ & $+\ [2] $ & $k\geq 2$  
		\\ \ref{item:ht=3_nu_3_WZ_k=2} &
		$[3 ,2  ,2,k ,2 ,2,2]$ & $+\ [3,(2)_{k-2}] $ & $+\ [3] $ & $+\ [2] $ & $k\in \{3,4\}$ 
		\\ \ref{item:ht=3_nu_3_WX_l=2} &
		$[2 ,2  ,2,k ,2 ,2 ,2]$ & $+\ [3,(2)_{k-2}]$ & $+\ [4] $ & $+\ [2] $ & $k\geq 3$ 
		\\ \ref{item:ht=3_nu_3_WX_k=2} &
		$[2 ,  k, (2)_{5}]$ & $+\ [4 ,(2)_{k-2}] $ & $+\ [3]$ & $+\ [2] $ & $k\geq 3$ 
		\\ \ref{item:ht=3_nu_3_WX_k,l=3} &
		$[2 ,3  ,2,3 ,2 ,2 ,2]$ & $+\ [4  ,2 ]$ & $+\ [3 ,2]$ & $+\ [2] $ &		
		\\ \ref{item:ht=3_nu_3_WY_k=2} &
		$[3,k,2,2,2]$ & $+\ [3  ,2 ,2]$ & $+\ [3,(2)_{k-2}] $ & $+\ [2] $ & $k\geq 2$ 
		\\ \ref{item:ht=3_nu_3_WY_k=3} &
		$[3,k,2,3,2]$ & $+\ [2  ,3,2 ,2]$ & $+\ [3,(2)_{k-2}] $ & $+\ [2] $ & $5\geq k\geq 2$ 
		\\ \ref{item:ht=3_nu_3_WY_k>3} &
		$[2,k,2,l,3]$ & $+\ [(2)_{k-2} ,3,2 ,2]$ & $+\ [(2)_{l-2},3] $ & $+\ [2] $ & $6\geq k\geq 4$, $l\in \{2,3\}$ 	
		\\ \ref{item:qL3_A0_chain} &  
		$[3,k,2]$ & $+\ [3,2,2]$ & $+\ [(2)_{k-2},3,2,2]$ & $+\ [2]$ & $k\geq 2$  
		\\ \ref{item:qL3_k1} &
		$[3,k,2]$ & $+\ [2,3,2,2]$ & $+\ [(2)_{k-2},3,3,2]$ & $+\ [2]$ & $9\geq k\geq 3$  
		\\ \ref{item:qL3_k2} &
		$[3,3,2]$ & $+\ [(2)_{k-2},3,2,2]$ & $+\ [2,3,k,2]$ & $+\ [2]$ & $9\geq k\geq 3$ 
		\\ \ref{item:qL3_4} &
		$[3,4,2]$ & $+\ [2,2,3,2,2]$ & $+\ [2,2,3,4,2]$ & $+\ [2]$ &
		\\[.5em] 
		 \ref{item:ht=3_nu_3_VZ_fork} &
		$\langle k;[2,3],[3],[2]\rangle $ & $+\ [2 ,3  ,2,2 ,2]$ & $+\ [3,(2)_{k-3},3] $ & $+\ [2] $ &
		\\ \ref{item:ht=3_nu_3_YV_fork_l=2,m=2} &
		$\langle 2;[2,2,3],[2],[2]\rangle$ &	$+\ [2,4,2,2 ,2]$ &  $+\ [5]$ & $+\ [2] $ &
		\\ \ref{item:ht=3_nu_3_YV_fork_l=2,m>2} &
		$\langle k;T,[2,3],[2]\rangle$ & $+\ [2 ,3  ,2,2 ,2]$ & $+\ [4,(2)_{k-2}]*T^{*}$ &  $+\ [2] $ & $k\geq 2$, $d(T)\leq 3$ 
		\\ \ref{item:ht=3_nu_3_YV_fork_l=3} &
		$\langle k;T,[3],[2]\rangle$ &	$+\ [2,3,2 ,2  ,2]$ & $+\ [2,4,(2)_{k-2}]*T^{*}$ &  $+\ [2] $ & $k\geq 2$, $d(T)\leq 5$
		\\ \ref{item:ht=3_nu_3_YV_fork_T=[2]_m=2} &
		$\langle
		2;[3],[2] ,[2]\rangle$ &	$+\ [2 ,k  ,2,2 ,2]$ &  $+\ [5,(2)_{k-2}] $ & $+\ [2] $ & $k\geq 4$ 
		\\ \ref{item:ht=3_nu_3_YV_fork_T=[2]_m>2} &
		$\langle
		k;[3],[2] ,[2]\rangle$ & 	$+\ [2 ,4  ,2,2 ,2]$ & $+\ [3,(2)_{k-3},4,2,2] $ &  $+\ [2] $ & $k\in \{3,4\}$ 
		\\ \ref{item:ht=3_nu_3_YV_fork_T=[3]} &
		$\langle 2;[3],[3] ,[2]\rangle$ & 	$+\ [2,4,2 ,2  ,2]$ & $+\ [2 ,5 ,2,2] $ & $+\ [2] $ &
		\\ \ref{item:ht=3_nu_3_YV_fork_T=[2,2]_m=3} &
		$\langle
		3;[2,2],[3] ,[2]\rangle$ &	$+\ [2 ,4  ,2,2 ,2]$ &  $+\ [4 ,4 ,2,2] $ & $+\ [2] $ &
		\\ \ref{item:ht=3_nu_3_YV_fork_T=[2,2]_m=2} &
		$\langle
		2;[2,2],[3] ,[2]\rangle$ &	$+\ [2 ,k  ,2,2 ,2]$ &  $+\ [6,(2)_{k-2}] $ & $+\ [2] $ & $7\geq k \geq 4$ 
		\\ \ref{item:ht=3_nu_3_YV_fork_T=[2,2,2]} &
		$\langle
		2;[2,2,2],[3] ,[2]\rangle$ & 	$+\ [2,4,2 ,2  ,2]$ & $+\ [7,2,2] $ & $+\ [2] $ &
		\\ \ref{item:ht=3_nu_3_ZV_U-3} &
		$\langle 2;[2,2,2] ,[3]  ,[2]\rangle$ & $+\ 
		[2  ,4,2 ,3]$ & $+\ [3] $ & $+\ [2] $ &
		\\ \ref{item:ht=3_nu_3_ZV_U-2} &
		$\langle 3;[2,2,2] ,[2]  ,[2]\rangle$ & $+\ 
		[3  ,3,2 ,3]$ & $+\ [3]$ & $ +\ [2] $ &
		\\ \ref{item:ht=3_nu_3_YV_U-2_k=2} &
		$\langle 2;[2,3,2] ,[2],[2]\rangle$ & $+\ 
		[4  ,2 ,2]$ & $+\ 
		[4,2 ]$ & $+\ [2] $ &
		\\ \ref{item:ht=3_nu_3_YV_U-2_k>2} &
		$\langle   k;[2,2,2] ,[2]  ,[2]\rangle$ & $+\ 
		[3,(2)_{k-2},3,2]$ & $+\ 
		[4]$ & $+\ [2] $ & $k\geq 3$ 
		\\ \ref{item:ht=3_nu_3_YV_U-2_m=3} &
		$\langle 3;[2,2,2] ,[2]  ,[2]\rangle$ & $+\ 
		[3  ,3,3 ,2]$ & $+\ 
		[4,2]$ & $+\ [2] $ &
		\\ \ref{item:ht=3_nu_3_YV_U-3_k>2} &
		$\langle   k;[2,2,2] ,[3]  ,[2]\rangle$ & $+\ 
		[2  ,3,(2)_{k-3},3,2 ,2]$ & $+\ 
		[4]$ & $+\ [2] $ & $k\geq 3$ 
		\\ \ref{item:ht=3_nu_3_YV_U-3_k=2} &
		$\langle 2;[2,2,2] ,[3]  ,[2]\rangle$ & $+\ 
		[2  ,4,k ,2]$ & $+\ 
		[4,(2)_{k-2}]$ & $+\ [2] $ & $k\in \{2,3\}$ 
		\\ \ref{item:ht=3_nu_3_YV_U'-2_l=3} &
		$\langle 2;[2,3,2],[2] ,[2]\rangle$ & $+\ 
		[2  ,3,2 ,2]$ & $+\ 
		[5]$ & $+\ [2] $ &
		\\ \ref{item:ht=3_nu_3_YV_U'-2_k=4} &
		$\langle 4;[2,2,2],[2] ,[2]\rangle$ & $+\ 
		[3  ,3 ,2]$ & $+\ 
		[2 ,4 ,2,3 ]$ & $+\ [2] $ &
		\\ \ref{item:ht=3_nu_3_YV_U'-2_k=3} &
		$\langle 3;[2,2,2],[2] ,[2]\rangle$ & $+\ 
		[3  ,k ,2]$ & $+\ 
		[(2)_{k-2},4 ,3 ]$ & $+\ [2] $ & $k\in \{3,4\}$ 
		\\ \ref{item:ht=3_nu_3_YV_U'-2_k>2_m=2} &
		$\langle k;[2,2,2],[2] ,[2]\rangle$ & $+\ 
		[3  ,2 ,2]$ & $+\ 
		[4,(2)_{k-3},3 ]$ & $+\ [2] $ & $k\geq 3$ 
		\\ \ref{item:ht=3_nu_3_YV_U'-3_chain_k>2} &
		$\langle k;[2,2,2],[3] ,[2]\rangle $ & $+\ 
		[4,3,(2)_{k-3},3,2 ]$& $+\ 
		[3  ,2 ,2]$  & $+\ [2] $ & $k\geq 3$ 
		\\ \ref{item:ht=3_nu_3_YV_U'-3_chain_k=2} &
		$\langle 2;[2,2,2],[3],[2]\rangle $ & $+\ 
		[2,4,4,(2)_{k-2}]$& $+\ 
		[3  ,k ,2]$  & $+\ [2] $ & $k\in \{2,3,4\}$ 
		\\ \ref{item:ht=3_nu_3_WXD} &
		$\langle   k;[(2)_{5}]  ,[2] ,[2]\rangle$ & $+\ [4 ,(2)_{k-3},3]  $ & $+\ [3]$ & $+\ [2]$ & $k\geq 3$
		\\ \ref{item:ht=3_nu_3_WYV} &
		$\langle k;[2,2,2] ,[3],[2]\rangle$ & $+\ [3  ,2 ,2]$ & $+\ [3,(2)_{k-3},3]$ & $+\ [2] $ & $k\geq 3$ 
		\\ \ref{item:ht=3_nu_3_WYV_k=2} &
		$\langle 2;[2,2,2] ,[3],[2]\rangle$ & $+\ [3  ,2 ,2]$ & $+\ [4]$ & $+\ [2] $ &  
		\\ \ref{item:ht=3_nu_3_WYC} &
$\langle 2;[3,2,2]  ,[2] ,[2]\rangle$ & $+\ [4  ,2 ,2]$ & $+\ [3] $ & $+\ [2] $ &
		\\ \ref{item:ht=3_nu_3_WYT} &
		$\langle k;T,[3],[2]\rangle $ & $+\ [(2)_{4},3 ,(2)_{k-2}]*T^{*} $ & $+\ [3] $ & $+\ [2] $ &$k\geq 2$, $d(T)\leq 5$ 
		\\ \ref{item:ht=3_nu_3_WYU} &
		$\langle 2;[3],[2] ,[2] \rangle $ & $+\ [4,k,2 ,2  ,2]$ & $+\ [3,(2)_{k-2}] $ & $+\ [2] $ & $k\geq 3$ 
		\\ \ref{item:ht=3_nu_3_WYUU} &
		$\langle 2;[2,2],[3],[2]\rangle $ & $+\ [5,k,2 ,2  ,2]$ & $+\ [3,(2)_{k-2}] $ & $+\ [2] $ & $k\in \{3,4\}$ 
		\\ \ref{item:qL3_A0_fork} &
		$\langle k;T,[3],[2]\rangle$  & $+\ [2,2,3,(2)_{k-2}]*T^{*}$ & $+\ [3,2,2]$ & $+\ [2]$ & $k\geq 2$,  $d(T)\leq 5$ 
		\\ \ref{item:qL3_rA0_T=[3]} &
		$\langle {2};[3],[3],[2]\rangle$ & $+\ [2,{3},2,2]$ & $+\ [2,4,3,2]$ & $+\ [2]$ &		 
		\\ \ref{item:qL3_rA0} & $\langle {2},[2],[3],[2]\rangle$ & $+\ [(2)_{k-2},{3},2,2]$ & $+\ [4,k,2]$ & $+\ [2]$ & $k\geq 3$  
		\\ \ref{item:qL3_rA0_l=3} & 
		$\langle {k};[2],[3],[2]\rangle$ & $+\ [2,{3},2,2]$ & $+\ [3,(2)_{k-3},3,2]$ & $+\ [2]$ & $k\in \{3,4\}$ 
		\\ \ref{item:qL3_rA0_s=2} & 
		$\langle {2};[2,2],[3],[2]\rangle$ & $+\ [(2)_{k-2},{3},2,2]$ & $+\ [5,k,2]$ & $+\ [2]$ &  $6\geq k\geq 3$ 
		\\ \ref{item:qL3_rA0_s=2_l=3} &
		$\langle {3};[2,2],[3],[2]\rangle$ & $+\ [2,{3},2,2]$ & $+\ [4,2,3,2]$ & $+\ [2]$ &
		\\ \ref{item:qL3_rA0_s=3_l=3} &
		$\langle {2};[2,2,2],[3],[2]\rangle$ & $+\ [2,{3},2,2]$ & $+\ [6,3,2]$ & $+\ [2]$ &
		\\ \ref{item:qL3_rA1} &
		$\langle 2;[(2)_{k-2},3],[2],[2]\rangle$ & $+\ [3,{k},2]$ & $+\ [{4},2,2]$ & $+\ [2]$ & $k\in \{3,4\}$  
		\\ \ref{item:qL3_rA1_k=3} &
		$\langle 3;[2,3],[2],[2]\rangle$ & $+\ [3,{3},2]$ & $+\ [3,{3},2,2]$ & $+\ [2]$ &
		\\[.5em] 
		\ref{item:ht=3_nu_3_YV_U'-2_T=[2]} &
		$\langle k;[2,2,2],[2] ,[2]\rangle$ & $+\ 
		\langle 2;[3],[2]  ,[2] \rangle$ & $+\ 
		[5,(2)_{k-3},3 ]$ & $+\ [2] $ & $k\geq 3$	
		\\ \ref{item:ht=3_nu_3_YV_U'-2_T=[2,2]} &
		$\langle 3;[2,2,2],[2] ,[2]\rangle$ & $+\ 
		\langle 2;[(2)_{k-4}],[3]  ,[2]\rangle$ & $+\ 
		[k,3 ]$ & $+\ [2] $ & $k\in \{5,6\}$
		\\ \ref{item:ht=3_nu_3_YV_U'-3_fork_k>2} &
		$\langle k;[2,2,2],[3] ,[2]\rangle $ & $+\ 
		\langle 2;[3],[2]  ,[2] \rangle$ & $+\ 
		[5,3,(2)_{k-3},3,2 ]$ & $+\ [2] $ & $k\geq 3$ 
		\\ \ref{item:ht=3_nu_3_YV_U'-3_fork_k=2} &
		$\langle 2;[2,2,2],[3] ,[2]\rangle $ & $+\ 
		\langle 2;[(2)_{k-4}],[3], [2]\rangle$ & $+\ 
		[k,4,2 ]$ & $+\ [2] $ & $k\in \{5,6\}$ 
		\\ 	\ref{item:qL3_[4]} &
		$\langle 2;[3],[2],[2]\rangle$ & $+\ \langle k;[4],T,[2]\rangle$ & $+\ [2,2,3,(2)_{k-2}]*T^{*}$ & 
		$+\ [2]$ & $d(T)\leq 3$  
		\\ \ref{item:qL3_[5]} &
		$\langle 2;[3],[2,2],[2]\rangle$ & $+\ \langle 2;[5],[2]\dec{1},[2]\rangle$ & 
		$+\ [4,2,2]$ & $+\ [2]$ &
	\end{longtable}
\vspace{-1em}

\begin{table}[htbp]
	\begin{tabular}{r|llll|l}
		\cite[3.12]{PaPe_ht_2} & \multicolumn{5}{c}{Singularities ($\cha\kk=2$)} \\ \hline\hline
		(1) & $\langle k;T, [2],[3]\rangle$ & $+\ [2,2,3,(2)_{k-1}]*T^{*}$ & $+\ [3,2,2]$ &$+\ [2]$ & $d(T)=6$, $k\geq 2$  \\
		(2) & $\langle k;T, [2],[3]\rangle$ & $+\ [2,4,(2)_{k-1}]*T^{*}$ & $+\ [2,3,2,2,2]$ & $+\ [2]$ & $d(T)=6$, $k\geq 2$  \\
		(3) & $\langle k;T,[2,2,2],[2] \rangle$ & $+\ [4,(2)_{k-1}]*T^{*}$ & $+\ [3,2,2]$ & $+\ [2]$ & $d(T)=4$, $k\geq 2$ 
	\end{tabular}
	\caption{Log canonical, but non--log terminal del Pezzo surfaces of rank one and height $3$, not admitting descendants with elliptic boundary; necessarily $\cha\kk=2$. Each type is realized by a unique surface, which swaps vertically to the one from Example \ref{ex:w=2_cha=2}\ref{item:swap-to-nu=3}.}
	\label{table:ht=3_non-lt}
\end{table}
\end{small}
\vspace{-1em}

\bibliographystyle{amsalpha}
\bibliography{bibl}
\end{document}